%% file: manuscript.11.tex
\documentclass[reqno,twoside,a4paper,11pt]{amsart}
\usepackage{amsmath,amssymb,dsfont,amsthm}
\usepackage{graphics}
\usepackage{xypic}
\usepackage{latexsym}
\usepackage{amsfonts}
\usepackage{geometry}
\usepackage{tensor}
\usepackage{bm,color}
\usepackage{relsize}
\usepackage{hyperref}

\hypersetup{
    colorlinks=true,
    linkcolor=blue,
    filecolor=magenta,      
    urlcolor=cyan,
    citecolor=blue,
    pdftitle={Sharelatex Example},
    bookmarks=true,
    pdfpagemode=FullScreen,
}

\date{\today}

\title[Torsion of spaces with conical singularities]{Intersection torsion and  analytic torsion of spaces with conical singularities}

\author{L. Hartmann}
\address{\tt UFSCar, Universidade Federal de So Carlos, S\~{a}o Carlos, Brazil.}
\email{hartmann@dm.ufscar.br}

\author{M. Spreafico}
\address[Mauro Spreafico]{\tt Department of mathematics and physics, Universit\'a del Salento, and INFN Lecce, Italy}
\email{mauro.spreafico@unisalento.it}

\numberwithin{equation}{section}

\newtheorem{theo}[subsection]{Theorem}
\newtheorem{lem}[subsection]{Lemma}
\newtheorem{corol}[subsection]{Corollary}
\newtheorem{defi}[subsection]{Definition}

\newtheorem{prop}[subsection]{Proposition}
\newtheorem{rem}[subsection]{Remark}

\newcommand{\beq}{\begin{equation}}
\newcommand{\eeq}{\end{equation}}

\newcommand{\z}{\zeta}
\newcommand{\te}{\theta}
\newcommand{\de}{{\delta}}
\newcommand{\vv}{{\varphi}}
\newcommand{\ep}{\epsilon}
\newcommand{\al}{\alpha}
\newcommand{\be}{\beta}
\newcommand{\y}{\upsilon}

\newcommand{\ka}{\kappa}
\newcommand{\la}{\lambda}
\newcommand{\ga}{\gamma}

\renewcommand{\Re}{{\rm Re}}
\renewcommand{\Im}{{\rm Im}}

\newcommand{\Aut}{{\rm Aut}}

\renewcommand{\b}{{\partial}}

\newcommand{\rk}{{\rm rk}}
\newcommand{\da}{{\dagger}}
\renewcommand{\det}{{\rm det\hspace{1pt}}}
\renewcommand{\ss}{{\hspace{10pt}}}
\newcommand{\bu}{{\bullet}}
\newcommand{\Sp}{{\rm Sp}}
\newcommand{\sgn}{{\rm sgn}}
\renewcommand{\d}{{\rm d}}
\newcommand{\supp}{{\rm supp}}
\newcommand{\ad}{{\rm ad}}
\newcommand{\oo}{\mathring}

\newcommand{\Det}{{\rm Det\hspace{1pt}}}
\newcommand{\Tr}{{\rm Tr\hspace{.5pt} }}
\renewcommand{\t}{\tilde}

\newcommand{\Z}{{\mathds{Z}}}
\newcommand{\R}{{\mathds{R}}}
\newcommand{\C}{{\mathds{C}}}
\newcommand{\F}{{\mathds{F}}}

\newcommand\sh{{\rm sh}}
\newcommand\ch{{\rm ch}}

\newcommand\e{{\rm e}}

\newcommand{\cs}{\mathsf{c}}
\newcommand{\bs}{\mathsf{b}}
\newcommand{\es}{\mathsf{e}}
\newcommand{\gs}{\mathsf{g}}
\newcommand{\hs}{\mathsf{h}}
\newcommand{\zs}{\mathsf{z}}

\newcommand{\vs}{\mathsf{v}}

\newcommand{\ms}{\mathsf{m}}
\newcommand{\ns}{\mathsf{n}}
\newcommand{\alphas}{\mathsf{\alpha}}
\newcommand{\betas}{\mathsf{\beta}}

\newcommand{\g}{{\mathsf{g}}}

\newcommand{\uf}{\mathfrak{u}}
\newcommand{\ff}{\mathfrak{f}}
\newcommand{\vf}{\mathfrak{v}}
\newcommand{\lf}{\mathfrak{l}}
\newcommand{\tf}{\mathfrak{t}}
\renewcommand{\sf}{\mathfrak{s}}
\newcommand{\df}{\mathfrak{d}}
\newcommand{\pf}{{\mathfrak{p}}}
\newcommand{\mf}{{\mathfrak{m}}}
\newcommand{\xs}{\mathsf{x}}
\newcommand{\ys}{\mathsf{y}}
\newcommand{\ks}{\mathsf{k}}
\newcommand{\as}{\mathsf{a}}
\newcommand{\af}{{\mathfrak a}}
\newcommand{\ds}{\mathsf{d}}

\renewcommand{\L}{\mathcal{L}}

\renewcommand{\P}{\mathcal{P}}
\renewcommand{\H}{\mathcal{H}}

\newcommand{\FF}{\mathcal{F}}
\newcommand{\BB}{\mathcal{B}}
\newcommand{\CC}{\mathcal{C}}
\renewcommand{\SS}{\mathcal{S}}

\newcommand{\AF}{\mathfrak{A}}
\newcommand{\PF}{\mathfrak{P}}
\newcommand{\DF}{\mathfrak{D}}

\newcommand{\CS}{\mathsf{C}}
\newcommand{\DS}{\mathsf{D}}

\newcommand{\E}{\mathcal{E}}
\renewcommand{\L}{\mathcal{L}}
\newcommand{\ES}{\mathsf{E}}

\newcommand{\QQ}{\mathcal{Q}}

\newcommand{\B}{{\mathcal{B}}}
\newcommand{\Ha}{{\mathcal{H}}}
\newcommand{\A}{{\mathcal{A}}}
\newcommand{\OO}{{\mathcal{O}}}

\newcommand{\Na}{\mathcal{N}}
\newcommand{\UU}{\mathcal{U}}

\DeclareMathOperator*{\Rz}{Res_0}
\DeclareMathOperator*{\Ru}{Res_1}

\newcommand{\semicolon}{}

\DeclareMathOperator*{\TTCS}{\CS\hspace{-3.5pt}| \hspace{1pt} }

\DeclareMathOperator*{\TTb}{\b\hspace{-3.5pt}| \hspace{1pt} }

\newcommand{\TCS}{{\TTCS}}

\newcommand{\Tb}{{\TTb}}

\DeclareMathOperator*{\BBCS}{\CS\hspace{-5.5pt}\slash\hspace{1pt}}

\DeclareMathOperator*{\BBcs}{\cs\hspace{-4.5pt}\slash \hspace{1pt} }
\DeclareMathOperator*{\BBhs}{\hs\hspace{-5pt}\slash \hspace{1pt} }

\newcommand{\BCS}{{\BBCS}}

\newcommand{\Bcs}{{\BBcs}}
\newcommand{\Bhs}{{\BBhs}}

\begin{document}

\maketitle
\begin{abstract} We prove an extension of the Cheeger-M\"{u}ller theorem to spaces with isolated conical singularities: the $L^2$-analytic torsion coincides with the Ray-Singer intersection torsion on an even dimensional space, and they are trivial, while the ratio is non trivial on an odd dimensional space, and the anomaly depends only on the link of the singularities. For this aim, we develop on one side a combinatorial cellular theory whose homology coincides with the intersection homology of Gregory and Macpherson, and where the Ray-Singer intersection torsion is well defined. On the other side, we  elaborate the  spectral theory for the Hodge-Laplace operator on the square integrable forms on a space with conical singularities {\it \'a la}  Cheeger, and we extend the classical results of the Hodge  theory and the analytic torsion.
\end{abstract}

\tableofcontents

\input{introduction.11}

\vspace{10pt}
\centerline{\bf PART I}
\input{part1.11}

\vspace{10pt}
\centerline{\bf PART II}
\input{part2.11}

\vspace{10pt}
\centerline{\bf PART III}
\input{part3.11}

\appendix
\input{appendices.11}

\bibliography{HarSprBibliography.bib}
\bibliographystyle{amsalpha}

\end{document}

%% file: introduction.11.tex

\section{Introduction}

This work is devoted to answer the following question: what would it be, if it existed, an extension of the Cheeger-M\"{u}ller theorem on a space with conical singularities? This question, even if not explicitly stated, emerges clearly in a series of  works of Jeff Cheeger \cite{Che1,Che2,Che3}, dedicated to the study of the spectral properties of the  Hodge-Laplace operator on singular spaces.  Since then, several works appeared where some aspects of this problem have been considered, manly from the analytic side \cite{BZ,LR,Luc,Vis,BM1,ALMP,MaV,BM2,Les2,AG,Pfa,HLV1,HLV2,ARS2,ARS3,Lud,Ver3}. However, observe that most of these results are given under the assumption of the Witt condition, that is not requested in our approach. Indeed, a prerequisite to the solution of this problem is the definition of suitable extensions of both analytic and combinatorial torsion for singular spaces. 

The analytic side has been deeply investigated: the Cheeger-M\"{u}ller theorem has been extended to manifolds with boundary,  initially for a product metric near the boundary, by L\"{u}ck \cite{Luc} (see also \cite{LR}), and consequently in the general case by Br\"{u}ning and Ma \cite{BM1}, exploiting  techniques of Bismut and Zang \cite{BZ}. In particular, this approach also allowed to produce explicit glueing formulas for analytic torsion on manifolds \cite{BM2}, extending the original formula of Vishick for a product metric structure near the glueing \cite{Vis}. This line of investigation lead to the recent work of Lesh \cite{Les2}, where the Vishick glueing formula for analytic torsion is generalised to the case of singular spaces. These results suggest to tackle the problem of the Cheeger-M\"{u}ller theorem on singular spaces decomposing the space into a regular part, a smooth manifold with boundary, and a singular part, the model of the singularity. A number of works appeared where analytic torsion for the model cone has been described, in the case of a flat cone \cite{Ver1,Ver2,HS1,HMS,HS2,HS3,Har,MV,HS4,HS5}, all essentially based on the technique of decomposing the spectrum of the relevant Hodge-Laplace operator on the cone over the one of the boundary, described in \cite{Spr9}, applied originally in \cite{Spr3}, and inspired by works of Br\"{u}ning and Seeley \cite{BS1,BS2}. 
On the combinatorial side there is much less available material, essentially the works of Dar \cite{Dar1,Dar2}, and \cite{HS4}, all for an even dimensional cone. 

The aim of this work is to develop on one side a theory that extends the classical theory  of Reidemeister combinatorial torsion to spaces with conical singularities, and on the other side a theory that extends the classical theory of analytic torsion to spaces with conical singularities, with the final aim of describing the relation between the two torsions.

The first part of the work, Sections \ref{tor} to \ref{ss8},  is devoted to the combinatorial aspects. Our main purpose is to introduce and discuss a combinatorial torsion for spaces with conical singularities (see Section \ref{s0}) that coincides with Reidemeister torsion when the links of the singularities are homology manifolds. Since spaces with conical singularities are pseudomanifolds with singular locus of dimension zero, the natural geometric setting is the one introduced by Goresky and Macpherson to define intersection homology \cite{GM1,GM2,GM3} (see also \cite{HS4}). We refer in particular to  \cite{GM1}, where the definition is topological/combinatorial, more that to  \cite{GM2}, where a more abstract sheaf theoretical point of view is presented. In this approach, the intersection homology groups are the homology groups of some chain complexes constructed from a simplicial triangulation of the underling space, by selecting some ad hoc basic sets \cite[3.4]{GM1} (see also \cite{HS4}, where the construction is revisited with a purpose similar to the present one). Here, we prefer to work with regular CW complexes instead than with simplicial complexes. We define CW basic sets, whose family we call intersection cellular family, Definition \ref{basicset}, that reduce to the basic sets of Goresky and Macpherson if the triangulation is simplicial. The intersection cellular family provides with a (non cellular) filtration of the given space, so it comes with a naturally associated chain complex, as the one associated to the classical CW structure. We call this chain complex the intersection chain complex of the given space, Definition \ref{intcellcomp}. 

Our construction is developed for proper pairs of pseudomanifolds with isolated singularities, definitions of Section \ref{s0}, covering the case of a space with isolated singularities and smooth boundary, namely the cone over a compact manifold. We  arrive to define intersection homology groups for a cone and for  a space with conical singularities, Definitions \ref{homcone} and  \ref{homsing}, that coincide,   Proposition \ref{p2.3} with the intersection homology groups of Goresky and MacPherson  \cite{GM1,KW}. In particular, recall that intersection homology was introduced with the aim of recovering for singular spaces the classical Poincar\'e duality of manifolds. We describe in Section \ref{Poincare} the construction of a dual complex for a pseudomanifold with isolated singularities and smooth boundary, and we prove duality of intersection homology groups in Theorems \ref{dualhomC} and \ref{dualhomX}. Next, we turn to the our main objective. For, we introduce in Section \ref{algtor}, an algebraic setting for the intersection chain complexes. Here a difficult technical  gap must be tackled, that is not clearly detectable using the standard approach in \cite{GM1}. In fact, the intersection complexes are free but a preferred graded chain basis is not well defined by the cells, see Remark \ref{intersection basis}. However, the topology of the cone permits to define a suitable graded chain basis for the intersection chain complex, see Lemmas \ref{l7.3} and \ref{l7.10}. We are then in the position of defining the intersection torsion of a cone and of a pseudomanifold with isolated singularities and smooth boundary, Theorems \ref{t7.1} and \ref{t7.2}. 

All the necessary combinatorial part of the theory having been developed, we need some geometry: in order to define the Ray-Singer torsion we need De Rham maps for the intersection homology theory so far defined. Since for that the development of the spectral analysis and of the Hodge theory is necessary, the definition of the Ray-Singer intersection torsion has been moved at the end of the second and of the third parts, the analytic ones, Sections \ref{hodgederhamcone} and \ref{hodgederhamX}, respectively. There, we  construct intersection De Rham maps from the spaces of the square integrable harmonic forms of some suitable Hodge-Laplace operators on the cone and on a space with conical singularity, Theorems \ref{derham1} and \ref{derham2}. As a corollary we have the Hodge Theorem. This part of the theory is developed only for some particular perversities, the ones that correspond to square integrable forms (compare also with \cite{Nag,You}). 
We may now define the Ray-Singer intersection torsion for a cone and for a space with conical singularity, Definitions \ref{torC} and \ref{torX}.  We give some natural algebraic  characterisation of them, Theorems \ref{t7.29} and \ref{t7.37}. This shows the intersection torsion is an invariant with an algebraic part that only depends on the homology of the underlying space, and a geometric part, that depends on the norm of the homology line. We conclude this part discussing the behaviour of intersection torsion under duality, Theorem \ref{t15.8}, and under a variation of the metric, Theorem \ref{variationRS}.

In the second part of the work, Sections \ref{ss2} to \ref{cheegercone}, we investigate the analytic side of the  problem. As in the approach of Cheeger, this program consists in an extension of the $L^2$ theory of the Hodge-Laplace operator on a space with conical singularity. Our approach to the problem, as observed, is to split the underground space into a model cone and a smooth manifold, and glue these two pieces. This is the original idea of Cheeger, see also \cite{CY,Gaf1,Gaf2}. However, our approach as some substantial differences. First, we tackle a cone with a generic slope, that we also call deformed cone, instead that a flat cone. 
This approach may be reinterpreted as follows. Let decompose the space $X$ with a conical tip as $X=C(W)\cup_W Y$, where $W=\b(Y)$, for some smooth manifold $Y$. Then, we add to a smooth chart for $Y$ a new open set diffeomorphic to the punctured Euclidian space, and we develop analysis on the resulting new chart. The second difference, is in the definition the Hodge-Laplace operator. The striking feature of the $L^2$ spectral geometry on a space with conical singularity $X$ is that Hodge theory produce topologic invariants of $X$ that indicate intersection cohomology rather than  classical cohomology, unless $X$ is a rational homology manifold. However, the original definition of the Hodge-Laplace operator introduced by Cheeger works only if the middle homology of the boundary of $Y$ is trivial when the dimension of $Y$ is even. We introduce some boundary conditions at the tip of the cone that permit to define the suitable Hodge-Laplace operator and recover the Hodge theorem in all cases. 

All this is in the second part of the work. 
An essential technical tool to deal,  first with the Hodge theory on the deformed cone, as defined in Section \ref{geocone},  and second with the asymptotic analysis of the relevant spectral functions, is the theory of the singular Sturm Liuouville operators on an interval $(0,l]$. In fact, in the previous cited works such analysis being performed on the flat cone, the classical results in the theory of Bessel functions were available and sufficient for all purpose (the idea originated probably in the works of Br\"{u}ning and Seeley \cite{BS1,BS2}). We need to extend this theory to a more general family of Sturm Liouville problems. This theory is developed in the Section \ref{ss2}. We develop some basic results on the solutions of main differential equation \ref{eqdiff1}, we  define the relevant self-adjoint bounded below operators, Theorem \ref{extensions} and  Definition \ref{defi0}, and we prove that they have compact resolvent with square integrable kernel, and a complete discrete spectral resolution, Corollaries \ref{c3.34}, \ref{c3.35}, \ref{c3.36}. This permits to introduce the main spectral function for these operators,  the logarithimic Gamma function, in Section \ref{spectralfunctions}. The main point is that the logarithmic Gamma function is a global function that contain deep spectral information. In particular, its asymptotic expansion for large values of the spectral variable determines the zeta determinant \cite[2.10]{Spr9}, and therefore the analytic torsion. Following an idea introduced in \cite{Spr10}, we show in Section \ref{spectralfunctions} that the asymptotic expansion of the logarithmic Gamma function is determined by that of a fundamental system of normalised solutions, Definition \ref{defi1},  of the eigenvalues equation of the formal differential operator, equation \ref{eqdiff1}. We need then asymptotic expansion of such normalised solutions,  both for large values of the spectral variable, and for large values of some parameter uniformly in the spectral variable. This technical part of the work is developed in Section \ref{largela}, applying standard techniques of asymptotic analysis, like the WKB  and Green Liouville  methods \cite{Mur,Olv}, and some new ideas based on introducing a perturbation of the flat case. 

Next, we briefly describe the geometry of the deformed cone, introducing the formal Hodge-Laplace operator, and we proceed to decompose the local analysis on the cone into local analysis on the interval and global  analysis on the section, Section \ref{ancone}. In particular, we characterise the solutions of the harmonic, Proposition \ref{harmonics}, and of the eigenvalues equations, Proposition \ref{l2b}. This permits to define some concrete Hodge-Laplace operators $\Delta_{\rm bc,\pf}$, with classical boundary conditions ${\rm bc}$ at the boundary of the cone, and some new boundary conditions at the tip of the cone, Section \ref{sec4.6}. More precisely, these conditions are of two types $\pf=\mf$ and $\pf=\mf^c$, that corresponds to the two middle perversities appearing in the intersection theory described in the first part. We prove that the operators $\Delta_{\rm bc,\pf}$ are self-adjoint and non negative, and have a compact resolvent with integral kernel, Section \ref{secspectral}. 
We give explicit description of the harmonics, Section \ref{harmonicforms}, of the spectrum and of the eigenfunctions, Section \ref{spectrumC}, obtaining the Hodge decomposition for the space of square integrable  function on the cone, Theorem \ref{completebasis}. A notable point is that we prove that the harmonics coincide with the intersection cohomology without any assumption, in particular without requiring the Witt condition \cite[Prop. 2.5]{Sie}, namely that the middle homology (of the base space) is trivial in the even dimension base case.

We are now in the position of defining the absolute and the relative  analytic torsion with middle perversities for the cone of a manifold, Section \ref{torsionzeta} and \ref{cone}. In particular, the information on the spectrum and the spectral resolution gives a duality result for the analytic torsion of the cone, Theorem \ref{dt}. We decompose the torsion zeta function into a global part and a boundary part, following the approach introduced and used in \cite{HS2,HS3,HS5}.  
The information on the spectrum makes it possible to apply the Spectral Decomposition Theorem \cite[3.9]{Spr9} (for simplicity we collect the main results of this theory in Appendix \ref{backss}) to study the logarithmic Gamma function associated to the operators $\Delta_{\rm bc,\pf}$. Essentially, this consists in decomposing the double sums appearing in the logarithmic Gamma function into sums of simple sums, plus regularising terms. The method is developed in Section \ref{cone}. In order to complete the description of the analytic torsion on the cone, it is necessary to identify the contribution of the boundary. We prove that this contribution is precisely the same as it would be if the cone were a smooth manifold. In other word we prove an extension of the work of Br\"{u}ning and Ma \cite{BM1,BM2} for the cone. 
This we do repeating the same analysis made on the cone on a deformed conical frustum, Section \ref{calcfrust}, and comparing the resulting boundary contributions in the analytic torsion, Section \ref{Sing}. Accomplished this final step, we complete our description  the absolute and the relative  analytic torsion with middle perversities on a cone, Theorem \ref{t1}. After that, as previously mentioned, we describe De Rham maps and Ray-Singer torsion, Section \ref{hodgederhamcone}, so to complete the proof of our extension of the Cheeger-M\"{u}ller Theorem \ref{tcm} on a cone.   We observe that in the even dimensional case, the result is analogous to the classical one, in the boundary version, while in the odd dimensional case two anomaly terms appear. In particular, we note that these anomaly terms depend only on the section of the cone. 

The third part of the work, starting in Section \ref{spaceX},  is devoted to spaces with singularity of conical type, namely we glue the results of the previous part with the classical results valid for smooth manifolds. A space with an isolated metric conical singularity, see Section \ref{spaceX}, decomposes by definition as $X=C(W)\cup Y$, where $Y$ is a compact connected orientable Riemannian manifold with boundary $W$, and $C(W)$ the deformed finite metric cone over $W$, considered in the second part. Exploiting this decomposition and the analysis developed on the cone, we may define concrete self-adjoint extensions $\Delta_{X,\pf}$, where  either $\pf=\mf$ or $\pf=\mf^c$,  of the formal Hodge-Laplace operator on square integrable forms on $X$, essentially by fixing some boundary conditions at the tip of the cone, Definition \ref{XLaplace}. 
We prove that the operators $\Delta_{X,\pf}$ are non negative with compact resolvent with square integrable kernel, Proposition  \ref{compresX}. So we have a complete discrete spectral resolution and the Hodge decomposition. In particular, this shows the existence of a complete asymptotic expansion for the trace of the kernel of the heat operator, and gives a direct approach to the determination of (at least the first) coefficients. This was the original problem faced by Br\"{u}ning and Seeley in \cite{BS1}, and  Cheeger in \cite{Che2}. 

Exploiting  the results in the second part, we have the Hodge Theorem that proves the existence of a natural isomorphism between the spaces of the $L^2$ harmonic forms and the intersection homology Theorem \ref{hodgetX}, and the De Rham maps \ref{derham2} (again without any homological conditions).  
Moreover, we may define the analytic torsion of $X$ with middle perversities, Section \ref{XX}. In particular, we identify the variation of the analytic torsion under the variation of the metric, provided that the last is constant near the singularity, and we prove that the variation is the same as in the smooth case, as described by Ray and Singer \cite{RS}, and coincides with the variation of the Ray-Singer intersection torsion, Theorem \ref{variationX}. 
In order to compare the analytic torsion  with the Ray-Singer intersection torsion, and  to complete the proof of our extension of the Cheeger-M\"{u}ller theorem on a space with  conical singularities, Theorem \ref{tcm}, we only need to prove that they coincide for some particular metric. This we do using a gluing formula of M. Lesh \cite{Les2}, that extends the original gluing formula of S.M. Vishik \cite{Vis}, and holds if the metric is a product near the gluing. In order to apply the formula of Lesh, we need to identify  $\Delta_{X,\pf}$ with the Laplace operator of some Hilbert complex with finite dimension spectrum, as defined in \cite{BL0}. For, we study the De Rham complex on the cone, and we prove that there exist Hilbert complexes with finite dimension spectrum whose associated Laplace operator is $\Delta_{X,\pf}$, Proposition \ref{p18.6}. Observe that we do not construct such an Hilbert complex (this should be an interesting project), but just prove its existence, using the explicit description of the self-adjoint extensions of the Hodge-Laplace formal operator given in Theorem \ref{selfext}. In other words, we have proved that $\Delta_{X,\pf}$ is an ideal boundary condition with finite dimension spectrum for some De Rham complex of forms on the cone. 
This permits to apply the glueing formula in \cite{Les2}, and to complete the proof of the extension of the Cheeger-M\"{u}ller Theorem \ref{tcm}. In particular, we show that the anomaly terms appear only if the dimension of the space is even, and they are the same as in the case of the cone, namely they depends on the link of the singularity. 

We observe that our analysis implies the extension of the mentioned results of Lesh Vishick (glueing formula) and Br\"{u}ning Ma (boundary term) to spaces with conical singularities, Theorems \ref{LV} and \ref{BMext}.

%% file: part1.11.tex

\section{Preliminaries on RS torsion} 
\label{s1}

\subsection{Whitehead torsion  of a chain complex}
\label{tor}

We briefly recall the definition of the torsion of a chain complex. We follows the classical definition of Milnor \cite{Mil0}, but with a little change of notation. Let $R$ be a ring (with unit) with the  invariant dimension property, and $M$  a finitely generated  free (left) $R$-module (all our modules will be of this type if not otherwise stated). Let $U$ be a subgroup of the group $R^\times$ of units of $R$, and let $K_U(R)=K_1(R)/U$ denotes the quotient of the {\it Whitehead group of the ring} $R$ by the subgroup generated by the  classes of the elements of $U$.  In particular, if  $\pi$ is  a group, and may consider the  group ring $\Z \pi$. It is well known that $\Z\pi$ satisfies the  invariant property dimension, and then above construction goes through giving  $K_U(\Z\pi)$. The group
\[
Wh(\pi)=K_{\pm \pi}(\Z\pi),
\]
is called the {\it Whitehead group of the group} $\pi$.

Let $\xs=\{x_1,\dots, x_n\}$ and $\ys=\{y_1,\dots, y_n\}$ be two bases for $M$. We denote by $(\ys/\xs)$ the non singular $n$-square matrix over $R$ defined by the change of bases ($y_j=\sum_k(\ys/\xs)_{j,k}x_k$), and we denote by $[(\ys/\xs)]$ the class of $(\ys/\xs)$ in the Whitehead group $K_U (R)$.

If we have an exact sequence of free $R$ modules
\[
\xymatrix{
0\ar[r]&M_1\ar[r]^i&M\ar[r]&M_2\ar[r]&0,
}
\]
with preferred bases $\ms_1$ for $M_1$, and $\ms_2$ for $M_2$, then a preferred basis for $M$ is naturally given by the set $i(\ms_1)\cup\hat \ms_2$, where $\hat \ms_2$ denotes a lift. We denote this basis of $M$ by: $i(\ms_1)\hat \ms_2$. It is clear that the Whitehead class of this basis does not depend on the lift. Moreover, if $\ms_1'$ for $M_1$, and $\ms_2'$ for $M_2$, then
\[
[(i(\ms'_1)\hat \ms'_2/i(\ms_1)\hat \ms_2)]=[(i(\ms'_1)/i(\ms_1))]+[(\hat \ms'_2/\hat \ms_2)].
\]

Two bases  $\ms$ and $\ms'$ of an $R$ module $M$ are said to be {\it equivalent bases} if $[(\ms'/\ms)]=0$ in the   group $ K_U(R)$.  

We will work in the category of the based modules, i.e. the family of the pairs $(M,\ms)$ where $M$ is an $R$ module and $\ms$ an equivalence class of preferred bases for $M$.

We say that an exact sequence 
\[
\xymatrix{
0\ar[r]&M_1\ar[r]^i&M\ar[r]&M_2\ar[r]&0,
}
\]
of $R$ modules is a {\it based exact sequence of $R$ modules} if the modules are based modules $(M_1,\ms_1)$ and $(M_2,\ms_2)$, and the bases are {\it compatible}, i.e. if
\[
[(i(\ms_1)\hat \ms_2/\ms)]=0.
\]

\begin{rem} If not necessary, we will omit explicit reference to the inclusion $i$ and to the lift, namely we will write $\ms_1 \ms_2$ for the basis $i(\ms_1)\hat \ms_2$ of $M$.
\end{rem}

Let
$$
\xymatrix{
\CS_\bu:& \CS_n\ar[r]^{\b_n}&\CS_{n-1}\ar[r]^{\b_{m-1}}&\dots\ar[r]^{\b_2}&\CS_1\ar[r]^{\b_1}&\CS_0,
}
$$
be a bounded chain complex of finite length $m$ of (finite dimensional) free left $R$-modules. When not otherwise stated, all our complexes will be of this type.

Denote by $Z_q=\ker (\b_q:\CS_q\to \CS_{q-1})$,  $B_q=\Im (\b_{q+1}:\CS_{q+1}\to \CS_q)$, and  $H_q=Z_q/B_q$ the homology groups of $\CS$. Assume that all the chain modules $\CS_q$ have preferred bases $\cs_q=\{c_{q,1},\dots, c_{q,\dim(\cs_q)}\}$.

The family of the chain complexes that have these properties forms the category of the {\it based chain complexes}, whose elements are denoted by pairs   $(\CS_\bu,\cs_\bu)$.

Next,  assume that the homology modules $H_q(\CS_\bu)$ are free with preferred bases $\hs_q$, and let $\hat\hs_q=\{\hat h_{q,j}\}$ be a set of independent elements of $Z_q$ that is a lift of $\hs_q$. Fix a set of linearly independent elements $\bs_q=\{b_{q,1},\dots, b_{q,\dim(\bs_q)}\}$ of $\CS_q$ such that $\b_q(\bs_q)$ is a basis for $B_{q-1}$ for each $q$ (in other words we are choosing a lift of a basis of $B_{q-1}$) (note that it is not necessary to assume  $B_q$ to be free  with preferred bases, since we can use stable bases if necessary; thus, in order to avoid unnecessary technical details, we work directly with free bases,  see Remark \ref{main remark} below for details).  Considering the sequence

\[
\xymatrix{
0\ar[r]&B_q=\Im \b_{q+1}\ar[r]&Z_q\ar[r]_p& H_q\ar[r]&0,}
\]
a basis for $Z_q$ is given by the basis $\b_{q+1}(\bs_{q+1})$ of $B_q$ and the set $\hat\hs_q$. We denote this basis by $\b_{q+1}(\bs_{q+1})\semicolon \hat\hs_q$ (see \cite{Mil1} for details). The same argument applied to the sequence

\[
\xymatrix{
0\ar[r]&Z_q\ar[r]&\CS_q\ar[r]_{\b_q}& B_{q-1}=\Im \b_q\ar[r]&0,
}
\]
determines the basis $\b_{q+1}(\bs_{q+1})\semicolon\hat\hs_q \semicolon \bs_q$ of $\CS_q$. 

The {\it Whitehead torsion of the based chain complex $(\CS_\bu,\cs_\bu)$  with respect to the bases $\cs_\bu=\{\cs_q\}$, and $\hs_\bu=\{\hs_q\}$} is the class
\[
\tau_{\rm W}(\CS_\bu;\cs_\bu,\hs_\bu)=\sum_{q=0}^n (-1)^q [(\b_{q+1}(\bs_{q+1})\hat\hs_q\bs_q/\cs_q)],
\]
in the  Whitehead group $ K_U(R)$. The definition is well posed since it is possible to show that the torsion does  not depend on the bases $\bs_q$, and on the lifts of the homology basis.

\begin{rem}\label{volumeinvariance} Note that torsion only depends on the equivalence classes of the bases, namely: if $[(\bs_q/\bs_q')]=[(\hs_q/\hs_q')]=0$, then $\tau_{\rm W}(\CS_\bu;\cs_\bu,\hs_\bu)=\tau_{\rm W}(\CS_\bu;\cs'_\bu,\hs'_\bu)$.
\end{rem}

\begin{rem}\label{main remark} In order to define torsion we do not really need free modules, stably free modules are sufficient, but the definitions a bit more involved. In order to guarantee that all the modules necessary in the definition of torsion are stably free, it is sufficient to require that the chain modules and  homology modules are free. In such a case, it follows that the modules of the cycles and those of the boundaries are stably free and torsion can be defined. Due to these facts, we will proceed assuming that all the modules of chains, homology, cycles and boundaries are actually free, in order to avoid unnecessary complications. In some particular instances we will recall explicitly that actually we only have stably free modules, and in those instances we assume the present remark and we will proceed by  considering those modules as free.

In order to completely understand the necessity of stably free modules, recall that in the general case (i.e rings that are not principal ideal domains) submodules of free modules are not necessarily free.

\end{rem}

\subsection{Main properties of the Whitehead torsion of a chain complex}
\label{alg}

\begin{lem} Let $(\CS_\bu, \cs_\bu)$ and $(\CS_\bu',\cs_\bu')$ be  two based chain complexes and $\varphi:(\CS_\bu, \cs_\bu)\to(\CS_\bu',\cs_\bu')$ a based chain isomorphism. Assume that the homology modules are free with graded preferred bases
$\hs_\bu$ and $\hs_\bu'$, and that $\varphi_{*,\bu}(\hs_\bu)=\hs_\bu'$. Then, $\tau(\CS_\bu;\cs_\bu,\hs_\bu)=\tau(\CS_\bu';\cs_\bu',\hs_\bu')$, and we say that $\varphi$ is a simple chain isomorphism of based chain complexes.
\end{lem}

\begin{theo}\label{mil0} \cite{Mil1} Let 
\[
\xymatrix{
0\ar[r]&(\CS_\bu,\cs_\bu)\ar[r]^{i_\bu}&(\CS'_\bu,\cs_\bu')\ar[r]^{p_\bu}&(\CS_\bu'',\cs_\bu'')\ar[r]&0,
}
\]
be an exact sequence of based chain complexes (i.e. $[(\cs_q'/i_q(\cs_q)s_q(\cs_q''))]=0$, where $s_\bu$ denote a splitting map, namely the $s_q(\cs_q'')=\hat \cs_q''$ are lifts of the $\cs_q''$). 

Assume that the homology modules are free with graded preferred bases $\hs_\bu$, $\hs'_\bu$ and $\hs_\bu''$, respectively. 
Let $\Ha$ denote the free acyclic chain complex  defined by the terms of the exact long homology sequence associated to the short exact chain sequence above. More precisely, let  
\begin{align*}
\Ha_q&=H_q(\CS_\bu''),&\Ha_{q+1}&=H_q(\CS'_\bu),&\Ha_{q+2}&=H_q(\CS_\bu),\\
\xs_q&=\hs_q'',&\xs_{q+1}&=\hs_q',&\xs_{q+2}&=\hs_q.
\end{align*}

Then,
\[
\tau_{\rm W}(\CS'_\bu;\cs'_\bu,\hs_\bu')=\tau_{\rm W}(\CS_\bu;\cs_\bu,\hs_\bu)+\tau_{\rm W}(\CS''_\bu;\cs''_\bu,\hs_\bu'')+\tau_{\rm W}(\Ha;\xs_\bu,\emptyset).
\]

\end{theo}

\begin{rem}\label{remmil0} Observe that  the formula of the previous theorem is independent on the homology graded bases. \end{rem}

\subsection{Whitehead torsion  of a cellular  complex}\label{sub3.3} 

By a cellular complex we mean a regular connected finite CW complex \cite[pg. 216]{Mun}. Recall that a regular CW complex can be triangulated in such a way that each closed cell is the polytope of a subcomplex. Thus, we can always associate to a regular CW complex a simplicial complex.  

Let $(K,L)$ be a pair of connected finite cellular  complexes of dimension $m$, and $(\tilde K,\tilde L)$ its universal covering complex pair, and identify the fundamental group $\pi=\pi_1(K)$ with the group of the covering transformations of $\tilde K$. Note that covering transformations are cellular. Let $\CS_\bu(\tilde K,\tilde L;\Z)$ be the cellular chain complex of $(\tilde K,\tilde L)$ with integer coefficients. The action of the group of covering transformations makes each chain group $\CS_q(\tilde K,\tilde L;\Z)$ into a module over the group ring $\Z\pi$, and since $K$ is finite, each of these modules is $\Z \pi$-free  and finitely generated with preferred basis $\cs_q$ determined by the natural choice of the $q$-cells of $K-L$. 
We obtain a complex of free finitely generated modules over $\Z\pi$ that we denote by $\CS_\bu( K,  L;\Z\pi)$. Any fixed geometric lift $\tilde\cs_\bu$ of the cells in $K-L$  provides a preferred graded basis for $\CS_\bu(  K, L; \Z\pi)$, and hence $(\CS_\bu( K,  L; \Z\pi),\tilde \cs_\bu)$ is a based chain complex.   

If  the homology modules $H_q(\CS_\bu(  K,  L; \Z\pi))$ are free with preferred bases $\hs_q$,  and taking a set of independent elements of $ Z_q$ that is a lift of $\hs_q$, we may consider the  Whitehead torsion $\tau_{\rm W}(\CS_\bu(  K,  L; \Z\pi);\tilde\cs_\bu,\hs_\bu)\in K_U(\Z\pi)$ has  defined as in Section \ref{tor},  
for any subgroups $U$ of units of $\Z\pi$. If $U=\pm\pi$, then the torsion does not depend any more on the choice of the representative lift of the graded cell basis $\cs_\bu$ used to define the preferred graded chain basis on the covering space, that will then be denoted by the same symbol $\cs_\bu$ (since any different choice will have change of basis matrix in $\pi$, and hence trivial torsion), thus we call the resulting element of $Wh(\pi)$,   the {\it Whitehead torsion of the cellular pair $(K,L)$ with respect to the graded homology basis $\hs_\bu$}, and we write:
\[
\tau_{\rm W}(K,L;\hs_\bu)=\tau_{\rm W}(\CS_\bu(  K,  L; \Z\pi); \cs_\bu,\hs_\bu).
\]
 
This torsion does not depend on the choice of the lifts of the cells used to define the preferred chain basis, since any different choice will have change of basis matrix in $\pi$, and hence trivial torsion. 

Given a CW complex $K$,  a CW complex $K'$ is a {\it subdivision} of $K$ if the underlying space $|K'|$ is equal to $|K|$, and if each open cell of $K'$ is contained in a open cell of $K$, so that the identity map $K\to K'$ is cellular. Similarly the pair $(K',L')$ is a subdivision of $(K,L)$ if $K'$ is a subdivision of $K$ and $L'$ is a subdivision of $L$  \cite[pg. 378]{Mil0}. For a finite regular CW complex, a subdivision is consistent with the subdivision of the associated triangulations, as subdivision of a simplicial complexes \cite[pg. 83]{Mun} (see also  \cite{Spa}). Since we may reduce all our complexes to simplicial complexes, saying a subdivision we mean if necessary the barycentric subdivision, see \cite[pg. 85]{Mun}.

Note that, if $K$ has a subdivision $K'$, the universal covering $\tilde K'$ of the last is a subdivision of $\tilde K$.

\begin{theo}\label{mil}({Milnor} \cite{Mil0}) Let  $(K,L)$ be a pair of connected finite cellular  complexes, and $(K';L')$ a subdivision pair. Assume the homology of $(K,L)$ is free with graded basis $\hs_\bu$. Then, 
$\tau_{\rm W}(\CS_\bu(  K,  L; \Z\pi); \cs_\bu,\hs_\bu)=\tau_{\rm W}(\CS_\bu(  K',  L'; \Z\pi); \cs'_\bu,\hs_\bu)$, where the graded chain bases are the cell bases, and we identified the homology.
\end{theo}

\subsection{Whitehead torsion of a twisted complex and R torsion}
\label{Rtor}

If 
\[
\varphi:R\to R',
\]
is a ring homomorphism, we may form the complex of free left $R'$-modules (using the homomorphism $\varphi$ to make $R'$ into a right $R$-module)
\[
\CS_{\varphi,q}=R' \otimes_\varphi \CS_q, 
\]
that has a preferred graded basis $1\otimes \cs_\bu$, induced by the preferred graded basis $\cs_\bu$ of $\CS_\bu$.  
Assuming that the homology modules $H_q(\CS_{\varphi,\bu})$ are free with preferred graded bases 
$\hs_\bu=\{\hs_q\}$, then the {\it Whitehead torsion of the twisted  complex} $\CS_{\varphi,\bu}$ (also called {\it Whitehead torsion with respect to the representation $\rho$) with respect to the graded bases $\cs_\bu$ and $\hs_\bu$} is well defined  in  $K_{\varphi(U)}(R')$, and we 
have the formula
\[
\tau_{\rm W}(\CS_{\varphi,\bu};\cs_\bu,\hs_\bu)=\sum_{q=0}^n (-1)^q [\varphi(\b_{q+1}(\bs_{q+1})\z_q\bs_q/\cs_q)]
=\varphi^* (\tau_{\rm W}(\CS_{\bu};\cs_\bu,\hs_\bu)).
\]

If $\rho:\pi\to \Aut_{R'}(M)$ is  a representation of $\pi$ in the group of the automorphisms of some free right module $M$ over  $R'$ (more precisely, using the sequence of ring homomorphisms
\[
\xymatrix{ \Z\pi\ar[r]^\rho&\Z \Aut_{R'}(M)\ar[r]& M(R',\rk(M)),}
\]
where the second homomorphism is the natural one), we may form the twisted complex 
\[
\CS_{\rho,q}=M \otimes_\rho \CS_q, 
\]
that however is only a complex of $\Z$-modules, unless $M$ is also a left $R''$-modules, in which case we obtain a complex of free finitely generated $R''$-modules. 
In particular, this is the case when $M$ is some vector space.

Fixing a basis $\ms$ for $M$, bases for these modules (and for cycles and boundary submodules) are given by tensoring with $\ms$, and $\ms$ will be omitted from the notation. All the argument above works and we can talk about based twisted chain complex.  Assuming that the homology modules $H_q(\CS_{\rho,\bu})$ are free with preferred graded bases $\hs_\bu=\{\hs_q\}$, then the {\it Whitehead torsion of the  twisted complex $\CS_{\rho,\bu}$ with respect to the graded bases $\cs_\bu$ and $\hs_\bu$} is a class in  $K_{{\pm }\rho(\pi)}(M(R',\rk(M)))=  K_{\pm 1}(M(R',\rk(M)))/\rho(\pi)=  K_{\pm 1}(R')/\rho(\pi)$. We 
have the formula
\[
\tau_{\rm W}(\CS_{\rho,\bu};\cs_\bu,\hs_\bu))=\tau_{\rm W}(M\otimes_\rho \CS_{\bu};\cs_\bu,\hs_\bu))=\sum_{q=0}^n (-1)^q [\rho(\b_{q+1}(\bs_{q+1})\hat\hs_q\bs_q/\cs_q)].
\]

\begin{rem}\label{r2.2} Note that if we consider the trivial representation $\rho_0:\pi\to 1$, then $\CS_{\rho_0,\bu}=\CS_\bu$ with coefficients in $\Z$. In fact, the induced ring homomorphism is 
\[
\rho_0:\Z\pi \to \Z 1.
\]

\end{rem}

If in particular the ring $R'=A$ is abelian, then $M$ is a left and a right $A$-module, and 
\[
\rho:\pi\to \Aut_{A}(V)=Gl(A,\rk(M)).
\]

In this case $\CS_{\rho,\bu}$ is a complex of $A$-modules
\[
\CS_{\rho,q}=M \otimes_\rho \CS_q,
\]
and the isomorphism $K_{\pm 1}(M(A,\rk(M)))/\rho(\pi)=  K_{\pm 1}(A)/\rho(\pi)=A^\times/\{\pm 1\}\rho(\pi)$ is the absolute value of determinant function
\[
|\det |:K_{\pm 1}(M(A,\rk(M)))/\rho(\pi)\to  A^\times/\{\pm 1\}\rho(\pi).
\]

An even more particular case is  when $M=V$ is some $k$-vector space  over a field $\F$ of characteristic zero, and $\rho:\pi\to \Aut_{\F}(V)=Gl(\F,k)$. In this case $\CS_{\rho,\bu}$ is a complex of vectors spaces 
\[
\CS_{\rho,q}=V \otimes_\rho \CS_q, 
\]
and the torsion is a class in $K_{\pm 1}(\F)/\rho(\pi)=\F^\times/\{\pm 1\}\rho(\pi)$. The isomorphism is induced by the determinant map $\det: Gl(\F)\to \F^\times$, and we denote by $|x|$ the class of $x\in \F^\times$.  If $\F=\R$, we consider (real) orthogonal representations of $\pi$. In such a case, $K_{\pm 1}(\F)/\rho(\pi)=\R^+$, the positive real numbers, the homology $H_q(\CS_{\rho,\bu})$ is free with some basis $\hs_q$, and the Whitehead class is just the module of the determinant of the matrix,  product notation is more usual, torsion is called  {\it R torsion of the complex $\CS_{\rho,q}$ with respect to the graded bases $\cs_\bu$ and $\hs_\bu$},
\beq\label{Rtor1}
\begin{aligned}
\tau_{\rm R}(\CS_{\rho,\bu};\cs_\bu,\hs_\bu))&=\tau_{\rm W}(V\otimes_\rho \CS_{\bu};\cs_\bu,\hs_\bu))\\
&=\sum_{q=0}^n (-1)^q [\rho(\b_{q+1}(\bs_{q+1})\hat\hs_q\bs_q/\cs_q)]\\
&=\prod_{q=0}^n (-1)^q |\det(\b_{q+1}(\bs_{q+1})\hat\hs_q\bs_q/\cs_q)|,
\end{aligned}
\eeq
and it appears as the change of basis of the homology determinant line. 

\subsection{Torsion with respect to a representation and $R$ torsion for cellular complexes}
\label{CW} 

Let $(K,L)$ be a cellular pair and $\tau_{\rm W}(K,L;\hs_\bu)=\tau_{\rm W}(\CS_\bu(  K,  L; \Z\pi);\cs_\bu,\hs_\bu)$ its Whitehead torsion as defined in Section \ref{sub3.3}. Let $\pi=\pi_1(K)$, and $\rho:\pi\to Aut_{R'}(M)$ be a representation of the fundamental group. Then, assuming that the homology groups $H_q(M\otimes_\rho \CS_q(  K,  L; \Z\pi)$ are free with basis $\hs_q$, proceeding as in Section \ref{Rtor}, we may define the {\it Whitehead torsion of the cellular pair $(K,L)$ with respect to representation $\rho$ of the fundamental group, and the graded homology basis $\hs_\bu$}:

\[
\tau_{\rm W}(K,L;M_\rho,\hs_\bu)=\tau_{\rm W}(\CS_{\rho,\bu}(  K, L;\Z\pi);\cs_\bu,\hs_\bu)
=\tau_{\rm W}(M\otimes_\rho \CS_{\bu}(  K,  L;\Z\pi);\cs_\bu,\hs_\bu).
\]
 
In particular, if  $\rho:\pi\to O(V)$  is a real orthogonal representation of the fundamental group  on a real vector space $V$, then the groups 
\[
H_q(V\otimes _\rho \CS_q(K,L;\Z\pi)),
\]
are free, and we denote by $\hs_q$ a basis. Then, we define the {\it R torsion of the cellular pair $(K,L)$ with respect to the representation $\rho$, and the graded homology basis $\hs_\bu$} to be  the torsion of the twisted complex
\[
\CS_{\rho,\bu}( K, L;\Z\pi)=V\otimes_{\rho} \CS_\bu(  K,  L;\Z\pi),
\]
that is sometimes denoted also by $\CS_\bu(K,L;V_\rho)$, and we use the notation 
\[
\tau_{\rm R}(K,L;V_\rho,\hs_\bu)=\tau_{\rm R}(\CS_{\rho,\bu}( K, L;\Z\pi);\cs_\bu,\hs_\bu)
=\tau_{\rm R}(V\otimes_\rho\CS_{\bu}(K, L;\Z\pi);\cs_\bu,\hs_\bu),
\]
where recall $\cs_\bu$ is a lift of the cell basis, and $\hs_\bu$ the homology basis of the twisted complex.

\begin{rem} In the case of the trivial representation, we obtain the complex (where $V$ has been identified with $\R^{\dim V}$ in the last formula)
\[
\CS_{\rho_0,\bu}( K,L;\Z\pi)=V\otimes_{\rho_0} \CS_\bu( K, L;\Z\pi)=V\otimes \CS_\bu( K, L;\R).
\]
\end{rem}

\begin{prop}  Let $(K,L)$ be a cellular pair of dimension $n$, and   $\rho_0:\pi\to O(V)$  the trivial  real orthogonal representation of the fundamental group  on a real vector space $V$. Let $\hs_\bu$ a graded basis for $H_\bu(V\otimes \CS_\bu(K,L;\R))$, and  $\ns_\bu$ be the standard integral graded basis (see Appendix \ref{standardbasis}). Then,
\[
\tau_{\rm R} (K,L; V_{\rho_0},\hs_\bu)=\prod_{q=0}^n |\det(\hs_q/\ns_q))|^{(-1)^q} \prod_{q=0}^n \# TH_q(K,L)^{(-1)^q}.
\]
\end{prop}


\subsection{R torsion for manifolds}
\label{torman}

Let $W$ be an orientable  compact connected Reimannian manifold of finite dimension $m$ without boundary and with metric $g$. Let $\rho:\pi_1(W)\to O(k,\R)$ be a representation of the fundamental group of $W$, and let  $E_\rho$ be the associated vector bundle over $W$ with fibre $\R^k$ (that we assume with a fixed basis) and group $O(k,\R)$, $E_\rho=\R^k \times_\rho \widetilde W$.  Let $L$  denote either a simplicial triangulation of $W$ or a CW decomposition. The homology groups (here there is a slight abuse of notation, since the bundle appears where its fibre should)
\[
H_q(\R^k\otimes_\rho \CS_q(L;\Z\pi)=H_q(W; E_\rho),
\]
are free, and we denote by $\hs_q$ a basis. Then, we can define the R torsion of $L$ as above and we define the {\it R torsion of $W$ with respect to the representation $\rho$, and the graded homology basis $\hs_\bu$} to be  the torsion of the twisted complex
\[
\CS_{\rho,\bu}( L;\Z\pi)=\R^k\otimes_{\rho} \CS_\bu(  L;\Z\pi),
\]
and we use the notation
\[
\tau_{\rm R}(W ;E_\rho,\hs_\bu)
=\tau_{\rm R}(\R^k\otimes_\rho\CS_{\bu}(L;\Z\pi);\cs_\bu,\hs_\bu).
\]

This definition is well posed since, given any other triangulation $L_1$ of $W$ it is well known that there exists a common subdivision $L'$ of $L$ and $L_1$, and therefore by Lemma the torsion computed with respect to $L$ coincides with the one computed with respect to $L_1$. 

If $Y$ is an orientable  compact connected Reimannian manifold of finite dimension $n$ with boundary $\b Y=W$, and with metric $g$, then the construction is exactly the same and we define the {\it relative R torsion of $Y$, or R torsion of the pair $(Y,W)$ with respect to the representation $\rho$, and the graded homology basis $\hs_\bu$} to be  the torsion of the twisted complex
\[
\CS_{\rho,\bu}( K,L;\Z\pi)=\R^k\otimes_{\rho} \CS_\bu(  K,L;\Z\pi),
\]
where $L$ is either a triangulation or a CW decomposition of $W$, and we use the notation
\[
\tau_{\rm R}(Y, W ;E_\rho,\hs_\bu)
=\tau_{\rm R}(\R^k\otimes_\rho\CS_{\bu}(K,L;\Z\pi);\cs_\bu,\hs_\bu).
\]

\subsection{RS torsion for a manifold}
\label{secRS}

In case of a manifold, it is possible to construct a particular basis for homology as follows \cite{RS}. Let $W$ a manifold as in the previous Section \ref{torman}. Let $\Omega(W,E_\rho)$ denote the graded linear space of smooth forms on $W$ with values in $E_\rho$.  The exterior differential on $W$ defines the exterior differential on $\Omega^q(W, E_\rho)$, 
$d:\Omega^q(W, E_\rho)\to\Omega^{q+1}(W, E_\rho)$. The metric $g$ defines an Hodge operator on $W$ and hence on 
$\Omega^q(W, E_\rho)$, $\star:\Omega^q(W, E_\rho)\to\Omega^{m-q}(W, E_\rho)$, and,  using the inner product 
$(\_, \_)$ in $\R^k$, an inner product on $\Omega^q(W, E_\rho)$ is defined by
\[
\langle\omega,\eta\rangle=\int_W ( \omega\wedge\star\eta ).
\]

The adjoint $d^\dagger$ and the Laplacian $\Delta=(d+d^\dagger)^2$ operator 
are defined on the space of sections with values in $E_\rho$, we obtain the twisted de Rham complex, and the Hodge decomposition holds. In particular we denote by

\[
\Ha^q(W,E_\rho)=\{\omega\in\Omega^q(W,E_\rho)~|~\Delta^{(q)}\omega=0\},
\]
the spaces of the harmonic forms.

In this setting, we have the  the  de Rham map $\A^q$ (that induce isomorphisms in cohomology),
\begin{align*}
\A^q:&\Ha^q(W,E_\rho)\to \CS^q(W;E_\rho),&
\end{align*}
with
\[
\A^q(\omega)(v\otimes_\rho c)=\int_{ c} (\omega,v),
\]
where $v\otimes_\rho c$ belongs to $\CS_q(W;E_\rho)$
, and $c$ is identified with the $q$-subcomplex (simplicial or cellular) that $c$ represents. 
Following Ray and Singer \cite{RS}, we introduce the de Rham map $\A_q$:
\begin{align*}
\A_q:&\Ha^q(W,E_\rho)\to C_q(W;E_\rho),\\
\A_q:&\omega\mapsto (-1)^{(m-1)q}\P_q^{-1}\A^{m-q}\star(\omega),
\end{align*}
defined by
\beq\label{aa}
\A_q(\omega)=(-1)^{(m-1)q}\sum_{j,i} \left(\int_{\hat c_{q,j}}(\star\omega,e_i)\right) c_{q,j}\otimes_\rho e_i,
\eeq
where the sum runs over all $q$-simplices $c_{q,j}$ of $L$,  $\P_q:C_q(L;\Z)\to C^{m-q}(\check L;\Z)$ is the Poincar\'e map, and $\check c$ denotes the dual block cell of $c$, see Sections \ref{asss} and \ref{Poincare}. 

The homology groups 
\[
H_q(\R^k\otimes_\rho \CS_q(L;\Z\pi))=H_q(\CS_\bu(L;E_\rho))=H_q(W; E_\rho),
\]
are free and have a natural basis coming from an orthonormal basis of the spaces of harmonic forms. Namely, let $\as_q$ be an orthonormal basis of $\Ha^q(W,E_\rho)$; then,  $\A_q(\as_q)$ is a basis of $H_q(W,E_\rho)$. Then, following Ray and Singer \cite{RS}, we  define the {\it RS torsion of $W$ with respect to the representation $\rho$} to be  the torsion of the twisted complex
\[
\CS_{\bu}( L;E_\rho)=\R^k\otimes_{\rho} \CS_\bu(  L;\Z\pi),
\]
with respect to the graded homology basis $\A_\bu(\as_\bu)$, and we use the notation
\[
\tau_{\rm R}(W ;E_\rho)
=\tau_{\rm R}(\R^k\otimes_\rho\CS_{\bu}(L;\Z\pi);\cs_\bu,\A_\bu(\as_\bu)).
\]

This torsion is well defined since as R torsion it does not depends on the choice of the triangulation.

For a manifold $Y$ with boundary $\b Y=W$, as in the previous section, recall that near the boundary there is a natural splitting of $\Lambda Y$ as direct sum of vector bundles $\Lambda T^*\b Y\oplus \Na^* Y\otimes \Lambda T^*\b Y$, where $\Na^*Y$ is the dual to the normal bundle to the boundary, and the smooth forms on $Y$ near the boundary decompose as $\omega=\omega_{\rm tan}+\omega_{\rm norm}$, where $\omega_{\rm norm}$ is the orthogonal projection on the subspace generated by $dx$, the one form corresponding to the outward pointing unit normal vector to the boundary, and $\omega_{\rm tan}$ is in $C^\infty(Y)\otimes\Lambda(\b Y)$. We  write $\omega=\omega_1+ dx \wedge\omega_{2}$, where $\omega_j\in C^\infty( Y, \Lambda(T^*\b Y))$, and
\[
\star\omega_2=dx \wedge \star\omega.
\]

Define absolute and relative boundary conditions by
\[
B_{\rm abs}(\omega)=\omega_{\rm norm}|_{\b Y}=\omega_2|_{\b Y}=0,\qquad 
B_{\rm rel}(\omega)=\omega_{\rm tan}|_{\b Y}=\omega_1|_{\b Y}=0.
\]

Let $\B(\omega)=B(\omega)\oplus B((d+d^\dagger)(\omega))$. The adjoint $d^\dagger$ and the Laplacian $\Delta=(d+d^\dagger)^2$ operators are defined on the space of sections with values in $E_\rho$,  the Laplacian with boundary conditions $\B(\omega)=0$  is self-adjoint, and the spaces  of the harmonic forms with boundary conditions are
\begin{align*}
\Ha_{\rm abs}^q(Y,E_\rho)&=\{\omega\in\Omega^q(Y,E_\rho)~|~d\omega=d^\da \omega=0, \B_{\rm abs}(\omega)=0\},\\
\Ha_{\rm rel}^q(Y,E_\rho)&=\{\omega\in\Omega^q(Y,E_\rho)~|~d\omega=d^\da \omega=0, \B_{\rm rel}(\omega)=0\}.
\end{align*}

Let $K$ be a regular cellular or simplicial decomposition of $Y$ and $L$ of $\b Y$. Let $C_q(K;E_\rho)=\R^k\otimes_\rho C_q( K;\Z\pi_1(Y))$ be complex of the twisted chains, as above. Then we have the following de Rham maps $\A^q$ (that induce isomorphisms in cohomology),
\begin{align*}
\A_{\rm abs}^q:&\Ha^q_{\rm abs}(Y,E_\rho)\to C^q(K;E_\rho),&
\A_{\rm rel}^q:&\Ha_{\rm rel}^q(Y ,E_\rho)\to C^q(K,L;E_\rho),
\end{align*}
with
\[
\A_{\rm abs}^q(\omega)(v\otimes_\rho c)=\A^q_{\rm rel}(\omega)(v\otimes_\rho c)=\int_{ c} (\omega,v),
\]
where $v\otimes_\rho c$ belongs to $C_q(K;E_\rho)$
, and $c$ is identified with the $q$-subcomplex (simplicial or cellular) that $c$ represents. 
Following Ray and Singer \cite{RS}, we introduce the de Rham maps $\A_q$:
\begin{align*}
\A^{\rm rel}_q:&\Ha_{\rm rel}^q(Y,E_\rho)\to C_q( K,L;E_\rho),&
\A^{\rm rel}_q:&\omega\mapsto (-1)^{(m-1)q}\P_q^{-1}\check\A_{\rm abs}^{m-q}\star(\omega),\\
\A^{\rm abs}_q:&\Ha_{\rm abs}^q(Y,E_\rho)\to C_q( K;E_\rho),&
\A^{\rm abs}_q:&\omega\mapsto (-1)^{(m-1)q}\P_q^{-1}\check\A_{\rm rel}^{m-q}\star(\omega),
\end{align*}
both defined by
\beq\label{aa1}
\A^{\rm rel}_q(\omega)=\A^{\rm abs}_q(\omega)=(-1)^{(m-1)q}\sum_{j,i} \left(\int_{\check c_{q,j}}(\star\omega,e_i)\right)
c_{q,j}\otimes_\rho e_i,
\eeq
where the sum runs over all $q$-simplices $c_{q,j}$ of $\check K-\tilde{\check L}$ in the first case, but runs over all $q$-simplices $c_{q,j}$ of $\check K\sqcup \tilde{\check L}$ in the second case (see Section \ref{Poincare} for details on the construction of the dual block complex). Here $\P_q:C_q(K,L;\Z)\to C^{m-q}(\check K-\check L;\Z)$ is the Poincar\'e map, and $\hat c$ denotes the dual block cell of $c$.

\section{Some elementary constructions with  chain complexes and their torsion}

\subsection{The cone of a chain complex}
\label{app-cone}

Let $\CS_\bu$ be  a finite chain complex of free finitely generated left $R$-modules 

\[
\xymatrix{
\CS_\bu:& \CS_m\ar[r]^{\b_m}&\CS_{m-1}\ar[r]^{\b_{m-1}}&\dots\ar[r]^{\b_2}&\CS_1\ar[r]^{\b_1}&\CS_0,
}
\]

We assume that $\CS_\bu$ has a graded preferred basis, and we denote it by $\cs_q=\{c_{q,k}\}$.  The {\it cone of} $\CS_\bu$ is the algebraic mapping cone of the chain identity of the augmented complex $\CS_\bu$, i.e.  the chain complex $C(\CS_\bu)_\bu$ of length $m+1$ with\footnote{It is important to observe that this sum is not multilinear, namely: $(x+y)\oplus z\not= x\oplus z+y\oplus z$. This  works only if one of the components is $0$.}
\[
 \dot\CS_q=C(\CS_\bu)_q=\left\{\begin{array}{cl}\CS_{q-1}\oplus \CS_{q},&q>0,\\   R[v]\oplus\CS_0,&q=0.\end{array}\right.
\]
and boundary operator 
\[
\dot\b_q=\left\{\begin{array}{cl}\left(\begin{matrix}\b_{q-1}&0\\1&-\b_{q}\end{matrix}\right),&q>1,\\
\left(\begin{matrix}\epsilon&0\\1&-\b_1\end{matrix}\right),&q=1,\\ 0,&q=0,\end{array}\right.
\]
where $\epsilon:\CS_0\to R[v]$ is the augmentation; in particular:
\beq\label{es1}
\xymatrix{
C(\CS_\bu)_\bu: &\dots\ar[r]& \dot\CS_1=\CS_0\oplus \CS_1\ar[r]^{\hspace{-25pt}\dot \b_1}&\dot\CS_0=R[v\oplus 0,0\oplus c_{0,k}]\ar[r]^{\hspace{45pt}\dot\b_0=0}&0.
}
\eeq





For $q>0$,
\[
\dot H_q=\frac{\dot Z_q}{\Im \dot \b_{q+1}}=\frac{\left\{\dot c=\b_q(y)\oplus y ~|~y\in  \CS_q\right\}}{\dot Z_q- (0\oplus \b_q(\CS_{q+1}))}=0,
\]
while
\[
\dot H_0=\frac{R[v\oplus 0,0\oplus c_{0,k}]}{R[v\oplus c_{0,k}]}=R[[v\oplus 0]].
\]

The chain inclusion 
\begin{align*}
j_q:&\CS_q\to C(\CS_q),\\
j_q:&c_q\mapsto 0\oplus c_q,
\end{align*}
induces the exact sequence
\begin{align}
\label{sC1}&\xymatrix{0\ar[r]& \CS_\bu\ar[r]_{j_\bu}&C(\CS_\bu)\ar[r]_{p_\bu}&(C(\CS_\bu),\CS_\bu)\ar[r]&0,}
\end{align}
where the relative complex is $\CS_q''=(C(\CS_\bu),\CS_\bu)_q=\dot \CS_q/\CS_q=\CS_{q-1}$, with boundary $\b_{q-1}$:
\beq\label{relrel}
\xymatrix{
(C(\CS_\bu),\CS_\bu)_\bu:& \CS_m''=\CS_{m-1}\ar[r]^{\b_m''=\b_{m-1}}&\dots\ar[r]^{\b''_2=\b_1}&\CS_1''=\CS_0\ar[r]^{\b''_1=\ep}&R[v]\ar[r]&0,
}
\eeq

Bases for the chain modules are $\cs''_q=\cs_{q-1}$ for $q>0$, and $c_{0,0}=\{v\}$.  


The exact sequence \eqref{sC1} induces an homology long exact sequence: for $q>1$
\[
\xymatrix@C=0.45cm{
\dots\ar[r]& H_{q}(C(\CS_\bu))=0\ar[r]^{}&H_{q}((C(\CS_\bu),\CS_\bu))\ar[r]^{p_{*,q}}&H_{q-1}(\CS_\bu)\ar[r]&H_{q-1}(C(\CS))=0,
}
\]
where the boundary homomorphism  is in fact the identity. Therefore, $H_q(C(\CS_\bu),\CS_\bu)=H_q(\CS_\bu)$, for $q>1$, and it is trivial otherwise. A basis for homology is $ \hs_q''=\hs_{q-1}$, for $q>1$.

\subsubsection{Torsion and relative torsion}

Since $\CS_\bu$ is a based chain complex, so is  $C(\CS_\bu)$  with preferred basis $\dot \cs_q=\{c_{q-1,j}\oplus 0,0\oplus c_{q,k}\}$,  for $q>0$, and $\dot \cs_0=\{v\oplus 0,0\oplus c_{0,j}\}$. The unique non trivial homology group of the cone is in dimension zero, and it is free, therefore the Whitehead torsions of $C(\CS_\bu)$ is  defined, and is related to the torsion of $\CS_\bu$.  We denote by $\dot \hs_0$ a fixed basis for $\dot H_0$. We have $\dot \hs_0= \{\dot h_{0,0}\}$, where $\dot h_{0,0}=\alpha [v\oplus 0]$, for some unit $\alpha\in R^\times$. 
The integral basis is $\dot n_{0,0}=[v\oplus 0]$.

 Assume $\dot \bs_q$ is a fixed set of elements of $\dot \CS_q$ with $\dot\b_q(\dot \bs_q)\not=0$, and l.i.. Since $H_q(C(\CS_\bu))=0$ for $q>0$
it follows that a basis for $\dot \CS_q$ is $\dot \b_{q+1}(\dot \bs_{q+1}),\dot \bs_q$, for $q>0$. A basis for $\dot \CS_0$ is  $\dot \b_{1}(\dot \bs_{1})\hat{\dot\hs}_0$, since $\dot \bs_0=\emptyset$. Applying the definition
\[
\tau_{\rm W}(C(\CS_\bu);\dot \cs_\bu,\dot \hs_\bu)=[(\dot\b_{1}(\dot \bs_{1})\hat{\dot\hs}_0/\dot \cs_0)]+\sum_{q=1}^{m+1} (-1)^q [(\dot\b_{q+1}(\dot \bs_{q+1})\dot \bs_q/\dot \cs_q)].
\]

We want to write the torsion of the cone as a function of the torsion of $\CS_\bu$.  For $q>1$,  if $c=x\oplus y\in \dot \CS_q=\CS_{q-1}\oplus \CS_{q}$, 
\[
\dot\b_q=\left(\begin{matrix}\b_{q-1}&0\\1&-\b_{q}\end{matrix}\right)\left(\begin{matrix}x\\y\end{matrix}\right)=\left(\begin{matrix} \b_{q-1}( x)\\ x-\b_{q}( y)\end{matrix}\right).
\]

Thus, when $q>1$,  
\begin{align*}
\dot\b_q(0\oplus \bs_q)&=0\oplus -\b_q(\bs_q)\not=0,&\dot\b_q(\cs_{q-1}\oplus 0)&=\b_{q-1}(\cs_{q-1})\oplus \cs_{q-1}\not=0,
\end{align*}
and this happens only for these elements (the notation $\langle\vs, \vs'\rangle$ for two sets of linearly independent vectors means the subspace direct sum of the two subspaces generated by these sets of vectors). However, these elements do not have linearly independent images: for if $\b_q(\bs_q)=\cs_{q-1}$ ($0\not=y=\b_q x$), then  
\begin{align*}
\dot\b_q(0\oplus \bs_q+\cs_{q-1}\oplus 0)&=\b_{q-1}(\cs_{q-1})\oplus \cs_{q-1}-\b_q(\bs_q)=0,
\end{align*}
the image will  vanish. Since $\cs_{q-1}=\b_q(\bs_q) \hat\hs_{q-1}\bs_{q-1}$, this problem  can be avoided taking only the elements in $\cs_{q-1}$ coming from $\hat\hs_{q-1}\bs_{q-1}$. Therefore, a set of  elements in $\dot \CS_q$ with non trivial l.i image is 
\[
\dot \bs_q=\hat\hs_{q-1}\bs_{q-1}\oplus 0\ss0\oplus \bs_q,
\]
and
\[
\dot\b_q(\dot \bs_q)= 0\oplus \hat\hs_{q-1}\ss\b_{q-1}(\bs_{q-1})\oplus \bs_{q-1}\ss 0\oplus \b_q(\bs_q),
\]
and hence the new basis for $\dot \CS_q$, $q>1$,  is (the equivalence follows since the matrix of the change of basis is block triangular).
\begin{align*}
\dot\b_{q+1}(\dot \bs_{q+1})\dot\bs_q&= 0\oplus \hat\hs_{q}\ss\b_{q}(\bs_{q})\oplus \bs_q\ss 0\oplus \b_{q+1}(\bs_{q+1})\ss   \hat\hs_{q-1}\bs_{q-1}\oplus 0\ss0\oplus \bs_q\\
&\cong 0\oplus\b_{q+1}(\bs_{q+1}) \hat\hs_{q}\bs_{q}\ss  \b_{q}(\bs_{q})\hat\hs_{q-1}\bs_{q-1}\oplus 0.
\end{align*}

At $q=1,0$, recall the right end of the sequence displayed in equation \eqref{es1} where
\[
\dot \b_1(x\oplus y)=\left(\begin{matrix}\epsilon&0\\1&-\b_{1}\end{matrix}\right)\left(\begin{matrix}x\\y\end{matrix}\right)=\left(\begin{matrix}\epsilon (x)\\x-\b_{1} (y)\end{matrix}\right)=\epsilon (x)\oplus x-\b_{1} (y),
\]
we see that $\dot \bs_1$ is as before, and $\bs_0=\al v\oplus 0$. So the new basis at $q=1$ is
\[
\dot\b_{2}(\dot \bs_{2})\dot \bs_1=0\oplus\b_{2}(\bs_{2}) \hat\hs_{1}\bs_{1}\ss  \b_{1}(\bs_{1})\hat\hs_{0}\bs_{0}\oplus 0.
\]

A short calculation gives the boundary:
\[
\dot\b_1(\dot \bs_1)=0\oplus \b_1(\bs_1) \hat\hs_{0}\ss\ep(\bs_{0})\oplus  \bs_{0},
\]
since $\dot\bs_0=\emptyset$ and $\hat{\dot \hs}_0= \alpha v\oplus 0$, the new basis at $q=0$ is
\begin{align*}
\dot\b_{1}(\dot \bs_{1})\hat{\dot\hs}_0\dot\bs_0
&=0\oplus \b_{1}(\bs_{1}) \hat\hs_{0}\ss\tilde\b_{0}(\tilde\bs_{0})\oplus \tilde \bs_{0}\ss \alpha v\oplus 0
\cong 0\oplus \b_{1}(\bs_{1}) \hat\hs_{0}\tilde\bs_{0} \ss  \alpha v\oplus 0.
\end{align*}

Therefore, for $q>1$
\begin{align*}
[(\dot\b_{q+1}(\dot \bs_{q+1})\dot \bs_q/\dot \cs_q)]=&[(0\oplus\b_{q+1}(\bs_{q+1}) \hat\hs_{q}\bs_{q}\ss  \b_{q}(\bs_{q})\hat\hs_{q-1}\bs_{q-1}\oplus 0/\cs_{q-1}\oplus 0\ss0\oplus\cs_{q})]\\
=&[(\b_{q}(\bs_{q})\hat\hs_{q-1}\bs_{q-1}/\cs_{q-1})]+[(\b_{q+1}(\bs_{q+1}) \hat\hs_{q}\bs_{q} /\cs_q)],
\end{align*}
and
\begin{align*}
[(\dot\b_{2}(\dot \bs_{2})\dot \bs_1/\dot \cs_1)]
&=[(\b_{1}(\bs_{1})\hat\hs_{0}\bs_{0}/\cs_{0})]+[(\b_{2}(\bs_{2}) \hat\hs_{1}\bs_{1} /\cs_1)],\\
[(\dot\b_{1}(\dot \bs_{1})\dot \bs_0/\dot \cs_0)]
&=[(\alpha v/v)]+[(\b_{1}(\bs_{1}) \hat\hs_{0}\bs_{0} /\cs_0)].
\end{align*}
 
This gives:
\begin{align*}
\tau_{\rm W}(C(\CS_\bu);\dot \cs_\bu,\dot \hs_\bu)=&[(\alpha v/v)]+[(\b_{1}(\bs_{1}) \hat\hs_{0}\bs_{0} /\cs_0)]
\\
&+\sum_{q=1}^{m+1} (-1)^q ([(\b_{q}(\bs_{q})\hat\hs_{q-1}\bs_{q-1}/\cs_{q-1})]+[(\b_{q+1}(\bs_{q+1}) \hat\hs_{q}\bs_{q} /\cs_q)])\\
=&[(\alpha v/v)]=\alpha.
\end{align*}

Note that if $R=\Z$, then $\alpha=1$.

For relative torsion, it is easy to see that in all degrees $q>1$, the relevant change of basis is
\[
(\b_{q+1}''(\bs_{q+1}'')\hat \hs_q''\bs_q''/\cs_q'')=(\b_q(\bs_q)\hat \hs_{q-1}\bs_{q-1}/\cs_{q-1}).
\]

It remains to consider the right end of the sequence as displayed in equation (\ref{relrel}). Recalling that homology is trivial, 
let $c_{0,0}$ be a fixed $0$ cell representing the zero homology of $\CS_\bu$, i.e. $\hat h_{0,0}=\beta c_{0,0}$, $\beta\in R^\times$. Then, we may choose $\bs_0''=\cs_0-\{c_{0,0}\}$. Whence
\[
(\b_{2}''(\bs_{2}'')\hat \hs_1''\bs_1''/\cs_1'')=(\b_1(\bs_1)\cs_0-\{c_{0,0}\}/\cs_{0}).
\]

Since the cells in the boundary of each $\b_1(\bs_1)$ are the same $R$ multiple of $0$ cells, it follows that 
\[
[(\b_1(\bs_1)\cs_0-\{c_{0,0}\}/\cs_0)]=[(\b_1(\bs_1)\hat \hs_0/\cs_0)].
\]

In degree $0$ the change of basis is trivial. So
\begin{align*}
\tau_{\rm W}((C(\CS_\bu),\CS_\bu); \cs''_\bu, \hs''_\bu)&=\sum_{q=0}^{m} (-1)^{q+1} [(\b_{q+1}(\bs_{q+1}) \hat\hs_{q}\bs_{q} /\cs_q)]
=-\tau_{\rm W}(\CS_\bu; \cs_\bu, \hs_\bu).
\end{align*}

Note that this follows as well applying Theorem \ref{mil0}, since the torsion of the homology sequence in this case is precisely $[(\alpha v/v)]=\alpha$.

\subsection{The algebraic mapping cone}
\label{B4}

Let $i_\bu:\CS_\bu\to \DS_\bu$ an inclusion of chain complexes, and consider its algebraic mapping cone
\[
\xymatrix{
\CS_\bu\ar[r]^{i_\bu}\ar[d]_{j_\bu}&\DS_\bu\ar[d]^{\bar j_\bu}\\
C(\CS_\bu)\ar[r]_{\bar i_\bu} &\ddot\CS_\bu=C(\CS_\bu)\sqcup_{i_\bu} \DS_\bu
}
\]


By definition,  $\ddot\CS_\bu$ is the following complex
\begin{align*}
\ddot\CS_q&=\CS_{q-1}\oplus \DS_q, \\
\ddot\b_q&=\b_{q-1}^{\CS_\bu}\oplus (i_{q-1}-\b_q^{\DS_\bu})=\left(\begin{matrix}\b^{\CS_\bu}_{q-1}&0\\i_{q-1}&-\b^{\DS_\bu}_{q}\end{matrix}\right).
\end{align*}

Note that this construction is functiorial: if $\varphi:(\DS_\bu, \CS_\bu)\to (\DS_\bu',\CS_\bu')$ is a chain map of pairs, then it induces a chain map 
\[
\ddot\varphi:\ddot\CS_\bu=C(\CS_\bu)\sqcup_{i_\bu} \DS_\bu\to \ddot\CS_\bu'=C(\CS'_\bu)\sqcup_{i'_\bu} \DS'_\bu,
\]
coinciding with $\varphi$ on $\DS_\bu$ and with its restriction on $\CS_\bu$. Existence follows by definition, commutativity with the boundary is easily verified. 

Note also that we have the inclusion $\bar i_\bu:C(\CS_\bu)\to \ddot \CS_\bu$, and it is easy to see that the quotient complex is
\[
(\ddot \CS_\bu ,C(\CS_\bu))_q=\ddot \CS_q /C(\CS_\bu)_q
=\frac{\CS_{q-1}\oplus \DS_q}{\CS_{q-1}\oplus \CS_q}=\DS_q/\CS_q=(\DS_\bu,\CS_\bu)_q.
\]
i.e. that the inclusion $\bar j_\bu:\DS_\bu\to \ddot\CS_\bu$ induces a chain  isomorphism on classes
\beq\label{eeee}
\tilde j_\bu:(\DS_\bu,\CS_\bu)\to (\ddot \CS_\bu ,C(\CS_\bu)).
\eeq

We have the  exact sequences
\[
\xymatrix{
0\ar[r]&\CS_\bu\ar[r]^{i_\bu}& \DS_\bu\ar[r]^{p_\bu}& \DS_\bu/\CS_\bu\ar[r]&0,
}
\]
and
\[
\xymatrix{
0\ar[r]&C(\CS_\bu)_\bu\ar[r]^{\bar i_\bu}&\ddot \CS_\bu\ar[r]^{\bar p_\bu}& \ddot \CS_\bu/C(\CS_\bu)_\bu\ar[r]&0,
}
\]
that induce the following commutative diagram of exact sequences
\[
\xymatrix{
&0\ar[d]&0\ar[d]&0\ar[d]&\\
0\ar[r]&\CS_\bu\ar[r]^{i_\bu}\ar[d]_{j_\bu}& \DS_\bu\ar[r]^{p_\bu}\ar[d]_{\bar j_\bu}& \DS_\bu/\CS_\bu\ar[d]_{\tilde j_\bu}\ar[r]&0\\
0\ar[r]&C(\CS_\bu)_\bu\ar[r]^{\bar i_\bu}\ar[d]_{q_\bu}&\ddot \CS_\bu\ar[r]^{\bar p_\bu}\ar[d]_{\bar q_\bu}& \ddot \CS_\bu/C(\CS_\bu)_\bu\ar[r]\ar[d]&0\\
0\ar[r]&C(\CS_\bu)_\bu/\CS_\bu\ar[r]^{\tilde i_\bu}\ar[d]&\ddot \CS_\bu/\DS_\bu\ar[r]\ar[d]& 0&\\
&0&0&&
}
\]

\subsubsection{Homology}

The homology of $\ddot\CS_\bu$ appears in several sequences.  
Consider the homology ladder induced by the map of pairs $(\bar j_\bu,j_\bu):(\DS_\bu,\CS_\bu)\to (\ddot \CS_\bu,C(\CS_\bu))$, coupled with the homology long exact sequence of the pairs $(C( \CS_\bu),\CS_\bu)$ and $(\ddot \CS_\bu,\DS_\bu)$:
\beq\label{ladder}
\xymatrix{
&\dots\ar[d]&\dots\ar[d]&&\\
0\ar[r]&H_{q+1}(C(\CS_\bu)/\CS_\bu)\ar[d]\ar[r]^{\tilde i_{*,q+1}}&H_{q+1}(\ddot\CS_\bu/\DS_\bu)\ar[r]\ar[d]_{\ddot\delta_{q+1}}&0\ar[d]\\
\dots\ar[r]&H_q(\CS_\bu)\ar[d]_{j_{*,q}}\ar[r]^{i_{*,q}}&H_q(\DS_\bu)\ar[r]^{p_{*,q}}\ar[d]_{\bar j_{*,q}}&H_q(\DS_\bu,\CS_\bu)\ar[d]_{\tilde j_{*,q}}\ar[r]&\dots\\
\dots\ar[r]&H_q(C(\CS_\bu))\ar[r]^{\bar i_{*,q}}\ar[d]_{q_{*,q}}&H_q(\ddot\CS_\bu)\ar[r]^{\bar p_{*,q}}\ar[d]_{\bar q_{*,q}}&H_q(\ddot\CS_\bu,C(\CS_\bu))\ar[d]\ar[r]&\dots\\
0\ar[r]&H_{q}(C(\CS_\bu)/\CS_\bu)\ar[d]\ar[r]^{\tilde i_{*,q}}&H_{q}(\ddot\CS_\bu/\DS_\bu)\ar[r]\ar[d]&0\\
&\dots&\dots&&
}
\eeq
where  $\bar p_{*,q}$ for all $q>1$, since the homology of the cone is trivial.

Using the chain isomorphism of pairs in equation \eqref{eeee}, we immediately obtain the isomorphism
\[
\tilde j_{*,\bu}:H_\bu(\DS_\bu,\CS_\bu)\to H_\bu(\ddot \CS_\bu ,C(\CS_\bu)).
\]

\begin{rem} If the chain complexes comes from topology, then this follows by excision.
\end{rem}

For $q>0$, composing with $p_{*,\bu}$, we have the isomorphism
\[
\bar p_{*,q}^{-1}\tilde j_{*,q}:H_q(\DS_\bu,\CS_\bu)\to H_q(\ddot \CS_\bu).
\]

For further use we also provide explicit computation of homology.  For $q>0$,  
if $c=x\oplus y\in \ddot \CS_q=\CS_{q-1}\oplus \DS_{q}$, 
\[
\ddot\b_q(c)=\left(\begin{matrix}\b_{q-1}&0\\1&-\b_{q}\end{matrix}\right)\left(\begin{matrix}x\\y\end{matrix}\right)=\left(\begin{matrix} \b_{q-1}( x)\\ x-\b^{\DS_\bu}_{q}( y)\end{matrix}\right).
\]

Thus, the kernel of $\ddot\b_q$ is generated by the elements of the form $\b^{\DS_\bu}_q(y)\oplus y$, with $y\in i_q(\CS_q)$, and $0\oplus \zs_q^{\DS_\bu}$, where $\zs_q^{\DS_\bu}\in \ker \b_q^{\DS_\bu}$. The image of $\ddot\b_{q+1}$ is
\[
\Im \ddot\b_{q+1}=\Im \b_q\oplus(\CS_q-\Im\b_{q+1}^{\DS_\bu})
=<\b_q(\bs_q)>\oplus <\cs_q \b^{\DS_\bu}_{q+1}(\bs^{\DS_\bu}_{q+1})>
\]
thus
\begin{align*}
H_q(\ddot\CS_\bu)&=\frac{\ker  \ddot\b_{q}}{\Im \ddot\b_{q+1}}=
\frac{\b^{\DS_\bu}_q(y)\oplus y~|~y\in i_q(\CS_q)>}{<\b_q(\bs_q)>}\oplus 
\frac{<\zs_q^{\DS_\bu}>}{<\cs_q \b^{\DS_\bu}_{q+1}(\bs^{\DS_\bu}_{q+1})>}
=\frac{<\zs_q^{\DS_\bu}>}{<\cs_q \b^{\DS_\bu}_{q+1}(\bs^{\DS_\bu}_{q+1})>}\\
&=H_q(\DS_\bu,\CS_\bu).
\end{align*}

\subsubsection{Some homology sequences and their torsion}
\label{ladder1}

We proceed to some identifications of homology sequences.

Recall that the inclusion of the base of the cone induces the exact sequence
\[
\xymatrix{0\ar[r]& \CS_\bu\ar[r]_{j_\bu}&\dot\CS_\bu\ar[r]_{q_\bu}&(\dot\CS_\bu,\CS_\bu)\ar[r]&0,}
\]
and the long exact sequence in homology 
\[
\xymatrix@C=0.45cm{
\dots\ar[r]& H_{q}(\dot\CS_\bu)=0\ar[r]^{}&H_{q}(\dot\CS_\bu,\CS_\bu)\ar[r]^{\dot\delta_{q}}&H_{q-1}(\CS_\bu)\ar[r]&H_{q-1}(\dot\CS_\bu)=0,\\
}
\]
where the boundary homomorphism $\dot\delta_{q}$ is in fact the identity. Using the diagram in figure \ref{ladder}, we have the following ladder
\[
\xymatrix{
\Ha:&\dots\ar[r]&H_q(\CS_\bu)\ar[r]^{i_{*,q}}\ar[d]_{\varphi_q}&H_q(\DS_\bu)\ar[r]^{p_{*,q}}\ar[d]_{id}&H_q(\DS_\bu/\CS_\bu)\ar[r]^{\delta_q}\ar[d]_{\psi_{q}}&H_{q-1}(\CS_\bu)\ar[r]\ar[d]&\dots\\
\ddot\Ha:&\dots\ar[r]&H_{q+1}(\ddot\CS_\bu/\DS_\bu)\ar[r]_{\ddot\delta_{q+1}}&H_q(\DS_\bu)\ar[r]_{\bar j_{*,q}}&H_q(\ddot\CS_\bu)\ar[r]_{\bar q_{*,q}}&H_{q}(\ddot\CS_\bu/\DS_\bu)\ar[r]&\dots
}
\]
where $\varphi_q=\tilde i_{*,q} \dot\delta_q^{-1}$ and $\psi_q=\bar p_{*,q}\tilde j_{*,q}$ are isomorphisms. Let $\hs_q$, $\hs_q^{\DS_\bu}$ and $\hs_q''$ be bases of $H_q(\CS_\bu)$, $H_q(\DS_\bu)$ and $H_q(\DS_\bu/\CS_\bu)$ respectively. Then, we fix the bases $\ddot\hs_q=\psi_q(\hs_q'')$, and $\ddot\hs_{q+1}''=\varphi_q(\hs_{q})$ for $H_q(\ddot\CS_\bu)$ and $H_{q+1}(\ddot \CS_\bu/\DS_\bu)$. 


When computing the torsion, we need to  identify  the sets of elements $\ys'_q$, $\ys_q$ and $\ys''_q$ in the modules of the two sequences $\Ha$ and $\ddot \Ha$.  We proceed as follows. 
We will denote by $\ys_q'$ a set on $H_q(\CS_\bu)$ such that its image is a basis for the image of $i_{*,q}$, and by $\ddot\ys_{q+1}''$ the corresponding set under $\varphi_q$; by $\ys_q$ a set in $H_q(\DS_\bu)$ such that its image is a basis for $p_{*,q}$;  by $\ys_q''$ a set in $H_q(\DS_\bu/\CS_\bu)$ such that its image is a basis for the image of $\d_q$ and by $\ddot \ys_q$ its image under $\psi_q$. Thus, the torsions reads

\begin{align*}
\tau(\Ha;\hs_q,\hs_q^{\DS_\bu},\hs_q'')=&\sum_{q=0}^{m+1}(-1)^q\left([(\de_{q+1}(\ys_{q+1}'')\ys_q'/\hs_q)]
-[(i_{*,q}(\ys_q') \ys_q/\hs_q^{\DS_\bu})]+[( p_{*,q}(\ys_q) \ys_q''/\hs_q'')]\right)\\
=&\sum_{q=0}^{m}(-1)^q [(\de_{q+1}(\ys_{q+1}'')\ys_{q}'/\hs_{q})]\\
&+\sum_{q=0}^{m+1}(-1)^q\left(
-[(i_{*,q}(\ys_q') \ys_q/\hs_q^{\DS_\bu})]+[( p_{*,q}(\ys_q) \ys_q''/\hs_q'')]\right),
\end{align*}
and
\[
\tau(\ddot\Ha;\ddot\hs_q,\hs_q^{\DS_\bu},\ddot\hs_q'')=\sum_{q=0}^{m+1}(-1)^q\left([(\ddot\de_{q+1}(\ddot\ys''_{q+1})\ddot\ys'_q/\hs_q^{\DS_\bu})]
-[(\bar j_{*,q}(\ddot\ys_q') \ddot\ys_q/\ddot\hs_q)]+[( \bar q_{*,q}(\ddot\ys_q) \ddot\ys_q''/\ddot\hs_q'')]
\right).
\]

Observing that
\begin{align*}
\ddot\ys_{q+1}''&=\varphi_q(\ys_q'),\\
\ddot\de_{q+1}(\ddot\ys_{q+1}'')&=\ddot\d_{q+1}''\varphi_q(\ys_q')=i_{*,q}(\ys_q'),\\
\ddot\ys_q'&=\ys_q,\\
\bar j_{*,q}(\ddot \ys_q')&=\bar j_{*,q}(\ys_q)=\psi_{q} p_{*,q}(\ys_q),\\
\ddot \ys_q&=\psi_q(\ys_q''),\\
\bar q_{*,q}(\ddot\ys_q)&=\bar q_{*,q}\psi_q(\ys_q'')=\varphi_{q-1}\d_q(\ys_q''),
\end{align*}
we have that
\[
(\ddot\de_{q+1}(\ddot\ys''_{q+1})\ddot\ys'_q/\hs_q^{\DS_\bu})=(i_{*,q}(\ys_q') \ys_q/\hs_q^{\DS_\bu}),
\]
\[
(\bar j_{*,q}(\ddot\ys_q') \ddot\ys_q/\ddot\hs_q)=(\psi_{q} p_{*,q}(\ys_q)\psi_q(\ys_q'')/\psi_q(\hs_q''))
=( p_{*,q}(\ys_q)\ys_q''/\hs_q''),
\]
and 
\[
( \bar q_{*,q}(\ddot\ys_q) \ddot\ys_q''/\ddot\hs_q'')=(\varphi_{q-1}\de_q(\ys_q'')\varphi_{q-1}(\ys_{q-1}')/\varphi_{q-1}(\hs_{q-1}))
=(\d_q(\ys_q'')\ys_{q-1}'/\hs_{q-1}).
\]

Therefore,
\begin{align*}
\tau(\ddot\Ha;\ddot\hs_q,\hs_q^{\DS_\bu},\ddot\hs_q'')=&\sum_{q=0}^{m+1}(-1)^q\left([(\ddot\de_{q+1}(\ddot\ys''_{q+1})\ddot\ys'_q/\hs_q^{\DS_\bu})]
-[(\bar j_{*,q}(\ddot\ys_q') \ddot\ys_q/\ddot\hs_q)]\right)\\
&+\sum_{q=0}^{m+1}(-1)^q
[( \bar q_{*,q}(\ddot\ys_q) \ddot\ys_q''/\ddot\hs_q'')]\\
=&\sum_{q=0}^{m+1}(-1)^q\left([(i_{*,q}(\ys_q') \ys_q/\hs_q^{\DS_\bu})]
-[( p_{*,q}(\ys_q)\ys_q''/\hs_q'')]\right)\\
&+\sum_{q=0}^{m}(-1)^{q+1}
[(\de_{q+1}(\ys_{q+1}'')\ys_{q}'/\hs_{q})]\\
&=-\tau(\Ha;\hs_q,\hs_q^{\DS_\bu},\hs_q'').
\end{align*}

\subsubsection{Torsion}\label{bb11} Assume $(\DS_\bu, \CS_\bu)$ are based consistently, and denote the bases by $\ds_\bu$ and $\cs_\bu$. Denote by $\ds_q''$ is the subset of elements of $\ds_\bu$ that lie in $\DS_\bu-i_\bu(\CS_\bu)$: this is a basis for $(\DS_\bu,\CS_\bu)$.  We have a  short exact sequence of based modules
\beq\label{lollo}
\xymatrix{0\ar[r]&(\CS_\bu,\cs_\bu)\ar[r]^{ i_\bu}&(\DS_\bu,\ds_\bu)\ar[r]^{p_\bu}&(\DS_\bu /\CS_\bu, \ds''_\bu)\ar[r]&0,}
\eeq
with
\[
\ds_\bu=i_\bu(\cs_\bu) \widehat{\ds_\bu''}.
\]

Using the chain isomorphism  described in equation \eqref{eeee}, $\tilde j_\bu(\ds_\bu'')$ is a chain basis for $(\ddot \CS_\bu ,C(\CS_\bu))$. Consider the short exact sequence of chain complexes
\[
\xymatrix{0\ar[r]&C(\CS_\bu)\ar[r]^{\bar i_\bu}&\ddot \CS_\bu\ar[r]^{\bar p_\bu}&(\ddot \CS_\bu ,C(\CS_\bu))\ar[r]&0,}
\]
and observe that the restriction to the inclusion 
\[
\bar s_\bu:=\bar p_\bu^{-1}|_{\bar j_\bu(\DS_\bu-i_\bu(\CS_\bu))}:(\ddot \CS_\bu ,C(\CS_\bu))\to \ddot\CS_\bu
\]
is a splitting map for $\bar p_\bu$. Then, a coherent basis for $\ddot \CS_\bu$ is given by
\[
\ddot \cs_\bu=\bar i_\bu(\dot\cs_\bu)\bar s_\bu\tilde j_\bu(\ds_\bu'').
\]

This reads
\[
\ddot \cs_q=\cs_{q-1}\oplus 0\ss 0\oplus \ds_q=\cs_{q-1}\oplus 0\ss 0\oplus i_q(\cs_q) \ss 0\oplus \ds_q''.
\]

For homology, note that  the induced map $\bar p_{*,q}$ is an isomorphism for $q>0$, while in dimension $q=0$ the map $\bar i_{*,q}$ is an isomorphism. Denote by $\hs_\bu''$ a graded homology basis for $H_\bu(\DS_\bu/\CS_\bu)$, and by $\ddot\hs_\bu$ a basis for $H_\bu(\ddot \CS_\bu)$.   Since $\tilde j_{*,\bu}$ is an isomorphism, a graded basis for $H_\bu(\ddot \CS_\bu,\CS_\bu)$ is $\tilde j_{*,q}(\hs_q'')$. Since $\bar p_{*,q}$ is an isomorphism for $q>0$, for such $q$,
we have $\ddot\hs_q= \bar p^{-1}_{*,q}\tilde j_{*,q}(\hs_q'')$. In the notation of direct sum, this reads
\[
\ddot\hs_\bu=0\oplus \hs''_\bu.
\]

Whence, in order to define the torsion of $\ddot \CS_\bu$, it is sufficient to assume that the relative homology groups $H_q(\DS_\bu,\CS_\bu)$ are free with the  basis above.

We proceed to compute the torsion.   We start by  computing the new bases for torsion. In these calculations we omit the inclusion and splitting maps, wherever possible.

For $q>0$,  if $c=x\oplus y\in \ddot \CS_q=\CS_{q-1}\oplus \DS_{q}$, 
\[
\ddot\b_q(c)=\left(\begin{matrix}\b_{q-1}&0\\1&-\b_{q}\end{matrix}\right)\left(\begin{matrix}x\\y\end{matrix}\right)=\left(\begin{matrix} \b_{q-1}( x)\\ x-\b^{\DS_\bu}_{q}( y)\end{matrix}\right).
\]

Thus, 
\begin{align*}
\ddot\b_q(0\oplus \bs^{\DS_\bu}_q)&=0\oplus -\b^{\DS_\bu}_q(\bs^{\DS_\bu}_q)\not=0,&\ddot\b_q(\cs_{q-1}\oplus 0)&=\b_{q-1}(\cs_{q-1})\oplus \cs_{q-1}\not=0,
\end{align*}
and this happens only for these elements (the notation $\langle\vs, \vs'\rangle$ for two sets of linearly independent vectors means the subspace direct sum of the two subspaces generated by these sets of vectors). However, these elements do not have linearly independent images: for if $\b_q^{\DS_\bu}(y)=x$ ($0\not=y=\b_q x$), then  
\begin{align*}
\ddot\b_q(0\oplus y+x\oplus 0)&=\b_{q-1}(x)\oplus x-\b^{\DS_\bu}_q(y)=0,
\end{align*}
the image will  vanish. Since $\cs_{q-1}=\b_q(\bs_q) \hat\hs_{q-1}\bs_{q-1}$, this problem  can be avoided taking only the elements in $\cs_{q-1}$ coming from $\hat\hs_{q-1}\bs_{q-1}$. Therefore, a set of  elements in $\ddot \CS_q$ with non trivial l.i. image is 
\[
\ddot \bs_q=\hat\hs_{q-1}\bs_{q-1}\oplus 0\ss0\oplus \bs^{\DS_\bu}_q.
\]

Recalling  that the basis $\ds_q$ splits as $i_q(\cs_q) \widehat{\ds_q''}$, thus choosing the set $\bs_q''$ for the quotient complex, we may choose the set $\ddot \bs_q$ as follows (this works since the homology of the cone is trivial!) 
\[
\ddot \bs_q=\cs_{q-1}\oplus 0\ss0\oplus \bs''_q,
\]
and
\[
\ddot\b_q(\ddot \bs_q)= \b_{q-1}(\cs_{q-1})\oplus \cs_{q-1}\ss 0\oplus \b''_{q}(\bs''_{q}),
\]
and hence the new basis for $\ddot \CS_q$, $q>1$,  is (the equivalence follows since the matrix of the change of basis is block triangular, see Section \ref{tor})
\begin{align*}
\ddot\b_{q+1}(\ddot \bs_{q+1})\ddot \hs_q\ddot\bs_q
&= \b_{q}(\cs_{q})\oplus \cs_{q}\ss 0\oplus \b''_{q+1}(\bs''_{q+1}) \ss\ddot \hs_q\ss\cs_{q-1}\oplus 0\ss0\oplus \bs''_q\\
&=\cs_{q-1}\oplus 0\ss\ddot \hs_q\ss 0\oplus \b''_{q+1}(\bs''_{q+1})\cs_{q}\bs''_q.
\end{align*}

We may now compute the torsion:
\begin{align*}
\tau(\ddot\CS_\bu;&\bar i_\bu(\dot\cs_\bu)\bar s_\bu\bar j_\bu(\ds_\bu''),\bar i_{*,\bu}(\dot\hs_\bu)\bar p^{-1}_{*,\bu}\tilde j_{*,\bu}(\hs_\bu''))\\
=&\sum_{q=0}^{m+1} (-1)^q [(\cs_{q-1}\oplus 0\ss0\oplus  \widehat{\hs''_q}\ss 0\oplus \b''_{q+1}(\bs''_{q+1})\cs_{q}\bs''_q/\cs_{q-1}\oplus 0\ss 0\oplus \cs_q \ds''_q)]\\
=&\sum_{q=0}^{m+1} (-1)^q [(\cs_{q-1}/\cs_{q-1})]
+\sum_{q=0}^{m+1} (-1)^q [( \b''_{q+1}(\bs''_{q+1})\widehat{\hs''_q}\bs''_q/\ds''_q)]\\
=&\tau(\DS_\bu/\CS_\bu;\ds''_\bu,\hs_\bu'').
\end{align*}

Observing that, considering the based exact sequence \ref{lollo}, by Lemma \ref{mil0},
\[
\tau(\DS_\bu;\ds_\bu,\hs_\bu^{\DS_\bu})=\tau(\CS_\bu;\hs_\bu,\cs_\bu)+\tau(\DS_\bu/\CS_\bu;\ds''_\bu,\hs_\bu'')+\tau(\Ha),
\]
where $\Ha$ is associated homology long the exact sequence, and that a simple comparison using the diagram in Figure \ref{ladder} shows that $\tau(\Ha)=-\tau(\ddot\Ha)$, as proved in Section \ref{ladder1}, we have that the two previous results are consistent.

\begin{prop} Making the suitable identifications of the grade homology bases, we have that
\[
\tau( \ddot \CS_\bu;\ddot\cs_\bu,\ddot\hs_\bu)=\tau(\DS_{\bu}/\CS_{\bu};\ds_\bu'',\hs_\bu'').
\]

\end{prop}

\begin{proof} Consider the short exact sequence of based chain complexes 
\[
\xymatrix{
(\dot \CS_\bu,\dot\cs_\bu)\ar[r]^{ \bar i_\bu}& (\ddot\CS_\bu,\ddot\cs_\bu)\ar[r]^{ \bar p_\bu}&( \ddot \CS_\bu/ \dot \CS_\bu,\ddot\cs_\bu'')
}
\]

Recalling the chain isomorphism $\tilde j_\bu :\DS_{\bu}/\CS_{\bu}\to \ddot \CS_\bu/\dot \CS_\bu$, see equation \eqref{eeee}, and making the suitable identifications of chain and homology bases, this turns out to be a simple chain isomorphism, so we may rewrite 
\[
\xymatrix{
(\dot \CS_\bu,\dot\cs_\bu)\ar[r]^{ \bar i_\bu}& (\ddot\CS_\bu,\ddot\cs_\bu)\ar[r]^{ \tilde j_q^{-1} \bar p_\bu}&( \DS_\bu/\CS_\bu, \ds_\bu'').
}
\]

Applying Theorem \ref{mil0}, we have
\[
\tau( \ddot \CS_\bu;\ddot\cs_\bu,\ddot\hs_\bu)=\tau( \dot\CS_\bu; \dot\cs_\bu, \dot\hs_\bu)+
\tau(\DS_{\bu}/\CS_{\bu};\ds_\bu'',\hs_\bu'')+\tau(\dot \Ha;\dot\hs_\bu,  \ddot\hs_\bu,\hs_\bu''),
\]
where $\dot \Ha$ is the associated long homology exact sequence. Since  the homology of $\dot\CS_\bu$ is trivial in positive degrees,  and with our identification of the homology bases, the sequence $\dot\Ha$ is a sequence of isomorphisms and $0$ maps, with zero torsion. We computed in Section \ref{app-cone}
\[
\tau(\dot\CS_\bu;\dot\cs_\bu,\dot\hs_\bu)=[(\dot \hs_0/\hs_0)], 
\]
so if we identify the two bases,  this completes the proof. 
\end{proof}

\subsection{Torsion and duality}

\label{asss}

A chain complex   $\CS_\bu$ of dimension $m$ is a {\it dualisable chain complex} if there exists a second complex $\check\CS_\bu$ of the same dimension and a  chain  isomorphism that we call {\it Poincar\'e  isomorphism},
\[
\P_q:\check\CS_q\to \CS_{m-q}^\da.
\]

It is clear that $\P_\bu$ induces an isomorphism $\P_{*,\bu}$ in homology.

\begin{rem} \label{dualbasis} Note that, if $(\CS_\bu,\cs_\bu)$ is based, then $\check\CS_q$  is naturally based with  basis $\check\cs_\bu=\P_\bu^{-1}(\cs_{m-q}^\da)$, and that the isomorphism $\P_\bu$ is an isometry with respect to the metrics induced by the chain bases. If $H_\bu(\CS_\bu)$ is free and has graded basis $\hs_\bu$, it follows that $H_\bu(\check \CS_\bu)$ is free and has graded basis $\check \hs_\bu=\P_{*,\bu}^{-1}(\hs_{m-\bu}^\da)$.
\end{rem}

Also note that we have the commutative diagram of isomorphisms
\[
\xymatrix{
\check \CS_q\ar[r]^{\check \psi_q}\ar[d]_{\P_q}&\check\CS_q^\da\\
\CS_{m-q}^\da&\CS_{m-q}\ar[u]_{\P^\da_{m-q}}\ar[l]^{\psi_{m-q}^\da}
}
\]

\begin{lem}\label{duale} Let be a dualizable based chain complex $\CS_\bu$  of dimension $m$ of free modules, with graded basis $\cs_\bu$, and  chain  isomorphism 
\[
\P_q:\check\CS_q\to \TCS_q= \CS_{m-q}^\da.
\]

Assume also the homology $H_\bu(\CS_\bu)$ is free with graded basis $\hs_\bu$. Let $\bs_q$ a set of linearly independent vectors in $\CS_q$ such that $\b_q(\bs_q)$ generate $\Im (\b_q)$.  Let $\check\cs_q=\P_q^{-1}(\cs_{m-q}^\da)$, $\check\bs_q=\P_q^{-1}((\b_{m-q+1}(\bs_{m-q+1}))^\da)$, and $\check \hs_q=\P_{*,q}^{-1}(\hs_{m-q}^\da)$. Then, 
\[
[(\check\b_{q+1}(\check \bs_{q+1}) \hat{\check \hs}_{q}\check\bs_q/\check\cs_q)]=-[(\b_{m-q+1}(\bs_{m-q+1}) \hat{ \hs}_{m-q}\bs_{m-q}/\cs_{m-q}))],
\]
and
\begin{align*}
\tau_{\rm W}(\check\CS_\bu;\check\cs_\bu,\check\hs_\bu)&=(-1)^{m+1}\tau_{\rm W}( \CS_\bu;\cs_\bu,\hs_\bu).
\end{align*}

\end{lem}

\begin{proof} Since  $\CS_\bu$ is based with basis $\cs_\bu$, then $\check \CS_\bu$ is naturally based (it is obvious that also $\TCS_\bu$ is based). Also, since the homologies are free with graded bases $\hs_\bu$ and $\check\hs_\bu=\P^{-1}_{*,\bu}(\hs_{m-\bu}^\da)$.

Let $\bs_q$ a set of linearly independent vectors in $\CS_q$ such that $\b_q(\bs_q)$ generate $\Im (\b_q)$. Then, the set 
\[
\b_{q+1}(\bs_{q+1}) \hat{ \hs}_{q}\bs_q,
\]
is a basis of $\CS_q$ and  the set 
\[
(\b_{m-q+1}(\bs_{m-q+1}))^\da\hat{ \hs}_{m-q}^\da\bs_{m-q}^\da,
\]
is a basis of $\CS_{m-q}^\da$, and  is dual in the proper sense to the basis $\b_{m-q+1}(\bs_{m-q+1}) \hat{ \hs}_{m-q}\bs_{m-q}$. Therefore, the matrices of the change of basis satisfy the equation
\beq\label{em2}
((\b_{m-q+1}(\bs_{m-q+1}))^\da\hat{ \hs}_{m-q}^\da\bs_{m-q}^\da/\cs_q^\da)= ((\b_{m-q+1}(\bs_{m-q+1}) \hat{ \hs}_{m-q}\bs_{m-q}/\cs_{m-q}))^T)^{-1}.
\eeq

Moreover, we may verify that (see \cite[pg. 544]{Spr20} for details)
\[
\Tb_q((\b_{m-q+1}(\bs_{m-q+1}))^\da)=\b_{m-q+1}^\da((\b_{m-q+1}(\bs_{m-q+1}))^\da)=\bs_{m-q+1}^\da.
\] 

Let $\check\bs_q=\P_q^{-1}((\b_{m-q+1}(\bs_{m-q+1}))^\da)$. 
Since $\Tb_q\P_q=\b_{m-q+1}^\da\P_q=\P_{q-1}\check \b_q$ \cite[pg. 530]{Spr20}, and 
\[
\left. \P_{q-1}\right|_{\Im(\check\b_q)}:\Im(\check\b_q)\to \Im (\Tb_q),
\]
is an isomorphism, and $\Tb_q((\b_{m-q+1}(\bs_{m-q+1}))^\da)=\bs_{m-q+1}^\da$ (see \cite[pg. 544]{Spr20} for details), the set  $\check\bs_q$ is  a set of linearly independent elements of $\check\CS_q$ such that $\check\b_q(\check\bs_q)$ generate $\Im (\check\b_q)$. Thus, the set 
\[
\check\b_{q+1}(\check \bs_{q+1}) \hat{\check \hs}_{q}\check\bs_q,
\]
is a basis of $\check \CS_{q}$, and the set 
\[
\Tb_{q+1}\P_{q+1}(\check \bs_{q+1}) \P_q(\hat{\check \hs}_{q})\P_q(\check\bs_q),
\]
is a basis of $ \TCS_q= \CS_{m-q}^\da$, and we have the following equality between the matrices of change of basis
\begin{align*}
(\check\b_{q+1}(\check \bs_{q+1}) \hat{\check \hs}_{q}\check\bs_q/\check\cs_q)
=&(\Tb_{q+1}\P_{q+1}(\check \bs_{q+1}) \P_q(\hat{\check \hs}_{q})\P_q(\check\bs_q)/\P_q(\check\cs_q))\\
=&(\bs_{m-q}^\da \hat{ \hs}_{m-q}^\da (\b_{m+1-q}(\bs_{m-q+1}))^\da/\cs_{m-q}^\da))\\
=&((\b_{m-q+1}(\bs_{m-q+1}) \hat{ \hs}_{m-q}\bs_{m-q}/\cs_{m-q}))^T)^{-1},
\end{align*}
where we used equation (\ref{em2}). For the Whitehead classes this means that
\begin{align*}
[(\check\b_{q+1}(\check \bs_{q+1}) \hat{\check \hs}_{q} \check\bs_q/\check\cs_{m-q})]
&=[((\b_{m-q+1}(\bs_{m-q+1}))^\da\hat{ \hs}_{m-q}^\da\bs_{m-q}^\da/\cs_q^\da)]\\
&= - [(\b_{m-q+1}(\bs_{m-q+1}) \hat{ \hs}_{m-q}\bs_{m-q}/\cs_{m-q}))].
\end{align*}

\end{proof}

\section{Algebraic  intersection homology}
\label{algtor}

\begin{defi}\label{perversity} A {\it perversity} is a finite sequence of integers $\pf=\{\pf_j\}_{j=2}^n$ such that $\pf_2=0$ and $\pf_{j+1}=\pf_j$ or $\pf_j+1$. If $\pf$ is a perversity (of length $n$),  we define the constant $\af=\af(\pf_n):=n-\pf_n$.
\end{defi}

The perversity: $\mf=\{\mf_j=[j/2]-1\}$ is called {\it lower middle perversity}. The {\it null perversity} is $0_j=0$, and the {\it top perversity} is $\tf_j=j-2$. Given a perversity $\pf$, the {\it complementary perversity} $\pf^c$ is $\pf^c_j=\tf_j-\pf_j=j-\pf_j-2$.

\subsection{Prelimiraries}

If not specified otherwise, all the chain complexes are chain complexes of free left $R$-modules. 
Let 
\[
\xymatrix{
\CS_\bu:& \CS_m\ar[r]^{\b_m}&\CS_{m-1}\ar[r]^{\b_{m-1}}&\dots\ar[r]^{\b_2}&\CS_1\ar[r]^{\b_1}&\CS_0,
}
\]
be a chain complex, and for $0\leq k\leq m$ denote by $\CS_\bu^{(k)}$ the truncated chain complex
\[
\xymatrix{
\CS_\bu^{(k)}:& \CS_k\ar[r]^{\b_k}&\CS_{k-1}\ar[r]^{\b_{k-1}}&\dots\ar[r]^{\b_2}&\CS_1\ar[r]^{\b_1}&\CS_0.
}
\]

Let $(\DS_\bu,\CS_\bu)$ a pair of chain complexes, with dimensions $(m+1,m)$. It is clear that we have inclusions of $\CS_\bu^{(k)}$ in $\CS_\bu^{(k+1)}$,  and an inclusion $i_\bu$ of $\CS_\bu$ in $\DS_\bu$. 

As in Section \ref{B4}, let $\ddot\CS_\bu=C(\CS_\bu)\sqcup_{i_\bu} \DS_\bu$ denote the mapping cone.

Let $\pf$ be a perversity, and  define the constant $\af=\af_n=n-\pf_n$. Then, we define the chain complex $\ES_{q,\bu}$ by
\[
\ES_{q,\bu}=\left\{\begin{array}{ll} \DS_\bu^{(q)}, & q<\af_{m+1},\\ \ddot \CS_\bu^{(q)}, &q\geq \af_{m+1}.\end{array}\right. 
\]
and with boundary homomorphism $\b^{\ES_\bu}_q$   defined by the natural compositions of the homology exact sequences associated to the chain pairs $(\ES_{q,\bu},\ES_{q-1,\bu})$:

\centerline{\xymatrix@C=0.4cm{
&&&\dots\ar[d]&\\
&&&H_{q-1}(E_{q-2})\ar[d]&\\
\dots\ar[r]&H_{q}(E_{q-1})\ar[r]&H_{q}( E_q, E_{q-1})\ar[dr]_{\b^{\ES_\bu}_q}\ar[r]^{}&H_{q-1}(E_{q-1})\ar[d]^{}\ar[r]&\dots\\
&&&H_{q-1}(E_{q-1},E_{q-2})\ar[d]&\\
&&&\dots&
}}

\subsection{The intersection chain complex of the cone of a chain complex} \label{s4.1}

Consider the particular case when $\DS_\bu$ is empty, then
\[
\ES_{q,\bu}=\left\{\begin{array}{ll} \CS_\bu^{(q)}, & q<\af_{m+1},\\ \dot \CS_\bu^{(q)}, &q\geq \af_{m+1},\end{array}\right. 
\]
and we define the {\it intersection chain complex} $I^\pf C(\CS_\bu)$ of perversity $\pf$ of $C(\CS_\bu)$ by
\[
I^\pf C(\CS_\bu)_q=H_q(\ES_{q,\bu},\ES_{q-1,\bu}).
\]

It is then clear that:

\begin{align*}
I^\pf C(\CS_\bu)_q&=\left\{\begin{array}{ll}\CS_q&q< \af,\\
H_{\af}(\dot \CS_\bu^{(\af)},\CS_\bu^{(\af-1)}),&q=\af, \\
\dot \CS_q&q> \af,
\end{array}\right.
\end{align*}
with boundary
\[
\begin{aligned} I^\pf \b_q&= j_{q-1}''' \b_q'''=\dot \b_q,&q\geq \af+2,\\ 
I^\pf \b_{\af+1}&=j''_\af \b_{\af+1}''',&\\
I^\pf \b_\af&=j_{\af-1}' \b''_\af,&\\
I^\pf \b_q&=j_{q-1}' \b_q '=\b_q, &q\leq \af-1,
\end{aligned}
\]
where the homomorphisms $i''',  j''', \b''$, $i'',j'',\b''$, and $i',j',\b'$ comes from the exact sequences:
\[
\xymatrix{ \dots\ar[r]&H_q(\CS_\bu^{(q)})\ar[r]^{j'_q}&H_q(\CS_\bu^{(q)},\CS_\bu^{(q-1)})\ar[r]^{\b_q'} &H_{q-1}(\CS_\bu^{(q-1)})\ar[r]&\dots,\\
 \dots\ar[r]&H_q(\dot \CS_\bu^{(q)})\ar[r]^{j'''_q}&H_q(\dot\CS_\bu^{(q)},\CS_\bu^{(q-1)})\ar[r]^{\b_q'''} &H_{q-1}(\CS_\bu^{(q-1)})\ar[r]&\dots,\\
 \dots\ar[r]&H_q(\dot \CS_\bu^{(q)})\ar[r]^{j'''_q}&H_q(\dot\CS_\bu^{(q)},\dot\CS_\bu^{(q-1)})\ar[r]^{\b_q'''} &H_{q-1}(\dot\CS_\bu^{(q-1)})\ar[r]&\dots.
}
\]

\begin{lem}\label{l4.1-1}
\begin{align*}
I^\pf C(\CS_\bu)_q&=\left\{\begin{array}{ll}\CS_q&q< \af,\\
 Z_{\af-1}\oplus\CS_\af,&q=\af, \\
\dot \CS_q&q> \af,
\end{array}\right.
\end{align*}
with boundary
\[
\begin{aligned} I^\pf \b_q&= \dot \b_q, &q&\geq \af,\\
I^\pf \b_q&=\b_q, &q&\leq \af-1,
\end{aligned}
\]
i.e. $I^\pf C(\CS_\bu)_\bu:$

\beq\label{cc}
\xymatrix{
\dot \CS_{m+1} \ar[r]^{\dot \b_{m+1}}& \dots \ar[r]^{\dot \b_{\af+2}}&\dot \CS_{\af+1}\ar[r]^{\dot \b_{\af+1}}&Z_{\af-1}\oplus \CS_{a}\ar[r]^{\dot \b_\af}& \CS_{\af-1}\ar[r]^{\b_{\af-1}}&\dots\ar[r]^{\b_1}&\CS_0.}
\eeq

\end{lem}

\begin{proof} We need only to consider the case $q=\af$.   Since 
\[
(I^{\pf} \dot\CS_\bu)_\af = H_\af(\dot \CS^{(\af)}_\bu, \CS^{(\af-1)}_\bu),
\] 
the relevant diagram is
\[
\xymatrix{
&\CS_{\af-1}\ar[d]^{\iota_{\af-1}}\ar[r]&\dots\\
\dot \CS_\af\ar[r]^{\dot \b_{\af}}&\dot\CS_{\af-1}\ar[d]\ar[r]&\dots\\
&\dot\CS_{\af-1}/\iota_{\af-1}(\CS_{\af-1})&
}
\]
where the vertical sequence is the exact sequence associated to the inclusion $\iota_\bu$. Recalling  that $\dot \CS_{\af-1} = \CS_{\af-2}\oplus \CS_{\af-1}$, and that the inclusion $\iota_{\af-1}: \CS_{\af-1}\to \dot \CS_{\af}$ is $\iota_{\af-1}(x) = 0\oplus x$, the associated  exact sequence reads
\[
\xymatrix{
0 \ar[r] & \CS_{\af-1} \ar[r]^{\hspace{-20pt}\iota_{\af-1}} & \dot \CS_{\af-1} =  \CS_{\af-2}\oplus \CS_{\af-1} \ar[r] & \CS_{\af-2} \ar[r] & 0,
}
\] 

Whence, we need to compute the kernel of the composition $pr_1\ddot \b_{\af}$:
\[
\ker (pr_1\dot \b_{\af})= \{(x,y)\in \dot \CS_{\af} =  \CS_{\af-1}\oplus \CS_{\af} : pr_1\dot \b_\af(x,y) =(\b_{\af-1}(x),0)= 0\}=Z_{\af-1}\oplus \CS_\af,
\] 
where $Z_{\af-1}=\ker \b_{\af-1}$ of the complex $\CS_\bu$.

\end{proof}

\begin{lem}\label{funct}  The construction is functorial in the category of the finite chain complexes of free $R$-modules: namely, if $\varphi_\bu:\CS_\bu\to \CS'_\bu$ is a chain map, we have a chain map $I^\pf C(\varphi_\bu)_\bu:I^\pf C(\CS_\bu)_\bu\to C(\CS'_\bu)_\bu$, where $C(\varphi_\bu)_\bu$ is the cone of the map $\varphi_\bu$.
\end{lem}

\begin{lem}\label{exact} The functor $I^\pf$ is exact. 
\end{lem}
\begin{proof} This follows by the description of the intersection chain complex given in Lemma \ref{l4.1-1}.
\end{proof}

\begin{lem}\label{homo} 
\[
H_q(I^\pf C(\CS_\bu))=\left\{\begin{array}{ll} H_q(\CS_\bu),& 0\leq q\leq \af-2,\\0,&\af-1\leq q\leq m+1.\end{array}\right.
\]
\end{lem}
\begin{proof}
By the Lemma \ref{l4.1-1}, if $q\leq \af-2$ then $I^{\pf} C(\CS_\bu)_q = \CS_q$ and $I^\pf \b_q=\b_q$, so $ H_q(I^\pf C(\CS_\bu))=H_q(\CS_\bu)$. If
$q\geq \af+1$ then $I^{\pf} C(\CS_\bu)_q = \dot \CS_q$ and $I^\pf \b_q=\dot \b_q$, so $ H_q(I^\pf C(\CS_\bu))=H_q(\dot \CS_\bu) = 0$. 

For the other cases, when $q=\af-1$, $\Im I^\pf \b_{\af} = Z_{\af-1} = \ker \b_{\af-1}$ then $H_{\af-1}(I^\pf C(\CS_\bu)) = 0$. If $q=\af$ then  
\[
\ker I^\pf\b_{\af} = \{(z_{\af-1},c_{\af})\in  Z_{\af-1}\oplus \CS_{\af}: \b_{\af}(c_{\af}) = z_{\af-1}\}=\hat H_{\af-1}\oplus 0.
\] 

Then take $(\b_{\af}(c_{\af}),c_{\af})\in \ker \dot\b_{\af}$ and consider the element $(c_{\af},0)\in I^\pf C(\CS_\bu)_{\af+1}$, then 
\[
I^{\pf} \b_{\af+1} (c_{\af},0) = (\b_{\af}(c_{\af}), c_{\af}),
\] 
and we obtain that $H_{\af}(I^\pf C(\CS_\bu)) = 0$.
\end{proof}

\begin{rem}\label{intersection basis}  Assume that $\CS_\bu$ is a based complex, with preferred graded basis $\cs_\bu$.  By Lemma \ref{l4.1-1}, $I^\pf (C(\CS_\bu))_q$ is free and a preferred basis determined  by that of $\CS_\bu$ for all $q\not=\af$, 
\begin{align*}
 I^\pf (C(\cs_\bu))_q&= \cs_q, \hspace{10pt}q\leq \af-1,& I^\pf (C(\cs_\bu))_q&=\dot \cs_q, \hspace{10pt}q\geq \af+1.
\end{align*}
but $I^\pf( C(\CS_\bu))_\af=Z_{\af-1}\oplus \CS_\af$, is not necessarily free, but is stably free by the next assumption. We assume some basis $\zs_{\af-1}$ of $Z_{\af-1}$ has been fixed and we set $I^\pf (C(\cs_\bu))_{\af}= \zs_{\af-1}\oplus \cs_\af$. 
Assuming that $H_\bu(\CS_\bu)$ is free with preferred graded basis $\hs_\bu$, using Lemma \ref{homo}, it follows that $ H_\bu (I^\pf C(\CS_\bu))$ is also free and has the preferred graded basis $I^\pf \dot\hs_\bu$ defined by
\begin{align*}
I^\pf\dot \hs_q&=\hs_q, ~q\leq \af-2, & I^\pf\dot \hs_q&=\emptyset, ~q\geq\af-1.
\end{align*}
\end{rem}

\subsection{The relative intersection chain complex of the cone of a chain complex} \label{s4.2} We have an inclusion 
\begin{align*}
\iota_q &:\CS_q\to I^\pf C(\CS_\bu)_q,\\
\iota_q&: x\mapsto \left\{\begin{array}{ll} x,& q< \af,\\x\oplus 0,& q\geq \af,\end{array}\right.
\end{align*}
that commutes with the boundary operator, thus we have an inclusion of complexes and we can consider the  pair $(I^\pf C(\CS_\bu),\CS_\bu)$.   We call the resulting relative (quotient) complex  the {\it  intersection chain complex of the pair} $(I^\pf C(\CS_\bu),\CS_\bu)$ of perversity $\pf$, and we denote it by $(I^\pf C(\CS_\bu),\CS_\bu)_\bu$. We have the following short exact sequence of chain complexes
\[
\xymatrix{
0 \ar[r] & \CS_\bu \ar[r]^{\iota_\bu} &I^\pf C(\CS_\bu)\ar[r] &(I^\pf C(\CS_\bu),\CS_\bu)_\bu=I^\pf C(\CS_\bu)/\CS_\bu \ar[r] & 0.
}
\]

\begin{lem}\label{l4.3}
\begin{align*}
(I^\pf C(\CS_\bu)_\bu,\CS_\bu)_q&=\left\{\begin{array}{ll} 0, &q< \af,\\
Z_{\af-1},&q=\af, \\
(C(\dot \CS),\CS)_q=\CS_{q-1},&q> \af,
\end{array}\right.
\end{align*}
with boundary
\[
\begin{aligned} I^\pf \b_q&=  \b_{q-1}, &q&\geq \af+1,\\
I^\pf \b_q&=0, &q&\leq \af,
\end{aligned}
\]
i.e. the complex $(I^\pf C(\dot \CS_\bu)_\bu,\CS_\bu)_\bu$:
\beq\label{cc1}
\xymatrix{ \hspace{-20pt}(\dot
\CS_\bu,\CS_\bu)_{m+1}=\CS_m \ar[r]^{\b_{m}}& \dots \ar[r]^{ \b_{\af+1}}&(\dot
\CS_\bu,\CS_\bu)_{\af+1}=\CS_{\af}\ar[r]^{ \b_{\af}}&
Z_{\af-1}\ar[r]^{0}& 0 \ar[r]^{0}&\dots.}
\eeq

\end{lem}

\begin{lem}\label{l4.4}
\[
H_q(I^\pf C(\CS_\bu),\CS_\bu)=\left\{\begin{array}{ll} 0,& 0\leq q\leq \af-1,\\
H_q(\dot \CS,\CS) = H_{q-1}(\CS_\bu),&\af\leq q\leq m+1.\end{array}\right.
\]
\end{lem}

\begin{proof}
We just observe that, for $q\geq \af+1$,
\[
I^\pf (C(\CS_\bu),\CS_\bu)_q = (\dot \CS,\CS)_q \cong \CS_{q-1},
\] 
and the result follows.
\end{proof}

\begin{rem}\label{relative intersection basis}  Proceeding as in Remark  \ref{intersection basis},  we assume that 
$\CS_\bu$ is a based complex, and has  got a preferred graded basis $\cs_\bu$. Then $(I^\pf C(\CS_\bu)_q,\CS_\bu)$ is free, and has the preferred basis in all degree different from $\af$:
\begin{align*}
 I^\pf (C(\cs_\bu))_{\rm rel,q}&= \emptyset, \hspace{10pt}q\leq \af-1,& I^\pf (C(\cs_\bu))_{\rm rel, q}&= \cs_q, \hspace{10pt}q\geq \af+1.
\end{align*}

 $I^\pf( C(\CS_\bu))_\af=Z_{\af-1}$, is not necessarily free, but is stably free by the next assumption. We assume some basis $\zs_{ \af-1}$ of $Z_{\af-1}$ has been fixed and we set $I^\pf (C(\cs_\bu))_{\rm rel, \af-1}=\zs_{ \af-1} $.

Assuming that $H_\bu(\CS_\bu)$ is free with preferred graded basis $\hs_\bu$,  then $ H_\bu (I^\pf C(\CS_\bu),\CS_\bu)$ is also free and has the preferred graded basis $I^\pf \hs_{{\rm rel},\bu}$ defined by
\begin{align*}
I^\pf\hs_{{\rm rel},q}&=\emptyset, ~q\leq \af-1, & I^\pf\hs_{{\rm rel},q}&=\hs_{q-1}, ~q\geq\af.
\end{align*}

\end{rem}

\subsection{The intersection chain complex of the mapping cone of a chain complex} 
\label{Imappingcone}

We consider now the general case of the mapping cone let $\ddot\CS_\bu=C(\CS_\bu)\sqcup_{i_\bu} \DS_\bu$. 
Given a perversity $\pf$, the complex $\ES_{q,\bu}$ is
\[
\ES_{q,\bu}=\left\{\begin{array}{ll} \DS_\bu^{(q)}, & q<\af_{m+1},\\ \ddot \CS_\bu^{(q)}, &q\geq \af_{m+1}.\end{array}\right. 
\]
and we define the {\it intersection chain complex} $I^\pf (C(\CS_\bu)\sqcup_{i_\bu} \DS_\bu))_\bu$ of perversity $\pf$ of $C(\CS_\bu)\sqcup_{i_\bu} \DS_\bu$ by
\[
I^\pf C(\CS_\bu)_q=H_q(\ES_{q,\bu},\ES_{q-1,\bu}).
\]

It is then clear that:

\begin{align*}
I^\pf (C(\CS_\bu)\sqcup_{i_\bu} \DS_\bu)_q&=\left\{\begin{array}{ll}\DS_q&q< \af,\\
H_{\af}(\ddot \CS_\bu^{(\af)},\DS_\bu^{(\af-1)}),&q=\af, \\
 (C(\CS_\bu)\sqcup_{i_\bu} \DS_\bu)_q&q> \af,
\end{array}\right.
\end{align*}
with boundary
\[
\begin{aligned} I^\pf \b_q&= j_{q-1}''' \b_q'''=\ddot \b_q,&q\geq \af+2,\\ 
I^\pf \b_{\af+1}&=j''_\af \b_{\af+1}''',&\\
I^\pf \b_\af&=j_{\af-1}' \b''_\af,&\\
I^\pf \b_q&=j_{q-1}' \b_q '=\b^{\DS_\bu}_q, &q\leq \af-1,
\end{aligned}
\]
where the homomorphisms $i''',  j''', \b''$, $i'',j'',\b''$, and $i',j',\b'$ comes from the exact sequences:
\[
\xymatrix{ \dots\ar[r]&H_q(\DS_\bu^{(q)})\ar[r]^{j'_q}&H_q(\DS_\bu^{(q)},\DS_\bu^{(q-1)})\ar[r]^{\b_q'} &H_{q-1}(\DS_\bu^{(q-1)})\ar[r]&\dots,\\
 \dots\ar[r]&H_q(\ddot \CS_\bu^{(q)})\ar[r]^{j'''_q}&H_q(\ddot\CS_\bu^{(q)},\DS_\bu^{(q-1)})\ar[r]^{\b_q'''} &H_{q-1}(\DS_\bu^{(q-1)})\ar[r]&\dots,\\
 \dots\ar[r]&H_q(\ddot \CS_\bu^{(q)})\ar[r]^{j'''_q}&H_q(\dot\CS_\bu^{(q)},\ddot\CS_\bu^{(q-1)})\ar[r]^{\b_q'''} &H_{q-1}(\ddot\CS_\bu^{(q-1)})\ar[r]&\dots.
}
\]

\begin{lem}\label{l4.1i}
\begin{align*}
I^\pf (C(\CS_\bu)\sqcup_{i_\bu} \DS_\bu)_q&=\left\{\begin{array}{ll}\DS_q&q< \af,\\
 Z_{\af-1}\oplus\DS_\af,&q=\af, \\
\ddot \CS_q&q> \af,
\end{array}\right.
\end{align*}
with boundary
\[
\begin{aligned} 
I^\pf \b_q&=\b_q^{\DS_\bu}, &q&\leq \af-1,\\
I^\pf \b_q&= \ddot \b_q, &q&\geq \af,
\end{aligned}
\]
i.e. $I^\pf (\ddot\CS_\bu)_\bu:$

\beq\label{cci}
\xymatrix{\ddot \CS_{m+1} \ar[r]^{\ddot \b_{m+1}}& \dots \ar[r]^{\ddot \b_{\af+2}}&\ddot \CS_{\af+1}\ar[r]^{\ddot \b_{\af+1}}&Z_{\af-1}\oplus \DS_{a}\ar[r]^{\ddot \b_\af}& \DS_{\af-1}\ar[r]^{\b^{\DS_\bu}_{\af-1}}&\dots\ar[r]^{\b^{\DS_\bu}_1}&\DS_0.}
\eeq

\end{lem}

\begin{proof} We need only to consider the case $q=\af$.   Since 
\[
I^{\pf} (\ddot\CS_\bu)_\af = H_\af(\ddot \CS^{(\af)}_\bu, \DS^{(\af-1)}_\bu),
\] 
the relevant diagram is
\[
\xymatrix{
&\DS_{\af-1}\ar[d]^{\bar j_{\af-1}}\ar[r]&\dots\\
\ddot \CS_\af\ar[r]^{\ddot \b_{\af}}&\ddot\CS_{\af-1}\ar[d]\ar[r]&\dots\\
&\ddot\CS_{\af-1}/\bar j_{\af-1}(\DS_{\af-1})&
}
\]
where the vertical sequence is the exact sequence associated to the inclusion $\bar j_\bu$.

Recalling  that $\ddot \CS_{\af-1} = \CS_{\af-2}\oplus \DS_{\af-1}$, and that the inclusion $\bar j_{\af-1}: \DS_{\af-1}\to \ddot \CS_{\af}$ is $\bar j_{\af-1}(x) = 0\oplus x$, the associated exact sequence reads
\[
\xymatrix{
0 \ar[r] & \DS_{\af-1} \ar[r]^{\hspace{-20pt}\bar j_{\af-1}} & \ddot \CS_{\af-1} =  \CS_{\af-2}\oplus \DS_{\af-1} \ar[r] & \CS_{\af-2} \ar[r] & 0,
}
\] 

Whence, we need to compute the kernel of the composition $pr_1\ddot \b_{\af}$:
\[
\ker (pr_1\ddot \b_{\af})= \{(x,y)\in \ddot \CS_{\af} =  \CS_{\af-1}\oplus \DS_{\af} : pr_1\ddot \b_\af(x,y) =(\b_{\af-1}(x),0)= 0\}=Z_{\af-1}\oplus \DS_\af,
\] 
where $Z_{\af-1}=\ker \b_{\af-1}$ of the complex $\CS_\bu$. \end{proof}

\begin{lem} We have the identification of chain complexes
\[
I^\pf (C(\CS_\bu)\sqcup_{i_\bu}\DS_\bu)_\bu=I^\pf C(\CS_\bu)_\bu \sqcup_{i_\bu}\DS_\bu.
\]
\end{lem}

\begin{lem}\label{homoi} 
\[
H_q(I^\pf (C(\CS_\bu)\sqcup_{i_\bu} \DS_\bu))=\left\{\begin{array}{ll} H_q(\DS_\bu),& 0\leq q\leq \af-2,\\
p_{*,\af-1}(H_{\af-1}(\DS_\bu))\leq H_{\af-1}(\DS_\bu/\CS_\bu),\\
H_q(\DS_\bu,\CS_\bu),&\af\leq q\leq m+1.\end{array}\right.
\]
\end{lem}
\begin{proof}
By the lemma \ref{l4.1i}, if $q\leq \af-2$ then $I^{\pf} (\ddot\CS_\bu)_q = \DS_q$ and $I^\pf \b_q=\b^{\DS_\bu}_q$, so $ H_q(I^\pf (C(\CS_\bu)\sqcup_{i_\bu} \DS_\bu))=H_q(\DS_\bu)$. 
If $q\geq \af+1$ then $I^{\pf} (\ddot\CS_\bu)_q = \ddot \CS_q$ and $I^\pf \b_q= \ddot\b_q$, so $ H_q(I^\pf (C(\CS_\bu)\sqcup_{i_\bu} \DS_\bu))=H_q( \ddot\CS_\bu) =H_q(\DS_\bu,\CS_\bu)$, since $q>0$. 

For the other cases,  the relevant part of the complex is
\[
\xymatrix@C=40pt{ \dots \ar[r]^{\hspace{-50pt}I^\pf \ddot\b_{\af+1}=\ddot \b_{\af+1}}& (I^\pf \ddot\DS_\bu)_{\af}= Z_{\af-1} \oplus \DS_{\af}\ar[r]^{I^\pf\b_\af=\ddot \b_\af}&(I^\pf \ddot\DS_\bu)_{\af-1}=\DS_{\af-1}\ar[r]^{\hspace{40pt}I^\pf\b_{\af-1}=\b_{\af-1}^{\DS_\bu}}&\dots}.
\]

When $q=\af-1$, $\ker I^\pf \b_{\af-1}=\ker \b^{\DS_\bu}_{\af-1}$, while $\Im I^\pf \b_{\af} = Z_{\af-1} +\Im \b_{\af}^{\DS_\bu}$, and hence $H_{\af-1}(I^\pf \ddot \CS_\bu)= \Im p_{*,\af-1}\leq H_{\af-1}(\DS_\bu,\CS_\bu)$. 

If $q=\af$, then, for $x\oplus y\in  Z_{\af-1}\oplus \DS_{\af}$, we have  
\[
I^\pf \b_\af(x\oplus y)=\b_{\af-1}^{\DS_\bu}(x)\oplus \bar j_{\af-1}(x)-\b^{\DS_\bu}_{\af}(y)=0\oplus \bar j_{\af-1}(x)-\b^{\DS_\bu}_{\af}(y),
\]
and this is zero if  $ \b^{\DS_\bu}_{\af}(y) = x$. Since
\[
I^\pf \b_{\af+1}(a\oplus b)=\b_{\af}(a)\oplus \bar j_{\af}(a)-\b^{\DS_\bu}_{\af+1}(b),
\]
$I^\pf \b_{\af+1}=Z_{\af-1}+\b^{\DS_\bu}_{\af+1}(\DS_{\af+1})$. It follows that $H_{\af}(I^\pf (C(\CS_\bu)\sqcup_{i_\bu} \DS_\bu)) = H_{\af-1}(\DS_\bu,\CS_\bu)$.
\end{proof}

\begin{lem}\label{funct1}  The construction is functorial in the category of the finite chain complexes of free $R$-modules: namely, if $\varphi_\bu:(\DS_\bu,\CS_\bu)\to (\DS'_\bu,\CS_\bu')$ is a chain map of pairs, we have a chain map $I^\pf \ddot\varphi_\bu:I^\pf (C(\CS_\bu)\sqcup_{i_\bu} \DS_\bu)_\bu\to I^\pf (C(\CS'_\bu)\sqcup_{i'_\bu} \DS'_\bu)_\bu$, where $\ddot \varphi_\bu$ is the push out of $\varphi$ and the cone of its restriction to $\CS_\bu$.
\end{lem}

\begin{lem}\label{exact1} The functor $I^\pf $ is exact. 
\end{lem}
\begin{proof} This follows by the description of the intersection chain complex given in Lemma \ref{l4.1}.
\end{proof}

\begin{rem}\label{intbasmapp} Suppose that   $\DS_\bu$, and $\CS_\bu$ are based with preferred bases $\ds_\bu$ and  $\cs_\bu$. Let $\ds''_\bu$ be a consistent  preferred basis for $(\DS_\bu,\CS_\bu)$. Then,  $I^\pf (C(\CS_\bu)\sqcup_{i_\bu} \DS_\bu)_\bu$ is free and a preferred basis is determined   for all $q\not=\af$: i.e. a preferred basis of $R$-modules for the intersection complex is (see Section \ref{bb11} for the notation and the maps):
\begin{align*}
 &I^\pf (C(\CS_\bu)\sqcup_{i_\bu} \DS_\bu)_q:&I^\pf \ddot\cs_q&= \ds_q, \hspace{10pt}q\leq \af-1,\\
  &I^\pf(C(\CS_\bu)\sqcup_{i_\bu} \DS_\bu)_\af:&I^\pf \ddot\cs_{\af-1}&=  \zs_{\af-1}\oplus 0\ss0\oplus \cs_\af, \\
 &I^\pf (C(\CS_\bu)\sqcup_{i_\bu} \DS_\bu)_q:&I^\pf \ddot\cs_q&=\cs_{q-1}\oplus 0\ss 0\oplus \ds_q
, \hspace{10pt}q\geq \af+1.
\end{align*}

Next, assuming that $H_\bu(\DS_\bu)$ is free with preferred graded basis $\hs_\bu^{\DS_\bu}$, and $H_\bu(\DS_\bu,\CS_\bu)$ is free with preferred basis $\hs_\bu''$,  it follows that $ H_\bu (I^\pf (C(\CS_\bu)\sqcup_{i_\bu} \DS_\bu)_\bu)$ is also free and has the preferred graded basis $I^\pf \ddot\hs_\bu$ defined by
\begin{align*}
I^\pf\ddot\hs_q&=\hs^{\DS_\bu}_q, ~q\leq \af-2, & I^\pf\ddot\hs_q&=p_{*,\af-1}(\hs^{\DS_\bu}),
&I^\pf\ddot\hs_q&=\hs_\bu'', ~q\geq\af.
\end{align*}

This assumption also guarantee that $Z_{\af-1}$ is free (stably free) and therefore so is $ I^\pf(C(\CS_\bu)\sqcup_{i_\bu} \DS_\bu)_\af$.

\end{rem}

\subsection{Subdivisions}

Let $\CS_\bu$ be a chain complex. A subdivision of $\CS_\bu$ is a second chain complex $\CS_\bu'$ with a chain quasi isomorphism $i_\bu:\CS_\bu\to \CS_\bu'$. Relative subdivision is defined analogously.

\begin{prop}\label{p4.11a} Let  Let $i_\bu:\CS_\bu\to \CS'_\bu$ be a subdivision of $\CS_\bu$. Then, $C(i_\bu):C(\CS_\bu)\to  C(\CS'_\bu)$ and $I^\pf C(i_\bu):I^\pf C(\CS_\bu)\to I^\pf C(\CS'_\bu)$ are subdivisions. 
\end{prop}
\begin{proof} The existence of $I^\pf C(i_\bu)$ follows by functoriality, see Lemma \ref{funct}. We just have to prove that $I^\pf C(i_\bu)$ induces an isomorphism in homology. We have the commutative diagram

\[
\xymatrix{ \dot \CS_{m+1} \ar[r]^{\dot \b_{m+1}}& \dots \ar[r]^{\dot \b_{\af+2}}&\dot \CS_{\af+1}\ar[r]^{\dot \b_{\af+1}}\ar[d]^{C(i_{\af+1})}&Z_{\af-1}\oplus \CS_{a}\ar[r]^{\dot \b_\af}\ar[d]^{i_{\af-1}\oplus i_{\af}}& \CS_{\af-1}\ar[r]^{\b_{\af-1}}&\dots\ar[r]^{\b_1}&\CS_0\\
\dot \CS_{m+1}' \ar[r]_{\dot \b_{m+1}'}& \dots \ar[r]_{\dot \b'_{\af+2}}&\dot \CS'_{\af+1}\ar[r]_{\dot \b'_{\af+1}}&Z'_{\af-1}\oplus \CS'_{a}\ar[r]_{\dot \b'_\af}& \CS'_{\af-1}\ar[r]_{\b'_{\af-1}}&\dots\ar[r]_{\b_1}&\CS'_0
}
\]

Since the intersection homology is trivial in all dimensions $q>\af-2$, and $I^\pf C(i_q)=i_q$ for $q\leq \af-2$, the thesis follows. \end{proof}

\begin{prop}\label{p4.11b}   Let $j_\bu:(\DS_\bu,\CS_\bu)\to (\DS'_\bu,\CS'_\bu)$ be a relative subdivision of $(\DS_\bu,\CS_\bu)$. Then, $\ddot j_\bu=(C(j_\bu|_{\CS_\bu})\sqcup j_\bu)_\bu:(C(\CS_\bu)\sqcup_{i_\bu} \DS_\bu)_\bu\to (C(\CS'_\bu)\sqcup_{i'_\bu} \DS'_\bu)_\bu$ and  $I^\pf (C(j_\bu|_{\CS_\bu})\sqcup j_\bu)_\bu:I^\pf (C(\CS_\bu)\sqcup_{i_\bu} \DS_\bu)_\bu\to I^\pf (C(\CS'_\bu)\sqcup_{i'_\bu} \DS'_\bu)_\bu$ are subdivisions. 
\end{prop}
\begin{proof} The existence of $I^\pf\ddot j_\bu$ follows by functoriality, see Lemma \ref{funct1}. We just have to prove that $I^\pf \ddot j_\bu$ induces an isomorphism in homology. We have the commutative diagram

\[
\xymatrix{ \ddot \CS_{m+1} \ar[r]^{\ddot \b_{m+1}}& \dots \ar[r]^{\ddot \b_{\af+2}}&\ddot \CS_{\af+1}=\CS_\af\oplus \DS_{\af+1}\ar[r]^{\ddot \b_{\af+1}}\ar[d]^{\ddot j_{\af+1}}&Z_{\af-1}\oplus \DS_{a}\ar[r]^{\ddot \b_\af}\ar[d]^{j_{\af-1}\oplus j_{\af}}& \DS_{\af-1}\ar[r]^{\b_{\af-1}}&\dots\ar[r]^{\b_1}&\DS_0\\
\ddot \CS_{m+1}' \ar[r]_{\ddot \b_{m+1}'}& \dots \ar[r]_{\ddot \b'_{\af+2}}&\ddot \CS'_{\af+1}=\CS'_\af\oplus \DS'_{\af+1}\ar[r]_{\ddot \b'_{\af+1}}&Z'_{\af-1}\oplus \DS'_{a}\ar[r]_{\ddot \b'_\af}& \DS'_{\af-1}\ar[r]_{\b'_{\af-1}}&\dots\ar[r]_{\b_1}&\DS'_0
}
\]

Since  $I^\pf \ddot j_q=j_q$ for $q\leq \af-2$, and  $I^\pf \ddot j_q=\ddot j_q=j_{q-1}\oplus j_q$ for $q\geq \af-1$, it is clear that it induces isomorphism in homology  \end{proof}

\subsection{Explicit homology formulas for the middle perversity}

First, we fix the value of some constants. Recall that, by definition,
\begin{align*}
\tf_q&=q-2,& \mf_q&=\left[\frac{q}{2}\right]-1,& \mf_q+\mf^c_q&=\tf_q,&\af&=n-\pf_n,
\end{align*}
and $n=\dim C(\CS_\bu)=m+1=\dim \DS_\bu$, where $m=\dim \CS_\bu$.

\subsubsection{Cone, odd section case: $m=2p-1$, $p\geq 1$}

If $m=2p-1$ is odd, $p\geq1$, then:
\begin{align*}
 n&=2p,&  \mf_{n=2p}&=p-1, & \mf^c_{n=2p}&=\tf_{2p}-\mf_{2p}=p-1=\mf_{n=2p},&\af&=n-\mf_{n=2p}=p+1,
\end{align*}
and hence by Lemmas \ref{homo} and \ref{l4.4} (compare \cite[Section 4.7]{KW})
\begin{align*}
H_q(I^\mf C(\CS_\bu))&=\left\{\begin{array}{ll} H_q(\CS_\bu),& 0\leq q\leq p-1,\\0,&p\leq q\leq 2p,\end{array}\right.\\
H_q(I^\mf C(\CS_\bu),\CS_\bu)&=\left\{\begin{array}{ll} 0,& 0\leq q\leq p,\\
 H_{q-1}(\CS_\bu),&p+1\leq q\leq 2p.\end{array}\right.
\end{align*}

Since $H_q(\CS_\bu)=H^\da_{m-q}(\CS_\bu)=H^\da_{2p-q-1}(\CS_\bu)$, we verify that
\[
H_q(I^\mf C(\CS_\bu))=H^\da_{n-q}(I^{\mf^c} C(\CS_\bu),\CS_\bu)=H^\da_{2p-q}(I^{\mf^c} C(\CS_\bu),\CS_\bu).
\]

In particular
\begin{align*}
H_p(I^\mf C(\CS_\bu))&=0,\\
H_{p-1}(I^\mf C(\CS_\bu))&=H^\da_{p+1}(I^{\mf^c} C(\CS_\bu),\CS_\bu)=H_{p-1}(\CS_\bu).
\end{align*}

\subsubsection{Cone, even section case: $m=2p$, $p\geq 1$}

If $m=2p$ is even, $p>1$, then:
\begin{align*}
n&=2p+1,&  \mf_{n=2p+1}&=p-1, & \af&=n-\mf_{n=2p+1}=p+2;\\
n&=2p+1, & \mf^c_{n=2p+1}&=\tf_{2p+1}-\mf_{2p+1}=p,& \af^c&=n-\mf^c_{n=2p+1}=p+1.
\end{align*}
and hence by Lemmas \ref{homo} and \ref{l4.4} (compare \cite[Section 4.7]{KW})
\begin{align*}
H_q(I^\mf C(\CS_\bu))&=\left\{\begin{array}{ll} H_q(\CS_\bu),& 0\leq q\leq p,\\0,&p+1\leq q\leq 2p+1,\end{array}\right.\\
H_q(I^\mf C(\CS_\bu),\CS_\bu)&=\left\{\begin{array}{ll} 0,& 0\leq q\leq p+1,\\
 H_{q-1}(\CS_\bu),&p+2\leq q\leq 2p+1.\end{array}\right.
\end{align*}
while
\begin{align*}
H_q(I^{\mf^c} C(\CS_\bu))&=\left\{\begin{array}{ll} H_q(\CS_\bu),& 0\leq q\leq p-1,\\0,&p\leq q\leq 2p+1,\end{array}\right.\\
H_q(I^{\mf^c} C(\CS_\bu),\CS_\bu)&=\left\{\begin{array}{ll} 0,& 0\leq q\leq p,\\
 H_{q-1}(\CS_\bu),&p+1\leq q\leq 2p+1.\end{array}\right.
\end{align*}

Since $H_q(\CS_\bu)=H^\da_{m-q}(\CS_\bu)=H^\da_{2p-q}(\CS_\bu)$, we verify that
\[
H_q(I^\mf C(\CS_\bu))=H^\da_{n-q}(I^{\mf^c} C(\CS_\bu),\CS_\bu)=H^\da_{2p+1-q}(I^{\mf^c} C(\CS_\bu),\CS_\bu),
\]
in particular
\[
H_p(I^\mf C(\CS_\bu))=H^\da_{p+1}(I^{\mf^c} C(\CS_\bu),\CS_\bu)=H_p(\CS_\bu).
\]

\subsubsection{Mapping cone, even dimensional case: $n=2p$, $p\geq 1$}

If $n=2p$, $m=2p-1$ is odd, $p\geq1$, then:
\begin{align*}
 n&=2p,&  \mf_{n=2p}&=p-1, & \mf^c_{n=2p}&=\tf_{2p}-\mf_{2p}=p-1=\mf_{n=2p},&\af&=n-\mf_{n=2p}=p+1,
\end{align*}
and hence by Lemmas \ref{homoi}  (compare \cite[Proposition 4.4.1]{KW})
\[
H_q(I^\mf (C(\CS_\bu)\sqcup_{i_\bu} \DS_\bu))=\left\{\begin{array}{ll} H_q(\DS_\bu),& 0\leq q\leq p-1,\\
p_{*,p}(H_{p}(\DS_\bu))\leq H_{p}(\DS_\bu/\CS_\bu),\\
H_q(\DS_\bu,\CS_\bu),&p+1\leq q\leq 2p.\end{array}\right.
\]

Since $H_q(\DS_\bu)=H_{2p-q-1}(\DS_\bu,\CS_\bu)$, it follows that
\begin{align*}
H_q(I^\mf \ddot\CS_\bu)=&\left\{\begin{array}{ll} H_q(\DS_\bu)=H_{2p-q-1}(\DS_\bu,\CS_\bu),& 0\leq q\leq p-1,\\
p_{*,p}(H_{p}(\DS_\bu))\leq H_{p}(\DS_\bu/\CS_\bu),\\
H_q(\DS_\bu,\CS_\bu)=H_{2p-q-1}(\DS_\bu),&p+1\leq q\leq 2p.\end{array}\right.\\
=&\left\{\begin{array}{ll} H_q(\DS_\bu,\CS_\bu),& p+1\leq q\leq 2p,\\
p_{*,p}(H_{p}(\DS_\bu))\leq H_{p}(\DS_\bu/\CS_\bu),\\
H_q(\DS_\bu),&0\leq q\leq p-1,\end{array}\right.\\
=&H_{n-q}(I^\mf \ddot\CS_\bu)
\end{align*}

\subsubsection{Mapping cone, odd dimensional: $n=2p+1$, $p\geq 1$}

If $n=2p+1$, $m=2p$ is even, $p\geq 1$, then:
\begin{align*}
n&=2p+1,&  \mf_{n=2p+1}&=p-1, & \af&=n-\mf_{n=2p+1}=p+2;\\
n&=2p+1, & \mf^c_{n=2p+1}&=\tf_{2p+1}-\mf_{2p+1}=p,& \af^c&=n-\mf^c_{n=2p+1}=p+1.
\end{align*}
and hence by Lemmas \ref{homoi} (compare \cite[Proposition 4.4.1]{KW})
\[
H_q(I^\mf (C(\CS_\bu)\sqcup_{i_\bu} \DS_\bu))=\left\{\begin{array}{ll} H_q(\DS_\bu),& 0\leq q\leq p,\\
p_{*,p+1}(H_{p+1}(\DS_\bu))\leq H_{p+1}(\DS_\bu/\CS_\bu),\\
H_q(\DS_\bu,\CS_\bu),&p+2\leq q\leq 2p+1,\end{array}\right.
\]
and
\[
H_q(I^{\mf^c} (C(\CS_\bu)\sqcup_{i_\bu} \DS_\bu))=\left\{\begin{array}{ll} H_q(\DS_\bu),& 0\leq q\leq p-1,\\
p_{*,p}(H_{p}(\DS_\bu))\leq H_{p}(\DS_\bu/\CS_\bu),\\
H_q(\DS_\bu,\CS_\bu),&p+1\leq q\leq 2p.\end{array}\right.
\]

\section{Whitehead torsion for the algebraic   intersection chain complexes}

Let $(\CS_\bu,\cs_\bu)$ be a based finite complex of free left $R$ modules with graded basis $\cs_\bu$. Assume the homology graded module $H_\bu(\CS_\bu,\cs_\bu)$ is free with graded basis $\hs_\bu$.

\subsection{The cone (absolute)}

Recall the complex in question is
\[
\xymatrix{ I^\pf C(\CS_\bu)_\bu:&\hspace{-20pt}\dot \CS_{m+1} \ar[r]^{\dot \b_{m+1}}& \dots \ar[r]^{\dot \b_{\af+2}}&\dot \CS_{\af+1}\ar[r]^{\dot \b_{\af+1}}&Z_{\af-1}\oplus \CS_{a}\ar[r]^{\dot \b_\af}& \CS_{\af-1}\ar[r]^{\b_{\af-1}}&\dots\ar[r]^{\b_1}&\CS_0.}
\]
with boundary
\[
\begin{aligned} I^\pf \dot\b_q&= \dot \b_q, &q&\geq \af,&I^\pf \dot\b_q&=\b_q, &q&\leq \af-1.
\end{aligned}
\]

By Lemma \ref{homo}, the intersection homology is 
\[
H_q(I^\pf C(\CS_\bu))=\left\{\begin{array}{ll} H_q(\CS_\bu),& 0\leq q\leq \af-2,\\0,&\af-1\leq q\leq m+1.\end{array}\right.
\]
so it is free and we denote by $I^\pf\dot\ks_\bu$ any graded basis, see Definition \ref{intbasmapp}.

We will use the notation $I^\pf \dot\bs_q$ to denote a set of independent vectors in $I^\pf C(\CS_\bu)_q$ with non trivial boundaries generating the image of the boundary operator.

For $q<\af-1$, a new chain basis is 
\[
I^\pf \dot\b_{q+1}(I^\pf \dot\bs_{q+1}) \hat{I^\pf  \dot\ks}_q I^\pf \dot\bs_q=\b_{q+1}( \bs_{q+1}) \hat{\ks}_q  \bs_q,
\]
and hence
\[
(I^\pf \dot\b_{q+1}(I^\pf \dot\bs_{q+1}) \hat{I^\pf  \dot\ks}_q I^\pf \dot\bs_q/I^\pf \dot\cs_q)=(\b_{q+1}( \bs_{q+1}) \hat{I^\pf  \dot\ks}_q  \bs_q/ \cs_q)=(\hat{I^\pf  \dot\ks}_q/\hat{  \hs}_q)(\b_{q+1}( \bs_{q+1}) \hat{  \hs}_q  \bs_q/ \cs_q).
\]

For $q>\af+1$ the intersection homology is trivial, however we may use the homology basis of the cone, and then a new chain basis is  
\[
I^\pf \dot\b_{q+1}(I^\pf \dot\bs_{q+1}) \hat{I^\pf  \dot\ks}_q I^\pf \dot\bs_q
=\dot\b_{q+1}(\dot \bs_{q+1}) \hat{I^\pf  \dot\hs}_q \dot \bs_q 
= 
\b_{q}( \bs_{q}) \hat{\hs}_{q-1} \bs_{q-1}\oplus 0\ss 0\oplus \b_{q+1}(\bs_{q+1}) \hat{\hs}_q  \bs_q,
\]
and hence 
\begin{align*}
(I^\pf \dot\b_{q+1}(I^\pf \dot\bs_{q+1}) \hat{I^\pf  \dot\ks}_q I^\pf \dot\bs_q/I^\pf \dot\cs_q)
&=(\dot\b_{q+1}(\dot \bs_{q+1}) \hat{I^\pf  \dot\hs}_q \dot \bs_q/\dot \cs_q) \\
&= (\b_{q+1}(\bs_{q+1}) \hat{\hs}_q  \bs_q/ \cs_q)
(\b_{q}( \bs_{q}) \hat{\hs}_{q-1} \bs_{q-1}/\cs_{q-1}).
\end{align*}

To deal with the remaining cases the relevant part of the complex is
\[
\xymatrix@C=40pt{ \dots \ar[r]^{\hspace{-50pt}I^\pf \dot\b_{\af+1}=\dot \b_{\af+1}}& I^\pf C(\CS_\bu)_{\af}= Z_{\af-1} \oplus \CS_{\af}\ar[r]^{I^\pf\b_\af=\dot \b_\af}&I^\pf C(\CS_\bu)_{\af-1}=\CS_{\af-1}\ar[r]^{\hspace{40pt}I^\pf\b_{\af-1}=\b_{\af-1}}&\dots}.
\]

It is clear that the case $q=\af+1$ is exactly the same as the cases with $q>\af+1$, since the image of the boundary is always contained in the kernel of the next one.

At $q=\af$, the intersection homology is trivial by Lemma \ref{homo}. Since th eimage of the boundary operator is contained in the kernel of the following one, the choice of the $I^{\pf}\dot\bs_{\af+1}$ made above is fine. Since for each $ z\oplus x\in Z_{\af-1}\oplus \CS_\af$, $\dot\b_\af(z\oplus x)=0\oplus z-\b_\af(x)$,  we can choose as set $I^\pf \dot\bs_\af= \hat {\hs}_{\af-1}\oplus 0\ss 0\oplus \bs_\af$, with image (up to sign)
\[
I^\pf \dot\b_\af(I^\pf \dot\bs_\af)=\dot \b_\af(0\oplus \bs_\af)\dot \b_\af(\hat{\hs}_{\af-1}\oplus 0)=\b_\af(\bs_{\af})\hat {\hs}_{\af-1}.
\]

Thus, the new basis of $I^\pf C(\CS_\bu)_{\af}$  is
\begin{align*}
I^\pf \dot\b_{\af+1}(I^\pf \dot\bs_{\af+1}) I^\pf \dot\bs_\af&=0\oplus  \hat{\hs}_{\af}\ss\b_{\af}(\bs_{\af})\oplus \bs_{\af}\ss \hat {\hs}_{\af-1}\oplus 0\ss 0\oplus \b_{\af+1}(\bs_{\af+1})\ss 0\oplus \bs_\af\\
&=\b_{\af}(\bs_{\af})\hat {\hs}_{\af-1}\oplus 0\ss 0\oplus \b_{\af+1}(\bs_{\af+1}) \hat{\hs}_{\af}\bs_\af,
\end{align*}
and the matrix of the change of basis is
\begin{align*}
(I^\pf \dot\b_{\af+1}(I^\pf \dot\bs_{\af+1}) \hat{I^\pf  \dot\ks}_\af I^\pf \dot\bs_\af/I^\pf \dot\cs_\af)
=&(\b_{\af}(\bs_{\af})\hs_{\af-1}\oplus 0\ss 0\oplus \b_{\af+1}(\bs_{\af+1}) \hat{\hs}_\af\bs_{\af}/\zs_{\af-1}\oplus 0\ss 0\oplus \cs_\af)\\
=&(\b_{\af}(\bs_{\af})\hat {\hs}_{\af-1}/\zs_{\af-1})( \b_{\af+1}(\bs_{\af+1}) \hat{\hs}_{\af}\bs_{\af}/ \cs_\af).
\end{align*}

In the case $q=\af-1$ we have that $I^\pf \dot\bs_{\af-1} = \bs_{\af-1}$, with image
\[
I^\pf \dot\b_{\af-1}(I^\pf \dot\bs_{\af-1})=\b_\af(\bs_{\af-1}).
\]

Since $\Im I^\pf\dot\b_{\af}=Z_{\af-1}$ then $I^{p}H_{\af-1}(C(\CS_{\bu})) = 0$, so the new basis of $I^\pf C(\CS_\bu)_{\af-1}$  is
\begin{align*}
I^\pf \dot\b_{\af}(I^\pf \dot\bs_{\af}) I^\pf \dot\bs_{\af-1}&=\b_{\af}(\bs_{\af})\hat {\hs}_{\af-1}\bs_{\af-1},
\end{align*}
and the matrix of the change of basis
\[
(I^\pf \dot\b_{\af}(I^\pf \dot\bs_{\af}) \hat{I^\pf  \dot\ks}_{\af-1} I^\pf \dot\bs_{\af-1}/I^\pf \dot\cs_{\af-1})=(\b_{\af}(\bs_{\af})\hat {\hs}_{\af-1}\bs_{\af-1}/\cs_{\af-1}).
\]

The torsion of the intersection chain complex of the cone  with perversity $\pf$ is
\begin{align*}
\tau_W (I^{\pf}C(\CS_{\bu});I^\pf \dot\cs, I^\pf\dot\ks)&= \sum_{q=0}^{m+1}(-1)^q [(I^\pf \dot\b_{q+1}(I^\pf \dot\bs_{q+1}) \hat{I^\pf \dot \ks}_q I^\pf \dot\bs_q/I^\pf \dot\cs_q)] \\
=& \sum_{q=0}^{\af-2}(-1)^q ([(\hat{\ks}_q/\hat{  \hs}_q)]+[(\b_{q+1}( \bs_{q+1}) \hat{  \ks}_q  \bs_q/ \cs_q)]) \\
&+(-1)^{\af-1}[(\b_{\af}(\bs_{\af})\hat \hs_{\af-1}\bs_{\af-1}/\cs_{\af-1})]\\
&+(-1)^\af[(\b_{\af}(\bs_{\af})\hat \hs_{\af-1}/\zs_{\af-1})]+(-1)^\af[( \b_{\af+1}(\bs_{\af+1}) \hat\hs_{\af}\bs_{\af}/ \cs_\af)]\\
&+\sum_{q=\af+1}^{m+1}(-1)^q ([(\b_{q+1}(\bs_{q+1}) \hat{\hs}_q  \bs_q/ \cs_q)]
+[(\b_{q}( \bs_{q}) \hat{\hs}_{q-1} \bs_{q-1}/\cs_{q-1})])\\
=& \sum_{q=0}^{\af-2}(-1)^q [(\hat{\ks}_q/\hat{  \hs}_q)]+\sum_{q=0}^{\af-2}(-1)^q[(\b_{q+1}( \bs_{q+1}) \hat{  \hs}_q  \bs_q/ \cs_q)]\\
&+(-1)^{\af-1}[(\b_{\af}(\bs_{\af})\hat \hs_{\af-1}\bs_{\af-1}/\cs_{\af-1})]\\
&+(-1)^\af[(\b_{\af}(\bs_{\af})\hat \hs_{\af-1}/\zs_{\af-1})].
\end{align*}

Observing that
\[
[(\b_{\af}(\bs_{\af})\hat \hs_{\af-1}/\zs_{\af-1})]
=[(\b_{\af}(\bs_{\af})\hat \hs_{\af-1}\bs_{\af-1}/\cs_{\af-1})]
-[(\zs_{\af-1}\bs_{\af-1}/\cs_{\af-1})],
\]
we have the alternative formula given in the proposition. We have proved the following result.

\begin{prop}\label{abstor} Let $I^\pf \dot\ks_\bu$ any graded homology basis for $H_\bu(I^\pf C(\CS_\bu))$, then the torsion of the intersection chain complex of the cone  of $\CS_\bu$ with perversity $\pf$ is
\begin{align*}
\tau_W (I^{\pf}C(\CS_{\bu})_\bu;I^\pf \dot\cs_\bu, I^\pf\dot\hs_\bu)=& \sum_{q=0}^{\af-2}(-1)^q [(\hat{\ks}_q/\hat{  \hs}_q)]
+\sum_{q=0}^{\af-2}(-1)^q[(\b_{q+1}( \bs_{q+1}) \hat{  \hs}_q  \bs_q/ \cs_q)]\\
&+(-1)^\af[(\zs_{\af-1}\bs_{\af-1}/\cs_{\af-1})]\\
=& \sum_{q=0}^{\af-2}(-1)^q [(\hat{\ks}_q/\hat{  \hs}_q)]
+\sum_{q=0}^{\af-1}(-1)^q[(\b_{q+1}( \bs_{q+1}) \hat{  \hs}_q  \bs_q/ \cs_q)]\\
&+(-1)^{\af-1}[(\b_{\af}(\bs_{\af})\hat \hs_{\af-1}/\zs_{\af-1})].
\end{align*}
\end{prop}

It is clear that the $\tau_W (I^{\pf}C(\CS_{\bu});I^\pf \dot\cs, I^\pf\dot\hs)$ coincides with the torsion of the following complex
\[
\xymatrix{ \BCS_\bu:&\hspace{-20pt}0 \ar[r]& \ker(\b_{\af-1}) \ar[r]&\CS_{\af-1}\ar[r]^{ \b_{\af-1}}&\dots\ar[r]
& \CS_q \ar[r]^{\b_{q}}&\dots\ar[r]^{\b_1}&0,}
\]
with basis $\Bcs_{\af}=\zs_{\af-1}$, $\Bcs_q=\cs_q$, $0\leq q\leq \af-1$, $\Bhs_q=\hs_q$, $0\leq q\leq \af-1$. Moreover, the natural inclusion $\iota_\bu:\BCS_\bu \to I^\pf C(\CS_\bu)$ is a quasi isomorphism.

\begin{defi} Let $\vv:\CS_\bu\to \DS_\bu$ be a quasi isomorphism of finite dimensional based chain complexes. If $\tau_W(\DS_\bu;\vv(\cs_\bu), \vv(\hs_\bu))=\tau_W(\CS_\bu;\cs_\bu, \hs_\bu)$, then we say that $\vv$ is a simple quasi isomorphism.
\end{defi} 

\begin{prop} The chain map $\iota_\bu:\BCS_\bu \to I^\pf C(\CS_\bu)$ is a simple quasi isomorphism of based chain complexes.
\end{prop}
\begin{proof} It is clear that $\iota_\bu$ is a quasi isomorphism. Now, the calculations above show that the torsion of $\tau_W (I^{\pf}C(\CS_{\bu});I^\pf \dot\cs, I^\pf\dot\hs)$ only depends on the basis in dimensions $q<\af$. In dimensions $q<\af-1$, $\iota_q(\Bcs_q)=\cs_q$ and $\iota_q(\Bhs_q)=\hs_q$, by the very definition. Thus, the torsion 
\end{proof}

\subsection{The cone (relative)}

Consider  the relative intersection chain complex $I^\pf (C(\CS_\bu),\CS_\bu)$. 
As previously, we assume that the homology of $I^\pf C(\CS_\bu)$ is free. Recall the preferred bases introduced in Definition \ref{relative intersection basis}. We have the splitting 
exact sequence  
\beq\label{eqq}
\xymatrix{
0 \ar[r] & \CS_\bu \ar[r]^{\iota_\bu} &I^\pf C(\CS_\bu)\ar[r] &(I^\pf
C(\CS_\bu),\CS_\bu)_\bu=I^\pf C(\CS_\bu)/\CS_\bu \ar[r]\ar@/^10pt/[l]^{ \xi_\bu} & 0,
}
\eeq
and the homology of $I^{\pf} (C(\CS_\bu),\CS_\bu)$ is free with a preferred basis
that we denote by $I^\pf\hs_{{\rm rel},\bu}$, with
\begin{align*}
I^\pf \hs_{{\rm rel},q}&= \emptyset , ~q\leq \af-1,&I^\pf \hs_{{\rm rel},q}=\hs_q,~q\geq \af.
\end{align*}

The calculation of the torsion splits into two parts. Recall the complex in question is $(I^\pf C(\dot \CS_\bu),\CS_\bu)_\bu:$
\[
\xymatrix{ 
m+1&&\af+1&\af&\af-1&&0\\
\CS_{m} \ar[r]^{\b_{m}}& \dots \ar[r]^{ \b_{\af+1}}&\CS_{\af}\ar[r]^{ \b_{\af}}&
Z_{\af-1}\ar[r]^{0}& 0 \ar[r]^{0}&\dots\ar[r]^{0}&0.}
\]

For $q<\af-1$, there is nothing to compute, 
while for $q\geq\af+1$ 
\[
(I^\pf \b_{q+1}(I^\pf \bs_{q+1}) \hat{I^\pf  \hs}_{{\rm rel},q} I^\pf \bs_q/I^\pf
\cs_{{\rm rel},q})=(\b_{q}( \bs_{q}) \hat{\hs}_{q-1} \bs_{q-1}/\cs_{q-1}).
\]

To deal with the case $q=\af$ the relevant part of the complex is
\[
\xymatrix@C=40pt{ \dots \ar[r]^{\hspace{-50pt}I^\pf \b_{\af+1}=  \b_{\af}}& I^\pf
(C(\CS_\bu),\CS_\bu)_{\af}= Z_{\af-1} \ar[r]^{I^\pf\b_\af= \b_{\af-1}}&I^\pf
(C(\CS_\bu),\CS)_{\af-1}=0\ar[r]^{\hspace{40pt}I^\pf\b_{\af-1}=0}&\dots}.
\]

At $q=\af$, the relative intersection homology is the homology of $\CS_\bu$ by Lemma
\ref{l4.4}. Since $I^\pf \b_{\af+1}=  \b_{\af}$. 
We have that $I^\pf \bs_{\af-1} = \emptyset$, so
\begin{align*}
(I^\pf \b_{\af+1}(I^\pf \bs_{\af+1}) \hat{I^\pf  \hs}_{{\rm rel},\af} I^\pf \bs_{\af}/I^\pf
\cs_{{\rm rel},\af})&=(\b_\af(\bs_{\af})\hat \hs_{\af-1}/\zs_{\af-1})\\
& = (\b_\af(\bs_{\af})\hat
\hs_{\af-1}\bs_{\af-1}/\cs_{\af-1})  (\zs_{\af-1}\bs_{\af-1}/\cs_{\af-1})^{-1}.
\end{align*}

Then, direct substitution in the definition
\begin{align*}
\tau_W ((I^{\pf}C(\CS_{\bu})_\bu,\CS_\bu);I^\pf C(\cs_\bu)_{{\rm rel},\bu}, I^\pf\ks_{{\rm rel},\bu})
=& \sum_{q=0}^{m+1}(-1)^q [(I^\pf
\b_{q+1}(I^\pf \bs_{q+1}) \hat{I^\pf  \hs}_{{\rm rel},q} I^\pf \bs_q/I^\pf \cs_{{\rm rel},q})],
\end{align*}
gives the following result.

\begin{prop}\label{reltor} Let $I^\pf \ks_{{\rm rel},\bu}$ any graded homology basis for $H_\bu(I^\pf C(\CS_\bu),\CS_\bu)$, then the torsion of the relative intersection chain complex of the cone of $\CS_\bu$ with perversity
$\pf$ is
\begin{align*}
\tau_W ((I^{\pf}C(\CS_{\bu})_\bu,\CS_\bu);&I^\pf C(\cs_\bu)_{{\rm rel},\bu}, I^\pf\ks_{{\rm rel},\bu})\\
=& \sum_{q=\af-1}^{m}(-1)^{q+1} \left([(I^\pf  \ks_{{\rm rel},q+1}/ I^\pf \hs_{{\rm rel},q+1})]+[(\b_{q+1}(
\bs_{q+1}) \hat{  \hs}_q  \bs_q/\cs_q)]\right)\\
&+(-1)^{\af-1}[(\zs_{\af-1}\bs_{\af-1}/\cs_{\af-1})]\\
=& \sum_{q=\af-1}^{m}(-1)^{q+1} [(I^\pf  \ks_{{\rm rel},q+1}/ I^\pf \hs_{{\rm rel},q+1})]
+(-1)^{\af}[(\b_\af(\bs_{\af})\hat \hs_{\af-1}/\zs_{\af-1})]\\
&+ \sum_{q=\af}^{m}(-1)^{q+1}[(\b_{q+1}(\bs_{q+1}) \hat{  \hs}_q  \bs_q/\cs_q)]).
\end{align*}
\end{prop}

Observe that $\tau_W ((I^{\pf}C(\CS_{\bu}),\CS_\bu);I^\pf \cs, I^\pf\hs)$ coincides with the
torsion of the following complex $\BCS_{\rm rel,\bu}:$
\[
\xymatrix{ 
m+1&m-1&&\af&\af-1&\af-2&&0\\
 \CS_m
\ar[r]&\CS_{m-1}\ar[r]^{ \b_{m-1}}& \cdots \ar[r]^{\b_{\af}}
& \CS_{\af-1} \ar[r]^{\b_{\af-1}}& \Im \b_{\af-1}\ar[r]&0\ar[r]&\dots\ar[r]&0,}
\]
with chain basis $\Bcs_{{\rm rel},\af-1}=\xs_{\af-1}$, 
$\Bcs_{{\rm rel},q}=\cs_q$,
$\af-1\leq q\leq m$, $\Bhs_{{\rm rel},q}=\hs_q$, $\af\leq q\leq m$.

\begin{prop} The chain map $\iota_\bu:\BCS_{{\rm rel},\bu} \to (I^\pf C(\CS_\bu),\CS_\bu)$ is a simple quasi isomorphism of based chain complexes.
\end{prop}
\begin{proof} We only need to check what happens at dimension $\af-1$. Since torsion does not depends on the choice of the $\bs_q$, we can chose $\bs_\af=\hat \xs$, a lift of the $\xs_{\af-1}$. This proves the thesis.
\end{proof}

\begin{prop} \label{milnorcone}
\begin{align*}
\tau_W (I^{\pf}C(\CS_{\bu})_\bu;I^\pf \dot\cs_\bu, I^\pf\dot\hs_\bu)=\tau_W(\CS_\bu; \cs_\bu,\hs_\bu)+\tau_W ((I^{\pf}C(\CS_{\bu})_\bu,\CS_\bu);&I^\pf C(\cs_\bu)_{{\rm rel},\bu}, I^\pf\ks_{{\rm rel},\bu}). 
\end{align*}
\end{prop}
\begin{proof} It is easy to see, considering the definition of the intersection chain bases, that the exact sequence in equation (\ref{eqq}), is an exact sequence of based complexes
\[
\xymatrix{
0 \ar[r] & \CS_\bu \ar[r]^{\iota_\bu} &I^\pf C(\CS_\bu)\ar[r] &(I^\pf
C(\CS_\bu),\CS_\bu)_\bu=I^\pf C(\CS_\bu)/\CS_\bu \ar[r] & 0,
}
\]
so we may apply Theorem \ref{mil0}. This gives
\begin{align*}
\tau_W (I^{\pf}C(\CS_{\bu})_\bu;I^\pf \dot\cs_\bu, I^\pf\dot\hs_\bu)=&\tau_W(\CS_\bu;\ \cs_\bu,\hs_\bu)
+\tau_W ((I^{\pf}C(\CS_{\bu})_\bu,\CS_\bu);I^\pf C(\cs_\bu)_{{\rm rel},\bu}, I^\pf\ks_{{\rm rel},\bu})\\
&+\tau(I^\pf \dot \Ha), 
\end{align*}
where $ \dot \Ha$ is the long homology exact sequence. Since the homology of $I^{\pf}C(\CS_{\bu})$ is trivial for $q\leq \af-1$, and $I^{\pf}C(\CS_{\bu})$ is trivial for $q\leq \af-1$, it follows that $\tau(I^\pf \dot \Ha)$ is trivial. 
\end{proof}

\subsection {The mapping cone}

We consider the complex
\[
\xymatrix{ I^\pf (\ddot\CS_\bu)_\bu:&\hspace{-20pt}\ddot \CS_{m+1} \ar[r]^{I^\pf\ddot \b_{m+1}}& \dots \ar[r]^{I^\pf\ddot \b_{\af+2}}&\ddot \CS_{\af+1}\ar[r]^{I^\pf\ddot \b_{\af+1}}&Z_{\af-1}\oplus \DS_{a}\ar[r]^{I^\pf\ddot \b_\af}& \DS_{\af-1}\ar[r]^{I^\pf\ddot \b_{\af-1}
}&\dots\ar[r]^{I^\pf\ddot \b_{0}
}&\DS_0.
}
\]
where $Z_{\af-1}=\ker \b_{\af-1}$ in the complex $\CS_\bu$, with boundary
\[
\begin{aligned} I^\pf \ddot\b_q&= \ddot \b_q, &q&\geq \af,\\
I^\pf \ddot\b_q&=\b_q^{\DS_\bu}, &q&\leq \af-1.
\end{aligned}
\]

By Lemma \ref{homoi}, the homology is 
\[
H_q(I^\pf (C(\CS_\bu)\sqcup_{i_\bu} \DS_\bu))=\left\{\begin{array}{ll} H_q(\DS_\bu),& 0\leq q\leq \af-2,\\
p_{*,\af-1}(H_{\af-1}(\DS_\bu))\leq H_{\af-1}(\DS_\bu/\CS_\bu),\\
H_q(\DS_\bu,\CS_\bu),&\af\leq q\leq m+1.\end{array}\right.
\]
so it is free, and we denote by $I^\pf \ddot\hs_\bu$ a graded basis.

We  use the notation $I^\pf \ddot\bs_q$ to denote a set of independent vectors in $I^\pf \ddot\CS_q$ with non trivial boundaries generating the image of the boundary operator.

For $q<\af-1$, a new chain basis is 
\[
I^\pf \ddot\b_{q+1}(I^\pf \ddot\bs_{q+1}) \hat{I^\pf  \ddot\hs}_q I^\pf \ddot\bs_q=\b^{\DS_\bu}_{q+1}( \bs^{\DS_\bu}_{q+1}) \widehat{I^\pf  \ddot\hs}_q  \bs^{\DS_\bu}_q,
\]
and hence
\begin{align*}
(I^\pf \ddot\b_{q+1}(I^\pf \ddot\bs_{q+1}) \widehat{I^\pf  \ddot\hs}_q I^\pf \ddot\bs_q/I^\pf \ddot\cs_q)&=(\b^{\DS_\bu}_{q+1}( \bs^{\DS_\bu}_{q+1}) \widehat{I^\pf  \ddot\hs}_q  \bs^{\DS_\bu}_q/ \ds_q)\\
&=(  \widehat{I^\pf\ddot\hs}_q/\hat\hs^{\DS_\bu}_q)(\b^{\DS_\bu}_{q+1}( \bs^{\DS_\bu}_{q+1}) \hat \hs^{\DS_\bu}_q  \bs^{\DS_\bu}_q/ \ds_q).
\end{align*}

For $q>\al=1$, by the result of Section \ref{bb11},
\begin{align*}
I^\pf \ddot\b_{q+1}(I^\pf \ddot\bs_{q+1}) \widehat{I^\pf  \ddot\hs}_q I^\pf \ddot\bs_q=&\ddot\b_{q+1}(\ddot \bs_{q+1})\ddot \hs_q\ddot\bs_q
=\cs_{q-1}\oplus 0\ss\widehat{I^\pf  \ddot\hs}_q\ss 0\oplus \b''_{q+1}(\bs''_{q+1})\cs_{q}\bs''_q,
\end{align*}
since ${I^\pf  \ddot\hs}_q=\ddot\hs_q=\hs_q''$, with lift $0\oplus \hat \hs_q''$, and 
$I^\pf \ddot \cs_q=\ddot \cs_q=\cs_{q-1}\oplus 0\ss 0\oplus \ds_q=\cs_{q-1}\oplus 0\ss 0\oplus \cs_q\hs_q''$, we get
\[
(I^\pf \ddot\b_{q+1}(I^\pf \ddot\bs_{q+1}) \widehat{I^\pf  \ddot\hs}_q I^\pf \ddot\bs_q/I^\pf \ddot\cs_\bu)
=(  \widehat{I^\pf\ddot\hs}_q/\hat{\ddot\hs}_q)( \b''_{q+1}(\bs''_{q+1})\widehat{\hs''_q}\bs''_q/\ds''_q).
\]

To deal with the remaining cases the relevant part of the complex is
\[
\xymatrix@C=40pt{ \dots \ar[r]^{\hspace{-50pt}I^\pf \ddot\b_{\af+1}=\ddot \b_{\af+1}}& I^\pf \ddot\CS_{\af}= Z_{\af-1} \oplus \DS_{\af}\ar[r]^{I^\pf\ddot\b_\af=\ddot \b_\af}&I^\pf \ddot\CS_{\af-1}=\DS_{\af-1}\ar[r]^{\hspace{40pt}I^\pf \ddot\b_{\af-1}=\b^{\DS_\bu}_{\af-1}}&\dots
}.
\]

It is clear that the case $q=\af+1$ is exactly the same as the cases with $q>\af+1$, since the image of the boundary is always contained in the kernel of the next one.

At $q=\af$, the intersection homology is the homology of the quotient complex. Since for each $ z\oplus x\in Z_{\af-1}\oplus \DS_\af$, $I^\pf \ddot\b_\af(z\oplus x)=\ddot\b_\af(z\oplus x)=0\oplus z-\b^{\DS_\bu}_\af(x)$,  i.e. $z$ is a cycle that is not a boundary,  we can choose as set $I^\pf \ddot\bs_\af$ the set
\[
I^\pf \ddot\bs_\af= \hat \hs_{\af-1}\oplus 0\ss 0\oplus \bs^{\DS_\bu}_\af,
\]
using the fact that $\ds_q=i_q(\cs_q)\hat \ds_q''$ (we will omit $i_q$ in the following, when there is not ambiguity), we may fix a set $\bs_q''$ for the quotient complex and then
\[
I^\pf \ddot\bs_\af= \b_{\af}(\bs_{\af})\hat \hs_{\af-1}\oplus 0\ss 0\oplus \bs''_\af,
\]
with image (up to sign)
\[
I^\pf \ddot\b_\af(I^\pf \ddot\bs_\af)=\ddot \b_\af(0\oplus \bs''_\af)\ddot \b_\af(\b_{\af}(\bs_{\af})\hat\hs_{\af-1}\oplus 0)
=\b_{\af}(\bs_{\af})\b''_\af(\bs''_{\af})\hat \hs_{\af-1}.
\]

However, the elements of this set are not l.i., for $\hat \hs_{\af-1}=\widehat{\delta_{\af}(\ys_\af'')} \hat \ys'_{\af-1}$, and the 
$\widehat{\delta_{\af}(\ys_\af'')}$ belong to the submodule generated by the $\b''_\af(\bs''_{\af})$. Whence, we need to reset 
\[
I^\pf \ddot\bs_\af= \b_{\af}(\bs_{\af})\widehat{\delta_{\af}(\ys_\af'')} \hat \ys'_{\af-1}\oplus 0\ss 0\oplus \bs'''_\af,
\]
where $\langle \b_{\af}''(\bs'''_\af)\rangle=\langle \b_{\af}''(\bs''_\af)\rangle-\langle \widehat{\delta_{\af}(\ys_\af'')}\rangle$, 
with image
\[
I^\pf \ddot\b_\af(I^\pf \ddot\bs_\af)=\ddot \b_\af(0\oplus \bs''_\af)\ddot \b_\af(\b_{\af}(\bs_{\af})\hat\hs_{\af-1}\oplus 0)
=\b_{\af}(\bs_{\af})\b''_\af(\bs''_{\af}) i_{\af-1}(\hat \ys'_{\af-1}).
\]

The new basis of $I^\pf \ddot\CS_{\af}$  is (recall the image of the previous module is as above, so see Section \ref{bb11} for it)
\begin{align*}
I^\pf \ddot\b_{\af+1}(I^\pf \ddot\bs_{\af+1}) I^\pf \ddot \hs_\af I^\pf \ddot\bs_\af
&= \b_{\af}(\cs_{\af})\oplus \cs_{\af}~ 0\oplus \b''_{\af+1}(\bs''_{\af+1})
~ \widehat{I^\pf \ddot \hs_\af}~\b_{\af}(\bs_{\af})\widehat{\delta_{\af}(\ys_\af'')} \hat \ys'_{\af-1}\oplus 0~ 0\oplus \bs''_\af\\
&=\b_{\af}(\bs_{\af})\widehat{\delta_{\af}(\ys_\af'')} \hat \ys'_{\af-1}\oplus 0\ss \widehat{I^\pf \ddot \hs_\af}\ss0\oplus \cs_q \b''_{\af+1}(\bs''_{\af+1}) \bs''_\af,
\end{align*}
and recalling that $I^\pf \ddot \hs_\af= \hs''_\af$,  with lift $0\oplus \hat \hs''_\af$, the matrix of the change of basis is
\begin{align*}
(I^\pf \ddot\b_{\af+1}&(I^\pf \ddot\bs_{\af+1}) \hat{I^\pf  \ddot\hs}_\af I^\pf \ddot\bs_\af/I^\pf \ddot\cs_\af)\\
=&(\widehat{I^\pf \ddot \hs_\af}/ \hat\hs''_\af)
(\b_{\af}(\bs_{\af})\widehat{\delta_{\af}(\ys_\af'')} \hat \ys'_{\af-1}\oplus 0\ss \widehat{I^\pf \ddot \hs_\af}\ss0\oplus \cs_q \b''_{\af+1}(\bs''_{\af+1}) \bs''_\af/\zs_{\af-1}\oplus 0\ss 0\oplus \ds_\af)\\
=&(\widehat{I^\pf \ddot \hs_\af}/\hat\hs''_\af)(\widehat{\delta_{\af}(\ys_\af'')} \hat \ys'_{\af-1}/\hat\hs_{\af-1})
(\b_{\af}(\bs_{\af})\hat \hs_{\af-1}/\zs_{\af-1})( \cs_q\b''_{\af+1}(\bs''_{\af+1}) \hat\hs''_{\af}\bs''_\af/ \cs_q\ds''_\af).
\end{align*}

In the case $q=\af-1$, the intersection homology is the image  $p_{*,\af-1}(H_{\af-1}(\DS_\bu))$ in $H_{\af-1}(\DS_\bu/\CS_\bu)$, and we have that $I^\pf \ddot\bs_{\af-1} = \bs^{\DS_\bu}_{\af-1}$, with image
\[
I^\pf \ddot\b_{\af-1}(I^\pf \ddot\bs_{\af-1})=\b^{\DS_\bu}_\af(\bs^{\DS_\bu}_{\af-1}).
\]

The new basis of $I^\pf \ddot\CS_{\af-1}$  is 
\begin{align*}
I^\pf \ddot\b_{\af}(I^\pf \ddot\bs_{\af}) I^\pf \ddot \hs_\af I^\pf \ddot\bs_{\af-1}
&=\b_{\af}(\bs_{\af})\b''_\af(\bs''_{\af})i_{\af-1}(\hat \ys'_{\af-1})\widehat{I^\pf \ddot \hs_\af}\bs_{\af-1}^{\DS_\bu}.
\end{align*}

Since we may choose as lifts of the homology classes in $p_{*,\af-1}(H_{\af-1}(\DS_\bu))$ precisely the lifts of the homology classes of the subset $\ys_{\af-1}$ of $\hs_{\af-1}^{\DS_\bu}$,   and since $i_{\af-1}(\hat \ys'_{\af-1})=\widehat{i_{*,\af-1}(\ys'_{\af-1})}$, the matrix of the change of basis is
\begin{align*}
(I^\pf \ddot\b_{\af}&(I^\pf \ddot\bs_{\af}) \hat{I^\pf  \ddot\hs}_{\af-1} I^\pf \ddot\bs_{\af-1}/I^\pf \dot\cs_{\af-1})\\
&=(\widehat{I^\pf \ddot \hs_{\af-1}}/\hat\ys_{\af-1})(\b_{\af}(\bs_{\af})\b''_\af(\bs''_{\af})\widehat{i_{*,\af-1}(\ys'_{\af-1})}\hat \ys_{\af-1}\bs^{\DS_\bu}_{\af-1}/\ds_{\af-1})\\
&=(\widehat{I^\pf \ddot \hs_{\af-1}}/\hat\ys_{\af-1})
(\widehat{i_{*,\af-1}(\ys'_{\af-1})}\hat \ys_{\af-1}/\hat\hs_{\af-1}^{\DS_\bu})
(\b_{\af}^{\DS_\bu}(\bs^{\DS_\bu}_{\af})\hat\hs_{\af-1}^{\DS_\bu}\bs^{\DS_\bu}_{\af-1}/\ds_{\af-1}).
\end{align*}

The torsion of the intersection chain complex of the mapping cone  with perversity $\pf$ is
\begin{align*}
\tau_W (I^{\pf}\ddot\CS_{\bu};I^\pf \ddot\cs, I^\pf\ddot\ks)&= \sum_{q=0}^{m+1}(-1)^q [(I^\pf \ddot\b_{q+1}(I^\pf \ddot\bs_{q+1}) \hat{I^\pf \ddot \ks}_q I^\pf \ddot\bs_q/I^\pf \ddot\cs_q)] \\
=& \sum_{q=0}^{\af-2}(-1)^q [(I^\pf \ddot\b_{q+1}(I^\pf \ddot\bs_{q+1}) \hat{I^\pf  \ddot\ks}_q I^\pf \ddot\bs_q/I^\pf \cs_q)] \\
&+(-1)^{\af-1}[(I^\pf \ddot\b_{\af}(I^\pf \ddot\bs_{\af}) \hat{I^\pf \ddot \ks}_{\af-1} I^\pf \ddot\bs_{\af-1}/I^\pf \ddot\cs_{\af-1})]\\
&+ (-1)^{\af}[(I^\pf \ddot\b_{\af+1}(I^\pf \ddot\bs_{\af+1}) \hat{I^\pf  \ddot\ks}_\af I^\pf \ddot\bs_\af/I^\pf \ddot\cs_\af)]\\
&+\sum_{q=\af+1}^{m+1}(-1)^q [(I^\pf \ddot\b_{q+1}(I^\pf \ddot\bs_{q+1}) \hat{I^\pf  \ddot\ks}_q I^\pf \ddot\bs_q/I^\pf \ddot\cs_q)]\\
=& \sum_{q=0}^{\af-2}(-1)^q \left([(\widehat{I^\pf\ddot\ks}_q/\hat\hs^{\DS_\bu}_q)]+[(\b^{\DS_\bu}_{q+1}( \bs^{\DS_\bu}_{q+1}) \hat \hs^{\DS_\bu}_q  \bs^{\DS_\bu}_q/ \ds_q)]\right) \\
&+(-1)^{\af-1}[(\widehat{I^\pf \ddot \ks_{\af-1}}/\hat\ys_{\af-1})]
+(-1)^\af[( \cs_q\b''_{\af+1}(\bs''_{\af+1}) \hat\ks''_{\af}\bs''_\af/ \cs_q\ds''_\af)]\\
&+(-1)^{\af-1}
\left( [(\widehat{i_{*,\af-1}(\ys'_{\af-1})}\hat \ys_{\af-1}/\hat\ks_{\af-1}^{\DS_\bu})]
+[(\b_{\af}^{\DS_\bu}(\bs^{\DS_\bu}_{\af})\hat\ks_{\af-1}^{\DS_\bu}\bs^{\DS_\bu}_{\af-1}/\ds_{\af-1})]\right)\\
&+(-1)^\af\left([(\widehat{I^\pf \ddot \ks_\af}/\hat\ks''_\af)]
+[(\widehat{\delta_{\af}(\ys_\af'')} \hat \ys'_{\af-1}/\hat\ks_{\af-1})]
+[(\b_{\af}(\bs_{\af})\hat \ks_{\af-1}/\zs_{\af-1})]\right)\\
&+\sum_{q=\af+1}^{m+1}(-1)^q \left([(  \widehat{I^\pf\ddot\ks}_q/\hat{\ks}''_q)]+[( \b''_{q+1}(\bs''_{q+1})\widehat{\ks''_q}\bs''_q/\ds''_q))]\right).
\end{align*}

Recalling the definition of the intersection homology basis for the mapping cone in Definition \ref{intbasmapp}, and 
identifying $p_{*,\af-1}(\hs_{\af-1}^{\DS_\bu})=p_{*,\af-1}( \ys_{\af-1})$ (see the proof of Proposition \ref{p5.7}), we have that
\begin{align*}
\sum_{q=0}^{\af-2}(-1)^q [(\widehat{I^\pf\ddot\ks}_q/\hat\hs^{\DS_\bu}_q)]
+&(-1)^{\af-1} [(\widehat{I^\pf \ddot \ks_{\af-1}}/\hat\ys_{\af-1})]
+\sum_{q=\af+1}^{m+1}(-1)^q [(  \widehat{I^\pf\ddot\ks}_q/\hat{\ks}''_q)]\\
&=\sum_{q=0}^{\af-2}(-1)^{m+1} [(\widehat{I^\pf\ddot\ks}_q/\widehat{I^\pf\ddot\hs}_q)].
\end{align*}

We have proved the following result.

\begin{prop} \label{p5.6} Let $\pf$ be a perversity,  $I^\pf \ddot\ks_{\bu}$ a graded homology basis of
$H_\bu(I^\pf (C(\CS_\bu)\sqcup_{i_\bu} \DS_\bu))$, then the torsion of  intersection chain complex of the mapping cone of 
the pair $(\DS_\bu,\CS_\bu)$ with perversity $\pf$ is

\begin{align*}
\tau_W (I^{\pf}\ddot\CS_{\bu};I^\pf \ddot\cs, I^\pf\ddot\ks)
=& \sum_{q=0}^{\af-2}(-1)^{m+1} [(\widehat{I^\pf\ddot\ks}_q/\widehat{I^\pf\ddot\hs}_q)]
+(-1)^{\af-1}
[(\widehat{i_{*,\af-1}(\ys'_{\af-1})}\hat \ys_{\af-1}/\hat\hs_{\af-1}^{\DS_\bu})]\\
&+(-1)^\af[(\b_{\af}(\bs_{\af})\hat \hs_{\af-1}/\zs_{\af-1})]
+\sum_{q=0}^{\af-1}(-1)^q [(\b^{\DS_\bu}_{q+1}( \bs^{\DS_\bu}_{q+1}) \hat \hs^{\DS_\bu}_q  \bs^{\DS_\bu}_q/ \ds_q)] \\
&+\sum_{q=\af}^{m+1}(-1)^q[( \b''_{q+1}(\bs''_{q+1})\widehat{\hs''_q}\bs''_q/\ds''_q))],
\end{align*}
where $I^\pf \ddot\hs_{\bu}$ is the graded homology basis defined in Remark \ref{intbasmapp}
\end{prop}

\begin{prop} \label{p5.7} Let $\pf$ be a perversity,  then the torsion of  intersection chain complex of the mapping cone of 
the pair $(\DS_\bu,\CS_\bu)$ with perversity $\pf$ is

\[
\tau(I^\pf \ddot \CS_\bu;I^\pf \ddot\cs_\bu,I^\pf \ddot\hs_\bu)=\tau(I^\pf \dot\CS_\bu;I^\pf \dot\cs_\bu,I^\pf \dot\hs_\bu)+
\tau(\DS_{\bu}/\CS_{\bu};\ds_\bu'',\hs_\bu'')+\tau(I^\pf \ddot\Ha;I^\pf \dot\hs_\bu, I^\pf \ddot\hs_\bu,\hs_\bu''),
\]
where
\begin{align*}
\tau(I^\pf \ddot\Ha;I^\pf \dot\hs_\bu,& I^\pf \ddot\hs_\bu,\hs_\bu'')
=(-1)^{\af-1}[(p_{*,\af-1}(\hs^{\DS_\bu}_{\af-1})  \ys_{\af-1}''/\hs_{\af-1}'')]\\
&+\sum_{q=0}^{\af-2}(-1)^q\left([(\b_{q+1}(\ys_{q+1}'')\ys_{q}'/\hs_{q})]
-[(i_{*,q}(\ys_q') \ys_q/\hs_q^{\DS_\bu})]+[( p_{*,q}(\ys_q) \ys_q''/\hs_q'')]\right).
\end{align*}

\end{prop}

\begin{proof} Consider the short exact sequence of based chain complexes (we assume some basis for $Z_{\af-1}$ to have been fixed)
\[
\xymatrix{
I^\pf \dot \CS_\bu\ar[r]^{I^\pf \bar i_\bu}&I^\pf \ddot\CS_\bu\ar[r]^{I^\pf \bar p_\bu}&I^\pf \ddot \CS_\bu/I^\pf \dot \CS_\bu
}
\]

This consists in the commutative diagram
\[
\xymatrix{
\dots\ar[r]&\dot \CS_{\af+1}\ar[r]\ar[d]_{I^\pf \bar i_{\af+1}=\bar i_{\af+1}}&Z_{\af-1}\oplus \CS_\af\ar[r]\ar[d]_{I^\pf \bar i_{\af}=id\oplus i_{\af}}&\CS_{\af-1}\ar[r]\ar[d]^{I^\pf \bar i_{\af-1}=i_{\af-1}}&\dots\\
\dots\ar[r]&\ddot \CS_{\af+1}\ar[r]\ar[d]_{I^\pf \bar p_{\af+1}=\bar p_{\af+1}}&Z_{\af-1}\oplus \DS_\af\ar[r]\ar[d]_{I^\pf \bar p_{\af+1}=0\oplus  p_{\af}}&\DS_{\af-1}\ar[r]\ar[d]^{I^\pf \bar p_{\af-1}= p_{\af-1}}&\dots\\
\dots\ar[r]&\ddot\CS_{\af+1}/\dot \CS_{\af+1}\ar[r]&\DS_\af/ \CS_\af\ar[r]&\DS_{\af-1}/\CS_{\af-1}\ar[r]&\dots\\
}
\]

Recalling the chain isomorphism $\tilde j_\bu :\DS_{\bu}/\CS_{\bu}\to \ddot \CS_\bu/\dot \CS_\bu$, see equation \eqref{eeee}, and making the suitable identifications of chain and homology bases, this turns out to be a simple chain isomorphism, and we may rewrite 
\[
\xymatrix{
(I^\pf \dot \CS_\bu,I^\pf \cs_\bu)\ar[r]^{I^\pf \bar i_\bu}&(I^\pf \ddot\CS_\bu,I^\pf \ddot\cs_\bu)\ar[r]^{\tilde j_\bu^{-1}I^\pf \bar p_\bu}&(\DS_{\bu}/\CS_{\bu},\ds_\bu'')
}
\]

Applying Theorem \ref{mil0}, we have
\[
\tau(I^\pf \ddot \CS_\bu;I^\pf \ddot\cs_\bu,I^\pf \ddot\hs_\bu)=\tau(I^\pf \dot\CS_\bu;I^\pf \dot\cs_\bu,I^\pf \dot\hs_\bu)+
\tau(\DS_{\bu}/\CS_{\bu};\ds_\bu'',\hs_\bu'')+\tau(I^\pf \ddot\Ha;I^\pf \dot\hs_\bu, I^\pf \ddot\hs_\bu,\hs_\bu''),
\]
where $I^\pf \ddot\Ha$ is the associated long homology exact sequence. Observe that in degrees $q>a-1$ the homology of $I^\pf\dot\CS_\bu$ is trivial, and therefore the sequence $I^\pf \ddot\Ha$ is a sequence of isomorphisms and $0$ maps, with zero torsion. The relevant part of the sequence is thus the following
\[
\xymatrix{
I^\pf \ddot\Ha:& 0\ar[r]&H_{\af-1}(I^\pf \ddot\CS_\bu)=p_{*,\af-1}(H_{\af-1}(\DS_\bu))\leq H_{\af-1}(\DS_\bu/\CS_\bu)\ar[r]^{\hspace{70pt}p_{*,\af-1}}& H_{\af-1}(\DS_\bu/\CS_\bu)
}
\]

With the choices of the homology bases made in Section \ref{ladder1}, we have $I^\pf \ddot \hs_{\af-1}=p_{*,\af-1}(\hs_{\af-1}^{\DS_\bu})$ in  $H_{\af-1}(I^\pf \ddot\CS_\bu)$ and $I^\pf \ddot \hs_{\af-1}=\hs_{\af-1}''$ in  
$H_{\af-1}(\DS_\bu/\CS_\bu)$. Therefore $p_{*,\af-1}$ is the inclusion,   $I^\pf y_{\af-1}=p_{*,\af-1}(\hs_{\af-1}^{\DS_\bu})$ and the matrix of the change of basis is the identity matrix in $H_{\af-1}(I^\pf \ddot\CS_\bu)$. In $H_{\af-1}(\DS_\bu/\CS_\bu)$, identifying $p_{*,\af-1}(\hs_{\af-1}^{\DS_\bu})=p_{*,\af-1}( \ys_{\af-1})$, the matrix of the change of basis is
\[
(p_{*,\af-1}(\ys_{\af-1})  \ys_{\af-1}''/\hs_{\af-1}'').
\]

The remaining part of the sequence is
\[
\xymatrix{
I^\pf \ddot\Ha:&  H_{\af-1}(\DS_\bu/\CS_\bu)\ar[r]^{\b_{\af-1}}&H_{\af-2}(\CS_\bu)\ar[r]^{i_{*,\af-2}}&H_{\af-2}(\DS_\bu)\ar[r]&H_{\af-2}(\DS_\bu/\CS_\bu)\ar[r]&\dots
}
\]
and coincides with the one of the inclusion $i_\bu:\CS_\bu\to \DS_\bu$, whence the matrices of the change of basis are precisely those of the sequence associated to  $i_\bu:\CS_\bu\to \DS_\bu$, as described in Section \ref{ladder1}. Thus the torsion is

\begin{align*}
\tau(I^\pf \ddot\Ha;I^\pf \dot\hs_\bu, &I^\pf \ddot\hs_\bu,\hs_\bu'')
=(-1)^{\af-1}[(p_{*,\af-1}(\ys_{\af-1})  \ys_{\af-1}''/\hs_{\af-1}'')]\\
&+\sum_{q=0}^{\af-2}(-1)^q\left([(\de_{q+1}(\ys_{q+1}'')\ys_{q}'/\hs_{q})]
-[(i_{*,q}(\ys_q') \ys_q/\hs_q^{\DS_\bu})]+[( p_{*,q}(\ys_q) \ys_q''/\hs_q'')]\right).
\end{align*}
\end{proof}

\begin{prop}\label{standarddualitytorsion} Let 
\[
\xymatrix{
0\ar[r]&(\CS_\bu,\cs_\bu)\ar[r]^{i_\bu}& (\DS_\bu,\ds_\bu)\ar[r]^{p_\bu}& (\DS_\bu/\CS_\bu,\ds_\bu'')\ar[r]&0,
}
\]
be a short exact sequence of based complexes, where $\CS_\bu$ has dimension $m$, and $\DS_\bu$ has dimension $m+1$. Assume that the homology modules are free with graded bases $\hs_\bu$,  $\hs_\bu^{\DS_\bu}$, and $\hs_\bu''$. Then,
\begin{align*}
\tau(\DS_{\bu};\ds_\bu,\hs^{\DS_\bu})=\tau(\CS_{\bu};\cs_\bu,\hs_\bu)+\tau((\DS_{\bu}/\CS_{\bu};\ds_\bu'',\hs_\bu'')
+\tau(\Ha;\hs_q,\hs_q^{\DS_\bu},\hs_q''),
\end{align*}
where 
\begin{align*}
\tau(\Ha;\hs_q,\hs_q^{\DS_\bu},\hs_q'')&=\sum_{q=0}^{m+1}(-1)^q\left([(\delta_{q+1}(\ys_{q+1}'')\ys_q'/\hs_q)]
-[(i_{*,q}(\ys_q') \ys_q/\hs_q^{\DS_\bu})]+[( p_{*,q}(\ys_q) \ys_q''/\hs_q'')]\right).
\end{align*}

\end{prop}

\begin{rem} Considering the formula in the above proposition for the truncated complexes at degree $\af-1$, we have the formula
\begin{align*}
\sum_{q=0}^{\af-1}(-1)^q&[(\b^{\DS_\bu}_{q+1}( \bs^{\DS_\bu}_{q+1}) \hat{  \hs}^{\DS_\bu}_q  \bs^{\DS_\bu}_q/ \cs^{\DS_\bu}_q)]\\
=&\sum_{q=0}^{\af-1}(-1)^q[(\b_{q+1}( \bs_{q+1}) \hat{  \hs}_q  \bs_q/ \cs_q)]+\sum_{q=0}^{\af-1}(-1)^q[(\b''_{q+1}( \bs''_{q+1}) \hat{  \hs}''_q  \bs''_q/ \cs''_q)]\\
&+\sum_{q=0}^{\af-1}(-1)^q\left([(\de_{q+1}(\ys_{q+1}'')\ys_q'/\hs_q)]
-[(i_{*,q}(\ys_q') \ys_q/\hs_q^{\DS_\bu})]+[( p_{*,q}(\ys_q) \ys_q''/\hs_q'')]\right).
\end{align*}
\end{rem}

\begin{rem} Direct verification shows that the  formulas for the torsion given in Propositions \ref{p5.6} and \ref{p5.7} coincide. 
\end{rem}

\section{Cellular intersection homology}
\label{s0}

\subsection{Pseudomanifolds with isolated singularities}
\label{s1-1}


Let $K$ be   connected CW complex of (finite) dimension $n$.  We denote  by $K_{(q)}$ the $q$-skeleton of $K$, and by $e$ an open cell. Recall that a CW complex is {\it regular} if all attaching maps are homeomorphisms (\cite{Rot} pg. 226, \cite{Mas} Definition 6.1). Note that a simplicial complex is a regular $CW$ complex. 
Recall that a face of a cell $e$ in a regular CW complex $K$ is a cell $b$ of $K$ such that $b\subseteq \bar e$   (see for example \cite{Mas} pg. 243),

\begin{defi}\label{d1} \cite[Definition 8.1]{Mas}  An $n$-dimensional  pseudomanifold with boundary $K$ is an $n$-dimensional finite, 
regular CW complex which satisfies the following three conditions: 
\begin{enumerate}
\item every cell is a face of some $n$-cell,
\item every $(n - 1)$-dimensional cell is a face of at most two $n$-cells,
\item  given any two $n$-cells, $e$ and $e'$, there exists a sequence of $n$-cells $\{e_0=e,\dots , e_k=e'\}$, such that  and $e_{j-1}$ and $e_j$  have a common $(n- 1)$-dimensional face, for all $j = 1, 2 ..... k$. 
\end{enumerate}

The boundary $\b K$ of $K$ is the subcomplex generated by the set of the $(n-1)$-cells which are faces of exactly one $n$-cell of $K$ (i.e. the subcomplex consisting of the set of the $(n-1)$-cells which are faces of exactly one $n$-cell of $K$ and all their faces). 
\end{defi}

Note that $(K,\b K)$ is a relative pseudomanifold. See \cite{Spa} pg. 150 or \cite{Mun} pg. 261 for the definition of relative pseudomanifolds in the simplicial category.

Note that the pair $(K,\b K)$ is a relative CW complex \cite{Whi} Chapter II. 
For the definition of homology manifold with boundary see \cite{Mit}.

\begin{rem} If a topological space $X$ has a CW decomposition $K$ (this means that there is an homeomorphism $X=K$), and $K$ is an orientable $n$-pseudomanifold with boundary (see \cite{Mas} Chapter VIII for orientability of pseudomanifold), then it is clear that there exists a stratification of $X$ by closed subsets, namely the sequence of closed subspaces
\[
X_0\subseteq X_1\subseteq \cdots\subseteq X_{n-1}\subseteq X_n=X,
\]
where $X_q=K_{(q)}$ (homeomorphism) 
Now consider the space $Y=X-X_{n-1}=K-K_{(n-1)}$. Because of condition 1 in Definition \ref{d1}, $Y$ is the disjoint union of a finite number of open sets, and therefore it is an oriented $n$-dimensional  manifold with boundary. Moreover, it is clear that $Y$ is dense in $X$. Since this can of course happens at some lower level, namely for some $X-X_q$, we can state that if the CW decomposition of $X$ is an orientable $n$-dimensional pseudomanifold with boundary, then there exists a closed subspace $\Sigma$ of $X$ such that $X-\Sigma$ is an orientable $n$-dimensional manifold with boundary, and $X-\Sigma$ is dense in $X$.

Next, condition (2) of the same definition, shows that if we actually re-collocate the $(n-1)$-cells in $Y$,  we still have a manifold with boundary, namely that the space $X-X_{(n-2)}$ is an orientable $n$-manifold with boundary, and therefore the codimension of $\Sigma$ is at least 2.
\end{rem}

\begin{defi} A  proper subcomplex of an $n$-dimensional pseudomanifold  $K$ is a subcomplex of $K$ disjoint from the singular set. We call the pair $(K,L)$ a  proper pair of pseudomanifolds, or a  proper relative pseudomanifold.
\end{defi}

\begin{defi}  An $n$-dimensional  pseudomanifold with isolated singularities and proper boundary  is  an $n$-dimensional pseudomanifold $K$ with boundary    such that there exists a (finite) subset of $0$-cells  $\Sigma=\{x_0, \dots, x_k\}$ in the interior of  $K$  such that $K-\Sigma$ is an $n$-dimensional   homology manifold with boundary.  The subspace $\Sigma$ is called  the  singular locus of $K$. 
\end{defi}


\begin{defi} An $n$-dimensional  space  with isolated singularity and proper boundary is an topological space $X$ that admits a CW decomposition $K$, where $K$ is  $n$-dimensional pseudomanifold with isolated singularities and proper boundary. \end{defi}

Let $C(X)$ denote the {\it cone over} $X$, namely the  quotient space $(I\times X)/(\{0\}\times X)$. Let $N$ be a regular oriented CW complex that is an homology $m$ manifold (for definition of orientation on regular CW complex see for example \cite{Mas}). Let $C(N)$ the cone over $N$. Since $(C(N),N)$ is a relative CW, it is clear that $C(N)$ is a CW \cite{Whi}. A regular oriented CW structure on $C(N)$ is given as follows. Let $v$ be the vertex of $C(N)$. For each cell $e$ in $N$ denote by $[v, e]$ the cone over $\bar e$ with vertex $v$. Since $N$ is regular, the set of the cells $[v,e]$ and $e$, with $e$ a cell of $N$, coincides  with  the standard CW structure of $C(N)$ induced by that of the product $I\times N$ (see for example \cite{Whi} ex. 4 and 5, pg. 51), and  gives a regular oriented CW structure on $C(N)$. Writing $C(N)$ we now means the cone over $N$ with this regular oriented CW structure. Also note that the pair $C(X)$ always come with a natural inclusion $i:X\to C(X)$, $i(x)=[(1,x)]$, that is cellular and defines the cellular pair $(C(X),X)$.

\begin{lem}\label{l2.1} An $n$-dimensional pseudomanifold $K$ with isolated singularities is the push out  of a regular CW complex $M$ that is an  homology $n$-manifold  with boundary $\b M=N_1\sqcup\dots \sqcup N_k$ by attaching the  cone $C (N_k)$ to $N_k$, for each $k$.
\end{lem} 
\begin{proof} Let $K'$ be a subdivision of $K$ such that the link in $K'$ of two  singular $0$-cell of $K$ are disjoint. Then, setting, for each $0$-cell $e_j$ of $K$, $N_j={\rm Link}(e_j)$, and $M=\overline{K'-\bigsqcup_{j=1}^k C(N_j)}$, we have the result.
\end{proof}


It is clear that we can consider the case of one isolated singularity, without loss of generality, and so we will do.

\begin{defi} Let $(K,L)$ be an  $n$-dimensional proper relative pseudomanifold  with one isolated singularity $\Sigma=\{x_0\}$. We call  the  decomposition of $K$
\[
K=M\sqcup_{N_0} C(N_0),
\]
where $N_0= {\rm Link}(x_0)$, $M=\overline{K-C(N_0)}$, the  canonical  decomposition  of $K$ (note that $N_0=\b M-(\b M\cap L)$).

\end{defi}

\begin{lem}\label{pi1} Let X a space with one isolated singularity $x_0$, and possible proper boundary, $\b X\cup \{x_0\}=\emptyset$. Let $K$ a CW decomposition  of $X$. Let $K=M\sqcup_{N_0} C(N_0)$ be the canonical  decomposition of $K$. Then, the inclusion $i:N_0\to K$ induces a trivial map in homotopy.
\end{lem}
\begin{proof} If $\gamma$ is any circle in a non trivial class of $\pi_1(N_0)$, the image $i_*([\gamma])$ is a class in $\pi_1(C(N_0))=1$, and therefore $\gamma$ is homotopy trivial in $K$.
\end{proof}

\begin{lem}\label{pic2} Let X a space with one isolated singularity $x_0$, and possible proper boundary, $\b X\cup \{x_0\}=\emptyset$. Let $K$ a CW decomposition  of $X$. Let $K=M\sqcup_{N_0} C(N_0)$ be the canonical  decomposition of $K$. Then, $M/N_0$, $K/C(N_0)$ and $K$ have the same homotopy type. \end{lem}
\begin{proof} It is clear that $M/N_0$ and $K/C(N_0)$ have the same homotopy type. The fact that they have the same homotopy type of $K$ follows for example by \cite[3.2.2]{Pic}, since the inclusion $N_0\to M$ is a cofibration. This in turns follows since $i:N_0\to M$ is the inclusion of a subcomplex (note that taking the long homotopy exact sequence associated to the inclusion $i$ it follows that the homotopy groups of $K$ and $K/C(N_0)$ are isomorphic).
\end{proof}

\subsection{Intersection cellular chain complex}
\label{s2.2}

Recall Definition \ref{perversity}:

\begin{defi} A {\it perversity} is a finite sequence of integers $\pf=\{\pf_j\}_{j=2}^n$ such that $\pf_2=0$ and $\pf_{j+1}=\pf_j$ or $\pf_j+1$. If $\pf$ is a perversity (of length $n$),  we define the constant $\af=\af(\pf_n):=n-\pf_n$. 
\end{defi}

The perversity: $\mf=\{\mf_j=[j/2]-1\}$ is called {\it lower middle perversity}. The {\it null perversity} is $0_j=0$, and the {\it top perversity} is $\tf_j=j-2$. Given a perversity $\pf$, the {\it complementary perversity} $\pf^c$ is $\pf^c_j=\tf_j-\pf_j=j-\pf_j-2$.

\begin{defi} Let  $K$ be an orientable $n$-pseudomanifold with isolated  singularities and smooth boundary and $\pf$ a perversity.  If $q$ is an integer and $\pf$ a perversity, a  cell $e$ of $K_{q}$ is said {\it $(\pf,q)$-allowable} if
\[
 \dim(\bar e\cap \Sigma)\leq q-n+\pf_{n}.
  \]
 \end{defi}

\begin{defi}\label{basicset} Let  $K$ be an orientable $n$-pseudomanifold with isolated  singularities and smooth boundary, and $\pf$ a perversity. The {\it intersection cellular family} of perversity $\pf$  associated to $K$ is the subfamily of the  $(\pf,q)$-allowable cells of $K$, namely $I^\pf \FF (K)=\{I^\pf K_{(q)}\}_{q=0}^m$, where $I^\pf K_{(q)}$ is the subcomplex  
\[
I^\pf K_{(q)}=\{e\in K_{(q)} ~|~e~{\rm is}~(\pf,q){\rm -allowable}\}.
\]

\end{defi}


\begin{rem}\label{r3.2} Note that for each $q$ there are cellular inclusions $I^\pf K_{(q)}\to I^\pf K_{(q+1)}$, and $I^\pf K_{(q)}\to  K_{(q)}$.
\end{rem}

\begin{lem}\label{l001} Let $N$ be a regular orientable CW complex that is an homology $m$-manifold,   $C(N)$ the cone over $N$, and $\pf$ a perversity. Let $n=m+1$, then, 
\[
I^\pf C(N)_{(q)}=\left\{\begin{array}{ll} N_{(q)},& q<n-\pf_n,\\ C(N)_{(q)}, &q\geq n-\pf_n.\end{array}\right.
\]
\end{lem}
\begin{proof} By definition 
\[
I^\pf C(N)_{(q)}=\{e\in C(N)_{(q)} ~|~\dim(\bar e \cap v)\leq q-n+\pf_n\}.
\]

Since $\dim(\bar e\cap v)$ is either $0$ or $-\infty$, if $q-n+\pf_n\geq 0$, then all cells of $C(N)_{(q)}$ belong to $I^\pf C(N)_{(q)}$. On the other side, if $q-n+\pf_n< 0$, then only the cells of $N_{(q)}$ belong to $I^\pf C(N)_{(q)}$.
\end{proof}

\begin{lem}\label{l3.3}  Let  $K$ be an orientable $n$-pseudomanifold with one isolated  singularity $x_0$,  standard decomposition $K=M\sqcup C(N_0)$, and smooth boundary, and $\pf$ a perversity. Then, 
\[
I^\pf K_{(q)}=\left\{\begin{array}{ll} M_{(q)},& q<n-\pf_n,\\ K_{(q)}, &q\geq n-\pf_n.\end{array}\right. 
\] 
\end{lem}
\begin{proof} By Lemma \ref{l2.1}.
\end{proof}

For each perversity $\pf$, the family $I^\pf \FF(K)$ defines a filtration (a relative filtration of $(K,L)$) (see for example \cite{Rot} pg. 212) of $K$ (note that for any perversity $I^\pf K_{(n)}=K$):
\[
K=I^\pf K_{(n)}\supset I^\pf K_{(n-1)}\supset \dots\supset I^\pf K_{(1)}\supset I^\pf K_{(0)}\supset 0.
\]

This filtration is not cellular, but any filtration has an associated chain complex defined using the homology of the pair and   where the boundary operators are defined by composing the exact sequences of the pairs as follows (see for example \cite{Rot} pg. 213):

\centerline{\xymatrix@C=0.3cm{
&&&\dots\ar[d]&\\
&&&H_{q-1}(I^\pf K_{(q-2)}\cup L)\ar[d]&\\
\dots\ar[r]&H_q(I^\pf K_{(q)}\cup L)\ar[r]&H_q(I^\pf K_{(q)}\cup L, I^\pf K_{(q-1)}\cup L)\ar[dr]_{\b_q}\ar[r]_{\delta}&H_{q-1}(I^\pf K_{(q-1)}\cup L)\ar[d]_j\ar[r]&\dots\\
&&&H_{q-1}(I^\pf K_{(q-1)}\cup L,I^\pf K_{(q-2)}\cup L)\ar[d]&\\
&&&\dots&
}}

\begin{defi}\label{intcellcomp} Let  $K$ be an orientable $n$-pseudomanifold with isolated  singularities and smooth boundary, and $\pf$ a perversity. The intersection cellular chain complex of $K$ with perversity $\pf$ is the chain complex $I^\pf \CS_\bu(K)$ with chain modules 
\[
I^\pf \CS_q(K;\Z)=H_q(I^\pf K_{(q)}, I^\pf K_{(q-1)}),
\]

Here the homology modules are any standard (singular, simplicial, CW) homology modules with integer coefficients.

Let $L$ be a proper subspace of $K$, the $q$ relative intersection cellular chain complex  of $(K,L)$ with perversity $\pf$ is the chain complex $I^\pf \CS_\bu(K,L)$ of chain modules
\[
I^\pf \CS_q(K,L;\Z)=H_q(I^\pf K_{(q)}\cup L, I^\pf K_{(q-1)}\cup L).
\]

In both cases, the boundary homomorphisms are defined by compositions of the restrictions of the boundary homomorphisms of the opportune long homology exact sequences (as described above).

\end{defi}

This definition is well posed, since by Remark \ref{r3.2}, $I^\pf K_{(q-1)}$ is a sub complex of $I^\pf K_{(q)}$ for all $q$.

Note that $I^\pf \CS_q(K,L)$  coincides with  the submodule  of $I^\pf \CS_q(K)$ generated by $(\pf,q)$-allowable  cells $e$ of  $K$ that are not in $L$.

\begin{lem} \label{freenes} Let  $K$ be an orientable compact connected $n$-pseudomanifold with isolated  singularities and smooth boundary,  $L$  a proper subspace of $K$, and $\pf$ a perversity, then $I^\pf \CS_q(K)$ and $ I^\pf \CS_q(K,L)$ are free finitely generate $\Z$ modules.
\end{lem}
\begin{proof} Since $K$ is compact, there are a finite number of cells. Thus   $H_q(I^\pf K_{(q)}, I^\pf K_{(q-1)})$ is the homology of a CW pair, whence generated by the some of $q$-cells (in the critical dimension $n-\pf_n$ not all the $n-\pf_n$-cells are cycles).
\end{proof}

\begin{rem} Let $\CS(K)=\{\CS_q(K)\}_{q=0}^n$, and $\CS(K,L)=\{\CS_q(K,L)\}_{q=0}^n$ be the cellular chain complex and the relative cellular chain complex of $K$ and of the pair $(K, L)$ respectively.  The {\it intersection cellular chain module} of perversity $\pf$  is the submodule  of $\CS_q(K)$ generated  by the  $(\pf,q)$-allowable cells with $(\pf, q-1)$-allowable boundary. Similarly for the relative case. The {\it relative intersection chain module}  of perversity $\pf$,  is $I^\pf \CS_q(K)/I^p \CS_q(L)$.  This shows the equivalence of our construction with the one in \cite{GM1}. 
\end{rem}

Note that we could have taken  the singular chain complex $\SS(K)$ and considered the filtration of the complex
\[
\SS(I^\pf K)\supset \SS(I^\pf K_{(n-1)})\supset\dots \supset \SS(I^\pf K_{(1)})\supset \SS(I^\pf K_{(0)})\supset 0,
\]
and then proceeded as in \cite{Mil0} pg. 371.

Next we provide an explicit description of the intersection cellular chain complex of a pseudomanifold with isolated singularities.

\begin{lem}\label{l2}  Let $N$ be a regular oriented CW complex that is an homology $m$-manifold,   $C(N)$ the cone over $N$, and $\pf$ a perversity. Let $n=m+1$, then, 
\begin{align*}
I^\pf \CS_q(C(N))&=\left\{
\begin{array}{ll}\CS_q(N)&q< n-\pf_n,\\
H_{a}(C(N)_{(a)},N_{(a-1)}),&q=\af=n-\pf_n, \\
\CS_q(C(N))&q> n-\pf_n.
\end{array}\right.\\
I^\pf \CS_q(C(N),N)&=\left\{
\begin{array}{ll} 0 &q< n-\pf_n,\\
H_{a}(C(N)_{(a)}\cup N,N),&q=\af=n-\pf_n, \\
\CS_q(C(N),N)&q> n-\pf_n.
\end{array}\right.
\end{align*}

\end{lem}
\begin{proof} This follows by the definition and Lemma \ref{l001}. 
\end{proof}

It will be useful to have a more explicit description of the boundary operators in this case. Consider the long homology exact sequences of the pairs $(C(N)_{(q)}$, $C(N)_{(q-1)})$, $(C(N)_{(q)},N_{(q-1)})$ and $(N_{(q)},N_{(q-1)})$. Denote by $i''',  j''', \b''$, $i'',j'',\b''$, and $i',j',\b'$ the homomorphisms, i.e. for example
\[
\xymatrix{ \dots\ar[r]&H_q(N_{(q)})\ar[r]^{j'_q}&H_q(N_{(q)},N_{(q-1)})\ar[r]^{\b_q} &H_{q-1}(N_{(q-1)})\ar[r]&\dots,
}
\]

Then, we find out that
\[
\begin{aligned} I^\pf \b_q&= j_{q-1}''' \b_q''',&q\geq \af+2,\\ 
I^\pf \b_{a+1}&=j''_a\b_{\af+1}''',&\\
I^\pf \b_\af&=j_{\af-1}' \b''_\af,&\\
I^\pf \b_q&=j_{q-1}' \b_q', &q\leq \af-1.
\end{aligned}
\]

The statement for $q\geq \af+2$ and $q\leq \af-1$ is clear. At $q=\af+1$,  the relevant diagram is

\centerline{\xymatrix@C=0.4cm{
&&&\dots\ar[d]&\\
&&&H_{\af}(I^\pf N_{(\af-1)})\ar[d]&\\
\dots\ar[r]&H_{\af+1}( C(N)_{(\af+1)})\ar[r]&H_{\af+1}( C(N)_{(\af+1)}, C(N)_{(\af)})\ar[dr]_{I^\pf \b_{\af+1}}\ar[r]^{\b'''_{\af+1}}&H_{\af}( C(N)_{(\af)})\ar[d]^{j''_{\af}}\ar[r]&\dots\\
&&&H_{\af}( C(N)_{(\af)}, N_{(\af-1)})\ar[d]&\\
&&&\dots&
}}

At $q=\af$,  the relevant diagram is

\centerline{\xymatrix@C=0.4cm{
&&&\dots\ar[d]&\\
&&&H_{\af-1}(I^\pf N_{(\af-2)})\ar[d]&\\
\dots\ar[r]&H_{\af}( C(N)_{(\af)})\ar[r]&H_\af( C(N)_{(\af)}, N_{(\af-1)})\ar[dr]_{I^\pf \b_{\af}}\ar[r]^{\b''_{\af}}&H_{\af-1}( N_{(\af-1)})\ar[d]^{j'_{\af-1}}\ar[r]&\dots\\
&&&H_{\af-1}( N_{(\af-1)}, N_{(\af-2)})\ar[d]&\\
&&&\dots&
}}

\begin{rem} Let  $K$ be a $n$-pseudomanifold  with an isolated  singularities pair and smooth boundary $\b K$, and $\pf$ a perversity. Then,
\begin{align*}
I^\pf \CS_q(K)&=\left\{\begin{array}{ll}\CS_q(K-\Sigma)&q< n-\pf_n,\\
\{c\in \CS_q(K)~|~\b c \in \CS_{q-1}(K-\Sigma)\},&q=n-\pf_n, \\
\CS_q(K)&q> n-\pf_n.
\end{array}\right.\\
I^\pf \CS_q(K,\b K)&=\left\{\begin{array}{ll}\CS_q(K-\Sigma,\b K)&q< n-\pf_n,\\
\{c\in \CS_q(K,\b K)~|~\b c \in \CS_{q-1}(K-\Sigma,\b K)\},&q=n-\pf_n, \\
\CS_q(K,\b K)&q> n-\pf_n.
\end{array}\right.
\end{align*}
\end{rem}

\begin{prop}\label{p2} Let  $K$ be a proper  $n$-pseudomanifold  with an isolated  singularities and smooth boundary,  and $\pf$ a perversity. Let $K=M\sqcup C(N_0)$ be the standard decomposition of $K$. Then,
\begin{align*}
I^\pf \CS_q(K)&=\left\{\begin{array}{ll}\CS_q(M)&q< n-\pf_n,\\
H_{a}(K_{(a)},M_{(a-1)}),&q=\af=n-\pf_n, \\
\CS_q(K)&q> n-\pf_n.
\end{array}\right.\\
I^\pf \CS_q(K,\b K)&=\left\{\begin{array}{ll} \CS_q(M,\b K) &q< n-\pf_n,\\
H_{a}(K_{(a)}\cup \b K, M_{(a-1)} \cup \b K),&q=\af=n-\pf_n, \\
\CS_q(K,\b K)&q> n-\pf_n.
\end{array}\right.
\end{align*}
\end{prop}

\begin{proof}
By the definition and and proceeding as in the proof of  Lemma \ref{l001}
\end{proof}

\begin{defi} The homology of $I^\pf \CS(K)$ is called the intersection homology of $K$ with perversity $\pf$, and denoted by $I^\pf H (K)$. The homology of $I^\pf \CS(K,\b K)$ is called the relative intersection homology of the pair $(K,\b K)$ with perversity $\pf$, and denoted by $I^\pf H(K,\b K)$. 
\end{defi}

\begin{prop}\label{l2.3-sec1}  Let $N$ be a regular oriented CW complex that is an homology $m$-manifold,   $C(N)$ the cone over $N$, and $\pf$ a perversity. Let $n=m+1$, then, 
\begin{align*}
I^\pf H_q(C(N))&=\left\{
\begin{array}{ll}
H_q(N), & 0\leq q\leq \af-2,\\
0, & \af-1 \leq 	q\leq  m+1,
\end{array}\right.
\\
I^\pf H_q(C(N),N)&=\left\{
\begin{array}{ll}
0 ,& 0\leq q\leq \af -1,\\
H_q(C(N),N) = H_{q-1}(N), & \af \leq q\leq m+1.
\end{array}
\right.
\end{align*}
\end{prop}

\begin{prop}\label{p2.3} Let  $K$ be an orientable $n$-pseudomanifold with isolated  singularities and smooth boundary  and $\pf$ a perversity. Let $K=M\sqcup C(N_0)$ be the standard decomposition of $K$. Then,
\begin{align*}
I^\pf H_q(K)&=\left\{\begin{array}{ll}H_q(M)&q\leq\af-2,\\
\Im\left(p_*:H_{\af-1}(M)\to H_{\af-1}(M,\b M)\right),\\
H_q(K)&q\geq \af-1,
\end{array}\right.
\end{align*}
where $p:M\to M/\b M$ is the canonical projection on the quotient. 
\end{prop}
\begin{proof} Recall the  intersection chain complex as given in Proposition \ref{p2}
\[
\xymatrix{
\ar[r]&\CS_{\af+1}(K)\ar[r]^{\b_{\af+1}^K}&I^\pf \CS_\af (K)\ar[r]^{I^\pf \b_\af}&\CS_{\af-1} (M)\ar[r]^{\b_{\af-1}^M}&\CS_{\af-2} (M)\ar[r]&
}
\]
where $I^\pf \CS_\af (K)=H_{\af}( K_{(\af)}, M_{(\af-1)})=\{c\in \CS_\af(K)~|~\b_\af^K(c)\in \CS_{\af-1}(M)\}\leq \CS_\af(K)$.
 
It is clear that $I^\pf H_q(K)=H_q(M)$ for $q\leq \af-2$, and that $I^\pf H_q(K)=H_q(K)$ for $q\geq \af+2$. To proceed we consider the homology sequences appearing in the definition of the intersection chain modules. The relevant diagram at dimension $q=\af+1$ is
\[
\centerline{\xymatrix@C=0.4cm{
&&&\dots\ar[d]&\\
&&&H_{\af}( M_{(\af-1)})=0\ar[d]&\\
\dots\ar[r]&H_{\af+1}( K_{(\af+1)})\ar[r]&H_{\af+1}( K_{(\af+1)}, K_{(\af)})\ar[dr]_{I^\pf \b_{\af+1}}\ar[r]^{\b'''_{\af+1}}&H_{\af}( K_{(\af)})\ar[d]^{j''_{\af}}\ar[r]&\dots\\
&&&H_{\af}( K_{(\af)}, M_{(\af-1)})\ar[d]&\\
&&&\dots&
}}
\]

Comparing it with the analogous diagram  for the complex $K$, since in both cases $j_\af''$ is injective, it follows that the $I^\pf \b_{\af+1}=\b_{\af+1}^K$, and therefore $I^\pf H_{\af+1}(K)=H_{\af+1}(K)$. This also implies that $\Im I^\pf \b_{\af+1}=\Im \b_{\af+1}^K$.

The relevant diagram at dimension $q=\af$ is
\[
\centerline{\xymatrix@C=0.4cm{
&&&\dots\ar[d]&\\
&&&H_{\af-1}( M_{(\af-2)})=0\ar[d]&\\
\dots\ar[r]&H_{\af}( K_{(\af)})\ar[r]&H_{\af}( K_{(\af)}, M_{(\af-1)})\ar[dr]_{I^\pf \b_{\af}}\ar[r]^{\b'''_{\af}}&H_{\af-1}( M_{(\af-1)})\ar[d]^{j''_{\af-1}}\ar[r]&\dots\\
&&&H_{\af-1}( M_{(\af-1)}, M_{(\af-2)})\ar[d]&\\
&&&\dots&
}}
\]

Injectivity of $j_{\af-1}''$, implies that $\ker I^\pf \b_{\af}=\ker \b'''_{\af}$, the boundary operator of the sequence of the pair. By its definition, $\b'''_{\af}([c])=[\b_\af^K(c)]$, for a chain $c$ in $\CS_\af(K_{(\af)})$ representing an  element of $H_{\af}( K_{(\af)}, M_{(\af-1)})$. Thus, $\ker \b'''_{\af}=\{[c]~|~[\b_\af^K(c)]=0\}$. But $\b_\af^K(c)$  is a cycle in $\CS_{\af-1}(M_{(\af-1)})$, so its homology class in $H_{\af-1}(M_{(\af-1)})$ can be zero only if it is the trivial cycle, namely if $\b_\af^K(c)=0$, i.e. if $c$ is a cycle in $\CS_\af(K_{(\af)})$. This shows that $\ker I^\pf \b_{\af}=\ker \b^K_{\af}$, and therefore $I^\pf H_{\af}(K)=H_{\af}(K)$.

The relevant diagram at dimension $q=\af-1$ is
\[
\centerline{\xymatrix@C=0.4cm{
&&&\dots\ar[d]&\\
&&&H_{\af-2}( M_{(\af-3)})=0\ar[d]&\\
\dots\ar[r]&H_{\af-1}( M_{(\af-1)})\ar[r]&H_{\af-1}( M_{(\af-1)}, M_{(\af-2)})\ar[dr]_{I^\pf \b_{\af-1}}\ar[r]^{\b'''_{\af-1}}&H_{\af-2}( M_{(\af-2)})\ar[d]^{j''_{\af-2}}\ar[r]&\dots\\
&&&H_{\af-2}( M_{(\af-2)}, M_{(\af-3)})\ar[d]&\\
&&&\dots&
}}
\]

It is clear that $I^\pf \b_{\af-1}=\b_{\af-1}^M$, and therefore its cycles are the cycles in $\CS_{\af-1}(M_{(\af-1)})$. 

Eventually, again by injectivity of $j_{\af-1}''$, the image of $I^\pf \b_{\af}$ is isomorphic to the image of the boundary operator 
$\b'''_{\af}$, and arguing as above, is given by the non zero elements of the type $\b_\af^K(c)$, that are cycles in $\CS_{\af-1}(M_{(\af-1)})$, where  $c$ is a chain in $\CS_\af(K_{(\af)})$. Therefore, $c$ may be either a chain in $\CS_{(\af)}(M_{(\af)})$ (that is not a cycle) with boundary in $\CS_{\af-1}(M_{(\af-1)})$, or a chain in $\CS_{(\af)}(C(N)_{(\af)})$ (that is not a cycle) with boundary a cycle in $\CS_{\af-1}(M_{(\af-1)})\cap \CS_{(\af-1)}(C(N)_{(\af)})=\CS_{\af-1}(N_{(\af-1)})$. The boundary in the first case coincides with the boundary of $\b^M_{(\af)}$. In the second case, with the module of the cycles of $\CS_{(\af-1)}(N_{(\af-1)})$. This means that the homology is the homology of $M$ quotient the cycles of $\CS_{(\af-1)}(N_{(\af-1)})$, i.e. the image of $p_*:H_{\af-1}(M)\to H_{\af-1}(M,\b M)$, where $p:M\to M/\b M$ is the canonical projection.
\end{proof}

Let $K$ be a CW complex, recall that a  subdivision of $K$ is a CW complex $K'$ such that $|K|=|K'|$, and each cell of $K'$ is contained in some cell of $K$ in such  a way that the identity map $i:K\to K'$ is cellular. A subdivision of the pair  $(K,L)$, is a pair $(K',L')$, where $K'$ is a subdivision of $K$, $L'$ is a subdivision of $L$. It is clear that  if $N'$ is a subdivision of $N$, then $C(N')$ is a subdivision of $C(N)$.

\begin{lem} Let  $(K,L)$ be a proper $n$-pseudomanifold  with isolated  singularities pair (in particular $L=\b K$). Then, we have the standard decomposition $K=M\cup_{\b M-(\b M\cap L)} C(N)$, for some $M$ and $N=\b M-(\b M\cap L)$. Let $M'$ be a subdivision of $M$, $L'$ be a subdivision of $L$. Then, 
\[
K'=M'\sqcup_{\b M'-(\b M'\cap L')} C(N'),
\]
is a subdivision of $K$.
\end{lem}

\begin{prop}\label{sub} Let  $(K,L)$ be a proper $n$-pseudomanifold  with isolated  singularities pair (in particular $L=\b K$), and $\pf$ a perversity. Then, there is a injective chain quasi isomorphism, $I^\pf j_\bu:I^\pf \CS_\bu(K,L)\to I^\pf \CS_\bu(K',L')$.
\end{prop}
\begin{proof} By definition,
\begin{align*}
I^\pf \CS_q(K',L')&=H_q(I^\pf K'_{(q)}\cup L', I^\pf K'_{(q-1)}\cup L'),\\
I^\pf \CS_q(K,L)&=H_q(I^\pf K_{(q)}\cup L, I^\pf K_{(q-1)}\cup L).
\end{align*}


By Lemma \ref{l3.3}, $I^\pf K'_{(q)}=M'_{(q)}$, and $I^\pf K_{(q)}=M_{(q)}$, if $q<n-\pf_n$, while $I^\pf K'_{(q)}=K'_{(q)}$, and $I^\pf K_{(q)}=K_{(q)}$, if $q\geq n-\pf_n$. In both cases $I^\pf K'_{(q)}$ is a decomposition of $I^\pf K_{(q)}$, and it is clear that the identity map is cellular and equivariant with respect to the relation of be an allowable cell, in other words the following diagram commute, where the vertical arrows are inclusions,
\[
\xymatrix{
I^\pf K_{(q)}\ar[d]\ar[r]^{id|_{I^\pf K_{(q)}}}&I^\pf K'_{(q)}\ar[d]\\
K_{(q)}\ar[r]_{id|_{K_{(q)}}}&K'_{(q)}
}
\]

The map $id|_{I^\pf K_{(q)}}$ is cellular, and for each $q$ induces a map in homology
\[
I^\pf j_{*,q}:H_q(I^\pf K'_{(q)}\cup L', I^\pf K'_{(q-1)}\cup L')\to H_q(I^\pf K_{(q)}\cup L, I^\pf K_{(q-1)}\cup L),
\]
that is a subdivision operator for the to chain complexes (see also \cite{Mun}. This completes the proof. 
\end{proof}

\subsection{Duality}
\label{Poincare}

In this section we recall some material on duality, and then define a dual cell decomposition for the intersection theory.

\subsubsection{Poincar\'e complexes}
\label{sec666}

A chain complex of $R$ modules $\CS_\bu$  of dimension $m$  is dualizable (or a Poincar\'e complex) if there exists a second chain complex of $R$ modules $\check \CS_\bu$ of dimension $m$, and a chain isomorphism 
\[
\P_q:\check\CS_q\to  \TCS_q=\CS_{m-q}^\da,
\]
where the notation $\da$ is for the algebraic dual, and the complex $\TCS_\bu$ is defined as follows:  $\TCS_q=\CS^\da_{m-q},\Tb_q:=\b^\da_{m-q+1}$ of degree -1:
\[
\xymatrix{
\TCS_\bu:&\TCS_m\ar[r]^{\Tb_m}&{{\ldots}} \ar[r]&\TCS_q=\CS^\da_{m-q}\ar[r]^{\Tb_q=\b^\da_{m-q+1}}&\TCS_{q-1}=\CS^\da_{m-q+1}\ar[r]&{\ldots}\ar[r]^{\Tb_1=\b^{m-1}}&\TCS_0=\CS^m.
}
\] 

Equivalently, we require a chain isomorphism
\[
\QQ_q:\CS_q=\TCS_{m-q}^\da\to \check \TCS_q=\check \CS_{m-q}^\da, \hspace{30pt} \QQ_q=\P_{m-q}^\da.
\]

The map $\P_q$ induces the following isomorphism in homology:
\beq\label{stand1}
\P_{*,q}:H_q( \check \CS_\bu) \to H_q(\TCS_\bu)=H_{m-q}( \CS_\bu^\da),
\eeq
the map $\QQ_q$:
\beq\label{stand2}
\QQ_{*,q}:H_q( \CS_\bu) \to H_q(\check \TCS_\bu)=H_{m-q}( \check\CS_\bu^\da),
\eeq

A Poincar\'e dual appears naturally in the following geometric context. Let $W$ be a compact connected oriented smooth manifold. Let $N$ be a regular CW decomposition of $W$, in particular a simplicial one. Then, passing through the baricentric subdivision, a dual block decomposition $\check N$ of $W$ may be obtained, having the baricentric subdivision as a common subdivision with $N$. Then, the previous results hold with the identifications

\begin{align*}
\CS_q&=\CS_q(N;\Z),\\
\check \CS_q&=\CS_q(\check N;\Z).
\end{align*}

In particular, if $\pi=\pi_1(N)$, these results hold for twisted coefficients as well, namely, with the identifications
\begin{align*}
\CS_q&=\CS_q(N;\Z\pi),\\
\check \CS_q&=\CS_q(\check N;\Z\pi),
\end{align*}
considering the complexes as complexes of $\Z\pi$ modules   \cite{Mil0}. Thus, in particular, considering an orthogonal  representation $\rho:\pi\to O(V)$, we have chain isomorphisms
\begin{align*}
\P_q&:\CS_q(\check N;V_\rho)\to  \CS_{m-q}^\da(N;V_\rho),\\
\QQ_q&:\CS_q(N;V_\rho)\to  \CS_{m-q}^\da(\check N;V_\rho),
\end{align*}
inducing isomorphisms in homology
\begin{align*}
\P_{*,q}&:H_q( \check N;V_\rho) \to H_{m-q}^\da(N;V_\rho),\\
\QQ_{*,q}&:H_q( N;V_\rho) \to H_q^\da(\check N;V_\rho).
\end{align*}

\subsubsection{Duality for a pair}\label{dualpair} Let $(M,N)$ be a simplicial complex pair of dimension $(n=m+1, m)$, that is a triangulation of a compact connected oriented smooth manifold with boundary $(Y,W)$, and  $N$ is a triangulation of the boundary. For the following construction see for example \cite[pg. 164]{RS} and \cite{Mil0}. Each $q$ simplex  $c_q$ of $N$ has a dual block cell $\tilde{\check c}_{m-q}$ of dimension $m-q$ in $N$. Note that the duals cells are not simplexes. Call the family of these cells $\tilde{\check N}$, this is a cell complex of dimension $m$. We have a cell isomorphism
\begin{align*}
\QQ&:N\to \tilde{\check N},\\
\QQ&:c\to \tilde{\check c}.
\end{align*}

Since $N$ is a basis for the chain complex (of $R$ modules) $\CS_\bu(N)$ and $\tilde{\check N}$ is  a basis for $\CS^\da_\bu(\tilde{\check N})$, the cell map $\QQ$ induces an  isomorphism of graded $R$ modules
\[
\QQ_q:\CS_q(N)\to \CS_{m-q}^\da(N).
\]

It is possible to verify that this is a chain map (i.e. that it commutes with the boundary operators), and therefore we have an isomorphism of chain complexes
\[
\QQ_q:\CS_q(N)\to \CS_{m-q}^\da(N).
\]

Each $q$ simplex  $c_q$ of $M$ has a dual block cell $\check c_{n-q}$ of dimension $n-q$. Call the family of these cells $\check M$. The family  $\check M$  is not a cell complex: (if $N$ is not empty) indeed,   the cells of $\check M$ that meet the boundary $N$ have not boundary in $\check M$. The boundary of these cells is in $\check N$. However, we have an isomorphism of sets
\begin{align*}
q&:M\to \check M,\\
q&:c\mapsto \check c.
\end{align*}

If we consider the union 
\[
\check G=\tilde{\check N}\sqcup \check M=\QQ(N)\sqcup q(M), 
\]
this is indeed a cell complex, and we have the isomorphism of cell complexs
\[
\hat \QQ=\QQ\sqcup q:N\sqcup M\to \check G.
\]

The cell complex $\check G$ is a cell decomposition of the geometric realisation $|\check M|$ of $\check M$, and it induces a chain complex
\[
\CS_\bu(\check G),
\]
and the family $\check M$ is a basis for the relative complex 
\[
\CS_\bu(\check G,\tilde{\check N})=\langle \check M\rangle.
\]

Thus the cell isomorphism $\hat \QQ$ induces a chain isomorphism
\[
\hat\QQ_q:\CS_q(M)\to \CS_{m+1-q}^\da(\check G,\tilde{\check N}).
\]

On the other side, we may consider the following  subcomplex of the cell complex $\check G$.  Let $c$ be a simplex in $F=M-N$, then $\check c$ is disjoint from $\tilde{\check N}$. Let $\check F$ denotes the family of such dual cells blocks
\[
\check F=\{\check c~|~c\in M-L\}.
\]

Then $\check F$ is a subset of $\check M$. Moreover, $\hat F$ is a  cell complex, and  a basis for the chain complex  $\CS_q(\check F)$. The restriction of $q$ induces a cell isomorphism
\[
\bar \QQ=q|_{F}:M-N\to \check F.
\]

Since $F=M-N$ is a basis for the relative chain complex 
\[
\CS_\bu(M,N)=\langle M-N\rangle,
\]
the cell map $\bar \QQ$ induces a chain isomorphism
\[
\bar \QQ_q:\CS_q(M,N)\to \CS^\da_{m+1-q}(\check F).
\]

Also, $\check F$ is a triangulation of the geometric realisation $|\check M|$ of $\check M$, and therefore we have the chain isomorphism
\[
\bar \QQ_q:\CS_q(M,N)\to \CS^\da_{m+1-q}(|\check M|).
\]

Back to the manifold, since any two decomposition have a common subdivision, $(\check G, \tilde{\check N})$ is a decomposition for $(Y,W)$, and $\check F$ a decomposition for $Y$. Therefore, taking twisted coefficients by an orthogonal real representation of the fundamental group $\rho:\pi=\pi_1(Y)\to O(V)$, we have  the chain isomorphisms
\begin{align*}
\bar\QQ_q&:\CS_q(Y,W;V_\rho)\to  \CS_{m-q}^\da(Y;V_\rho),\\
\hat \QQ_q&:\CS_q(Y;V_\rho)\to  \CS_{m-q}^\da(Y,W;V_\rho),
\end{align*}
inducing isomorphisms in homology
\begin{align*}
\bar\QQ_{*,q}&:H_q(Y,W;V_\rho) \to H_{m-q}^\da(Y;V_\rho),\\
\hat \QQ_{*,q}&:H_q( Y;V_\rho) \to H_q^\da(Y,W;V_\rho).
\end{align*}

\subsubsection{Duality for the mapping cone}

In this section we extend the previous section to the mapping cone. Let $X=C(W)\sqcup_{W} Y$, where $Y$ is a compact connected orientable smooth manifold of dimension $n=m+1$ with boundary $W$. Let $(M,N)$ a triangulation of $(Y,W)$. We have the cell decomposition  $K=C(N)\sqcup_N M$ of $X$.  We construct a new cell complex $X^*$ that makes the role of the dual decomposition as follows. Let $c$ a cell of $M-N$, then denote by $c^*$ the dual cell $q(c)$ in $\check M$.  The family of these block cells is $\check F$. Let $c$ be a cell of $N$. Then, we have three block cells associated to $c$: the first is $\tilde{\check x}$, the dual block cell of $c$ in $\tilde{\check N}$, the second is $\check c$, the dual block cell of $c$ in $\tilde M$, the third is the $C(\tilde{\check c})$, the cone over $\tilde{\check c}$. The cells $\check c$ and $C(\tilde{\check c})$ have part of the boundary in common, and this part of the boundary is precisely the cell $\tilde{\check c}$. We define the dual block cell $c^*$ of $c$ to be the mapping cone  $C(\tilde{\check c})\sqcup_j \check c$, where $j:\tilde{\check c}\to \check c$ is the inclusion. The family $K^*$ of the dual block cells $c^*$ for $c\in K$ is a cell decomposition of $K$, and we have the bijection 
\[
\QQ:M\to K^*.
\]

This induces a chain isomorphism
\[
\QQ_q:\CS_q(M)\to \CS^\da_{m+1-q}(K^*),
\]
and therefore an isomorphism in homology
\[
\QQ_{*,q}:H_q(|M|)=H_q(\CS_\bu(M))\to H_{m+1-q}^\da(\CS_\bu(K^*))=H_{m+1-q}^\da(|K^*|).
\]

On the other side, the restriction of $\QQ$ on the subcomplex $M-N$ induces a bijection onto the subcomplex $\check F$:
\[
\bar \QQ=\QQ|_{M-N}:M-N\to \check F.
\]

Since $M-N$ is a basis for the chain complex $\CS_\bu(M,N)=\CS_\bu(K,C(N))$ and $\check F$ is a basis for the chain complex $\CS_\bu^\da(\check F)=\CS_\bu^\da(|\check M|)$, we have a chain isomorphism
\[
\bar \QQ_q:\CS_q(M,N)=\CS_q(K,C(N))\to \CS_{m+1-q}^\da(\check F)=\CS_{m+1-q}^\da(|\check M|),
\]
that induces an isomorphism in homology
\[
\bar \QQ_{*,q}:H_q(|K|)=H_q(\CS_\bu(K,C(N)))\to H_{m+1-q}^\da(M)=H_{m+1-q}^\da(\CS_\bu(|\check M|)).
\]

In the particular case where $M=N$, i.e. $K=C(N)$, the cell complex $K^*$ reduces to the family $\check A$ of the cells of the cells complex $C(\tilde {\check N})$, the cone over the dual complex of $N$, that are the cone of the cells of $\tilde{\check N}$. So the previous cell isomorphisms reduce to the bijection of sets
\[
\QQ':N\to \tilde {\check A}.
\]

Observe that $\tilde {\check A}$ is a basis for the relative chain complex $\CS_\bu(C(\tilde {\check N}), \tilde {\check N})$, $\QQ'$ induces the chain isomorphism
\[
\QQ_q':\CS_q(N)\to \CS_{m+1-q}^\da(C(\tilde {\check N}), \tilde {\check N}).
\]

In particular, since $X$ has the same homotopy type  of $Y/W$ (see Lemma \ref{pic2}), taking  twisted coefficients by an orthogonal real representation of the fundamental group $\rho:\pi=\pi_1(X)\to O(V)$, we have  the chain isomorphisms
\begin{align*}
\QQ_q&:\CS_q(Y;V_\rho)\to  \CS_{m-q}^\da(X;V_\rho),\\
\bar \QQ_q&:\CS_q(Y,W;V_\rho)\to  \CS_{m-q}^\da(Y;V_\rho),
\end{align*}
inducing isomorphisms in homology
\begin{align*}
\QQ_{*,q}&:H_q(Y;V_\rho) \to H_{m-q}^\da(X;V_\rho),\\
\hat \QQ_{*,q}&:H_q( X;V_\rho) \to H_q^\da(Y;V_\rho).
\end{align*}

\section{Intersection torsion of pseudomanifolds}

\subsection{The intersection torsion of the cone of a CW complex}\label{sec7.1}

Let $N$ be a  connected finite regular CW complex of dimension $m$, and  $C(N)$ the cone over $N$ as defined in Subsection \ref{s1-1}. 
\begin{lem} Let $N$ be a  connected finite regular CW complex, let $\pf$ be a perversity.  Let  $\CS_\bu(N;\Z)$ denote the cellular (simplicial) chain complex of $N$. Then, the algebraic intersection chain complex $I^\pf \dot \CS_\bu(N,Z)=I^\pf C(\CS_\bu(N;\Z))$ of the cone of $\CS_\bu(N;\Z)$ coincides with the cellular intersection  chain complex $I^\pf \CS_\bu(C(N);\Z)$ of $C(N)$ with coefficients in $\Z$, and the relative abstract intersection chain complex $(I^\pf \dot \CS_\bu(N;\Z),\CS_\bu(N;\Z))$ coincides with the relative cellular intersection chain complex $I^\pf \CS_\bu(C(N),N;\Z)$. Both these complexes are complexes of free finitely generated $\Z$ modules.
\end{lem}
\begin{proof} This follows by direct comparison of the chain modules and of the boundary homomorphisms, as described in details in Lemma \ref{l2} and Lemmas \ref{l4.1} and \ref{l4.3}.
\end{proof}

\begin{lem}\label{l7.3} Let $N$ be a  connected finite regular CW complex, let $\pf$ be a perversity. Then,  the cellular intersection  chain complex $I^\pf \CS_\bu(C(N);\Z)$ has a natural class of equivalent graded chain bases, and this induces a natural graded class of chain basis  for the algebraic intersection chain complex $I^\pf \dot \CS_\bu(N;\Z)$. Representatives of this class are:
\begin{align*}
 I^\pf (\dot\cs_\bu)_q&= \cs_q, \hspace{10pt}q\leq \af-1,& I^\pf (\dot\cs_\bu)_\af&=\zs_{\af-1}\oplus \cs_\af,
 &I^\pf (\dot\cs_\bu)_q&=\dot \cs_q, \hspace{10pt}q\geq \af+1,
\end{align*}
where $\cs_\bu$ is the  graded cell basis of $\CS_\bu(N;\Z)$, $\dot\cs_\bu$ the graded cell basis of $\dot\CS_\bu(N;\Z)$, and $\zs_{\af-1}$ is any integral basis for $\ker \b_{\af-1}:\CS_{\af-1}(N;\Z)\to \CS_{\af-2}(N;\Z)$. In a similar way we define the relative intersection graded chain basis $I^\pf \dot \cs_{\rm rel, \bu}$.
\end{lem}
\begin{proof} As observed in Section \ref{s0}, the complexes $\CS_\bu(N;\Z)$ and $\CS_\bu(C(N);\Z)$ have a natural bases given by the cells. Therefore, the statement is clear in all degrees $q\not=\af-1$. In degree $\af$, since we are working with the ring $\Z$, all submodules are free, and therefore $Z_{\af-1}$ is free. We do not  have  one specific basis, but any two bases are equivalent since $\pi_1(C(N))$ is trivial and thus the  Whitehead group $\Z\pi$ is trivial (in other words, the matrix of the change of basis has integral entries, and therefore trivial Whitehead class). 
\end{proof}

\begin{lem} Let $N$ be a  connected finite regular CW complex, let $\pf$ be a perversity. 
Assume that $H_q(N;\Z)$ is free with preferred graded basis $\hs_q$, then $ H_q (I^\pf C(\CS_\bu);\Z)$ is free (stably free) and has the preferred graded basis $I^\pf \dot\hs_\bu$ defined by
\begin{align*}
I^\pf\dot \hs_q&=\hs_q, ~q\leq \af-2, & I^\pf\dot \hs_q&=\emptyset, ~q\geq\af-1.
\end{align*}

In a similar way we define the relative intersection graded homology basis $I^\pf \dot \hs_{\rm rel, \bu}$.
\end{lem}

\begin{prop} Let $N$ be a  connected finite regular CW complex, let $\pf$ be a perversity. Assume that the homology modules $H_q(N;\Z)$ are free with basis $\hs_q$. Then, the torsion of the cellular intersection  chain complex $I^\pf \CS_\bu(C(N);\Z)$ with the natural graded chain basis $I^\pf \dot\cs_\bu$ induced by the cells described in Lemma \ref{l7.3}, and the  graded homology basis $I^\pf\dot \hs_\bu$ is
\begin{align*}
\tau_W (I^\pf \CS_\bu(C(N);\Z); I^\pf \dot\cs_\bu,  I^\pf\dot\hs_\bu)
=&\tau_W (I^{\pf}C(\CS_{\bu}(N;\Z))_\bu;I^\pf \dot\cs_\bu, I^\pf\dot\hs_\bu)\\
=& \sum_{q=0}^{\af-1}(-1)^q[(\b_{q+1}( \bs_{q+1}) \hat{  \hs}_q  \bs_q/ \cs_q)]\\
&\qquad \qquad+(-1)^\af[(\b_{\af}(\bs_{\af})\hat \hs_{\af-1}/\zs_{\af-1})].
\end{align*}
\end{prop}
\begin{proof} This follows by  Proposition \ref{abstor}. \end{proof}

Consider now the twisted complex of real vector spaces
\[
I^\pf \CS_\bu(C(N);V_\rho)=V\otimes_\rho I^\pf \CS_\bu(C(N);\Z),
\]
where $\rho:\pi_1(C(N))\to O(V)$ is a representation of the fundamental group on a real vector space $V$. Since in the present case the first is trivial, this is the trivial representation and the twisted complex is just the complex $V\otimes I^\pf \CS_\bu(C(N);\Z)$, that will only depend on the representation through its rank, we chose for simplicity the trivial representation $\rho_0$ itself.

\begin{corol}\label{c7.6} The complex $I^\pf \CS_\bu(C(N);V_{\rho_0})$ has a natural class of equivalent graded bases induced by that of the complex $I^\pf \CS_\bu(C(N);\Z)$. The same is true in the relative case, and for the homology and relative homology. We will use the same notation for these bases as above.
\end{corol}

\begin{theo}\label{t7.1} Let $N$ be a  connected finite regular CW complex, let $\pf$ be a perversity.  Let $\ns_\bu$ be the graded standard basis, and  $\hs_\bu$ any other graded basis for the homology of $\CS_\bu(N;V_\rho)$. Then, the torsion of the cellular intersection  chain complex $I^\pf \CS_\bu(C(N);V_{\rho_0})$ with   the natural graded chain basis $I^\pf \dot\cs_\bu$, and the graded homology basis $I^\pf\dot \hs_\bu$ is well defined and given by the following formula
\begin{align*}
\tau_{\rm R} (I^\pf \CS_\bu(C(N);V_{\rho_0}); I^\pf \dot\cs_\bu, I^\pf\dot\hs_\bu)
=& \prod_{q=0}^{\af-2} \left(\det (\hs_q/\ns_q)\right)^{(-1)^q}
\prod_{q=0}^{\af-2} \left(\# TH_q(N;\Z)\right)^{(-1)^q}.
\end{align*}
\end{theo}
\begin{proof} Consider the chain complex $\CS_\bu(N;\Z)$. Since this is a complex of $\Z$ modules, beside the cell basis there exists the standard basis $\es_\bu$ introduced in Appendix  \ref{standardbasis}, and $[(\cs_q/\es_q)]=1$, for all $q$. The same for the induced basis in $\CS_\bu(N;{\rho_0})$, and  therefore for the graded chain bases of $I^\pf \CS_\bu(C(N);V_{\rho_0})$. Moreover, we may consider the corresponding standard bases for homology and for cycles as defined in Appendix \ref{standardbasis}. More precisely, denote by $\ns_q$ the standard basis for the free part of $H_q(N;\Z)$, by $\vs_q$ the standard basis of the cycles in $\CS_q(N;\Z)$ and by $\hs_q$ any fixed basis for $H_q(N;V_{\rho_0})$ (that is free as being a vector space).   With these choices of basis, and choosing the $\bs_q$ as in the appendix, we have that the class in $\R$ of the matrix of the change of basis is
\[
\det (\b_{q+1}(\bs_{q+1}) \hat \ns_q \bs_q/\es_q)=\# T H_q(\CS_\bu).
\]

Therefore, for all $q<\af-1$, 
\[
\det (I^\pf \dot\b_{q+1}(I^\pf \dot\bs_{q+1}) \widehat{I^\pf \dot\ns_q} I^\pf \dot \bs_q/I^\pf \dot\es_q)=\# T H_q(N;\Z).
\]

In degree $q=\af-1$, we have that
\[
\det (I^\pf \dot\b_{\af}(I^\pf \dot\bs_{\af}) \widehat{I^\pf \dot\ns_{\af-1}} I^\pf \dot \bs_{\af-1}/I^\pf \dot\es_{\af-1})
=\det (\b_{\af}(\bs_{\af})  \hat \hs_{\af-1}/\vs_{\af-1}).
\]

As observed in Appendix \ref{standardbasis}, $Z_q$ is uniquely determined, and its basis is given by the first $\dim Z_q$ elements of the basis $\es_q$, whence, choosing again the $\bs_q$ as in the appendix,  we have 
\[
\det (I^\pf \dot\b_{\af}(I^\pf \dot\bs_{\af}) \widehat{ I^\pf \dot\ns_{\af-1}} I^\pf \dot \bs_{\af-1}/I^\pf \dot\es_{\af-1})
=\# T H_{\af-1}(N;\Z),
\]
that cancel out with the last factor in the previous product, giving the stated result.  
\end{proof}

\begin{rem} Observe that the standard basis $\ns_q$ of $H_q(N;\Z)$ may be interpreted geometrically as induced by the cycles, instead that by the cells. In fact, it comes from change of integral bases, where the first elements of the new basis are precisely the cycles representing the free part of homology, see Appendix \ref{standardbasis}.
\end{rem}

\begin{rem} If the dimension of $N$ is odd, $m=2p-1$, $p\geq 1$, then it is easy to see that (compare with \cite{HS3,Spr10})
\begin{align*}
\tau_{\rm R} (I^\mf \CS_\bu(C(N);V_{\rho_0});I^\mf \dot\cs_\bu, I^\mf\dot\hs_\bu)
&=\sqrt{\tau_{\rm R} ( \CS_\bu(N;V_{\rho_0});\cs_\bu, \hs_\bu)}.
\end{align*}
\end{rem}

\begin{prop} If $N'$ is a subdivision of $N$, the inclusion $i:N\to N'$ induces a chain quasi isomorphism $I^\pf C(i_\bu):I^\pf C(\CS_\bu(N;\Z)) \to I^\pf C(\CS_\bu(N',\Z))$, i.e. $I^\pf C(i_{\bu})_{*,\bu}:H_\bu(I^\pf C(\CS_\bu(N;\Z)))\to H_\bu(I^\pf C(\CS_\bu(N';\Z)))$ is an isomorphism.
\end{prop}
\begin{proof} This follows by Proposition \ref{sub}. 
\end{proof}

\begin{corol} Let $N$ be a  connected finite regular CW complex, let $\pf$ be a perversity.  Then, the torsion of the intersection complex of the cone of $N$ is invariant under subdivisions.
\end{corol}

\begin{theo}\label{t7.1.rel} Let $N$ be a  connected finite regular CW complex, let $\pf$ be a perversity.  Let $\ns_\bu$ be the graded standard basis, and  $\hs_\bu$ any other graded basis for the homology of $\CS_\bu(N;V_{\rho_0})$. Then, the torsion of the relative cellular intersection  chain complex $I^\pf \CS_\bu(C(N),N;V_{\rho_0})$ with graded chain basis $I^\pf \dot\cs_{\rm rel,\bu}$ and graded homology basis $I^\pf\dot \hs_{\rm rel, \bu}$ is well defined and given by the following formula
\small{\begin{align*}
\tau_{\rm R} (I^\pf \CS_\bu(C(N),N;V_{\rho_0});I^\pf \dot\cs_{{\rm rel},\bu}, I^\pf\dot\hs_{{\rm rel},\bu})
=& \prod_{q=\af-1}^{m} \left(\det (\hs_{q}/\ns_{q})\right)^{(-1)^{q+1}}
 \left(\# TH_q(N;\Z)\right)^{(-1)^{q+1}}.
\end{align*}}
\end{theo}

We conclude with some results on duality.

\begin{prop}\label{dualcone} Let $N$ be an $m$-dimensional a  connected finite regular CW complex, let $\pf$ be a perversity. 
Then there exists a chain map form the cone $C(N)$ in a dual cell decomposition of $C(N)$ (described in Section \ref{Poincare}, since there is not ambiguity here, we identify $\check N$ with $\tilde{\check N}$)
\[
I^\pf\QQ'_{q}=\left\{\begin{array}{ll} \QQ_q&q\leq \af-1,\\ \QQ'_q,&q\geq \af,\end{array}\right.,
\] 
that induces isomorphism in homology:
\[
I^\pf \QQ'_{*,q}
:H_q(I^\pf \CS_\bu(C(N);V_{\rho_0}))\to H^\da_{m+1-q}(I^{\pf^c}\CS_\bu(C(\check N), \check N;V_{\rho_0})),
\]
\end{prop}
\begin{proof} This is a particular case of Proposition \ref{dualpseudo}. The maps are described at the end of Section \ref{Poincare}. The relevant one is
\[
\QQ'_{\af-1}:(I^\pf\CS_\bu(C(N)))_{\af-1} =\CS_q(N)\to \CS_{m-\af+2=\af^c-1}^\da(C(\check N),\check N)=
(I^{\pf^c} \CS_\bu(C(\check N, \check N))_{\af^c-1}.
\]
\end{proof}

Since homology is invariant under subdivisions, we have the following result.

\begin{corol} If  $N$ and its dual complex $\check N$ have a common subdivision, then 
there is an isomorphism in homology 
\[
I^\pf \QQ'_{*,q}:H_q(I^\pf \CS_\bu(C(N);V_{\rho_0}))\to H^\da_{m+1-q}(I^{\pf^c}\CS_\bu(C( N),  N;V_{\rho_0})).
\]
\end{corol}

\begin{prop}\label{dualtorcone} Let $N$ be an $m$-dimensional a  connected finite regular CW complex, let $\pf$ be a perversity, then 
\[
\tau_{\rm R} (I^\pf \CS_\bu(C(N);V_{\rho_0});I^\pf \dot\cs_\bu, I^\pf\dot\hs_\bu)
=\left(\tau_{\rm R} (I^\pf \CS_\bu(C(\check N),\check N;V_{\rho_0});I^\pf \dot{\check\cs}_{\rm rel, \bu}, I^\pf\dot{\check\hs}_{\rm rel, \bu})\right)^{(-1)^{m}}.
\]
\end{prop}

\begin{proof} We use sum notation. By Theorem \ref{t7.1}, 
\begin{align*}
\tau_{\rm R} (I^\pf \CS_\bu(C(N);V_{\rho_0});I^\pf \dot\cs_\bu, I^\pf\dot\hs_\bu)
=& \sum_{q=0}^{\af-2}(-1)^q [(\hs_q/\ns_q))]
+\sum_{q=0}^{\af-2}(-1)^q[(\b_{q+1}( \bs_{q+1}) \hat{  \ks}_q  \bs_q/ \cs_q)].
\end{align*}

We use the Poincar\'e isomorphism $I^\pf \QQ_q$. This reduces to the classical Poincar\'e isomorphism $\P_q$, when $q\leq \af-2$, and therefore using Lemma \ref{duale} we have

\begin{align*}
\tau_{\rm R} (I^\pf &\CS_\bu(C(N);V_{\rho_0});I^\pf \dot\cs_\bu, I^\pf\dot\hs_\bu)\\
=& \sum_{q=0}^{\af-2}(-1)^{q+1} [(\check \hs_{ m-q}/\check \ns_{ m-q})]
+\sum_{q=0}^{\af-2}(-1)^{q+1}[( \check\b_{m-q+1}( \check\bs_{m-q+1}) \hat{ \check \ns}_{m-q} \check \bs_{m-q}/\check\cs_{m-q})]\\
=& \sum_{q=m}^{m-\af+2}(-1)^{m+q+1} [(\check \hs_{q}/\check \ns_{ q})]
+\sum_{q=m}^{m-\af+2}(-1)^{m+q+1}[( \check\b_{q+1}( \check\bs_{q+1}) \hat{ \check \ns}_{q}  \bs_q/\check\cs_{q})]\\
\end{align*}
recalling that $\af^c=m-\af+3$
\begin{align*}
\tau_{\rm R} (I^\pf &\CS_\bu(C(N);V_{\rho_0});I^\pf \dot\cs_\bu, I^\pf\dot\hs_\bu)\\
=& (-1)^m\sum_{q=\af^c-1}^m(-1)^{q+1} [(\check \hs_{q}/\check \ns_{ q})]
+(-1)^m\sum_{q=\af^c-1}^{m}(-1)^{q+1}[( \check\b_{q+1}( \check\bs_{q+1}) \hat{  \check\hs}_{q}  \check\bs_q/\check\cs_{q})].
\end{align*}

\end{proof}


\subsection{The intersection torsion of  a pseudomanifold}
\label{sec7.2}

Let $K$ be an  $n$-dimensional proper  pseudomanifold  with one isolated singularity $\Sigma=\{x_0\}$. For simplicity we assume that $K$ has no boundary, the analysis in the case of non trivial boundary is analogous. Let $K=M\sqcup_{N_0} C(N_0)$ be the standard decomposition of $K$. Let $\pi=\pi_1(K)$ be the fundamental group of $K$. Since the inclusion of $\pi_1(N)\to \pi_1(K)$ factors through $\pi_1(C(N))$, the  universal covering space $\tilde K$ of $K$ has an induced decomposition
\[
\tilde K=\tilde M\sqcup_{\tilde N} \bigsqcup_{g\in \pi} C(\tilde N_g),
\]
where $N_g$ denotes the boundary of the $g$ sheet of $\tilde M$.

Consider the cellular chain complex $\CS_\bu(\tilde K;\Z)$. The action of $\pi$ onto $\tilde K$ makes each chain group into a module over the group ring $\Z\pi$, and since $K$ is finite, each of these modules is free over $\Z\pi$ and finitely generated by any lift of the  cells of $K$. However, observe that $\pi$ acts trivially  onto the cells of $C(N_0)$, and therefore $\pi$ acts globally  on its lifts without mixing the sheets. In other words, we may identify the $\pi$ action on the following chain complexes 
\[
\CS_q\left(\bigsqcup_{g\in \pi} C(\tilde N_g);\Z\right)=\pi \CS_q(C(N);\Z),
\]
where $\pi  \CS_q(C(N);\Z)$ means  the $\Z\pi$ module with basis the cells of $C(N)$. Using the decomposition of the chain modules introduced in Section \ref{B4} for the mapping cone, we may write 
\[
\CS_q(\tilde K;\Z)= \CS_q( K;\Z\pi)=  \pi\CS_{q-1}(N;\Z)\oplus \CS_q(M;\Z\pi),
\]
where the notation $\CS_q( K;\Z\pi)$ means a $\Z\pi$ modules with basis the lift of the cells of $K$,  we make $\pi$ acting on the universal covering of $M$ as subspace of $\tilde K$. The key point here is to observe that the boundary operator $\ddot\b_q$ of $\CS_q(\tilde K;\Z)$ when restricted to the chains of  the $\bigsqcup_{g\in \pi} C(\tilde N_g)$ does not mix the sheets. Also recall that, by the Seifert - Van Kampen Theorem, $\pi_1(K)$ is the quotient $\pi_1(M)/N(i_*(\pi_1(N_0))$ of $\pi_1(M)$ by the normalised of  $i_*(\pi_1(N_0))$; this means that the cycles of $\pi_1(M)$ homotopic to some image of a cycles of 
$\pi_1(N_0)$ are trivial in $\pi_1(K)$. Using the decomposition of the boundary operator introduced in  Section \ref{B4}
\begin{align*}
\ddot\b_q&=\b_{q-1}^{\CS_\bu}\oplus (i_{q-1}-\b_q^{\DS_\bu})=\left(\begin{matrix}\b^{\CS_\bu}_{q-1}&0\\i_{q-1}&-\b^{\DS_\bu}_{q}\end{matrix}\right),
\end{align*}
where $\CS_\bu=\pi \CS_q(N;\Z)$, and $\DS_\bu=\CS_q(M;\Z\pi)$, what happens is that if $c=x\oplus y\in \CS_q( K;\Z\pi)$, then we may write $c=g t\oplus y$, for some $g\in \pi$, $t\in  \CS_q(N;\Z)$, and then
\begin{align*}
\ddot\b_q(x\oplus y)&=\b_{q-1}^{\CS_\bu}(g t)\oplus (i_{q-1}-\b_q^{\DS_\bu})(y)
=g\b_{q-1}^{\CS_\bu}( t)\oplus (i_{q-1}-\b_q^{\DS_\bu})(y).
\end{align*}

We have proved the following result.

\begin{lem} The cellular chain complex $\CS_\bu(K;\Z\pi)$ (of $\Z\pi$ modules) of the universal covering $\tilde K=\tilde M\sqcup_{\tilde N} \bigsqcup_{g\in \pi} C(\tilde N_g)$ of the mapping cone $K=M\sqcup_{N_0} C(N_0)$ of the inclusion 
$i:N_0\to M$ coincides with the (abstract) mapping cone $\pi C(\CS_\bu(N;\Z))\sqcup_{i_\bu} \CS_\bu(M;\pi\Z)$ of the chain inclusion
$i_\bu:\pi \CS_\bu(N;\Z)\to \CS_\bu(M;\pi\Z)$ of chain complexes of $\Z\pi$ modules.
\end{lem}

We proceed assuming the chain bases described in Section \ref{bb11}: i.e. we denote by $\cs_{q-1}$, $\ds_q$, and $\ds''_q$  the chain bases of $\pi C(\CS_q(N;\Z))=\pi\CS_{q-1}(N;\Z)$,  of $\CS_q(M;\Z\pi)$, and of $\CS_q(M;\Z\pi)/\pi\CS_q(N_0;\Z)=\CS_q(M/N_0;\Z\pi)$ determined by the lifts of the cells, respectively. Observe that for these chain bases, it is true that 
$[(\ds_q/i_q(\cs_q)\hat \ds_q'')]=1$, since we may choose all the lifts in the same sheet. Then, the chain complex $\CS_\bu(K;\Z\pi)$ has preferred graded chain basis $\ddot \cs_\bu$ (determined by lifts of the cells) that using the direct sum decomposition reads: 
$\ddot \cs_q=\cs_{q-1}\oplus 0\ss 0\oplus\ds_q= \cs_{q-1}\oplus 0\ss 0\oplus i_q(\cs_q)\ss 0\oplus \ds_q''$.

\begin{lem}\label{l7.10}  Let $K$ be an $n$-dimensional pseudomanifold without boundary  and with an isolated singularity  $x_0$, and  $\pf$  a perversity.  Let $K=M\sqcup_{N_0} C(N_0)$ be the standard decomposition of $K$. 
Then, the cellular intersection chain complex $I^\pf \CS_\bu(K;\Z\pi)$ coincides with the algebraic intersection chain complex $I^\pf (\pi\CS_\bu(N;\Z))\sqcup_{i_\bu} \CS_\bu(M;\Z\pi)$ of the mapping cone of the  chain inclusion induces by the geometric inclusion $i:N_0\to M$, is a chain complex of free finitely generated $\Z\pi$ modules and is naturally based by the lifts of the cells of $K$ with a preferred class of equivalent graded bases, with representavives:
\begin{align*}
&I^\pf \CS_\bu(K;\Z\pi)_q:&I^\pf \ddot\cs_q&= \ds_q, \hspace{10pt}q\leq \af-1,\\
&I^\pf \CS_\bu(K;\Z\pi)_{\af-1}:&I^\pf \ddot\cs_{\af}&=  \zs_{\af-1}\oplus 0\ss0\oplus \cs_\af, \\
&I^\pf \CS_\bu(K;\Z\pi)_q:&I^\pf \ddot\cs_q&=\cs_{q-1}\oplus 0\ss 0\oplus \ds_q=\cs_{q-1}\oplus 0\ss 0\oplus i_q(\cs_q) \ss 0\oplus \ds_q'', \hspace{10pt}q\geq \af+1.
\end{align*}

\end{lem}
\begin{proof} The first part of the statement follows by the previous considerations and the results of Section \ref{Imappingcone}. The graded cell basis is described in Remark \ref{intbasmapp}. Since $Z_{\af-1}$ is a submodule 
$\pi\CS_{\af-1}(N_0;\Z)$, it is free, as in the proof of Lemma \ref{l7.3}. Moreover, any two bases differs by a mtrix with integer coefficients, and are therefore in the same Whitehead class. 
\end{proof}

We consider now the twisted complex
\[
I^\pf \CS_\bu(K;V_\rho)=V\otimes_\rho I^\pf \CS_\bu(K;\Z\pi)=I^\pf (\CS_\bu(N;V_{\rho_0})\sqcup_{i_\bu} \CS_\bu(M;V_\rho)),
\]
where $\rho:\pi\to O(V)$ is a real orthogonal representation of $\pi$. This is a complex of free finitely generated $\Z\pi$ modules with graded basis induced by the pone described in Lemma \ref{l7.10}, that will be denoted with the same notation.

About the homology, recall that, as shown in Section \ref{bb11},  the projection $p_q: K \to K/C(N_0)$ induces an isomorphism in homology (with any coefficients), that composed with the natural identification of the homologies of $K/C(N_0)$ and $M/N_0$ gives a natural isomorphism of  $H_q(K;V_\rho)$ with $H_q(M,N_0;V_\rho)$. We denote the basis of these vector spaces by $\ddot\hs_q$ and $\hs_q''$. Also, denote by $\hs_q$ and $\hs_q^{\DS_\bu}$ graded bases for the homology vector spaces $H_q(N_0;V_{\rho})$ and $H_q(M;V_\rho)$, respectively, and by $\ks_q$  the standard basis of $H_q(N_0;V_{\rho})$. Also recall the long exact sequence 
\[
\xymatrix{
\Ha:&  \dots\ar[r]&H_q(N_0;V_{\rho})\ar[r]^{i_{*,q}}&H_q(M;V_\rho)\ar[r]^{p_{*,q}}&H_q(M,N_0;V_{\rho})\ar[r]&\dots
}
\]
induced by 
\[
\xymatrix{
0\ar[r]&\CS_\bu(N_0;V_\rho)\ar[r]^{i_\bu}&\CS_\bu(M;V_\rho)\ar[r]^{p_\bu}&\CS_\bu(M,N_0;V_\rho)\ar[r]&0,
}
\]

Then, by Remark \ref{intbasmapp}, a preferred basis $I^\pf \ddot\hs_\bu$ for the intersection homology of $K$ is:
\begin{align*}
I^\pf\ddot\hs_q&=\hs^{\DS_\bu}_q, ~q\leq \af-2, & I^\pf\ddot\hs_q&=p_{*,\af-1}(\hs^{\DS_\bu}),
&I^\pf\ddot\hs_q&=\hs_\bu'', ~q\geq\af.
\end{align*}

\begin{theo}\label{t7.2} Let $K$ be a pseudomanifold without boundary  and with an isolated singularity  $x_0$, and  $\pf$  a perversity.  Let $K=M\sqcup_{N_0} C(N_0)$ be the standard decomposition of $K$. Let $\rho:\pi_1(K)\to O(V)$ be some orthogonal representation of the fundamental group of $K$. Assume the chain bases and the homology bases defined above. Then,  the torsion of the cellular intersection  chain complex $I^\pf \CS_\bu(K;V_\rho)$ with graded chain basis $I^\pf \ddot\cs_\bu$ and graded homology basis $I^\pf\ddot \hs_\bu$ is well defined and given by the following formula
\begin{align*}
\tau_{\rm R}(I^\pf \CS_\bu(K;V_\rho);I^\pf \ddot\cs_\bu,I^\pf \ddot\hs_\bu)
=&\tau_{\rm R} (I^\pf \CS_\bu(C(N);V_\rho);I^\pf \dot\cs_\bu, I^\pf\dot\hs_\bu)
\tau_{\rm R}(\CS_\bu(M,N_{0};V_\rho);\ds_\bu'',\hs_\bu'')\\
&\tau_{\rm R}(I^\pf \ddot\Ha;I^\pf \dot\hs_\bu, I^\pf \ddot\hs_\bu,\hs_\bu''),
\end{align*}
where $\tau_{\rm R}(I^\pf \ddot\Ha;I^\pf \dot\hs_\bu, I^\pf \ddot\hs_\bu,\hs_\bu'')$ is the torsion of the long exact sequence
$I^\pf \ddot\Ha:$
\[
\xymatrix{
  \dots\ar[r]&H_q(I^\pf \CS_\bu(C(N);V_\rho))\ar[r]&H_q((I^\pf \CS_\bu(K;V_\rho)\ar[r]&H_q(K,N;V_{\rho})\ar[r]&\dots,
}
\]
with graded bases $I^\pf \dot\hs_\bu$, $I^\pf \ddot\hs_\bu$, and $\hs_\bu''$. 
\end{theo}
\begin{proof} Proposition \ref{p5.7}. 
\end{proof}

\begin{rem} Here is an equivalent formula for torsion
\begin{align*}
\tau_{\rm R}(I^\pf \CS_\bu(K;V_\rho);I^\pf \ddot\cs_\bu,I^\pf \ddot\hs_\bu)
=&\tau_{\rm R}(I^\pf \CS_\bu(C(N_0);V_{\rho_0});I^\pf \dot\cs_\bu,I^\pf \dot\hs_\bu)\tau_{\rm R}(\CS_\bu(M,N_{0});\ds_\bu'',\hs_\bu'')\\
&\tau_{\rm R}(I^\pf \ddot\Ha;I^\pf \dot\hs_\bu, I^\pf \ddot\hs_\bu,\hs_\bu''),
\end{align*}
where $\tau_{\rm R}(I^\pf \ddot\Ha;I^\pf \dot\hs_\bu, I^\pf \ddot\hs_\bu,\hs_\bu'')$ is the following part
\begin{align*}
\tau_{\rm R}(I^\pf \ddot\Ha;I^\pf \dot\hs_\bu, I^\pf \ddot\hs_\bu,\hs_\bu'')
=&|\det(p_{*,\af-1}(\hs^{\DS_\bu}_{\af-1})  \ys_{\af-1}''/\hs_{\af-1}'')|^{(-1)^{\af-1}}\\
&\prod_{q=0}^{\af-2}\left|\frac{\det(\delta_{q+1}(\ys_{q+1}'')\ys_{q}'/\hs_{q})\det ( p_{*,q}(\ys_q) \ys_q''/\hs_q'')}
{\det(i_{*,q}(\ys_q') \ys_q/\hs_q^{\DS_\bu})}\right|^{(-1)^q},
\end{align*}
of the torsion of the long exact sequence
\[
\xymatrix{
\Ha:&  \dots\ar[r]&H_q(N_0;V_{\rho})\ar[r]^{i_{*,q}}&H_q(M;V_\rho)\ar[r]^{p_{*,q}}&H_q(M,N_0;V_{\rho})\ar[r]&\dots.
}
\]

Using the calculations in Section \ref{ladder1}, we identify the term coming from the sequence $I^\pf \ddot\Ha$ with the corresponding one in the sequence $\Ha$. 
Next, we use Theorem \ref{t7.1} to complete the proof. 
\end{rem}

\begin{corol} Let $K$ be an $n$-dimensional pseudomanifold without boundary  and with an isolated singularity  $x_0$, and  $\pf$  a perversity.  Then, the torsion of the intersection complex $I^\pf \CS_\bu(K;V_\rho)$ of $K$  with twisted coefficients $V_\rho$ is invariant under subdivisions.
\end{corol}

We conclude this section with some results on duality.

\begin{prop}\label{dualpseudo} Let $K$ be an $n$-dimensional pseudomanifold without boundary  and with an isolated singularity, and  $\pf$  a perversity. Then, the map 
\[
I^\pf \QQ_{q}=\left\{\begin{array}{ll} \QQ_q&q\leq \af-1,\\\bar \QQ_q,&q\geq \af,\end{array}\right.:K\to K^*,
\]
from $K$ to its dual $K^*$, induces isomorphism in homology
\[
I^\pf\QQ_{*,q}:H_q(I^\pf \CS_\bu(K;V_\rho))\to H^\da_{n-q}(I^{\pf^c}\CS_\bu(K^*;V_\rho)).
\]
\end{prop}

\begin{proof}

Let $K=C(N)\sqcup_N M$ the canonical decomposition of $K$. Using the maps $\QQ_\bu$ introduced in Section \ref{Poincare}, and the definition of the intersection complex, we have the commutative diagram
\[
\xymatrix@C=0.4cm{
\dots\ar[r]&(I^\pf \CS_\bu(K;V_\rho))_{\af+1}\ar[r]\ar[d]_=&(I^\pf \CS_\bu(K;V_\rho))_{\af}\ar[r] \ar[d]_=&(I^\pf \CS_\bu(K;V_\rho))_{\af-1} \ar[r]\ar[d]_=&\dots
\\
\dots\ar[r]&\CS_{\af+1}(M;V_\rho)\ar[r]\ar[d]_{ \QQ_{\af+1}}&\CS_{\af}(M;V_\rho)\ar[r]\ar[d]_{ \QQ_{\af}}&\CS_{\af-1}(M,N;V_\rho)\ar[r]\ar[d]_{ \bar\QQ_{\af-1}}&\dots
\\
\dots\ar[r]&\CS^\da_{m+1-\af-1}(K^*;V_\rho)\ar[r]\ar[d]^=&\CS^\da_{m+1-\af}(K^*;V_\rho)\ar[r]\ar[d]_=&\CS^\da_{m+1-\af+1}(\check M;V_\rho)\ar[r]\ar[d]_=&\dots\\
\dots\ar[r]&(I^{\pf^c}\CS_\bu(K^*;V_\rho))^\da_{\af^c-3}\ar[r]&(I^{\pf^c}\CS_\bu(K^*;V_\rho))^\da_{\af^c-2}\ar[r]&(I^{\pf^c}\CS_\bu(K^*;V_\rho))^\da_{\af^c-1}\ar[r]&\dots
}
\]

The maps $\QQ_\bu$ induces isomorphisms in homology, this gives the result in all degree different from $\af-1$. In that degree, composing the homology sequences of the two pairs, we have the commutative diagram 
\[
\xymatrix{
\dots\ar[r]&H_{\af-1}(M)\ar[d]_{ \QQ_{*,\af-1}}\ar[r]^{p_{*,\af-1}}&H_{\af-1}(M,N)\ar[r]\ar[d]^{\bar \QQ_{*,\af-1}}&\dots\\
\dots&H^\da_{\af-1}(M,N)\ar[l]&H^\da_{\af-1}(M,N)\ar[l]^{p_{*,\af-1}^\da}&\ar[l]\dots
}
\]
that shows that the map $\bar \QQ_{*,\af-1}$ is an isomorphism also in degree $\af-1$. 
\end{proof}

Since homology is invariant under subdivisions, we have the following result.

\begin{corol} There is an isomorphism in homology 
\[
I^\pf \QQ_{*,q}:H_q(I^\pf \CS_\bu(K;V_\rho))\to H^\da_{n-q}(I^{\pf^c}\CS_\bu(K^*;V_\rho)).
\]
\end{corol}

\begin{prop} \label{p7.24}
\begin{align*}
\tau_{\rm R}(I^\pf \CS_\bu(K;V_\rho);I^\pf \ddot\cs_\bu,I^\pf \ddot\hs_\bu)
=
\left(\tau_{\rm R}(I^{\pf^c} \CS_\bu(K^*;V_\rho);I^{\pf^c} \ddot\cs^*_\bu,I^{\pf^c} \ddot\hs^*_\bu)\right)^{(-1)^m}.
\end{align*}
\end{prop}

\begin{proof} By standard duality for torsion (use sum notation)
\[
\tau_{\rm R}(\CS_\bu(M,N_{0};V_\rho);\ds_\bu'',\hs_\bu'')=(-1)^{m} \tau_{\rm R}(\CS_\bu(\check M),V_\rho);\check \ds_\bu,\check \hs_\bu^{\DS_\bu}),
\]
where the bases are defined in Section \ref{ladder1}. By duality of intersection torsion of the cone, Proposition \ref{dualcone}, 
\[
\tau_{\rm R} (I^\pf \CS_\bu(C(N);V_\rho);I^\pf \dot\cs_\bu, I^\pf\dot\hs_\bu)
={(-1)^{m}}\tau_{\rm R} (I^{\pf^c} \CS_\bu(C(\check N),\check N;V_\rho);I^{\pf^c} \dot{\check\cs}_{\rm rel, \bu}, I^{\pf^c}\dot{\check\hs}_{\rm rel, \bu}).
\]

So

\begin{align*}
\tau_{\rm R}(I^\pf \CS_\bu(K;V_\rho);I^\pf \ddot\cs_\bu,I^\pf \ddot\hs_\bu)
=&\tau_{\rm R}(I^\pf \CS_\bu(C(N_0);V_{\rho_0});I^\pf \dot\cs_\bu,I^\pf \dot\hs_\bu)+\tau_{\rm R}(\CS_\bu(M,N_{0});\ds_\bu'',\hs_\bu'')\\
&+\tau_{\rm R}(I^\pf \ddot\Ha;I^\pf \dot\hs_\bu, I^\pf \ddot\hs_\bu,\hs_\bu'')\\
=&(-1)^m \tau_{\rm R} (I^{\pf^c} \CS_\bu(C(\check N),\check N;V_\rho);I^{\pf^c} \dot{\check\cs}_{\rm rel, \bu}, I^{\pf^c}\dot{\check\hs}_{\rm rel, \bu})\\
&+(-1)^m\tau_{\rm R}(\CS_\bu(\check M),V_\rho;\check \ds_\bu,\check \hs_\bu^{\DS_\bu})
+\tau_{\rm R}(I^\pf \ddot\Ha;I^\pf \dot\hs_\bu, I^\pf \ddot\hs_\bu,\hs_\bu'').
\end{align*}

By Proposition \ref{milnorcone},
\begin{align*}
\tau_{\rm R} (I^{\pf^c} \CS_\bu(C(\check N),\check N&;V_\rho);I^{\pf^c} \dot{\check\cs}_{\rm rel, \bu}, I^{\pf^c}\dot{\check\hs}_{\rm rel, \bu})\\
&=\tau_{\rm R} (I^{\pf^c} \CS_\bu(C(\check N);V_\rho);I^{\pf^c} \dot{\check\cs}_{ \bu}, I^{\pf^c}\dot{\check\hs}_{ \bu})
-\tau_{\rm R}(\CS_\bu(\check N;V_\rho);\check\cs_\bu,\check \hs_\bu),
\end{align*}
by Proposition \ref{standarddualitytorsion},
\[
\tau_{\rm R}(\CS_\bu(\check M),V_\rho;\check \ds_\bu,\check \hs_\bu^{\DS_\bu})
=\tau_{\rm R}(\CS_\bu(\check M,\check N_{0});\check \ds_\bu'',\check \hs_\bu'')+\tau_{\rm R}(\CS_\bu(\check N;V_\rho);\check\cs_\bu,\check \hs_\bu)
+\tau(\Ha;\check \hs_\bu, \check \hs_\bu^{\DS_\bu}, \check \hs_\bu'').
\]
thus
\begin{align*}
\tau_{\rm R}(I^\pf \CS_\bu(K;V_\rho);I^\pf \ddot\cs_\bu,I^\pf \ddot\hs_\bu)
=&(-1)^m \left(\tau_{\rm R} (I^{\pf^c} \CS_\bu(C(\check N);V_\rho);I^{\pf^c} \dot{\check\cs}_{ \bu}, I^{\pf^c}\dot{\check\hs}_{ \bu})
+ \tau(\Ha;\check \hs_\bu, \check \hs_\bu^{\DS_\bu}, \check \hs_\bu'')\right)\\
&+(-1)^m\tau_{\rm R}(\CS_\bu(\check M,\check N_{0});\check \ds_\bu'',\check \hs_\bu'')
+\tau_{\rm R}(I^\pf \ddot\Ha;I^\pf \dot\hs_\bu, I^\pf \ddot\hs_\bu,\hs_\bu'').
\end{align*}

Consider the diagram 
\[
\centerline{
\xymatrix@C=0.4cm{
&\ys_q',\hs_q&\ys_q, \hs_q^{\DS_\bu}&\ys''_q,\hs_q''&\ys_{q-1}',\hs_{q-1}\\
\dots \ar[r]&H_q(L;V_\rho)\ar[r]^{i_{*,q}}\ar[d]_{\QQ_q}&H_q(M;V_\rho)\ar[r]^{p_{*,q}}\ar[d]_{\hat\QQ_q}&H_q(M,L;V_\rho)\ar[r]^{\delta_q}\ar[d]_{\bar\QQ_q}&H_{q-1}(N;V_\rho)\ar[r]\ar[d]_{\QQ_{q-1}}&\dots\\
\dots \ar[r]&H^\da_{m-q}(\check L;V_\rho)\ar[r]_{\check \delta^\da_{m+1-q}}&H^\da_{m+1-q}(\check G,\check N;V_\rho)\ar[r]^{\check p^\da_{*,m+1-q}}&H^\da_{m+1-q}(\check M;V_\rho)\ar[r]^{\check i^\da_{*, m+1-q}}&H^\da_{m+1-q}(\check N;V_\rho)\ar[r]&\dots\\
&\xs'_{m-q}, \check \hs_{m-q}^\da& \xs''_{m+1-q}, (\check \hs_{m+1-q}'')^\da& \xs_{m+1-q}, (\check \hs_{m+1-q}^{\DS_\bu})^\da
&\xs'_{m+1-q}, \check \hs_{m+1-q}^\da}}
\]
where in the first and in the last lines are indicated elements whose images are a basis for the image of the following map, and an homology basis. In the last lines, the homology bases have been chosen as the dual of the corresponding homology bases.

By commutativity of the middle square of the diagram
$
\bar Q_q p_{*,q}(\ys_q)=\check p_{*,m+1-q}^\da\hat Q_q(\ys_q).
$
 This allow us to fix the following set of elements in $H^\da_{m+1-q}(\check M)$:
$
\check \ys_{m+1-q}^\da=\bar Q_q p_{*,q}(\ys_q).
$
 The notation means that we can choose the  elements $\check \ys_{m+1-q}$ in $H_{m+1-q}(\check M)$ in such a way that their duals satisfy the stated equation. Then, 
$
1=\check p_{*,m+1-q}^\da\hat Q_q(\ys_q)(\check \ys_{m+1-q}), 
$
the $\hat Q_q(\ys_q)$ are the duals of the $\check p_{*,m+1-q}(\check \ys_{m+1-q})$, i.e.
$
\hat Q_q(\ys_q)=(\check p_{*,m+1-q}(\check \ys_{m+1-q}))^\da.
$
 In a similar way, commutativity of the left square gives
$
\hat Q_q i_{*,q}(\ys_q')=\check \delta^\da_{m+1-q} Q_q(\ys_q'),
$ 
so fixing 
$
(\check \ys_{m+1-q}'')^\da=\hat Q_q i_{*,q}(\ys_q'),
$ 
in $H^\da_{m+1-q}(\check G, \check N)$, since 
$
1=\check \delta^\da_{m+1-q} Q_q(\ys_q')(\check \ys_{m+1-q}''), 
$ 
we have that
$
Q_q(\ys_q')=(\check \delta_{m+1-q} (\check \ys_{m+1-q}''))^\da.
$
 And eventually,
$
Q_{q-1}\delta_q(\ys_q'')=\check i^\da_{*,m+1-q}\bar Q_q(\ys_q''),
$ 
so fixing 
$
(\check \ys_{m+1-q}')^\da=Q_{q-1}\delta_q(\ys_q''),
$  
in $H^\da_{m+1-q}(\check N)$, since 
$
1=\check i^\da_{*,m+1-q}\bar Q_q(\ys_q'')(\check \ys_{m+1-q}'), 
$ 
it follows that
$
\bar Q_q(\ys_q'')=(\check i^\da_{*,m+1-q}(\check \ys_{m+1-q}'))^\da.
$

We may now proceed to deal with torsion. We may choose the $\xs$ as follows:
\begin{align*}
\xs_{m+1-q}&=\bar Q(\ys_q''),&
\xs_{m-1}'&=Q_q(\ys_q'),&
\xs_{m+1-q}''&=\hat Q_q(\ys_q).
\end{align*}

Thus, for the matrix of the change of basis in $H^\da_{m+1-q}(\check N;V_\rho)$, form one side
\begin{align*}
\left( \check i_{*,m+1-q}^\da(\xs_{m+1-q}) \xs'_{m+1-q}/\check\hs_{m+1-q}^\da\right)
&=\left( \check i_{*,m+1-q}^\da(\bar Q_q(\ys_q'')) Q_{q-1}(\ys_{q-1}')/Q_{q-1}(\hs_{q-1})\right)\\
&=\left( \bar Q_{q-1}\delta_q(\ys_q'') Q_{q-1}(\ys_{q-1}')/Q_{q-1}(\hs_{q-1})\right)\\
&=\left( \delta_q(\ys_q'')\ys_{q-1}'/\hs_{q-1}\right);
\end{align*}
from the other
\begin{align*}
\left( \check i_{*,m+1-q}^\da(\xs_{m+1-q}) \xs'_{m+1-q}/\check\hs_{m+1-q}^\da\right)
&=\left( (\check\ys'_{m+1-q})^\da (\check \delta_{m+2-q}(\check\ys_{m+2-q}''))^\da/\check \hs^\da_{m+q-1}\right)\\
&=\left(\left( \check\ys'_{m+1-q} \check \delta_{m+2-q}(\check\ys_{m+2-q}'')/\check \hs_{m+1-q}\right)^T\right)^{-1}.
\end{align*}

For the matrix of the change of basis in $H^\da_{m+1-q}(\check M;V_\rho)$:
\begin{align*}
\left( \check p_{*,m+1-q}^\da(\xs_{m+1-q}'') \xs_{m+1-q}/(\check\hs^{\DS_\bu}_{m+1-q})^\da\right)
&=\left( \check p_{*,m+1-q}^\da(\hat Q_q(\ys_q)) \bar Q_{q}(\ys_q'')/\bar Q_q(\hs''_{q})\right)\\
&=\left( \bar Q_q p_{*,q}(\ys_q) \bar Q_{q}(\ys_q'')/\bar Q_q(\hs''_{q})\right)\\
&=\left( p_{*,q}(\ys_q) \ys_q''/\hs''_{q}\right)\\
&=\left( \check \ys_{m+1-q}^\da (\check i_{*,m+1-q}(\check \ys_{m+1-q}'))^\da/(\check\hs^{\DS_\bu}_{m+1-q})^\da\right)\\
&=\left( \left(\check \ys_{m+1-q} \check i_{*,m+1-q}(\check \ys_{m+1-q}')/\check\hs^{\DS_\bu}_{m+1-q}\right)^T\right)^{-1}.
\end{align*}

For the matrix of the change of basis in $H^\da_{m+1-q}(\check G,\check N;V_\rho)$:
\begin{align*}
\left( \check \delta_{*,m+1-q}^\da(\xs_{m+1-q}') \xs_{m+1-q}''/(\check\hs''_{m+1-q})^\da\right)
&=\left( \check \delta_{*,m+1-q}^\da(Q_q(\ys_q')) \hat Q_q(\ys_q)/\hat Q_q (\hs_q^{\DS_\bu})\right)\\
&=\left( \hat Q_q i_{*,q}(\ys_q') \hat Q_q(\ys_q)/\hat Q_q (\hs_q^{\DS_\bu})\right)\\
&=\left(  i_{*,q}(\ys_q') \ys_q/\hs_q^{\DS_\bu}\right)\\
&=\left( (\check \ys_{m+1-q}'')^\da (\check p_{*,m+1-q}(\check \ys_{m+1-q}))^\da/(\check\hs''_{m+1-q})^\da\right)\\
&=\left( \left(\check \ys_{m+1-q}''\check p_{*,m+1-q}(\check \ys_{m+1-q})/\check\hs''_{m+1-q}\right)^T\right)^{-1}.
\end{align*}

We may compute the torsion of the dual homology complex. 
\begin{align*}
\tau(\Ha&;\check \hs_\bu, \check \hs_\bu^{\DS_\bu}, \check \hs_\bu'')\\
=&\sum_{q=0}^{m+1} (-1)^q\left( [( \check \delta_{q+1}(\check\ys_{q+1}'')\check\ys'_{q} /\check \hs_{q})]
-[(\check i_{*,q}(\check \ys_{q}')\check \ys_{q} /\check\hs^{\DS_\bu}_{q})]
+[(\check p_{*,q}(\check \ys_{q})\check \ys_{q}''/\check\hs''_{q})]\right)\\
=&\sum_{q=0}^{m+1} (-1)^{q+1}\left([( \delta_{m+1-q}(\ys_{m+1-q}'')\ys_{m-q}'/\hs_{m-q})]
-[( p_{*,m+1-q}(\ys_{m+1-q}) \ys_{m+1-q}''/\hs''_{m+1-q})]\right.\\
&\left.+[(  i_{*,m+1-q}(\ys_{m+1-q}) \ys_{m+1-q}/\hs_{m+1-q}^{\DS_\bu})]
\right)\\
=&(-1)^m\sum_{q=0}^{m+1} (-1)^{q}\left([( \delta_{q}(\ys_{q}'')\ys_{q-1}'/\hs_{q-1})]
-[( p_{*,q}(\ys_{q}) \ys_{q}''/\hs''_{q})]
+[(  i_{*,q}(\ys_{q}) \ys_{q}/\hs_{q}^{\DS_\bu})]
\right)
\end{align*}

\begin{align*}
\tau_{\rm R}(I^\pf \ddot\Ha;&I^\pf \dot\hs_\bu, I^\pf \ddot\hs_\bu,\hs_\bu'')\\
&=(-1)^{\af-1}[(p_{*,\af-1}(\hs^{\DS_\bu}_{\af-1})  \ys_{\af-1}''/\hs_{\af-1}'')]\\
&+\sum_{q=0}^{\af-2}(-1)^q\left([(\delta_{q+1}(\ys_{q+1}'')\ys_{q}'/\hs_{q})]
-[(i_{*,q}(\ys_q') \ys_q/\hs_q^{\DS_\bu})]
+[( p_{*,q}(\ys_q) \ys_q''/\hs_q'')]\right)\\
=&(-1)^{\af}[(\check \ys_{m+2-\af} \check i_{*,m+2-\af}(\check \ys_{m+2-\af}')/\check\hs^{\DS_\bu}_{m+2-\af})]\\
&+\sum_{q=0}^{\af-2}(-1)^{q+1}\left([( \check\ys'_{m-q} \check \delta_{m+1-q}(\check\ys_{m+1-q}'')/\check \hs_{m-q})]\right.\\
&\left.-[(\check \ys_{m+1-q}''\check p_{*,m+1-q}(\check \ys_{m+1-q})/\check\hs''_{m+1-q})]
+[(\check \ys_{m+1-q} \check i_{*,m+1-q}(\check \ys_{m+1-q}')/\check\hs^{\DS_\bu}_{m+1-q})]
\right)\\
=&(-1)^{\af}[(\check \ys_{m+2-\af} \check i_{*,m+2-\af}(\check \ys_{m+2-\af}')/\check\hs^{\DS_\bu}_{m+2-\af})]\\
&+(-1)^m\sum_{q=m+3-\af}^{m+1}(-1)^{q}\left([( \check\ys'_{q-1} \check \delta_{q}(\check\ys_{q}'')/\check \hs_{q-1})]
-[(\check \ys_{q}''\check p_{*,q}(\check \ys_{q})/\check\hs''_{q})]
+[(\check \ys_{q} \check i_{*,q}(\check \ys_{q}')/\check\hs^{\DS_\bu}_{q})]
\right).
\end{align*}

So 
\begin{align*}
(-1)^m &\tau(\Ha;\check \hs_\bu, \check \hs_\bu^{\DS_\bu}, \check \hs_\bu'')+\tau_{\rm R}(I^\pf \ddot\Ha;I^\pf \dot\hs_\bu, I^\pf \ddot\hs_\bu,\hs_\bu'')\\
=&(-1)^m\sum_{q=0}^{m+1} (-1)^q\left( [( \check \delta_{q+1}(\check\ys_{q+1}'')\check\ys'_{q} /\check \hs_{q})]
-[(\check i_{*,q}(\check \ys_{q}')\check \ys_{q} /\check\hs^{\DS_\bu}_{q})]
+[(\check p_{*,q}(\check \ys_{q})\check \ys_{q}''/\check\hs''_{q})]\right)\\
&+(-1)^{\af}[(\check \ys_{m+2-\af} \check i_{*,m+2-\af}(\check \ys_{m+2-\af}')/\check\hs^{\DS_\bu}_{m+2-\af})]\\
&+(-1)^m\sum_{q=m+3-\af}^{m+1}(-1)^{q}\left([( \check\ys'_{q-1} \check \delta_{q}(\check\ys_{q}'')/\check \hs_{q-1})]
-[(\check \ys_{q}''\check p_{*,q}(\check \ys_{q})/\check\hs''_{q})]
+[(\check \ys_{q} \check i_{*,q}(\check \ys_{q}')/\check\hs^{\DS_\bu}_{q})]
\right)\\
=&(-1)^m\sum_{q=0}^{m+1} (-1)^q\left( [( \check \delta_{q+1}(\check\ys_{q+1}'')\check\ys'_{q} /\check \hs_{q})]
-[(\check i_{*,q}(\check \ys_{q}')\check \ys_{q} /\check\hs^{\DS_\bu}_{q})]
+[(\check p_{*,q}(\check \ys_{q})\check \ys_{q}''/\check\hs''_{q})]\right)\\
&+(-1)^{\af}[(\check \ys_{m+2-\af} \check i_{*,m+2-\af}(\check \ys_{m+2-\af}')/\check\hs^{\DS_\bu}_{m+2-\af})]\\
&+(-1)^{1-\af}[( \check\ys'_{m+2-\af} \check \delta_{m+3-\af}(\check\ys_{m+3-\af}'')/\check \hs_{m+2-\af})]\\
&+(-1)^m\sum_{q=m+3-\af}^{m+1}(-1)^{q}\left([( \check\ys'_{q} \check \delta_{q+1}(\check\ys_{q+1}'')/\check \hs_{q})]
-[(\check \ys_{q}''\check p_{*,q}(\check \ys_{q})/\check\hs''_{q})]
+[(\check \ys_{q} \check i_{*,q}(\check \ys_{q}')/\check\hs^{\DS_\bu}_{q})]
\right),
\end{align*}
continuing
\begin{align*}
(-1)^m &\tau(\Ha;\check \hs_\bu, \check \hs_\bu^{\DS_\bu}, \check \hs_\bu'')+\tau_{\rm R}(I^\pf \ddot\Ha;I^\pf \dot\hs_\bu, I^\pf \ddot\hs_\bu,\hs_\bu'')\\
=&(-1)^m\sum_{q=0}^{m+2-\af} (-1)^q\left( [( \check \delta_{q+1}(\check\ys_{q+1}'')\check\ys'_{q} /\check \hs_{q})]
-[(\check i_{*,q}(\check \ys_{q}')\check \ys_{q} /\check\hs^{\DS_\bu}_{q})]
+[(\check p_{*,q}(\check \ys_{q})\check \ys_{q}''/\check\hs''_{q})]\right)\\
&+(-1)^{\af}[(\check \ys_{m+2-\af} \check i_{*,m+2-\af}(\check \ys_{m+2-\af}')/\check\hs^{\DS_\bu}_{m+2-\af})]\\
&+(-1)^{1-\af}[( \check\ys'_{m+2-\af} \check \delta_{m+3-\af}(\check\ys_{m+3-\af}'')/\check \hs_{m+2-\af})]\\
=&(-1)^m\sum_{q=0}^{m+1-\af} (-1)^q\left( [( \check \delta_{q+1}(\check\ys_{q+1}'')\check\ys'_{q} /\check \hs_{q})]
-[(\check i_{*,q}(\check \ys_{q}')\check \ys_{q} /\check\hs^{\DS_\bu}_{q})]
+[(\check p_{*,q}(\check \ys_{q})\check \ys_{q}''/\check\hs''_{q})]\right)\\
&+(-1)^{2+\af}[(\check p_{*,m+2-\af}(\check \ys_{m+2-\af})\check \ys_{m+2-\af}''/\check\hs''_{m+2-\af})]\\
=&(-1)^m\left(\sum_{q=0}^{\af^c-2}(-1)^q\left( [( \check \delta_{q+1}(\check\ys_{q+1}'')\check\ys'_{q} /\check \hs_{q})]
-[(\check i_{*,q}(\check \ys_{q}')\check \ys_{q} /\check\hs^{\DS_\bu}_{q})]
+[(\check p_{*,q}(\check \ys_{q})\check \ys_{q}''/\check\hs''_{q})]\right)\right.\\
&\left.+(-1)^{\af^c-1}[(\check p_{*,\af^c-1}(\check \ys_{\af^c-1})\check \ys_{\af^c-1}''/\check\hs''_{\af^c-1})]\right).
\end{align*}
\end{proof}

\section{Intersection homology of  cones  and spaces with conical singularities}
\label{ss8}

Let $(W,g)$ be a compact connected oriented Riemannian manifold without boundary. 
Let $C(W)$ be the cone over $W$. Since $W$ is a smooth manifold, it admits a smooth triangulation, and hence it is a regular CW complex. Moreover, any two such triangulations admit a common subdivision. It follows that we may select any one triangulation and all the result of the previous sections hold. In particular, using the results of Section \ref{sec7.1}, we may introduce the following definitions and prove the following results.

\begin{defi}\label{homcone} Let $(W,g)$ be a compact connected  manifold of dimension $m$ without boundary. Let $\pf$ be a perversity. Let $\rho:\pi_1(W)\to O(V)$. Let $N$ be any regular  cellular decomposition of $W$. Let $C(W)$ the cone over $W$. We call {\it intersection homology groups of $C(W)$ with perversity $\pf$ and coefficients in $V_\rho$}, the groups 
$I^\pf H^q(C(W);V_\rho)= H^q(I^\pf \CS_\bu(C(N);V_\rho)$. The relative groups are defined similarly.
\end{defi}

It is clear that these groups do not depend on the cellular decomposition $N$. Formulas for these groups are given in Lemma \ref{homo}.

\begin{theo}\label{dualhomC} Let $(W,g)$ be a compact connected manifold of dimension $m$ without boundary, of dimension $m$. Let $\pf$ be a perversity.   Let $C(W)$ the cone over $W$. Then, there is an isomorphism 
\[
I^\pf  Q'_{*,q}:I^\pf H^q(C(W);V_{\rho_0})\to (I^{\pf^c} H_{m-q})^\da(C(W), W;V_{\rho_0}).
\]
\end{theo}

Let $X$ be a space with a conical singularity, and dimension $n$, as defined in Section \ref{spaceX}. Then, $X$ is the smooth glueing of a manifold with boundary $(Y,W)$, with the cone over $W$. So, there exists a cellular decomposition $K$ of $X$, such that $K$ is an $n$ pseudomanifold with one isolated singularity (the tip of the cone). We have the standard decomposition $K=M\sqcup_{N_0} C(N_0)$.

\begin{defi} Let $X=Y\sqcup_W C(W)$ be a space with a conical singularity. A cellular decomposition $K$ of $X$ is said to be coherent if in the standard decomposition $K=M\sqcup_{N_0} C(N_0)$, $M$ is a regular cellular decomposition of $Y$, and $N_0$ a regular cellular decomposition of $W$.
\end{defi}

It is clear that a coherent decomposition exists. 
So, we may use the results of Section \ref{sec7.2}, to introduce the following definitions and prove the following theorems.

\begin{defi}\label{homsing} Let $X=Y\sqcup_W C(W)$ be a space with a conical singularity of dimension $n=m+1$,  where $(Y,W)$ is a compact connected orientable smooth manifold of dimension $n$, with boundary $W$. Let $\pf$ be a perversity. Let $\rho:\pi_1(X)\to O(V)$ be an orthogonal representation on a $k$ dimensional vector space $V$. Let 
$K=M\sqcup_{N_0} C(N_0)$ be any coherent cellular decomposition of $X$. We call {\it intersection homology groups of $X$ with perversity $\pf$ and coefficients in $V_\rho$}, the groups 
$I^\pf H^q(X;V_\rho)= H^q(I^\pf \CS_\bu(K;V_\rho)$. The relative groups are defined similarly.
\end{defi}

Since the cellular decompositions  $K$ and $K^*$ of $X$ are constructed using the triangulations of the corresponding decompositions of the composing manifolds, they admit a common subdivision, see Proposition \ref{sub}. This shows that the intersection homology groups of $X$ are independent on the cellular decomposition. Formulas for these groups are given in Propositions \ref{p2.3}.

\begin{theo}\label{dualhomX} Let $X=Y\sqcup_W C(W)$ be a space with a conical singularity of dimension $n=m+1$,  where $(Y,W)$ is a compact connected orientable smooth Riemannian manifold of dimension $n$, with boundary $W$. Let $\pf$ be a perversity. Let $\rho:\pi_1(X)\to O(V)$ be an orthogonal representation on a $k$ dimensional vector space $V$. Then, there is an isomorphism 
\[
I^\pf \QQ_{*,q}:I^\pf  H_q(X;V_\rho)\to I^{\pf^c} H^\da_{n-q}(X;V_\rho).
\]
\end{theo}

%% file: part2.11.tex

\section{Some results on  Sturm-Liouville operators}
\label{ss2}

In this section we collect some technical results on  Sturm Liouville operator necessary in the following analysis. Several   results are either classic and well known, or may be deduced applying classic results, some are new, and in that cases complete proofs are presented. We follow a classical approach and work with real of smooth functions, and we refer to  Dunford Schwartz \cite{DS1} \cite{DS2} and Bocher \cite{Boc}, since this seems to be more convenient for a direct application to the context of the analysis on differentiable manifolds that is the main purpose of this work, however see also Weidmann  \cite{Wei}, and Zettl \cite{Zet}.


\subsection{Formal operator, differential equation and fundamental system of solutions}
\label{ss1.1}

We consider the formal linear differential  Sturm Liouville operator
\beq\label{ell}
\lf_{\nu,\al}=-\frac{\d^2}{ \d x^2} +q_{\nu,\al}(x),
\eeq
on the space of square integrable functions $L^ 2(0,l)$ (that is a (complete) separable Hilbert space), where 
\[
q_{\nu,\al}(x)=-\left(\al-\frac{1}{2}\right)\frac{h''(x)}{h(x)}+\left(\al^2-\frac{1}{4}\right)\frac{(h'(x))^2}{h^2(x)}+\frac{\nu^2-\al^2}{h^2(x)},
\]
where $\nu=\sqrt{\tilde\la+\al^2}$,  $\al$ and $\tilde \la$ are a real numbers, with $\tilde\la\geq 0$, and 
\[
h(x)=x H(x),
\]
where   $H$ is a non vanishing smooth  function    on $[0,l]$, with $H(0)=1$.  

\begin{rem} Note that these requirements guarantee that 
\[
|1-H^2(x)|=O(x^\ep),
\]
for some $\ep>0$, with $\ep=1$. For Since $H$ is smooth, so is $H^2$, and therefore
\[
\lim_{x\to 0^+} \frac{|H^2(x)-1|}{x}=\lim_{x\to 0^+} \left|\frac{d}{dx}H^2(x)\right|=\left|\frac{d}{dx}H^2(x)\right|(x=0).
\]
\end{rem}


This assumptions guarantee that the coefficients of the relevant differential equation satisfy the hypothesis of Theorem \ref{boc}. Observe that the results outlined in this section certainly hold under much weaker requirements, we introduced the stated hypothesis since they are sufficient for our application and avoid several technical details. For example, we could take ore generally $q_{\nu,\al}\in L_{\rm loc}(0,l)$ and proceed following \cite{Zet}. It is clear that our choice of $q_{\nu,\al}$ is in that class, so that we can use the results of the last work.

Before to proceed, we outline the behaviour of the potential near $x=0$. By the assumptions on $h$:
\begin{align*}
q_{\nu,\al}(x)=&\left(\nu^2-\frac{1}{4}\right)\frac{1}{x^2}+(\nu^2-\al^2)\frac{1-H^2(x)}{x^2H^2(x)}
-\left(\al-\frac{1}{2}\right)\frac{{H''}(x)}{H(x)}\\
&+\left(\al^2-\frac{1}{4}\right)\frac{{H'}^2(x)}{H^2(x)}
+2\left(\al-\frac{1}{2}\right)^2 \frac{H'(x)}{H(x)}x.
\end{align*}

For further use, we set
\[
p_{\nu,\al}=q_{\nu,\al}-\frac{\nu^2-\frac{1}{4}}{h^2}.
\]

Beside  the formal operator $\lf_{\nu,\al}$, we also introduce  the following "constant coefficients case" operator
\[
\lf^0_\nu=-\frac{d^2}{d x^2}+\frac{\nu^2-\frac{1}{4}}{x^2},
\]
which we also will refer to as the "flat case". It is clear that this is a particular case of $\lf_{\nu,\al}$, and precisely the one in which $h(x)=x$. Therefore all the results obtained for $\lf_{\nu,\al}$ hold for $\lf^0_\nu$. However, for $\lf^0_\nu$ we have some more explicit results that will be useful to obtain properties of $\lf_{\nu,\al}$. For this reason, we will occasionally  state the explicit expression of the quantities relative to $\lf^0_\nu$, that will be denoted by a superscript  $0$. We will use the decomposition
\beq\label{xyz}
\lf_{\nu,\al}=\lf_{\nu}^0+r,
\eeq
where 
\[
r(x)=q_{\nu,\al}-\frac{\nu^2-\frac{1}{4}}{x^2}.
\]

Next, summarise  some information about the solutions of the differential equation associated to $\lf_{\nu,\al}$ that will be useful in the following. The differential equation associated to $\lf_{\nu,\al}$ is the  second order regular singular equation (a regular one if $\nu=\frac{1}{2}$)
\beq\label{eqdiff1}
u''-q_{\nu,\al}u=u''+\left(\al-\frac{1}{2}\right)\frac{h''}{h}u-\left(\al^2-\frac{1}{4}\right)\frac{{h'}^2}{h^2}u-\frac{\nu^2-\al^2}{h^2}u=\la u,
\eeq
whose indicial equation  is
\[
\mu(\mu-1)+\frac{1}{4}-\nu^2=0,
\]
with solutions $\mu_\pm=\frac{1}{2}\pm \nu$. In the following, it will be convenient to have at our disposal  equation (\ref{eqdiff1}) reformulated in the new variable $f=h^{\al-\frac{1}{2}}u$:
\beq\label{eqdiff2}
f''+(1-2\al)\frac{h'}{h}f'-\frac{\nu^2-\al^2}{h^2}f=\la f,
\eeq
whose indicial equation is
\[
s(s-1)+(1-2\al)s+\al^2-\nu^2=0,
\]
with solutions $s_\pm=\al\pm\nu$ (we fix the order $s_+\geq s_-$). The corresponding formal Sturm Liouville operator is
\[
\sf_{\nu,\al}=T\lf_{\nu,\al} T^{-1}=-\frac{d^2}{d x^2}+(1-2\al)\frac{h'(x)}{h(x)}\frac{d}{dx}-\frac{\nu^2-\al^2}{h^2(x)},
\]
where $T$ is the isometry 
\begin{align*}
T:&L^2((a,b),d x)\to L^2((a,b),h^{1-2\al}d x),\\
T:&u\mapsto h^{\al-\frac{1}{2}} u.
\end{align*}

Note that
\[
\sf^0_{\nu,\al}=T\lf_{\nu}^0 T^{-1}=-\frac{d^2}{d x^2}+(1-2\al)\frac{1}{x}\frac{d}{dx}-\frac{\nu^2-\al^2}{x^2}.
\]

A complete system  of two linearly independent solutions $u_\pm$ of  equation (\ref{eqdiff2}), and consequently of equation (\ref{eqdiff1}) is described in   Theorem \ref{boc}. For present use, we normalise these solutions according to the following definition.

\begin{defi}\label{defi1} A fundamental system of normalised solutions $\uf_\pm$ of equation \eqref{eqdiff1} is:

\begin{enumerate}
\item if $\nu\not=0$, then
\begin{align*}
\uf_\pm(x)&=x^{\frac{1}{2}\pm\nu}\varphi_\pm(x),&\uf'_\pm(x)&=x^{-\frac{1}{2}\pm\nu}\Phi_\pm(x),
\end{align*}
where the $\varphi_\pm$ and the $\Phi_\pm$ are continuous  in some interval $[0,l]$,    $\varphi_\pm(0)=1$, and $\Phi_\pm(0)=s_\pm$; 
\item if $\nu=0$, then
\begin{align*}
\uf_+(x)&=x^{\frac{1}{2}+\nu}\varphi_+(x),&\uf'_+(x)&= x^{-\frac{1}{2}+\nu}\Phi_+(x),\\
\uf_-(x)&=x^{\frac{1}{2}+\nu}\varphi_-(x)\log x,&\uf'_-(x)&=x^{-\frac{1}{2}+\nu}\Phi_-(x)\log x
\end{align*}
where the $\psi_\pm$,   and the $\Psi_\pm$ are continuous  in some interval $[0,l]$,   $\psi_\pm(0)=1$, and $\Psi_\pm(0)=s_+$;
\end{enumerate}

The corresponding system of solutions of equation \eqref{eqdiff2} is
$\ff_\pm=h^{\al-\frac{1}{2}}\uf_\pm$.


The solutions above depend smoothly on the parameters, if necessary we will write $\uf_\pm=\uf_\pm(x,\la,\nu)$, and $\ff_\pm=\ff_\pm(x,\la,\nu)$. 
\end{defi}

\begin{rem} Since $h$ is smooth in $(0,l]$, it follows that the solutions $\uf_\pm$ and $\ff_\pm$ are smooth on $(0,l]$, see Remark \ref{Appsmooth}.
\end{rem}

\begin{rem} Definition \ref{defi1} covers the particular case  $\tilde\la=0$ and $\al=\frac{1}{2}$, as limit case with  $\nu=\frac{1}{2}$. In such a case, the problem reduces to a regular Sturm Liouville problem and the solutions of the fundamental system are smooth at $x=0$.
\end{rem}

\begin{rem} \label{nu=alpha} We observe the following particular solutions of the harmonic equation when $\nu=\pm \al$. In such a case, the harmonic equation is
\[
u'' +\left(\alpha-\frac{1}{2}\right)\frac{h''}{h} u
- \left(\alpha^2-\frac{1}{4}\right)\frac{(h')^2}{h^2}u=0.
\]

Setting 
\[
p=(1-2\al)\frac{h'}{h},
\]
this equation reduces to the one studied in Appendix \ref{meta}, and therefore has the two linearly independent solutions:
\begin{align*}
u_1&=h^{\frac{1}{2}-\al},& u_2&=h^{\frac{1}{2}-\al}\int h^{2\al-1},
\end{align*}
and therefore, according to the normalisation introduced in Definition \ref{defi1}, 
\begin{align*}
\uf_+&=u_2=2\al h^{\frac{1}{2}-\al}\int h^{2\al-1},&\uf_-&=u_1=h^{\frac{1}{2}-\al},& \al&>0,\\
\uf_+&=h^\frac{1}{2} ,& \uf_-&=h^{\frac{1}{2}}\int \frac{1}{h}, & \al&=0,\\
\uf_+&=u_1=h^{\frac{1}{2}-\al},& \uf_-&=u_2=2\al h^{\frac{1}{2}-\al}\int h^{2\al-1},& \al&<0,
\end{align*}
and \begin{align*}
\ff_+&=2\al \int h^{2\al-1},&\ff_-&=1,& \al&>0,\\
\ff_+&= 1,& \ff_-&=\int \frac{1}{h}, & \al&=0,\\
\ff_+&=1,& \ff_-&=2\al \int h^{2\al-1},& \al&<0.
\end{align*}

\end{rem}

\begin{rem} \label{IC} If $h$ is analytic in $[0,l]$ and  $\nu>0$, the solutions $\uf_{-}$ and $\uf_{+}$ are the unique solutions of equation \eqref{eqdiff1}, with characteristic exponent $\pm\nu$ respectively, and satisfying the initial value conditions:
\begin{align*}
IC_{0}(\uf_{-})&=1,&
IC_{0}'(\uf_{+})&=1,
\end{align*}
where
\begin{align*}
IC_{0}(u)&=\Rz_{x=0} x^{\nu-\frac{1}{2}} u(x),&IC_{0}'(u)&=\Rz_{x=0} x^{-\nu-\frac{1}{2}} u(x).
\end{align*}

In the regular case $\nu=\frac{1}{2}$, these conditions reduces to the classical ones:
\begin{align*}
IC_{0}(u)&=u(0),&IC'_{0}(u)&=u'(0).
\end{align*}
\end{rem}

\begin{rem} If $h$ is analytic in $[0,l]$ and  $\nu>0$,  the requirement of having characteristic exponent $\pm\mu$ corresponds either to the requirement of having the following form
\begin{align*}
u_{\pm}(x,\lambda,\nu)&= x^{\frac{1}{2}\pm\nu} h_\pm (x,\la,\nu),\\
h_\pm(x,\la,\nu)&=\sum_{j=0}^\infty a_{j,\pm}(\la,\nu)x^j,\\
a_0(\la,\nu)&=1,
\end{align*}
where the power series converges in any bounded interval $(0,l]$, or to the following initial conditions
\begin{align*}
IC_{0}'(\uf_{-})&=0,
&IC_{0}(\uf_{+})&=0.
\end{align*}

Note also that 
\begin{align*}
\lim_{x\to 0^+} x^{\nu-\frac{1}{2}} \uf_{+}(x,\la,\nu)&=0,&
\lim_{x\to 0^+} x^{-\nu-\frac{1}{2}} \uf_{-}(x,\la,\nu)&=+\infty.
\end{align*}

\end{rem}

\begin{rem}\label{rem3.1}  In the flat case, namely when $h(x)=x$, the coefficients of the Sturm Liouville differential equation are analytic in all the interval, so we have a fundamental system of analytic solutions, compare with Theorem \ref{theoA2}. This is the classical well known case, and the  normalised solutions may be written in terms of classical functions as follows ($z=i\sqrt{-\la}$). See Lemmas \ref{l2.2-a1} and \ref{l2.3-a1} (in particular note that by Remark \ref{xxx}, if $\nu \in \frac{1}{2}(2\Z+1)$, then we may use the expansions in Lemma \ref{l2.2-a1}, since the critical index appear only for even indices) of Appendix for details.

If $\nu\not\in \Z$:
\begin{align*}
\uf^0_{\pm}(x,\lambda,\nu)=&\frac{2^{\pm\nu} \Gamma(\pm\nu+1)}{\la^\frac{\pm\nu}{2}}
\sqrt{x}J_{\pm\nu}(\sqrt{\lambda}x)
=\frac{2^{\pm\nu} \Gamma(\pm\nu+1)}{(-\la)^\frac{\pm\nu}{2}}\sqrt{x}I_{\pm\nu}(\sqrt{-\la} x).
\end{align*}

If $\nu\in \Z$, $\nu>0$:
\begin{align*}
\uf^0_{+}(x,\la,\nu)=&\frac{2^{\nu} \Gamma(1+\nu)}{(-\la)^\frac{\nu}{2}} \sqrt{x}I_{\nu}(\sqrt{-\la} x),\\
\uf^0_{-}(x,\la,\nu)=&\frac{(-1)^\nu\log (-\la)}{2^{\nu} \Gamma(\nu)(-\la)^{-\frac{\nu}{2}}} \sqrt{x}I_{\nu}(\sqrt{-\la} x)
-\frac{(-1)^\nu\log 2}{2^{\nu-1} \Gamma(\nu)(-\la)^{-\frac{\nu}{2}}} \sqrt{x}I_{\nu}(\sqrt{-\la} x)\\
&+\frac{1}{2^{\nu-1} \Gamma(\nu)(-\la)^{-\frac{\nu}{2}}} \sqrt{x}K_{\nu}(\sqrt{-\la} x).
\end{align*}

If $\nu=0$:
\begin{align*}
\uf_{+}^0 (x,\la,0)&=\sqrt{x}I_0(\sqrt{-\la} x),\\
\uf_-^0(x,\la,0)&=-\sqrt{x}K_0(\sqrt{-\la} x)-\frac{1}{2}\sqrt{x}\log (-\la) I_0(\sqrt{-\la}x)+\left(
\log 2-\ga\right)\sqrt{x}I_0(\sqrt{-\la}x).
\end{align*}

If $\nu=\frac{1}{2}$ ($\mu_-=0$ and  $\mu_+=1$, this is a particular instance of the first case), 
\begin{align*}
\uf^0_{-}\left(x,\lambda,\frac{1}{2}\right)&=
\cos (\sqrt{\la}x)=
\ch (\sqrt{-\la}x),\\
\uf^0_{+}\left(x,\lambda,\frac{1}{2}\right)&=
\frac{1}{\sqrt{\la}}\sin (\sqrt{\la}x)=
\frac{1}{\sqrt{-\la}}\sh (\sqrt{-\la}x).
\end{align*}
\end{rem}

\begin{rem} As observed in Remark \ref{ra6}, the solution $\ff_-$ and $\uf_-$ are not determined univocally when $\nu>0$. This is not in general a problem, but produce technical difficulties in particular when the asymptotic expansions are discussed. For this reason, it is convenient to fix univocally also the minus solution. We proceed as follows. As observed in Remark \ref{ra6}, if $\uf_\pm$ are  solutions satisfying the normalisation in Definition \ref{defi1}, then
\[
u_-=\uf_-+c \uf_+,
\]
is again a solution of type $-$ according to the same definition. Now let $h_\ep(x)$ be a family of function converging smoothly to $h_0(x)=0$ (see the proof of Lemma \ref{explambda}). 
Then, the family of solutions
\[
u^\ep_{-}=\uf^\ep_{-}+c \uf^\ep_{+},
\]
converge to an analytic solution
\[
u^0_{-}=\uf^0_{-}+c \uf^0_{+}, 
\]
that may be determined univocally fixing the constant $c$ by requiring that $u^0_-=\uf^0_-$, where $\uf^0_-$ is given in Remark \ref{rem3.1}.

Observe that, if $\nu=0$, also the $-$ solution is determined univocally in the smooth case by Remark \ref{case0}, se also Appendix \ref{a12}.

\end{rem}

\begin{rem}\label{function v} 
Occasionally,  it will be more convenient to use the alternative fundamental system of solutions   $\uf_+$ and $\vf$, where the last is defined as follows. 

If $\nu\notin \Z$, and $\nu\not= \frac{1}{2}$:
\begin{align*}
\vf(x,\la,\nu )&=2^{\nu} \Gamma(\nu+1)\uf_-(x,\la,\nu )
-\frac{2^{-\nu} \Gamma(-\nu+1)}{(-\la)^{-\nu}}\uf_+(x,\la,\nu )\\
&=\frac{\pi\nu}{\sin\pi\nu}\left(\frac{1}{2^{-\nu} \Gamma(-\nu+1)}\uf_-(x,\la,\nu )-\frac{(-\la)^\nu}{2^{\nu} \Gamma(\nu+1)}\uf_+(x,\la,\nu )\right),
\end{align*}
satisfying the IC
\begin{align*}
IC_{0}(\vf)&=2^{\nu} \Gamma(\nu+1),
&IC_{0}'(\vf)&=-\frac{2^{-\nu} \Gamma(-\nu+1)}{(-\la)^{-\nu}};
\end{align*}
and with limit
\begin{align*}
\vf^0(x,\la,\nu)&=\frac{\pi\nu}{\sin\pi\nu}(-\la)^{\frac{\nu}{2}} \sqrt{x}\left(\e^{\frac{\pi}{2}\nu i} J_{-\nu}(\sqrt{\la} x)-\e^{-\frac{\pi}{2}\nu i}J_\nu(\sqrt{\la}x)\right)\\
&=\frac{\pi\nu}{\sin\pi\nu}(-\la)^{\frac{\nu}{2}} \sqrt{x}\left( I_{-\nu}(\sqrt{-\la} x)-I_\nu(\sqrt{-\la}x)\right)\\
&=\frac{2\nu}{(-\la)^{-\frac{\nu}{2}}} \sqrt{x}K_\nu(\sqrt{-\la} x).
\end{align*}

If $\nu\in \Z$, $\nu>0$:
\[
\vf(x,\la,\nu)=2^{\nu} \Gamma(\nu+1)\uf_-(x,\la,\nu)-\frac{2(-1)^\nu \nu}{(-\la)^{-\frac{\nu}{2}}}\left(\frac{1}{2}\log(-\la)-\log 2\right) \uf_+(x,\la,\nu),
\]
and
\begin{align*}
\vf^0(x,\la,\nu)=&\frac{2 \nu}{(-\la)^{-\frac{\nu}{2}}} \sqrt{x}K_{\nu}(\sqrt{-\la} x).
\end{align*}

If $\nu=0$:
\[
\vf(x,\la,0)=-\uf_-(x,\la,0)-\left(\log 2-\ga-\frac{1}{2}\log(-\la)\right)\uf_+(x,\la,0),
\]
and
\begin{align*}
\vf^0(x,\la,0)&= \sqrt{x}K_0(\sqrt{-\la} x).
\end{align*}

\end{rem}

\subsection{Minimal and maximal operators, boundary values and self-adjoint extensions}
\label{bv}

The operator $\lf_{\nu,\al}$ is a regular formal operator of order $2$ on the interval $(0,l]$, according to \cite[XIII.1.1, pg. 1280]{DS2},
with $a_2(x)=-1$, $a_0(x)=q_{\nu,\al}(x)$, $a_k(x)=0$, $k\not= 0,2$.

According to \cite[XIII.2.1, pg. 1287]{DS2}, the boundary matrix of $\lf_\nu$ is
\[
F_x=\left(\begin{array}{cc}0&-1\\1&0\end{array}\right),
\]
and the formal adjoint is
\[
\lf_{\nu,\al}^\da=-\frac{d^2}{dx^2}+q_{\nu,\al}(x),
\]
and therefore $\lf_\nu$ is formally self-adjoint or formally symmetric.

For an interval $I\subseteq \R$,  let \cite[XIII.1.2, pg. 1280]{DS2}
\[
A^2(I)=\{u\in C^1(I)~|~ u'\in AC_0(I)\},
\]
where $AC_0(I)$ is the set of the absolutely continuous functions over compact subsets (see also \cite[pg. 157]{Wei}).  Observe that if $f\in A^2(I)$, then $u''$ exists almost everywhere and is integrable on compact subsets. Let \cite[XIII.2.3, pg. 1287]{DS2}
\[
H^2(I)=\{u\in A^2(I)~|~u,u''\in L^2(I)\},
\]
and for a given formal differential operator $\tf$:
\[
H_\tf^2(I)=\{u\in A^2(I)~|~u,\tf u\in L^2(I)\}.
\]

For any set of functions $X$, let $X_0$ denote the subset of $X$ of functions with compact support \cite[XIII.2.7, pg. 1291]{DS2}.

We have the following inclusions of sets: $C_0^\infty(I)\subseteq C^\infty(I)\subseteq A^2(I)$, $H^2(\bar I)\subseteq H_\tf^2(\bar I)$ 
\cite[pg. 1288]{DS2}.

We have the Green formula: for any $u\in H^2((0,l])$, $v\in A^2((0,l])$, if either $u$ or $v$ has compact support, then
\[
\int_0^l (\lf_{\nu,\al} u)(x)\overline{g(x)} dx=\int_0^l u(x)\overline{(\lf_{\nu,\al}^\da v)(x)} dx+F_l(u,v)-F_0(u,v)
\]
\cite[XIII.2.5, pg. 1288]{DS2}, where 
\[
F_x(u,v)=\sum_{j,k=0}^{2-1} F^{j,k}_x u^{j}(x) v^{k}(x)=-u(x)v'(x)+u'(x)v(x)=-W(u,v)(x),
\]
where $W(f,g)(x)$ is called the Wronskian of the pair $(f,g)$ at $x$ \cite[pg. 262]{Wei}.

According to \cite[XIII.2.8, pg. 1291]{DS2}, we define the minimal and the maximal operators associated to $\lf_{\nu,\al}$ 
\beq\label{minmax}\begin{aligned}
D(L_{{\nu,\al},{\rm min}})&=H^2((0,l])\cap H_0^2((0,l))=H_0^2((0,l)),\\
D(L_{{\nu,\al},{\rm max}})&=H_{\lf_{\nu,\al}}^2((0,l]).
\end{aligned}
\eeq

Note that we could equivalently chose $D(L_{{\nu,\al},{\rm min}})=C^\infty_0((0,l))$, as in \cite[pg. 160]{Wei}.

\begin{rem}\label{rdom} We may find an explicit description of $D(L_{{\nu,\al},{\rm max}})$. By definition if $u$ is any function in $A^2((0,l])$ then $u$ and $\lf u$ are square integrable on compact subsets. Thus the domain is given by those of the functions $u$ in $A^2((0,l])$ such that $u$ and $\lf u$ are square integrable near $x=0$. Assuming that $u(x)\sim x^a$, then square integrability requires that $a>\frac{3}{2}$. However, there are other solutions, and precisely those that satisfy the equation $\lf_{\nu,\al} u=\la u$, with $\la\not=0$ and $u$ square integrable near $x=0$. These are the $\uf_\pm$ described in Definition \ref{defi1}. Thus,
\[
D(L_{{\nu,\al},{\rm max}})=\left\{ u\in A^2((0,l])~|~ u(x)\sim x^a, a>\frac{3}{2}\right\}\cup \langle\uf_\pm\rangle.
\]
\end{rem}

Observe that $L_{{\nu,\al},{\rm min}}$ and $L_{{\nu,\al},{\rm max}}$ are unbounded densely defined operator on $L^2(0,l)$, i.e.
\begin{align*}
\overline{D(L_{{\nu,\al},{\rm min}})}&=\overline{H_0^2((0,l])}=L^2(0,l),\\
\overline{D(L_{{\nu,\al},{\rm max}})}&=\overline{H_{\lf_{\nu,\al}}^2((0,l])}=L^2(0,l).
\end{align*}

Since $\lf_{\nu,\al}$ is regular and formally self-adjoint, it follows \cite[XIII.2.11, pg. 1295]{DS2} that $L_{\nu,{\rm min}}\subseteq L_{{\nu,\al},{\rm max}}$, i.e. that $L_{{\nu,\al},{\rm min}}$ is symmetric \cite[XII.1.7, pg. 1190]{DS2}. Moreover, since $\lf_{\nu,\al}$ is formally self-adjoint, it follows that \cite[XIII.2.10, pg. 1294]{DS2}
\[
L_{{\nu,\al},{\rm min}}^\da=L_{{\nu,\al},{\rm max}},
\]
i.e. $L_{{\nu,\al},{\rm max}}$ is the adjoint of $L_{{\nu,\al},{\rm min}}$.

By \cite[XII.1.6(a), pg. 1189]{DS2}, it follows that $L_{{\nu,\al},{\rm max}}$ is closed. The formula for the domain in Remark \ref{rdom} shows that $L_{{\nu,\al},{\rm max}}$ is not symmetric, for if it were, then it should be self-adjoint, while is is not as we may verify using equation \eqref{bl}, that shows that the boundary value does not vanishes on all the combinations of the functions $\uf_\pm$. Indeed, the formula in Remark \ref{rdom} is the first von Neumann for the adjoint of $L_{{\nu,\al},{\rm min}}$, that is indeed $L_{{\nu,\al},{\rm max}}$.

Moreover,  all the self-adjoint extensions of $L_{{\nu,\al},{\rm min}}$ are restrictions of $L_{{\nu,\al},{\rm max}}$. We want to characterise these extensions, for we need to introduce deficiency indices and boundary values.

The deficiency indices $d_\pm$ are two positive integer numbers or infinity, as defined in \cite[XII.4.9]{DS2}. 
By \cite[XIII.2.14, pg. 1295]{DS2} , the deficiency indices of $L_{{\nu,\al},{\rm min}}$ are $d_+=d_-\leq 2$, and by 
\cite[XIII.2.24, pg. 1301]{DS2}  (Weyl-Kodaira) $d_++d_-\geq 2$. Hence, $1\leq d_\pm\leq 2$, i.e $(d_+,d_-)=(1,1)$ or
$(d_+, d_-)=(2,2)$.

This means that we can apply Theorem XII.4.30, pg.  1238 to affirm that all self-adjoint extensions $L_{\nu,\al}$ of the operator $L_{{\nu,\al},{\rm min}}$ are restrictions of $L_{{\nu,\al},{\rm max}}$ to the subspace determined by a symmetric family of $2$ linearly independent boundary conditions for $L_{{\nu,\al},{\rm min}}$.

Therefore, we determine such a set of boundary conditions, and for this a set of boundary values for $L_{{\nu,\al},{\rm min}}$. Boundary conditions and boundary values for an abstract operator are defined in XII.4.20, pg. 1234 and XII.4.25, pg. 1235 of \cite{DS2}. However, in the actual concrete case, a more effective definition is that in XIII.4.17, pg. 1297 of \cite{DS2}, where boundary values and boundary conditions are defined for a formal differential operator, like $\lf_{\nu,\al}$. The two definition coincide in  the actual case by \cite[XIII.2.18, pg. 1298]{DS2}. Accordingly, a boundary value for $\lf_{\nu,\al}$ is a continuous linear functional $BV$ on $D(L_{{\nu,\al},{\rm max}})$ that vanishes on $D(L_{{\nu,\al},{\rm min}})$. If $BV(f)=0$ for each function $f\in D(L_{{\nu,\al},{\rm max}})$ which vanishes in a neighbourhood of either $0$ or $l$, then $BV$ is called a boundary value for $\lf_{\nu,\al}$ at $0$, respectively at $l$ (in other words, the support of $f$ does not contains $0$, $l$), and we use the notation $BV(0)$, $BV(l)$, respectively.

By \cite[XIII.2.19, pg. 1298]{DS2}, each boundary value for $\lf_{\nu,\al}$ is the sum of a boundary value for $\lf_{\nu,\al}$ at $0$ and a   
boundary value for $\lf_{\nu,\al}$ at $l$, and the maximum number of independent boundary values for $\lf_{\nu,\al}$ at either $0$ or $l$ is $2$, by \cite[XIII.2.22, pg. 1300]{DS2}. 

By \cite[XIII.2.23, pg. 1301]{DS2}, a complete set of boundary values for $\lf_{\nu,\al}$ at $l$ (that is a fixed end of the interval) is 
the set $\{BV(l), BV'(l)\}$ of the functionals
\begin{align*}
BV(l)(u)&=u(l),\\
BV'(l)(u)&=u'(l),
\end{align*}
$u\in D(\lf_{{\nu,\al},{\rm max}})$, $\supp(u)\cap \{l\}=\emptyset$. This shows also that $\lf_{\nu,\al}$ has two linearly independent boundary values at $l$. 

It remains to understand the boundary values for $\lf_{\nu,\al}$ at $0$. By \cite[XII.4.21, pg. 1234 ]{DS2}, the number of boundary values (sum of the number of boundary values at $0$ and at $l$) is $n_0+n_l=d_++d_-$, and is therefore either $2$ or $4$. Since $n_l=2$, it follows that $n_0$ is either $0$ or $2$.

Since
\[
{\rm liminf}_{x\to 0^+}x^2q(x)
={\rm liminf}_{x\to 0^+}x^2q(x)
=\nu^2-\frac{1}{4}
=\left\{\begin{array}{cc}>\frac{3}{4}&{\rm for}~\nu>1,\\
<\frac{3}{4}&{\rm for}~\nu<1,\end{array}\right.
\]
we can used \cite[XIII.6.23, pg. 1414]{DS2} , to state that
\[
n_0=\left\{\begin{array}{cc}0&{\rm for}~\nu>1,\\2&{\rm for}~\nu<1.\end{array}\right.
\]

For a complete answer, we need the H. Weyl criterium, that affirms that the number of boundary values for $\lf_{\nu,\al}$ at $0$ is $n_0=2$ if there are two linearly independent  solutions of the equation $\lf_{\nu,\al} u=\la u$, square integrable at $0$, with 
$\Im(\la)\not=0$, while it is $n_0=0$ if there is only one of such solutions \cite[end pg. 1305]{DS2}. Using the expansions of the solution $\uf_\pm$ of the equation $\lf_{\nu,\al} u=\la u$ given in Theorem \ref{boc}, we see that $\uf_+$ is square integrable (on $(0,l]$ and consequently) at $0$ for all $\nu$, while $\uf_-$ is square integrable at $0$ if and only is $\nu<1$. More precisely, if $\nu\not= 0$, we are in the first case $\mu_+\not= \mu_-$, so the solutions are
\begin{align*}
u_+(x)&=x^{\frac{1}{2}+\nu}\vv_+(x),&u_-(x)&=x^{\frac{1}{2}-\nu}\vv_-(x),
\end{align*}
$u_+$ is square integrable for all $\nu$, while $u_-$ is square integrable if and only if $\nu<1$;  if $\nu=0$, then $\mu_+=\mu_-=\frac{1}{2}$, and then the two solutions are
\begin{align*}
u_+(x)&=x^{\frac{1}{2}}\vv_+(x),&u_-(x)&=x^{\frac{1}{2}}\vv_-(x)\log x,
\end{align*}
that are both square integrable. In summary, we have the following table:
\[
\begin{array}{ccccccccc}
&(d_+,d_-)& L^2~sol.~ at~ 0&0 &L^2~sol.~ at~ l& l&n=n_0+n_l&n_0
& n_l\\
\nu <1&   (2,2)&2&LCC&2&LCC&4&2&2\\
\nu\geq 1&(1,1)&1&LPC&2&LCC&2&0&2
\end{array}
\]
where $LCC$ and $LPC$ means limit circle case and limit point case, respectively. 
Whence there are $2$ boundary values for $\lf_{\nu,\al}$ at $0$ when $\nu<1$, and none when $\nu\geq 1$.

We find out the boundary values for $\lf_{\nu,\al}$ at $0$ when $\nu<1$. 



Let $\nu<1$ and $\uf_\pm$ be the the two solutions of the associated differential equation (\ref{eqdiff1}) with $\la=0$, according to the description outlined above.  
Fix a point $x_0$, with  $0<x_0<l$, and let $v_\pm$ be two smooth functions on $(0,l]$ vanishing  for $x>x_0$ and equal to 
$\uf_\pm$ near $0$. Then, 
\[
(\lf_{\nu,\al} v_\pm)(x)=0,
\]
near $0$, so that $\lf_{\nu,\al} v_\pm$ is square integrable on $(0,l]$ and therefore $v_\pm\in D(L_{{\nu,\al},{\rm max}})$. 
Thus, by \cite[XII.4.20, pg. 1234]{DS2}, for all $u\in D(L_{{\nu,\al}, {\rm max}})$, 
\beq\label{bl}
BV_{\nu,\pm}(0)(u)=(\lf_{\nu,\al} u,v_\pm)-( u,\lf_{\nu,\al}v_\pm),
\eeq
are boundary values for $\lf_{\nu,\al}$ at $0$ (see also \cite[XIII.2.27, pg. 1302]{DS2}). In fact, these are continuous functionals on $D(L_{{\nu,\al},{\rm max}})$ that vanish on $D(L_{{\nu,\al},{\rm min}})$. We may compute
\begin{align*}
BV_{\nu,\pm}(0)(u)&=\int_0^l \left( (\lf_{\nu,\al}u)(x) \overline{v_\pm(x)}- u(x)\overline{(\lf_{\nu,\al}v_\pm)(x)}\right) dx\\
&=\int_0^l \left( -u''(x) v_\pm(x)- u(x)v''_\pm(x)\right) dx\\
&=\lim_{x\to 0^+} \left(u'(x)v_\pm(x)-u(x)v'_\pm(x)\right)\\
&=\lim_{x\to 0^+} W(v_\pm,u)(x),
\end{align*}
where $u\in D(L_{{\nu,\al},{\rm max}})$, and therefore we have the following two boundary values for $\lf_{\nu,\al}$ at $0$:
\begin{align*}
BV_{\nu,+}(0)(u)&=\lim_{x\to 0^+} W(v_\pm,u)(x)=\lim_{x\to 0^+} \left(v_+(x)u'(x)-v'_+(x)u(x)\right),\\
BV_{\nu,-}(0)(u)&=\lim_{x\to 0^+} W(v_\pm,u)(x)=\lim_{x\to 0^+} \left(v_-(x)u'(x)-v'_-(x)u(x)\right).
\end{align*}

\begin{rem} The intrinsic meaning of this construction is the following: equation (\ref{bl}) is the obstruction for 
$L_{{\nu,\al}, {\rm max}}$ to be self-adjoint, i.e. the restriction of the its domain to the subspace of the functions that satisfy the boundary condition makes the boundary value to vanish and therefore defines the domain of the adjoint $L_{{\nu,\al}, {\rm max}}^\da$, that consequently is a restriction of $L_{{\nu,\al}, {\rm max}}$.
\end{rem}

We compute the boundary values on the solutions:
\begin{enumerate}
\item if $\nu \not= 0$ ($\nu< 1$), then 
\begin{align*}
BV_{\nu,+}(0)(v_+)&=0,&BV_{\nu,+}(0)(v_-)&=-2\nu\vv_+(0)\vv_-(0)=-2\nu,\\
BV_{\nu,-}(0)(v_+)&=2\nu\vv_+(0)\vv_-(0)=2\nu,&BV_{\nu,-}(0)(v_-)&=0;
\end{align*}
\item if $\nu=0$, then
\begin{align*}
BV_{\nu,+}(0)(v_+)&=0,&BV_{\nu,+}(0)(v_-)&=\vv^2_+(0)=1,\\
BV_{\nu,-}(0)(v_+)&=-\vv^2_+(0)=-1,&BV_{\nu,-}(0)(v_-)&=0.
\end{align*}
\end{enumerate}

This shows that these boundary values are independent and therefore determine a complete system of boundary values for $\lf_{\nu,\al}$ at $0$.


We have proved the following theorem, where boundary conditions are defined according to \cite[XIII.2.29, pg. 1305]{DS2} \cite[8.29]{Wei}. 

\begin{theo} \label{extensions} The self-adjoint extensions $L_{\nu,\al}$ of the minimal operator $L_{{\nu,\al},{\rm min}}$ associated to the  formal differential operator $\lf_{\nu,\al}$ in the Hilbert space $L^2(0,l)$ are:
\begin{align*}
L_{\nu,\al} u&=\lf_{\nu,\al} u,\\
D(L_{\nu,\al})&=\left\{ u\in D(L_{{\nu,\al},{\rm max}})~|~ BC_\nu(0)(u)=BC(l)(u)=0\right\},
\end{align*}
where the boundary conditions are
\begin{align*}
BC_\nu(0)(u):&\hspace{60pt} \left\{
\begin{array}{cc}\be_+ BV_{\nu,+}(0)(u)+\be_- BV_{\nu,-}(0)(u)=0,&{\rm if~}\nu<1,\\
{\rm none,}&{\rm if}~\nu\geq 1,\end{array}\right.
\\
BC(l)(u):& \hspace{80pt}\be BV(l)(u)+\be' BV'(l)(u)=0,
\end{align*}
for real $\be_\pm, \be, \be'$, with $\be_+^2+\be_-^2=\be^2+{\be'}^2=1$. We denote these operators by $L_{\nu,\al, \be,\be',\be_\pm}$. We denote by $S_{\nu,\al, \be,\be',\be_\pm}=T L_{\nu,\al, \be,\be',\be_\pm}T^{-1}$ the corresponding extensions of $\sf_{\nu,\al}=T\lf_{\nu,\al}T^{-1}$.
\end{theo}

Proceeding exactly in the same way, we may consider the formal operator $\lf_{\nu,\al}$ acting on the space $L^2(a,b)$, with $a>0$. Due to our hypothesis on $q_{\nu,\al}$ in this case we have a regular operator, so we immediately have the following characterisation if its self-adjoint extensions.

\begin{theo} The self-adjoint extensions $R_{\nu,\al}$ of the minimal operator $R_{{\nu,\al},{\rm min}}$ associated to the  formal differential operator $\lf_{\nu,\al}$ in the Hilbert space $L^2(a,b)$, $a>0$,  are:
\begin{align*}
R_{\nu,\al} u&=\lf_\nu u,\\
D(R_{\nu,\al})&=\left\{ u\in D(R_{{\nu,\al},{\rm max}})~|~ BC(a)(u)=BC(b)(u)=0\right\},
\end{align*}
where the boundary conditions are
\begin{align*}
BC(a)(u):& \hspace{80pt}\al BV(a)(u)+\al' BV'(a)(u)=0,\\
BC(b)(u):& \hspace{80pt}\be BV(b)(u)+\be' BV'(b)(u)=0,
\end{align*}
for real  $\al, \al', \be, \be'$, with $\al^2+{\al'}^2=\be^2+{\be'}^2=1$.
\end{theo}

\begin{defi}\label{defi0} 
We consider now some particular  self-adjoint extensions of $\lf_{\nu,\al}$ on $L^2(0,l)$ determined by the possible different combinations of the following boundary conditions at $x=0$:
\begin{align*}
BC_{\nu, +}(0)(u):& &BV_{\nu,+}(0)(u)&=0,\\
BC_{\nu, -}(0)(u):& &BV_{\nu,-}(0)(u)&=0,
\end{align*}
and at $x=l$:
\begin{align*}
BC_{\rm abs}(l)(u):& &\left(h^{\al-\frac{1}{2}} u\right)'(l)&=0,\\
BC_{\rm rel}(l)(u):&&u(l)&=0.
\end{align*}

If $\nu<1$, we denote by $L_{\nu,\al,{\rm abs, \pm}}$ the operator defined by the boundary condition $BC_{\rm abs}(l)$ and $BC_{\rm \pm}(0)$. If $\nu\geq 1$, we denote by   $L_{\nu,\al,{\rm abs}}$ the operator 
defined  by the boundary condition $BC_{\rm abs}(l)$. Similarly in the relative case.

Moreover, we consider 
the self-adjoint extension $R_{\nu,\al}=R_{\nu,\al, {\rm abs, rel}}$ of $\lf_\nu$ on $L^2(a,b)$, $0<a<b$,  determined by  following boundary conditions: 
\begin{align*}
BC_{\rm rel}(a)(u):&&u(a)&=0,\\
BC_{\rm abs}(b)(u):&&\left(h^{\al-\frac{1}{2}} u\right)'(b)&=0,
\end{align*}

We will write $L_{\nu,\al}$ meaning any one of the operator defined above on $L^2(0,l)$.
\end{defi}

\begin{rem}\label{abc} Observe that the fundamental solution $\uf_\pm$ belongs to the domain of $L_{\nu,\al, {\rm bc}, \pm}$ but not to that of $L_{\nu,\al, {\rm bc}, \mp}$.  Both the fundamental solutions belong to the domain of  $R_{\nu,\al, {\rm bc}}$, while only $\uf_+$ belongs to the domain of $L_{\nu,\al, {\rm bc}}$. 
This follows immediately by the definition of the domain and the calculations of the boundary values on the solutions given above.
\end{rem}

\subsection{Spectrum, kernel and spectral functions}
\label{spectral sequences}

Our next aim is to describe the spectrum of the operators introduced in Definition \ref{defi0}.   More precisely, we will tackle only the operators $L_{\nu,\al, {\rm bc}}$, 
$L_{\nu,\al, {\rm bc}, +}$,  $R_{\nu,\al}$ and $L_{\al,\al, {\rm bc}, -}$ if 
$0< \al\leq \frac{1}{2}$. Analogous results for the operator $L_{\nu,\al, {\rm bc}, -}$ would require more sophisticated tools and are not necessary for the applications we have in mind in the present work. 

Before giving the explicit results, we recall some information on the solutions of the main differential equation in the case $\nu=|\al|$. In such a case, the main equation reads
\[
f''+(1-2\al)\frac{h'}{h} f'=\la f,
\]
with exponents $s_\pm=0, 2\al$, and $\ff_\pm(x)=x^{\al\pm|\al|}\psi_\pm(x)$, so if $\al>0$:
\begin{align*}
\ff_+(x)&=x^{2\al}\psi_+(x), & \ff_-(x)&=\psi_-(x);
\end{align*}
if $\al<0$:
\begin{align*}
\ff_+(x)&=\psi_+(x), & \ff_-(x)&=x^{2\al}\psi_-(x);
\end{align*}
if $\al=0$:
\begin{align*}
\ff_+(x)&=\psi_+(x), & \ff_-(x)&=\psi_-(x)\log x.
\end{align*}

\begin{lem}\label{3.18} The operators  $L_{\nu,\al, {\rm bc}}$, 
$L_{\nu,\al, {\rm bc}, +}$,  $R_{\nu,\al}$ and $L_{\al,\al, {\rm bc}, -}$ if 
$0< \al\leq \frac{1}{2}$, are bounded below by zero.
\end{lem}
\begin{proof} Note that
\[
-h^{2\al-1}\left(h^{1-2\al}f'\right)'=-f''-(1-2\al)\frac{h'}{h} f'.
\]

Let $S$ denote the operator corresponding to one of the operators in the statement  under the transformation $T$. Then,  
\begin{align*}
\langle  S f,f\rangle&=\int_a^b \left(h^{1-2\al} \sf_{\nu,\al} (f) \bar f\right)(x) d x\\
&=\int_a^b \left(h^{1-2\al} \left( -f''-(1-2\al)\frac{h'}{h}f'-\frac{\al^2-\nu^2}{h^2}f\right)\bar f\right)(x) d x\\
&=-\int_a^b \left( \left(h^{1-2\al}f'\right)' \bar f\right)(x)dx
+(\nu^2-\al^2)\int_a^b \left(h^{1-2\al}\frac{|f|^2}{h^2}\right)(x) d x\\
&=-\left[h^{1-2\al}f' \bar f\right]_a^b+
\int_a^b \left( h^{1-2\al}f' \bar f'\right)(x)dx
+(\nu^2-\al^2)\int_a^b \left(h^{1-2\al}\frac{|f|^2}{h^2}\right)(x) d x,
\end{align*}
where either $a=0$ and $b=l$ for the operators $L$ or $a>0$ for the operator $R$.

Next, assume that $f$ is an eigenfunction of $S$. Then, 
\[
(h^{1-2\al}f' \bar f)(b)=0,
\]
since $f'(b)=0$ with absolute bc, while $f(b)=0$ with relative bc.

The analysis for the point $a$ requires more work. We distinguish the different cases.  

The easiest case of course is if  $S$ corresponds to $R_{\nu,\al}$, and $a>0$, since then
\[
(h^{1-2\al}f' \bar f)(a)=0, 
\]
because $f'(a)=0$ with absolute bc, while $f(a)=0$ with relative bc.

Next, consider the case where operator $S$ corresponds to   $L=L_{\nu,\al, {\rm bc}}$ or 
$L=L_{\nu,\al, {\rm bc}, +}$,  and $a=0$. In both cases, $f$ is a multiple of $\ff_+$. In the first, since $\ff_+$ is the unique square integrable solution (recall $\nu\geq 1$), and in the second by Remark \ref{abc}. Whence, near $x=0$, 
\[
(h^{1-2\al}f' \bar f)(x)\sim x^{1-2\al} x^{\al+\nu-1} x^{\al+\nu}=x^{2\nu},
\]
and therefore
\[
(h^{1-2\al}f' \bar f)(0)=0.
\] 

This works also for the last case,  namely when $S$ corresponds to $L_{\al,\al, {\rm bc}, -}$, with $0<\al< \frac{1}{2}$, and $a=0$. For in this case $f$ is a multiple of $\ff_-$ by Remark \ref{abc}. Whence,  if $0<\al<\frac{1}{2}$, 
\[
(h^{1-2\al}f' \bar f)(x)=
x^{1-2\al} H^{1-2\al}\psi'_-(x)\psi_-(x),
\]
since $\nu=\al$, and therefore
\[
(h^{1-2\al}f' \bar f)(0)=0.
\] 

If $\nu=\al=\frac{1}{2}$, we reduces to the regular case, and therefore either $f$ or $f'$ vanishes at $x=0$ by the boundary condition there.


Thus, in all cases
\begin{align*}
\langle  S f,f\rangle&=\la |f|^2=\int_a^b \left( h^{1-2\al}f' \bar f'\right)(x)dx
+(\nu^2-\al^2)\int_a^b \left(h^{1-2\al}\frac{|f|^2}{h^2}\right)(x) d x,
\end{align*}
is the sum of two non negative quantities, and the proof is completed. 
\end{proof}

\begin{corol} The operators $L_{\nu,\al, {\rm bc}}$, 
$L_{\nu,\al, {\rm bc}, +}$,  $R_{\nu,\al}$ and $L_{\al,\al, {\rm bc}, -}$ if 
$0< \al\leq \frac{1}{2}$ are non negative.
\end{corol}

\begin{lem}\label{bca} If $\nu>0$, the operators $L_{\nu, \al, {\rm bc}, \pm}$ are bounded below and their essential spectrum is void.
\end{lem}
\begin{proof} This follows by  \cite{DS2} XIII.10.C25, since
\[
{\rm liminf}_{x\to 0^+}\left|x^2 q_\nu(x)\right|=\lim_{x\to 0^+} \left|x^2 q(x)\right|=\nu^2-\frac{1}{4}.
\]
\end{proof}

\begin{lem}\label{eigenL} The operator $L_\nu=L_{\nu,\al,{\rm bc}, \pm}$ (excluding the case $L_{0,0,{\rm bc}, -}$) has a pure point non negative real spectrum, with unique accumulation point at infinity 
and all the eigenvalues are simple.
\end{lem}
\begin{proof} We have just seen that $L_\nu$ is bounded below zero, either by Lemma \ref{abc} of Lemma \ref{bca}. Since $L_\nu$ is self-adjoint, the spectrum is real. Since it is closed, it is easy to see that the continuum spectrum is contained in the essential spectrum. The last is void by \cite[XIII.7.17, pg. 1449]{DS2}. More precisely, by point (a) when $\nu>\frac{1}{2}$, since then
\[
\lim_{x\to 0^+} q_{\nu,\al}(x)=\lim_{x\to 0^+}\frac{\nu^2-\frac{1}{4}}{x^2}= +\infty;  
\]
by point (b) when $\nu<1$, since then ${\rm limsup}_{x\to 0^+}\left|x^2 q_{\nu,\al}(x)\right|<\frac{3}{4}$. The residual spectrum is void by \cite{Yos} XI.8.1. It follows that the spectrum is real pure point. Since $L_\nu$ is bounded below, by \cite{DS2} XIII.7.50, the eigenvalues have unique accumulation point at $+\infty$, and are all simple (see also \cite{Wei}  8.29).
\end{proof}

\begin{lem}\label{eigenR} The operator $R_{\nu,\al,{\rm bc}}$ has a pure point bounded below real spectrum, with unique accumulation point at infinity,  
and all the eigenvalues are simple.
\end{lem}
\begin{proof} The spectrum is real pure point with unique accumulation point at infinity by the spectral theorem for compact resolvent \cite[XIII.4.2, pg. 1331]{DS2}, since the resolvent is compact for non real $\la$ by \cite[XIII.4.1(1), pg. 1330]{DS2}, because the interval is compact. 
\end{proof} 

Whence,   spectrum of $L$ and $R$, denoted by $\Sp (L)$, and $\Sp (R)$, respectively, is an   infinite set of non negative different real numbers and with unique accumulation point at infinite.

\begin{rem} The formal operator $\lf_{\nu,\al}$ reduces to the formal operator $\lf^0_\nu$ when $h(x)=x$. We may reduce accordingly the domains of the operators $L_{\nu,\al, {\rm bc}}$, $L_{\nu,\al, {\rm bc}, \pm}$ and $R_{\nu,\al, {\rm bc}}$ in the case $h(x)=x$, and in this way we obtain the flat version of these concrete operators, denoted by $L^0_{\nu, {\rm bc}}$, $L^0_{\nu, {\rm bc}, \pm}$ and $R^0_{\nu, {\rm bc}}$. It is clear enough that the results proved for the general case hold true for the flat case. 
\end{rem}

\begin{lem}\label{eigen}  Let $\la_n$ be in the spectrum of either of $L_{\nu,\al, {\rm bc}}$, $L_{\nu,\al, {\rm bc}, +}$,  $R_{\nu,\al}$ or $L_{\al,\al, {\rm bc}, -}$, and $\la_n^0$ denote the corresponding eigenvalue of the  corresponding flat operator, then
\[
|\lambda_{\nu,n}-\lambda^0_{\nu,n}|\leq M,
\]
for all $n$ and  some positive constant $M$.
\end{lem}

\begin{proof} For simplicity denote by $L$ and $L^0$ the general and the flat operators. Let $u_{n}$ be the eigenvector of $L$ corresponding to the eigenvalue $\la_{n}$, and $u^0_{n}$ be the eigenvector of $L^0$ corresponding to the eigenvalue $\la^0_{n}$. 

Then,
\[
\la_n\langle u^0_n, u_n\rangle=\langle u^0_n, L u_n\rangle 
=\langle L^0 u^0_n,  u_n\rangle+\langle u^0_n, r u_n\rangle,
\]
where (see equation (\ref{xyz})
\begin{align*}
r(x)=&q_{\nu,\al}(x)-\frac{\nu^2-\frac{1}{4}}{x^2}\\
=&(\nu^2-\al^2)\frac{1-H^2(x)}{x^2H^2(x)}
-\left(\al-\frac{1}{2}\right)\frac{{H''}(x)}{H(x)}\\
&+\left(\al^2-\frac{1}{4}\right)\frac{{H'}^2(x)}{H^2(x)}
+2\left(\al-\frac{1}{2}\right)^2 \frac{H'(x)}{H(x)}x.
\end{align*}

In the case of the operator $R$, $\langle u^0_n, P u_n\rangle$ is finite, and the result follows immediately. For the operator $L_{\nu,\al, {\rm bc}}$, $L_{\nu,\al, {\rm bc}, +}$, if $\nu\not=|\al|$, by our assumption, near $x=0$
\[
r(x)\sim x^{\ep-2},
\]
for some $\ep>0$. On the other side, 
\[
u_n(x)\sim u_n^0(x)\sim x^{\frac{1}{2}+\nu}\log x,
\]
and hence 
\[
(r u^0_n,u_n)=\int_0^l r(x) u^0_n(x) u_n(x) dx,
\]
is finite, and the result follows. For the operator $L_{\al,\al, {\rm bc}, -}$, near $x=0$
\[
r(x)=O(1),
\]
and  
\[
u_n(x)\sim u_n^0(x)\sim x^{\frac{1}{2}-\al}\log x,
\]
so that 
\[
r(x) u^0_n(x)u_n(x)\sim x^{1-2\al}\log^2 x.
\]

Since $\al=\nu<1$ in this case, again the integral is finite, and this completes the proof. 
\end{proof}

Since the sequences $\Sp(L^0_{{\nu,\al}, {\rm bc}})$, $\Sp(L^0_{{\nu,\al}, {\rm bc}, \pm})$ and $\Sp(R^0_{{\nu,\al}, {\rm bc}})$ have order $\frac{1}{2}$ \cite{Spr4}, we have the following fact.

\begin{corol}\label{order} The sequences   $\Sp(L_{\nu,\al, {\rm bc}})$, $\Sp(L_{\nu,\al, {\rm bc}, +})$,  $\Sp(R_{\nu,\al})$ and  $\Sp(L_{\al,\al, {\rm bc}, -})$, have order $\frac{1}{2}$ and genus  $0$.
\end{corol}

Next, we want to characterise the kernels of the relevant operators.

\begin{lem}\label{kerL} The kernel of the operators $L_{\nu,\al, {\rm bc}}$, $L_{\nu,\al, {\rm bc}, +}$,  and  $L_{\al,\al, {\rm bc}, -}$, with $0<\al\leq \frac{1}{2}$, are as follows:

\begin{align*}
\ker L_{\nu,\al,{\rm abs}}&=\left\{\begin{array}{cc} \langle 0 \rangle, & \al\not=-\nu,\\
\langle h^{\frac{1}{2}+\nu}\rangle, &\al=-\nu\end{array}\right.\\
\ker L_{\nu,\al,{\rm rel}}&=\langle 0 \rangle,\\
\ker L_{\nu,\al,{\rm abs, +}}&=\left\{\begin{array}{cc} \langle 0 \rangle, & \al\not=-\nu,\\
\langle h^{\frac{1}{2}+\nu}\rangle, &\al=-\nu\end{array}\right.\\
\ker L_{\nu,\al,{\rm rel, +}}&=\langle 0 \rangle,\\
\ker L_{\al,\al,{\rm abs, -}}&=\langle h^{\frac{1}{2}-\nu}\rangle, \hspace{10pt} 0< \al\leq \frac{1}{2},\\
\ker L_{\al,\al,{\rm rel, -}}&=\langle 0 \rangle, \hspace{10pt} 0< \al\leq \frac{1}{2}.
\end{align*}

\end{lem}

\begin{proof} Starting with the last formula in the proof of  Lemma \ref{abc}, if $f$ is an harmonic 
\begin{align*}
0=\langle  S f,f\rangle&=\int_a^b \left( h^{1-2\al}f' \bar f'\right)(x)dx
+(\nu^2-\al^2)\int_a^b \left(h^{1-2\al}\frac{|f|^2}{h^2}\right)(x) d x,
\end{align*}
and therefore $\nu=|\al|$ and $f'=0$. This means that the kernel of all the operators is trivial if $\nu\not=|\al|$, and is generated by the possible constant solutions of the harmonic equation. Since this solution must belong to $L^2(h^{1-2\al} dx)$, it follows that the kernel is trivial if $\al\geq 1$. 

Assume now that $\nu=|\al|$. The solutions of the harmonic equation when $\nu=|\al|$ have been explicitly described in Remark \ref{nu=alpha}.  We need to identify what among these solutions belong to the domain of the relevant operator.

For $L_{\nu,\al, {\rm abs}}$ or $L_{\nu,\al, {\rm abs}, +}$, if $\al=\nu>0$, the constant function corresponds to the minus solution $\ff_-$, and therefore is not in the domain, because it is not in $L^2(h^{1-2\al}dx)$  in the first case, and  does not satisfy the boundary condition at $x=0$ in the second case. It follows that the kernel is trivial. If $-\al=\nu>0$, the constant function is the plus solution $\ff_+$, that is in the domain since it is in $L^2(h^{1-2\al}dx)$, it satisfies the boundary condition at $x=0$, and the boundary condition at $x=l$; therefore, the kernel is generated by $\uf_+$.

For $L_{\nu,\al, {\rm rel}}$ and $L_{\nu,\al, {\rm rel}, +}$, if $\al=\nu>0$, the constant function correspond to the minus solution $\ff_-$, and therefore is not in the domain, because it is not in $L^2(h^{1-2\al}dx)$  in the first case, and  does not satisfy the boundary condition at $x=0$ in the second case. It follows that the kernel is trivial. If $-\al=\nu>0$, the constant function is the plus solution $\ff_+$, that is in  $L^2(h^{1-2\al}dx)$, satisfies the boundary condition at $x=0$, but does not satisfy the boundary condition at $x=l$, and therefore the kernel is trivial.

For $L_{\al,\al, {\rm abs}, -}$, with $0<\al\leq \frac{1}{2}$, necessarily  $\al=\nu>0$, so that the constant function corresponds to the minus solution $\ff_-$, and  satisfies the boundary condition at $x=0$, and that at $x=l$; moreover it is square integrable since $\al\leq \frac{1}{2}$, and therefore the kernel is generated by $\uf_-$.  

For $L_{\al,\al, {\rm rel}, -}$, with $0<\al\leq \frac{1}{2}$, necessarily  $\al=\nu>0$, so that the constant function corresponds to the minus solution $\ff_-$, and  satisfies the boundary condition at $x=0$, but does not satisfy  that at $x=l$, so the kernel is trivial. \end{proof}

\begin{lem}\label{kerR} The kernel of the operators $R_{\nu,\al}$ is trival.
\end{lem}
\begin{proof} Proceeding as in the proof of Lemma \ref{kerR}, the constant function satisfy the absolute boundary condition at $x=b$ but not the relative boundary condition at $x=a$.
\end{proof}

\begin{rem}\label{resolvent} Note that in the singular case with $\nu< 1$, the deficiency indices of $L_{\nu,\al}$ are equal to the rank of the operator, therefore the resolvent is compact by \cite[XIII.4.1, pg. 1330]{DS2}. It follows that in that case we may apply the spectral theorem for compact resolvent to prove the result on the spectrum as in the case of the operator $R$ in Proposition \ref{eigenR} \cite[XIII.4.2, pg. 1330]{DS2}.
\end{rem}

\subsection{Resolvent}

\begin{prop}\label{p3.33} Let $S_{\nu,\al, \be,\be',\be_\pm}$ be any of the self-adjoint extensions of the formal operator $\sf_{\nu,\al}$ described in Theorem \ref{extensions}. Then, the resolvent $(\la I-S_{\nu,\al, \be,\be',\be_\pm})^{-1}$ is an integral operator with kernel 
\[
k(x,y)=\frac{1}{h(x)^{1-2\al}W(\ff_+,\ff_-)}h^{\frac{1}{2}-\al}(x)h^{\frac{1}{2}-\al}(y)\left\{\begin{array}{ll}\ff_{\beta,\beta'} (x)\ff_{\beta_\pm}(y),&x\geq y,\\ \ff_{\beta_\pm} (x)\ff_{\beta, \beta'}(y),&x< y,\end{array}\right.
\]
where $\ff_{\beta,\beta'}$ is the either (unique) solution of the equation $s_{\nu,\al, \be,\be',\be_\pm}f=\la f$ satisfying the boundary condition for $S_{\nu,\al, \be,\be',\be_\pm}$ at $x=l$ if we have the LCC at $x=l$, or the (unique) square integrable solution $\ff_{\beta,\beta'}$ of of the equation $s_{\nu,\al, \be,\be',\be_\pm}f=\la f$ if we have the LPC at $x=l$, and $\ff_{\beta\pm}$ is the either (unique) solution of the equation $s_{\nu,\al, \be,\be',\be_\pm}f=\la f$ satisfying the boundary condition for $S_{\nu,\al, \be,\be',\be_\pm}$ at $x=0$ if we have the LCC at $x=0$, or the (unique) square integrable solution $\ff_{\beta_\pm}$ of of the equation $s_{\nu,\al, \be,\be',\be_\pm}f=\la f$ if we have the LPC at $x=0$. The kernel of the operator $S_{\nu,\al, \be,\be',\be_\pm}$ may be obtained replacing the solutions $\ff$ by the solutions $\uf$.
\end{prop}
\begin{proof} The construction of the kernel may be found in  \cite[XIII.3]{DS2}. Alternatively, using \cite[8.29]{Wei}, with $p=r=h^{1-2\al}$, $q(x)=\frac{\nu^2-\al^2}{h^2(x)}$, then
\[
\sf_{\nu,\al} f=-\frac{1}{r}(pf')'+qf.
\]
\end{proof}

\begin{corol}\label{c3.34} Let $L_{\nu,\al, \be,\be',\be_\pm}$ be any of the self-adjoint extensions of the formal operator $\lf_{\nu,\al}$ described in Theorem \ref{extensions}. Then, $L_{\nu,\al, \be,\be',\be_\pm}$ has compact resolvent.
\end{corol}
\begin{proof} The resolvent os an integral operator by the Theorem. So for it to be compact it is sufficient to verify that it is square integrable on $(0,l)\times (0,l)$ (see for example \cite[pg. 1331]{DS2}). Observe that that $h(x)^{1-2\al} W(\ff_+,\ff_-)$ is constant and we denoted this number by $c$. Then we calculate
\begin{align*}
c^2\int_0^l\int_0^l |k(x,y)|^2 dx dy&=\int_0^l |\uf_{\beta,\beta'} (x)|^2\int_0^x |\uf_{\beta_\pm}(y)|^2 dy dx+
\int_0^l |\uf_{\beta_\pm} (x)|^2\int_x^l |\uf_{\beta, \beta'}(y)|^2dy.
\end{align*}

These functions are smooth, so problems appear only near $x=0$, and there the solutions $\uf_{\beta,\beta'}$ are multiple of $\uf_{\beta_\pm}$. Near $x=0$, $\uf_{\beta_\pm}(x)=\beta_+x^{\frac{1}{2}-\nu}+\beta_- x^{\frac{1}{2}+\nu}$. So near zero 
\[
 \int_0^\ep\int_0^\ep |k(x,y)|^2 dx dy={\rm const}  \int_0^\ep x^2 dx<\infty.
\]
\end{proof}

\begin{corol}\label{c3.35} Let $L_{\nu,\al, \be,\be',\be_\pm}$ be any of the self-adjoint extensions of the formal operator $\lf_{\nu,\al}$ described in Theorem \ref{extensions}. Then, $L_{\nu,\al, \be,\be',\be_\pm}$ has compact resolvent.
\end{corol}

\begin{corol}\label{c3.36} Let $L_{\nu,\al, \be,\be',\be_\pm}$ be any of the self-adjoint extensions of the formal operator $\lf_{\nu,\al}$ described in Theorem \ref{extensions}. Then, there exists a spectral resolution of $L_{\nu,\al, \be,\be',\be_\pm}$, i.e. the spectrum $\Sp (L_{\nu,\al, \be,\be',\be_\pm})$ of $L_{\nu,\al, \be,\be',\be_\pm}$ is real and pure point, i.e. coincides with the set of the eigenvalues, and discrete, i.e. the unique point of accumulation is infinite, and all eigenspaces have finite dimension. We write  $\Sp (L_{\nu,\al, \be,\be',\be_\pm})=\{\la_n\}_{n\in \Z}$. All the eigenvalues are simple. All eigenfunctions are smooth, and the set of the corresponding eigenfunction $\vv_{\la_n}$ is a complete orthonormal basis of $L^2(0,l)$.
\end{corol}
\begin{proof} See \cite[XIII.4.2, XIII.4.3]{DS2}. See \cite[8.29]{Wei} for simplicity of eigenvalues. See also \cite[10.6.1]{Zet} for the case where one point is LCC. 
\end{proof}

\begin{rem} Note that the kernel of $L_{\nu,\al, \be,\be',\be_\pm}$ is continuous, and therefore we may use \cite[Lemma pg. 65]{RS3} to compute the trace of the resolvent, namely
\[
\Tr (\la I-L_{\nu,\al, \be,\be',\be_\pm})^{-1}=\int_0^l k(x,x) dx.
\]
\end{rem}

\subsection{Other spectral functions}
\label{spectralfunctions}

Out last point in this section is to introduce the main spectral functions associated to the operators $R_{\nu,\al}$, $L_{\nu,\al, {\rm abs}}$,  $L_{\nu,\al, {\rm abs}, +}$,  and  $L_{\al,\al, {\rm abs}, -}$, with $0<\al\leq \frac{1}{2}$. Let $\Sigma$ any sector contained in $\C-\Sp_+(L)$ (where $L$ is any of the above operators), and $\Lambda$ the boundary of $\Sigma$. We assume the variable $\la$ always restricted in $\Sigma$. We denote by $-\la$ the complex variable in $\Sigma$ with $\arg (-\la)=0$ on the negative part of the real axes contained in $\Sigma$. With this convention, if $z=\sqrt{\la}$, then $iz =-\sqrt{-\la}$.

The eigenvalues of $L=L_{\nu,\al,{\rm abs} }$ and $L=L_{\nu,\al,{\rm abs, +} }$  may be characterised as follows. If $Lu=\la u$, then $u$ is a solution of equation $\lf_{\nu,\al}u=\la u$ that belongs to $D(L)$. If $\nu\geq 1$, the unique solution that is square integrable is $\uf_+$, that satisfies the boundary condition at $x=0$, thus for this solution to belong to the domain of $L_{\nu,\al,{\rm abs} }$ it is necessary and sufficient that it satisfies the boundary condition at $x=l$. If $\nu<1$, there are two square integrable solutions, $\uf_\pm$, but only $\uf_+$ satisfies the boundary condition at $x=0$, whence again  for this solution to belong to the domain of $L_{\nu,\al,{\rm abs, +} }$ it is necessary and sufficient that it satisfies the boundary condition at $x=l$. Whence, the eigenvalues of the operator $L$ are the zeros of the function
\[
B_{\nu,\al,{\rm abs,+}}(l,\la )=\left(h^{\al-\frac{1}{2}}(x) \uf_{+}(x,\la,\nu)\right)'(l),
\]
as a function of $\la$.  Since the solutions are analytic in $\la$, $B$ is an entire  function. 
By Lemma \ref{order},  $B$ has order  $\frac{1}{2}$, thus, we have the factorisation 
\[
B_{\nu,\al,{\rm abs,+}}(l,\la )=B_{\nu,\al,{\rm abs,+}}(l,0 ) \la^{\dim\ker L}\prod_{\la_{n}\in \Sp_0(L_{\nu,\al,{\rm abs,+}})} \left(1+\frac{-\la}{\la_{n}}\right),
\] 
where (recall that $\dim\ker L $ is either 0 or 1)
\[
B_{\nu,\al,{\rm abs,+}}(l,0)=\lim_{\la\to 0} \frac{B_{\nu,\al,{\rm abs,+}}(l,\la)}{\la^{\dim\ker L}}.
\]

It follows that:
\beq\label{ww}
\begin{aligned}
\log\Gamma(-\lambda, L)=&\log\Gamma(-\lambda,\Sp_0 (L_{\nu,\al,{\rm abs,+}}))\\
=&-\log\prod_{\la_{n}\in \Sp_0(L_{\nu,\al,{\rm abs,+}})} \left( 1+\frac{-\lambda}{\lambda_{n}}\right)\\
=&\log B_{\nu,\al,{\rm abs,+}}(l,0 )+ \dim\ker L_{\nu,\al, {\rm abs, +}}\ln\la\\
&-\log B_{\nu,\al,{\rm abs,+}}(l,\la ).
\end{aligned}
\eeq

The Gamma function for the operators $L_{\nu,\al,{\rm rel}}$ and $L_{\nu,\al,{\rm rel,+}}$ is the same but with the function
\[
B_{\nu,\al,{\rm rel,+}}(l,\la )= \uf_{+}(l,\la,\nu)(l).
\]

Next, consider the operator $L_{\al,\al,{\rm abs, -}}$, with $0<\al\leq \frac{1}{2}$. The analysis is the same as for $L_{\al,\al,{\rm abs,+}}$, up to the fact that the solution satisfying the boundary condition at $x=0$  is now $\uf_-$ instead that $\uf_+$. This gives the following function $B$:
\[
B_{\al,\al,{\rm abs, -}}(l,\la )=\left(h^{\al-\frac{1}{2}}(x) \uf_{-}(x,\la,\al)\right)'(l).
\]

Similarly, for the operator $L_{\nu,\al,{\rm rel, -}}$, we have
\[
B_{\al,\al,{\rm rel, -}}(l,\la )= \uf_{-}(l,\la,\al).
\]

We conclude with  the operator $R_{\nu,\al, {\rm abs, rel}}$.  A similar analysis shows that the eigenvalues of $R_{\nu, \al }$  are the zeros of the function 
\begin{align*}
A_{\nu,\al, {\rm abs, rel}}(a,b,\la )=& \uf_+(a,\la,\nu)\left. \left(h^{\al-\frac{1}{2}}(x) \uf_-(x,\la,\nu)\right)'\right|_{x=b}\\
&-\uf_-(a,\la,\nu)\left. \left( h^{\al-\frac{1}{2}}(x)\uf_+(x,\la,\nu)\right)'\right|_{x=b},
\end{align*}
as a function of $\la$.  Since the solutions are analytic in $\la$, $A_\nu$ is an entire  function. 
By Corollary  \ref{order},  $A_\nu$ has order  $\frac{1}{2}$, and by Lemma \ref{kerR}, $R_{\nu,\al}$ has trivial kernel, thus, we have the factorisation
\[
A_{\nu,\al, {\rm abs, rel}}(a,b,\la )=A_{\nu,\al, {\rm abs, rel}}(a,b,0 )  \prod_{\la_{n }\in \Sp(R_{\nu,\al, {\rm abs, rel}})} \left(1+\frac{-\la}{\la_{n }}\right),
\] 
and hence
\beq\label{ww1}
\begin{aligned}
\log\Gamma(-\lambda, R_{\nu,\al, {\rm abs, rel}})&=\log\Gamma(-\lambda,\Sp_0 (R_{\nu,\al, {\rm abs, rel}}))\\
&=-\log\prod_{\la_{n }\in \Sp_0(R_{\nu,\al, {\rm abs, rel}})} \left( 1+\frac{-\lambda}{\lambda_{n }}\right)\\
&=\log A_{\nu,\al, {\rm abs, rel}}(a,b,0 )-\log A_{\nu,\al, {\rm abs, rel}}(a,b,\la ).
\end{aligned}
\eeq

\subsection{Asymptotic expansions of the solutions of the fundamental system}
\label{largela}

In this section we provide the asymptotic expansions of the solution in the fundamental system of solution of equation (\ref{eqdiff1}), first for large values of $\la$ and second for large values of $\nu$. Note that these solutions are normalised as in Definition \ref{defi1}. The proof are given using classical method of asymptotic analysis,   and a new approach consisting in introducing a perturbation of the flat case, main reference is Olver \cite{Olv} (see also \cite{Mur} and \cite{Mar}).

\begin{lem}\label{l2.3} The  equation
\[
v''(x)+\left(z^2-\frac{\nu^2-\frac{1}{4}}{h^2(x)}- p(x)\right) v(x)=0,
\]
has two linearly independent solutions $v_\pm(x,z )$,   that for large $z$ (in the suitable sector) have the following asymptotic expansions
\begin{align*}
v_{\mp}(x,z,\nu )
=&C_\mp \e^{\pm i x z}\left(1+\left(\mp\frac{i}{2}\left(\frac{1}{4}-\nu^2\right)\int \frac{1}{h^2(x)}\pm \frac{i}{2} \int p(x) dx\right)\frac{1}{z}+\dots\right).
\end{align*}
uniform  in $x$ for $x$ in any compact subset of $(0,l]$. The coefficients are functions of $\nu^2$.
\end{lem}
\
\begin{proof} Assume $z$ large in module in some suitable sector of the complex plane, and write 
\[
v(x,z,\nu )=\e^{z w_0(x,\nu )+w_1(x,\nu )+z^{-1} w_2(x,\nu )+\dots}=\e^{F(x,z,\nu )},
\]
asymptotically, where the coefficients $w_j$ are smooth in $x$ for $x\in(0,l]$. Thus,
\begin{align*}
v'(x,z,\nu )=&F'(x,z,\nu )\e^{F(x,z,\nu )},\\
v''(x,z,\nu )=&(F'(x,z,\nu ))^2\e^{F(x,z,\nu )}+F''(x,z,\nu )\e^{F(x,z,\nu )},
\end{align*}

Substitution gives
\[
F''+{F'}^2+z^2-\frac{\nu^2-\frac{1}{4}}{h^2}- p=0,
\]
i.e.

\begin{align*}
z w''_0+w''_1+z^{-1} w''_2+\dots+(z w'_0+w'_1+z^{-1} w'_2+\dots)^2
+z^2-\frac{\nu^2-\frac{1}{4}}{h^2}- p =0.
\end{align*}

We proceed considering the coefficients of the powers of $z$.  For $z^2$, we have
\[
{w_0'}^2+1=0,
\]
that gives: 
\[
w_{0\pm}(x,\nu )=\pm i x,
\]
up to an additive  constant. Observe that this constant will give a multiplicative constant in the final expansion. 
Next, the term in $z$ is:
\[
w_0''+2w_0'w_1'=0,
\]
that gives
\begin{align*}
w_1'(x,\nu )&=0,&w_1(x,\nu )&=0,
\end{align*}
up to an additive constant, that again will give a multiplicative constant in the final expansion. For this reason, we will not consider the single additive constants, but we collect them all in a  unique multiplicative one in the final formula. The constant term is:
\[
w_1''+{w_1'}^2+2w_0'w_2'=\frac{\nu^2-\frac{1}{4}}{h^2}+ p,
\]
that gives
\begin{align*}
w_{2\pm}'(x,\nu )&=\mp\frac{i}{2}\left(\left(\nu^2-\frac{1}{4}\right)\frac{1}{h^2(x)}+ p(x)\right)\\
w_{2\pm}(x,\nu )&=\mp\frac{i}{2}\int \left(\left(\nu^2-\frac{1}{4}\right)\frac{1}{h^2(x)}+ p(x)\right) dx,
\end{align*}
since we are collecting the constants. This gives
\begin{align*}
v_{\mp}(x,z,\nu )\sim& \e^{\pm i x z\mp\frac{i}{2}\int \left(\left(\nu^2-\frac{1}{4}\right)\frac{1}{h^2(x)}+ p(x)\right) dx\frac{1}{z}+\dots}\\
\sim&\e^{\pm i x z}\left(1\mp\frac{i}{2}\int \left(\left(\nu^2-\frac{1}{4}\right)\frac{1}{h^2(x)}+ p(x)\right) dx\frac{1}{z}
+\dots\right).
\end{align*}
\end{proof}

It remains to fix the values of the constants. This is easy enough in the regular case $\nu=\frac{1}{2}$, but requires a little more work for general $\nu$. We present a separate proof for the case $\nu=\frac{1}{2}$, to be superseded by the following one given for any $\nu$ (in the allowed range).

\begin{lem} The solutions $\uf_\pm$ in the fundamental system of solutions of equation (\ref{eqdiff1}), normalised as in Definition \ref{defi1},  and with $\nu=\frac{1}{2}$, have the following expansions for large $\la$: 
\begin{align*}
\uf_{-}(x,\la )=&\frac{1}{2}\left(\e^{ x\sqrt{-\la}}\left(1-\frac{1}{2} \int p(x) dx\frac{1}{\sqrt{-\la}}+\dots\right)\right.\\
&\left.+\e^{- x\sqrt{-\la}}\left(1+\frac{1}{2} \int p(x) dx\frac{1}{\sqrt{-\la}}+\dots\right)\right),\\
\uf_{+}(x,\la )=&\frac{1}{2\sqrt{-\la}}\left(\e^{ x\sqrt{-\la}}\left(1-\frac{1}{2} \int p(x) dx\frac{1}{\sqrt{-\la}}+\dots\right)\right.\\
&\left.-\e^{- x\sqrt{-\la}}\left(1+\frac{1}{2} \int p(x) dx\frac{1}{\sqrt{-\la}}+\dots\right)\right),
\end{align*}
uniformly   in $x$ for $x\in [0,l]$.

\end{lem}

\begin{proof} Since $\nu=\frac{1}{2}$, the expansion for large $\la$ is uniform in $x$ for small $x$, so we may impose the initial conditions described in Remark \ref{IC} in order to deduce the missing multiplicative constants. Let
\[
v(x)=c_+ v_+(x)+c_- v_-(x),
\]
be the general solution for large $\la$. Then, imposing $IC_0$, we obtain the equations
\[
c_+\left(1-\frac{1}{2} \int p(x) dx\frac{1}{\sqrt{-\la}}+\dots\right)
+c_-\left(1+\frac{1}{2} \int p(x) dx\frac{1}{\sqrt{-\la}}+\dots\right)
=1,
\]
and 
\begin{align*}
&c_+\left(\sqrt{-\la}\left(1-\frac{1}{2} \int p(x) dx\frac{1}{\sqrt{-\la}}+\dots\right)+\left(0-\frac{1}{2}  p(0) \frac{1}{\sqrt{-\la}}+\dots\right)\right)\\
&\hspace{40pt}+c_-\left(-\sqrt{-\la}\left(1+\frac{1}{2} \int p(x) dx\frac{1}{\sqrt{-\la}}+\dots\right)+\left(0+\frac{1}{2}  p(0) \frac{1}{\sqrt{-\la}}+\dots\right)\right)
=0,
\end{align*}
equating the terms of order 0 in the first equation, and those of order $\sqrt{-\la}$ in the second gives $c_-=c_+=\frac{1}{2}$. Thus $\uf_-=\frac{1}{2}(v_++v_-)$, and this gives the expansion of $u_-$. Next, for $u_+$, we impose $IC'_0$.  
\end{proof}

\begin{lem}\label{explambda} The solutions $\uf_\pm$  in the fundamental system of solutions of equation (\ref{eqdiff1}), and their combination $\vf$, normalised as in Definition \ref{defi1},   have the following expansions for large $\la$, uniformly  in $x$ for $x$ in any compact subset of $(0,l]$.  
If $\nu\not\in\Z$:

\begin{align*}
\uf_{+}&(x,\lambda,\nu )\\
=&\frac{2^{\nu} \Gamma(\nu+1)}{\sqrt{2\pi}(-\la)^{\frac{1}{2}\left(\nu+\frac{1}{2}\right)}}
\left(\e^{ x\sqrt{-\la}}\left(1-\frac{1}{2}\int \left(\left(\nu^2-\frac{1}{4}\right)\frac{1}{h^2(x)}+ p(x)\right) dx\frac{1}{\sqrt{-\la}}
+\dots\right)\right.\\
&\hspace{50pt}\left.+\e^{-x\sqrt{-\la}}\left(1+\frac{1}{2}\int \left(\left(\nu^2-\frac{1}{4}\right)\frac{1}{h^2(x)}+ p(x)\right) dx\frac{1}{\sqrt{-\la}}
+\dots\right)\right),\\
\uf_{-}&(x,\lambda,\nu)\\
=&\frac{2^{-\nu} \Gamma(-\nu+1)}{\sqrt{2\pi}(-\la)^{\frac{1}{2}\left(-\nu+\frac{1}{2}\right)}}
\left(\e^{ x\sqrt{-\la}}\left(1-\frac{1}{2}\int \left(\left(\nu^2-\frac{1}{4}\right)\frac{1}{h^2(x)}+ p(x)\right) dx\frac{1}{\sqrt{-\la}}
+\dots\right)\right.\\
&\hspace{50pt}\left.+\e^{-x\sqrt{-\la}}\left(1+\frac{1}{2}\int\left(\left(\nu^2-\frac{1}{4}\right)\frac{1}{h^2(x)}+  p(x) \right)dx\frac{1}{\sqrt{-\la}}
+\dots\right)\right),
\end{align*}
\begin{align*}
\vf(x,\la,\nu)=& \frac{ \sqrt{2\pi}\nu}{(-\la)^{\frac{1}{2}(-\nu+\frac{1}{2})}} \e^{- x\sqrt{-\la}}\left(1+\frac{1}{2}\int\left(\left(\nu^2-\frac{1}{4}\right)\frac{1}{h^2(x)}
+  p(x) \right)dx\frac{1}{\sqrt{-\la}}+\dots\right).
\end{align*}

If $\nu\in\Z$, $\nu>0$:
\begin{align*}
\uf_{-}(x,\la,\nu)&=\frac{2(-1)^\nu \nu\left(\frac{1}{2}\log(-\la)-\log 2\right)}{\sqrt{2\pi}2^{\nu} \Gamma(\nu+1)(-\la)^{\frac{1}{4}} }\e^{x\sqrt{-\la}}\left(1+O\left(\frac{1}{\sqrt{-\la}}\right)\right)\\
&\qquad+\frac{ \sqrt{2\pi}\nu}{2^{\nu} \Gamma( \nu+1)(-\la)^{\frac{1}{2}(-\nu+\frac{1}{2})}} \e^{-x\sqrt{-\la}}
\left(1+O\left(\frac{1}{\sqrt{-\la}}\right)\right).
\end{align*}

If $\nu=0$:
\begin{equation*}
\begin{aligned}
\uf_{+}(x,\lambda,0 )=&\frac{\e^{ x\sqrt{-\la}}}{\sqrt{2\pi}(-\la)^{\frac{1}{4}}}\left(1-\frac{1}{2}\int \left(-\frac{1}{4}\frac{1}{h^2(x)}+ p(x)\right) dx\frac{1}{\sqrt{-\la}}
+\dots\right)\\
&\hspace{50pt}+\frac{\e^{-x\sqrt{-\la}}}{\sqrt{2\pi}(-\la)^{\frac{1}{4}}}\left(1+\frac{1}{2}\int \left(-\frac{1}{4}\frac{1}{h^2(x)}+ p(x)\right) dx\frac{1}{\sqrt{-\la}}
+\dots\right),\\
\uf_{-}(x,\la,0)&=\frac{\log 2 - \gamma - \log\sqrt{-\la}}{\sqrt{2\pi}(-\la)^{\frac{1}{4}} }\e^{x\sqrt{-\la}}\left(1+O\left(\frac{1}{\sqrt{-\la}}\right)\right)\\
&\qquad\left(\frac{\log 2 - \gamma - \log\sqrt{-\la}}{\sqrt{2\pi}}-\sqrt{\frac{2}{\pi}} \right)\frac{ 1}{(-\la)^{\frac{1}{4}}} \e^{-x\sqrt{-\la}}
\left(1+O\left(\frac{1}{\sqrt{-\la}}\right)\right),\\
\vf(x,\la,0)&=\sqrt{\frac{\pi}{2}} \frac{ 1}{(-\la)^{\frac{1}{4}}} \e^{-x\sqrt{-\la}}
\left(1+O\left(\frac{1}{\sqrt{-\la}}\right)\right).
\end{aligned}
\end{equation*}

\end{lem}
\begin{proof} To compute the constants for any $\nu$ we consider a perturbation of the equation. More precisely, consider the equation 
\beq\label{eee}
v''(x)+\left(\la-\frac{\nu^2-\frac{1}{4}}{h_\ep^2(x)}- p_\ep(x)\right) v(x)=0,
\eeq
where  $\ep\in [0,1]$,    $h_\ep$ is a smooth family of smooth functions   on $[0,l]$, satisfying
\begin{align*}
\lim_{x\to 0^+}\frac{h_\ep(x)}{x}&=1,
\end{align*}
uniformly in $\ep$, and 
\begin{align*}
\lim_{\ep\to 0^+}h_\ep(x)&=x,
\end{align*}
uniformly in $x$, and  $p_\ep$ is  smooth family of smooth functions  on $(0,l]$, satisfying
\[
\lim_{x\to 0^+}x^2 p_\ep(x)=0,
\]
uniformly in $\ep$. Thus the flat case corresponds to $\ep=0$. This may be obtained for example taking $h_\ep(x)=xH(1+\ep x^\ep)$. 
In this setting, and reviewing the proof of Lemma \ref{l2.3}, we  observe two things. First, that the coefficient  of the leading term namely the one appearing in the determination of the function $\omega_0$ does not depend on $\ep$. Only the coefficients appearing in the following terms do depend on $\ep$, and second, that these other coefficients are smooth in $\ep$. This means that we may consider one multiplicative constant that does not depend on $\ep$, and additive constant in each term in the expansion that has a (negative) power of $\la$. Therefore, we may determine the multiplicative  constant taking the limit $\ep\to 0^+$ in  the expansion of the solution and requiring that it coincides with the expansion of the explicit solution given for $\ep=0$, i.e. for the flat case,  in Remark \ref{rem3.1}. 
The last can be found using the known expansion for the Bessel functions. We give details for the case where $\nu$ is not an  integer. 
{\small
\begin{align*}
\uf^0_{\pm}(x,\lambda,\nu)=&\frac{2^{\pm\nu} \Gamma(\pm\nu+1)}{(-\la)^\frac{\pm\nu}{2}}\sqrt{x}I_{\pm\nu}(\sqrt{-\la} x)\\
=&\frac{2^{\pm\nu} \Gamma(\pm\nu+1)}{(-\la)^\frac{\pm\nu}{2}} \left(\frac{\e^{\sqrt{-\la}x}}{\sqrt{2\pi \sqrt{-\la}}}\left(1+O\left(\frac{1}{\sqrt{-\la}}\right)\right)+\frac{\e^{-\sqrt{-\la}x}}{\sqrt{2\pi \sqrt{-\la}}}\left(1+O\left(\frac{1}{\sqrt{-\la}}\right)\right)\right)
\end{align*}
}

Let 
\begin{align*}
v_\ep(x)&=c_+v_+(x,\la,\ep)+c_- v_-(x,\la,\ep)\\
&=c_+ \e^{ x\sqrt{-\la}}\left(1-O\left(\frac{1}{\sqrt{-\la}}\right)\right)
+c_- \e^{- x\sqrt{-\la}}\left(1+O\left(\frac{1}{\sqrt{-\la}}\right)\right),
\end{align*}
be the expansion of the general solution of equation (\ref{eee}), then imposing 
\[
\lim_{\ep\to 0^+} v_\ep(x)=u_{\ep,+}(x,\lambda,\nu),
\]
gives
\[
c_+=c_-=\frac{2^{\nu} \Gamma(\nu+1)}{\sqrt{2\pi}(-\la)^{\frac{1}{2}\left(\nu+\frac{1}{2}\right)}},
\]
while imposing  
\[
\lim_{\ep\to 0^+} v_\ep(x)=u_{\ep,-}(x,\lambda,\nu),
\]
gives
\[
c_+=c_-=\frac{2^{-\nu} \Gamma(-\nu+1)}{\sqrt{2\pi}(-\la)^{\frac{1}{2}\left(-\nu+\frac{1}{2}\right)}}.
\]

If $\nu=0$,  consider 
\[
\vf(x,\la,0) = - \uf_{-}(x,\la,0) + (\gamma-\log 2 + \log\sqrt{-\la}) \uf_{+}(x,\la,0);
\]
then
\begin{align*}
\uf^0_{+}(x,\lambda,0)=&\sqrt{x}I_{0}(\sqrt{-\la} x),\\
\uf^0_{-}(x,\lambda,0) =&- \sqrt{x}K_0(\sqrt{-\la} \ x) + \left(\log 2 - \gamma - \log\sqrt{-\la})\right) \sqrt{x} I_0(\sqrt{-\la} \ x),\\
\vf^0(x,\la,0) =& \sqrt{x} K_0(\sqrt{-\la}\ x),
\end{align*}
and  we use the same strategy as above to obtain the expansions.

\end{proof}

\begin{lem}\label{explambdader} The derivative of the solutions $\uf_\pm$ in the fundamental system of solutions of equation (\ref{eqdiff1}) and that of the function $\vf$, normalised as in Definition \ref{defi1},   have the following expansions for large $\la$, uniformly   in $x$ for $x$ in any compact subset of $(0,l]$.

If $\nu\not\in\Z$:

\begin{align*}
\uf_{+}'&(x,\lambda,\nu )\\
=&\frac{2^{\nu} \Gamma(\nu+1)}{\sqrt{2\pi}(-\la)^{\frac{1}{2}\left(\nu-\frac{1}{2}\right)}}
\left(\e^{ x\sqrt{-\la}}\left(1-\frac{1}{2}\int\left(\left(\nu^2-\frac{1}{4}\right)\frac{1}{h^2(x)}+  p(x) \right)dx\frac{1}{\sqrt{-\la}}
+\dots\right)\right.\\
&\hspace{50pt}\left.+\e^{-x\sqrt{-\la}}\left(1+\frac{1}{2}\int\left(\left(\nu^2-\frac{1}{4}\right)\frac{1}{h^2(x)}+  p(x) \right)dx\frac{1}{\sqrt{-\la}}
+\dots\right)\right),\\
\uf_{-}'&(x,\lambda,\nu)\\
=&\frac{2^{-\nu} \Gamma(-\nu+1)}{\sqrt{2\pi}(-\la)^{\frac{1}{2}\left(-\nu-\frac{1}{2}\right)}}
\left(\e^{ x\sqrt{-\la}}\left(1-\frac{1}{2}\int\left(\left(\nu^2-\frac{1}{4}\right)\frac{1}{h^2(x)}+  p(x) \right)dx\frac{1}{\sqrt{-\la}}
+\dots\right)\right.\\
&\hspace{50pt}\left.+\e^{-x\sqrt{-\la}}\left(1+\frac{1}{2}\int\left(\left(\nu^2-\frac{1}{4}\right)\frac{1}{h^2(x)}+  p(x) \right)dx\frac{1}{\sqrt{-\la}}
+\dots\right)\right),
\end{align*}
\begin{align*}
\vf'(x,\la,\nu)&= \frac{ \sqrt{2\pi}\nu}{(-\la)^{\frac{1}{2}(-\nu-\frac{1}{2})}} \e^{- x\sqrt{-\la}}\left(1+\frac{1}{2}\int\left(\left(\nu^2-\frac{1}{4}\right)\frac{1}{h^2(x)}+  p(x) \right)dx\frac{1}{\sqrt{-\la}}
+\dots\right).
\end{align*}

If $\nu\in\Z$, $\nu>0$:
\begin{align*}
\uf'_{-}(x,\la,\nu)&=\frac{2(-1)^\nu \nu\left(\frac{1}{2}\log(-\la)-\log 2\right)}{\sqrt{2\pi}2^{\nu-1} \Gamma(\nu)(-\la)^{\frac{1}{2}(\nu-\frac{1}{2})} }\e^{x\sqrt{-\la}}\left(1+O\left(\frac{1}{\sqrt{-\la}}\right)\right)\\
&\qquad+\frac{ \sqrt{2\pi}\nu}{2^{|\al |} \Gamma( \nu+1)(-\la)^{\frac{1}{2}(-\nu-\frac{1}{2})}} \e^{-x\sqrt{-\la}}
\left(1+O\left(\frac{1}{\sqrt{-\la}}\right)\right),\\
\vf'(x,\la,\nu)&=\frac{ \sqrt{2\pi}\nu}{(-\la)^{\frac{1}{2}(-\nu-\frac{1}{2})}} \e^{-x\sqrt{-\la}}
\left(1+O\left(\frac{1}{\sqrt{-\la}}\right)\right).
\end{align*}

If $\nu=0$: 
\begin{align*}
\uf_{+}'(x,\lambda,0 )=&\frac{1}{\sqrt{2\pi}(-\la)^{-\frac{1}{4}}}
\left(\e^{ x\sqrt{-\la}}\left(1-\frac{1}{2}\int\left(-\frac{1}{4}\frac{1}{h^2(x)}+  p(x) \right)dx\frac{1}{\sqrt{-\la}}
+\dots\right)\right.\\
&\hspace{50pt}\left.-\e^{-x\sqrt{-\la}}\left(1+\frac{1}{2}\int\left(-\frac{1}{4}\frac{1}{h^2(x)}+  p(x) \right)dx\frac{1}{\sqrt{-\la}}
+\dots\right)\right),\\
\uf'_{-}(x,\la,0)&=\frac{\log 2 - \gamma - \log\sqrt{-\la}}{\sqrt{2\pi}(-\la)^{-\frac{1}{4}} }\e^{x\sqrt{-\la}}\left(1+O\left(\frac{1}{\sqrt{-\la}}\right)\right)\\
&\qquad-\left(\frac{\log 2 - \gamma - \log\sqrt{-\la}}{\sqrt{2\pi}}-\sqrt{\frac{2}{\pi}} \right)\frac{ 1}{(-\la)^{-\frac{1}{4}}} \e^{-x\sqrt{-\la}}
\left(1+O\left(\frac{1}{\sqrt{-\la}}\right)\right),\\
\vf'(x,\la,0)&=-\sqrt{\frac{\pi}{2}} \frac{ 1}{(-\la)^{-\frac{1}{4}}} \e^{-x\sqrt{-\la}}
\left(1+O\left(\frac{1}{\sqrt{-\la}}\right)\right).
\end{align*}

\end{lem} 
\begin{proof} These expansion are obtained by taking the derivative of the ones given in Lemma \ref{explambda}. \end{proof}

Next, we study the asymptotic expansions for large $\nu$.

\begin{lem}\label{l3.5} The  equation
\beq
\label{neweq}
w''(x)+\left(z^2\nu^2-\frac{\nu^2-\frac{1}{4}}{h(x)^2}- p(x)\right) w(x)=0,
\eeq
has two linearly independent solutions $w_{\pm}(x,z,\nu)$,   that for large $\nu$  have the following asymptotic expansions
\begin{align*}
w_{\pm}(x,z,\nu)
=\frac{C_\pm}{\left(\frac{1}{h^2(x)}-z^2\right)^\frac{1}{4}}\e^{\pm\nu\int \sqrt{\frac{1}{h^2(x)}-z^2} dx}\left(\sum_{j=0}^m U_j(x,z) (\pm\nu)^{-j}+ O\left(\frac{1}{\nu^m}\right)\right).
\end{align*}
uniformely in $x$ for $x$ in any compact subset of $(0,l]$,   analytic and uniform in $z$, for $z\in \Sigma_{\theta,c}=\left\{z\in \C~|~|\arg(z-c)|> \frac{\theta}{2}\right\}$, $c,\theta>0$.
\end{lem}
\begin{proof} We proceed as in Chapter 10 of \cite{Olv}. Also, as in the proof of Lemma \ref{l2.3}, we collect the constants in one undetermined multiplicative constant. Setting
\begin{align*}
f(x,z)&=\frac{1}{h(x)^2}-z^2,\\
g(x)&= p(x)-\frac{1}{4 h(x)^2},
\end{align*}
and with the change of variable 
\[
t=t(x)=\int \sqrt{\frac{1}{h(x)^2}-z^2} dx,
\]
the equation 
\[
w''(x)+\left(z^2\nu^2-\frac{\nu^2-\frac{1}{4}}{h(x)^2}- p(x)\right) w(x)=w''(x)-\left(\nu^2 f(x,z)+g(x)\right)w(x)=0,
\]
becomes
\[
W''+(\nu^2 +\psi)W=0,
\]
where
\[
w(x,z,\nu)=f^{-\frac{1}{4}}(x,z) W(t(x),z,\nu),
\]
and
\[
\psi(t,z)=\frac{f(x(t),z,\ep)}{g(x(t))}-\frac{1}{f^\frac{3}{4}(x(t),z)}\left. \frac{d^2}{d x^2}f^{-\frac{1}{4}}(x,z)\right|_{x=x(t)}.
\]

Considering series solutions
\[
W(t,z,\pm\nu)=\e^{\pm\nu t}\sum_{j=0}^\infty U_j(t,z) (\pm\nu)^{-j},
\]
we obtain $U_0=1$, and the relation
\[
U_{j+1}(x,z)=-\frac{1}{2}f^{-\frac{1}{2}}(x,z,\ep)U'_j(x,z,\ep)+\int U(x,z) U_j(x,z),
\]
where
\[
U(x,z)=\frac{16 f^2(x,z)g(x)+4f(x,z)f''(x,z)-5(f'(x,z))^2}
{32 f^\frac{5}{2}(x,z)}.
\]

This prove the first part of the lemma. Explicit evaluation of the coefficients shows that
\[
U_j(t,z)=O(z^{-k}),
\]
with $k\geq 1$, for $j>1$, and concludes the proof. 
\end{proof}

\begin{lem}\label{expnu} The solution $\uf_+$  of the fundamental system of solutions of equation (\ref{eqdiff1}), and the  combination $\vf$ of such solutions, normalised as in Definition \ref{defi1}, and with $\lambda$ replaced by $\la\nu^2$, have the following expansions for large $\nu$: 
\begin{align*}
\uf_{+}(x,\la\nu^2,\nu)=&\frac{2^{\nu} \Gamma(\nu+1)}{\sqrt{2\pi\nu}(-\la)^\frac{\nu}{2}}
\frac{\sqrt{h(x)}}{\left(1-\la h^2(x)\right)^\frac{1}{4}}\e^{\nu\mathlarger{\mathlarger{\int}} \frac{\sqrt{1-\la h^2(x)}}{ h(x)} dx}\\
&\left(\sum_{j=0}^J U_j(x,i\sqrt{-\la}) \nu^{-j}+ O\left(\frac{1}{\nu^{J+1}}\right)\right),\\
\vf(x,\la\nu^2,\nu)
=&\sqrt{2\pi\nu}\frac{\sqrt{h(x)}}{\left(1-\la h^2(x)\right)^\frac{1}{4}}\e^{-\nu\mathlarger{\mathlarger{\int}} \frac{\sqrt{1-\la h^2(x)}}{ h(x)} dx}\\
&\left(\sum_{j=0}^J U_j(x,i\sqrt{-\la}) (-\nu)^{-j}+ O\left(\frac{1}{\nu^{J+1}}\right)\right)
\end{align*}
uniformly   in $x$ for $x$ in any compact subset of $(0,l]$, analytic and uniform in $\la$, for $\la\in \Sigma_{\theta,c}=\left\{z\in \C~|~|\arg(z-c)|< \frac{\theta}{2}\right\}$, $c,\theta>0$, where  $U_0=1$, and
\[
U_j(x,i\sqrt{-\la})=O\left(\frac{1}{(-\la)^\frac{k}{2}}\right),
\]
for all $j>1$, with some $k>1$.
\end{lem}
\begin{proof} If $u(x,\la,\nu)$ $(z=i\sqrt{-\la}$) is a solution of the equation 
\[
u''(x)+\left(z^2-\frac{\nu^2-\frac{1}{4}}{h(x)^2}- p(x)\right) u(x)=0,
\]
it is clear that $u(x,\la\nu^2,\nu)$ is a solution of equation (\ref{neweq}), and hence
\[
u_\pm(x,\la\nu^2,\nu)=c_+w_+(x,\la,\nu)+c_- w_-(x,\la,\nu).
\]

Therefore, for large $\nu$,
\begin{align*}
u_\pm(x,\la\nu^2,\nu)=&\frac{C_+}{\left(\frac{1}{h^2(x)}-z^2\right)^\frac{1}{4}}\e^{\nu\int \sqrt{\frac{1}{h^2(x)}-z^2} dx}\left(1+ O\left(\frac{1}{\nu}\right)\right)\\
&+\frac{C_-}{\left(\frac{1}{h^2(x)}-z^2\right)^\frac{1}{4}}\e^{-\nu\int \sqrt{\frac{1}{h^2(x)}-z^2} dx}\left(1+ O\left(\frac{1}{\nu}\right)\right)
\end{align*}

Next, observe that the integration constants appear exactly as in the proof of Lemma \ref{l2.3}, and therefore we may tackle them by the same mean as in the proof of Lemma \ref{explambda}, namely considering a the operator as a perturbation of the flat one. 

In the flat case,  $\ep=0$,  writing $u=-izx=\sqrt{-\la} x$,
\[
\int \sqrt{\frac{1}{x^2}-z^2} dx=\sqrt{1-z^2 x^2}+\log\frac{-izx}{\sqrt{1-z^2 x^2}+1}.
\]

Whence the solutions have the expansions
\begin{align*}
w^0_{\pm}(x,i\sqrt{-\la},\nu)
&=\frac{\sqrt{x}}{\left(1-\la x^2\right)^\frac{1}{4}}\e^{\pm \nu\sqrt{1-\la x^2}\pm\nu\log\frac{\sqrt{-\la}x}{\sqrt{1-\la x^2}+1}}\left(1+O\left(\frac{1}{\nu^2}\right)\right).
\end{align*}

We compare these solutions with the two normalised solutions given in Remark \ref{rem3.1}. For the dominant one we have
\begin{align*}
\uf^0_{+}(x,\la\nu^2)&=\frac{2^{\nu} \Gamma(\nu+1)}{(-\la)^\frac{\nu}{2}}\sqrt{x}I_{\nu}(\sqrt{-\la} \nu x)\\
&=\frac{2^{\nu} \Gamma(\nu+1)}{(-\la)^\frac{\nu}{2}}
\frac{\sqrt{x}}{\sqrt{2\pi\nu}\left(1-\la x^2\right)^\frac{1}{4}}\e^{ \nu\sqrt{1+\la x^2}+\nu\log\frac{\sqrt{-\la}x}{\sqrt{1-\la x^2}+1}}
\left(1+O\left(\frac{1}{\nu}\right)\right).
\end{align*}

This shows that $\uf_{0,+}$ is a multiple of $w_{0,+}$, and therefore for large $\nu$
\begin{align*}
\uf^0_{+}(x,\la\nu^2,\nu,\ep)&=\frac{2^{\nu} \Gamma(\nu+1)}{\sqrt{2\pi\nu}(-\la)^\frac{\nu}{2}}
\frac{\sqrt{h_{\ep}(x)}}{\left(1-\la h^2_{\ep}(x)\right)^\frac{1}{4}}\e^{\nu\mathlarger{\mathlarger{\int}} \frac{\sqrt{1-\la h^2_{\ep}(x)}}{ h_{\ep}(x)} dx}\left(1+O\left(\frac{1}{\nu}\right)\right).
\end{align*}

Next, for the recessive solution, observing  that 
\begin{align*}
\vf^0(x,\la\nu^2,\nu)&=2\nu \sqrt{x}K_\nu(\sqrt{-\la} x)\\
&=\sqrt{2\pi\nu}\frac{\sqrt{x}}{\left(1-\la x^2\right)^\frac{1}{4}}\e^{ -\nu\sqrt{1+\la x^2}-\nu\log\frac{\sqrt{-\la}x}{\sqrt{1-\la x^2}+1}}
\left(1+O\left(\frac{1}{-\nu}\right)\right),
\end{align*}
we see that that $\vf^0$ is a multiple of $w^0_{-}$, and this completes the proof. 
\end{proof}

\begin{lem}\label{expnuprimo} The derivative of the solution $\uf_+$  in the fundamental system of solutions of equation (\ref{eqdiff1}), and the combination $\vf$ of such solutions, normalised as in Definition \ref{defi1}, and with $\la$ replaced by $\la\nu^2$, have the following expansions for large $\nu$: 
\begin{align*}
\uf'_{+}(x,\la\nu^2,\nu)
=&\frac{2^{\nu} \Gamma(\nu+1)}{(-\la)^\frac{\nu}{2}}\sqrt{\frac{\nu}{2\pi}}
\frac{\left(1-\la h^2(x)\right)^\frac{1}{4}}{\sqrt{h(x)}}\e^{\nu\mathlarger{\mathlarger{\int}} \frac{\sqrt{1-\la h^2(x)}}{ h(x)} dx}\\
&\left(\sum_{j=0}^m V_j(x,i\sqrt{-\la}) \nu^{-j}+ O\left(\frac{1}{\nu^{j+1}}\right)\right),\\
\vf'(x,\la\nu^2,\nu)
=&-\sqrt{2\pi\nu^3}\frac{\sqrt{h(x)}}{\left(1-\la h^2(x)\right)^\frac{1}{4}}\e^{-\nu\mathlarger{\mathlarger{\int}} \frac{\sqrt{1-\la h^2(x)}}{ h(x)} dx}\\
&\left(\sum_{j=0}^J V_j(x,i\sqrt{-\la}) (-\nu)^{-j}+ O\left(\frac{1}{\nu^{J+1}}\right)\right),
\end{align*}
uniformly  in $x$ for $x$ in any compact subset of $(0,l]$, and analytic in $\la$, for $\la\in \Sigma_{\theta,c}=\left\{z\in \C~|~|\arg(z-c)|> \frac{\theta}{2}\right\}$, $c,\theta>0$, where $V_0=1$, and
\[
V_j(x,i\sqrt{-\la})=O\left(\frac{1}{(-\la)^\frac{k}{2}}\right),
\]
for all $j>1$, with some $k>1$.
\end{lem}
\begin{proof} Since the expansion of the last lemma are uniform in $x$ we can differentiate in $x$ to obtain
\begin{align*}
\uf'_{+}(x,\la\nu^2,\nu)
=&\frac{2^{\nu} \Gamma(\nu+1)}{(-\la)^\frac{\nu}{2}}\sqrt{\frac{\nu}{2\pi}}
\frac{\left(1-\la h^2(x)\right)^\frac{1}{4}}{\sqrt{h(x)}}\\
&\e^{\nu\mathlarger{\mathlarger{\int}} \frac{\sqrt{1-\la h^2(x)}}{ h(x)} dx}\left(\sum_{j=0}^m V_j(x,i\sqrt{-\la}) \nu^{-j}+ O\left(\frac{1}{\nu^{j+1}}\right)\right),
\end{align*}
where
\begin{equation*}
\begin{aligned}
V_j(x,i\sqrt{-\la}) =& U_{j}(x,i\sqrt{-\la}) + \left(\frac{\sqrt{h(x)}}{(1-\la h^2(x))^{\frac{1}{4}}}\right)' \frac{\sqrt{h(x)}}{(1-\la h^2(x))^{\frac{1}{4}}} \ U_{j-1}(x,i\sqrt{-\la})\\
&+ \frac{h(x)}{(1-\la h^2(x))^{\frac{1}{2}}} \ U'_{j-1}(x,i\sqrt{-\la})\\
=&U_{j}(x,i\sqrt{-\la}) + \frac{1}{2}\frac{h(x)}{(1-\la h^2(x))^{\frac{1}{2}}} \ U_{j-1}(x,i\sqrt{-\la}) \\
&+ \frac{h(x)}{(1-\la h^2(x))^{\frac{1}{2}}} \ U'_{j-1}(x,i\sqrt{-\la}).
\end{aligned}
\end{equation*}
\end{proof}

We conclude this section and this part with a technical lemma that solves a key point in the determination of the boundary contribution to analytic torsion, see the proof of Theorem ?.

\subsection{A technical lemma}  
\label{s2.26}

For further use we want to compare the uniform expansions for large $\nu$ of the functions $\ln\Gamma(-\lambda\nu^2, R_{\nu,\al,{\rm abs, rel} })$ and $\ln\Gamma(-\lambda\nu^2, L_{\nu, \al,{\rm abs}, \pm })$, with $a=l_1$, $b=l_2=l$, $a<b$, see equations (\ref{ww}) and (\ref{ww1}). This means to compare the expansions of the functions
$\log A_{\nu,\al,{\rm abs, rel}}(a,b,\nu\la )$ and $\ln B_{\nu, \al,{\rm abs}, \pm }(l,\nu\la )$. Note that, since we are interested here in large $\nu$, it is equivalent to work with the $+$ or the $-$ extension. We will just write $A_\nu$ and $B_\nu$ for these two functions.

By definition
\begin{align*}
A_\nu(a,b,\la\nu^2 )=&\left(\al   \uf_+(a,\la\nu^2,\nu )+\al  ' \uf_+'(a,\la\nu^2,\nu )\right)\left(\be   \uf_-(b,\la\nu^2,\nu )+\be  ' u_-'(b,\la\nu^2,\nu )\right)\\
&-\left(\al   \uf_-(a,\la\nu^2,\nu )+\al  ' \uf_-'(a,\la\nu^2,\nu )\right)\left(\be   \uf_+(b,\la\nu^2,\nu )+\be  ' \uf_+'(b,\la\nu^2,\nu )\right).
\end{align*}
and
\[
B_\nu(l,\la\nu^2 )=\ga   \uf_{+}(l,\la\nu^2,\nu )+\ga  ' \uf_{+}'(l,\la\nu^2,\nu ),
\]
with suitable $\al,\al',\be,\be'\ga,\ga'$.

Using the expansions in Lemmas \ref{expnu} and \ref{expnuprimo}, and assuming that 
\[
\Re \mathlarger{\mathlarger{\int}} \frac{\sqrt{1-\la h^2(x)}}{ h(x)} dx \Big{|}_{x=b}
>
\Re \mathlarger{\mathlarger{\int}} \frac{\sqrt{1-\la h^2(x)}}{ h(x)} dx \Big{|}_{x=a},
\]
we obtain
\begin{align*}
A_\nu(a,b,\la\nu^2 )=&-\frac{2^\nu\Gamma(\nu+1)}{\sin\pi\nu}(-\la)^{-\frac{\nu}{2}}\e^{\nu \left(\mathlarger{\mathlarger{\int}} \frac{\sqrt{1-\la h^2(x)}}{ h(x)} dx \Big{|}_{x=b}
-
\mathlarger{\mathlarger{\int}} \frac{\sqrt{1-\la h^2(x)}}{ h(x)} dx \Big{|}_{x=a}\right)}\\
&\left(
\left( \be   \frac{\sqrt{h(b)}}{\left(1-\la h^2(b)\right)^\frac{1}{4}}\left(\sum_{j=0}^J U_j(b,i\sqrt{-\la} ) \nu^{-j}+ O\left(\frac{1}{\nu^{J+1}}\right)\right)\right.\right.\\
&\hspace{20pt}\left.-\nu \be'   \frac{\left(1-\la h^2(b)\right)^\frac{1}{4}}{\sqrt{h(b)}}\left(\sum_{j=0}^J V_j(b,i\sqrt{-\la} ) \nu^{-j}+ O\left(\frac{1}{\nu^{J+1}}\right)\right)\right)\\
&\left.
\left( \al   \frac{\sqrt{h(a)}}{\left(1-\la h^2(a)\right)^\frac{1}{4}}\left(\sum_{j=0}^J U_j(a,i\sqrt{-\la} ) (-\nu)^{-j}+ O\left(\frac{1}{\nu^{J+1}}\right)\right)\right.\right.\\
&\hspace{20pt}\left.-\nu \al'   \frac{\left(1-\la h^2(a)\right)^\frac{1}{4}}{\sqrt{h(a)}}\left(\sum_{j=0}^J V_j(a,i\sqrt{-\la} ) (-\nu)^{-j}+ O\left(\frac{1}{\nu^{J+1}}\right)\right)\right)\\
&\hspace{40pt}\left.+O\left(\e^{-2\nu \left(\mathlarger{\mathlarger{\int}} \frac{\sqrt{1-\la h^2(x)}}{ h(x)} dx \Big{|}_{x=b}
-
\mathlarger{\mathlarger{\int}} \frac{\sqrt{1-\la h^2(x)}}{ h(x)} dx \Big{|}_{x=a}\right)}\right)\right).
\end{align*}

We are interested in the coefficients of the terms in negative powers of $\nu$, that all come from the following terms:
\begin{align*}
\log A_\nu(a,b,\la\nu^2 )=&\dots +\log\Gamma(\nu+1)\\
&+\log \left( \nu \al'   \frac{\left(1-\la h^2(a)\right)^\frac{1}{4}}{\sqrt{h(a)}}\left(\sum_{j=0}^J V_j(a,i\sqrt{-\la} ) (-\nu)^{-j}+ O\left(\frac{1}{\nu^{J+1}}\right)\right)\right.\\
&\hspace{20pt}\left.-\al   \frac{\sqrt{h(a)}}{\left(1-\la h^2(a)\right)^\frac{1}{4}}\left(\sum_{j=0}^J U_j(a,i\sqrt{-\la} ) (-\nu)^{-j}+ O\left(\frac{1}{\nu^{J+1}}\right)\right)\right)\\
&+\log \left( \be   \frac{\sqrt{h(b)}}{\left(1-\la h^2(b)\right)^\frac{1}{4}}\left(\sum_{j=0}^J U_j(b,i\sqrt{-\la} ) \nu^{-j}+ O\left(\frac{1}{\nu^{J+1}}\right)\right)\right.\\
&\hspace{20pt}\left.-\nu \be'   \frac{\left(1-\la h^2(b)\right)^\frac{1}{4}}{\sqrt{h(b)}}\left(\sum_{j=0}^J V_j(b,i\sqrt{-\la} ) \nu^{-j}+ O\left(\frac{1}{\nu^{J+1}}\right)\right)\right)+\dots,
\end{align*}

The same substitution gives 
\begin{align*}
\log B_\nu(l,\la\nu^2 )=&\dots +\log\Gamma(\nu+1)\\
&+\log \left( \ga   \frac{\sqrt{h(b)}}{\left(1-\la h^2(b)\right)^\frac{1}{4}}\left(\sum_{j=0}^J U_j(b,i\sqrt{-\la} ) \nu^{-j}+ O\left(\frac{1}{\nu^{J+1}}\right)\right)\right.\\
&\hspace{20pt}\left.-\nu \ga'   \frac{\left(1-\la h^2(b)\right)^\frac{1}{4}}{\sqrt{h(b)}}\left(\sum_{j=0}^J V_j(b,i\sqrt{-\la} ) \nu^{-j}+ O\left(\frac{1}{\nu^{J+1}}\right)\right)\right)+\dots,
\end{align*}

A direct computation gives the following result. There, the sign comes from the same analysis performed with $a$ and $b$ switched.

\begin{prop} \label{p2.1}The functions $\ln\Gamma(-\lambda\nu^2, L_{\nu,\al, {\rm abs}, + })$ and $\ln\Gamma(-\lambda\nu^2, R_{\nu,\al })$ have  asymptotic expansions for large $\nu$, uniform in $\lambda$, for $\la\in \Sigma_{c,\theta}$. Let 
\[
\ln\Gamma(-\lambda\nu^2, L_{\nu, \al,{\rm abs}, + })=\dots+\sum_{j=1}^\infty \left(\phi_j(l_2,\la )+C_j(l_2 )\right)\nu^{-j}+\dots,
\]
and
\[
\ln\Gamma(-\lambda\nu^2, R_{\nu, \al,{\rm abs, rel} })=\dots+\sum_{j=1}^\infty \left(\psi_j(l_1,l_2,\la )+D_j(l_1,l_2 )\right)\nu^{-j}+\dots.
\]
then
\[
\psi_j(l_1,l_2,\la )=\sgn(l_2-l_1)^j\phi_j(l_2,\la )|_{\ga  =\be  ,\ga'  =\be'  }+\sgn(l_1-l_2)^j \phi_j(l_1,\la )|_{\ga  =\al  ,\ga'  =\al'  },
\]
where the notation means that the boundary conditions of the operator $L_{\nu,\al }$ are choosen identifying  the  values of $\ga  $ and $\ga'  $ as indicated.
\end{prop}

\section{Analysis on the deformed cone}
\label{ancone}

\subsection{The underlying geometry} 
\label{geocone}

Let $(W,g)$ be an orientable  compact connected Riemannian manifold of finite dimension $m$ without boundary  and with Riemannian structure $g$. Let $a<b$ be two non negative  real numbers. We denote by $I_{a,b}$ either the interval $[a,b]$, if $a>0$,  or the interval $(0,l]$. Consider the manifold $C_{a,b}(W)=I_{a,b}\times W$. This is connected manifold of dimension $n=m+1$ with boundary, either compact when $a>0$ or open when $a=0$. When $a>0$, the boundary of $C_{a,b}(W)$ is 
\[
\b C_{a,b}(W)=\b_{a}C_{a,b}(W)\cup \b_{a}C_{a,b}(W)=\{a,b\}\times W,
\]
when $a=0$, the boundary of $C_{0,b}(W)$ is 
\[
\b C_{0,b}(W)= \b C_{a,b}(W)=\{b\}\times W.
\]

Let $x$ the natural global coordinate on $I_{a,b}$, and $h(x)=xH(x)$, with  $H$  a smooth non vanishing  function    on $[0,l]$, with $H(0)=1$. Then 
\beq\label{g1}
\g_h=dx\otimes dx+h^2(x) g,
\eeq
is a Riemannian metric on $C_{a,b}(W)$. Each connected component of the boundary  is of course diffeomorphic to $W$, and isometric to $(W,h(l)^2g)$, where either $l=a$ or $b$.  For the  global coordinate $x$ corresponds to the local coordinate $x'=l-x$, where $x'$ is the geodesic distance from the boundary. Therefore, $ g_\b(x')=h(l-x)^2g$, and if $i:W\to C_{a,b}(W)$ denotes the inclusion, $i^*(dx\otimes dx+h(x)^2g)=g_\b(0)=h(l)^2 g$. If $y$ is a local coordinate system on $W$, then $(x,y)$ is a local coordinate system on $C_{a,b}(W)$. Following common notation, we call $(W,  g)$ the {\it section } of $C_{a,b}(W)$. Also following usual notation, a tilde will denote operations on the section (of course $\tilde g=g$), and not on the boundary.

In particular we call the space $C_{a>0,b}(W)$ the {\it finite metric frustum} over $W$, the space $C_{0,b}(W)$ the {\it open finite metric cone} over $W$, and the space $\overline{C_{0,b}(W)}=C_{0,b}(W)\cup \{a\}$ the {\it completed finite metric cone} over $W$.

\subsection{The formal Hodge-Laplace operator}\label{hodge}
For a differentiable manifold $M$, we denote by $\Omega^{q}(M)$ the space of smooth sections of forms on $M$, $\Gamma(M,\Lambda^{(q)}T^* M)$. Let $\star$, $\d$, $\d^\dagger$ and $\Delta$ denote the Hodge star, the exterior derivative, its dual and  Hodge-Laplace operator on $C_{a,b}(W)$ induced by the metric $\g_h$. In this section we give the explicit form of these formal operators. For $\omega\in \Omega^{q}(C_{a,b}(W))$, set
\[
\omega(x,y)=f_1(x)\tilde\omega_1(y)+f_2(x)dx\wedge \tilde\omega_2(y),
\]
with smooth functions $f_1,f_2 \in C^{\infty}(I_{a,b})$, and  forms, $\tilde\omega_1\in \Omega^{q}(W)$ and $\tilde\omega_2 \in \Omega^{q-1}(W)$. Then,
\begin{align*}
\star \omega(x,y)=&
h(x)^{m-2(q-1)} f_2(x)\tilde\star \tilde\omega_2(y)+(-1)^q (h(x))^{m-2q}f_1(x) 
dx\wedge\tilde\star \tilde\omega_1(y),\\
\d\omega (x,y) =&f_1(x) \tilde d\tilde\omega_1(y) + d x\wedge\left(  
f'_1(x) \tilde\omega_1(y) - f_2(x)  \tilde \d\tilde\omega_2(y)\right),\\
\d^{\dag} \omega(x,y) =& (-1)^{(m+1)q+(m+1)+1}\star d \star\omega(x,y)\\
=&-h(x)^{-m+2(q-1)} \left(h(x)^{m-2(q-1)} f_2(x)\right)'\tilde\omega_2(y) + 
h(x)^{-2}f_1(x) \tilde d^{\dag} \tilde\omega_1(y) \\
& - h(x)^{-2}f_2(x)dx \wedge \tilde d^{\dag} \tilde\omega_2(y),\\
\end{align*}
and, after some simplifications,
\begin{equation*}
\begin{aligned}
\Delta^{(q)}\omega(x,y)&= - \left( \left(m-2q\right) h(x)^{-1}  h'(x) 
f'_1(x) +  f''_1(x)
\right)\tilde\omega_1(y)+ h(x)^{-2}f_1(x)\tilde \Delta^{(q)}\tilde\omega_1(y)\\
&\qquad - 2 h(x)^{-1}h'(x) f_2(x)   \tilde d\tilde\omega_2(y) \\
&\qquad - \left( \left(m - 
2(q-1)\right) (h(x)^{-1}h'(x)f_2(x))' +  f''_2(x) \right)dx\wedge
\tilde\omega_2(y)\\
&\qquad +dx\wedge \tilde d^{\dag}\tilde\omega_1(y)\left( -2h(x)^{-3} h'(x) f_1(x) 
\right) + h(x)^{-2}f_2(x) dx\wedge\tilde\Delta^{(q-1)}\tilde \omega_2(y).
\end{aligned}
\end{equation*}

Observe that all structure are product, thus we may apply the exponential low and consider the adjoint functions. In fact, we have  the bilinear bijection $\ad_q$ of the space $\Omega^q(C_{a,b}(W))$ onto 
$C^\infty(I_{a,b}, \Omega^q(W)\times \Omega^{q-1}(W))=C^\infty (I_{a,b},\Gamma(W,\Lambda^{(q)}T^* W \times \Lambda^{(q-1)}T^*W))$: 
\begin{align*}
\ad_q&: \Omega^q(C_{a,b}(W))\to C^\infty(I_{a,b}, \Omega^q(W)\times \Omega^{q-1}(W)),\\
\ad_q&:f_1(x)\tilde\omega_1(y)+f_2(x)dx\wedge\tilde\omega_2(y)\mapsto (h(x)^{-\alpha_q 
+\frac{1}{2}}f_1(x)\tilde\omega_1(y),h(x)^{-\alpha_{q-1} + 
\frac{1}{2}}f_2(x)\tilde\omega_2(y)).
\end{align*}
with inverse
\begin{align*}
\ad_q^{-1}&:C^\infty(I_{a,b}, \Omega^q(W)\times \Omega^{q-1}(W))\to \Omega^q(C_{a,b}(W)),\\
\ad_q^{-1}&:(\omega^{(q)}(x),\omega^{(q-1)}(x))\mapsto h(x)^{\alpha_q - 
\frac{1}{2}}p^*\omega^{(q)}(x)+h(x)^{\alpha_{q-1}-\frac{1}{2}} \d x\wedge 
p^*\omega^{(q-1)}(x),
\end{align*}
where 
\begin{align*}
p&:C_{a,b}(M)\to W,\\
p&:(x,y)\mapsto y,
\end{align*}
is the projection, and
\[
\alpha_q=\frac{1}{2}(1+2q-m)=q+\frac{1}{2}(1-m),
\]
(note that $\alpha_{q\pm 1}=\alpha_q\pm 1$).

We will identify $C^\infty(I_{a,b}, \Omega^q(W)\times \Omega^{q-1}(W))$ with $C^\infty (I_{a,b})\otimes \Omega^{q}( W)\times C^\infty (I_{a,b}) \otimes\Omega^{q-1}(W)$, and we will write the vector in the last space corresponding to the form $\omega(x,y)=f_1(x)\tilde\omega_1(y)+f_2(x)dx\wedge \tilde\omega_2(y)$ as
\[
\ad(\omega)(x)=(u_1(x),u_2(x)),
\]

Therefore, the ``change of basis''  on the components of the vectors is:
\[
\left\{
\begin{array}{ll}
f_1(x)\tilde\omega_1 = h(x)^{\al_q-\frac{1}{2}} u_1(x), \\
f_2(x)\tilde\omega_2 = h(x)^{\al_{q-1}-\frac{1}{2}} u_2(x).
\end{array}
\right.
\]

The  Laplace operator on $C^\infty (I_{a,b})\otimes \Omega^{q}( W)\times C^\infty (I_{a,b}) \otimes\Omega^{q-1}(W)$ reads
\begin{equation}\label{lapdec}
\AF^q(x)=\ad_q\Delta^{(q)}\ad^{-1}_q=\left(\begin{matrix} -\frac{d^2}{d x^2}  & 0\\
0 & -\frac{d^2}{d x^2}
\end{matrix}
\right)+\frac{ h''(x)}{h(x)} B^{(q)}+\frac{A^{(q)}(x)}{h(x)^2},
\end{equation}
where
\[
B^{(q)}
=\left(\begin{matrix} -\alpha_{q}+\frac{1}{2}  & 0\\
0 & \alpha_{q-2}+\frac{1}{2}
\end{matrix}
\right),
\]
\begin{align*}
A^{(q)}(x)
&=\left(\begin{matrix}\tilde\Delta^{(q)}+ (\alpha_q^2-\frac{1}{4})(h'(x))^2 & -2 h'(x) \tilde \d\\-2  h'(x) \tilde \d^\dagger& \tilde\Delta^{(q-1)}+ \left(\alpha_{q-2}^2-\frac{1}{4}\right) (h'(x))^2\end{matrix}\right).
\end{align*}

The  operators  $\d$ and $\d^{\dag}$  in $C^\infty (I_{a,b})\otimes \Omega^{q}( W)\times C^\infty (I_{a,b}) \otimes\Omega^{q-1}(W)$ are:
\[
\df= \ad_{q+1} \d~ \ad_q^{-1} = 
\frac{1}{h(x)}\left(\begin{matrix} 
\tilde\d  & 0 \\
\left(\al_q-\frac{1}{2}\right)h'(x)+h(x)\frac{d}{d x}
 & -\tilde\d
\end{matrix}
\right)
\]

and

\[
\df^\da = \ad_{q-1} \d^{\dag} \ad_q^{-1} = 
\frac{1}{h(x)}\left(\begin{matrix} 
\tilde\d^\dag & 
\left(\al_{q-1}-\frac{1}{2}\right)h'(x)-h(x)\frac{d}{d x}\\
0 & -\tilde\d^\dag
\end{matrix}
\right)
\]

The following commutative diagram illustrates the setting
\[
\xymatrix{\Omega^{q}(C_{a,b}(W)) \ar[r]^{\d}\ar[d]_{\ad_{q}}& 
\Omega^{q}(C_{a,b}(W)) \ar[d]^{\ad_{q+1}} \\
C^{\infty}(I_{a,b},\Omega^{q}(W) \times \Omega^{q-1}(W)) 
\ar[r]_{\ad_q \d~\ad_{q+1}^{-1}}  & 
C^{\infty}(I_{a,b},\Omega^{q+1}(W)\times \Omega^{q}(W))
}
\]

We denote by $L^2(C_{a,b}(W))$ the complete separable Hilbert space of the square integrable forms $\omega$ on $C_{a,b}(W)$, where the inner product is 
\beq\label{innerF}
\langle\omega,\omega'\rangle=\int_0^l \int_W \omega\wedge\star\omega' .
\eeq

The inner product on $C^\infty(I_{a,b}, \Omega^q(W)\times \Omega^{q-1}(W))$:
\[
\langle(\omega_1,\omega_2),(\eta_1,\eta_2)\rangle_{I_{a,b}}=\int_0^l \left(\langle\omega_1(x),\eta_1(x)\rangle_W+\langle\omega_2(x),\eta_2(x)\rangle_W\right)\d x,
\] 
where 
\[
\langle\tilde\omega,\tilde\eta\rangle_W=\int_W \tilde\omega\wedge \tilde\star \tilde\eta,
\]
Then we have a complete separable Hilbert space, denoted by 
$L^2(I_{a,b},\Omega^q(W)\times \Omega^{q-1}(W))$, and  it makes $\ad_q$  an isometry.

\begin{rem} Note that freezing the section, we have smooth functions on the interval $C^\infty(I_{a,b})$ in both cases, however the measure on the corresponding spaces of square integrable functions are different. For $L^2(C_{a,b}(W))$ gives the space  $L^2(I_{a,b}, h^{1-2\al_q}(x)dx)$, while the adjoint space gives $L^2(I_{a,b},dx)$. This two spaces are isometric under the exponential law.
\end{rem}

\subsection{Geometric boundary conditions}\label{BC} 

We describe in this section the classical boundary conditions \cite[3.2]{RS}. For $\omega\in \Omega^{q}(C_{a,b}(W))$ with
\[
\omega(x,y)=f_1(x)\tilde\omega_1(y)+f_2(x)dx\wedge\tilde\omega_2(y),
\]
with smooth functions $f_1,f_2 \in C^{\infty}(I_{a,b})$, and  forms, $\tilde\omega_1\in \Omega^{q}(W)$ and $\tilde\omega_2 \in \Omega^{q-1}(W)$, we call $f_1(x)\tilde\omega_1(y)$ the {\it tangent component} of $\omega$, denoted by $\omega_{\rm tg}$,  and $f_2(x) \tilde\omega_2(y)$ the {\it normal component} of $\omega$, denoted by $\omega_{\rm norm}$. Then the absolute BC  on the boundary of $C_{a,b}(W)$ are:
\[
\BB_{\rm abs}(\omega)=0\hspace{20pt}{\rm if~ and~ only~ if}\hspace{20pt}\left\{\begin{array}{l}\omega_{\rm norm}|_{\b W}=0,\\
(\d\omega)_{\rm norm}|_{\b W}=0,\\
       \end{array}
\right.
\]
the relative BC are
\[
\BB_{\rm rel}(\omega)=0\hspace{20pt}{\rm if~ and~ only~ if}\hspace{20pt}\left\{\begin{array}{l}\omega_{\rm tg}|_{\b W}=0,\\
(\d^\dagger\omega)_{\rm tg}|_{\b W}=0.\\
       \end{array}
\right.
\]

A simple calculation, recalling that $h$ may not vanish on $(0,l]$,  gives

\beq\label{abs1}
\BB_{\rm abs}(\omega)=0\iff\left\{\begin{array}{l}f_2(l) \tilde\omega_2(y)=0,\\
 f_1'(l) \tilde\omega_1(y) =0,
       \end{array}
\right.
\eeq
\beq\label{rel1}
\BB_{\rm rel}(\omega)=0\iff\left\{\begin{array}{l}f_1(l)\tilde\omega_1(y)=0,\\
\left((h(x)^{m-2(q-1)}f_2(x)\right)'_{x=l}\tilde\omega_2(y)=0.
       \end{array}
\right.
\eeq

On the vector $(u_1,u_2)$ in the adjoint space 
$C^\infty (I_{a,b})\otimes \Omega^{q}( W)\times C^\infty (I_{a,b}) \otimes\Omega^{q-1}(W)$, the BC read
\beq\label{abs2}
\BB_{\rm abs}(u_1,u_2)=0\iff\left\{\begin{array}{l} u_2(l)=0,\\
 (h(x)^{\al_q-\frac{1}{2}} u_1(x))'_{x=l} =0,
       \end{array}
\right.
\eeq
\beq\label{rel2}
\BB_{\rm rel}(u_1,u_2)=0\iff\left\{\begin{array}{l}u_1(l)=0,\\
\left((h(x)^{\frac{1}{2}-\al_{q-1}}u_2(x)\right)'_{x=l}=0.
       \end{array}
\right.
\eeq

\subsection{Decomposition of the Hodge-Laplace formal operator}
\label{declap}

The adjoint space  $C^\infty(I_{a,b}, \Omega^q(W)\times \Omega^{q-1}(W))$ is the space of the smooth functions on the interval $I_{a,b}$ with values in $\Gamma(W,\Lambda^{(q)}T^* W\times \Lambda^{(q-1)}T^*W)$.  The corresponding formal operator $\AF^q$ (see equation \ref{lapdec}) has the following particular form (compare with \cite{BS2})
\[
\AF^q=-\frac{d^2}{dx^2}+\tilde\AF^q(x),
\]
where $\tilde\AF^q(x)$ is a function on $I_{a,b}$ with values in space of the operators on the section. However, in the present case, we may go a little further and write
\[
\AF^q=-\frac{d^2}{dx^2}+a(x)\tilde\AF^q,
\]
where $\tilde \AF$ are some fixed operators on the section, and $a$ some smooth functions. More precisely, the operators appearing in $\tilde \AF$ are the Hodge-Laplace operator, the exterior derivative and its adjoint. Thus,  we may consider the corresponding concrete operators on $L^2(W)$, and we have an orthonormal  basis of the last space given by  eigenfunctions of $\tilde \Delta$. We now develop this argument in details.

Let $\tilde \lambda_{q,n}\not=0$ be a non zero   eigenvalue of 
$\tilde\Delta^{(q)}$ on $W$, and $\tilde\E^{(q)}_{\tilde\lambda_{q,n}}$ the associated 
eigenspace. Then, the Hodge decomposition induces  the decomposition (see \cite{RS} pg. 154) (observe that any of these two spaces may be trivial)
\[
\tilde\E^{(q)}_{\tilde\lambda_{q,n}}=\tilde\E^{(q)}_{\tilde\lambda_{q,n}, {\rm ex}}\oplus \tilde\E^{(q)}_{\tilde\lambda_{q,n},{\rm cex}},
\]
where 
\begin{align*}
\tilde\E^{(q)}_{\tilde\lambda_{q,n},{\rm ex}}&=\{\tilde\omega\in\Omega^{q}(W,E_\rho)~|~\tilde\Delta\tilde\omega=\tilde 
\lambda_{q,n}\tilde\omega, \, \tilde\omega=\tilde d\tilde\alpha\},\\
\tilde\E^{(q)}_{\tilde\lambda_{q,n},{\rm cex}}&=\{\tilde\omega\in\Omega^{q}(W,E_\rho)~|~\tilde\Delta\tilde\omega=\tilde 
\lambda_{q,n}\tilde\omega, \, \tilde\omega=\tilde d^\dagger\tilde\alpha\}.
\end{align*}

Let $\{\tilde\varphi_{\tilde\lambda_{q,n},{\rm cex},j}\}$ and $\{\tilde\varphi_{\tilde\lambda_{q,n},{\rm ex},k}\}$ be orthonormal bases of $\tilde \E^{(q)}_{\tilde\lambda_{q,n},{\rm cex}}$ and $\tilde\E^{(q)}_{\tilde\lambda_{q,n},{\rm ex}}$, respectively, and let $\{\tilde\varphi_{{\rm har},j_0}\}$ be a basis of the space of the harmonic forms
\[
\H^q(W)=\{\tilde\omega\in \Omega^{(q)}(W)~|~\tilde\Delta\tilde\omega=0\}.
\]

Then we have a complete basis for $\Omega^q(W)$ (we will omit the indices $j,k$ where not necessary in the following)
\beq\label{bas}
\bigcup_{q=0}^m\left(\left\{\tilde\varphi_{{\rm har},l}^{(q)}\right\}_l\cup\bigcup_{n=1}^\infty\left\{\tilde\varphi_{\tilde\lambda_{q,n},{\rm cex},j}^{(q)},\tilde\varphi_{\tilde\lambda_{q,n},{\rm ex},k}^{(q)}\right\}_{j,k}\right),
\eeq
that corresponds to the  decomposition of  the space of the forms as 
\[
\Omega^q(W)=\H^q(W)\oplus\Omega_{\rm cex}^q(W)\oplus\Omega^q_{\rm ex}(W)
=\H^q(W)\oplus \bigoplus_{n=1}\left(\tilde\E^{(q)}_{\tilde\lambda_{q,n},{\rm cex}}
\oplus \tilde\E^{(q)}_{\tilde\lambda_{q,n},{\rm ex}}\right).
\]

The above decomposition induces the direct sum decomposition
\begin{align*}
C^\infty(I_{a,b}, &\Omega^q(W)\times \Omega^{q-1}(W))\\
=&C^\infty(I_{a,b},\H^q(W))\oplus C^\infty(I_{a,b},\Omega_{\rm cex}^q(W)\oplus C^\infty(I_{a,b},\Omega^{q}_{\rm ex}(W))\\
&\oplus C^\infty(I_{a,b},\H^{q-1}(W))\oplus C^\infty(I_{a,b},\Omega_{\rm cex}^{q-1}(W)\oplus C^\infty(I_{a,b}, \Omega^{q-1}_{\rm ex}(W))\\
=&C^\infty(I_{a,b},\H^q(W))\oplus 
\bigoplus_{n_1} \left(C^\infty(I_{a,b},\tilde\E^{(q)}_{\tilde\lambda_{q,n_1},{\rm cex}})\oplus C^\infty(I_{a,b}, \tilde\E^{(q)}_{\tilde\lambda_{q,n_1},{\rm ex}})\right)\\
&\oplus C^\infty(I_{a,b},\H^{q-1}(W))\oplus 
\bigoplus_{n_2} \left(C^\infty(I_{a,b},\tilde\E^{(q-1)}_{\tilde\lambda_{q-1,n_2},{\rm cex}})\oplus C^\infty(I_{a,b}, \tilde\E^{(q-1)}_{\tilde\lambda_{q-1,n_2},{\rm ex}})\right)\\
&=\bigoplus_{  n_1,n_2,w_{n_1},w_{n_2}} V_{n_1,w_{n_1}}\times V_{n_2,w_{n_2}},
\end{align*}
where $n_1,n_2=0, 1, 2, \dots$, $w_0={\rm har}$, $w_{n\geq 1}={\rm  cex, ex}$, and
\[
V_{n,w_n}=C^\infty (I_{a,b},\tilde \E^{(q)}_{\tilde\lambda_{q,n}, w_n}),
\]
with the convention that 
\[
C^\infty (I_{a,b},\tilde \E^{(q)}_{\tilde\lambda_{q,0},{\rm har}})= C^\infty(I_{a,b},\H^{q-1}(W)).
\]

Using these facts in equation (\ref{lapdec}), we have obtain the announced  decomposition of the Hodge-Laplace operator:

\[
\AF^q=\bigoplus_{  n_1,n_2} \AF^q_{n_1,n_2},
\]
where
\[
\AF^q_{n_1,n_2}
=\left(\begin{matrix} \tf_{\tilde\la_{q,n_1},\al_q} & -2\frac{h'}{h^2} \tilde \d\\
 -2\frac{h'}{h^2} \tilde \d^\da& \tf_{\tilde\la_{q-1,n_2},-\al_{q-2}}\end{matrix}\right),
\]
acts on the space
\[
C^\infty (I_{a,b},\tilde \E^{(q)}_{\tilde\la_{q,n_1}}\times \tilde\E^{(q-1)}_{\tilde\la_{q-1,n_2}}), 
\] 
and
\begin{equation}\label{L0}
\tf_{\la,\al} = -\frac{d^2}{dx^2} - \frac{h''(x)}{h(x)}\left(\alpha-\frac{1}{2}\right) + \frac{\la +\left(\alpha^2-\frac{1}{4}\right)h'(x)^2}{h(x)^2}.
\end{equation}

We may go a little bit further, using the basis of forms described in equation (\ref{bas}). In fact, we have the decomposition
\[
V_{n,w_n}=C^\infty (I_{a,b},\tilde \E^{(q)}_{\tilde\lambda_{q,n}, w_n})=\bigoplus_{j_{w_n}} C^\infty (I_{a,b},\langle \tilde\vv_{\tilde\la_{q,n},w_n,j_{w_n}}\rangle),
\] 
and writing
\[
C^\infty (I_{a,b},\langle \tilde\vv_{\tilde\la_{q,n},w_n,j_{w_n}}\rangle)=C^\infty(I_{a,b})\otimes \langle\tilde\vv_{\tilde\la_{q,n},w_n,j_{w_n}}\rangle,
\]
we have the decomposition
\[
\AF^q_{n_1,n_2}=\bigoplus_{w_{n_1},{w_{n_2}}} \AF^q_{n_1,n_2,w_{n_1},w_{n_2}}
=\bigoplus_{w_{n_1},{w_{n_2}}}  \bigoplus_{j_{w_{n_1}},j_{w_{n_2}}} \AF^q_{n_1,n_2,w_{n_1},w_{n_2},j_{w_{n_1}},j_{w_{n_2}}},
\]
where
\[
\AF^q_{n_1,n_2,w_{n_1},w_{n_2},j_{w_{n_1}},j_{w_{n_2}}}
=\left(\begin{matrix} \tf_{\tilde\la_{q,n_1},\al_q}\tilde\vv_{\tilde\la_{q,n_1},w_{n_1},j_{w_{n_1}}} & -2\frac{h'}{h^2} \tilde \d\tilde\vv_{\tilde\la_{q-1,n_2},w_{n_2},j_{w_{n_2}}}\\
 -2\frac{h'}{h^2} \tilde \d^\da \tilde\vv_{\tilde\la_{q,n_1},w_{n_1},j_{w_{n_1}}}& \tf_{\tilde\la_{q-1,n_2},-\al_{q-2}}\tilde\vv_{\tilde\la_{q-1,n_2},w_{n_2},j_{w_{n_2}}}\end{matrix}\right),
\]
acts on the space $C^\infty (I_{a,b})\times C^\infty (I_{a,b})$.

\subsection{Concrete operators}\label{sec4.6} In this section we discuss a concrete realisation of the formal operator $\AF$ introduced in Section \ref{hodge}, and consequently of the Hodge-Laplace operator $\Delta$ on $C_{a,b}(W)$. We will use the decomposition given in Section \ref{declap}. Accordingly, we saw at the end of Section \ref{hodge}  that a given form $\omega$ in $\Omega^q(C_{a,b}(W))$ belongs to $L^2(C_{a,b}(W))$ if and only if the adjoint form $(u_1,u_2)$ in 
$C^\infty(I_{a,b}, \Omega^q(W)\times \Omega^{q-1}(W))$ belongs to $L^2(I_{a,b},\Omega^q(W)\times \Omega^{q-1}(W))$. Using the decomposition on the eigenspaces of the Laplace operator on the section, 
$(u_1,u_2)$ belongs to $L^2(I_{a,b},\Omega^q(W)\times \Omega^{q-1}(W))$ if and only if $u_1$ and $u_2$ belongs to $L^2(I_{a,b},dx)$. We reduced the problem of describing the concrete operator induced by the formal operator $\AF$ on the cone in that of describing the concrete operator on the line segment induced by the formal operator $\lf_{\nu,\al}$ given in equation (\ref{L0}). The last problem has been studied in details in Section \ref{declap}, so we will use now the results of that section.  We obtained the direct sum decomposition
\[
\AF^q=\bigoplus_{  n_1,n_2} \AF^q_{n_1,n_2}=\bigoplus_{  n_1,n_2}\bigoplus_{  w_{n_1},w_{n_2},j_{w_{n_1}},j_{w_{n_2}}} \AF^q_{n_1,n_2,w_{n_1},w_{n_2},j_{w_{n_1}},j_{w_{n_2}}},
\]
where
\[
\AF^q_{n_1,n_2,w_{n_1},w_{n_2},j_{w_{n_1}},j_{w_{n_2}}}
=\left(\begin{matrix} \tf_{\tilde\la_{q,n_1},\al_q}\tilde\vv_{\tilde\la_{q,n_1},w_{n_1},j_{w_{n_1}}} & -2\frac{h'}{h^2} \tilde \d\tilde\vv_{\tilde\la_{q-1,n_2},w_{n_2},j_{w_{n_2}}}\\
 -2\frac{h'}{h^2} \tilde \d^\da \tilde\vv_{\tilde\la_{q,n_1},w_{n_1},j_{w_{n_1}}}& \tf_{\tilde\la_{q-1,n_2},-\al_{q-2}}\tilde\vv_{\tilde\la_{q-1,n_2},w_{n_2},j_{w_{n_2}}}\end{matrix}\right),
\]
with
\begin{equation}\label{L0-1}
\tf_{\la,\al} = -\frac{d^2}{dx^2} - \frac{h''(x)}{h(x)}\left(\alpha-\frac{1}{2}\right) + \frac{\la^2 +\left(\alpha^2-\frac{1}{4}\right)h'(x)^2}{h(x)^2},
\end{equation}
and  $(\tilde \la_{q,n_1},\tilde \la_{q-1,n_2})$ is a pair of eigenvalues of the Laplace operators $\tilde \Delta^{(q)}$ and $\tilde \Delta^{(q-1)}$ on the smooth forms on the section,  $\al_q=q+\frac{1}{2}(1-m)$, and $(\tilde\vv_{\tilde\la_{q,n_1},w_{n_1},j_{w_{n_1}}}, \tilde\vv_{\tilde\la_{q-1,n_2},w_{n_2},j_{w_{n_2}}})$ two eigenfunctions in the corresponding eigenspaces. Thus, to study the operator $\AF^q$ is equivalent to study each of the operators $\AF^q_{n_1,n_2,w_{n_1},w_{n_2},j_{w_{n_1}},j_{w_{n_2}}}$, acting on $C^\infty(I_{a,b})\times C^\infty(I_{a,b})$, and in particular we have the identification
\[
\tf_{\la,\al}=\lf_{\sqrt{ \la+\al^2}, \al},
\]
where the formal operator $\lf_{\nu,\al}$ is the one studied in Section \ref{ss2}, so we may use the results described in that section. In the following, when possible, we will  write $\AF^q_{n_1,n_2}$, omitting the other indices.

The minimal and the maximal operators associated to $\lf_{\nu,\al}$  are given in equation \eqref{minmax}, and therefore those associated to $\AF^q_{n_1,n_2}$ are 
\begin{align*}
D(A^q_{n_1,n_2,{\rm min}})
=& C_0^\infty (\mathring{I}_{a,b},\tilde \E^{(q)}_{\tilde\la_{q,n_1}}\times \tilde\E^{(q-1)}_{\tilde\la_{q-1,n_2}}),\\
D(A^q_{n_1,n_2,{\rm max}})=&\left\{(u_1,u_2)\in A^2(I_{a,b},\tilde \E^{(q)}_{\tilde\la_{q,n_1}}\times \tilde\E^{(q-1)}_{\tilde\la_{q-1,n_2}})\right.\\
&\left.~|~(u_1,u_2), \AF^q_{n_1,n_2} (u_1,u_2)\in 
L^2(I_{a,b},\tilde \E^{(q)}_{\tilde\la_{q,n_1}}\times \tilde\E^{(q-1)}_{\tilde\la_{q-1,n_2}})\right\}.
\end{align*}
and by the considerations above may be identified with the spaces
\begin{align*}
D(A^q_{n_1,n_2,{\rm min}})
&=C_0^\infty (\mathring{I}_{a,b})\times C_0^\infty (\mathring{I}_{a,b}),\\
D(A^q_{n_1,n_2,{\rm max}})=&\left\{(u_1,u_2)\in A^2(I_{a,b})\times A^2(I_{a,b})\right.\\
&\left.~|~(u_1,u_2), \AF^q_{n_1,n_2}  (u_1,u_2)\in 
L^2(I_{a,b})\times L^2(I_{a,b})\right\}.
\end{align*}

Observe that $A^q_{n_1,n_2,{\rm min}}$ and $A^q_{n_1,n_2,{\rm max}}$ are unbounded densely defined formally self-adjoint operators on $L^2(I_{a,b})$.

Moreover, again proceeding as in Section \ref{bv}, we have the following result.
\begin{theo}\label{selfext} The self-adjoint extensions of the formal Hodge-Laplace operator $\Delta^{(q)}$ in the Hilbert space 
$L^2(C_{a,b}(W))$ are completely determined by the self-adjoint extensions $A^q_{n_1,n_2}$ of the minimal operator $A^q_{n_1,n_2,{\rm min}}$ associated to the formal differential operator $\AF^q_{n_1,n_2}$ in the Hilbert space $L^2(I_{a,b},dx)\times L^2(I_{a,b},dx)$, defined as follows: 
{\small
\begin{align*}
A^q_{n_1,n_2} (u_1,u_2)&=\AF^q_{n_1,n_2} (u_1,u_2),\\
D(A^q_{n_1,n_2})&=\left\{ (u_1,u_2)\in D(A^q_{n_1,n_2,{\rm max}})~|~ BC^q_{n_1,n_2}(0)(u_1,u_2)=BC^q_{n_1,n_2}(l)(u_1,u_2)=0\right\},
\end{align*}
}
where, setting  $\nu_{1,q,n}=\sqrt{\tilde\la_{q,n}+\al_q^2}$, and $\nu_{2,q,n}=\sqrt{\tilde\la_{q,n}+\al_{q-1}^2}$, the boundary conditions are
\begin{align*}
BC&^q_{n_1,n_2}(0)(u_1,u_2):\\
& \left\{\begin{array}{c}\left\{\begin{array}{cc}\ga_{q,n_1,+} BV_{\nu_{1,q,n_1},+}(0)(u_1)+\ga_{q,n_1,-} BV_{\nu_{1,q,n_1},-}(0)(u_1)=0,&{\rm ~if~}\nu_{1,q,n_1}<1,\\
{\rm none,}&{~\rm if~}\nu_{1,q,n_1}\geq 1,
\end{array}\right.\\
\left\{\begin{array}{cc}
\ga_{q,n_1,+} BV_{\nu_{2,q-1,n_2},+}(0)(u_2)+\ga_{q,n_2,-} BV_{\nu_{2,q-1,n_2},-}(0)(u_2)=0
,&{\rm ~if~}\nu_{2,q-1,n_2}<1,\\
{\rm none,}&{~\rm if~}\nu_{2,q-1,n_2}\geq 1,
\end{array}\right.\\
\end{array}\right.\\
\end{align*}
\begin{align*}
BC^q_{n_1,n_2}(l)(u_1,u_2):& \hspace{30pt}\left\{\begin{array}{c}\ga_{q,n_1} BV(l)(u_1)+\ga_{q,n_1}' BV'(l)(u_1)=0,\\
\ga_{q-1,n_2} BV(l)(u_2)+\ga_{q-1,n_2}' BV'(l)(u_2)=0;
\end{array}\right.
\end{align*}
for real $\ga_{q,n_j,\pm}, \ga_{q,n_j}, \ga_{q,n_j}'$, with $\ga_{q,n_j,+}^2+\ga_{q,n_j,-}^2=\ga_{q,n_j}^2+{\ga_{q,n_j}'}^2=1$.
\end{theo}

\begin{rem}\label{mixedter} Observe that the boundary conditions are given on the two components separately. This is not overlooking the possibility of having $\nu_{q,n_1}$ and $\nu_{q-1,n_2}$ both less that one. For in such a case there appears indeed a term mixing the two components when we apply the formal operator $\AF^q_{n_1,n_2}$ to the vector $(u_1,u_2)$, however the mixed terms in each component do not involve differential operators, but just multiplication by a function, and therefore do not appear in the inner product defining the boundary values as described in Section \ref{bv}.
\end{rem}

For the reader sake, we recall here the explicit form of the boundary values. The boundary values at $x=l$ are 
\begin{align*}
BV(l)(u)&=u(l), & BV'(l)(u)&=u'(l),
\end{align*}
the boundary values at $x=0$ are
\begin{align*}
BV_{\nu,+}(0)(u)&=\lim_{x\to 0^+} \left(v_{\nu,+}(x)u'(x)-v'_{\nu,+}(x)u(x)\right),\\
BV_{\nu,-}(0)(u)&=\lim_{x\to 0^+} \left(v_{\nu,-}(x)u'(x)-v'_{\nu,-}(x)u(x)\right),
\end{align*}
where the $v_{\nu,\pm}$ are a smooth functions on $(0,l]$ vanishing near $l$ and equal to the functions $\uf_{\nu,\pm}$ near $0$, and the $\uf_{\nu,\pm}$ are the following functions (normalised as in Definition \ref{defi1}): 
\begin{align*}
\uf_{\nu,+}(x)&=x^{\frac{1}{2}+\nu}\vv_+(x),&\uf_{\nu,-}(x)&=x^{\frac{1}{2}-\nu}\vv_-(x),& {\rm ~if~}\nu&\not= 0,\\
\uf_{\nu,+}(x)&=x^{\frac{1}{2}+\nu}\vv_+(x),&\uf_{\nu,-}(x)&=x^{\frac{1}{2}}\vv_+(x)\log x +x^\frac{3}{2}\tilde\vv_-(x),
& {\rm ~if~}\nu&=0,
\end{align*}
where the functions $\vv_\pm$, and $\tilde \vv_-$  are continuous and non vanishing, with continuous non vanishing derivative in $[0,l]$, and $\vv_\pm(0)=1$.

Using the classification of the self-adjoint extensions of the Laplace Hodge operator, we may now proceed to define the suitable one for our geometric setting. First, we determine the constant $\ga$ and $\ga'$. These are given by equation (\ref{abs2}) for the absolute case, and we find: 
\begin{align*}
\ga_{q,n_1}&=\left(\al_q-\frac{1}{2}\right) h^{\al_q-\frac{3}{2}}(l)h'(l), &\ga_{q,n_1}' &=h^{\al_q-\frac{1}{2}}(l),\\
\ga_{q-1,n_2}&=1, &\ga_{q-1,n_2}' &=0,
\end{align*}
and by equation (\ref{rel2}) in the relative case:
\begin{align*}
\ga_{q,n_1}&=1, &\ga_{q,n_j}' &=0,\\
\ga_{q-1,n_2}&=\left(\frac{1}{2}-\al_{q-1}\right)h^{-\frac{1}{2}-\al_{q-1}}(l)h'(l), &\ga_{q-1,n_2}' &=h^{\frac{1}{2}-\al_{q-1}}(l).
\end{align*}

Therefore, we may rewrite the absolute boundary conditions at $x=l$ as
\begin{align*}
BC^q_{n_1,n_2}(l)(u_1,u_2):& \hspace{30pt}\left\{\begin{aligned}\left(\al_q-\frac{1}{2}\right) h'(l)u_1(l)+h(l)u'_1(l)=\left(h^{\al_q-\frac{1}{2}}u_1\right)'(l)=0,\\
u_2(l)=0,
\end{aligned}\right.
\end{align*}
and the relative boundary conditions at $x=l$ as
\begin{align*}
BC^q_{n_1,n_2}(l)(u_1,u_2):& \hspace{30pt}\left\{\begin{aligned}
u_1(l)=0.\\
\left(\frac{1}{2}-\al_{q-1}\right) h'(l)u_2(l)+h(l)u'_2(l)=\left(h^{\frac{1}{2}-\al_{q-1}}u_2\right)'(l)=0,\\
\end{aligned}\right.
\end{align*}

This leads to the definition of the following operators.

\begin{defi} \label{abs.lower} Let $\tilde \la_{q,n}$ an eigenvalue of the Hodge-Laplace operator $\tilde\Delta^{(q)}$ in $W$, with $n=0,1,\dots$, and $\tilde \la_{q,0}=0$. Let $\tilde\vv_{\tilde\la_{q,n},w_n,j_n}$ an eigenform of $\tilde\la_{q,n}$ in a given complete orthonormal basis of eigenforms of $\tilde \Delta^{(q)}$, where $w_0={\rm har}$ denotes an harmonic forms, $w_{n>0}={\rm ex, cex}$ identifies exact and coexact forms. Let  $\al_q=\frac{1}{2}(1+2q-m)$ ($m=\dim W$). We define the  operators 
\begin{align*}
A^q_{\rm abs}&=\bigoplus_{  n_1,n_2} A^q_{n_1,n_2, {\rm  abs}},\\
A^q_{\rm abs, \pm}&=\bigoplus_{  n_1,n_2} A^q_{n_1,n_2, {\rm  abs, \pm}},
\end{align*}
in the Hilbert space $L^2(I_{a,b},\Omega^q(W)\times \Omega^{q-1}(W))$, where the operators $A^q_{n_1,n_2, {\rm  abs}}$ and 
$A^q_{n_1,n_2, {\rm  abs, \pm}}$ in the Hilbert space $L^2(I_{a,b},dx)\times L^2(I_{a,b},dx)$ are defined as follows:

\begin{align*}
A^q_{n_1,n_2,{\rm  abs}} (u_1,u_2)&=\AF^q_{n_1,n_2} (u_1,u_2),\\
D(A^q_{n_1,n_2,{\rm  abs}})&=\left\{ (u_1,u_2)\in D(A^q_{n_1,n_2,{\rm max}})~|~ BC^q_{n_1,n_2, {\rm abs}}(l)(u_1,u_2)=0\right\},
\end{align*}
and
\begin{align*}
A^q_{n_1,n_2,{\rm  abs, \pm}} (u_1,u_2)&=\AF^q_{n_1,n_2} (u_1,u_2),\\
D(A^q_{n_1,n_2,{\rm  abs, \pm}})=&\left\{ (u_1,u_2)\in D(A^q_{n_1,n_2,{\rm max}}) \right.\\
&\left.~|~BC^q_{n_1,n_2, \pm}(0)(u_1,u_2)=BC^q_{n_1,n_2, {\rm abs}}(l)(u_1,u_2)=0\right\},
\end{align*}
where
\begin{align*}
D(A^q_{n_1,n_2,{\rm max}})=&\left\{(u_1,u_2)\in A^2(I_{a,b})\times A^2(I_{a,b})
\right.\\
&\left. ~|~(u_1,u_2), \AF^q_{n_1,n_2}  (u_1,u_2)\in 
L^2(I_{a,b},dx)\times L^2(I_{a,b},dx)\right\},\\
\AF^q_{n_1,n_2}
=&\left(\begin{matrix} \tf_{\tilde\la_{q,n_1},\al_q} & -2\frac{h'}{h^2} \tilde \d\\
 -2\frac{h'}{h^2} \tilde \d^\da & \tf_{\tilde\la_{q-1,n_2},-\al_{q-2}}\end{matrix}\right),
\end{align*}
with
\[
\tf_{\la,\al} = -\frac{d^2}{dx^2} - \frac{h''(x)}{h(x)}\left(\alpha-\frac{1}{2}\right) + \frac{\la +\left(\alpha^2-\frac{1}{4}\right)h'(x)^2}{h(x)^2},
\]
and  the boundary condition are:
\begin{align*}
BC^q_{n_1,n_2,{\rm \pm}}(0)(u_1,u_2):&\hspace{30pt} \left\{\begin{array}{l}\left\{\begin{array}{cc}BV_{\nu_{1,q,n_1},\pm}(0)(u_1)=0,&{\rm ~if~}\nu_{1,q,n_1}<1,\\
{\rm none,}&{~\rm if~}\nu_{1,q,n_1}\geq 1,
\end{array}\right.\\
\left\{\begin{array}{cc}
 BV_{\nu_{2,q-1,n_2},\pm}(0)(u_2)=0,&{\rm ~if~}\nu_{2,q-1,n_2}<1,\\
{\rm none,}&{~\rm if~}\nu_{2,q-1,n_2}\geq 1,
\end{array}\right.\\
\end{array}\right.\\
BC^q_{n_1,n_2,{\rm abs}}(l)(u_1,u_2):&\hspace{30pt} \left\{\begin{array}{c}\left(h^{\al_q-\frac{1}{2}}u_1\right)'(l)=0,\\
u_2(l)=0,
\end{array}\right.
\end{align*}
with $\nu_{1,q,n}=\sqrt{\tilde\la_{q,n}+\al_q^2}$,  $\nu_{2,q,n}=\sqrt{\tilde\la_{q,n}+\al_{q-1}^2}$.

We will denote by $\Delta^{(q)}_{\rm abs}$ and $\Delta^{(q)}_{\rm abs, \pm}$ the corresponding operators under the isometry $\ad_q$.
\end{defi}

Note that $\nu_{j,q,n}\geq 1$ up to a finite number of cases, so we have a finite number of boundary conditions.

\begin{defi} \label{abs.upper} 
The operators $A^q_{\rm rel}$ and $A^q_{\rm rel, \pm}$, are defined as the operators $A^q_{\rm abs}$ and $A^q_{\rm abs, \pm}$ replacing the boundary condition at $x=l$ with the following one:
\begin{align*}
BC^q_{n_1,n_2,{\rm rel}}(l)(u_1,u_2):&\hspace{30pt} \left\{\begin{array}{c}u_1(l)=0,\\
\left(h^{-\al_{q-2}-\frac{1}{2}}u_2\right)'(l)=0.
\end{array}\right.
\end{align*}
We will denote by $\Delta^{(q)}_{\rm rel}$, and $\Delta^{(q)}_{\rm rel, \pm}$ the corresponding operators under the isometry $\ad_q$.

\end{defi}

\begin{defi} \label{HodgeLaplace} We define the graded operators on the Hilbert space  $L^2(I_{a,b},\Omega^\bu(W)\times \Omega^{\bu-1}(W))$ (recall $m=\dim W$).

If $m=2p-1$, $p\geq 1$:
\begin{align*}
A_{{\rm abs}, \mf^c}=A_{{\rm abs}, \mf}=&\bigoplus_{q=0}^{p-2} A^q_{\rm abs}\oplus A^{p-1}_{{\rm abs}, +}\oplus A^p_{{\rm abs}}\oplus A^{p+1}_{{\rm abs},+}\oplus \bigoplus_{q=p+2}^{2p} A^q_{\rm abs};\\
A_{\rm rel,\mf^c}=A_{\rm rel,\mf}=&\bigoplus_{q=0}^{p-2} A^q_{\rm rel}\oplus A^{p-1}_{{\rm rel}, +}\oplus A^p_{{\rm rel}}\oplus A^{p+1}_{{\rm abs},+}\oplus \bigoplus_{q=p+2}^{2p} A^q_{\rm rel}.
\end{align*}

If $m=2p$, $p\geq 1$:
\begin{align*}
A_{\rm abs, \mf^c}=&\bigoplus_{q=0}^{p-2} A^q_{\rm abs}\oplus A^{p-1}_{{\rm abs}, +}\oplus A^p_{{\rm abs},+}\oplus A^{p+1}_{0,0,{\rm abs}, -}\oplus  \bigoplus_{n_1>0,n_2>0} A^{p+1}_{n_1,n_2,{\rm abs},+}\oplus A^{p+2}_{{\rm abs}, +}\\
&\oplus \bigoplus_{q=p+3}^{2p+1} A^q_{\rm abs},\\
A_{\rm abs, \mf}=&\bigoplus_{q=0}^{p-2} A^q_{\rm abs}\oplus A^{p-1}_{{\rm abs}, +}\oplus A^{p}_{0,0,{\rm abs}, -}\oplus  \bigoplus_{n_1>0,n_2>0} A^{p}_{n_1,n_2,{\rm abs}, +}\oplus A^{p+1}_{{\rm abs},+}\oplus A^{p+2}_{{\rm abs}, +}\\
&\oplus \bigoplus_{q=p+3}^{2p+1} A^q_{\rm abs};\\
A_{\rm rel, \mf^c}=&\bigoplus_{q=0}^{p-2} A^q_{\rm rel}\oplus A^{p-1}_{{\rm rel}, +}\oplus A^{p}_{{\rm rel}, +}\oplus A^{p+1}_{0,0,{\rm rel}, -}\oplus  \bigoplus_{n_1>0,n_2>0} A^{p+1}_{n_1,n_2,{\rm rel}, +}\oplus A^{p+2}_{{\rm rel}, +}\\
&\oplus \bigoplus_{q=p+3}^{2p+1} A^q_{\rm rel},\\
A_{\rm rel, \mf}=&\bigoplus_{q=0}^{p-2} A^q_{\rm rel}\oplus A^{p-1}_{{\rm rel}, +}\oplus A^{p}_{0,0,{\rm rel}, -}\oplus  \bigoplus_{n_1>0,n_2>0} A^{p}_{n_1,n_2,{\rm rel}, +}\oplus A^{p+1}_{{\rm rel},+}\oplus A^{p+2}_{{\rm rel}, +}\\
&\oplus \bigoplus_{q=p+3}^{2p+1} A^q_{\rm rel}.
\end{align*}

We will denote by $\Delta^{(q)}_{\rm abs,\mf}$ and $\Delta^{(q)}_{\rm abs, \mf^c}$ the corresponding operators under the isometry $\ad_q$.

\end{defi}

\begin{rem}\label{lapche} Following Cheeger, we also define the operator $\Delta_{\rm abs}$ (and the correspondent relative version)
\begin{align*}
D(\Delta_{\rm abs})&=\left\{ \omega\in \Omega^\bu(C_{a,b}(W))~|~ \omega, \d\omega , \d^\da \omega, \Delta \omega \in L^2(C_{a,b}(W))  \land \BB_{\rm abs}(l)(f)=0\right\},
\end{align*}
\end{rem}


\begin{defi}\label{defiL}  We introduce some notation. Setting $\nu=\sqrt{\la+\al^2}$, the operator
\[
\tf_{\la,\al} = -\frac{d^2}{dx^2} - \frac{h''(x)}{h(x)}\left(\alpha-\frac{1}{2}\right) + \frac{\la}{h(x)^2} +\left(\alpha^2-\frac{1}{4}\right)\frac{h'(x)^2}{h(x)^2},
\]
reads 
\[
\lf_{\nu,\al}  = \tf_{\nu^2-\al^2,\al}=-\frac{d^2}{dx^2} - \frac{h''(x)}{h(x)}\left(\alpha-\frac{1}{2}\right) + \frac{\nu^2 -\al^2}{h(x)^2}+\left(\alpha^2-\frac{1}{4}\right)\frac{h'(x)^2}{h(x)^2}.
\]

We denote by $\L_{\mu, \al,\pm}(x,\lambda)=\uf_\pm(x)$ the two normalised solutions of the differential equation
\[
\lf_{\nu,\al}u=\la u,
\]
i.e. 
\begin{align*}
u'' + \frac{h''}{h}\left(\alpha-\frac{1}{2}\right) u
 -\left(\alpha^2-\frac{1}{4}\right)\frac{{h'}^2}{h^2}u
 +\left(\la+\frac{\alpha^2-\mu^2}{h^2}\right) u&=0,
\end{align*}
as described in Definition \ref{defi1}. We also introduce the notation  $\FF_{\mu, \al,\pm}(x,\la)=\ff_\pm$ for the two normalised  independent solutions of the corresponding equation 
\[
f''+\frac{(1-2\alpha) h'}{h}f'+\left(\la+\frac{\alpha^2-\mu^2}{h^2}\right)f=0,
\]

It is clear that
\beq\label{relfuncts}
\FF_{\mu, \al,\pm}=h^{\al-\frac{1}{2}}\L_{\mu, \al,\pm}.
\eeq

\end{defi}

\subsection{Solutions of the eigenvalues equation}
\label{soleq}

In this section we give an explicit description of the solutions of the eigenvalues equation for the formal operator $\AF$ and consequently for the formal Laplace operator $\Delta$. 
Now we proceed to identify the solutions of the equation $\AF u=\la u$, with $\la\not=0$. 
According to the decomposition introduced in Section \ref{declap}, we decompose a $q$ form 
$\ad(\omega^{(q)})=(u_1,u_2)$ in the basis in equation (\ref{bas}). However, for further use, it is more convenient to reduce all the forms on the section to coexact forms. This may be do as follows.  Recall that the differential defines a bijection (an isometry if normalised by $\frac{1}{\la}$)
\begin{align*}
\tilde \d:&\tilde\E^{(q)}_{\tilde\lambda_{q,n},{\rm cex}}\to \tilde\E^{(q+1)}_{\tilde\lambda_{q,n}, {\rm ex}},\\
\tilde \d:&\tilde\varphi\mapsto \tilde d\tilde\varphi,
\end{align*}
whose inverse is $\tilde \d^\da$. This means that all the eigenvalues of the $q$ exact forms appear as eigenvalues of a $q-1$ coexact form. 

As a consequence, it is more convenient to index the eigenvalues by the dimension of the coexact form they belong to, namely to  denote by $\tilde\la_{q,n}$ one of the eigenvalues of $\tilde\Delta^{(q)}$ such that
\[
\tilde \Delta^{(q)}\tilde\omega_{\tilde\la_{q,n}, {\rm cex}, j}=\tilde\la_{q, n} \tilde\omega_{\tilde\la_{q, n}, {\rm cex},j},
\]
for some coexact $q$-form $\tilde\omega_{\tilde\la_{q, n},{\rm cex}, j}$. This means that there may or may not exist an exact $q$ form with eigenvalue $\tilde\la_{q, n}$, but certainly it does exist an exact  $q+1$-form $\tilde d\tilde \omega_{\tilde\la_{q,n},{\rm cex},j}$ with eigenvalue $\tilde\la_{q, n}$. Using this convention,  $m_{\tilde\lambda_{q,n},{\rm cex}}$ denotes the multiplicity of the eigenspace of coexact forms with eigenvalue $\tilde \lambda_{q,n}$, i.e. $m_{\tilde\lambda_{q,n},{\rm cex}}=\dim 
\tilde\E^{(q)}_{\tilde\lambda_{q,n},{\rm cex}}=\dim \tilde\E^{(q+1)}_{\tilde\lambda_{q,n},{\rm ex}}$.

Thus we may decide to take as basis for the whole eigenspace of a given eigenvalue $\tilde\la_{q,n}$ only coexact forms and differential of coexact forms, i.e. we identify
\[
\tilde \E^{(q)}_{\tilde\lambda_{q, n},{\rm ex}}=\langle 
\tilde\vv^{(q)}_{\tilde\lambda_{q, n},{\rm ex}, k}\rangle
\equiv\langle 
\tilde\d\tilde\vv^{(q-1)}_{\tilde\lambda_{q-1, n},{\rm cex},j}\rangle=\tilde\d 
\tilde 
\E^{(q-1)}_{\tilde\lambda_{q-1, n},{\rm cex}},
\]
(where $\tilde\la_{q, n}=\tilde\la_{q-1, n}$) and hence we can take the basis (we will omit the indices $j_k$ where not necessary in the following)
\beq\label{bas1}
\bigcup_{q=0}^m\left(\left\{\tilde\varphi_{{\rm har}, l}^{(q)}\right\}_l\cup
\bigcup_{\tilde\la_{q, n}=\tilde\la_{q-1, n}, n=1}^\infty\left\{\tilde\varphi_{\tilde\lambda_{q, n},{\rm cex}, j}^{(q)},\tilde \d \tilde\varphi_{\tilde\lambda_{q-1,n},{\rm cex}, k}^{(q-1)}\right\}_{j,k}\right),
\eeq
for $\Omega^q(W)$, that corresponds to the  decomposition of  the space of the forms as 
\begin{align*}
\Omega^q(W)&=\H^q(W)\oplus\Omega_{\rm cex}^q(W)\oplus\left(\Omega^q_{\rm ex}(W)=\tilde d \Omega^{q-1}_{\rm cex}(W)\right)\\
&=\H^q(W)\oplus \bigoplus_{n=1}\left(\tilde\E^{(q)}_{\tilde\lambda_{q,n},{\rm cex}}\oplus
\tilde \d \tilde\E^{(q-1)}_{\tilde\lambda_{q-1,n},{\rm cex}}\right).
\end{align*}

Whence, we arrive at the following decomposition of a $q$ form 
$\ad(\omega^{(q)})=(u_1,u_2)$ in the basis in equation (\ref{bas1}): 
\begin{align*}
u_1(x)=& \sum_{l_1} u^{(q)}_{{\rm har}, l_1}(x) \tilde\varphi^{(q)}_{{\rm har}, l_1} \\
&+\sum_{n_1,j_1}
u^{(q)}_{\tilde\lambda_{q, n_1},{\rm cex}, j_1}(x)\tilde\varphi^{(q)}_{\tilde\lambda_{q, n_1},{\rm cex}, j_1} 
+ \sum_{n_1,k_1}    u^{(q)}_{\tilde\lambda_{q, n_1}, {\rm ex}, k_1}(x) \tilde 
\d\tilde\varphi^{(q-1)}_{\tilde\lambda_{q-1, n_1},{\rm cex}, k_1},\\
u_2(x)=& \sum_{l_2} u^{(q-1)}_{{\rm har}, l_2}(x) \tilde\varphi^{(q-1)}_{{\rm har},l_2} \\
&+\sum_{n_2,j_2}
u^{(q-1)}_{\tilde\lambda_{q-1, n_2}, {\rm cex}, j_2}(x)\tilde\varphi^{(q-1)}_{\tilde\lambda_{q-1,n_2},{\rm cex}, j_2} + \sum_{n_2k_2}    u^{(q-1)}_{\tilde\lambda_{q-1, n_2}, {\rm ex}, k_2}(x) \tilde \d\tilde\varphi^{(q-2)}_{\tilde\lambda_{q-2, n_2},{\rm cex}, k_2},
\end{align*}
where now the $u$ are functions with complex values. As a consequence,  $\AF(\ad(\omega^{(q)}))$ reads
\begin{align*}
\AF (u_1,u_2)(x) =& \left(\sum_{l_1} (\tf_{0, \alpha_q}\ 
u^{(q)}_{{\rm har}, l_1})(x)\tilde\varphi_{{\rm har}, l_1}^{(q)},\sum_{l_2} 
(\tf_{0,-\alpha_{q-2}} \ u^{(q-1)}_{{\rm har},l_2})(x) \tilde\varphi_{{\rm 
har},l_2}^{(q-1)}\right)\\
&+ \left(\sum_{n_1,j_1} (\tf_{\tilde\la_{q,n_1},\alpha_q} \ 
u^{(q)}_{\tilde\lambda_{q,n_1},{\rm 
cex},j_1})(x)\tilde\varphi_{\tilde\lambda_{q,n_1},{\rm cex},j_1}^{(q)}\right.\\
&+ \sum_{n_1,k_1}  
(\tf_{\tilde\la_{q,n_1},\alpha_{q}} \ u^{(q)}_{\tilde\lambda_{q,n_1},{\rm ex},k_1})(x)\tilde\d\tilde\varphi_{\tilde\lambda_{q-1,n_1},{\rm 
cex},k_1}^{(q-1)}\\
&\left.-2\frac{h'(x)}{h(x)^2} 
\sum_{n_2,j_2}u^{(q-1)}_{\tilde\lambda_{q-1,n_2},{\rm cex},j_2}(x)
\tilde \d\tilde\varphi_{\tilde\lambda_{q-1,n}, {\rm cex},j_2}^{(q-1)},0\right)\\
&+\left(0, -2\frac{h'(x)}{h(x)^2} 
\sum_{n_1,k_1}\tilde\lambda_{q-1,n_1}u^{(q)}_{\tilde\lambda_{q,n_1},{\rm ex},k_1}(x)
\tilde\varphi_{\tilde\lambda_{q-1,n_1},{\rm cex},k_1}^{(q-1)}\right.\\
&+\sum_{n_2,j_2}(\tf_{\tilde\la_{q-1,n_2},-\alpha_{q-2}} \ 
u^{(q-1)}_{\tilde\lambda_{q-1,n_2},{\rm cex},j_2})(x)\tilde\varphi_{\tilde\lambda_{q-1,n_2},{\rm cex},j_2}^{(q-1)}\\
&\left.+\sum_{n,k_2} (\tf_{\tilde\la_{q-1,n_2},-\alpha_{q-2}} \ 
u^{(q-1)}_{\tilde\lambda_{q-1,n_2},{\rm ex},k_2})(x) 
\tilde\d\tilde\varphi_{\tilde\lambda_{q-2,n_2},{\rm cex},k_2}^{(q-2)}\right).
\end{align*}

After some simplifications, the equation 
$\AF(u_1,u_2)=\la(u_1,u_2)$, gives the following set of equations
{\small
\begin{align}
\label{eq-har1} \tf_{0,\alpha_q} \ u^{(q)}_{{\rm har},l_1}&=\lambda \ 
u^{(q)}_{{\rm har}, l_1},\\
\label{eq-har2} \tf_{0,-\alpha_{q-2}} \ u^{(q-1)}_{{\rm har}, l_2}& = 
\lambda \  u^{(q-1)}_{{\rm har},l_2},\\
\label{cex1} \tf_{\tilde\la_{q,n_1},\alpha_q} \ u^{(q)}_{\tilde\lambda_{q,n_1},{\rm cex},j_1} 
&= \lambda \ u^{(q)}_{\tilde \lambda_{q,n_1},{\rm cex},j_1},\\
\label{ex1}&\hspace{-130pt}
\begin{aligned}
\sum_{n_1,k_1}  (\tf_{\tilde\la_{q,n_1},\alpha_q} \ u^{(q)}_{\tilde\lambda_{q,n_1},{\rm ex},k_1})(x)
&\tilde\d\tilde\varphi_{\tilde\lambda_{q-1,n_1},{\rm cex},k_1}^{(q-1)}\\
&\hspace{30pt}-2\frac{h'(x)}{h(x)^2} \sum_{n_2,j_2} u^{(q-1)}_{\tilde\lambda_{q-1,n_2},{\rm cex},j_2}(x)\tilde \d\tilde\varphi_{\tilde\lambda_{q-1,n_2},{\rm cex},j_2}^{(q-1)}\\
=&\lambda \sum_{n_1, k_1} \  u^{(q)}_{\tilde\lambda_{q,n_1},{\rm ex},k_1}(x)
\tilde\d\tilde\varphi_{\tilde\lambda_{q-1,n_1},{\rm cex},k_1}^{(q-1)},
\end{aligned}\\
\label{ex2}
&\hspace{-130pt}
\begin{aligned}
\sum_{n_2,j_2}(\tf_{\tilde\la_{q-1,n_2},-\alpha_{q-2}} \ u^{(q-1)}_{\tilde\la_{q-1,n_2},{\rm cex},j_2})(x)&\tilde\vv^{(q-1)}_{\tilde\lambda_{q-1,n_2},{\rm cex},j_2}\\
&-2\frac{h'(x)}{h(x)^2} \sum_{n_1,k_1} \tilde\lambda_{q-1,n_1} \  u^{(q)}_{\tilde\lambda_{q,n_1},{\rm ex}, k_1}(x)\tilde\vv^{(q-1)}_{\tilde\lambda_{q-1,n_1},{\rm cex},k_1}\\
=&\lambda  \ \sum_{n_2,j_2}u^{(q-1)}_{\tilde\lambda_{q-1,n_2},{\rm cex},j_2}(x)
\tilde\vv^{(q-1)}_{\tilde\lambda_{q-1,n_2},{\rm cex},j_2},
\end{aligned}\\
\label{cex2} \tf_{\tilde\la_{q-1,n_2},-\alpha_{q-2}} \ 
u^{(q-1)}_{\tilde\lambda_{q-1,n_2},{\rm ex},k_2} &=\lambda \ u^{(q-1)}_{\tilde\lambda_{q-1,n_2},{\rm ex},k_2}.
\end{align} }

The first three  equations and the last one are each other independent, and are independent from the other two. Moreover, applying the convention introduced before of indexing the eigenvalues with respect to the coexact forms, in the last euqation we identify $\tilde\la_{q-1,n_2}$ with $\tilde\la_{q-2,n_2}$, and therefore the equation becomes
\[
\tf_{\tilde\la_{q-2,n_2},-\alpha_{q-2}} \ 
u^{(q-1)}_{\tilde\lambda_{q-2,n_2},{\rm ex},k_2} =\lambda \ u^{(q-1)}_{\tilde\lambda_{q-2,n_2},{\rm ex},k_2},
\]
and in each instance the operator $\tf_{\la,\al}$ may be identified with the operator $\lf_{\nu=\sqrt{\la^2+\al^2},\al}$, so that the solutions of the differential equation are precisely the functions introduced in Definition \ref{defiL}, and as a consequence, in each dimension $q$, we have the following four types of solutions of the eigenvalue equations (where we omit the index of the eigenvector)\eqref{eq-har1},\eqref{eq-har2}, \eqref{cex1}, and \eqref{cex2}, respectively,
\begin{align*}
\psi^{(q)}_{E,  \pm}(x,\la)&=\left(\L_{|\alpha_q|,\alpha_q,\pm}(x,\la) \ 
\tilde\vv_{\rm har}^{(q)},0\right),\\
\psi^{(q)}_{O,  \pm}(x,\la)&=\left(0, \L_{|\alpha_{q-2}|,-\alpha_{q-2}, \pm}(x,\la)  \tilde\vv_{\rm har}^{(q-1)}\right),\\
\psi^{(q)}_{I, \tilde\lambda_{q,n}, \pm}(x,\la)&=\left(\L_{{\mu_{q,n},\alpha_q},\pm}(x,\lambda) 
\ \tilde\vv^{(q)}_{\tilde\lambda_{q,n}, {\rm cex}},0\right),\\
\psi^{(q)}_{IV, \tilde\la_{q-2,n}, \pm}(x,\la)&=\left(0, \L_{{\mu_{q-2,n}},-\alpha_{q-2},\pm}(x,\lambda) \  \tilde \d\tilde\vv^{(q-2)}_{\tilde\lambda_{q-2,n}, {\rm cex}}\right).
\end{align*}

Readably, the equations \eqref{ex1} and \eqref{ex2} need  $k_1 = j_2$, $n_1 = n_2$; moreover, recalling the identification of the eigenvalues we have that $\tilde \la_{q,n_1}=\tilde \la_{q-1,n_1}$. Thus, we may put these two equations in the following system
\[\scriptstyle
\left(\begin{matrix} 
\tf_{\tilde\la_{q-1,n_1},\al_q} & -2\frac{h'}{h^2} \tilde \d\\
 -2\frac{h'}{h^2} \tilde \d^\da& \tf_{\tilde\la_{q-1,n_1},-\al_{q-2}}
 \end{matrix}\right) 
 \left(
\begin{matrix}
u^{(q)}_{\tilde\lambda_{q,n_1},{\rm ex},k_1} \tilde\d\tilde\varphi_{\tilde\lambda_{q-1,n_1},{\rm cex},k_1}^{(q-1)}\\
u^{(q-1)}_{\tilde\lambda_{q-1,n_1},{\rm cex},k_1} \tilde\varphi_{\tilde\lambda_{q-1,n_1},{\rm cex},k_1}^{(q-1)}
\end{matrix} \right) 
= \la \left(
\begin{matrix}
u^{(q)}_{\tilde\lambda_{q,n_1},{\rm ex},k_1} \tilde\d\tilde\varphi_{\tilde\lambda_{q-1,n_1},{\rm cex},k_1}^{(q-1)}\\
u^{(q-1)}_{\tilde\lambda_{q-1,n_1},{\rm cex},k_1}\tilde \varphi_{\tilde\lambda_{q-1,n_1},{\rm cex},k_1}^{(q-1)}
\end{matrix} \right) .
\]

The first matrix in the last equation is $\AF^q_{n_1,n_1,{\rm ex}, {\rm cex}}$. Since $\d^{q-1} \Delta^{(q-1)} = \Delta^{(q)} \d^{(q-1)}$, we may construct a first solution of the system as follows: let $\psi^{(q-1)}_{I, \tilde\lambda_{q-1,n}, \pm}$ be a solution of type I of degree $q-1$; then 
\begin{align*}
\df  \psi^{(q-1)}_{I, \tilde\lambda_{q-1,n}, \pm}(x,\la) =&\left(h(x)^{-1} \L_{\mu_{q-1,n_1},\al_{q-1},\pm}(x,\la) \tilde\d \tilde\varphi_{\tilde\la_{q-1,n_1},{\rm cex}}\right., \\
&\left.h(x)^{\al_{q-1}+\frac{1}{2}} (h(x)^{\al_{q-1}-\frac{1}{2}}\L_{\mu_{q-1,n_1},\al_{q-1}}(x,\la))'\tilde\varphi_{\tilde\la_{q-1,n_1},{\rm cex}}\right),
\end{align*}
and  
\[
\begin{aligned}
\AF^{q}_{n_1,n_1,{\rm ex}, {\rm cex}} \df \psi^{(q-1)}_{I, \tilde\lambda_{q-1,n}, \pm} (x,\la )&= \df \AF^{q-1}_{n_1,n_1,{\rm cex}, {\rm ex}}\psi^{(q-1)}_{I, \tilde\lambda_{q-1,n}, \pm}(x,\la)\\
&=\la \df\psi^{(q-1)}_{I, \tilde\lambda_{q-1,n}, \pm}(x,\la),
\end{aligned}
\]
as desired. In a similar way we produce a second independent solution using the fact that  $(\d^{\dag})^{(q+1)} \Delta^{(q+1)} = \Delta^{(q)} (\d^{\dag})^{(q+1)}$, which implies that $\df^{\dag}\AF^{q}_{n_1,n_1,{\rm cex}, {\rm ex}} =  \AF^{q+1}_{n_1,n_1,{\rm ex}, {\rm cex}} \df^{\dag}$.

We  followed the notation of Cheeger, introducing six types of eigenforms 
(the type is indicated by the first index). The forms of type E are two linearly 
independent solutions of the equation (\ref{eq-har1}). The forms of type O are two 
linearly independent solutions of the equation (\ref{eq-har2}). The forms of type 
I are two linearly independent solutions of the equation (\ref{cex1}). The 
forms of type IV are two linearly independent solutions of the equation 
(\ref{cex2}).


We have proved the following result.

\begin{lem}\label{l2b} Let $\{\tilde\varphi_{{\rm har}}^{(q)},\tilde\varphi_{{\rm 
cex},n}^{(q)},\tilde d\tilde\varphi_{{\rm cex},n}^{(q-1)}\}$ be an orthonormal  basis 
of $\Omega^q(W)$
consisting of harmonic,  coexact and exact eigenforms of 
$\tilde\Delta^{(q)}$, as described above in details. Let $\tilde\lambda_{q,n}$
denote the  eigenvalue of $\tilde\varphi_{{\rm cex},n}^{(q)}$ and $m_{{\rm cex},q,n}=\dim \tilde\E^{(q)}_{{\rm cex},n}$.
Then, the solutions of the equation $\Delta^{(q)} \omega=\lambda \omega$, 
with $\lambda\not=0$, are  of the following six types:
\begin{align*}
\psi^{(q)}_{E,  \pm}(x,\la)&=\L_{|\alpha_q|,\alpha_q,\pm}(x,\lambda) \  
h(x)^{\alpha_q-\frac{1}{2}} \tilde\vv_{\rm har}^{(q)},\\
\psi^{(q)}_{O,  \pm}(x,\la)&=\L_{|\alpha_{q-2}|,-\alpha_{q-2}, 
\pm}(x,\lambda) \ h(x)^{\alpha_{q-1}-\frac{1}{2}} \ \d x\wedge 
\tilde\vv_{\rm har}^{(q-1)},\\
\psi^{(q)}_{I, \tilde\lambda_{q,n}, \pm}(x,\la)&=\L_{{\mu_{q,n},\alpha_q},\pm}(x,\lambda) \ 
h(x)^{\alpha_q - \frac{1}{2}} \ \tilde\vv^{(q)}_{\tilde\lambda_{q,n}, {\rm cex}},& 
n>0,\\
\psi^{(q)}_{II, \tilde\lambda_{q-1,n}, 
\pm}(x,\la)&=\L_{{\mu_{q-1,n},\alpha_{q-1}},\pm}(x,\lambda) \ 
h(x)^{\alpha_{q-1}-\frac{1}{2}} \ 
\tilde\d\tilde\vv^{(q-1)}_{\tilde\lambda_{q-1,n},{\rm 
cex} }\\
\qquad &+ (\L_{{\mu_{q-1,n},\alpha_{q-1}},\pm}(x,\lambda) \ 
h(x)^{\alpha_{q-1}-\frac{1}{2}})' dx 
\wedge\tilde\vv^{(q-1)}_{\tilde\lambda_{q-1,n}, {\rm cex}},& 
n>0,\\
\psi^{(q)}_{III, \tilde\lambda_{q-1,n}, \pm}(x,\la)&= 
h(x)^{2\alpha_{q-1}+1}(h(x)^{-\alpha_{q-1}-\frac{1}{2}} 
\L_{\mu_{q-1,n},-\alpha_{q-1},\pm}(x,\lambda))' \tilde 
\d\tilde\vv^{(q-1)}_{\tilde\lambda_{q-1,n}, {\rm cex}} 
\\
&\qquad +  \L_{{\mu_{q-1,n},-\alpha_{q-1}},\pm}(x,\lambda) \ 
h(x)^{\alpha_{q-1}-\frac{3}{2}}\d x\wedge 
\tilde\d^{\dag}\tilde 
\d\tilde\vv^{(q-1)}_{\tilde\lambda_{q-1,n}, {\rm cex}}& n>0\\
\psi^{(q)}_{IV, \tilde\lambda_{q-2,n}, 
\pm}(x,\la)&=\L_{{\mu_{q-2,n},-\alpha_{q-2}},\pm}(x,\lambda) \ 
h(x)^{\alpha_{q-2}+\frac{1}{2}}\
\d x\wedge \tilde \d\tilde\vv^{(q-2)}_{\tilde\lambda_{q-2,n}, {\rm cex}}& n>0,
\end{align*}
where $\mu_{q,n}^2 = \tilde\lambda_{q,n}+\alpha_q^2$, the $\L_{{\mu,\alpha}, \pm}(x,\lambda)$ are two linearly independent solutions of the equation (see Definition \ref{defiL})
\[
\lf_{\mu,\alpha} u = \lambda u.
\]
\end{lem}

\begin{rem}\label{r1} The solutions of type $I$, $III$ and $E$ are coexact, and those of types $II$, $IV$ and $O$ are exacts. The operator $\d$ sends forms of types $I$, $III$ and $E$ in forms of types $II$, $IV$ and $O$, 
while $\d^\dag$ sends forms of types $II$, $IV$ and $O$ in forms of types $I$, $III$ and $E$, respectively.  
The Hodge operator sends forms of type $I$ in forms of 
type $IV$, $II$ in $III$, and $E$ in $0$. 

We verify the last statement. For observe that as $\tilde\lambda_{q,n}$ are eigenvalues of co-exact eigenforms 
they satisfy the relation $\tilde\lambda_{q,n} = \tilde\lambda_{m-1-q,n}$. Now 
recalling the definition of $\lf_{\nu,\alpha}$ and that $\alpha_q = -\alpha_{(m+1-q)-2}$
\begin{equation}\label{HodgeDonLK}
\lf_{\mu_{q,n},\alpha_q} = \lf_{\mu_{(m+1-q)-2},\alpha_{(m+1-q)-2}},
\end{equation}
therefore
\[
\L_{{\mu_{q,n},\alpha_q}, \pm}(x,\lambda) = 
\mathcal{L}_{{\mu_{m-q-1,n},-\alpha_{m-q-1},\pm}}(x,\lambda).
\]

Or explicitly we have,
\begin{align*}
\star\psi_{ I,\tilde\lambda_{q,n},\pm}^{(q)}(x,\la) &= 
\star\left(\L_{{\mu_{q,n},\alpha_q},\pm}(x,\lambda) \ 
h(x)^{\alpha_q - \frac{1}{2}} \ \tilde\vv^{(q)}_{\tilde\lambda_{q,n}, {\rm 
cex}}\right)\\
&=(-1)^{q}h(x)^{m-2q} \L_{{\mu_{q,n},\alpha_q},\pm}(x,\lambda) 
\  
h(x)^{\alpha_q-\frac{1}{2}} dx \wedge\tilde\star\tilde\vv^{(q)}_{\tilde\lambda_{q,n}, 
{\rm cex}}\\
&=(-1)^{q}h(x)^{-2\alpha_q+1} \L_{{\mu_{q,n},\alpha_q},\pm}(x,\lambda) 
\ h(x)^{\alpha_q-\frac{1}{2}} dx \wedge
\tilde\star\tilde\vv^{(q)}_{\tilde\lambda_{q,n}, {\rm cex}}\\
&=(-1)^{q} \L_{{\mu_{q,n},\alpha_q},\pm}(x,\lambda) \ 
h(x)^{-\alpha_q+\frac{1}{2}}  dx 
\wedge
\tilde d\tilde\vv^{(m-1-q)}_{\tilde\lambda_{m-1-q,n}, {\rm cex}}\\
&=(-1)^{q} \L_{{\mu_{m-1-q,n},\alpha_{m-q-1}},\pm}(x,\lambda) \ 
h(x)^{-\alpha_q+\frac{1}{2}}  dx 
\wedge
\tilde d\tilde\vv^{(m-1-q)}_{\tilde\lambda_{m-1-q,n}, {\rm cex}}\\
&=(-1)^{q} \L_{{\mu_{(m+1-q)-2,n},-\alpha_q},\pm}(x,\lambda) \ 
h(x)^{-\alpha_q+\frac{1}{2}}  dx 
\wedge
\tilde d \tilde\vv^{((m+1-q)-2)}_{\tilde\lambda_{(m+1-q)-2,n}, {\rm cex}}\\
&=(-1)^q\psi_{ IV,\tilde\lambda_{(m+1-q-)2,n},\pm}^{(m+1-q)}(x,\la).
\end{align*}

\end{rem}

\begin{rem} The solutions of the eigenvalues equation may be obtained also working directly with the forms as written in the equation at the beginning of Section \ref{hodge}. For example, direct substitution of a form $f \tilde\omega_{\rm cex}^{(q)}$ shows that this satisfies the eigenvalues equation, while a form $f \tilde\omega_{\rm ex}^{(q)}$ does not. This gives solutions of type I. Similarly, $f dx\wedge \tilde\omega^{(q-1)}_{\rm cex}$ does not satisfy the eigenvalues equation, while $f dx\wedge\tilde \omega^{(q-1)}_{\rm ex}$ does. This gives solutions of type IV.
\end{rem}

\subsection{Solutions of the harmonic equation}

Proceeding as in the previous Section \ref{soleq}, and with the same notation, we determine the solutions of the harmonic equation.

\begin{lem}\label{harmonics} The solutions of the harmonic equation $\Delta^{(q)}u=0$  are of the following types:
\begin{align*}
\theta^{(q)}_{E,\pm}(x,0)&=\L_{|\alpha_q|,\alpha_q,\pm}(x,0)
h(x)^{\alpha_q-\frac{1}{2}} \tilde\vv_{\rm har}^{(q)},\\
\theta^{(q)}_{O,\pm}(x,0)&=\L_{|\alpha_{q-2}|,-\alpha_{q-2},\pm}(x,0)
h(x)^{\alpha_{q-1}-\frac{1}{2}} \d 
x\wedge 
\tilde\vv_{\rm 
har}^{(q-1)},\\
\theta^{(q)}_{I,\tilde\la_{q,n},\pm}(x,0)&=\L_{\mu_{q,n},\alpha_q,\pm}(x,0) 
h(x)^{\alpha_{q}-\frac{1}{2}}\tilde\vv^{(q)}_{\lambda_{q,n},{\rm cex}},& n>0,\\
\theta^{(q)}_{II, \tilde\la_{q-1,n}, \pm}(x,0)&=\L_{\mu_{q-1,n},\alpha_{q-1},\pm}(x,0) 
h(x)^{\alpha_{q-1}-\frac{1}{2}} 
\tilde\d\tilde\vv^{(q-1)}_{\lambda_{q-1,n},{\rm cex} }\\
&\qquad+ 
\left(\L_{\mu_{q-1,n},\alpha_{q-1},\pm}(x,0)h(x)^{\alpha_{q-1}-\frac{1}{2}}\right)'
dx \wedge\tilde\vv^{(q-1)}_{\lambda_{q-1,n}, {\rm cex}},& 
n>0,\\
\theta^{(q)}_{III,\tilde\la_{q-1,n},\pm}(x,0)&=h(x)^{2\alpha_{q-1}+1}
\left(h(x)^{-\alpha_{q-1}-\frac{1}{2}}\L_{\mu_{q-1,n},-\alpha_{q-1},\pm}(x,0)\right)' \tilde 
\d\tilde\vv^{(q-1)}_{\lambda_{q-1,n}, 
{\rm cex}} \\
&\qquad + 
h(x)^{\alpha_{q-1}- \frac{3}{2}} \L_{\mu_{q-1,n},-\alpha_{q-1},\pm}(x,0)\d 
x\wedge 
\tilde\d^{\dag}\tilde 
\d\tilde\vv^{(q-1)}_{\lambda_{q-1,n}, {\rm cex}}& n>0\\
\theta^{(q)}_{IV, \tilde\la_{q-2,n}, \pm}(x,0)&=\L_{\mu_{q-2,n},-\alpha_{q-2},\pm}(x,0)
 h(x)^{\alpha_{q-1}-\frac{1}{2}}\d x\wedge \tilde 
\d\tilde\vv^{(q-2)}_{\lambda_{q-2,n}, {\rm cex}}& n>0.
\end{align*}

\end{lem}
%
%
%

\begin{rem}\label{II=III} Observe that the forms of type II are of type III, and viceversa. For  let $\t\vv$ and $\t\be$ be  a co exact and an exact form on the section of degrees $q-1$ and $q$, respectively. Set $\vv=f\t\vv$, and $\be=g dx\wedge \t\be$. Then, 
\[
d\vv=f \tilde  d\tilde\vv+f' dx \wedge\tilde \vv,
\]
and
\[
d^\da\be=-h^{-m+2q}(h^{m-2q} g)' \t\be-\frac{1}{h^2} g dx \wedge \t d^\da \t \be.
\]

The equation $d^\da\be= d\vv$  requires that 
\[
\left\{
\begin{array}{l} h^{m-2q}(h^{m-2q} g)'\t \be=-f \t d\t\vv,\\
\frac{1}{h^2} g \t d^\da \t \be=-f' \t\vv.
\end{array}
\right.
\]

Solving the second in $\t\vv$ and substituting in the first we have
\[
h^{-m+2q}(h^{m-2q} g)'  f' \t\be=\frac{fg}{h^2} \t\la \t\be,
\]
where $\t\Delta \t\be=\t\la\be$, that gives
\[
h^{-m+2q}(h^{m-2q} g)'  f' =\frac{\t\la}{h^2} fg,
\]

Let $g=h^2 f'$, then we have 
\[
h^{-m+2q}(h^{m-2q+2} f')'   =\frac{\t\la}{h^2} h^2 f,
\]
i.e.
\[
-f''-(m-2q+2) \frac{h'}{h} f'+\frac{\t\la}{h^2} f=0.
\]

For $f$ to be a solution of this equation, it is necessary that $\t\vv$ has the same eigenvalue $\t\la$ of $\t\be$, and then the solutions is 
\begin{align*}
\t\be&=-\frac{1}{\t\la}\t d\t\vv,&\t\vv&=-\t d^\da \t \be\\
f&:&f''+(m-2q+2)\frac{h'}{h}f'-\frac{\t\la}{h^2}f&=0,\\
g&=h^2 f'.
\end{align*}

It is now easy to verify that  $\Delta g=0$.  
\end{rem}

\subsection{Square integrable solutions}


We discuss square integrability of the solutions given in Lemmas \ref{l2} and \ref{harmonics}.

\begin{lem}\label{l3} 
Let $\psi^{(q)}$ be one of the  forms given in Lemma \ref{l2b}. Then, $\psi^{(q)}$, $d\psi^{(q)}$ and $d^\dagger \psi^{(q)}$  are square integrable  according to the following tables. If  $m=\dim W=2p-1$, $p\geq 1$:
\begin{center}
\begin{tabular}{|c|c|c|c|c|}
\hline
forms type& signal & $\psi^{(q)}\in L^2$ &$d\psi^{(q)}\in L^2$&$d^\dagger \psi^{(q)} \in L^2$\\
\hline
E &$+$&$\forall q$& $\forall q$& $\forall q$\\
E &$-$& $q=p-1$& $q=p$& $\forall q$\\
O &$+$&$\forall q$& $\forall q$& $\forall q$\\
O &$-$&$q=p+1$& $\forall q$& $q =p$\\
I &$+$&$\forall q$    & $\forall q$   & $\forall q$   \\
I &$-$ &  $q=p-1, \tilde\la_{p-1,k} <1$  & $\not \exists q$   & $\forall q$\\
II &$+$&$\forall q$ &$\forall q$ &$\forall q$\\
II &$-$& $\not\exists q$& $\forall q$ & $q=p,\tilde\la_{p-1,k} <1$\\
III &$+$&$\forall q$ &$\forall q$ &$\forall q$\\
III &$-$&$\not\exists q$&$q=p, \tilde\la_{p-1,k} <1$&$\forall q$\\
IV &$+$&$\forall q$&$\forall q$&$\forall q$\\
IV &$-$&$q=p+1, \tilde\la_{p-1,k}<1$& $\forall q$ & $\not\exists q$ \\
\hline
\end{tabular}
\end{center}

If $m=\dim W=2p$, $p\geq 1$:

\begin{center}
\begin{tabular}{|c|c|c|c|c|}
\hline
forms type& signal & $\psi^{(q)}\in L^2$ &$d\psi^{(q)}\in L^2$&$d^\dagger \psi^{(q)}\in L^2$\\
\hline
E &$+$&$\forall q$&$\forall q$&$\forall q$\\
E &$-$& $q=p-1$ or $q=p$& $q=p$ or $q=p+1$& $\forall q$\\
O &$+$&$\forall q$&$\forall q$&$\forall q$\\
O &$-$&$q= p+1$ or $q=p+2$&$\forall q$& $q=p$ or $q=p+1$\\
I &$+$&$\forall q$&$\forall q$   &$\forall q$  \\
I &$-$&$\begin{array}{c}q=p-1,  \tilde\la_{p-1,k} < \frac{3}{4} \\ q=p, \tilde\la_{p,k} < \frac{3}{4}\end{array}$&$\not\exists q$ & $\forall q$\\
II &$+$&$\forall q$&$\forall q$&$\forall q$\\
II &$-$&$\not\exists q$&$\forall q$& $\begin{array}{c}q=p-1,  \tilde\la_{p-1,k} < \frac{3}{4} \\ q=p, \tilde\la_{p,k} < \frac{3}{4}\end{array}$\\
III& $+$&$\forall q$&$\forall q$&$\forall q$\\
III &$-$&$\not\exists q$&$\begin{array}{c}q=p-1,  \tilde\la_{p-1,k} < \frac{3}{4} \\ q=p, \tilde\la_{p,k} < \frac{3}{4}\end{array}$&$\forall q$\\
IV &$+$&$\forall q$&$\forall q$&$\forall q$\\
IV &$-$& $\begin{array}{c} q=p+1,  \tilde\la_{p-1,k} < \frac{3}{4} \\  q=p+2,  \tilde\la_{p,k} < \frac{3}{4} \end{array} $   & $\forall q$& $\not\exists q$\\
\hline
\end{tabular}
\end{center}

\end{lem}

\begin{proof} Consider the forms of type $I$, these forms are
\[
\psi^{(q)}_{I, \tilde\la_{q,n}, \pm}(x,\la)=\L_{{\mu_{q,n},\alpha_q},\pm}(x,\lambda) 
h(x)^{\alpha_q-\frac{1}{2}} 
\tilde\vv^{(q)}_{\tilde\lambda_{q,n}, {\rm cex}},\; n>0.
\] 

From Definition \ref{defi1} the functions $\L_{{\mu_{q,n},\alpha_q},\pm}(x,\lambda)$ have the following behaviour near $x=0$,
\[
\L_{{\mu_{q,n},\alpha_q},\pm}(x,\lambda) \sim x^{\frac{1}{2}\pm\mu_{q,n}}.
\]

Then,
\begin{align*}
\frac{||\psi^{(q)}_{I, \tilde\la_{q,n}, \pm}(x,\la)||^2}{ ||\tilde\vv^{(q)}_{\lambda_{q,n}, {\rm cex}}||^2} &=\int_{0}^l \L^2_{{\mu_{q,n}},\alpha_q,\pm}(x,\lambda) dx<\infty \Leftrightarrow \pm\mu_{q,n}+1>0.
\end{align*}

Therefore all solutions $+$ are in $L^2$ as $\mu_{q,n}\geq 0$. For the solutions $-$ we have
\begin{align*}
\mu_{q,n} = \sqrt{\tilde \lambda^2_{q,n} + \alpha_q^2}<1 &\Rightarrow |\alpha_q| 
< 1\; \text{and}\; \tilde\la_{q,n}<1 \Rightarrow -2<2q-m+1<2 \; \text{and} \;\tilde\la_{q,n}<1 \\
&\Rightarrow \frac{m-3}{2}<q<\frac{m+1}{2} \; \text{and} \;\tilde\la_{q,n}<1.
\end{align*} 

Now we have two cases. If $m = 2p-1$, then $p-2<q<p$, i.e, $q=p-1$ and $\tilde \la_{q,n}<1$. 
For the even case, consider $m=2p$ then $p-\frac{3}{2} < q< p+\frac{1}{2} \Rightarrow q = p-1\; {\rm or}\; q=p$. Therefore, since $\al_{p-1} = -\frac{1}{2}$ and $\al_{p} = \frac{1}{2}$ we have that $\tilde\la_{q,n}<1-\al_{q}^2 = \frac{3}{4}$, for $q= p-1$ or $q=p$.

%

Next we deal with the forms of type $II$,
\begin{align*}
\psi^{(q)}_{II, \tilde\lambda_{q-1,n}, 
\pm}(x,\la)&=\L_{{\mu_{q-1,n},\alpha_{q-1}},\pm}(x,\lambda) \ 
h(x)^{\alpha_{q-1}-\frac{1}{2}} \ 
\tilde\d\tilde\vv^{(q-1)}_{\tilde\lambda_{q-1,n},{\rm 
cex} }\\
\qquad &+ (\L_{{\mu_{q-1,n},\alpha_{q-1}},\pm}(x,\lambda) \ 
h(x)^{\alpha_{q-1}-\frac{1}{2}})' dx 
\wedge\tilde\vv^{(q-1)}_{\tilde\lambda_{q-1,n}, {\rm cex}},& 
n>0.
\end{align*}

Near $x=0$,
\begin{align*}
||\psi^{(q)}_{II, \tilde\la_{q,n}, \pm}(x,\la)||^2  \sim \int_0^{\epsilon} x^{\pm 2\mu_{q,n}-1}dx<\infty \Leftrightarrow \pm 
\mu_{q,n}>0,
\end{align*}
and then, $\psi^{(q)}_{II, \tilde\la_{q,n}, +}(x,\la) \in L^2$ and 
$\psi^{(q)}_{II, \tilde\la_{q,n}, -}(x,\la) \not \in L^2(0,l), \forall q$.

For the other types of forms we just need to use the Remark \ref{r1} to obtain the result. 
%
%
%
%
\end{proof}

\begin{lem}\label{l6-1}
Let $\theta_\pm$ be one of the  forms given in Lemma \ref{harmonics}. Then, $\theta_{\pm}$, $d\theta_{\pm}$ and $d^\dagger \theta_\pm$  are square integrable  according to the following tables. If  $m=\dim W=2p-1$, $p\geq 1$:
\begin{center}
\begin{tabular}{|c|c|c|c|c|c|}
\hline
forms type & signal & $\theta^{(q)}\in L^2$ &$d\theta^{(q)}\in L^2$&$d^\dagger \theta^{(q)}\in L^2$\\
\hline
E &$+$&$\forall q$& $\forall q$& $\forall q$\\
E &$-$& $q=p-1$& $q=p$& $\forall q$\\
O &$+$&$\forall q$& $\forall q$& $\forall q$\\
O &$-$&$q=p+1$& $\forall q$& $q =p$\\
I &$+$&$\forall q$    & $\forall q$   & $\forall q$   \\
I &$-$ &  $q=p-1, \tilde\la_{p-1,k} <1$  & $\not \exists q$   & $\forall q$\\
II &$+$&$\forall q$ &$\forall q$ &$\forall q$\\
II &$-$& $\not\exists q$& $\forall q$ & $q=p, \tilde\la_{p-1,k} <1$\\
III &$+$&$\forall q$ &$\forall q$ &$\forall q$\\
III &$-$&$\not\exists q$&$q=p,\tilde\la_{p-1,k} <1$&$\forall q$\\
IV &$+$&$\forall q$&$\forall q$&$\forall q$\\
IV &$-$&$q=p+1, \tilde\la_{p-1,k}<1$& $\forall q$ & $\not\exists q$ \\
\hline
\end{tabular}
\end{center}

If $m=\dim W=2p$, $p\geq 1$:

\begin{center}
\begin{tabular}{|c|c|c|c|c|c|}
\hline
forms type & signal  & $\theta^{(q)}\in L^2$ &$d\theta^{(q)}\in L^2$&$d^\dagger \theta^{(q)}\in L^2$\\
\hline
E &$+$&$\forall q$&$\forall q$&$\forall q$\\
E &$-$& $q=p-1$ or $q=p$& $q=p$ or $q=p+1$& $\forall q$\\
O &$+$&$\forall q$&$\forall q$&$\forall q$\\
O &$-$&$q= p+1$ or $q=p+2$&$\forall q$& $q=p$ or $q=p+1$\\
I &$+$&$\forall q$&$\forall q$   &$\forall q$  \\
I &$-$&$\begin{array}{c}q=p-1,  \tilde\la_{p-1,k} < \frac{3}{4} \\ q=p, \tilde\la_{p,k} < \frac{3}{4}\end{array}$&$\not\exists q$ & $\forall q$\\
II &$+$&$\forall q$&$\forall q$&$\forall q$\\
II &$-$&$\not\exists q$&$\forall q$& $\begin{array}{c}q=p-1,  \tilde\la_{p-1,k} < \frac{3}{4} \\ q=p, \tilde\la_{p,k} < \frac{3}{4}\end{array}$\\
III& $+$&$\forall q$&$\forall q$&$\forall q$\\
III &$-$&$\not\exists q$&$\begin{array}{c}q=p-1,  \tilde\la_{p-1,k} < \frac{3}{4} \\ q=p, \tilde\la_{p,k} < \frac{3}{4}\end{array}$&$\forall q$\\
IV &$+$&$\forall q$&$\forall q$&$\forall q$\\
IV &$-$& $\begin{array}{c} q=p+1,  \tilde\la_{p-1,k} < \frac{3}{4} \\  q=p+2,  \tilde\la_{p,k} < \frac{3}{4} \end{array} $   & $\forall q$& $\not\exists q$\\
\hline
\end{tabular}
\end{center}

\end{lem}

\subsection{Spectral properties}
\label{secspectral}

We give the Green formula for the cone. Next we state the more relevant spectral properties of the Hodge-Laplace operator.

\begin{lem}\label{green cone1} Let $\omega$ be a square integrable $q$ form on the cone, then
\begin{align*}
\langle \Delta^{(q)}_C\omega, \omega\rangle=&-\left[h^{1-2\al_q}(x) f_1'(x)f_1(x)\right]_0^l \|\tilde\omega_1\|^2
-\left[\left(h^{1-2\al_{q-1}}(x)f_2\right)'(x)f_2(x)\right]_0^l \|\tilde\omega_2\|^2\\
&+\int_0^l  h^{1-2\al_q}(f_1'(x))^2  dx  \|\tilde\omega_1\|^2
+\int_0^l h^{2\al_{q-1}-1}(x)\left(\left(h^{1-2\al_{q-1}}f_2(x)\right)'\right)^2dx \|\tilde\omega_2\|^2\\
&+\tilde\lambda_{q,n_1}\int_0^l h^{-2\al_q-1}(x)f_1^2(x) dx \|\tilde\omega_1\|^2
+\tilde\lambda_{q-1,n_2}\int_0^l h^{1-2\al_{q-1}}\frac{f_2^2(x) }{h^2(x)} dx \|\tilde\omega_2\|^2\\
&-4\int_0^l h'(x) h^{-2\al_q}(x) f_1(x) f_2(x)dx
\langle\tilde d\tilde \omega_2, \tilde\omega_1\rangle.
\end{align*}
\end{lem}
\begin{proof}

Let $\te$ a square integrable form on the cone. Then, by definition
\begin{align*}
\int_C\Delta^{(q)}_{C}\theta\wedge \star_C \theta =&\int_0^l h^{1-2\al_q}(x)\left(\FF^{(q)}_{1} f_1\right)(x) f_1(x) dx \int_W\omega_1\wedge\tilde\star\omega_1\\
&+\int_0^l h^{1-2\al_{q-1}}(x)\left(\FF^{(q)}_{2} f_2\right)(x) f_2(x) dx \int_W \omega_2\wedge\tilde\star\omega_2\\
&-2\int_0^l h'(x) h^{1-\al_q-\al_{q-1}}(x) f_1(x) f_2(x)dx
\left(\int_W \tilde d\omega_2\wedge\tilde\star\omega_1+\int_W \tilde d^\da\omega_1\wedge\tilde\star\omega_2\right).
\end{align*}
where
\begin{align*}
\FF^{(q)}_1&=-\frac{d^2}{dx^2}-(1-2\al_q)\frac{h'(x)}{h(x)}\frac{d}{dx}
+\frac{\tilde\lambda_{q,n_1}}{h^2(x)}\\
&=-h^{2\al_q-1}\frac{d}{dx}(h^{1-2\al_q}\frac{d}{dx})
+\frac{\tilde\lambda_{q,n_1}}{h^2(x)},
\\
\FF^{(q)}_2&=-\frac{d^2}{dx^2}
+(2\al_{q-1}-1)\left(\frac{h'(x)}{h(x)}\frac{d}{dx}+\frac{h''(x)}{h(x)}-\frac{(h'(x))^2}{h^2(x)}\right)
+\frac{\tilde\lambda_{q-1,n_2}}{h^2(x)}\\
&=-\frac{d}{dx}\left(h^{2\al_{q-1}-1}\frac{d}{dx}h^{1-2\al_{q-1}}\right)+\frac{\tilde\lambda_{q-1,n_2}}{h^2(x)}.
\end{align*}

We may apply integration by parts in the integrals on the line obtaining

\begin{align*}
\int_0^l h^{1-2\al_q}(x)\left(\FF^{(q)}_{1} f_1\right)(x) f_1(x) dx 
=&-\left[h^{1-2\al_q}(x)f'_1(x)f_1(x)\right]_\epsilon^l
+\int_0^l  h^{1-2\al_q}(x)(f_1'(x))^2  dx \\
&+\tilde\lambda_{q,n_1}\int_0^l h^{-2\al_q-1}(x)f_1^2(x) dx,
\end{align*}
and
\begin{align*}
\int_0^l h^{1-2\al_{q-1}}(x)\left(\FF^{(q)}_{2} f_2\right)(x) f_2(x) dx 
=&-\left[\left(h^{1-2\al_{q-1}}f_2\right)'(x)f_2(x)\right]_\epsilon^l\\
&+\int_0^l h^{2\al_{q-1}-1}\left(\left(h^{1-2\al_{q-1}}f_2\right)'(x)\right)^2dx \\
&+\tilde\lambda_{q-1,n_2}\int_0^l h^{1-2\al_{q-1}}(x)\frac{f_2^2(x) }{h^2(x)} dx. 
\end{align*}

Whence

\begin{align*}
\langle \Delta^{(q)}_C\omega,\omega\rangle=&-\left[h^{1-2\al_q}(x) f_1'(x)f_1(x)\right]_0^l \|\tilde\omega_1\|^2]
-\left[\left(h^{1-2\al_{q-1}}(x)f_2\right)'(x)f_2(x)\right]_0^l \|\tilde\omega_2\|^2\\
&+\int_0^l  h^{1-2\al_q}(f_1'(x))^2  dx  \|\tilde\omega_1\|^2
+\int_0^l h^{2\al_{q-1}-1}(x)\left(\left(h^{1-2\al_{q-1}}f_2(x)\right)'\right)^2dx \|\tilde\omega_2\|^2\\
&+\tilde\lambda_{q,n_1}\int_0^l h^{-2\al_q-1}(x)f_1^2(x) dx \|\tilde\omega_1\|^2
+\tilde\lambda_{q-1,n_2}\int_0^l h^{1-2\al_{q-1}}\frac{f_2^2(x) }{h^2(x)} dx \|\tilde\omega_2\|^2\\
&-2\int_0^l h'(x) h^{-2\al_q}(x) f_1(x) f_2(x)dx
\left((\tilde d\tilde \omega_2, \tilde\omega_1)+( \tilde d^\da\tilde \omega_1,\tilde\omega_2)\right).
\end{align*}

Note applying this formula for the solutions of the eigenvalues equation, since all the forms on the section are co exact,  last term in the above equation vanishes. \end{proof}

\begin{prop}\label{green cone2} Let $\omega$ be a square integrable $q$ form on the cone, then
\begin{align*}
\langle \Delta^{(q)}_C\omega, \omega\rangle=&-\left[h^{1-2\al_q}(x) f_1'(x)f_1(x)\right]_0^l \|\tilde\omega_1\|^2
-\left[\left(h^{1-2\al_{q-1}}(x)f_2\right)'(x)f_2(x)\right]_0^l \|\tilde\omega_2\|^2\\
&+2\left[h^{1-2\al_q}(x)f_1(x)f_2(x)\right]_0^l \langle\tilde d\tilde \omega_2, \tilde\omega_1\rangle\\
&+\|d\omega\|^2+\|d^\da\omega\|^2.
\end{align*}
\end{prop}
\begin{proof} We compute 
\begin{align*}
\langle d\omega,  d\omega\rangle&=\int_0^l h^{1-2\al_q} (f_1'(x))^2 dx \|\tilde \omega_1\|^2 + \int_0^l h^{1-2\al_q}(x)\frac{f_1^2(x)}{h^{2}(x)} dx \|\tilde d \tilde\omega_1\|^2\\
&+ \int_0^l h^{1-2\al_{q-1}}(x) \frac{f_2^2(x)}{h^2(x)} dx \|\tilde d \tilde \omega_2\|^2 \\
&-2\int_0^l h^{1-2\al_q}(x) f_1'(x) f_2(x) dx \langle \tilde \omega_1, \tilde d\tilde\omega_2\rangle,
\end{align*}
and
\begin{align*}
\langle d^\da \omega, d^\da\omega\rangle &=\int_0^l h^{2\al_{q-1}-1}(x)\left((h^{1-2\al_{q-1}}(x) f_2(x))'\right)^2 dx \|\tilde\omega_2\|^2 \\
&+ \int_0^l h^{1-2\al_q}(x) \frac{f_1(x)}{h(x)^2}^2 dx \|\tilde d^{\dag}\tilde \omega_1\|^2
+ \int_0^l h^{1-2\al_{q-1}}(x) \frac{f_2^2(x)}{h^2(x)} dx \|\tilde d^\dag \tilde \omega_2\|^2\\
&-2\int_0^l (h^{1-2\al_{q-1}}(x)  f_2(x))' \frac{f_1(x)}{h^2(x)} dx \langle \tilde d^{\dag}\tilde \omega_1, \tilde\omega_2\rangle,
\end{align*}

Observe that if $\omega$ is co exact $\|\t d\t\omega\|^2=\t\la\|\t\omega\|^2$,  if it is  exact $\|\t d^\da\t\omega\|^2=\t\la\|\t\omega\|^2$, and that
\[
h^{-2} f_1(h^{1-2\al_{q-1}}f_2)'=h^{-2} f_1(h^2 h^{1-2\al_q}f_2)'=
2h'h^{-2\al_q} f_1 f_2+(h^{1-2\al_q}f_1f_2)',
\]
we have the thesis. 
\end{proof}

\begin{theo}\label{teo1} The operators $\Delta^{(q)}_{\rm abs,\mf}$, $\Delta^{(q)}_{\rm abs, \mf^c}$, $\Delta^{(q)}_{\rm rel,\mf}$ and $\Delta^{(q)}_{\rm rel, \mf^c}$ are self-adjoint, non negative and have compact resolvent. The spectrum is discrete. The eigenfunctions determine a complete orthonormal basis of $L^2(I_{a,b}; \Omega^{(q)}(W))$.
\end{theo}
\begin{proof} Non negativity follows by Proposition \ref{green cone2}. By definition, the operator  $\Delta^{(q)}_{\rm bc,\pf}$ is the direct sum of the operators
\[
\Delta^{(q)}_{\rm bc,\pf}=\bigoplus_{  n_1,n_2} \Delta^{(q)}_{{\rm bc},\pf, n_1,n_2}=\bigoplus_{  n_1,n_2}\bigoplus_{  w_{n_1},w_{n_2},j_{w_{n_1}},j_{w_{n_2}}} 
\Delta^{(q)}_{{\rm bc},\pf, n_1,n_2,w_{n_1},w_{n_2},j_{w_{n_1}},j_{w_{n_2}}},
\]
corresponding to the decomposition introduced in Section \ref{sec4.6}. So 
\[
\Delta^{(q)}_{{\rm bc},\pf, n_1,n_2,w_{n_1},w_{n_2},j_{w_{n_1}},j_{w_{n_2}}}=\ad A^q_{ {\rm bc},\pf,n_1,n_2,w_{n_1},w_{n_2},j_{w_{n_1}},j_{w_{n_2}}} \ad^{-1},
\]
where (formally)
\[
A^{q}_{{\rm bc},\pf, n_1,n_2,w_{n_1},w_{n_2},j_{w_{n_1}},j_{w_{n_2}}}
=\left(\begin{matrix} \tf_{\tilde\la_{q,n_1},\al_q}\tilde\vv_{\tilde\la_{q,n_1},w_{n_1},j_{w_{n_1}}} & -2\frac{h'}{h^2} \tilde \d\tilde\vv_{\tilde\la_{q-1,n_2},w_{n_2},j_{w_{n_2}}}\\
 -2\frac{h'}{h^2} \tilde \d^\da \tilde\vv_{\tilde\la_{q,n_1},w_{n_1},j_{w_{n_1}}}& \tf_{\tilde\la_{q-1,n_2},-\al_{q-2}}\tilde\vv_{\tilde\la_{q-1,n_2},w_{n_2},j_{w_{n_2}}}\end{matrix}\right),
\]
with
\[
\tf_{\la,\al} = -\frac{d^2}{dx^2} - \frac{h''(x)}{h(x)}\left(\alpha-\frac{1}{2}\right) + \frac{\la^2 +\left(\alpha^2-\frac{1}{4}\right)h'(x)^2}{h(x)^2},
\]
and  $(\tilde \la_{q,n_1},\tilde \la_{q-1,n_2})$ is a pair of eigenvalues of the Laplace operators $\tilde \Delta^{(q)}$ and $\tilde \Delta^{(q-1)}$ on the smooth forms on the section,  $\al_q=q+\frac{1}{2}(1-m)$, and $(\tilde\vv_{\tilde\la_{q,n_1},w_{n_1},j_{w_{n_1}}}, \tilde\vv_{\tilde\la_{q-1,n_2},w_{n_2},j_{w_{n_2}}})$ two eigenfunctions in the corresponding eigenspaces. Thus, to study the operator $A^q$ is equivalent to study each of the operators $A^q_{n_1,n_2,w_{n_1},w_{n_2},j_{w_{n_1}},j_{w_{n_2}}}$, acting on $C^\infty(I_{a,b})\times C^\infty(I_{a,b})$, and in particular we have the identification
\[
\tf_{\la,\al}=\lf_{\sqrt{\la+\al^2}, \al},
\]
where the formal operator $\lf_{\nu,\al}$ is the one studied in Section \ref{ss1.1}.

It is fundamental to observe that in the operator $\Delta^{(q)}_{{\rm bc},\pf, n_1,n_2,w_{n_1},w_{n_2},j_{w_{n_1}},j_{w_{n_2}}}$ not only the indices $n$ are fixed, but also the particular forms $\tilde\vv$ as appearing in the matrix above. Therefore, when we consider the solutions of the associated differential equations, we always fall in one particular type of solutions, among those described in Lemmas \ref{l2}, \ref{l2b}. 

Let $\psi$ be any square integrable form on the cone. Then, using the decomposition on the section in terms of coexact forms, we have the following possibilities: 
\begin{align*}
a.&&  \psi_a&=f_a \tilde\sigma_a,\\
b.&&\psi_b&=f_b \tilde d\tilde\sigma_b,\\ 
c.&&\psi_c&=f_c dx\wedge  \tilde\sigma_c,\\ 
d.&&\psi_d&=f_d dx\wedge\tilde d\tilde\sigma_d, 
\end{align*}
where $\sigma_j$ are coexact forms on the section. Consider these cases independently. Take a square integrable  form of type a, $\psi_a$, then the equation
\[
(\la-\Delta^{(q)}_{{\rm bc},\pf, n_1,n_2,w_{n_1},w_{n_2},j_{w_{n_1}},j_{w_{n_2}}})\omega=\psi_a,
\]
reduces to the equation 
\[
(\la- \tf_{\tilde\la_{q,n_1},\al_q}\tilde\vv_{\tilde\la_{q,n_1},w_{n_1},j_{w_{n_1}}})u=u_a,
\]
where $u$ and $u_a$ are the first components after application of the isomorphism $\ad$. Since $\tf_{\la,\al}=\lf_{\sqrt{\la+\al^2}, \al}$, the last equation has a solution and this solution is in the domain of the corresponding operator $\la-L_{\nu,\al})$ is given by the integral operator $(\la-\L_{\nu,\al})^{-1}$, i.e. the resolvent of $L_{\nu,\al}$, described in Proposition \ref{p3.33}, i.e. 
\[
u_1=(\la-\L_{\nu,\al})^{-1}v_{a,1}.
\] 

By that proposition and its corollary the result follows for forms of type a.

Take a square integrable  form of type d, $\psi_d$, then the equation
\[
(\la-\Delta^{(q)}_{{\rm bc},\pf, n_1,n_2,w_{n_1},w_{n_2},j_{w_{n_1}},j_{w_{n_2}}})\omega=\psi_d,
\]
reduces to the equation 
\[
(\la- \tf_{\tilde\la_{q-1,n_2},-\al_{q-2}})u_2=v_{a,2},
\]
where $u_2$ and $v_{a,2}$ are the second components after application of the isomorphism $\ad$. Since $\tf_{\la,\al}=\lf_{\sqrt{ \la+\al^2}, \al}$, the last equation has a solution and this solution is in the domain of the corresponding operator $\la-L_{\nu,\al}$ is given by the integral operator $(\la-\L_{\nu,\al})^{-1}$, i.e. the resolvent of $L_{\nu,\al}$, described in Proposition \ref{p3.33}, i.e. 
\[
u_2=(\la-\L_{\nu,\al})^{-1}v_{a,2}.
\] 

By the that proposition and its corollary the result follows for forms of type b. 

It remains to deal with the forms of type b and c. As in the determination of the solutions of the eigenvalues equation, we observe that the subspace  generates by forms of types b and c coincides with the space generated by the forms 
$d \psi_a$ and $d^\da\psi_d$. Therefore, we can work with the last space. So let $d \psi_a$ be square integrable, since 
\[
d(\la-\Delta^{(q)}_{{\rm bc},\pf, n_1,n_2,w_{n_1},w_{n_2},j_{w_{n_1}},j_{w_{n_2}}})=(\la-\Delta^{(q)}_{{\rm bc},\pf, n_1,n_2,w_{n_1},w_{n_2},j_{w_{n_1}},j_{w_{n_2}}})d,
\]
if $\omega$ is a solution of 
\[
(\la-\Delta^{(q)}_{{\rm bc},\pf, n_1,n_2,w_{n_1},w_{n_2},j_{w_{n_1}},j_{w_{n_2}}})\omega=\psi_a,
\]
then $d\omega$ is a solution of 
\[
(\la-\Delta^{(q)}_{{\rm bc},\pf, n_1,n_2,w_{n_1},w_{n_2},j_{w_{n_1}},j_{w_{n_2}}})d\omega =d\psi_a.
\]

Such a $\omega$ is given by 
\[
((\la-\L_{\nu,\al})^{-1}v_{a,1},0),
\]
so $d\omega$ is 
\[
\df ((\la-\L_{\nu,\al})^{-1}v_{a,1},0).
\]

This shows that also in this case we have an integral kernel resolvent constructed by means of the resolvent of the operator $L_{\nu,\al}$, and therefore the result follows also in this case by Proposition \ref{p3.33} and its corollaries.

Thus, each of the operators $(\la-\Delta^{(q)}_{{\rm bc},\pf, n_1,n_2,w_{n_1},w_{n_2},j_{w_{n_1}},j_{w_{n_2}}})^{-1}$ is a bounded compact operator. By the second resolvent identity
\begin{align*}
&\left\|(\la-\Delta^{(q)}_{{\rm bc},\pf, n_1,n_2,w_{n_1},w_{n_2},j_{w_{n_1}},j_{w_{n_2}}})^{-1}
-(\la-\Delta^{(q)}_{{\rm bc},\pf})^{-1}\right\|\\
&\hspace{30pt}=\left\|(\la-\Delta^{(q)}_{{\rm bc},\pf, n_1,n_2,w_{n_1},w_{n_2},j_{w_{n_1}},j_{w_{n_2}}})^{-1}\right\|
\left\|\Delta^{(q)}_{{\rm bc},\pf, n_1,n_2,w_{n_1},w_{n_2},j_{w_{n_1}},j_{w_{n_2}}}-\Delta^{(q)}_{{\rm bc},\pf}\right\|\\
&\hspace{40pt}
\left\|(\la-\Delta^{(q)}_{{\rm bc},\pf})^{-1}\right\|.
\end{align*}

Since by definition $\Delta^{(q)}_{{\rm bc},\pf, n_1,n_2,w_{n_1},w_{n_2},j_{w_{n_1}},j_{w_{n_2}}}
\to \Delta^{(q)}_{{\rm bc},\pf}$ (it is a direct sum decomposition, projection on a complete orthonormal system), it follows that 
\[
\left(\la-\Delta^{(q)}_{{\rm bc},\pf, n_1,n_2,w_{n_1},w_{n_2},j_{w_{n_1}},j_{w_{n_2}}}\right)^{-1}
\to \left(\la-\Delta^{(q)}_{{\rm bc},\pf}\right)^{-1}.
\]

Since each $\left(\la-\Delta^{(q)}_{{\rm bc},\pf, n_1,n_2,w_{n_1},w_{n_2},j_{w_{n_1}},j_{w_{n_2}}}\right)^{-1}$ is (bounded) compact, $\left(\la-\Delta^{(q)}_{{\rm bc},\pf}\right)^{-1}$ is compact \cite[6.4]{Wei}.
\end{proof}

\subsection{The harmonic forms}
\label{harmonicforms}

In this section we describe the space of the harmonic forms of the operator $A_{\nu,\al}$ and consequently of the operator $\Delta$. Even if the calculations are made in the adjoint space, i.e., for the operator $A_{\nu,\al}$, the statments are given for the geometric operator $\Delta$.

\begin{prop} \label{har1}
There are the following natural isomorphisms between the  spaces of the harmonic forms $\H^\bu_{\rm abs, \mf^c}(C_{0,l}(W))$ of the operator $\Delta_{\rm abs, \mf^c}$ and those of the operator $\tilde \Delta$:
\begin{align*}
\H^q_{\rm abs, \mf^c}(C_{0,l}(W))&=\begin{cases}\H^q(W), &0\leq q\leq p-1,\\
 \{0\}, & p\leq q\leq m.\end{cases}
\end{align*}
where $m=2p-1$ or $m=2p$, with $p\geq 1$.

The precise isomorphism is as follows: let $\H^q(W)=\langle \tilde\vv^{(q)}_{{\rm har},j}\rangle$, then 
\[
\H^q_{\rm abs, \mf^c}(C_{0,l}(W))=\langle \L_{|\alpha_q|,\alpha_q,+}(x,0)
h(x)^{\alpha_q-\frac{1}{2}} \tilde\vv_{{\rm har}, j}^{(q)}\rangle=\langle  \tilde\vv_{{\rm har}, j}^{(q)}\rangle,
\]
for $0\leq q\leq p-1$.

\end{prop}

\begin{proof}  We look for the harmonics of $\Delta_{\rm abs, \mf^c}$ among the solutions of the harmonic equation described in Lemma \ref{harmonics}. By definition of $\Delta_{\rm abs, \mf^c}$, for $\theta^{(q)}_{T,\tilde\la_{q,n}\pm}$ ($T=E$, $O$, $I$, $II$, $III$, $IV$) to be an harmonic, it is necessary that it belongs to the domain of one of the following operators: $\Delta_{\rm abs}^q$, $\Delta_{\rm abs, +}^q$,  $\Delta_{\rm abs, -}^p$ or $\Delta_{\rm abs, -}^{p-1}$, where the last two appears only if the dimension $m=2p$ is even. Thus it is easier to proceed distinguishing the dimensions.

Assume $m=2p-1$, $p\geq 1$ is odd.

\begin{itemize}

\item[I:] According  to Lemma \ref{harmonics} the solutions of the harmonic equation of type I are
\[
\theta^{(q)}_{I,\tilde\la_{q,n},\pm}(x,0)=\L_{\mu_{q,n},\alpha_q,\pm}(x,0)h^{\al_q-\frac{1}{2}}(x)\tilde\vv^{(q)}_{\rm har}.
\]

According to Definition \ref{defiL}, the function $\L_{\mu_{q,n},\alpha_q,\pm}(x,0)$ is a solution of the harmonic equation for the formal operator $\lf_{\nu,\al}$.

By definition of $\Delta_{\rm abs, \mf^c}$, for $\theta^{(q)}_{E,\tilde\la_{q,n}\pm}$ to be an harmonic, it is necessary that it belongs to the domain of either the operator $\Delta_{\rm abs}^q$ or the operator $\Delta_{\rm abs, +}^q$. This means that the function $\L_{\mu_{q,n},\alpha_q,\pm}(x,0)$ must belong to the domain of the operator $L_{\mu_{q,n},\alpha_q, {\rm abs}, +}$. So $\L_{\mu_{q,n},\alpha_q,\pm}(x,0)$ should be an harmonic of $L_{\mu_{q,n},\alpha_q, {\rm abs}, +}$. By Lemma \ref{kerL}, this means that either $\L_{\mu_{q,n},\alpha_q,  \pm}=0$ or $\mu_{q,n}=\alpha_q$, so the analysis reduces to that of solutions of type E.

\item[II:] By Lemma \ref{harmonics} the solution of the harmonic equation of type II are
\begin{equation*}
\begin{aligned}
\theta^{(q)}_{II, \tilde\la_{q-1,n}, \pm}(x,0)=&h(x)^{\al_{q-1}-\frac{1}{2}} \L_{{\mu_{q-1,n},\alpha_{q-1}},\pm}(x,0)  \tilde\d \tilde\vv^{(q-1)}_{\tilde\lambda_{q-1,n}{\rm cex} }\\
&+  \left(\L_{{\mu_{q-1,n},\alpha_{q-1}},\pm}(x,0) 
h(x)^{\alpha_{q-1}-\frac{1}{2}}\right)' dx \wedge \tilde\vv^{(q-1)}_{\tilde\lambda_{q-1,n}, {\rm cex}}.
\end{aligned}
\end{equation*}

According to Definition \ref{defiL}, the function $\L_{\mu_{q,n},\alpha_q,\pm}(x,0)$ is a solution of the harmonic equation for the formal operator $\lf_{\nu,\al}$.

By definition of $\Delta_{\rm abs, \mf^c}$, for $\theta^{(q)}_{II,\tilde\la_{q,n}\pm}$ to be an harmonic, it is necessary that it belongs to the domain of either the operator $\Delta_{\rm abs}^q$ or the operator $\Delta_{\rm abs, +}^q$.

By Lemma \ref{l6-1}, the minus solutions are not in $L^2$, so we just need to consider the plus solutions. Write the solutions in the adjoint form and apply the operator $\AF^{q}_{n,n,{\rm ex}, {\rm cex}} $; we obtain
\begin{align*}
\AF^{q}_{n,n,{\rm ex},{\rm cex}} &\theta^{(q)}_{II, \tilde\la_{q-1,n}, +}(x,0) \\
&= \left(\begin{matrix} 
h(x)^{-1} \lf_{\mu_{q-1,n},\al_{q-1}}  \L_{{\mu_{q-1,n},\alpha_{q-1}},+}(x,0)  \tilde\d\tilde\vv^{(q-1)}_{\tilde\lambda_{q-1,n}{\rm cex}}\\
h(x)^{-\al_{q-1}+\frac{1}{2}} \left(\lf_{\mu_{q-1,n},\al_{q-1}} \L_{{\mu_{q-1,n},\alpha_{q-1}},+}(x,0) h(x)^{\alpha_{q-1}-\frac{1}{2}}\right)' \tilde\vv^{(q-1)}_{\tilde\lambda_{q-1,n}, {\rm cex}}
\end{matrix}\right).
\end{align*}

This means that the function $\L_{\mu_{q-1,n},\alpha_{q-1},+}(x,0)$ must belong to the domain of the operator $L_{\mu_{q-1,n},\alpha_{q-1}, {\rm abs}, +}$. So $\L_{\mu_{q-1,n},\alpha_{q-1},+}(x,0)$ should be an harmonic of the operator $L_{\mu_{q-1,n},\alpha_{q-1}, {\rm abs}, +}$. By Lemma \ref{kerL}, this means that either $\L_{\mu_{q-1,n},\alpha_{q-1},  +}(x,0)=0$ or $\mu_{q-1,n}=\alpha_{q-1}$. In the last case, $\mu_{q-1,n}=\alpha_{q-1}$ implies that $\tilde\la_{q-1,n}=0$, i.e. that $\tilde\vv^{(q-1)}_{\tilde\lambda_{q-1,n}, {\rm cex}}$ is an harmonic of the section, whence  the analysis reduces to that of solutions of type O.

\item[III:] Proceeding as for the solutions of type II, we obtain that $\mu_{q-1}  = -\al_{q-1}$ and then $\tilde \la_{q-1,n}  = 0$. Therefore $\tilde\vv^{(q-1)}_{\tilde\lambda_{q-1,n}, {\rm cex}}$ is an harmonic of the section and the solutions of type III are trivial.

\item[IV:] Proceeding as for the solutions of type I, we obtain that $\tilde\vv^{(q-2)}_{\tilde\lambda_{q-2,n}, {\rm cex}}$ is an harmonic of the section, therefore the solutions of type IV are trivial.

\item[E:] According  Lemma \ref{harmonics} the solutions of the harmonic equation of type E are
\[
\theta^{(q)}_{E, \pm}(x,0)=\L_{|\alpha_q|,\alpha_q,\pm}(x,0)h^{\al_q-\frac{1}{2}}(x)\tilde\vv^{(q)}_{\rm har}.
\]

According to Definition \ref{defiL}, the function $\L_{|\alpha_q|,\alpha_q,\pm}(x,0)$ is a solution of the harmonic equation for the formal operator $\lf_{|\al_q|,\al_q}$.

By definition of $\Delta_{\rm abs, \mf^c}$, for $\theta^{(q)}_{E,\pm}(x,0)$ to be an harmonic, it is necessary that it belongs to the domain of either the operator $\Delta_{\rm abs}^q$ or the operator $\Delta_{\rm abs,+}^q$. This means that the function $\L_{|\al_q|,\alpha_q,\pm}(x,0)$ must belong to the domain of the operator $L_{|\al_q|,\al_q, {\rm abs}, +}$. So $\L_{|\al_q|,\alpha_q,\pm}(x,0)$ should be an harmonic of $L_{|\al_q|,\al_q, {\rm abs}, +}$. By Lemma \ref{kerL}, this means that either $\L_{|\al_q|,\al_q, {\rm abs}, +}=0$ or $|\al_q|=-\al_q$. 

Note that the condition $|\al_q|=-\al_q$ gives $q\leq p-1$, independently on the parity of the dimension $m$. It follows that if  $q\geq p$,  the kernel of $\Delta_{\rm abs, \mf^c}$ is trivial. If $q< p-1$, the function $\L_{|\al_q|,\alpha_q,-}(x,0)$ is not square integrable by Lemma \ref{l3}, thus may not be in the kernel. If $q=p-1$ then $\L_{|\al_{p-1}|,\al_{p-1}}(x,0)$ does not satisfy the plus boundary condition at $x=0$, thus may not be in the kernel too. 
Thus it remains only $\L_{|\al_q|,\alpha_q,+}(x,0)=h^{\frac{1}{2}-\al_q}(x)$, that according to Lemma \ref{kerL} generates the kernel of $L_{|\al_q|,\al_q, {\rm abs}, +}$.

\item[O:] According  Lemma \ref{harmonics} the solutions of the harmonic equation of type E are
\[
\theta^{(q)}_{O,\pm}(x,0)=\L_{|\alpha_{q-2}|,-\alpha_{q-2},\pm}(x,0)h(x)^{\alpha_{q-1}-\frac{1}{2}} \d x\wedge\tilde\vv_{\rm har}^{(q-1)}.
\]

According to Definition \ref{defiL}, the function $\L_{|\alpha_{q-2}|,-\alpha_{q-2},\pm}(x,0)$ is a solution of the harmonic equation for the formal operator $\lf_{|\alpha_{q-2}|,-\alpha_{q-2}}$.

By definition of $\Delta_{\rm abs, \mf^c}$, for $\theta^{(q)}_{O,\pm}(x,0)$ to be an harmonic, it is necessary that it belongs to the domain of either  the operator $\Delta_{\rm abs}^q$ or the operator $\Delta_{\rm abs, +}^q$. This means that the function $\L_{|\al_{q-2}|,-\alpha_{q-2},\pm}(x,0)$ must belong to the domain of $L_{|\al_{q-2}|,-\al_{q-2}, {\rm rel}, +}$. So $\L_{|\al_{q-2}|,-\alpha_{q-2},\pm}(x,0)$ should be an harmonic of $L_{|\al_{q-2}|,-\al_{q-2}, {\rm rel}, +}$. Since the kernel of this operator is trivial by Lemma  \ref{kerL}, there are not harmonics of type O.

\end{itemize}

Assume $m=2p$, $p\geq 1$, is even.

\begin{itemize}

\item[I:] According  to Lemma \ref{harmonics} the solutions of the harmonic equation of type I are
\[
\theta^{(q)}_{I,\tilde\la_{q,n},\pm}(x,0)=\L_{\mu_{q,n},\alpha_q,\pm}(x,0)h^{\al_q-\frac{1}{2}}(x)\tilde\vv^{(q)}_{\rm har}.
\]

According to Definition \ref{defiL}, the function $\L_{\mu_{q,n},\alpha_q,\pm}(x,0)$ is a solution of the harmonic equation for the formal operator $\lf_{\nu,\al}$.

By definition of $\Delta_{\rm abs, \mf^c}$, for $\theta^{(q)}_{I,\tilde\la_{q,n}\pm}$ to be an harmonic, it is necessary that it belongs to the domain of either the operator $\Delta_{\rm abs}^q$ or the operator $\Delta_{\rm abs, +}^q$, when $q\not= p+1$, or to the domain of the operator  $\Delta_{\rm abs, -}^{p+1}$. The analysis for the + operators is the same as for the case of odd dimension, and shows that the solutions reduces to solutions of type E to be in the kernel. It remains to consider the case of degree $p+1$. 
But observe that in the definition of $\Delta_{\rm abs, \mf^c}$, the operator $\Delta_{\rm abs, -}^{p+1}$ actually appears only if $\tilde\la_{q,n}=0$. Otherwise the plus operator appears. This means that also in degree $p+1$ the solutions that may belong to the kernel reduce to solutions of type E.


\item[II:] By Lemma \ref{harmonics} the solution of the harmonic equation of type II are
\begin{equation*}
\begin{aligned}
\theta^{(q)}_{II, \tilde\la_{q-1,n}, \pm}(x,0)=&h(x)^{\al_{q-1}-\frac{1}{2}} \L_{{\mu_{q-1,n},\alpha_{q-1}},\pm}(x,0)  \tilde\d\tilde\vv^{(q-1)}_{\tilde\lambda_{q-1,n}{\rm cex} }\\
&+  \left(\L_{{\mu_{q-1,n},\alpha_{q-1}},\pm}(x,0) 
h(x)^{\alpha_{q-1}-\frac{1}{2}}\right)' dx \wedge \tilde\vv^{(q-1)}_{\tilde\lambda_{q-1,n}, {\rm cex}}.
\end{aligned}
\end{equation*}

According to Definition \ref{defiL}, the function $\L_{\mu_{q,n},\alpha_q,\pm}(x,0)$ is a solution of the harmonic equation for the formal operator $\lf_{\nu,\al}$.

By definition of $\Delta_{\rm abs, \mf^c}$, for $\theta^{(q)}_{II,\tilde\la_{q,n}\pm}$ to be an harmonic, it is necessary that it belongs to the domain of either the operator $\Delta_{\rm abs}^q$ or the operator $\Delta_{\rm abs, +}^q$.

By Lemma \ref{l6-1}, the minus solutions are not in $L^2$, so we just need to consider the plus solutions. Write the solutions in the adjoint form and apply the operator $\AF^{q}_{n,n,{\rm ex}, {\rm cex}} $; we obtain
\begin{align*}
\AF^{q}_{n,n,{\rm ex}, {\rm cex}} &\theta^{(q)}_{II, \tilde\la_{q-1,n}, +}(x,0) \\
&= \left(\begin{matrix} 
h(x)^{-1} \lf_{\mu_{q-1,n},\al_{q-1}}  \L_{{\mu_{q-1,n},\alpha_{q-1}},+}(x,0)  \tilde\d\tilde\vv^{(q-1)}_{\tilde\lambda_{q-1,n}{\rm cex}}\\
h(x)^{-\al_{q-1}+\frac{1}{2}} \left(\lf_{\mu_{q-1,n},\al_{q-1}} \L_{{\mu_{q-1,n},\alpha_{q-1}},+}(x,0) h(x)^{\alpha_{q-1}-\frac{1}{2}}\right)' \tilde\vv^{(q-1)}_{\tilde\lambda_{q-1,n}, {\rm cex}}
\end{matrix}\right).
\end{align*}

This means that the function $\L_{\mu_{q-1,n},\alpha_{q-1},+}(x,0)$ must belong to the domain of the operator $L_{\mu_{q-1,n},\alpha_{q-1}, {\rm abs}, +}$. So $\L_{\mu_{q-1,n},\alpha_{q-1},+}(x,0)$ should be an harmonic of the operator $L_{\mu_{q-1,n},\alpha_{q-1}, {\rm abs}, +}$. By Lemma \ref{kerL}, this means that either $\L_{\mu_{q-1,n},\alpha_{q-1},  +}=0$ or $\mu_{q-1,n}=\alpha_{q-1}$. In the last case, $\mu_{q-1,n}=\alpha_{q-1}$ implies that $\tilde\la_{q-1,n}=0$, i.e. that $\vv^{(q-1)}_{\tilde\lambda_{q-1,n}, {\rm cex}}$ is an harmonic of the section, whence  the analysis reduces to that of solutions of type O.

\item[III:] Proceeding as for the solutions of type II, we obtain that $\mu_{q-1}  = -\al_{q-1}$ and then $\tilde \la_{q-1,n}  = 0$. Therefore, $ \tilde\vv^{(q-1)}_{\tilde\lambda_{q-1,n}, {\rm cex}}$ is an harmonic of the section and the solutions of type III are trivial.

\item[IV:] Proceeding as for the solutions of type I, we obtain that $\tilde \vv^{(q-2)}_{\tilde\lambda_{q-2,n}, {\rm cex}}$ is an harmonic of the section, therefore the solutions of type IV are trivial.

\item[E:] According  Lemma \ref{harmonics} the solutions of the harmonic equation of type E are
\[
\theta^{(q)}_{E, \pm}(x,0)=\L_{|\alpha_q|,\alpha_q,\pm}(x,0)h^{\al_q-\frac{1}{2}}(x)\tilde\vv^{(q)}_{\rm har}.
\]

According to Definition \ref{defiL}, the function $\L_{|\alpha_q|,\alpha_q,\pm}(x,0)$ is a solution of the harmonic equation for the formal operator $\lf_{|\al_q|,\al_q}$.

By definition of $\Delta_{\rm abs, \mf^c}$, for $\theta^{(q)}_{E,\pm}(x,0)$ to be an harmonic, it is necessary that it belongs to the domain of either  the operator $\Delta_{\rm abs}^q$ or the operator $\Delta_{\rm abs, +}^q$, when $q\not=  p+1$, or to the domain of the operator  $\Delta_{\rm abs, -}^{p+1}$. The analysis for the + operators, i.e. $q\not= p+1$, is the same as for the case of odd dimension, whence It follows that there are not harmonics of type E in the kernel of $\Delta_{\rm abs, \mf^c}$ is trivial if  $q\geq p$, $q\not=p+1$, and while  $\L_{|\al_q|,\alpha_q,+}(x,0)=h^{\frac{1}{2}-\al_q}(x)$, is in the kernel of of $\Delta_{\rm abs, \mf^c}$ if  $q\leq p-1$.

It remains to consider the case of degree $p+1$. Since $\al_{p+1}=\frac{3}{2}$, recalling the definition of $\Delta_{\rm abs, -}^{p+1}$, in this case the function $\L_{\frac{3}{2},\frac{3}{2},\pm}(x,0)$ must belong to the domain of the operator $L_{\frac{3}{2},\frac{3}{2}, {\rm abs}}$, since $\nu=\frac{3}{2}>1$. So $\L_{\frac{3}{2},\frac{3}{2},\pm}(x,0)$ should be an harmonic of $L_{\frac{3}{2},\frac{3}{2}, {\rm abs}}$.  By Lemma \ref{kerL}, the kernel of $L_{\frac{3}{2},\frac{3}{2}, {\rm abs}}$ is trivial, so we do not have harmonics of type E in the kernel of of $\Delta_{\rm abs, \mf^c}$ in degree $p+1$.

\item[O:] According  Lemma \ref{harmonics} the solutions of the harmonic equation of type E are
\[
\theta^{(q)}_{O,\pm}(x,0)=\L_{|\alpha_{q-2}|,-\alpha_{q-2},\pm}(x,0)h(x)^{\alpha_{q-1}-\frac{1}{2}} \d x\wedge \tilde\vv_{\rm har}^{(q-1)}.
\]

According to Definition \ref{defiL}, the function $\L_{|\alpha_{q-2}|,\alpha_{q-2},\pm}(x,0)$ is a solution of the harmonic equation for the formal operator $\lf_{|\alpha_{q-2}|,\alpha_{q-2}}$.

By definition of $\Delta_{\rm abs, \mf^c}$, for $\theta^{(q)}_{O,0,\pm}(x,0)$ to be an harmonic, it is necessary that it belongs to the domain of either  the operator $\Delta_{\rm abs}^q$ or the operator $\Delta_{\rm abs, +}^q$, when $q\not=  p+1$, or to the domain of the operator  $\Delta_{\rm abs, -}^{p+1}$. The analysis for the + operators, i.e. $q\not= p+1$, is the same as for the case of odd dimension, whence It follows that there are not harmonics of type O in the kernel of $\Delta_{\rm abs, \mf^c}$ if  $q\leq p+1$.

It remains to consider the case of degree $p+1$. Since $\al_{p+1}=\frac{3}{2}$, recalling the definition of $\Delta_{\rm abs, -}^{p+1}$, in this case the function $\L_{\frac{3}{2},\frac{3}{2},\pm}(x,0)$ must belong to the domain of the operator $L_{\frac{3}{2},\frac{3}{2}, {\rm rel}}$, since $\nu=\frac{3}{2}>1$. So $\L_{\frac{3}{2},\frac{3}{2},\pm}(x,0)$ should be an harmonic of $L_{\frac{3}{2},\frac{3}{2}, {\rm rel}}$.  By Lemma \ref{kerL}, the kernel of $L_{\frac{3}{2},\frac{3}{2}, {\rm rel}}$ is trivial, so we do not have harmonics in the kernel of $\Delta_{\rm abs, \mf^c}$ in degree $p+1$.

\end{itemize}

\end{proof}

Note that if  $m=2p-1$, $\Delta_{{\rm abs}, \mf}  = \Delta_{{\rm abs}, \mf^c}$, whence $\H^q_{\rm abs, \mf}(C_{0,l}( W)) = \H^q_{\rm abs, \mf^c}(C_{0,l}( W))$.

\begin{prop}\label{har2}
If $m=2p$, $p\geq 1$, there is the following natural isomorphism between the  spaces of the harmonic forms $\H^\bu_{\rm abs, \mf}(C_{0,l}(W))$ of the operator $\Delta_{\rm abs, \mf}$ and that of the operator $\tilde \Delta$:
\begin{align*}
\H^q_{\rm abs, \mf}(C_{0,l}( W))&=
\begin{cases} 
\H^q (W) &  0\leq q\leq p,\\
 \{0\} &  p+1\leq q\leq 2p+1.
\end{cases}
\end{align*}

The precise isomorphisms are as follows: let $\H^q(W)=\langle \tilde\vv^{(q)}_{{\rm har},j}\rangle$, then:
\[
\H^q_{\rm abs, \mf}(C_{0,l}(W))
=\left\{\begin{array}{cc}\langle \L_{|\alpha_q|,\alpha_q,+}(x,0) h(x)^{\alpha_q-\frac{1}{2}} \tilde\vv_{{\rm har}, j}^{(q)}\rangle=\langle \tilde\vv_{{\rm har}, j}^{(q)}\rangle, &0\leq q\leq p-1,\\
\langle \L_{|\alpha_q|,\alpha_q,-}(x,0) h(x)^{\alpha_q-\frac{1}{2}} \tilde\vv_{{\rm har}, j}^{(q)}\rangle=\langle \tilde\vv_{{\rm har}, j}^{(p)}\rangle, &q= p.
\end{array}\right.
\]

\end{prop}

\begin{proof} The proof is very much the same as the last proposition, so we just outline here the key part of it, namely we verify that the forms of type $E-$ indeed verify the boundary condition at $x=0$ in degree  $q=p$. For, if $q=p$, $\al_{p}=\frac{1}{2}$, $\nu_{p,0}=\frac{1}{2}$, $\L_{\frac{1}{2},\frac{1}{2},-}(x,0)= 1$, generates the kernel of $L_{\frac{1}{2},\frac{1}{2},{\rm abs},-}$ by Lemma \ref{kerL}.  So we have verified that  $\theta^{(p)}_{E,-}$ is in the domain of $\Delta^{(p)}_{\rm abs, \mf}$ and it is an harmonic.  
\end{proof}

\begin{prop}\label{har3} There are the following natural isomorphisms between the  spaces of the harmonic forms $\H^\bu_{\rm rel, \mf^c}(C_{0,l}(W))$ of the operator $\Delta_{\rm rel, \mf^c}$ and those of the operator $\tilde \Delta$:
\begin{align*}
\H^q_{\rm rel, \mf^c}(C_{0,l}(W))&=\begin{cases}\{0\}, &0\leq q\leq p,\\
 \H^{q-1} (W), & p+1\leq q\leq m,\end{cases}
\end{align*}
where $m=2p-1$ or $m=2p$, $p\geq 1$. The precise isomorphism is as follows: let $\H^q(W)=\langle\tilde \vv^{(q)}_{{\rm har},j}\rangle$, 
\begin{align*}
\H^q_{\rm rel, \mf^c}(C_{0,l}(W))&=\langle h(x)^{2\al_{q-2}+1}  dx \wedge  \tilde\vv^{(q-1)}_{{\rm har},j} \rangle.
\end{align*}

\end{prop}

\begin{proof} 
The proof is very similar to the one of the Proposition \ref{propharmonicsabs}, the forms of type $I$, $II$ reduce to  forms of type E and O, and those of type  and $III$ and $IV$ are trivial. We discuss only the forms of types E and O.

\begin{itemize}
\item[E$\pm$:] 
 According  Lemma \ref{harmonics} 
\[
\theta^{(q)}_{E,\pm}(x,0)=\L_{|\alpha_q|,\alpha_q,\pm}(x,0)h(x)^{\al_q-\frac{1}{2}} \tilde\vv^{(q)}_{\rm har}.
\]

If $m=2p-1$, the forms $\theta^{(q)}_{E,-}(x,0)$ are in $L^2$ only if $q=p-1$.  But these forms do not satisfy the plus boundary condition at $x=0$, and therefore they are not in the domain of 
$\Delta_{{\rm rel}, \mf^c}$. The forms $\theta^{(q)}_{E,+}(x,0)$ are in $L^2$ and $\L_{|\al_q|,\al_q,+}(x,0)$ should be in the kernel of $L_{|\al_q|,\al_q,{\rm rel},+}$ or $L_{|\al_q|,\al_q,{\rm rel}}$ depending on $q$. But by Lemma \ref{kerL} both kernels are trivial. If $m=2p$ the only change is that the forms $\theta^{(q)}_{E,-}(x,0)$ are in $L^2$ only if $q=p-1$ or $q=p$, but these forms do not satisfy the boundary condition at $x=0$ and therefore are not in the domain $\Delta_{{\rm rel}, \mf^c}$. For the forms $\theta^{(q)}_{E,+}(x,0)$ the argument is the same as before, and then they are not harmonic.

\item[O$\pm$:]   By Lemma \ref{harmonics} 
\[
\theta^{(q)}_{O,\pm}(x,0)= \L_{|\alpha_{q-2}|,-\alpha_{q-2},\pm}(x,0) h(x)^{\al_{q-2}+\frac{1}{2}}\tilde\vv^{(q-1)}_{\rm har}.
\]

According to Definition \ref{defiL}, the function $\L_{|\alpha_{q-2}|,-\alpha_{q-2},\pm}(x,0)$ is a solution of the harmonic equation for the formal operator $\lf_{|\alpha_{q-2}|,-\alpha_{q-2}}$.

Consider first the odd case, $m=2p-1$. By definition of $\Delta_{\rm rel, \mf^c}$, for $\theta^{(q)}_{O,\pm}(x,0)$ to be an harmonic, it is necessary that it belongs to the domain of either  the operator $\Delta_{\rm rel}^q$, with $q\not=p-1,p+1$, or the operator $\Delta_{\rm rel, +}^q$, when $q= p-1, p+1$.  Thus we need to consider only forms of type $+$. 
Then, the requirement of satisfying the boundary condition at $x=0$ traduces in the requirement that the function 
$\L_{|\alpha_{q-2}|,-\alpha_{q-2},+}(x,0)$ belongs to the domain of either the operators $L_{|\al_{q-2}|,-\al_{q-2}, \rm{abs}}$, with $q\not=p-1,p+1$ or $L_{|\al_{q-2}|,-\al_{q-2}, \rm{abs},+}$, with $q=p-1,p+1$.

If $q\geq p+1$ then $\al_{q-2}\geq 0$ and by Lemma \ref{kerL} the kernel of both operators
is
\[
\L_{|\alpha_{q-2}|,-\alpha_{q-2},+}(x,0) =  h^{\frac{1}{2} +\al_{q-2}}(x).
\]

Then the forms $\theta^{(q)}_{O,+}(x)$ are harmonics.

If $q\leq p$ then $\al_{q-2}<0$ and by Lemma \ref{kerL}, the kernel both the operators 
is trivial. Then any $\theta^{(q)}_{O,+}(x,0)$ is not harmonic.


Next, assume $m=2p$ is even. By the definition of $\Delta_{{\rm rel},\mf^c}$ the forms of  type $+$ satisfy the boundary condition at $x=0$ for all $q \not = p+1$ while if $q=p+1$ we have the form $\theta^{(q)}_{O,-}(x)$.

If $q\geq p+2$ then $\al_{q-2}> 0$ and by Lemma \ref{kerL}, the kernel of $L_{|\al_{q-2}|,-\al_{q-2},{\rm abs},+}$ is generated by
\[
\L_{|\alpha_{q-2}|,-\alpha_{q-2},+}(x,0) =  h^{\frac{1}{2} +\al_{q-2}}(x).
\]
Therefore, the forms $\theta^{(q)}_{O,+}(x,0) = h(x)^{2\al_{q-2}+1} dx \wedge \tilde\vv^{q-1}_{{\rm har},j}$ are harmonics.

If $q< p+1$ then $\al_{q-2}<0$ and by Lemma \ref{kerL}, the kernel of $L_{|\al_{q-2}|,-\al_{q-2},{\rm abs},+}$ is trivial. This implies that
the forms $\theta^{(q)}_{O,+}(x,0)$ are not harmonics.

If $q=p+1$ then $\al_{q-2} = \al_{p-1} = -\frac{1}{2}$. By Lemma \ref{kerL}, the kernel of $L_{\frac{1}{2},\frac{1}{2},{\rm abs}, -}$ is generated by
\[
\L_{\frac{1}{2},\frac{1}{2},-}(x,0) = 1.
\]
Therefore, the form $\theta^{(p+1)}_{O,-}(x,0) =  dx \wedge \tilde \vv^{(p)}_{{\rm har},j}$ is harmonic.
\end{itemize}

\end{proof}

\begin{prop}\label{har4} If $m=2p$, $p\geq 1$, there are the following natural isomorphisms between the  spaces of the harmonic forms $\H^\bu_{\rm rel, \mf}(C_{0,l}(W))$ of the operator $\Delta_{\rm rel, \mf}$ and those of the operator $\tilde \Delta$:
\begin{align*}
\H^q_{\rm rel, \mf}(C_{0,l}( W))&=
\begin{cases} 
 \{0\} &  0\leq q\leq p+1,\\
\H^{q-1}(W) &  p+2\leq q\leq m,
\end{cases}
\end{align*}
where $m=2p-1$ or $m=2p$, $p\geq 1$. The precise isomorphisms are as follows: let $\H^q(W)=\langle \tilde\vv^{(q)}_{{\rm har},j}\rangle$, 
then 
\begin{align*}
\H^q_{\rm rel, \mf}(C_{0,l}(W))
&=\langle  h(x)^{2\alpha_{q-2}+1} dx \wedge \tilde\vv_{{\rm har}, j}^{(q-1)}\rangle,
\end{align*}
\end{prop}

\begin{proof}
The proof is very similar to the proof of the last proposition. We will discuss only the cases of forms of type $E$ and $O$, that are the cases where differences occur.  
The only changes appear in dimension $q=p$ and $q=p+1$.
If $q=p$, then we have the $-$ boundary condition at $x=0$. The only $L^2$  form of type $-$ in this dimension is $\theta_{E,-}^{(p)}$. Note that $\al_p = \frac{1}{2}$, whence, by Lemma \ref{kerL}, the kernel of $L_{\frac{1}{2},\frac{1}{2},{\rm rel},-}$ is trivial, and then $\theta_{E,-}^{(p)}$ is not a harmonic form.
If $q=p+1$, then we have the $+$ boundary condition at $0$. This implies  that the relevant form is $\theta_{O,+}^{(p+1)}$. Note that $\al_{q-2} = \al_{p-1} = -\frac{1}{2}$ and, by Lemma \ref{kerL}, the kernel of $L_{\frac{1}{2},\frac{1}{2},{\rm abs},+}$ is trivial. Therefore, the form $\theta_{O,+}^{(p+1)}$ is not a harmonic form.
\end{proof}

\begin{rem}\label{fields=forms} It follows by the proof of the previous propositions that harmonic fields and harmonic forms coincide on the cone. For this reason we will use the standard notation of harmonic forms.
\end{rem}

\begin{prop}\label{Hodge-duality} The Hodge start $\star:\Omega^q(C_{0,l}(W))\to \Omega^{m+1-q}(C_{0,l}(W))$ induces an isomorphism
\[
\star :\H^q_{\rm abs, \mf}(C_{0,l}(W))\to \H^{m+1-q}_{\rm rel, \mf^c}(C_{0,l}(W)).
\]
\end{prop}
\begin{proof} 
Note that  $m-2q = 2\al_{m-(q-1)-2}+1$, then the proof is a direct consequence of the propositions of Section \ref{harmonicforms} and the definition of the Hodge star in Section \ref{hodge}.
\end{proof}


\subsection{The spectrum and the eigenfunctions}
\label{spectrumC}

In this section we exhibit a useful description of the spectrum of the operators $A^q$ with the different boundary conditions, and consequently of the Laplace operators. As in the previous section, we prefer to write down the result for the geometric operators.

As above, we fix a complete orthonormal  basis $\left\{\tilde\varphi_{{\rm har}}^{(q)},\tilde\varphi_{{\rm cex},n}^{(q)},\tilde d\tilde\varphi_{{\rm cex},n}^{(q-1)}\right\}$ of $\Omega^q(W)$ consisting of harmonic,  coexact and exact eigenforms of 
$\tilde\Delta^{(q)}$, and we set $\tilde\lambda_{q,n}$ denote the  eigenvalue of $\tilde\varphi_{{\rm cex},n}^{(q)}$, and $m_{{\rm cex},q,n}=\dim \tilde\E^{(q)}_{{\rm cex},n}$.

\begin{prop}\label{l4}  The positive part of the spectrum of the lower Hodge-Laplace operators  $\Delta_{\rm abs, \mf}$  and $\Delta_{\rm abs, \mf^c}$ on $C_{0,l}(W)$, with absolute boundary conditions on $\b C_{0,l} (W)$ is as follows. If $m=\dim W=2p-1$, $p\geq 1$:
\begin{align*}
\Sp_+ \Delta_{\rm abs,\mf}^{(q)} =&\Sp_+ \Delta_{\rm abs,\mf^c}^{(q)}\\
=& \left\{m_{{\rm cex},q,n} : \hat 
\ell_{\mu_{q,n},\alpha_q,k} \right\}_{n,k=1}^{\infty}
\cup
\left\{m_{{\rm cex},q-1,n} : \hat 
\ell_{\mu_{q-1,n},\alpha_{q-1},k}\right\}_{n,k=1}^{\infty} \\
&\cup \left\{m_{{\rm cex},q-1,n} : 
\ell_{\mu_{q-1,n},-\alpha_{q-1},k}\right\}_{n,k=1}^{\infty} \cup \left\{m 
_{{\rm cex},q-2,n} :
\ell_{\mu_{q-2,n},-\alpha_{q-2},k}\right\}_{n,k=1}^{\infty} \\
&\cup \left\{m_{{\rm har},q}: 
\ell_{|\alpha_{q-1}|,-\alpha_{q-1},k}\right\}_{k=1}^{\infty} \cup \left\{ 
m_{{\rm har},q-1}:
\ell_{|\alpha_{q-2}|,-\alpha_{q-2},k}\right\}_{k=1}^{\infty}.
\end{align*}

If $m=\dim W=2p$, $p\geq 1$:
\begin{align*}
\Sp_+ \Delta_{\rm abs, \mf}^{(q\not= p, p+1)} =&\Sp_+ \Delta_{\rm abs,\mf^c}^{(q\not= p, p+1)} \\
=& \left\{m_{{\rm cex},q,n} : \hat 
\ell_{\mu_{q,n},\alpha_q,k}\right\}_{n,k=1}^{\infty}
\cup
\left\{m_{{\rm cex},q-1,n} : \hat 
\ell_{\mu_{q-1,n},\alpha_{q-1},k}\right\}_{n,k=1}^{\infty} \\
&\cup \left\{m_{{\rm cex},q-1,n} : 
\ell_{\mu_{q-1,n},-\alpha_{q-1},k}\right\}_{n,k=1}^{\infty} \cup \left\{m 
_{{\rm cex},q-2,n} :
\ell_{\mu_{q-2,n},-\alpha_{q-2},k}\right\}_{n,k=1}^{\infty} \\
&\cup \left\{m_{{\rm har},q}:
\ell_{|\alpha_{q-1}|,-\alpha_{q-1},k}\right\}_{k=1}^{\infty} \cup \left\{ 
m_{{\rm har},q-1}: 
\ell_{|\alpha_{q-2}|,-\alpha_{q-2},k}\right\}_{k=1}^{\infty},
\end{align*}

\begin{align*}
\Sp_+ \Delta_{\rm abs, \mf^c}^{(p)} 
=& \left\{m_{{\rm cex},p,n} : \hat 
\ell_{\mu_{p,n},\alpha_p,k}\right\}_{n,k=1}^{\infty}
\cup
\left\{m_{{\rm cex},p-1,n} : \hat 
\ell_{\mu_{p-1,n},\alpha_{p-1},k}\right\}_{n,k=1}^{\infty} \\
&\cup \left\{m_{{\rm cex},p-1,n} : 
\ell_{\mu_{p-1,n},-\alpha_{p-1},k}\right\}_{n,k=1}^{\infty} \cup \left\{m 
_{{\rm cex},p-2,n} :
\ell_{\mu_{p-2,n},-\alpha_{p-2},k}\right\}_{n,k=1}^{\infty} \\
&\cup \left\{ m_{{\rm har},p}: 
\ell_{\frac{1}{2},\frac{1}{2},-,k}\right\}_{k=1}^{\infty} 
\cup \left\{ m_{{\rm har},p-1}:  \ell_{|\al_{p-2}|,-\al_{p-2},k}\right\}_{k=1}^{\infty},
\end{align*}

\begin{align*}
\Sp_+ \Delta_{\rm abs, \mf^c}^{(p+1)} 
=& \left\{m_{{\rm cex},p+1,n} : \hat 
\ell_{\mu_{p+1,n},\alpha_{p+1},k}\right\}_{n,k=1}^{\infty}
\cup
\left\{m_{{\rm cex},p,n} : \hat 
\ell_{\mu_{p,n},\alpha_{p},k}\right\}_{n,k=1}^{\infty} \\
&\cup \left\{m_{{\rm cex},p,n} : 
\ell_{\mu_{p,n},-\alpha_p,k}\right\}_{n,k=1}^{\infty} \cup \left\{m 
_{{\rm cex},p-1,n} :
\ell_{\mu_{p-1,n},-\alpha_{p-1},k}\right\}_{n,k=1}^{\infty} \\
&\cup \left\{ m_{{\rm har},p+1}:
\ell_{|\al_{p}|,-\al_{p},k}\right\}_{k=1}^{\infty} \cup 
\left\{ m_{{\rm har},p}: 
\ell_{\frac{1}{2},\frac{1}{2},-,k}\right\}_{k=1}^{\infty}, 
\end{align*}

\begin{align*}
\Sp_+ \Delta_{\rm abs, \mf}^{(p)} 
=& \left\{m_{{\rm cex},p,n} : \hat 
\ell_{\mu_{p,n},\alpha_p,k}\right\}_{n,k=1}^{\infty}
\cup
\left\{m_{{\rm cex},p-1,n} : \hat 
\ell_{\mu_{p-1,n},\alpha_{p-1},k}\right\}_{n,k=1}^{\infty} \\
&\cup \left\{m_{{\rm cex},p-1,n} : 
\ell_{\mu_{p-1,n},-\alpha_{p-1},k}\right\}_{n,k=1}^{\infty} \cup \left\{m 
_{{\rm cex},p-2,n} :
\ell_{\mu_{p-2,n},-\alpha_{p-2},k}\right\}_{n,k=1}^{\infty} \\
&
\cup \left\{ m_{{\rm har},p}: \ell_{\frac{1}{2},\frac{1}{2},+,k}\right\}_{k=1}^{\infty}
\cup \left\{ m_{{\rm har},p-1}: \ell_{|\al_{p-2}|,-\al_{p-2},k}\right\}_{k=1}^{\infty},
\end{align*}

\begin{align*}
\Sp_+ \Delta_{\rm abs, \mf}^{(p+1)} 
=& \left\{m_{{\rm cex},p+1,n} : \hat 
\ell_{\mu_{p+1,n},\alpha_{p+1},k}\right\}_{n,k=1}^{\infty}
\cup
\left\{m_{{\rm cex},p,n} : \hat 
\ell_{\mu_{p,n},\alpha_{p},k}\right\}_{n,k=1}^{\infty} \\
&\cup \left\{m_{{\rm cex},p,n} : 
\ell_{\mu_{p,n},-\alpha_p,k}\right\}_{n,k=1}^{\infty} \cup \left\{m 
_{{\rm cex},p-1,n} :
\ell_{\mu_{p-1,n},-\alpha_{p-1},k}\right\}_{n,k=1}^{\infty} \\
&\cup \left\{ m_{{\rm har},p+1}:
\ell_{|\al_{p}|,-\al_{p},k}\right\}_{k=1}^{\infty} 
\cup 
\left\{m_{{\rm har},p}: 
\ell_{\frac{1}{2},\frac{1}{2},+,k}\right\}_{k=1}^{\infty},
\end{align*}
where  the $\ell_{\mu,\alpha,k}$ are the non vanishing  zeros of the  function 
$\L_{\mu,\alpha,+}(l,\lambda)$, and  
$\hat \ell_{\mu,\alpha,k}$ are the non vanishing  zeros of the function 
\[
\left(\alpha-\frac{1}{2}\right)h'(l)\L_{\mu,\alpha,+}(l,\lambda)
+h(l)\L_{\mu,\alpha,+}'(l,\lambda) .
\] 
with  $\alpha\in \R$, $\alpha_q$ and $\mu_{q,n}$ are defined in Lemma 
\ref{l2}. The particular case, $\hat \ell_{\frac{1}{2},\frac{1}{2},\pm,k}$ 
are zeros of $\L_{\frac{1}{2},\frac{1}{2},\pm}(l,\lambda)$ discussed in Remark \ref{ummejo}. 
The enumeration of the eigenvalues in this lemma is as follows: both the indices start denoting by their minimum allowed value, i.e. 1, the smaller of the positive eigenvalues; respectively: the index $n$ for the eigenvalues of the section, the index $k$ for the eigenvalues of the operator on the segment.
\end{prop}

\begin{proof} The result follows by considering all those solutions  of the eigenvalues equation as described in Lemma \ref{l2b} that are square integrable, according to the table in Lemma \ref{l3}, and that satisfy the boundary condition at $x=0$ according to the Definition \ref{HodgeLaplace}. 
Then we only need to consider the positive solutions if $m = 2p-1$ and if $m=2p$, we need to consider negative solutions of type $E$ and $O$ in dimensions $q=p$ and $q=p+1$.

We consider in some details the odd case, $m=2p-1$,  for the operator $\Delta_{\rm abs, \mf^c}$ and afterwards the we outline the main points in the even case $m=2p$. 

\begin{itemize}

\item[I+:] A solution of type I+ is
\[
\psi^{(q)}_{I, \tilde\lambda_{q,n}, +}(x,\la)=\L_{{\mu_{q,n},\alpha_q},+}(x,\lambda) h^{\al_q -\frac{1}{2}}(x)\tilde\vv^{(q)}_{\tilde\lambda_{q,n}, {\rm cex}},
\]

This solution is square integrable and satisfies the plus boundary condition at $x=0$ for all $q$.
We may apply the the boundary condition at $x=l$, that reads
\[
\left. \left(h^{\al_q-\frac{1}{2}}(x)\L_{{\mu_{q,n},\alpha_q},+}(x,\lambda)\right)'\right|_{x=l}=0,
\]
and this gives the first set of eigenvalues $\hat \ell_{\mu_{q,n},\alpha_q,k} $, $n,k=1, 2, \dots$.

\item[II+] 
A solution of type II is 	
\[
\begin{aligned}
\psi^{(q)}_{II, \tilde\lambda_{q-1,n}, +}(x,\la)&=\L_{\mu_{q-1,\alpha_{q-1},n},+}(x,\lambda) h(x)^{\al_{q-1}-\frac{1}{2}}  \ \tilde\d\tilde\vv^{(q-1)}_{\tilde \lambda_{q-1,n},{\rm cex} } \\
&+ (\L_{\mu_{q-1,\alpha_{q-1},n},+}(x,\lambda) \ h(x)^{\alpha_{q-1}-\frac{1}{2}})' dx \wedge \tilde\vv^{(q-1)}_{\tilde \lambda_{q-1,n}, {\rm cex}}).
\end{aligned}
\] 
Then the boundary condition at $x=l$ \eqref{abs1} is
\[
\left.\left(\L_{\mu_{q-1,n},\alpha_{q-1},+}(x,\lambda) \ 
h(x)^{\alpha_{q-1}-\frac{1}{2}}\right)'\right|_{x=l}=0,
\] 
and  gives the second set of eigenvalues $\hat \ell_{\mu_{q-1,n},\alpha_{q-1},k} $, $n,k=1,2,\dots$.

\item[III+] 
	
For the solution of type $III$ we have,
\[
\begin{aligned}
\psi^{(q)}_{III, \tilde\lambda_{q-1,n},+}(x,\la)=& h(x)^{2\alpha_{q-1}+1}\b_x(h(x)^{-\alpha_{q-1}-\frac{1}{2}} \L_{\mu_{q-1,n},-\alpha_{q-1},+}(x,\lambda)) \tilde \d\tilde\vv^{(q-1)}_{\tilde\lambda_{q-1,n}, {\rm cex}} \\
&\qquad +  \L_{\mu_{q-1,n},-\alpha_{q-1},+}(x,\lambda) \ h(x)^{\alpha_{q-1}-\frac{3}{2}}\d x\wedge \tilde\d^{\dag}\tilde \d\tilde \vv^{(q-1)}_{\tilde\lambda_{q-1,n}, {\rm cex}},
\end{aligned}
\]
and the boundary condition reads
\[
\left\{\begin{array}{l}
\L_{\mu_{q-1,n},-\alpha_{q-1},+}(l,\lambda) \ h(l)^{\alpha_{q-1}-\frac{3}{2}}=0,\\
\b_x(h(x)^{2\alpha_{q-1}+1}\b_x(h(x)^{-\alpha_{q-1}-\frac{1}{2}} 
\L_{\mu_{q-1,n},-\alpha_{q-1},+}(x,\lambda)))|_{x=l}=0,
\end{array}
\right.
\] 
since
\begin{equation*}
\begin{aligned} \b_x(h(x)^{2\alpha_{q-1}+1}&\b_x(h(x)^{-\alpha_{q-1}-\frac{1}{2}} 
\L_{\mu_{q-1,n},-\al_{q-1},+}(x,\lambda)))  \\
&= - h(x)^{\alpha_{q-1}+\frac{1}{2}}\left(\la  - \frac{\tilde \lambda_{q,n}}{h(x)^2}  \right)\L_{\mu_{q-1,n},-\al_{q-1},+}(x,\lambda).
\end{aligned}
\end{equation*}
this gives the third  set of eigenvalues $ \ell_{\mu_{q-1,n},-\alpha_{q-1},k} $, $n,k=1,2,\dots$.	

\item[IV+]  A solution of type $IV+$ is
\begin{equation*}
\psi^{(q)}_{IV, \tilde\lambda_{q-2,n},
+}(x,\la)=\L_{{\mu_{q-2,n},-\alpha_{q-2}},+}(x,\lambda) \ 
h(x)^{\alpha_{q-2}+\frac{1}{2}}\
\d x\wedge \tilde \d\tilde\vv^{(q-2)}_{\tilde\lambda_{q-2,n}, {\rm cex}}.
\end{equation*}
If we apply the boundary condition in equation \eqref{abs1}, we obtain the fourth set of eigenvalues $\left\{m 
_{{\rm cex},q-2,n} :
\ell_{\mu_{q-2,n},-\alpha_{q-2},k}\right\}_{n,k=1}^{\infty} $.

\item[E+]
A solution of type $E+$ is
\[
\psi^{(q)}_{E, +}(x,\la)=\L_{|\alpha_q|,\alpha_q,+}(x,\lambda) \  
h(x)^{\alpha_q-\frac{1}{2}} \tilde\vv_{\rm har}^{(q)},
\]
and the boundary condition at $x=l$ gives
\[
\left( h(x)^{\al_q - \frac{1}{2}}\L_{|\alpha_{q}|,\alpha_{q}, +}(x,l) \right)'|_{x=l}=0,
\]
since by Lemma \ref{applem2} 
\[
\b_x \left(h(x)^{\al_{q}-\frac{1}{2}} \L_{|\al_{q}|,\al_{q},+}(x,\la)\right) = A_+ h(x)^{\al_{q-1}+\frac{1}{2}} \L_{|\al_{q-1}|,-\al_{q-1},+ }(x,\la)
\]
the boundary condition at $x=l$ is   
\[
\L_{|\al_{q-1}|,-\al_{q-1},+}(l,\la)=0,
\]
and this gives the set $ \ell_{|\al_{q-1}|,-\alpha_{q-1},k}$, $n,k=1,2,\dots$.

\item[O+] A solution of type $O+$ is
\[
\psi^{(q)}_{O,  +}(x,\la)= \L_{|\alpha_{q-2}|,-\alpha_{q-2}, +}(x,\la)  h(x)^{\al_{q-1}-\frac{1}{2}} \tilde\vv_{\rm har}^{(q-1)},
\]
the boundary condition is
\[
\L_{|\alpha_{q-2}|,-\alpha_{q-2}, +}(l,\la)  \tilde\vv_{\rm har}^{(q-1)}=0,
\]
and this gives the set $\ell_{|\al_{q-2}|,-\alpha_{q-2},k}$, $n,k=1,2,\dots$.

\end{itemize}

We now consider the case
 $\dim W = 2p$. In such a case,  we need some care with the forms $E$ and $O$ in dimensions $q=p$ and $q=p+1$. 

If we are in the case $\Delta^{(p)}_{\rm abs, \mf^c}$, then all the forms of type $E \pm$ are in $L^2$, but only the $+$ one satisfies the boundary condition at $x=0$;  applying the boundary condition at $x=l$ and using Lemma \ref{applem2},  we obtain the set  $\ell_{\frac{1}{2},\frac{1}{2},-,k}$, $k=1,2,\dots$. 
In the case $\Delta^{(p+1)}_{\rm abs, \mf^c}$,  all the forms of type $O \pm$ are in $L^2$, but only the $-$ one  satisfies the boundary condition at $x=0$; if we apply the boundary condition at $x=l$ we obtain the set $ \ell_{\frac{1}{2},\frac{1}{2},-,k}$, $k=1,2,\dots$. 

In the case $\Delta^{(p)}_{\rm abs, \mf}$,  only the form $E-$ satisfies the boundary condition at $x=0$, applying the boundary condition at $x=l$ and Lemma \ref{applem2}, we obtain the set  $ \ell_{\frac{1}{2},\frac{1}{2},+,k}$, $k=1,2,\dots$. 
The last case is  $\Delta^{(p+1)}_{\rm abs, \mf}$,  then only the form $O-$ satisfies the boundary condition at $x=0$, and  if we apply the boundary condition at $x=l$,  the set of the eigenvalues is $ \ell_{\frac{1}{2},\frac{1}{2},-,k}$, $k=1,2,\dots$. 
\end{proof}

\begin{rem}\label{ummejo} 
In this remark we discuss the particular case of the eigenvalues $\ell_{\frac{1}{2},\frac{1}{2},\pm,k}$ appearing in the even case $m=2p$ in dimensions $p$ or $p+1$. By definition, these numbers are the zeros of the functions $\L_{\left|\frac{1}{2}\right|,\frac{1}{2}, \pm}$, that are two independent  solutions of the equation (see Section \ref{soleq})
\[
u''+\la u=0.
\]

Therefore, 
\begin{align*}
\L_{\left|\frac{1}{2}\right|,\frac{1}{2}, +}(x,\la)&=\frac{1}{\sqrt{\la}}\sin(\sqrt{\la}x),&\L_{\left|\frac{1}{2}\right|,\frac{1}{2}, -}(x,\la)&=\cos(\sqrt{\la} x),
\end{align*}
adopting the normalisation introduced in Remark \ref{rem3.1}. It follows that,
\begin{align*}
\ell_{\frac{1}{2},\frac{1}{2},+,k}&=\left(\frac{\pi}{l}{k}\right)^2, 
&\ell_{\frac{1}{2},\frac{1}{2},-,k}&=\left(\frac{\pi}{2l}(2k-1)\right)^2.
\end{align*}
	
The associated zeta functions are
\begin{align*}
\zeta_-(s)&=\left(\frac{\pi}{2l}\right)^{-2s}\sum_{k=0}^\infty (2k+1)^{-2s}=\left(\frac{\pi}{2l}\right)^{-2s}\left(\zeta_{\rm R}(2s)-2^{-2s}\zeta_{\rm R}(2s)\right),\\
\zeta_+(s)&=\left(\frac{\pi}{l}\right)^{-2s}\sum_{k=1}^\infty (2k)^{-2s}=\left(\frac{\pi}{l}\right)^{-2s}\zeta_{\rm R}(2s).
\end{align*}

In the following we will need their derivative at $s=0$, that may be easily computed using that of the Riemann zeta function.
\end{rem}

\begin{prop}\label{l4.1}  The positive part of the spectrum of the  Hodge-Laplace operators $\Delta_{\rm rel, \mf}$  and $\Delta_{\rm rel, \mf^c}$ on $C_{0,l}(W)$, with relative boundary conditions on $\b C_{0,l} (W)$ is as follows. If $m=\dim W=2p-1$, $p\geq 1$:

\begin{align*}
\Sp_+ \Delta_{\rm rel, \mf}^{(q)}=&\Sp_+ \Delta_{\rm rel, \mf^c}^{(q)}\\
 =& \left\{m_{{\rm cex},q,n} :  
\ell_{\mu_{q,n},\alpha_q,k} \right\}_{n,k=1}^{\infty}
\cup
\left\{m_{{\rm cex},q-1,n} :  
\ell_{\mu_{q-1,n},\alpha_{q-1},k}\right\}_{n,k=1}^{\infty} \\
&\cup \left\{m_{{\rm cex},q-1,n} : 
\hat \ell_{\mu_{q-1,n},-\alpha_{q-1},k}\right\}_{n,k=1}^{\infty} \cup \left\{m 
_{{\rm cex},q-2,n} :
\hat \ell_{\mu_{q-2,n},-\alpha_{q-2},k}\right\}_{n,k=1}^{\infty} \\
&\cup \left\{m_{{\rm har},q}: 
\ell_{|\alpha_{q}|,\alpha_{q},k}\right\}_{k=1}^{\infty} \cup \left\{ 
m_{{\rm har},q-1}:
\ell_{|\alpha_{q-1}|,\alpha_{q-1},k}\right\}_{k=1}^{\infty}.
\end{align*}

If $m=\dim W=2p$, $p\geq 1$:
\begin{align*}
\Sp_+ \Delta_{\rm rel, \mf}^{(q\not= p, p+1)} =&\Sp_+ \Delta_{\rm rel, \mf^c}^{(q\not= p, p+1)} \\
=& \left\{m_{{\rm cex},q,n} : 
\ell_{\mu_{q,n},\alpha_q,k}\right\}_{n,k=1}^{\infty}
\cup
\left\{m_{{\rm cex},q-1,n} :  
\ell_{\mu_{q-1,n},\alpha_{q-1},k}\right\}_{n,k=1}^{\infty} \\
&\cup \left\{m_{{\rm cex},q-1,n} : 
\hat \ell_{\mu_{q-1,n},-\alpha_{q-1},k}\right\}_{n,k=1}^{\infty} \cup \left\{m 
_{{\rm cex},q-2,n} :
\hat \ell_{\mu_{q-2,n},-\alpha_{q-2},k}\right\}_{n,k=1}^{\infty} \\
&\cup \left\{m_{{\rm har},q}:
\ell_{|\alpha_{q}|,\alpha_{q},k}\right\}_{k=1}^{\infty} \cup \left\{ 
m_{{\rm har},q-1}: 
\ell_{|\alpha_{q-1}|,\alpha_{q-1},k}\right\}_{k=1}^{\infty},
\end{align*}

\begin{align*}
\Sp_+ \Delta_{\rm rel, \mf^c}^{(p)} 
=& \left\{m_{{\rm cex},p,n} : 
\ell_{\mu_{p,n},\alpha_p,k}\right\}_{n,k=1}^{\infty}
\cup
\left\{m_{{\rm cex},p-1,n} :  
\ell_{\mu_{p-1,n},\alpha_{p-1},k}\right\}_{n,k=1}^{\infty} \\
&\cup \left\{m_{{\rm cex},p-1,n} : 
\hat \ell_{\mu_{p-1,n},-\alpha_{p-1},k}\right\}_{n,k=1}^{\infty} \cup \left\{m 
_{{\rm cex},p-2,n} :
\hat \ell_{\mu_{p-2,n},-\alpha_{p-2},k}\right\}_{n,k=1}^{\infty} \\
&\cup \left\{ m_{{\rm har},p}: 
\ell_{\frac{1}{2},\frac{1}{2},+,k}\right\}_{k=1}^{\infty} 
\cup \left\{ m_{{\rm har},p-1}:  \ell_{|\al_{p-1}|,\al_{p-1},k}\right\}_{k=1}^{\infty},
\end{align*}

\begin{align*}
\Sp_+ \Delta_{\rm rel, \mf^c}^{(p+1)} 
=& \left\{m_{{\rm cex},p+1,n} :
\ell_{\mu_{p+1,n},\alpha_{p+1},k}\right\}_{n,k=1}^{\infty}
\cup
\left\{m_{{\rm cex},p,n} : 
\ell_{\mu_{p,n},\alpha_{p},k}\right\}_{n,k=1}^{\infty} \\
&\cup \left\{m_{{\rm cex},p,n} : 
\hat \ell_{\mu_{p,n},-\alpha_p,k}\right\}_{n,k=1}^{\infty} \cup \left\{m 
_{{\rm cex},p-1,n} :
\hat \ell_{\mu_{p-1,n},-\alpha_{p-1},k}\right\}_{n,k=1}^{\infty} \\
&\cup \left\{ m_{{\rm har},p+1}:
\ell_{|\al_{p+1}|,\al_{p+1},k}\right\}_{k=1}^{\infty} \cup 
\left\{ m_{{\rm har},p}: 
\ell_{\frac{1}{2},\frac{1}{2},+,k}\right\}_{k=1}^{\infty}, 
\end{align*}

\begin{align*}
\Sp_+ \Delta_{\rm rel, \mf}^{(p)} 
=& \left\{m_{{\rm cex},p,n} : 
\ell_{\mu_{p,n},\alpha_p,k}\right\}_{n,k=1}^{\infty}
\cup
\left\{m_{{\rm cex},p-1,n} :  
\ell_{\mu_{p-1,n},\alpha_{p-1},k}\right\}_{n,k=1}^{\infty} \\
&\cup \left\{m_{{\rm cex},p-1,n} : 
\hat \ell_{\mu_{p-1,n},-\alpha_{p-1},k}\right\}_{n,k=1}^{\infty} \cup \left\{m 
_{{\rm cex},p-2,n} :
\hat \ell_{\mu_{p-2,n},-\alpha_{p-2},k}\right\}_{n,k=1}^{\infty} \\
&
\cup \left\{ m_{{\rm har},p}: \ell_{\frac{1}{2},\frac{1}{2},-,k}\right\}_{k=1}^{\infty}
\cup \left\{ m_{{\rm har},p-1}: \ell_{|\al_{p-1}|,\al_{p-1},k}\right\}_{k=1}^{\infty},
\end{align*}

\begin{align*}
\Sp_+ \Delta_{\rm rel, \mf}^{(p+1)} 
=& \left\{m_{{\rm cex},p+1,n} :  
\ell_{\mu_{p+1,n},\alpha_{p+1},k}\right\}_{n,k=1}^{\infty}
\cup
\left\{m_{{\rm cex},p,n} : 
\ell_{\mu_{p,n},\alpha_{p},k}\right\}_{n,k=1}^{\infty} \\
&\cup \left\{m_{{\rm cex},p,n} : 
\hat \ell_{\mu_{p,n},-\alpha_p,k}\right\}_{n,k=1}^{\infty} \cup \left\{m 
_{{\rm cex},p-1,n} :
\hat \ell_{\mu_{p-1,n},-\alpha_{p-1},k}\right\}_{n,k=1}^{\infty} \\
&\cup \left\{ m_{{\rm har},p+1}:
\ell_{|\al_{p+1}|,\al_{p+1},k}\right\}_{k=1}^{\infty} 
\cup 
\left\{m_{{\rm har},p}: 
\ell_{\frac{1}{2},\frac{1}{2},-,k}\right\}_{k=1}^{\infty},
\end{align*}

\end{prop}

\begin{proof} The proof is very similar to the proof of Proposition \ref{l4}: we consider the solutions of the eigenvalue equation described in Lemma \ref{l2} that are square integrable and satisfy the boundary condition at $x=0$ described in Definition \ref{HodgeLaplace} according to the domain of the operator with $\mf$ and $\mf^c$, and then we apply the relative boundary condition at $x=l$, Equation \eqref{rel2}. 
 We just observe that in the case $m=2p$ and dimensions $q=p$ and $q=p+1$ we also used Lemma \ref{applem2}. 
\end{proof}

\begin{corol}\label{tracecorol} A power of the resolvent of the  operators $\Delta_{\rm bc, \pf}^{(q)}$ is of trace class.
\end{corol}

\begin{theo}\label{completebasis} (Hodge decomposition) The operators $\Delta_{\rm bc, \pf}$ have a spectral resolution. The family of the harmonic forms described in Propositions \ref{har1} and \ref{har2} plus the family of the eigenfunctions corresponding to the eigenvalues described in Proposition \ref{l4} determines a  complete basis for the space $L^2(C_{0,l}(W))$ consisting in eigenfunctions of $\Delta_{\rm abs, \pf}$. Those in Propositions \ref{har3}, \ref{har4}, and \ref{l4.1} gives the corresponding basis for the operators $\Delta_{\rm rel, \pf}$.
\end{theo}
\begin{proof} This follows by  spectral decomposition on a spectral resolution of the section, and Corollary \ref{c3.36}.
\end{proof}

\begin{rem}
With the results presented in propositions \ref{Hodge-duality}, \ref{l4} and \ref{l4.1} we have the following isomorphisms
\[
\star: D(\Delta_{\rm abs, \mf}) \to D(\Delta_{\rm rel, \mf^c}), \qquad \text{and} \qquad \star: D(\Delta_{\rm abs, \mf^c}) \to D(\Delta_{\rm rel, \mf}).
\]
\end{rem}

We conclude comparing the operator $\Delta_{\rm bc, \pf}$ with the Hodge-Laplace operator $\Delta_{\rm bc}$ in the definition of Cheeger , see Remark \ref{lapche}. We recall that $\Delta_{\rm bc}$ is defined in all cases, up to when $m=2p$ and $H_p(W)\not=0$,  just by restricting the maximal domain requiring that $\d \omega$ and $\d^\da\omega$ are square integrable. When $m=2p$ and $H_p(W)\not=0$, some ideal boundary conditions are introduced (we referee to our work \cite{HS5} for definition of ideal boundary condition and the corresponding results for the analytic torsion, in the case of a flat cone). 

\begin{theo} Assume that if $m=2p$, $p\geq 1$, is even, then $H_p(W)=0$. Then, the operators $\Delta_{\rm abs, \pf}$  coincide with the operator $\Delta_{\rm abs}$;  the operators $\Delta_{\rm rel, \pf}$ coincide with the operator $\Delta_{\rm rel}$. 
\end{theo}
\begin{proof} A long but straightforward verification shows that the Cheeger conditions on square integrability coincide with our conditions in the definitions of the operators $\Delta_{\rm bc, \pf}$ on the solutions of the harmonic and eigenvalues equations, respectively. Therefore, these operators have the same spectral resolutions, and the theorem follows. 
\end{proof}

\begin{lem}\label{l6} The positive part of the spectrum of the Hodge-Laplace operator \\$\Delta_{{\rm rel} \ \b_1,{\rm abs} \ \b_2}$  on $C_{a,b}(W)$, $a>0$, with relative boundary condition on $\b_1 = \{a\}\times W$ and absolute boundary conditions on $\b_2 = \{b\}\times W$ is as follows:

\begin{equation*}
\begin{aligned}
\Sp_+ &\Delta_{{\rm rel} \ \b_1,{\rm abs} \ \b_2}^{(q)} 
= \left\{m_{{\rm cex},q,n} : \hat f_{\mu_{q,n},\alpha_q,k}(a,b) 
\right\}_{n,k=1}^{\infty}\cup \left\{m_{{\rm cex},q-1,n} : \hat 
f_{\mu_{q-1,n},\alpha_{q-1},k}(a,b)\right\}_{n,k=1}^{\infty} \\
&\cup \left\{m_{{\rm cex},q-1,n} : \hat 
f_{\mu_{q-1,n},-\alpha_{q-1},k}(b,a)\right\}_{n,k=1}^{\infty} \cup 
\left\{m_{{\rm cex},q-2,n} : \hat 
f_{\mu_{q-2,n},-\alpha_{q-2},k}(b,a)\right\}_{n,k=1}^{\infty} \\
&\cup \left\{m_{{\rm har},q}:\hat 
f_{|\alpha_q|,\alpha_q,k}(a,b)\right\}_{k=1}^{\infty} \cup \left\{m_{{\rm 
har},q-1}: \hat 
f_{|\alpha_{q-1}|,\alpha_{q-1},k}(a,b)\right\}_{k=1}^{\infty},
\end{aligned}
\end{equation*}
where $\hat f_{\mu,c,k}(a,b)$ are non vanishing zeros of 
\begin{align*}
\hat F_{\mu,c}(a,b,\lambda) =&\L_{\mu,c,+}(a,\lambda) \b_x 
\left(\L_{\mu,c,-}(b,\lambda) h(b)^{c-\frac{1}{2}}\right)' -\L_{\mu,c,-}(a,\lambda)\left(\L_{\mu,c,+}(b,\lambda)  h(b)^{c-\frac{1}{2}}\right)'.
\end{align*}

\end{lem}
\begin{proof} All solutions are square integrable, so we just need to apply the boundary conditions. We give some details for  solution of Type I,  from Lemma \ref{l2}, these have the following form
\begin{equation*}
\psi_{ I,\tilde\lambda_{q,n}}^{(q)}(x,\la)= \left(A \L_{\mu_{q,n},\alpha_q,+}(x,\lambda) + 
B\L_{\mu_{q,n},\alpha_q,-}(x,\lambda)\right) h(x)^{\alpha-\frac{1}{2}} 
\tilde\varphi_{\tilde{\lambda}_{q,n},{\rm cex}}^{(q)},
\end{equation*}
where $A$ and $B$ are constants. The boundary conditions on $\partial_1$ and $\partial_2$ read as follows:
\begin{equation*}
\begin{cases}
A\L_{\mu_{q,n},\alpha_q,+}(a,\lambda) + 
B\L_{\mu_{q,n},\alpha_q,-}(a,\lambda) = 0 \\
A \frac{d}{dx} 
(\L_{\mu_{q,n},\alpha_q,+}(x,\lambda)h(x)^{\alpha-\frac{1}{2}})|_{x=b} 
+ 
B
(\L_{\mu_{q,n},\alpha_q,-}(x,\lambda)h(x)^{\alpha-\frac{1}{2}})|_{x=b} = 0.
\end{cases}
\end{equation*}

Since $AB \not = 0$, we have that $\lambda$ is the solution of the following equation
\begin{equation*}
\begin{aligned}
\L_{\mu_{q,n},\alpha_q,+}(a,\lambda) 
\frac{d}{dx}&\left(\L_{\mu_{q,n},\alpha_q,-}(x,\lambda)h(x)^{\alpha-\frac{1}{2}}
\right)|_{x=b} \\
&- \L_{\mu_{q,n},\alpha_q,-}(a,\lambda)\frac{d}{dx}
\left(\L_{\mu_{q,n},\alpha_q,+}(x,\lambda)h(x)^{\alpha-\frac{1}{2}}
\right)|_{x=b}=0.
\end{aligned}
\end{equation*}

The others eigenvalues come from the forms of Type II, III, IV, and E and 
the proofs are similar. The forms of type O require more care. Using Lemma \ref{applem2}, we have that
\begin{equation*}
\psi^{(q)}_{O}(x,\la)=(0, A_{\pm} h(x)^{-\al_{q-2}-\frac{1}{2}} (h(x)^{\al_{q-1}-\frac{1}{2}} \L_{|\alpha_{q-1}|,\alpha_{q-1}, \pm}(x,\lambda))'  \tilde\vv_{\rm har}^{(q-1)}).
\end{equation*}

Now we consider the boundary conditions at $x=a$ and $x=b$. In the boundary condition at $x=a$, we use the relation
\begin{equation*}
\begin{aligned}
(h(x)^{-2\al_{q-1}+1} (h(x)^{\al_{q-1}-\frac{1}{2}} \L_{|\alpha_{q-1}|,\alpha_{q-1}, \pm}(x,\lambda))' )'
=-\la h(x)^{\al_{q-1}+\frac{1}{2}} \L_{|\alpha_{q-1}|,\alpha_{q-1}, \pm}(x,\lambda),
\end{aligned}
\end{equation*}
and combining this with the boundary condition at $x=b$, we have the following system
\begin{equation*}
\begin{cases}
A \L_{|\alpha_{q-1}|,\alpha_{q-1}, +}(a,\lambda)+  B\L_{|\alpha_{q-1}|,\alpha_{q-1}, -}(a,\lambda)= 0 \\
A(h(x)^{\al_{q-1}-\frac{1}{2}} \L_{|\alpha_{q-1}|,\alpha_{q-1}, +}(x,\lambda))'|_{x=b} + B  (h(x)^{\al_{q-1}-\frac{1}{2}} \L_{|\alpha_{q-1}|,\alpha_{q-1}, -}(x,\lambda))'|_{x=b} = 0.
\end{cases}
\end{equation*}

Solving the system in $\lambda$ gives the $ \hat{f}_{|\al_{q-1}|,\al_{q-1},k}(a,b)$. 
\end{proof}

\section{The torsion zeta function}
\label{torsionzeta}

In this section we introduce the main objects of our work, the torsion zeta function and the analytic torsion. In the first subsection we decompose the torsion zeta function of the cone and of the frustum in some simpler zeta functions that may be tackle directly by the zeta regularisation techniques introduced in Section \ref{backss}. In the second subsection, we give a decomposition of the analytic torsion that clarifies its geometric structure, and we compare the two decompositions.

The {\it torsion zeta function} is defined by
\begin{equation*}
t_{C_{a,b}^m,{\rm BC}} (s)= \frac{1}{2} \sum_{q=1}^m (-1)^q q
\zeta(s,\Delta^{(q)}_{\rm BC}),
\end{equation*}
and the {\it analytic torsion} is
\begin{equation*}
T_{\rm bc}(C^{m}_{a,b}) = \exp (t'_{C_{a,b}^m,{\rm BC}} (0))
\end{equation*}

\begin{theo}[Duality]\label{dt}
Let $(W,g)$ be an oriented closed Riemannian manifold of dimension $m$ then
\[
\log T_{\rm abs, \pf} (C_{0,l} W) = (-1)^m \log T_{\rm rel, \pf^c} (C_{0,l} W),
\]
where $\pf = \mf$ or $\pf = \mf^c$.
\end{theo}

\begin{proof}
Using the fact that $\tilde \la_{q,n} = \tilde \la_{(m+1-q)-2,n}$ and $\al_q = -\al_{(m+1-q)-2}$ we obtain that $\mu_{q,n} = \mu_{(m+1-q)-2,n}$, for all $n$. 
Then, it is easy to see by Propositions \ref{l4} and \ref{l4.1} that 
\[
\Sp_+ \Delta^{q}_{\rm abs,\pf} = \Sp_+ \Delta^{m+1-q}_{\rm rel,\pf^c}.
\]

The eigenvalues are associated to the eigenforms described in the Lemma \ref{l2b}. By Remark \ref{r1}, 
the forms of type $I$, $III$, and $E$ are coexact and the forms of type $II$, $IV$ and $O$ are exact. 
The operator $\d$ sends forms of type $I$, $III$ and $E$ in forms of type $II$, $IV$ and $O$, respectively, 
while the operator $\d^\dag$ acting as inverse. Consider $F^{(q)}_{\rm ccl, rel, \pf}$ and $F^{(q)}_{\rm cl, rel,\pf}$ being the  subsets of $\Sp_+ \Delta^{(q)}_{\rm rel,\pf}$ consisting of the 
eigenvalues associated to co-closed and closed eigenforms, respectively, satisfying the relative  and the $\pf$ boundary  conditions. Then we have
\begin{equation*}
\begin{aligned}
t_{C_{0,l}W, \rm abs,\pf}(s) &= \frac{1}{2}\sum_{q=0}^{m+1} (-1)^q q \zeta(s,\Delta^{(q)}_{\rm abs,\pf})= \frac{1}{2}\sum_{q=0}^{m+1} (-1)^q q \zeta(s,\Delta^{(m+1-q)}_{\rm rel,\pf^c})\\
&= \frac{1}{2}\sum_{q=0}^{m+1} (-1)^{m+1-q} (m+1 - q) \zeta(s,\Delta^{(q)}_{\rm rel,\pf^c})\\
&=(-1)^m t_{C_{0,l}W, \rm rel,\pf^c}(s)   +(-1)^{m+1} \frac{m+1 }{2}\sum_{q=0}^{m+1} (-1)^{q}  \zeta(s,\Delta^{(q)}_{\rm rel,\pf^c})\\
&=(-1)^m t_{C_{0,l}W, \rm rel,\pf^c}(s)   +(-1)^{m+1} \frac{m+1 }{2}\sum_{q=0}^{m+1} (-1)^{q}  (\zeta(s,F^{(q+1)}_{\rm ccl, rel,\pf^c})+\zeta(s,F^{(q)}_{\rm cl, rel,\pf^c}))\\
&=(-1)^m t_{C_{0,l}W, \rm rel,\pf^c}(s), 
\end{aligned}
\end{equation*}
since $F^{(q)}_{\rm ccl, rel,\pf^c} = F^{(q+1)}_{\rm cl, rel,\pf^c}$.
\end{proof}

\subsection{Decomposition of the torsion zeta function}
\label{dec1}

\begin{prop}\label{ppp}
The torsion zeta function of the deformed cone with absolute boundary condition at $l$ is:
\begin{equation*}
t_{C_{0,l}^m,{\rm abs,\pf}} (s) = t^{(m)}_0(s) + t^{(m)}_1(s) + t^{(m)}_{2}(s) + t^{(m)}_{3,{\pf}}(s),
\end{equation*}
where
\begin{equation*}
\begin{aligned}
t^{(m)}_0(s)=&\frac{1}{2}\sum_{q=0}^{[\frac{m}{2}]-1}(-1)^{q}
\left((\mathcal{Z}_{q,-}(s) - \hat{\mathcal{Z}}_{q,+}(s)) 
 +(-1)^{m-1} 
(\mathcal{Z}_{q,+}(s)-\hat{\mathcal{Z}}_{q,-}(s))\right)\\
t^{(2p-1)}_1(s)=&(-1)^{p-1}\frac{1}{2} 
\left(\mathcal{Z}_{p-1,-}(s)-\hat{\mathcal{Z}}_{p-1,+}(s)\right),\qquad 
t^{(2p)}_1(s) = 0;\\
t^{(m)}_{2}(s)=&\frac{1}{2}\sum_{q=0}^{[\frac{m-1}{2}]} (-1)^{q+1} 
\left( z_{q-1,-} (s)+(-1)^{m} z_{q,+} (s)\right),\\
t^{(2p-1)}_3(s)=&0, \quad\quad \quad\quad t^{(2p)}_{3,{\mf^c}}(s)=\frac{(-1)^{p+1}}{2}   \zeta_{-} (s),\\
\qquad& \qquad\quad\quad \quad t^{(2p)}_{3,{\mf}}(s)=\frac{(-1)^{p+1}}{2}  \zeta_{+} (s),
\end{aligned}
\end{equation*}
and
\begin{align*}
\mathcal{Z}_{q,\pm}(s) &= \sum_{n,k=1}^\infty m_{{\rm cex},q,n}\ell^{-s}_{\mu_{q,n},\pm\alpha_q,k},&
\hat{\mathcal{Z}}_{q,\pm}(s) &= \sum_{n,k=1}^\infty m_{{\rm cex},q,n}\hat \ell^{-s}_{\mu_{q,n},\pm\alpha_q,k},\\
z_{q,\pm}(s)&=\sum_{k=1}^{\infty} m_{{\rm har},q}\ell^{-s}_{|\alpha_q|,\pm\alpha_q,k},&
\zeta_{\pm}(s)&=\sum_{k=1}^{\infty} m_{{\rm har},p}\ell^{-s}_{\frac{1}{2},\frac{1}{2},\pm,k}.
\end{align*}
\end{prop}

\begin{proof} 	By Lemma \ref{l4} at level $q$ we have the following eigenvalues that come from the co-exact eigenforms 
\begin{equation*}
\begin{aligned}
&\left\{m_{{\rm cex},q,n} : \hat 
\ell_{\mu_{q,n},\alpha_q,k}\right\}_{n,k=1}^{\infty}
\cup
\left\{m_{{\rm cex},q-1,n} : \hat 
\ell_{\mu_{q-1,n},\alpha_{q-1},k}\right\}_{n,k=1}^{\infty} \\
&\cup \left\{m_{{\rm cex},q-1,n} : 
\ell_{\mu_{q-1,n},-\alpha_{q-1},k}\right\}_{n,k=1}^{\infty} \cup 
\left\{m_{{\rm cex},q-2,n} :
\ell_{\mu_{q-2,n},-\alpha_{q-2},k}\right\}_{n,k=1}^{\infty},
\end{aligned}
\end{equation*}
therefore in $t_{C_{a,b}^m} (s)$ we have
\begin{equation*}
\begin{aligned}
& m_{{\rm cex},q,n}  \left( q  (-1)^q \hat 
\ell^{-s}_{\mu_{q,n},\alpha_q,k} +(q+1)  (-1)^{q+1} 
\hat \ell^{-s}_{\mu_{q,n},\alpha_q,k}\right)\\
&+ m_{{\rm cex},q,n}\left((q+1)  (-1)^{q+1} 
\ell^{-s}_{\mu_{q,n},-\alpha_{q},k}+(q+2)  (-1)^{q+2}
\ell^{-s}_{\mu_{q,n},-\alpha_q,k}\right) \\
&=m_{{\rm cex},q,n} 
 (-1)^{q}\left( \ell^{-s}_{\mu_{q,n},-\alpha_q,k} - 
\hat\ell^{-s}_{\mu_{q,n},\alpha_q,k} \right).
\end{aligned}
\end{equation*}

By Hodge duality on the coexact forms on the section, we obtain:
\begin{equation*}
\begin{aligned}
m_{{\rm cex},q,n} &= m_{{\rm cex},m-1-q,n},\qquad &\alpha_q &= 
-\alpha_{m-1-q}\\
\tilde \lambda_{q,n} &= \tilde{\lambda}_{m-1-q,n},\qquad  &\mu_{q,n} 
&= \mu_{m-1-q,n}\\
\ell_{\mu_{q,n},\alpha_q,n} &= \ell_{\mu_{m-1-q},-\alpha_{m-1-q},n},\qquad &
\hat \ell_{\mu_{q,n},\alpha_q} &= \hat \ell_{\mu_{m-1-q},-\alpha_{m-1-q},n}.
\end{aligned}
\end{equation*}

Using this information we have
\begin{equation*}
\begin{aligned}
\sum_{q=0}^{m-1}& (-1)^{q} \sum_{n,k=1}^\infty m_{{\rm cex},q,n} 
 \left( \ell^{-s}_{\mu_{q,n},-\alpha_q,k} - 
\hat\ell^{-s}_{\mu_{q,n},\alpha_q,k} \right) \\
=& 
\frac{1}{2}\sum_{q=0}^{[\frac{m}{2}]-1} (-1)^q \sum_{n,k=1}^\infty 
m_{{\rm cex},q,n}
\left( \left( \ell^{-s}_{\mu_{q,n},-\alpha_q,k} - 
\hat\ell^{-s}_{\mu_{q,n},\alpha_q,k} \right) + (-1)^{m-1}\left(
\ell^{-s}_{\mu_{q,n},\alpha_q,k} - 
\hat\ell^{-s}_{\mu_{q,n},-\alpha_q,k} \right)\right)\\
&+\left(\text{term with}\ q=\frac{m-1}{2}
\right).
\end{aligned}
\end{equation*}

The last term appear only if $\frac{m-1}{2}$ is an integer, i.e., $m=2p-1$ is odd. As $q=p-1$ and $m=2p-1$ we obtain $\alpha_q =0$, then
\begin{equation*}
\left(\text{term with}\ q=\frac{m-1}{2}
\right) = 
(-1)^{p-1}\frac{m_{{\rm cex},p-1,n}}{2} \sum_{n,k=1}^\infty \left( 
\ell^{-2s}_{\mu_{p-1,n},0,k} - 
\hat\ell^{-2s}_{\mu_{p-1,n},0,k} \right).
\end{equation*}

Next we pass to consider the terms coming from harmonic forms of the section. Collecting the eigenvalues coming from harmonic forms of the section as given in the spectrum (Lemma \ref{l4}), and proceeding as above, we find for $t_2^{(m)}(s)$ and $t_3^{(m)}(s)$,
\begin{align*}
\sum_{\stackrel{q=0}{q\not = \frac{m}{2}}}^{m}(-1)^q m_{{\rm har},q}&
\left(\sum_{k=1}^\infty q \ \ell^{-s}_{|\al_{q-1}|,-\al_{q-1},k} - (q+1) \ \ell^{-s}_{|\al_{q-1}|,-\al_{q-1},k}\right) +\left(\text{term with}\ q=\frac{m}{2}\right)=\\
&= \sum_{\stackrel{q=0}{q\not = \frac{m}{2}}}^{m}(-1)^{q+1} m_{{\rm har},q}
\sum_{k=1}^\infty \  \ell^{-s}_{|\al_{q-1}|,-\al_{q-1},k} + \left(\text{term with}\ q=\frac{m}{2}\right),
\end{align*}
where the last term appear only if $q=\frac{m}{2}$ is an integer, i.e., $m=2p$ is even.
In this case, $q = p$, $\al_q = \frac{1}{2}$. In the $\mf^c$ case we have
\begin{align*}
\left(\text{term with}\ q=\frac{m}{2}\right) &= (-1)^{p} m_{{\rm har},p} \sum_{k=1}^\infty \left(p \ \ell^{-s}_{\frac{1}{2},\frac{1}{2},-,k}  - (p+1) \ \ell^{-s}_{\frac{1}{2},\frac{1}{2},-,k} \right)\\
&= (-1)^{p+1} m_{{\rm har},p} \sum_{k=1}^\infty \ \ell^{-2s}_{\frac{1}{2},\frac{1}{2},-,k} =2t_{3,\mf^c}^{(2p)}(s),
\end{align*}
and in the $\mf$ case we obtain
\begin{align*}
\left(\text{term with}\ q=\frac{m}{2}\right) &= (-1)^{p} m_{{\rm har},p} \sum_{k=1}^\infty \left(p \ \ell^{-s}_{\frac{1}{2},\frac{1}{2},+,k}  - (p+1) \ \ell^{-s}_{\frac{1}{2},\frac{1}{2},-,k} \right)\\
&= (-1)^{p+1} m_{{\rm har},p} \sum_{k=1}^\infty \ \ell^{-2s}_{\frac{1}{2},\frac{1}{2},+,k} =2t_{3,\mf}^{(2p)}(s).
\end{align*}

Back to the other term, it is convenient to distinguish the odd and the even case. Note that in both cases $\left[\frac{m-1}{2}\right] = p-1$. If $m=2p-1$, then 
\begin{equation*}
\begin{aligned}
& \sum_{q=0}^{2p-1}(-1)^{q+1} m_{{\rm har},q}
\sum_{k=1}^\infty \  \ell^{-s}_{|\al_{q-1}|,-\al_{q-1},k} =\\
&=\sum_{q=0}^{p-1}(-1)^{q+1} m_{{\rm har},q}
\left(\sum_{k=1}^\infty  
\ell^{-s}_{|\al_{q-1}|,-\al_{q-1},k}\right) + \sum_{q=p}^{2p-1}(-1)^{q+1} m_{{\rm har},q}
\left(\sum_{k=1}^\infty 
\ell^{-s}_{|\al_{q-1}|,-\al_{q-1},k}\right) \\
&=\sum_{q=0}^{p-1}(-1)^{q+1} m_{{\rm har},q}
\left(\sum_{k=1}^\infty 
\ell^{-s}_{|\al_{q-1}|,-\al_{q-1},k}\right)
- \sum_{q=0}^{p-1}(-1)^{q+1} m_{{\rm har},2p-1-q}
\left(\sum_{k=1}^\infty 
\ell^{-2s}_{|\al_{2p-1-q-1}|,-\al_{2p-1-q-1},k}\right) \\
&=\sum_{q=0}^{p-1}(-1)^{q+1} m_{{\rm har},q}
\left(\sum_{k=1}^\infty 
\ell^{-s}_{|\al_{q-1}|,-\al_{q-1},k}\right)
- \sum_{q=0}^{p-1}(-1)^{q+1} m_{{\rm har},q}
\left(\sum_{k=1}^\infty 
\ell^{-s}_{|\al_{q}|,\al_{q},k}\right).
\end{aligned}
\end{equation*}

If $m=2p$, we already separated $t_3$, so we just need to consider the remaining term. We have 
\begin{equation*}
\begin{aligned}
 \sum_{\stackrel{q=0}{q\not=p}}^{2p}&(-1)^{q+1} m_{{\rm har},q}
\sum_{k=1}^\infty \ \ell^{-s}_{|\al_q|,\al_q,k} \\
&=\sum_{q=0}^{p-1}(-1)^{q+1} m_{{\rm har},q} \left(\sum_{k=1}^\infty   \ell^{-s}_{|\al_{q-1}|,-\al_{q-1},k}\right) 
+ \sum_{q=p+1}^{2p}(-1)^{q+1} m_{{\rm har},q}\left(\sum_{k=1}^\infty  \ell^{-s}_{|\al_{q-1}|,-\al_{q-1},k}\right) \\
&=\sum_{q=0}^{p-1}(-1)^{q+1} m_{{\rm har},q} \left(\sum_{k=1}^\infty  \ell^{-s}_{|\al_{q-1}|,-\al_{q-1},k}\right) 
+ \sum_{q=0}^{p-1}(-1)^{q+1} m_{{\rm har},2p-q}\left(\sum_{k=1}^\infty  \ell^{-s}_{|\al_{2p-q-1}|,-\al_{2p-q-1},k}\right) \\
&=\sum_{q=0}^{p-1}(-1)^{q+1} m_{{\rm har},q} \left(\sum_{k=1}^\infty   \ell^{-s}_{|\al_{q-1}|,-\al_{q-1},k}\right) 
+ \sum_{q=0}^{p-1}(-1)^{q+1} m_{{\rm har},q}\left(\sum_{k=1}^\infty  \ell^{-s}_{|\al_{q}|,\al_{q},k}\right).
\end{aligned}
\end{equation*}


Concluding the proof we make a remark, that may be useful in the following.  
Observe that in the previous formulas the index $q$ has the range $0\leq q\leq p-1$, where $m$ is either $2p-1$ or $2p$, $p\geq 1$. Then, $\al_q$ varies in the range $1-p\leq \al_q\leq 0$, if $m=2p-1$, and in the range $\frac{1}{2}-p\leq \al_q\leq -\frac{1}{2}$, if $m=2p$,  and hence $\al_q$ is always non positive. On the other side, 
$-\al_{q-1}=1-\al_q$ varies in the range $1\leq 1-\al_q\leq p$, if $m=2p-1$, and in the range $\frac{1}{2}\leq 1-\al_q\leq \frac{1}{2}+p$, if $m=2p$, and hence $-\al_{q-1}$ is always greater than $\frac{1}{2}$ beside the singular instance of $-\al_{p-1}=\frac{1}{2}$ that appears in $t_3$, and having been already treated will be disconsidered here. 
\end{proof}

\begin{prop} \label{p5.3}
The torsion zeta function of the deformed frustum with relative boundary condition at $\b_1$ and absolute boundary condition at $\b_2$ is ($l_2>l_1>0$):
\begin{equation*}
t_{C_{l_1,l_2(W)}^m,{\rm mix}} (s) = w^{(m)}_0(s) + w^{(m)}_1(s) + w^{(m)}_2(s)+w^{(m)}_3(s), 
\end{equation*}
where
\begin{equation*}
\begin{aligned}
w^{(m)}_0(s)&=\frac{1}{2}\sum_{q=0}^{[\frac{m}{2}]-1}(-1)^{q}
\left(\hat {\mathcal{D}}_{q,-}(s;b,a)-\hat {\mathcal{D}}_{q,+}(s;a,b)
+(-1)^{m-1} \left(\hat {\mathcal{D}}_{q,+}(s;b,a)-\hat {\mathcal{D}}_{q,-}(s;a,b)\right)
\right)\\
w^{(2p-1)}_1(s)&= (-1)^{p-1} \frac{1}{2} (\hat{\mathcal{D}}_{p-1,-}(s;b,a) - \hat{\mathcal{D}}_{p-1,+}(s;a,b)), \qquad 
w^{(2p)}_1(s) = 0,\\
w^{(m)}_2(s)&=\frac{1}{2}\sum_{q=0}^{[\frac{m-1}{2}]} (-1)^{q+1} 
 \left(\hat{d}_{q,+}(s;a,b)+(-1)^{m} \hat{d}_{q-1,-}(s;a,b)\right),\\
w^{(2p-1)}_3(s)&=0, \qquad \qquad w^{(2p)}_3(s)=\frac{(-1)^{p}}{2} 
 \hat{d}_{p,+} (s;a,b),
\end{aligned}
\end{equation*}
and

\begin{equation*}
\hat{\mathcal{D}}_{q,\pm}(s;a,b) = \sum_{n,k=1}^\infty m_{{\rm cex},q,n}\hat f^{-s}_{\mu_{q,n},\pm\alpha_q,k}(a,b), \qquad 
\hat d_{q,\pm}(s;a,b) = \sum_{n,k=1}^\infty m_{{\rm har},q}\hat f^{-s}_{|\alpha_q|,\pm\alpha_q,k}(a,b)
\end{equation*}
\end{prop}
\begin{proof} Recalling   the definition of the the torsion zeta function, first we collect the eigenvalues associated to  the coexact forms as given in Lemma \ref{l4}, and we find that
\begin{equation*}
w^{(m)}_0(s)+w^{(m)}_1(s)=\sum_{q=0}^{m-1}(-1)^{q}\left( \hat{\mathcal{D}}_{q,-}(s;b,a)) -\hat{\mathcal{D}}_{q,+}(s;a,b) \right).
\end{equation*}

Then we use the duality as follows.  Since $\alpha_{m-q-1} = -\alpha_q$ and $\tilde\lambda_{q,n}=\tilde\lambda_{m-q-1,n}$ we have the following equality
\begin{align*}
\lf_{\mu_{m-1-q,n},\al_{m-1-q}} &= -\frac{d^2}{dx^2} 
+\frac{h''(x)}{h(x)}\left(-\alpha_q+\frac{1}{2}\right) + 
\frac{\mu_{q,n}^2-\frac{1}{4}+\left(\alpha_{q}^2-\frac{1}{4}\right)(h'(x)^2-1)}{h(x)^2}\\
&=\lf_{\mu_{q,n},\al_q},
\end{align*}
and hence
\begin{equation*}\scriptsize
\begin{aligned}
\hat{F}_{\mu_{m-1-q},\pm\alpha_{m-1-q}}(a,b;\la) &=\L_{\mu_{m-1-q},\pm\alpha_{m-1-q},-}(a,\lambda) \b_x 
(\L_{\mu_{m-1-q},\pm\al_{m-1-q},+}(x,\lambda) h(x)^{\alpha_{m-1-q}-\frac{1}{2}})|_{x=b} \\
&\qquad-\L_{\mu_{m-1-q},\pm\al_{m-1+q},+}(a,\lambda)\b_x (\L_{\mu_{m-1+q},\pm\al_{m-1+q},-}(x,\lambda)  h(x)^{\alpha_{m-1-q}-\frac{1}{2}})|_{x=b}\\
&=\L_{\mu_{q},\mp\alpha_{q},-}(a,\lambda) \b_x (\L_{\mu_{q},\mp\al_{q},+}(x,\lambda) h(x)^{-\alpha_{q}-\frac{1}{2}})|_{x=b} \\
&\qquad-\L_{\mu_{q},\mp\al_{q},+}(a,\lambda)\b_x (\L_{\mu_{q},\mp\al_{q},-}(x,\lambda)  h(x)^{-\alpha_{q}-\frac{1}{2}})|_{x=b}\\
&= \hat{F}_{\mu_{q},\mp\alpha_{q}}(a,b;\la).
\end{aligned}
\end{equation*}

Next,  we collect the eigenvalues with multiplicity given by the harmonic forms:
\begin{equation*}
w^{(m)}_2(s)+w^{(m)}_3(s)=\frac{1}{2}\sum_{q=0}^{m}(-1)^{q+1} \hat{d}_{q,+}(s;a,b).
\end{equation*}

Remember that the $\hat f_{|\al|,\al,k}(a,b)$ are the zeros of 
\begin{equation*}
\begin{aligned}
\hat F_{|\al|,\al}(a,b,\la)= 
\L&_{|\al|,\al,-}(a,\la)\b_x(h(x)^{\al-\frac{1}{2}} \L_{|\al|,\al,+}(x,\la))|_{x=b}\\ 
&- \L_{|\al|,\al,+}(a,\la)\b_x(h(x)^{\al-\frac{1}{2}} \L_{|\al|,\al,-}(x,\la))|_{x=b}.
\end{aligned}
\end{equation*}

Using  duality on the harmonic forms,  $\alpha_{m-q} = -\alpha_{q-1}$, and hence
\begin{equation*}
\hat F_{|\al_{m-q}|,\al_{m-q}}(a,b,\la)= \hat F_{|\al_{q-1}|,-\al_{q-1}}(a,b,\la).
\end{equation*}

Then, if $m=2p-1$,
\begin{equation*}
\begin{aligned}
w^{(2p-1)}_2(s)&=\frac{1}{2}\sum_{q=0}^{2p-1}(-1)^{q+1}  \hat{d}_{q,+}(s;a,b)\\
&=\frac{1}{2}\sum_{q=0}^{p-1}(-1)^{q+1}  \hat{d}_{q,+}(s;a,b) + \frac{1}{2}\sum_{q=p}^{2p-1}(-1)^{q+1} \hat{d}_{q,+}(s;a,b)\\
&=\frac{1}{2}\sum_{q=0}^{p-1}(-1)^{q+1} \hat{d}_{q,+}(s;a,b) + \frac{1}{2}\sum_{q=0}^{p-1}(-1)^{2p-1-q+1}  \hat{d}_{2p-1-q,+}(s;a,b)\\
&=\frac{1}{2}\sum_{q=0}^{p-1}(-1)^{q+1}(\hat{d}_{q,+}(s;a,b) +(-1)^{2p-1}  \hat{d}_{q-1,-}(s;a,b)),
\end{aligned}
\end{equation*}
if $m=2p$, 
\begin{equation*}
\begin{aligned}
w^{(2p)}_2(s)+w^{(2p)}_3(s)&=\frac{1}{2}\sum_{q=0}^{2p}(-1)^{q+1} \hat{d}_{q,+}(s;a,b)\\
=&\frac{1}{2}\sum_{q=0}^{p}(-1)^{q+1}  \hat{d}_{q,+}(s;a,b) + \frac{1}{2}\sum_{q=p+1}^{2p}(-1)^{q+1}  \hat{d}_{q,+}(s;a,b)\\
=&\frac{1}{2}\sum_{q=0}^{p}(-1)^{q+1}  \hat{d}_{q,+}(s;a,b) + \frac{1}{2}\sum_{q=0}^{p-1}(-1)^{2p-q+1}  \hat{d}_{2p-q,+}(s;a,b)\\
=&\frac{1}{2}\sum_{q=0}^{p-1}(-1)^{q+1}   (\hat{d}_{q,+}(s;a,b)+(-1)^{2p} \hat{d}_{q-1,-}(s;a,b))\\
& +(-1)^{p+1}  \hat{d}_{p,+}(s;a,b),
\end{aligned}
\end{equation*}
and  the result follows.
\end{proof}

In the decompositions given in the two propositions above, the first two terms (those with indices 0 and 1) are zeta functions associated to double sequences (spectrally decomposable), the last two (those with indices 2 and 3) are zeta functions associated to simple sequences (of spectral type).  We consider the double sequences here. According to the Spectral Decomposition Lemma (Theorem \ref{sdl}), the derivative at zero of these zeta functions decomposes into a regular and a singular contribution, that give  the  following decompositions of the analytic torsion:
{\small
\begin{align*}
\log T_{\rm abs,\mf / \mf^c}(C^{m}_{0,l}) =& t'_{C_{0,l}^m,{\rm abs, \mf / \mf^c}} (0) \\
=& {t^{(m)}_{0,{\rm reg}}}'(0) +{t^{(m)}_{0,{\rm sing}}}'(0)+{t^{(m)}_{1,{\rm reg}}}'(0)+{t^{(m)}_{1,{\rm sing}}}'(0)+ {t^{(m)}_2}'(0) + {t^{(m)}_{3,\rm low/up}}'(0),\\
\log T_{\rm rel, abs}(C^{m}_{l_1,l_2}) =& t'_{C_{l_1,l_2}^m,{\rm rel, abs}} (0) \\
=& {w^{(m)}_{0,{\rm reg}}}'(0) +{w^{(m)}_{0,{\rm sing}}}'(0)+{w^{(m)}_{1,{\rm reg}}}'(0)+{w^{(m)}_{1,{\rm sing}}}'(0)+ {w^{(m)}_2}'(0) + {w^{(m)}_3}'(0).
\end{align*}}
We study $\log T_{\rm abs, \mf / \mf^c}(C^{m}_{0,l}) $ in Section \ref{cone} and $\log T_{\rm rel, abs}(C^{m}_{l_1,l_2})$ in Sections \ref{calcfrust} and \ref{Sing}.

\subsection{Decomposition of the analytic torsion}
\label{decompo}

For a manifold $M$ with boundary $\b M$, we have the decomposition
\[
\log T_{\rm BC}(M)=\log T_{\rm BC, global}(M)+\log T_{\rm BC, bound}(M),
\]
where
\begin{align*}
\log T_{\rm BC, global}(M)&=\left\{\begin{array}{ll}\log \tau(M,\b M), & {\rm BC=rel},\\\log \tau(M), & {\rm BC=abs},\end{array}\right.,\\
\log T_{\rm BC, bound}(M)&=\frac{1}{4}\chi(\b M)\log 2+A_{\rm BM, BC}(\b M),
\end{align*}
where $\tau(M,\b M)$ is the Reidemeister torsion of the pair $(M,\b M)$ with the Ray and Singer homology basis, 
$\chi(M)$ denotes the Euler characteristic of $M$, and $A_{\rm BM, BC}(\b M)$ the anomaly boundary term.

\begin{prop}\label{mainfrustum} In the notation introduced above, we have that ($l_2>l_1>0$):
\begin{align*}
\log T_{\rm global, rel, abs}(C^{m}_{l_1,l_2})+\frac{1}{4}\chi(\b C^{m}_{l_1,l_2})&= {w^{(m)}_{0,{\rm reg}}}'(0) +{w^{(m)}_{1,{\rm reg}}}'(0)+ {w^{(m)}_2}'(0) + {w^{(m)}_3}'(0),\\
\log T_{\rm bound, rel, abs}(C^{m}_{l_1,l_2})&= {w^{(m)}_{0,{\rm sing}}}'(0)+{w^{(m)}_{1,{\rm sing}}}'(0).
\end{align*}
\end{prop}
\begin{proof} See Section \ref{calcfrust}. \end{proof}

The zeta functions $w_0$ and $w_2$ involve double series and, tackled by the Spectral Decomposition Lemma, originate two types of contribution, the regular and the singular ones. The zeta functions $w_3$ and $w_4$ involve simple series and originate just one type of contribution. It happens that the global part (or interior part) of the torsion is given precisely by the regular contributions, while the boundary part is given by the singular contribution.

This suggests to introduce a similar decomposition for the cone, namely
\begin{align*}
\log T_{\rm BC}(C^{m}_{0,l})&=\log T_{\rm global, BC}(C^{m}_{0,l})+\log T_{\rm bound, BC}(C^{m}_{0,l})\\
&=\log T_{\rm global, BC}(C^{m}_{0,l})+\frac{1}{4}\chi(\b C^{m}_{0,l})\log 2+A_{\rm BM, BC}(\b C^{m}_{0,l}),
\end{align*}
where
\begin{align*}
\log T_{\rm global,  abs,\pf}(C^{m}_{0,l})+\frac{1}{4}\chi(\b C^{m}_{0,l})\log 2&= {t^{(m)}_{0,{\rm reg}}}'(0) +{t^{(m)}_{1,{\rm reg}}}'(0)+ {t^{(m)}_2}'(0) + {t^{(m)}_{3,\pf}}'(0),\\
\log T_{\rm bound,  abs,\pf}(C^{m}_{0,l})&= {t^{(m)}_{0,{\rm sing}}}'(0)+{t^{(m)}_{1,{\rm sing}}}'(0).
\end{align*}

We proceed next to understand the different contributions in the previous formulas, thus obtaining a characterization of the analytic torsion with respect to the contribution coming from Spectral Decomposition Lemma. of the cone. We consider the cone in Section \ref{cone}, and the frustum in Section \ref{calcfrust}.

\section{The analytic torsion of a deformed cone}
\label{cone}

In this section we consider the analytic torsion of the cone according to  the decomposition  given in  the previous section (see Proposition \ref{ppp}), 
\begin{align*}
\log T_{\rm abs,\pf}(C^{m}_{0,l}) =& t'_{C_{0,l}^m,{\rm abs, \pf}} (0) \\
=& {t^{(m)}_{0,{\rm reg}}}'(0) +{t^{(m)}_{0,{\rm sing}}}'(0)+{t^{(m)}_{1,{\rm reg}}}'(0)+{t^{(m)}_{1,{\rm sing}}}'(0)+ {t^{(m)}_2}'(0) + {t^{(m)}_{3,\pf}}'(0),
\end{align*}
and we proceed to compute the different contributions.  
We  split the calculations into two main parts: in the first in Section \ref{double}, we consider the terms $t_0$ and $t_1$, in the second in Section \ref{simple}, we consider the terms  $t_2$ and $t_3$.

\subsection{The contributions involving double series}
\label{double}

In this section we study the zeta functions (see  Proposition \ref{ppp}) 
\begin{align*}
t^{(m)}_0(s)=&\frac{1}{2}\sum_{q=0}^{[\frac{m}{2}]-1}(-1)^{q}
\left((\mathcal{Z}_{q,-}(s) - \hat{\mathcal{Z}}_{q,+}(s)) 
 +(-1)^{m-1} 
(\mathcal{Z}_{q,+}(s)-\hat{\mathcal{Z}}_{q,-}(s))\right),\\
t^{(2p-1)}_1(s)=&(-1)^{p-1}\frac{1}{2} 
\left(\mathcal{Z}_{p-1,-}(s)-\hat{\mathcal{Z}}_{p-1,+}(s)\right),\qquad 
t^{(2p)}_1(s) = 0.
\end{align*}

As observed, these zeta functions involve double series and are dealt with the methods introduced in Section \ref{backss}. We proceed in several steps, as here described. In Section \ref{3s1}, we describe in some details the double sequences and the relevant zeta functions. In Section \ref{s6.1}, we verify the hypothesis of Theorem \ref{sdl}, and we give the values of the different parameters. We then decompose the relevant zeta functions into regular and singular part  according Theorem \ref{sdl}. In Section \ref{reg} we compute the regular part, and in Section \ref{sing} the singular part.

\subsubsection{The relevant zeta functions}
\label{3s1}

We study the zeta functions
\begin{align*}
\mathcal{Z}_{q,\pm}(s) &= \sum_{n,k=1}^\infty m_{{\rm cex},q,n}\ell^{-2s}_{\mu_{q,n},\pm\al_q,k},& \hat{\mathcal{Z}}_{q,\pm}(s) &= 
\sum_{n,k=1}^\infty m_{{\rm cex},q,n}\hat \ell^{-2s}_{\mu_{q,n},\pm\alpha_q,k}.
\end{align*}
associated respectively to the double sequence
\begin{align*}
 S_{q,\pm}&=\left\{m_{{\rm cex},q,n}\ : \  \ell_{\mu_{q,n},\pm\al_q,k}\right\},&\hat S_{q,\pm}&=\left\{m_{{\rm cex},q,n}\ : \  \hat\ell_{\mu_{q,n},\pm\alpha_q,k}\right\}.
\end{align*}

In order to apply the tools described in Section \ref{ss2}, we need better characterisation of the numbers in the two sequences. By Lemma \ref{l4}, the numbers $\ell_{\mu_{q,n},\pm\al_q,k}$  are the  zeros of the function
\begin{align*}
&\L^{(q)}_{\mu_{q,n},\pm\al_q,+}(l,\la),
\end{align*}
and the numbers $\hat \ell_{\mu_{q,n},\pm\al_q,k}$ are the zeros  of the functions
\[
\left(\al_q-\frac{1}{2}\right)h'(l)\L^{(q)}_{\mu_{q,n},\pm\al_q,+}(l,\la)+\left. h(l)\b_x\L^{(q)}_{\mu_{q,n},\pm\al_q,+}(l,\la)\right|_{x=l}.
\]


The   function $\L^{(q)}_{\mu_{q,n},\pm\al_q,+}(x,\la)$ is the square integrable solution of the equation
\[
(\mathfrak{L}^{(q)}_{\mu_{q,n},\pm\al_q} -\la)u(x,\la) =0,
\]
where $\mathfrak{L}_{\mu_{q,n},\pm\al_q }$ is the formal operator 
\[
\mathfrak{L}_{\mu_{q,n},\pm\al_q }=-\frac{d^2}{dx^2}+q_{\mu_{q,n},\pm\al_q }(x),
\]
with
\begin{align*}
q_{\mu_{q,n},\pm\al_q }(x)
&= \frac{\mu^2_{q,n}-\frac{1}{4}+\left(\alpha^2_{q}-\frac{1}{4}\right)((h'(x))^2-1)}{h(x)^2}-\frac{h''(x)}{h(x)}\left(\pm\alpha_{q}-\frac{1}{2}\right),
\end{align*}
and
\[
\mu_{q,n}=\sqrt{\tilde\la_{q,n}+\al_q^2}.
\]

The formal operator $\mathfrak{L}_{\mu,\al}$ is of the formal operator $\lf_{\nu,\al}$ considered in  Section \ref{ss1.1}, with
$\nu=\mu_{q,n}=\sqrt{\tilde\la_{q,n}+\al_q^2}$. Whence, we may consider the concrete linear operators $L_{\mu_{q,n},\al_q, \rm rel}$, and $L_{\mu_{q,n},\al_q, \rm rel, +}$ on  $L^2[0,l]$ defined by  the  boundary condition at $x=l$:
\begin{align*}
\BB_l:&&u(l)=0,
\end{align*}
and the  boundary condition at $x=0$:
\begin{align*}
BC_{\rm +}(0)(u):&\hspace{60pt} \left\{
\begin{array}{cc} BV_{\nu,+}(0)(u)=0,&{\rm if~}\nu<1,\\
{\rm none,}&{\rm if}~\nu\geq 1,\end{array}\right.
\end{align*}
and the operators $L_{\mu_{q,n},\al_q, \rm abs}$, and $L_{\mu_{q,n},\al_q, \rm abs, +}$  with   the  boundary condition at $x=l$:
\begin{align*}
\hat\BB_l^\pm:&&\left(\pm\alpha_{q}-\frac{1}{2}\right)h'(l)u(l)+h(l)u'(l)=0,
\end{align*}
and the same boundary conditions at $x=0$. The numbers $\ell_{\mu_{q,n},\pm\al_q,k}$ are the positive eigenvalues of the operator $L_{\mu_{q,n},\al_q, \rm rel, +}$, and the numbers $\hat \ell_{\mu_{q,n},\pm\al_q,k}$  are the positive eigenvalues of the operator $L_{\mu_{q,n},\al_q, \rm abs, +}$. For simplicity we will write $L_{\mu_{q,n},\al_q, \rm bc, +}$ implying that if $\nu\geq 1$ the boundary condition at $x=0$ is irrelevant.

These operators are particular instances of the operators defined at the end of Section \ref{bv}. Note that in the following analysis $\nu\geq 1$ in most of the cases. Also note that only the $+$ solution $\L_{\mu_{q,n},\pm\al_q,+}$ of the relevant Sturm Liouville equation appears, and therefore in the whole of Section \ref{cone} it is not necessary to distinguish the case when the roots of the indicial equation coincide.

Recall that the functions  $\L_{\mu_{q,n},\al_q,\pm}(x,\la)$ are  the two linearly independent solutions of the eigenvalue equation, see Definition \ref{defiL}, 
\begin{align*}
\left(-\frac{d^2}{dx^2}  +\frac{\mu^2_{q,n}-\frac{1}{4}}{h(x)^2}+\frac{\left(\alpha^2_{q}-\frac{1}{4}\right)((h'(x))^2-1)}{h(x)^2}-\frac{h''(x)}{h(x)}\left(\alpha_{q}-\frac{1}{2}\right)-\la\right) u(x,\la)=0.
\end{align*}

\subsubsection{Spectral decomposition}
\label{s6.1}

We start by verifying the conditions necessary to apply the spectral decomposition theorem.
First, recall that  $\tilde S_q=\left\{ m_{q,n}~:~\mu_{q,n}=\sqrt{\tilde\la_{q,n}+\al_q^2}\right\}$  is a totally regular simple sequence of spectral type with infinite order, exponent of convergence and genus $\es(\tilde S_q)=\gs(\tilde S_q)=m=\dim W$, by \cite[Proposition 3.1]{Spr9}. The associated zeta function
\[
\zeta(s,\tilde S_q)=\zeta_{\rm cex}\left(\frac{s}{2},\tilde\Delta^{(q)}\right),
\]
has possible simple poles at $s=m-h$, $h=0,2,4, \dots, m-1$, \cite[Proposition 3.2]{Spr9}.

Thus, by Lemma \ref{eigen} and its corollary, both the sequence $S_q$ and $\hat S_{q,\pm}$ are double sequences  of relative order   $\left(\frac{m+1}{2},\frac{m}{2},\frac{1}{2}\right)$ and relative genus $\left(\left[\frac{m+1}{2}\right],\left[\frac{m}{2}\right],0\right)$ see \cite[Section 3]{Spr9}.  We show that the sequences $S$ are spectrally decomposable (with power $\ka=2$) over the sequence $\tilde S_q$.  We need to consider the associated logarithmic spectral Gamma functions of the quotients, namely the functions:
\[
\log \Gamma(-\la \mu^2_{q,n},S_{q,\pm})=-\log\prod_{k=1}^\infty \left(1+\frac{-\la \mu^2_{q,n}}{\ell_{\mu_{q,n},\pm\al_q,k}}\right),
\]
and
\[
\log\Gamma(-\la \mu^2_{q,n}, \hat S_{q,\pm})=-\log\prod_{k=1}^\infty \left(1+\frac{-\la \mu^2_{q,n}}{\hat\ell_{\mu_{q,n},\pm\alpha_q,k}}\right).
\]

We need the uniform (in $\la$) asymptotic expansion of these functions for large $\tilde\la_{q,n}$. Proceeding as at the end of Section \ref{spectral sequences}, we have 
\[
\log \Gamma(-\la \mu^2_{q,n},S_{q,\pm})
=\log \mathfrak{B}_{\mu_{q,n},\pm\al_q}(l,0)+\dim\ker L_{\mu_{q,n},\pm\al_q, +}-\log \mathfrak{B}_{\mu_{q,n},\pm\al_q}(l,\la\mu_{q,n}^2),
\]
and 
\[
\log\Gamma(-\la \mu^2_{q,n}, \hat S_{q,\pm})
=\log \hat{\mathfrak{B}}_{\mu_{q,n},\pm\al_q}(l,0)+\dim\ker \hat L_{\mu_{q,n},\pm\al_q, +}-\log\hat{\mathfrak{B}}_{\mu_{q,n},,\pm\al_q}(\la \mu^2_{q,n}),
\]
where
\begin{align*}
\mathfrak{B}_{\mu_{q,n},\pm\al_q}(l,\la )&=\L_{\mu_{q,n},\pm\al_q,+}(l,\la),\\
\hat{\mathfrak{B}}_{\mu_{q,n},\pm\al_q}(l,\la)&=\left(\pm\alpha_{q}-\frac{1}{2}\right)h'(l)\L_{\mu_{q,n},\pm\al_q,+}(l,\la)+h(l)\L_{\mu_{q,n},\pm\al_q,+}'(l,\la),
\end{align*}
and
\begin{align*}
 \mathfrak{B}_{\mu_{q,n},\pm\al_q}(l,0)&=\lim_{\la\to 0} \frac{ \mathfrak{B}_{\mu_{q,n},\pm\al_q}(l,\la)}{\la^{\dim\ker L_{\mu_{q,n},\pm\al_q, {\rm rel}, +}}},\\
 \hat{\mathfrak{B}}_{\mu_{q,n},\pm\al_q}(l,0)&=\lim_{\la\to 0} \frac{ \hat{\mathfrak{B}}_{\mu_{q,n},\pm\al_q}(l,\la)}{\la^{\dim\ker \hat L_{\mu_{q,n},\pm\al_q,{\rm abs}, +}}}.
\end{align*}

Since in the present case $\tilde\la_{q,n}\not=0$, it follows that $\mu_{q,n}\not=-\al_q$, and therefore by Lemma \ref{kerL}, the kernel of the operators $L$ and $\hat L$ are trivial.

Using the expansions of the solutions for large $\nu=\mu_{q,n}$ and fixed $x=l$, given in Lemmas \ref{expnu} and \ref{expnuprimo}, we obtain the required expansion of the logarithmic Gamma functions. In details, consider first the sequence $S_{q,\pm}$:
\begin{align*}
\log\Gamma(-\la \mu^2_{q,n},  S_{q,\pm})
=&\log {\mathfrak{B}}_{\mu_{q,n},,\pm\al_q}(l,0)-\log \frac{2^{\mu_{q,n}} \Gamma(\mu_{q,n}+1)}{\sqrt{2\pi\mu_{q,n}}(-\la)^\frac{\mu_{q,n}}{2}}\\
&-\log \frac{\sqrt{h(l)}}{\left(1-\la h^2(l)\right)^\frac{1}{4}}
-\mu_{q,n} \mathlarger{\mathlarger{\int}} \frac{\sqrt{1-\la h^2(x)}}{ h(x)} dx\mathlarger{\mathlarger{\mathlarger{|}}}_{x=l}\\
&-\log \left(1+\sum_{j=1}^J U_{q,j,\pm}(l,i\sqrt{-\la}) \mu_{q,n}^{-j}+ O\left(\frac{1}{\mu_{q,n}^{J+1}}\right)\right),
\end{align*}

By Remark \ref{last}, we realise that there are not relevant logarithmic terms, and therefore $L$ can take any value, while 
the relevant terms in powers of $\mu_{q,n}$ are all negative,  $\sigma_h=m-h$, $h=0,1,2,\dots m-1$, and the functions $\phi_{\sigma_h}$ are given by the following equation
\[
\sum_{h=0}^{m-1} \phi_{q,m-h,\pm}(l,\la)\mu_{q,n}^{-h}=-\log \left(1+\sum_{j=1}^{m-1} U_{q,j,\pm}(l,z) \mu_{q,n}^{-j}\right)+ O\left(\frac{1}{\mu_{q,n}^{m}}\right).
\]

This shows that indeed the sequence $S_{q,\pm}$ is spectral decomposable on the sequence $\tilde S_q$, according to Definition \ref{spdec}. We give here the values of the parameters appearing in the definition. 
\begin{align*}
(s_0,s_1,s_2)&=\left(\frac{m+1}{2},\frac{m}{2},\frac{1}{2}\right)  ,&
(p_0,p_1,p_2)&=\left(\left[\frac{m+1}{2}\right],\left[\frac{m}{2}\right],0\right),\\
r_0&=m,&q&=m,\\
\ka&=2,&\ell&=m.
\end{align*}

Second, for the sequence $\hat S_{q,\pm}$:
\begin{align*}
\log\Gamma(-\la \mu^2_{q,n},  \hat S_{q,\pm})
=&\log \hat{\mathfrak{B}}_{\mu_{q,n},\pm\al_q}(l,0)\\
&-\log\left(\left(\pm\alpha_{q}-\frac{1}{2}\right)h'(l)\L_{\mu_{q,n},\pm\al_q,+}(l,\la\mu^2_{q,n})+h(l)\L_{\mu_{q,n},\pm\al_q,+}'(l,\la\mu^2_{q,n})\right)\\
=&\log \hat{\mathfrak{B}}_{\mu_{q,n},\pm\al_q}(0)\\
&-\log \frac{2^{\nu} \Gamma(\nu+1)}{(-\la)^\frac{\nu}{2}}\sqrt{\frac{\nu}{2\pi}}
\sqrt{h(l)}\left(1-\la h^2(l)\right)^\frac{1}{4}\e^{\nu \mathlarger{\mathlarger{\int}} \frac{\sqrt{1-\la h^2(x)}}{ h(x)} dx\mathlarger{\mathlarger{\mathlarger{|}}}_{x=l}}\\
&-\log \left(1+\sum_{j=1}^J \left(V_{q,j,\pm}(l,z)+\frac{\left(\pm\alpha_{q}-\frac{1}{2}\right)h'(l)}{\left(1-\la h^2(l)\right)^\frac{1}{4}} U_{q,j-1,\pm}(l,z)\right)\mu_{q,n}^{-j}\right)\\
&+ O\left(\frac{1}{\mu_{q,n}^{m}}\right).
\end{align*}

By Remark \ref{last}, we realise that there are not relevant logarithmic terms, and therefore $L$ can take any value, while 
the relevant terms in powers of $\mu_{q,n}$ are all negative,  $\sigma_h=m-h$, $h=0,1,2,\dots m-1$, and the functions $\phi_{\sigma_h}$ are given by the following equation
\begin{align*}
\sum_{h=0}^{m-1} \hat\phi_{q,m-h,\pm}(l,\la)\mu_{q,n}^{-h}=&-\log \left(1+\sum_{j=1}^{m-1} \left(V_{q,j,\pm}(l,z)+\frac{\left(\pm\alpha_{q}-\frac{1}{2}\right)h'(l)}{\left(1-\la h^2(l\right)^\frac{1}{4}} U_{q,j-1,\pm}(l,z)\right) \mu_{q,n}^{-j}\right)\\
&+ O\left(\frac{1}{\mu_{q,n}^{m}}\right).
\end{align*}

This shows that indeed the sequences $\hat S_{q,\pm}$ is spectral decomposable on the sequence $\tilde S_q$, according to Definition \ref{spdec}. We give here the values of the parameters appearing in the definition. 
\begin{align*}
(s_0,s_1,s_2)&=\left(\frac{m+1}{2},\frac{m}{2},\frac{1}{2}\right)  ,&
(p_0,p_1,p_2)&=\left(\left[\frac{m+1}{2}\right],\left[\frac{m}{2}\right],0\right),\\
r_0&=m,&q&=m,\\
\ka&=2,&\ell&=m.
\end{align*}


\subsubsection{The regular part}
\label{reg}

Applying  the formulas in the Theorem \ref{sdl},  we need to identify the quantities $A_{0,0}(0)$, and $A_{0,1}'(0)$. 
Before starting calculations, observe that all the coefficients $b_{\sigma_h, j,0/1}$ vanish. For observe that the expansions of the functions $\phi_{q,m-h,\pm}(l,\la)$ and $\hat \phi_{q,m-h,\pm}(l,\la)$ for large $\lambda$ have terms only with negative powers of $-\la$ and negative powers of $-\la$ times $\log (-\la)$, as follows by the asymptotic characterisation of the functions $U_j(x,i\sqrt{-\la})$ and $V_j(x,i\sqrt{-\la})$ given in Lemmas \ref{expnu} and \ref{expnuprimo}. Whence, we have the following formulas
\begin{align*}
\mathcal{Z}'_{{\rm reg},q,\pm}(0) &=-A_{0,0,q,\pm}(0)-A'_{0,1,q,\pm}(0),&
 \hat{\mathcal{Z}'}_{{\rm reg},q,+}(0)=-\hat A_{0,0,q,\pm}(0)-\hat A'_{0,1,q,\pm}(0),
\end{align*}
where 
\begin{align*}
A_{0,0,q,\pm}(s)&=\sum_{n=1}^\infty m_{{\rm cex},q,n} a_{0,0,q,\pm}\mu_{q,n}^{-2s},&a_{0,0,q,\pm}=\Rz_{\la=\infty} \log\Gamma(-\la \mu^2_{q,n},  S_{q,\pm}),\\ 
A_{0,1,q,\pm}(s)&=\sum_{n=1}^\infty m_{{\rm cex},q,n} a_{0,0,q,\pm}\mu_{q,n}^{-2s},&a_{0,1,q,\pm}=\Rz_{\la=\infty} \frac{\log\Gamma(-\la \mu^2_{q,n},  S_{q,\pm})}{\log(-\la)},\\
\hat A_{0,0,q,\pm}(s)&=\sum_{n=1}^\infty m_{{\rm cex},q,n} a_{0,0,q,\pm}\mu_{q,n}^{-2s},&\hat a_{0,0,q,\pm}=\Rz_{\la=\infty} \log\Gamma(-\la \mu^2_{q,n}, \hat S_{q,\pm}),\\
\hat A_{0,1,q,\pm}(s)&=\sum_{n=1}^\infty m_{{\rm cex},q,n} a_{0,1,q,\pm}\mu_{q,n}^{-2s},&\hat a_{0,1,q,\pm}=\Rz_{\la=\infty} \frac{\log\Gamma(-\la \mu^2_{q,n}, \hat S_{q,\pm})}{\log(-\la)}.
\end{align*}

 According to the definition in equation (\ref{fi2}),  we need the asymptotic expansion of the associate logarithmic Gamma functions for large $\la$ (see equation (\ref{form}). The logarithmic Gamma function associated to the zeta functions 
$\mathcal{Z}$ and $\hat{\mathcal{Z}}$ are respectively
\[
\log \Gamma(-\la \mu^2_{q,n},S_{q,\pm})=-\log\prod_{k=1}^\infty \left(1+\frac{-\la \mu^2_{q,n}}{\ell_{\mu_{q,n},\pm\al_q,k}}\right),
\]
and
\[
\log\Gamma(-\la \mu^2_{q,n}, \hat S_{q,\pm})=-\log\prod_{k=1}^\infty \left(1+\frac{-\la \mu^2_{q,n}}{\hat\ell_{\mu_{q,n},\pm\alpha_q,k}}\right).
\]

We need the expansions for large $\la$. Recall that
\[
\log \Gamma(-\la \mu^2_{q,n},S_{q,\pm})
=\log \mathfrak{B}_{\mu_{q,n},\pm\al_q}(l,0)
-\log \mathfrak{B}_{\mu_{q,n},\pm\al_q}(l,\la\mu_{q,n}^2),
\]
and 
\[
\log\Gamma(-\la \mu^2_{q,n}, \hat S_{q,\pm})
=\log \hat{\mathfrak{B}}_{\mu_{q,n},\pm\al_q}(l,0)
-\log\hat{\mathfrak{B}}_{\mu_{q,n},,\pm\al_q}(\la \mu^2_{q,n}),
\]
where
\begin{align*}
\mathfrak{B}_{\mu_{q,n},\al_q}(l,\la)&=\L_{\mu_{q,n},\al_q,+}(l,\la),\\
\hat{\mathfrak{B}}_{\mu_{q,n},\al_q}(l,\la)&=\left(\alpha_{q}-\frac{1}{2}\right)h'(l)\L_{\mu_{q,n},\al_q,+}(l,\la)+h(l)\L_{\mu_{q,n},\al_q,+}'(l,\la).
\end{align*}

First, we need the expansion for large $\la$ of the function $\L_{\mu,\al,+}(l,\lambda)$ and its derivative. Such expansions are given in Lemmas \ref{explambda} and \ref{explambdader}. We find
\begin{align*}
\log\mathfrak{B}_{\mu,\al}(l,\mu^2\la)=&l\mu\sqrt{\la}-\frac{1}{2}\left(\mu+\frac{1}{2}\right)\log (-\la)-\left(\mu+\frac{1}{2}\right)\log\mu+\log 2^\mu \Gamma(\mu+1)\\
&-\frac{1}{2}\log 2\pi+O(1/\sqrt{-\la}),
\end{align*}
and
\begin{align*}
\log\hat{\mathfrak{B}}_{\mu,\al}(l,\mu^2\la)=&l\mu\sqrt{\la}-\frac{1}{2}\left(\mu-\frac{1}{2}\right)\log (-\la)-\left(\mu-\frac{1}{2}\right)\log \mu+\log 2^\mu \Gamma(\mu+1)\\
&-\frac{1}{2}\log 2\pi+\log h(l)+O(1/\sqrt{-\la}).
\end{align*}


Second, in order to proceed, it is convenient to deal directly with the sum
\beq\label{diffG}
\log \Gamma(-\la \mu^2_{q,n},S_{q,\pm})
-\log\Gamma(-\la \mu^2_{q,n}, \hat S_{q,\pm}),
\eeq
that appears in the definition of $t_0^{(m)}$ (see Proposition \ref{ppp}).

In fact, it is clear that
\begin{align*}
\frac{\mathfrak{B}_{\mu,-\al}(l,\la)}{\hat{\mathfrak{B}}_{\mu,\al}(l,\la)}
&=\frac{h^{2\al-1}(l)h^{-\al-\frac{1}{2}}(l)\L_{\mu,-\al,+}(l,\la)}
{\left.\left(h^{\al-\frac{1}{2}}(x)\L_{\mu,\al,+}(x,\la)\right)'\right|_{x=l}}.
\end{align*}

Setting  $\FF_{\mu,\al,\pm}=h^{\al-\frac{1}{2}}\L_{\mu,\al,\pm}$, then
\begin{align*}
\frac{\mathfrak{B}_{\mu,-\al}(x,\la)}{\hat{\mathfrak{B}}_{\mu,\al}(x,\la)}
&=\frac{h^{2\al-1}(x)\FF_{\mu,-\al,+}(x,\la)}{\FF'_{\mu,\al,+}(x,\la)},
\end{align*}
we  compute
\[
-\log \frac{ \hat{\mathfrak{B}}_{\mu,\al}(l,0)}{{\mathfrak{B}}_{\mu,-\al}(l,0)}=
-\log\frac{\FF'_{\mu,\al,+}(l,0)}{\FF_{\mu,-\al,+}(l,0)}+(2\al-1) \log h(l).
\]

Since $b=\al^2-\mu^2$,  by the Lemma \ref{applem1} in the Appendix, 
\[
\frac{\FF'_{\mu,\al,+}(x,0)}{\FF_{\mu,-\al,+}(x,0)}= (\mu+\al)h^{2\al-1}(x).
\]
 
Therefore,
\[
-\log \frac{ \hat{\mathfrak{B}}_{\mu,\al}(l,0)}{{\mathfrak{B}}_{\mu,-\al}(l,0)}=
(2\al-1) \log h(l)-\log(\mu+\al)-\log h^{2\al-1}(l)=-\log(\mu+\al).
\]

Collecting,
\begin{align*}
a_{0,0,q,-}-\hat a_{0,0,q,+}=&\Rz_{\la=\infty} \log\Gamma(-\la \mu^2_{q,n},  S_{q,-}) -\Rz_{\la=\infty} \log\Gamma(-\la \mu^2_{q,n}, \hat S_{q,+})\\
=&\log  \frac{ {\mathfrak{B}}_{\mu_{q,n},-\al_q}(l,0)}{\hat{\mathfrak{B}}_{\mu_{q,n},\al_q}(l,0)}\\
&-\Rz_{\la=\infty} \log \mathfrak{B}_{\mu_{q,n},-\al_q}(l,\la \mu^2_{q,n})
+\Rz_{\la=\infty} \log \hat{\mathfrak{B}}_{\mu_{q,n},\al_q}(l,\la \mu^2_{q,n})\\
=&-\log(\mu_{q,n}+\al_q)+\log \mu_{q,n}+\log h(l)\\
=&-\log\left(1+\frac{\al_q}{\mu_{q,n}}\right)+\log h(l),
\end{align*}
and
\begin{align*}
a_{0,1,q,-}-\hat a_{0,1,q,+}=&\Rz_{\la=\infty} \frac{\log\Gamma(-\la \mu^2_{q,n},  S_{q,-})}{\log(-\la)} -\Rz_{\la=\infty} \frac{\log\Gamma(-\la \mu^2_{q,n}, \hat S_{q,+})}{\log (-\la)}\\
=&\frac{1}{2}.
\end{align*}
and therefore
\beq\label{ZZZ}
\begin{aligned}
\mathcal{Z}'_{{\rm reg},q,-}(0) - \hat{\mathcal{Z}'}_{{\rm reg},q,+}(0)=&\sum_{n=1}^\infty m_{{\rm cex},q,n} \left(\log\left(\mu_{q,n}+\al_q\right)-\log h(l)\right).
\end{aligned}
\eeq

\leftline{\it Calculation of the regular part of $t_0$:}

Remember that
\begin{equation*}
t^{(m)}_{0,{\rm reg}}(s)=\frac{1}{2}\sum_{q=0}^{[\frac{m}{2}]-1}(-1)^{q}
\left((\mathcal{Z}_{{\rm reg},q,-}(s) - \hat{\mathcal{Z}}_{{\rm reg},q,+}(s)) 
+(-1)^{m-1} 
(\mathcal{Z}_{{\rm reg},q,+}(s)-\hat{\mathcal{Z}}_{{\rm reg},q,-}(s))\right).
\end{equation*}

It is convenient to analyse the odd and the even case separately. 
First, suppose that $m= 2p-1$. Then,  we have

\begin{equation*}
\begin{aligned}
(t^{(2p-1)}_{0,{\rm reg}})'(0)&=\frac{1}{2}\sum_{q=0}^{p-2}(-1)^{q}
\left((\mathcal{Z}'_{{\rm reg},q,-}(0) - \hat{\mathcal{Z}}_{{\rm reg},q,+}'(0)) 
+ 
\mathcal{Z}'_{q,+}(0)-\hat{\mathcal{Z}}'_{q,-}(0)\right)\\
&= \frac{1}{2}\sum_{q=0}^{p-2}(-1)^{q} \sum_{n=1}^{\infty} m_{{\rm 
cex},q,n}\left(
\log (\mu_{q,n}+\alpha_{q}) - \log h(l) +\log (\mu_{q,n}-\alpha_{q}) - \log 
h(l)
\right)\\
&= \frac{1}{2}\sum_{q=0}^{p-2}(-1)^{q} \sum_{n=1}^{\infty} 
m_{{\rm cex},q,n}\log \tilde \la_{q,n} - \sum_{q=0}^{p-2}(-1)^{q} 
\sum_{n=1}^{\infty} m_{{\rm cex},q,n} \log h(l)\\
&= \sum_{q=0}^{p-2}(-1)^{q} 
\left(-\frac{1}{2}\zeta'(0,\tilde{\Delta}^{(q)}_{\rm cex}) - 
\zeta(0,\tilde{\Delta}^{(q)}_{\rm cex}) \ \log h(l)\right).
\end{aligned}
\end{equation*}

Second, suppose that $m=2p$. Then, we have

\begin{equation*}
\begin{aligned}
(t^{(2p)}_{0,{\rm reg}})'(0)&=\frac{1}{2}\sum_{q=0}^{p-1}(-1)^{q}
\left((\mathcal{Z}_{{\rm reg},q,-}'(0) - \hat{\mathcal{Z}}_{{\rm reg},q,+}'(0)) 
-
(\mathcal{Z}_{{\rm reg},q,+}'(0)-\hat{\mathcal{Z}}_{{\rm reg},q,-}'(0))\right)\\
&=\frac{1}{2}\sum_{q=0}^{p-1}(-1)^{q} \sum_{n=1}^{\infty} 
m_{{\rm cex},q,n}\left(\log (\mu_{q,n}+\alpha_{q}) - \log h(l) -\log 
(\mu_{q,n}-\alpha_{q}) + \log 
h(l)\right)\\
&=\frac{1}{2}\sum_{q=0}^{p-1}(-1)^{q} \sum_{n=1}^{\infty} 
m_{{\rm cex},q,n} \log \frac{\mu_{q,n}+\alpha_{q}}{\mu_{q,n}-\alpha_{q}}.
\end{aligned}
\end{equation*}

\leftline{\it Calculation of the regular part of $t_1$:}

This is a particular case of the previous one, i.e. of $t_0$. We use  equation \eqref{ZZZ} with $q=p$, and we  
observe that $\alpha_{p-1} = 0$,  since  $m=2p-1$. Then we have,
\begin{equation*}
\begin{aligned}
(t_{1,{\rm reg}}^{(2p-1)})'(0) &= (-1)^{p-1} \frac{1}{2}\left((\mathcal{Z}_{{\rm reg},p-1,-}'(0) - 
\hat{\mathcal{Z}}_{{\rm reg},p-1,+}'(0)\right)\\
&=(-1)^{p-1} \frac{1}{2}\sum_{n=1}^\infty m_{{\rm cex},p-1,n} \left(
-\log h(l) + \log \mu_{p-1,n}\right)\\
&=(-1)^{p-1} \left(-\frac{1}{4}\zeta'(0,\tilde{\Delta}^{(p-1)}_{\rm cex}) - 
\frac{1}{2 }\zeta(0,\tilde{\Delta}^{(p-1)}_{\rm cex}) \ \log h(l)\right)
\end{aligned}
\end{equation*}

\subsubsection{The singular part}
\label{sing}

See Section \ref{Sing} below.

\subsection{The contributions involving simple series}
\label{simple}

In this section we study the zeta functions (see Proposition \ref{ppp})

\begin{align}\label{T2}
t^{(m)}_{2}(s)=&\frac{1}{2}\sum_{q=0}^{[\frac{m-1}{2}]} (-1)^{q+1} 
\left( z^{-2s}_{q-1,-} (s)+(-1)^{m} z^{-2s}_{q,+} (s)\right),
\end{align}
\begin{align*}
t^{(2p-1)}_3(s)=&0, \quad\quad \quad\quad t^{(2p)}_{3,{\mf}}(s)=\frac{(-1)^{p+1}}{2}   \zeta_{-} (s),\\
\qquad& \qquad\quad\quad \quad t^{(2p)}_{3,{\mf^c}}(s)=\frac{(-1)^{p+1}}{2} \zeta_{+} (s),
\end{align*}
where
\begin{align*}
z_{q,\pm}(s)&=\sum_{k=1}^{\infty} m_{{\rm har},q}\ell^{-2s}_{|\alpha_q|,\pm\alpha_q,k},&
\zeta_{\pm}(s)&=\sum_{k=1}^{\infty} m_{{\rm har},p}\ell^{-2s}_{\frac{1}{2},\frac{1}{2},\pm,k}.
\end{align*}

Here $\zeta_{\pm}$ was defined in Remark \ref{ummejo}, and may be tackled explicitly without problem. On the other side,  $z_{q,\pm}$ is the zeta function associated to the simple sequence 
\begin{align*}
Y_{q,\pm}&=\left\{m_{{\rm har},q}\ : \ \ell_{|\alpha_q|,\pm\alpha_q,k}\right\},
\end{align*}
where by Lemma \ref{l4}, and Proposition \ref{ppp}, the numbers  $\ell_{|\al|,\al,k}$ are 
the zeros of the function $\L_{|\al|, \al,+}(l,\lambda)$, where $\L_{|\al|, \al,+}(x,\lambda)$ is the normalised $+$ solution, as defined in Definition \ref{defiL}, of the eigenvalue equation 
\[
\mathfrak{L}_{|\al|,\al} u=\la u,
\]
where $\mathfrak{L}_{|\al|,\al }$ is the formal operator 
\begin{equation*}
\mathfrak{L}_{|\al|,\al} := -\frac{d^2}{dx^2} - \left(\al -\frac{1}{2}\right) 
\frac{h''(x)}{h(x)} +\left(\al^2 - \frac{1}{4}\right) \frac{h'(x)^2}{h(x)^2}.
\end{equation*}

Observe that $\left[\frac{m-1}{2}\right]=p-1$, whence $0\leq q\leq p-1$.

The formal operator $\mathfrak{L}_{\nu,\al}$ is of the formal operator $\lf_{\nu,\al}$ considered in  Section \ref{ss1.1}, with
$\nu=|\al|$. Whence, we may consider the concrete linear operators $L_{|\al_q|,\al_q, \rm rel, +}$ on  $L^2[0,l]$ defined by  the  boundary condition at $x=l$:
\begin{align*}
\BB_l:&&u(l)=0,
\end{align*}
and  the  boundary condition at $x=0$:
\begin{align*}
BC_{\rm +}(0)(u):&\hspace{60pt} \left\{
\begin{array}{cc} BV_{\nu,+}(0)(u)=0,&{\rm if~}\nu<1,\\
{\rm none,}&{\rm if}~\nu\geq 1,\end{array}\right.
\end{align*}
and the $\ell_{|\al_q|,\al_q,k}$ are the eigenvalues of  $L_{|\al_q|,\al_q, \rm rel, +}$ (recall $0\leq q\leq p-1$). 
In all cases only  the  $+$ solution $\L_{|\al|,\pm\al,+}$ of the relevant Sturm Liouville equation appears, and therefore it is not necessary to distinguish the case where the roots of the indicial equation coincide.

As observed, the case of the $\ell_{\frac{1}{2},\frac{1}{2},\pm,k}$, that are the eigenvalues of  $L_{\frac{1}{2},\frac{1}{2}, \rm rel, \pm}$, has been treated explicitely in Remark \ref{ummejo}. 

Note that in all case the kernel is trivial, by Lemma \ref{kerL}.

By Lemma \ref{eigen} and its corollary,  the sequence  $ Y_{q,\pm}(a,b)$ has   genus 0.  We show that is a sequence of spectral type. For consider the associated logarithmic spectral Gamma function:
\[
\log\Gamma(-\la,  Y_{q,\pm})=-\log\prod_{k=1}^\infty \left(1+\frac{-\la }{\ell_{|\al_{q}|,\pm\alpha_q,k}}\right).
\]

Since
\begin{align*}
\log\Gamma(-\lambda, Y_{q,\pm})&=\ln B_{|\al_q|,\pm\al_q}(l,0 )-\log B_{|\al_q|,\pm\al_q}(l,\la ),
\end{align*}
with
\[
B_{|\al|,\al}(l,\la)=\L_{|\al|, \al,+}(l,\lambda),
\]
and  the functions $\L^{(q)}_{|\al|,\al,+}(x,\lambda)$ have asymptotic expansions for large $\la$ by the results in Section \ref{largela}, it follows that the sequence $\hat Y_{q,\pm}$ is of spectral type. Moreover, a simple calculation show that it is indeed a regular sequence of spectral type.

We want to compute the constant term in the expansion of $\log\Gamma(-\la, Y_{q,\pm})$ for large $\la$, since by Theorem \ref{teo1-1},
\begin{align*}
z_{q,\pm}'(0)&=-\Rz_{\la=+\infty} \log\Gamma(-\lambda, Y_{q,\pm})\\
&=-\ln B_{|\al_q|,\pm\al_q}(l,0 )+\Rz_{\la=+\infty}\log B_{|\al_q|,\pm\al_q}(l,\la ).
\end{align*}

The asymptotic expansion can be find in Lemmas \ref{explambda} and \ref{explambdader}; we compute for large $\lambda$:
\begin{equation*}
\log \L_{|\al|,\al,+}(x,\lambda)= \log \frac{2^{|\al| 
-\frac{1}{2}}\Gamma(1+|\al|)}{\sqrt{\pi}} - 
\frac{1}{2}\left(|\al|+\frac{1}{2}\right)\log(-\la) + l\sqrt{-\la} + 
O\left(\frac{1}{\sqrt{-\la}}\right),
\end{equation*}
and hence
\beq\label{eee1}
\Rz_{\lambda \to \infty} \log \L_{|\al|,\al,+}(l,\lambda) = \log \frac{2^{|\al| 
		-\frac{1}{2}}\Gamma(1+|\al|)}{\sqrt{\pi}}.
\eeq

Next, we compute the value of $B_{|\al|,\al}(l,0 )$.

Proceeding as in the proof of Lemma \ref{kerL}, the solution of the harmonic equation above are
\begin{align*}
u_1&=h^{\frac{1}{2}-\al},& u_2&=h^{\frac{1}{2}-\al}\int h^{2\al-1},
\end{align*}
and therefore, according to the normalisation introduced in Definition \ref{defi1}, 
\begin{align*}
\L_{|\al|,\al,+}=\uf_+&=u_2=2\al h^{\frac{1}{2}-\al}\int h^{2\al-1},&\uf_-&=u_1=h^{\frac{1}{2}-\al},& \al&>0,\\
\L_{|\al|,\al,+}=\uf_+&=h^\frac{1}{2} ,& \uf_-&=u_2=h^{\frac{1}{2}}\int h^{-1}, & \al&=0,\\
\L_{|\al|,\al,+}=\uf_+&=u_1=h^{\frac{1}{2}-\al},& \uf_-&=u_2=2\al h^{\frac{1}{2}-\al}\int h^{2\al-1},& \al&<0,
\end{align*}
and since $B_{|\al|,\al}(l,0)= \L_{|\al|,\al,+}(l,0)$, 
this gives that, if $\al>0$, then 
\beq\label{B1}
B_{|\al|,\al}(l,0)= \L_{|\al|,\al,+}(l,0)= u_2(l)=2\al h^{\frac{1}{2}-\al}(l) \int_{0}^{l} 
h(x)^{2\al-1}dx,
\eeq
if $\al\leq 0$, then
\beq\label{B2}
B_{|\al|,\al}(l,0)= \L_{|\al|,\al,+}(l,0)= u_1(l)= h^{\frac{1}{2}-\al}(l).
\eeq

We may complete the calculation of $t_2$. For observe that in its definition the index $q$ has the range $0\leq q\leq p-1$, where $m$ is either $2p-1$ or $2p$, $p\geq 1$. We need to compute the derivative at $s=0$ of $z_{q-1,-} $ and $ z_{q,+}$. First, since
\[
z_{q-1,-} (s)=\sum_{k=1}^{\infty} \ell^{-2s}_{|\alpha_{q-1}|,-\alpha_{q-1},k},
\]
and $-\al_{q-1}=1-\al_q$ varies in the range $1\leq 1-\al_q\leq p$, if $m=2p-1$, and in the range $\frac{1}{2}\leq 1-\al_q\leq \frac{1}{2}+p$, if $m=2p$, we may always apply the formula in equation (\ref{B1}) for $B_{|\al_{q-1}|,-\al_{q-1}}(l,0)$, and this together with the formula in equation (\ref{eee1}) (that holds for all $\al$) gives
\[
z'_{q-1,-}(0) = \log \frac{2^{|1-\al_q| -\frac{1}{2}}\Gamma(1+|1-\al_q|)}{2(1-\al_q)\sqrt{\pi}}h^{-\al_q-\frac{1}{2}}(l)-\log  \int_{0}^{l} h(x)^{1-2\al_q}dx,
\]
for $0\leq q\leq p-1$. Second, since
\[
z_{q,+} (s)=\sum_{k=1}^{\infty} \ell^{-2s}_{|\alpha_{q}|,\alpha_{q},k},
\]
and $\al_q$ varies in the range $1-p\leq \al_q\leq 0$, if $m=2p-1$, and in the range $\frac{1}{2}-p\leq \al_q\leq -\frac{1}{2}$, if $m=2p$,  we may always apply the formula in equation (\ref{B2}) for $B_{|\al_{q}|,\al_{q}}(l,0)$, and this together with the formula in equation (\ref{eee1}) (that holds for all $\al$) gives
\[
z'_{q,+}(0) = \log \frac{2^{|\al_q| -\frac{1}{2}}\Gamma(1+|\al_q|)}{\sqrt{\pi}} h^{\al_q-\frac{1}{2}}(l),
\]
for $0\leq q\leq p-1$.

We substitute these quantities in the formula in equation (\ref{T2}), distinguishing the even and odd cases. First, let $m=2p-1$, then
\begin{equation*}
\begin{aligned}
(t_2^{(2p-1)})'(0) &=\frac{1}{2}\sum_{q=0}^{p-1}(-1)^{q+1} m_{{\rm har},q} \log \frac{2^{\al_q-\al_{q-1}}}{(-2\al_{q-1})} \frac{\Gamma(1-\al_{q-1})}{\Gamma(1-\al_q)}\\
&-\frac{1}{2}\sum_{q=0}^{p-1}(-1)^{q+1} m_{{\rm har},q} \log h(l)^{\al_q+\al_{q-1}} \int_{0}^{l} h(x)^{-2\al_{q-1}-1} dx\\
&=\frac{1}{2}\sum_{q=0}^{p-1}(-1)^{q} m_{{\rm har},q} \log h(l)^{2\al_q-1} \int_{0}^{l} h(x)^{-2\al_{q-1}-1} dx.
\end{aligned}
\end{equation*}

Second, let $m=2p$, then
\begin{equation*}
\begin{aligned}
(t_2^{(2p)})'(0) &=\frac{1}{2}\sum_{q=0}^{p-1}(-1)^{q+1} m_{{\rm har},q} \log \frac{2^{-\al_q-\al_{q-1}-1}}{(-2\al_{q-1})} \frac{\Gamma(1-\al_{q-1})\Gamma(1-\al_q)}{\pi}\\
&+\frac{1}{2}\sum_{q=0}^{p-1}(-1)^{q+1} m_{{\rm har},q} \log h(l)^{\al_q-\al_{q-1}-1} \int_{0}^{l} h(x)^{-2\al_{q-1}-1} dx\\
&=\frac{1}{2}\sum_{q=0}^{p-1}(-1)^{q+1} m_{{\rm har},q} \log \frac{ ((2p-2q-1)!!)^2}{2^2}\int_{0}^{l} h(x)^{-2\al_{q-1}-1} dx.
\end{aligned}
\end{equation*}

We conclude by computing $t_3$. In this case, we need the derivative of the zeta functions (recall that $\al_p=1/2$)
\[
\zeta_{\pm}(s)=\sum_{k=1}^{\infty} \ell^{-2s}_{|\frac{1}{2}|,\frac{1}{2},\pm,k}.
\]
that have been described in Remark \ref{ummejo}, we obtain:
\begin{align*}
\zeta_+'(0)&=-\log 2l,& \zeta_-'(0)&=-\log 2,
\end{align*}
and hence
\begin{align*}
(t^{(2p)}_{3,\mf^c})'(0)&=\frac{(-1)^{p}}{2} m_{{\rm har},p}\log 2,\\
(t^{(2p)}_{3,\mf})'(0)&=\frac{(-1)^{p}}{2} m_{{\rm har},p}\log 2l.
\end{align*}

\subsection{The global part of the torsion}

We may now collect the regular part coming from the double series computed in Section \ref{reg}, i.e. that of $t_0$ and $t_1$, and the part coming from the simple series computed in Section \ref{simple}, i.e. that of $t_2$ and $t_3$. This gives the total regular or global  part of the torsion, as defined in Section \ref{decompo}
\begin{align*}
\log T_{\rm global,  abs}(C^{m}_{0,l})+\frac{1}{4}\chi(\b C^{m}_{0,l})&= {t^{(m)}_{0,{\rm reg}}}'(0) +{t^{(m)}_{1,{\rm reg}}}'(0)+ {t^{(m)}_2}'(0) + {t^{(m)}_3}'(0).
\end{align*}

It is convenient to distinguish the odd and even cases. 

\subsubsection{Odd case} If $m=2p-1$, summing up the different contributions, we obtain

\begin{equation*}
\begin{aligned}
\log T_{\rm global,  abs, \pf}(C^{m}_{0,l})+\frac{1}{4}\chi(\b C^{m}_{0,l})
=&{t^{2p-1}_{0,{\rm reg}}}'(0)+ {t^{2p-1}_{1,{\rm reg}}}'(0)+{t^{2p-1}_2}'(0)\\
=& 
\sum_{q=0}^{p-2}(-1)^{q+1} \frac{1}{2}\zeta'(0,\tilde{\Delta}^{(q)}_{\rm 
cex}) 
+ (-1)^{p} \frac{1}{4}\zeta'(0,\tilde{\Delta}^{(p-1)}_{\rm cex}) \\
&+\log h(l) \left(\sum_{q=0}^{p-2}(-1)^{q+1} 
\zeta(0,\tilde{\Delta}^{(q)}_{\rm cex})+(-1)^{p}
\frac{1}{2}\zeta(0,\tilde{\Delta}^{(p-1)}_{\rm cex})\right)\\
&+\frac{1}{2}\sum_{q=0}^{p-1}(-1)^{q} m_{{\rm har},q} 
\log h(l)^{2\al_q-1} \int_{0}^{l} h(x)^{-2\al_{q-1}-1} dx.
\end{aligned}
\end{equation*}

By \cite[equation (A.4)]{HS5},
\begin{equation*}
\begin{aligned}
\sum_{q=0}^{p-2}(-1)^{q+1} 
\zeta(0,\tilde{\Delta}^{(q)}_{\rm cex})+(-1)^{p}
\frac{1}{2}\zeta(0,\tilde{\Delta}^{(p-1)}_{\rm cex}) 
&=\frac{1}{2}\sum_{q=0}^{p-1} (-1)^q m_{{\rm har},q} (1-2\alpha_q),
\end{aligned}
\end{equation*}
and by \cite[equation (A.1)]{HS5},
\begin{equation*}
\log T(W,g) = \sum_{q=0}^{p-2}(-1)^{q+1} 
\zeta'(0,\tilde{\Delta}^{(q)}_{\rm 
	cex}) 
+ (-1)^{p} \frac{1}{2}\zeta'(0,\tilde{\Delta}^{(p-1)}_{\rm cex}).
\end{equation*}

Therefore we obtain that
\begin{equation*}
\log T_{\rm global,  abs, \pf}(C^{m}_{0,l})+\frac{1}{4}\chi(\b C^{m}_{0,l}) = \frac{1}{2}\log T(W,g) 
+ \frac{1}{2} \sum_{q=0}^{p-1}(-1)^{q} m_{{\rm har},q}  \log \int_{0}^{l} 
h(x)^{-2\al_{q-1}-1} dx.
\end{equation*}

\subsubsection{Even Case} If $m=2p$, we have

\[
\log T_{\rm global,  abs, \pf}(C^{m}_{0,l})+\frac{1}{4}\chi(\b C^{m}_{0,l})=
{t^{2p}_{0,{\rm reg}}}'(0)+ {t^{2p}_2}'(0)+{t^{2p}_{3,\pf}}'(0) 
\]
therefore it is convenient to separate the last term. We first compute

\begin{equation*}
\begin{aligned}
{t^{2p}_{0,{\rm reg}}}'(0)+ {t^{2p}_2}'(0)=& 
\frac{1}{2}\sum_{q=0}^{p-1}(-1)^{q} \sum_{n=1}^{\infty} 
m_{{\rm cex},q,n} \log \frac{\mu_{q,n}+\alpha_{q}}{\mu_{q,n}-\alpha_{q}}\\
&+\frac{1}{2}\sum_{q=0}^{p-1}(-1)^{q+1} m_{{\rm har},q} \log \frac{ ((2p-2q-1)!!)^2}{2^2}\int_{0}^{l} h(x)^{-2\al_{q-1}-1} dx\\
=& \frac{1}{2}\sum_{q=0}^{p-1}(-1)^{q} \sum_{n=1}^{\infty} m_{{\rm cex},q,n} \log \frac{\mu_{q,n}+\alpha_{q}}{\mu_{q,n}-\alpha_{q}}\\
&+\sum_{q=0}^{p-1}(-1)^{q+1} m_{{\rm har},q} \log (2p-2q-1)!!+\frac{1}{2} \chi(W)\log 2\\
&+\frac{1}{2}\sum_{q=0}^{p-1}(-1)^{q} m_{{\rm har},q}\log \int_{0}^{l} h(x)^{-2\al_{q-1}-1} dx
-(-1)^{p}\frac{m_{{\rm har},p}}{2} \log 2.
\end{aligned}
\end{equation*}

This gives
\begin{align*}
\log T_{\rm global,  abs, \mf^c}(C^{m}_{0,l})+\frac{1}{4}\chi(\b C^{m}_{0,l})
=&\frac{1}{2}\sum_{q=0}^{p-1}(-1)^{q} \sum_{n=1}^{\infty} m_{{\rm cex},q,n} \log \frac{\mu_{q,n}+\alpha_{q}}{\mu_{q,n}-\alpha_{q}}\\
&+\sum_{q=0}^{p-1}(-1)^{q+1} m_{{\rm har},q} \log (2p-2q-1)!!+\frac{1}{2} \chi(W)\log 2\\
&+\frac{1}{2}\sum_{q=0}^{p-1}(-1)^{q} m_{{\rm har},q}\log \int_{0}^{l} h(x)^{-2\al_{q-1}-1} dx,
\end{align*}
and
\begin{align*}
\log T_{\rm global,  abs, \mf}(C^{m}_{0,l})+\frac{1}{4}\chi(\b C^{m}_{0,l})
=&\frac{1}{2}\sum_{q=0}^{p-1}(-1)^{q} \sum_{n=1}^{\infty} m_{{\rm cex},q,n} \log \frac{\mu_{q,n}+\alpha_{q}}{\mu_{q,n}-\alpha_{q}}\\
&+\sum_{q=0}^{p-1}(-1)^{q+1} m_{{\rm har},q} \log (2p-2q-1)!!+\frac{1}{2} \chi(W)\log 2\\
&+\frac{1}{2}\sum_{q=0}^{p}(-1)^{q} m_{{\rm har},q}\log \int_{0}^{l} h(x)^{-2\al_{q-1}-1} dx.
\end{align*}


\section{The analytic torsion of a frustum}\label{calcfrust}

In this section we consider the analytic torsion of the frustum according to the decomposition  given in  the previous section (see Proposition \ref{p5.3}), 
\begin{align*}
\log T_{\rm rel, abs}(C^{m}_{l_1,l_2}(W)) =& t'_{C_{l_1,l_2}^m(W),{\rm rel, abs}} (s) \\
=& {w^{(m)}_{0,{\rm reg}}}'(0) +{w^{(m)}_{0,{\rm sing}}}'(0)+{w^{(m)}_{1,{\rm reg}}}'(0)+{w^{(m)}_{1,{\rm sing}}}'(0)+ {w^{(m)}_2}'(0) + {w^{(m)}_3}'(0),
\end{align*}
and we proceed to compute the different contributions.  
We  split the calculations into two main parts: in the first in Section \ref{doublefrust}, we consider the terms $w_0$ and $w_1$, in the second in Section \ref{simplefrust}, we consider the terms  $w_2$ and $w_3$. 

\subsection{The contributions involving double series}
\label{doublefrust}

In this section we study the zeta functions appearing in the contributions (see  Proposition \ref{p5.3}) 
\begin{align*}
w^{(m)}_0(s)&=\frac{1}{2}\sum_{q=0}^{[\frac{m}{2}]-1}(-1)^{q}
\left(\hat {\mathcal{D}}_{q,-}(s;b,a)-\hat {\mathcal{D}}_{q,+}(s;a,b)\right.\\
&\qquad \qquad \qquad \qquad \left.+(-1)^{m-1} \left(\hat {\mathcal{D}}_{q,+}(s;b,a)-\hat {\mathcal{D}}_{q,-}(s;a,b)\right)
\right),\\
w^{(2p-1)}_1(s)&= (-1)^{p-1} \frac{1}{2} \ m_{{\rm 
cex},q,n}(\hat{\mathcal{D}}^{-2s}_{p-1,-}(s;b,a) - \hat{\mathcal{D}}^{-2s}_{p-1,+}(s;a,b)), \qquad 
w^{(2p)}_1(s) = 0.
\end{align*}

As observed, these zeta functions involve double series and are dealt with the methods introduced in Section \ref{backss}. We proceed in several steps, as here described. In Section \ref{rlz1}, we describe in some details the double sequences and the relevant zeta functions. In Section \ref{specfrust}, we verify the hypothesis of Theorem \ref{sdl}, and we give the values of the different parameters. We then decompose the relevant zeta functions into regular and singular part  according Theorem \ref{sdl}. In Section \ref{regfrust} we compute the regular part, and in Section \ref{Sing} the singular part.

\subsubsection{The relevant zeta functions}
\label{rlz1}

The zeta function we want to study is
\begin{align*}
\hat {\mathcal{D}}_{q,\pm}(s;a,b) &= \sum_{n,k=1}^\infty m_{{\rm cex},q,n} \hat f^{-2s}_{\mu_{q,n},\pm\al_q,k}(a,b),
\end{align*}
associated to the double sequence
 \begin{align*}
\hat Q_{q,\pm}(a,b)&=\left\{m_{{\rm cex},q,n}\ : \  \hat f_{\mu_{q,n},\pm\al_q,k}(a,b)\right\}.
\end{align*}

In order to apply the tools described in Section \ref{ss2}, we need better characterisation of the numbers in the  sequences: 
the $\hat f_{\mu,\al,k}(a,b)$ are zeros of 
\begin{equation*}
\begin{aligned}
\L^{(q)}_{\mu,\al,+}(a,\lambda) \b_x (\L^{(q)}_{\mu,\al,-}(x,\lambda) h(x)^{\al-\frac{1}{2}})|_{x=b} -\L^{(q)}_{\mu,\al,-}(a,\lambda)\b_x (\L^{(q)}_{\mu,\al,+}(x,\lambda)  h(x)^{\al-\frac{1}{2}})|_{x=b}.
\end{aligned}
\end{equation*}
where  the   functions $\L^{(q)}_{\mu_{q,n},\pm\al_q,\pm}(x,\la)$ are the two linearly independent solutions of the equation
\[
(\mathfrak{L}^{(q)}_{\mu_{q,n},\pm\al_q} -\la)u(x,\la) =0,
\]
and $\mathfrak{L}_{\mu_{q,n},\pm\al_q }$ is the formal operator 
\[
\mathfrak{L}_{\mu_{q,n},\pm\al_q }=-\frac{d^2}{dx^2}+q_{\mu_{q,n},\pm\al_q }(x),
\]
with
\begin{align*}
q_{\mu_{q,n},\pm\al_q }(x)
&= \frac{\mu^2_{q,n}-\frac{1}{4}+\left(\alpha^2_{q}-\frac{1}{4}\right)((h'(x))^2-1)}{h(x)^2}-\frac{h''(x)}{h(x)}\left(\pm\alpha_{q}-\frac{1}{2}\right),
\end{align*}
and
\[
\mu_{q,n}=\sqrt{\tilde\la_{q,n}+\al_q^2}.
\]


The operator $\mathfrak{L}_{\nu,\al}$ is the operator $\lf_{\nu,\al}$  considered in  Section \ref{ss1.1}, with
$\nu=\mu_{q,n}=\sqrt{\tilde\la_{q,n}+\al_q^2}$. Proceeding as for the cone, we define a concrete  operator $\hat R_{\mu_{q,n},\pm\al_q}$ on  $L^2[a,b]$ by  the  boundary conditions:
 \begin{align*}
\BB_a:&&u(a)=0,\\
\hat \BB_b^\pm:&&\left(\pm\alpha_{q}-\frac{1}{2}\right)h'(b)u(b)+h(b)u'(b)=0.
\end{align*}

It is easy to verify that  these BC satisfy the condition in Section \ref{ss2}. This means that we can apply the results of Section \ref{ss2} to the operator $ R_{\mu_{q,n},\pm\al_q}=\hat R_{\mu_{q,n},\pm\al_q}$.

Observe that, differently from the case of the cone, in the present case of the frustum both the $+$ and the $-$ solutions $\L^{(q)}_{\mu_{q,n},\pm\al_q,\pm}$ of the relevant Sturm Liouville equation appear, whence for the $-$ solution the case when the roots of the indicial equation coincide need independent treatment. As described in Appendix \ref{SturmLouville}, this happens only if $\mu_{q,n}=0$, and this is possible only if $\tilde\la_{q,n}=0$, and $\mu_{q,n}=|\al_q|$. Whence, in the following two sections this problem will not occur, while it will occur in Section \ref{simplefrust}.

\subsubsection{Spectral decomposition}\label{specfrust} Proceeding as  in Section \ref{s6.1},  by Lemma \ref{eigen},  the sequence  $\hat Q_{q,\pm}(a,b)$ is a double sequence   of relative order   $\left(\frac{m+1}{2},\frac{m}{2},\frac{1}{2}\right)$ and relative genus $\left(\left[\frac{m+1}{2}\right],\left[\frac{m}{2}\right],0\right)$ see \cite{Spr9} Section 3.  We show that it is also spectrally decomposable (with power $\ka=2$) over the sequence $\tilde S_q$.  We need to consider the associated logarithmic spectral Gamma function of the quotient, namely the function:
\[
\log\Gamma(-\la \mu^2_{q,n},  \hat Q_{q,\pm}(a,b))=-\log\prod_{k=1}^\infty \left(1+\frac{-\la \mu^2_{q,n}}{\hat f_{\mu_{q,n},\pm\alpha_q,k}(a,b)}\right).
\]

We need the uniform (in $\la$) asymptotic expansion of these functions for large $\mu_{q,n}$. Writing
\begin{equation*}
\begin{aligned}
\hat{\mathcal{A}}_{\mu,\al}(a,b,\lambda) =\hat F_{\mu,\al}(a,b,\la)=&\L_{\mu,\al,+}(a,\lambda) \b_x (\L_{\mu,\al,-}(x,\lambda) h(x)^{\al-\frac{1}{2}})|_{x=b} \\
&-\L_{\mu,\al,-}(a,\lambda)\b_x (\L_{\mu,\al,+}(x,\lambda)  h(x)^{\al-\frac{1}{2}})|_{x=b},
\end{aligned}
\end{equation*}
and proceeding as at the end of  Section \ref{spectral sequences}, we have 
\begin{align*}
\log\Gamma(-\la \mu^2_{q,n}, \hat Q_{q,\pm}(a,b))
=&\log \hat{\mathcal{A}}_{\mu_{q,n},\pm\al_q}(a,b,0)+ \dim\ker \hat R_{\mu_{q,n},\pm\al_q}\\
&-\log\hat{\mathcal{A}}_{\mu_{q,n},\pm\al_q}(a,b,\la \mu_{q,n}^2).
\end{align*}
where
\begin{align*}
\log \mathcal{A}_{\mu_{q,n},\pm\al_q}(a,b,0)&=\lim_{\la\to 0} \frac{\log \mathcal{A}_{\mu_{q,n},\pm\al_q}(a,b,\la)}{\la^{\dim\ker \hat R_{\mu_{q,n},\pm\al_q}}}.
\end{align*}

However, by Lemma \ref{kerR}, the kernel of these operators is trivial, and therefore the previous expression simplify accordingly.

Proceeding as in Section \ref{s6.1}, using the expansions of the solutions for large $\nu=\mu_{q,n}$ and fixed $x=a,b$, given in Lemmas \ref{expnu} and \ref{expnuprimo}, we obtain the required expansion of the logarithmic Gamma functions. 
By Remark \ref{last}, we realise that there are not relevant logarithmic terms, and therefore $L$ can take any value, while 
the relevant terms in powers of $\mu_{q,n}$ are all negative,  with $\sigma_h=m-h$, $h=0,1,2,\dots m-1$, namely 
\[
\sum_{h=0}^{m-1} \hat\psi_{q,m-h,\pm}(a,b,\la)\mu_{q,n}^{-h}.
\]

This shows that indeed the sequences $S_q$ is spectral decomposable on the sequence $\tilde S_q$, according to Definition \ref{spdec}. We give here the values of the parameters appearing in the definition. 
\begin{align*}
(s_0,s_1,s_2)&=\left(\frac{m+1}{2},\frac{m}{2},\frac{1}{2}\right)  ,&
(p_0,p_1,p_2)&=\left(\left[\frac{m+1}{2}\right],\left[\frac{m}{2}\right],0\right),\\
r_0&=m,&q&=m,\\
\ka&=2,&\ell&=m.
\end{align*}

\subsubsection{The regular part}
\label{regfrust}

Applying  the formulas in the Theorem \ref{sdl},  we need to identify the quantities $A_{0,0}(0)$, and $A_{0,1}'(0)$. By the same argument as detailed at the beginning of Section \ref{reg}, all the $b$ coefficients vanish. So,  we have that
\begin{align*}
\hat {\mathcal{D}'}_{{\rm reg}, q,\pm}(0;a,b) 
 &=-\hat A_{0,0,q,\pm}(0;a,b)-\hat A'_{0,1,q,\pm}(0;a,b),
\end{align*}
where 
\begin{align*}
\hat A_{0,0,q,\pm}(s;a,b)&=\sum_{n=1}^\infty m_{{\rm cex},q,n} a_{0,0,q,\pm}(a,b)\mu_{q,n}^{-2s},&a_{0,0,q,\pm}=\Rz_{\la=\infty} \log\Gamma(-\la \mu^2_{q,n},  \hat Q_{q,\pm}(a,b)),\\ 
\hat A_{0,1,q,\pm}(s;a,b)&=\sum_{n=1}^\infty m_{{\rm cex},q,n} a_{0,1,q,\pm}(a,b)\mu_{q,n}^{-2s},&a_{0,1,q,\pm}=\Rz_{\la=\infty} \frac{\log\Gamma(-\la \mu^2_{q,n},  \hat Q_{q,\pm}(a,b))}{\log(-\la)}.
\end{align*}

 According to the definition in equation (\ref{fi2}),  we need the asymptotic expansion of the associated logarithmic Gamma functions for large $\la$ (see equation (\ref{form}). The logarithmic Gamma function associated to the zeta function $\hat{\mathcal{D}}$ is
\[
\log\Gamma(-\la \mu^2_{q,n}, \hat Q_{q,\pm}(a,b))=-\log\prod_{k=1}^\infty \left(1+\frac{-\la \mu^2_{q,n}}{\hat f_{\mu_{q,n},\pm\alpha_q,k}(a,b)}\right).
\]

We need the expansions for large $\la$. Recall that
\begin{align*}
\log\Gamma(-\la \mu^2_{q,n}, \hat Q_{q,\pm}(a,b))
=&\log \hat{\mathcal{A}}_{\mu_{q,n},\pm\al_q}(a,b,0)
-\log\hat{\mathcal{A}}_{\mu_{q,n},\pm\al_q}(a,b,\la \mu_{q,n}^2).
\end{align*}
where
\begin{equation*}
\begin{aligned}
\hat{\mathcal{A}}_{\mu,\al}(a,b,\lambda) =&\L_{\mu,\al,+}(a,\lambda) \b_x (\L_{\mu,\al,-}(x,\lambda) h(x)^{\al-\frac{1}{2}})|_{x=b} \\
&-\L_{\mu,\al,-}(a,\lambda)\b_x (\L_{\mu,\al,+}(x,\lambda)  h(x)^{\al-\frac{1}{2}})|_{x=b}.
\end{aligned}
\end{equation*}

First, we need the expansion for large $\la$ of the functions $\L_{\mu,\al,\pm}(a,\lambda)$ and its derivative. Recalling that these functions are particular instances of the functions $\uf_\pm$ introduced in Section \ref{ss2}, such expansions are given in Lemmas \ref{explambda} and \ref{explambdader}. In this case, however, it is more convenient to use the function $\vf$ instead of the function $\uf_-$. Since, \begin{equation*}
\L_{\mu,\al,-}(x,\la) = \frac{1}{2^{\mu}\Gamma(1+\mu)} \vf_{\mu,\al} + 
2^{-2\mu} \frac{\Gamma(1-\mu)}{\Gamma(1+\mu)} \frac{1}{(-\la)^{-\mu}} 
\L_{\mu,\al,+}(x,\la),
\end{equation*} 
we rewrite $\hat{\mathcal{A}}_{\mu,\al}$ as follows:
\begin{equation*}
\begin{aligned}
\hat{\mathcal{A}}_{\mu,\al}(a,b;\la) = \frac{1}{2^{\mu}\Gamma(1+\mu)} 
&\left(\L_{\mu,\al,+} (a,\la) \b_x (h(x)^{\al-\frac{1}{2}} \vf_{\mu,\al}(x,\la))|_{x=b}\right.  \\
&-\left.\vf_{\mu,\al}(a,\la) \b_x (h(x)^{\al-\frac{1}{2}} \L_{\mu,\al,+}(x,\la))|_{x=b}\right).
\end{aligned}
\end{equation*}

Since
\begin{equation*}
\b_x(h(x)^{\al-\frac{1}{2}}\L_{\mu,\al,+}(x,\la))|_{x=b} = h^{\al - 
\frac{1}{2}}(b) \left(\left(\al- \frac{1}{2}\right) \frac{h'(b)}{h(b)} 
\L_{\mu,\la,+}(b,\la) - \L_{\mu,\al,+}'(b,\la)\right),
\end{equation*}
using Lemmas \ref{explambda} and \ref{explambdader}, we obtain the following 
expansion for large $\la$:
\begin{equation*}
\begin{aligned}
\b_x(h(x)^{\al-\frac{1}{2}} 
\L_{\mu,\al}(x,\la))|_{x=b}&=
h(b)^{\al-\frac{1}{2}}  \frac{2^{\mu-\frac{1}{2}}\Gamma(1+\mu)}{\sqrt{\pi} \ 
(-\la)^{\frac{1}{2}(\mu-\frac{1}{2})}} \\
& \left(e^{b\sqrt{-\la}}
	\left(1-O\left(\frac{1}{\sqrt{-\la}}\right)\right)+e^{-b\sqrt{-\la}}
	\left(1-O\left(\frac{1}{\sqrt{-\la}}\right)\right)\right)
\end{aligned}
\end{equation*}

Following the same idea we have
\begin{equation*}
\begin{aligned}
\b_x(h(x)^{\al-\frac{1}{2}} 
\vf_{\mu,\al}(x,\la))|_{x=a}=
h(a)^{\al-\frac{1}{2}}  \frac{\mu \ (2\pi)^{\frac{1}{2}}}{
	(-\la)^{-\frac{1}{2}(\mu+\frac{1}{2})}} 
\ e^{-a\sqrt{-\la}}
\left(1+O\left(\frac{1}{\sqrt{-\la}}\right)\right).
\end{aligned}
\end{equation*}

Collecting, the expansion of $\hat{\mathcal{A}}_{\mu,\al}(a,b;\la)$ for large  $\la$ is
\begin{equation}\label{Eq-AsympA}
\hat{\mathcal{A}}_{\mu,\al}(a,b;\la)=
-h(b)^{\al-\frac{1}{2}} \mu 
e^{(b-a)\sqrt{-\la}} \left(1+O\left(\frac{1}{\sqrt{-\la}}\right)\right),
\end{equation}
and hence
\begin{equation*}
\begin{aligned}
\Rz_{\la=\infty} \log \hat{\mathcal{A}}_{\mu,\al}(a,b;\la) &= \log \mu +\left(\al - \frac{1}{2}\right)\log h(b)+\log (-1).
\end{aligned}
\end{equation*}

Second, we need to understand the  term $\hat{\mathcal{A}}_{\mu_{q,n},\al_q}(a,b,0)$. It is convenient to deal directly with the sum
\[
\log \hat{\mathcal{A}}_{\mu_{q,n},-\al_q}(b,a,0)-\log \hat{\mathcal{A}}_{\mu_{q,n},\al_q}(a,b,0).
\]

With   $\FF_{\mu,\al,\pm}=h^{\al-\frac{1}{2}}\L_{\mu,\al,\pm}$, 
\[
h^{\al-\frac{1}{2}}(a) \hat{\mathcal{A}}_{\mu,\al}(a,b,0)=
\FF_{\mu,\al,+}(a,0) \FF'_{\mu,\al,-}(b,0) 
-\FF_{\mu,\al,-}(a,0) \FF'_{\mu,\al,+}(b,0).
\]

By the Lemma \ref{applem1} in the Appendix, 
\[
\frac{\FF'_{\mu,\al,\pm}(x,0)}{\FF_{\mu,-\al,\pm}(x,0)}= Ah^{2\al-1}(x).
\]

Since for small $x$
\[
\frac{\FF'_{\mu,\al,\pm}(x,0)}{\FF_{\mu,-\al,\pm}(x,0)}\sim(\pm\mu+\al)x^{2\al-1},
\]
and $h(x)\sim x$, it follows that $A=\al\pm\mu$. Therefore,

\begin{align*}
h^{\al-\frac{1}{2}}(a) \hat{\mathcal{A}}_{\mu,\al}(a,x,0)=&
h^{2\al-1}(x)\left(\al\left(\FF_{\mu,\al,+}(a,0) \FF_{\mu,-\al,-}(x,0) 
-\FF_{\mu,\al,-}(a,0) \FF_{\mu,-\al,+}(x,0)\right)\right.\\
&\left.-\mu\left(\FF_{\mu,\al,+}(a,0) \FF_{\mu,-\al,-}(x,0) 
+\FF_{\mu,\al,-}(a,0) \FF_{\mu,-\al,+}(x,0)\right)\right),
\end{align*}
and a simple calculation gives
\[
\hat{\mathcal{A}}_{\mu,\al}(a,x,0) 
=h^{\al+\frac{3}{2}}(a)h^{\al-\frac{3}{2}}(x)  \hat{\mathcal{A}}_{\mu,-\al}(x,a,0),
\]
and therefore
\[
\frac{\hat{\mathcal{A}}_{\mu,\al}(a,b,0)}{ \hat{\mathcal{A}}_{\mu,-\al}(b,a,0)}
=h^{\al+\frac{3}{2}}(a)h^{\al-\frac{3}{2}}(b).
\]

Collecting,
\begin{align*}
\hat a_{0,0,q,-}(b,a)&-\hat a_{0,0,q,+}(a,b)=\\
=&\Rz_{\la=\infty} \log\Gamma(-\la 
\mu^2_{q,n},  \hat{Q}_{q,-}(b,a)) -\Rz_{\la=\infty} \log\Gamma(-\la 
\mu^2_{q,n}, 
\hat Q_{q,+}(a,b))\\
=&\log  
\frac{\hat{\mathcal{A}}_{\mu_{q,n},-\al_q}(b,a,0)}{\hat{\mathcal{A}}_{\mu_{q,n},\al_q}(a,b,0)}\\
&-\Rz_{\la=\infty} \log \hat{\mathcal{A}}_{\mu_{q,n},-\al_q}(b,a,\la 
\mu^2_{q,n})
+\Rz_{\la=\infty} \log \hat{\mathcal{A}}_{\mu_{q,n},\al_q}(a,b,\la 
\mu^2_{q,n})\\
=&-\log \frac{h(a)}{h(b)},
\end{align*}
and
\begin{equation*}
\hat a_{0,1,q,-}(b,a)=\hat a_{0,1,q,+}(a,b)= 0;
\end{equation*}
whence
\begin{align*}
\hat {\mathcal{D}'}_{{\rm reg}, q,-}(0;b,a) -\hat {\mathcal{D}'}_{{\rm reg}, q,+}(0;a,b) 
&=-\hat A_{0,0,q,-}(0;b,a)+\hat A_{0,0,q,+}(0;a,b)\\
&=\sum_{n=1}^\infty m_{{\rm cex}, q, n} 
\log \frac{h(b)}{h(a)} = \zeta_{\rm cex}(0,\tilde \Delta^{(q)}) \log 
\frac{h(b)}{h(a)}.
\end{align*}

We can then compute the regular part of the terms $(w^{(m)}_0)'(0)$ and $(w^{(m)}_1)'(0)$:
\begin{equation*}
\begin{aligned}
(w_{0,{\rm reg}}^{(m)})'(0) &= -\frac{1}{2}\log\frac{h(b)}{h(a)} 
\sum_{q=0}^{p-2} (-1)^q \left(\zeta_{\rm 
cex}(0,\tilde{\Delta}^{(q)}) + (-1)^{m-1} \zeta_{\rm 
cex}(0,\tilde{\Delta}^{(q)})\right),\\
(w_{1,{\rm reg}}^{(2p-1)})'(0) &= -\frac{1}{2}\log\frac{h(b)}{h(a)} 
(-1)^{p-1} \zeta_{\rm 
	cex}(0,\tilde{\Delta}^{(p-1)}).
\end{aligned}
\end{equation*}

\subsubsection{The singular part}

See Section \ref{Sing} below.


\subsection{The contribution involving simple series}
\label{simplefrust}

In this section we compute the contribution coming from 
\begin{equation*}
\begin{aligned}
w^{(m)}_2(s)&=\frac{1}{2}\sum_{q=0}^{[\frac{m-1}{2}]} (-1)^{q+1} 
 \left(\hat{d}_{q,+}(s;a,b)+(-1)^{m} \hat{d}_{q-1,-}(s;a,b)\right),\\
w^{(2p-1)}_3(s)&=0, \qquad \qquad w^{(2p)}_3(s)=\frac{(-1)^{p}}{2} \hat{d}_{p,+} (s;a,b),
\end{aligned}
\end{equation*}
where the relevant zeta function is 

\begin{equation*}
\hat d_{q,\pm}(s;a,b) = \sum_{n,k=1}^\infty m_{{\rm har},q}\hat f^{-2s}_{|\alpha_q|,\pm\alpha_q,k}(a,b)
\end{equation*}
introduced in Proposition \ref{p5.3}. This is the zeta function associated to the simple sequence
\begin{align*}
\hat P_{q,\pm}(a,b)&=\left\{m_{{\rm har},q}\ : \  \hat f_{|\al_q|,\pm\al_q,k}(a,b)\right\},
\end{align*}
where the $\hat f_{|\al|,\al,k}(a,b)$ are zeros of


\begin{equation*}
\begin{aligned}
\hat F^{(q)}_{|\al|,\al}(a,b,\lambda) =\L^{(q)}_{|\al|,\al,+}(a,\lambda) \b_x (\L^{(q)}_{|\al|,\al,-}(x,\lambda) &h(x)^{\al-\frac{1}{2}})|_{x=b} \\
&-\L^{(q)}_{|\al|,\al,-}(a,\lambda)\b_x (\L^{(q)}_{|\al|,\al,+}(x,\lambda)  h(x)^{\al-\frac{1}{2}})|_{x=b}.
\end{aligned}
\end{equation*}

The functions $\L^{(q)}_{|\al|,\al,\pm}(x,\lambda)$ are two l.i. solutions of the equation $\mathfrak{L}_{|\al|,\pm\al}u=\la u$.
where
\[
\mathfrak{L}_{|\al|,\al}=-\frac{d^2}{dx^2}+q_{|\al|,\al}(x),
\]
with
\begin{align*}
q_{|\al|,\al}(x)
&= \frac{\left(\alpha^2-\frac{1}{4}\right)(h'(x))^2}{h(x)^2}-\frac{h''(x)}{h(x)}\left(\alpha-\frac{1}{2}\right).
\end{align*}

In the present case, the difference of the roots of the indicial equation of the relevant Sturm Liouville equation may vanish, and indeed is zero when $\al=0$. So we may need to distinguish the cases $\al\not=0$ and $\al=0$. We will proceed our analysis as follows: we write nothing when ever the result follows independently of the value of $\al$, while we separate the case $\al=0$ whenever the result is different in that case.

The formal operator $\mathfrak{L}_{|\al|,\al}$ is the formal operator $\lf_{|\al|,\al}$ studied in Section \ref{ss1.1}, and as in that section we may define a  concrete linear operator $\hat R_{|\al|,\al}$ on  $L^2[a,b]$ by  the  boundary conditions:
 \begin{align*}
\BB_a:&&u(a)=0,\\
\hat \BB_b^\pm:&&\left(\pm\alpha_{q}-\frac{1}{2}\right)h'(b)u(b)+h(b)u'(b)=0,
\end{align*}
and this is precisely the operator $R_{|\al|,\al}$ defined at the end of Section \ref{bv}. Whence the numbers $\hat f_{|\al|,\al,k}(a,b)$ are the positive eigenvalues of $\hat R_{|\al|,\al}$, and by Lemma \ref{kerR}, the kernel of $\hat R_{|\al|,\al}$ is trivial.


By Lemma \ref{eigen} and its corollary,  the sequence  $\hat P_{q,\pm}(a,b)$ has   genus 0.  We show that is a sequence of spectral type. For consider the associated logarithmic spectral Gamma function:
\[
\log\Gamma(-\la,  \hat P_{q,\pm}(a,b))=-\log\prod_{k=1}^\infty \left(1+\frac{-\la }{\hat f_{|\al_{q}|,\pm\alpha_q,k}(a,b)}\right).
\]

Writing
\begin{equation*}
\begin{aligned}
\hat{\mathcal{A}}_{|\al|,\al}(a,b,\lambda) =&\L_{|\al|,\al,+}(a,\lambda) \b_x (\L_{|\al|,\al,-}(x,\lambda) h(x)^{\al-\frac{1}{2}})|_{x=b} \\
&-\L_{|\al|,\al,-}(a,\lambda)\b_x (\L_{|\al|,\al,+}(x,\lambda)  h(x)^{\al-\frac{1}{2}})|_{x=b},
\end{aligned}
\end{equation*}
we have 
\[
\log\Gamma(-\la, \hat P_{q,\pm}(a,b))
=\log \hat{\mathcal{A}}_{|\al_{q}|,\pm\al_q}(a,b,0)-\log\hat{\mathcal{A}}_{|\al_{q}|\pm\al_q}(a,b,\la ).
\]

Since the functions $\L^{(q)}_{|\al|,\al,\pm}(x,\lambda)$ have asymptotic expansions for large $\la$ by the results in Section \ref{largela}, it follows that the sequence $\hat P_{q,\pm}(a,b)$ is of spectral type. Moreover, a simple calculation show that it is indeed a regular sequence of spectral type.

We want to compute the constant term in the expansion of $\log\Gamma(-\la, \hat P_{q,\pm}(a,b))$ for large $\la$, since
\[
\hat d'_{q,\pm}(0;a,b)=-\Rz_{\la=\infty} \log\Gamma(-\la, \hat P_{q,\pm}(a,b)).
\]

In order to proceed, we need to distinguish the case $\al=0$. So first assume $\al\not=0$. Then, using the expansions in Lemmas \ref{explambda} and \ref{explambdader},  we compute  
\begin{equation*}
\Rz_{\la=\infty} \log \hat{A}_{|\al|,\al}(a,b,\la) = \log |\al| + \left(\al - \frac{1}{2}\right)\log h(b)+\log(-1).
\end{equation*}

Next, we compute $\hat{\mathcal{A}}_{|\al_{q}|,\pm\al_q}(a,b,0)$. The harmonic equation is
\[
\hat R_{|\al_q|,\pm\al_q}u=0,
\]
and proceeding as in the proof of Lemma \ref{kerL}, we have the following solutions
\begin{align*}
u_1&=h^{\frac{1}{2}-\al},& u_2&=h^{\frac{1}{2}-\al}\int h^{2\al-1},
\end{align*}
that according to the normalisation introduced in Definition \ref{defi1}, give
\begin{align*}
\L_{|\al|,\al,+}=\uf_+&=u_2=2\al h^{\frac{1}{2}-\al}\int h^{2\al-1},&\L_{|\al|,\al,-}&=\uf_-=u_1=h^{\frac{1}{2}-\al},& \al&>0,\\
\L_{|\al|,\al,+}=\uf_+&=h^\frac{1}{2} ,& \L_{|\al|,\al,-}&=\uf_-=u_2=h^{\frac{1}{2}}\int h^{-1}, & \al&=0,\\
\L_{|\al|,\al,+}=\uf_+&=u_1=h^{\frac{1}{2}-\al},& \L_{|\al|,\al,-}&=\uf_-=u_2=2\al h^{\frac{1}{2}-\al}\int h^{2\al-1},& \al&<0.
\end{align*}

Thus, direct substitution gives for $\al>0$:
\begin{align*}
\hat{\mathcal{A}}_{|\al|,\al}(a,b,0) =&-2\al h^{\frac{1}{2}-\al}(a)h^{2\al-1}(b),
\end{align*}
and
\begin{equation*}
\Rz_{\la=\infty}\log \hat{A}_{|\al |,\al }(a,b;\la) - 
\log (\hat{A}_{|\al |,\al }(a,b;0))=-\log 2 + \left(-\al +\frac{1}{2}\right)\log \frac{h(b)}{h(a)},
\end{equation*}
for $\al<0$:
\begin{equation*}
\begin{aligned}
\hat{A}_{|\al|,\al}(a,b;0)
&=- h(a)^{-\al + \frac{1}{2}} (-2\al) h(b)^{2\al-1},
\end{aligned}
\end{equation*}
and
\begin{equation*}
\Rz_{\la=\infty}\log \hat{A}_{|\al |,\al }(a,b;\la) - 
\log (-\hat{A}_{|\al |,\al }(a,b;0))=-\log 2 + \left(-\al +\frac{1}{2}\right)\log \frac{h(b)}{h(a)},
\end{equation*}
also when $\al<0$.

Therefore, for all $\al\not=0$, 
\[
\hat d'_{q,\pm}(0;a,b)=-\Rz_{\la=\infty} \log\Gamma(-\la, \hat P_{q,\pm}(a,b))=-\log 2 + \left(-\al +\frac{1}{2}\right)\log \frac{h(b)}{h(a)}.
\]

It remains to tackle the case $\al=0$. Using the expansions in Lemmas \ref{explambda} and \ref{explambdader} with $\nu=0$,    we compute  
\begin{equation*}
\Rz_{\la=\infty} \log \hat{A}_{0,0}(a,b,\la) =-\log 2 - \frac{1}{2}\log h(b)+\log(-1).
\end{equation*}

Using the solutions described above we compute directly
\begin{equation*}
\begin{aligned}
\hat{A}_{0,0}(a,b;0)&=h(a)^{\frac{1}{2}} h(b)^{-1}.
\end{aligned}
\end{equation*}

Altogether,
\begin{equation*}
\Rz_{\la=\infty}\log \hat{A}_{0,0 }(a,b;\la) - \log (\hat{A}_{0,0 }(a,b;0))=-\log 2+\frac{1}{2} \log \frac{h(b)}{h(a)}.
\end{equation*}

We can now complete the calculations of $w_2$, using the previous formulas for the different ranges of $\al$ to compute $\hat d'_{q,\pm}(0;a,b)$. It is convenient to distinguish the odd and the even cases. First, if $m=2p-1$, then $\al_q = 0$, if $q = p-1$, while $\al_{q-1}\not=0$ for all $q$, therefore:
\begin{equation*}
\begin{aligned}
(w_{2}^{(2p-1)})'(0) &= \frac{1}{2}\sum_{q=0}^{p-1}(-1)^{q+1}m_{{\rm har},q}  (\hat{d}'_{q,+}(0;a,b) +(-1)^{2p-1}  \hat{d}'_{q-1,-}(0;a,b))\\
=& \frac{1}{2}\sum_{q=0}^{p-2}(-1)^{q+1} m_{{\rm har},q} \left( \left(-\al_q -\al_{q-1}\right)\log \frac{h(b)}{h(a)}\right) +  (-1)^{p} m_{{\rm har},p-1} \frac{1}{2}\log \frac{h(b)}{h(a)}\\
=& \frac{1}{2}\log \frac{h(b)}{h(a)}\sum_{q=0}^{p-1}(-1)^{q+1} m_{{\rm har},q} \left(-2\al_q +1\right).
\end{aligned}
\end{equation*}

Second, if $m=2p$, $\al_q$ and $\al_{q-1}$ never vanish, and we have
\begin{equation*}
\begin{aligned}
(w^{(2p)}_2)'(0)+(w^{(2p)}_3)'(0)&=\frac{1}{2}\sum_{q=0}^{p-1}(-1)^{q+1}m_{{\rm har},q}  (\hat{d}_{q,+}(s;a,b)+(-1)^{2p} \hat{d}_{q-1,-}(s;a,b)) \\
&+(-1)^{p+1}m_{{\rm har},p}  \hat{d}_{p,+}(s;a,b)\\
=& \frac{1}{2}\sum_{q=0}^{p-1}(-1)^{q+1} m_{{\rm har},q} \left(-2\log 2 + \left(-\al_q +\al_{q-1}+1\right)\log \frac{h(b)}{h(a)}\right)\\
& +  (-1)^{p} m_{{\rm har},q} \frac{1}{2}\log 2 \\
=& \frac{1}{2}\left(\sum_{q=0}^{p-1}(-1)^{q} m_{{\rm har},q} 2\log 2  +  (-1)^{p} m_{{\rm har},q} \log 2 \right)\\
=& \frac{1}{2} \chi(W) \log 2.
\end{aligned}
\end{equation*}

\subsection{The global part of the torsion}
\label{globfrustum}

We may now collect the regular part coming from the double series computed in Section \ref{doublefrust}, i.e. that of $w_0$ and $w_1$, and the part coming from the simple series computed in Section \ref{simplefrust}. This gives the  global  part of the torsion, as defined in Section \ref{decompo}
\begin{align*}
\log T_{\rm global,  abs}(C^{m}_{l_1,l_2}(W))+\frac{1}{4}\chi(\b C^{m}_{l_1,l_2}(W))\log 2&= {w^{(m)}_{0,{\rm reg}}}'(0) +{w^{(m)}_{1,{\rm reg}}}'(0)\\
&\qquad+ {w^{(m)}_2}'(0) + {w^{(m)}_3}'(0).
\end{align*}

It is convenient to distinguish the odd and even cases. 

\subsubsection{Odd case} If $m=2p-1$, we have

\begin{align*}
\log & T_{\rm global,  abs}(C^{2p-1}_{l_1,l_2}(W))+\frac{1}{4}\chi(\b C^{2p-1}_{l_1,l_2}(W))\log 2= {w^{(2p-1)}_{0,{\rm reg}}}'(0) +{w^{(2p-1)}_{1,{\rm reg}}}'(0)+ {w^{(2p-1)}_2}'(0)\\
=&-\log\frac{h(b)}{h(a)} \left(
\sum_{q=0}^{p-2} (-1)^q  \zeta_{\rm cex}(0,\tilde{\Delta}^{(q)})
+\frac{1}{2}(-1)^{p-1} \zeta_{\rm cex}(0,\tilde{\Delta}^{(p-1)})\right)+ {w^{(2p-1)}_2}'(0)\\
=&\frac{1}{2}\log\frac{h(b)}{h(a)} \sum_{q=0}^{p-1}(-1)^{q+1} m_{{\rm har},q} (2\al_q-1)+\frac{1}{2}\log \frac{h(b)}{h(a)}\sum_{q=0}^{p-1}(-1)^{q+1} m_{{\rm har},q} \left(-2\al_q +1\right)\\
=&0.
\end{align*}

\subsubsection{Even case} If $m=2p$, we have

\begin{align*}
\log T_{\rm global,  abs}(C^{2p}_{l_1,l_2}(W))+\frac{1}{4}\chi(\b C^{2p}_{l_1,l_2}(W))\log 2=& {w^{(2p)}_{0,{\rm reg}}}'(0) + {w^{(2p)}_2}'(0)+{w^{(2p)}_{3}}'(0)\\
=&\frac{1}{2} \chi(W) \log 2.
\end{align*}

\subsection{Proof of Proposition \ref{mainfrustum}} 

The results obtained in the previous sections allow us to complete the proof of Proposition \ref{mainfrustum}. For being the frustum a smooth manifold with boundary, the classical decomposition given in Section \ref{decompo} for the torsion holds. 
Therefore,
\begin{align*}
\log T_{\rm rel,abs}(C^{(m)}_{[l_1,l_2]}(W))=&\log T_{\rm global, rel,abs}(C^{(m)}_{[l_1,l_2]}(W))+\log T_{\rm boundary,rel,abs}(C^{(m)}_{[l_1,l_2]}(W)),
\end{align*}
where
\begin{align*}
\log T_{\rm global,rel,abs}(C^{(m)}_{[l_1,l_2]}(W))=&\log \tau(C^{(m)}_{[l_1,l_2]}(W),\{l_1\}\times W)+\frac{1}{4}\chi (\b C^{(m)}_{[l_1,l_2]}(W))\log 2.
\end{align*}

Since an explicit calculation of the Reidemeister torsion (with mixed BC) gives immediately the same quantities as computed in Section \ref{globfrustum} for the global part of the analytic torsion, this proves the first formula in the statement, namely that
\begin{align*}
\log T_{\rm global,rel,abs}(C^{(m)}_{[l_1,l_2]}(W))=&\log \tau(C^{(m)}_{[l_1,l_2]}(W),\{l_1\}\times W)+\frac{1}{4}\chi (\b C^{(m)}_{[l_1,l_2]}(W))\log 2\\
=&{w^{(m)}_{0,{\rm reg}}}'(0)+{w^{(m)}_{1,{\rm reg}}}'(0)+{w^{(m)}_{2}}'(0)+{w^{(m)}_{3}}'(0),
\end{align*}
and consequently
\[
\log T_{\rm boundary,rel,abs}(C^{(m)}_{[l_1,l_2]}(W))=A_{\rm BM,mixed}(\b C^{(m)}_{[l_1,l_2]}(W))
={w^{(m)}_{0,{\rm sing}}}'(0)+{w^{(m)}_{1,{\rm reg}}}'(0)
\]
and the second formula is proved.

\section{The singular part and the boundary contribution in the analytic torsion}
\label{Sing}

In this section we deal with the singular part of the torsion of the cone and of the frustum. More precisely, in the first subsection,  we prove a relationship between these two part, and in the second subsection, and we show that this singular part coincides with the boundary term for the cone (for the frustum this was already proved in Proposition \ref{mainfrustum}).

\subsection{A relation between the singular part of the torsion of the cone and of the frustum}

Recalling the definition in Section \ref{decompo}, the singular part for the frustum appears applying the SDL \ref{sdl} to  the zeta functions $w^{(m)}_0(s)$ and $w^{(m)}_1(s)$, the singular part for the cone applying the SDL to the zeta functions $t^{(m)}_0(s)$ and $t^{(m)}_1(s)$. We give all details for $w_0$ and $t_0$, the calculations for $w_1$ and $t_1$ are exactly the same and will be omitted. For the frustum, using the values of the parameters computed in Section \ref{specfrust}, we have
\begin{align*}
{w_{0,\rm sing}^{(m)}}'(0)=&\frac{\gamma}{2}\sum_{j=0}^{m-1}\Ru_{s=0}\Psi_{j}(s)\Ru_{s=j}\zeta(s,U)\\
&+\frac{1}{2}\sum_{j=0}^{m-1}\Rz_{s=0}\Psi_{j}(s)\Ru_{s=j}\zeta(s,U)+{\sum_{j=0}^{m-1}}{^{\displaystyle
'}}\Ru_{s=0}\Psi_{j}(s)\Rz_{s=j}\zeta(s,U),
\end{align*}
where 
\[
\Psi_{j}(s)=\int_0^\infty t^{s-1}\frac{1}{2\pi i}\int_{\Lambda_{\theta,c}}\frac{\e^{-\lambda t}}{-\lambda} \psi_{j}(\lambda) d\lambda dt,
\]
and the $\psi_j$ are the coefficients of the powers of $\mu_{q,n}$ in the expansion for large $\mu_{q,n}$ of the function
$\log\Gamma(-\la \mu^2_{q,n},  \hat Q_{q,\pm}(a,b))$ discussed in Section \ref{specfrust}.

Since, by Proposition \ref{p5.3},  
\[
w^{(m)}_0(s)=\frac{1}{2}\sum_{q=0}^{[\frac{m}{2}]-1}(-1)^{q}
\left(\hat {\mathcal{D}}_{q,-}(s;b,a)-\hat {\mathcal{D}}_{q,+}(s;a,b)
+(-1)^{m-1} \left(\hat {\mathcal{D}}_{q,+}(s;b,a)-\hat {\mathcal{D}}_{q,-}(s;a,b)\right)
\right),
\]
each $\psi_j$ is given by a combination of the corresponding coefficients $\hat\psi_{q,j,\pm}$ appearing in the singular part of the derivative of the zeta functions $\hat D_{q,\pm}$, namely define
\[
\Psi_{q,j}^{(m)}(\la):= \left(\hat\psi_{q,j,-}(l_2,l_1,\la)-\hat\psi_{q,j,+}(l_1,l_2,\la)+(-1)^{m-1} \left(\hat\psi_{q,j,+}(l_2,l_1,\la)
-\hat\psi_{q,j,-}(l_1,l_2,\la)\right)\right),
\] 
then
\begin{align*}
{w_{0,\rm sing}^{(m)}}'(0,S)=&\frac{\gamma}{2}\sum_{j=0}^{m-1}\sum_{q=0}^{[\frac{m}{2}]-1}(-1)^{q} \Ru_{s=0}\int_0^\infty t^{s-1}\frac{1}{2\pi i}\int_{\Lambda_{\theta,c}}\frac{\e^{-\lambda t}}{-\lambda} \Psi_{q,j}^{(m)}(\la) d\la dt \Ru_{s=j}\zeta(s,U)\\
&+\frac{1}{2}\sum_{j=0}^{m-1}\sum_{q=0}^{[\frac{m}{2}]-1}(-1)^{q} \Rz_{s=0}\int_0^\infty t^{s-1}\frac{1}{2\pi i}\int_{\Lambda_{\theta,c}}\frac{\e^{-\lambda t}}{-\lambda} \Psi_{q,j}^{(m)}(\la) d\la dt \Ru_{s=j}\zeta(s,U)\\
&+{\sum_{j=0}^{m-1}}{^{\displaystyle'}}\sum_{q=0}^{[\frac{m}{2}]-1}(-1)^{q}\Ru_{s=0}\int_0^\infty t^{s-1}\frac{1}{2\pi i}\int_{\Lambda_{\theta,c}}\frac{\e^{-\lambda t}}{-\lambda} \Psi_{q,j}^{(m)}(\la) d\la dt
\Rz_{s=j}\zeta(s,U)
\end{align*}

The same analysis applied to the zeta functions for the cone (see Proposition \ref{ppp})
\begin{align*}
t^{(m)}_0(s)&=\sum_{q=0}^{[\frac{m}{2}]-1}(-1)^{q}
\left((\mathcal{Z}_{q,-}(s) - \hat{\mathcal{Z}}_{q,+}(s)) +(-1)^{m-1} 
(\mathcal{Z}_{q,+}(s)-\hat{\mathcal{Z}}_{q,-}(s))\right),
\end{align*}
gives
\begin{align*}
{t^{(m)}_{0,{\rm sing}}}'(0,l)=&\frac{\gamma}{2}\sum_{j=0}^{m-1}\sum_{q=0}^{[\frac{m}{2}]-1}(-1)^{q} \Ru_{s=0}\int_0^\infty t^{s-1}\frac{1}{2\pi i}\int_{\Lambda_{\theta,c}}\frac{\e^{-\lambda t}}{-\lambda} \Phi_{q,j}^{(m)}(\la) d\la dt\Ru_{s=j}\zeta(s,U)\\
&+\frac{1}{2}\sum_{j=0}^{m-1}\sum_{q=0}^{[\frac{m}{2}]-1}(-1)^{q} \Rz_{s=0}\int_0^\infty t^{s-1}\frac{1}{2\pi i}\int_{\Lambda_{\theta,c}}\frac{\e^{-\lambda t}}{-\lambda}\Phi_{q,j}^{(m)}(\la) d\la dt \Ru_{s=j}\zeta(s,U)\\
&+{\sum_{j=0}^{m-1}}{^{\displaystyle'}}\sum_{q=0}^{[\frac{m}{2}]-1}(-1)^{q}\Ru_{s=0}\int_0^\infty t^{s-1}\frac{1}{2\pi i}\int_{\Lambda_{\theta,c}}\frac{\e^{-\lambda t}}{-\lambda}\Phi_{q,j}^{(m)}(\la) d\la dt
\Rz_{s=j}\zeta(s,U),
\end{align*}
where 
\[
\Phi_{q,j}^{(m)}(\la):= \left(- \phi_{q,j,+}(l,\la)+\hat\phi_{q,j,-}(l,\la)
+(-1)^{m-1} \left(\phi_{q,j,+}(l,\la)-\hat\phi_{q,j,-}(l,\la)\right)\right),
\]
and we inserted $l$ in the argument meaning that the cone  of length $l$ is to be considered, and the $\phi_j$ come from  the coefficients of the powers of $\mu_{q,n}$ in the expansion for large $\mu_{q,n}$ of the functions $\log\Gamma(-\la \mu^2_{q,n},   S_{q})$ and 
$\log\Gamma(-\la \mu^2_{q,n},  \hat S_{q,\pm})$ discussed in Section \ref{s6.1}.

Now, the functions $\log\Gamma(-\la \mu^2_{q,n},  \hat Q_{q,\pm}(a,b))$, $\log\Gamma(-\la \mu^2_{q,n},  \hat S_{q,\pm})$, and $\log\Gamma(-\la \mu^2_{q,n},   S_{q})$), are the logarithmic Gamma functions associated to the operators  introduced in Sections \ref{rlz1} and \ref{3s1}, respectively, that are particular instances of the  abstract operators $R_{\nu,\al}$ and $L_{\nu,\al}$  introduced in Section \ref{bv}. Therefore, we may compare the coefficients of the expansions of these two operators by means of the formula proved in Proposition \ref{p2.1} of Section \ref{s2.26}.  Adapting the notation, we 
have 
\[
\hat\psi_{q,j,\pm}(l_1,l_2,\la)=\sgn(l_1-l_2)^j \phi_{q,j,\pm}(l_1,\la)+\sgn(l_2-l_1)^j\hat\phi_{q,j,\pm}(l_2,\la),
\]
that gives
\begin{align*}
\hat\psi_{q,j,-}&(l_2,l_1,\la)-\hat\psi_{q,j,+}(l_1,l_2,\la)+(-1)^{m-1} \left(\hat\psi_{q,j,+}(l_2,l_1,\la)
-\hat\psi_{q,j,-}(l_1,l_2,\la)\right)\\
=&\phi_{q,j,-}(l_2,\la)
+(-1)^j\hat\phi_{q,j,-}(l_1,\la)
-\left((-1)^j \phi_{q,j,+}(l_1,\la)+\hat\phi_{q,j,+}(l_2,\la)\right)\\
&+(-1)^{m-1} \left( \phi_{q,j,+}(l_2,\la)+(-1)^j\hat\phi_{q,j,+}(l_1,\la)
-\left((-1)^j \phi_{q,j,-}(l_1,\la)+\hat\phi_{q,j,-}(l_2,\la)\right)\right)\\
=&(-1)^j\left(- \phi_{q,j,+}(l_1,\la)+\hat\phi_{q,j,-}(l_1,\la)
+(-1)^{m-1} \left(-\phi_{q,j,-}(l_1,\la)+\hat\phi_{q,j,+}(l_1,\la)\right)\right)\\
& \phi_{q,j,-}(l_2,\la)-\hat\phi_{q,j,+}(l_2,\la)
+(-1)^{m-1} \left(\phi_{q,j,+}(l_2,\la)-\hat\phi_{q,j,-}(l_2,\la)\right).
\end{align*}

It is convenient to distinguish odd and even cases. First, assume $m=2p-1$ is odd, then
\begin{align*}
\hat\psi_{q,j,-}(l_2,l_1,\la)&-\hat\psi_{q,j,+}(l_1,l_2,\la)+\hat\psi_{q,j,+}(l_2,l_1,\la)
-\hat\psi_{q,j,-}(l_1,l_2,\la)\\
=&(-1)^j\left(- \phi_{q,j,+}(l_1,\la)+\hat\phi_{q,j,-}(l_1,\la)
+\hat\phi_{q,j,+}(l_1,\la)-\phi_{q,j,-}(l_1,\la)\right)\\
&- \phi_{q,j,+}(l_2,\la)+\hat\phi_{q,j,-}(l_2,\la)
+\hat\phi_{q,j,+}(l_2,\la)-\phi_{q,j,-}(l_2,\la);
\end{align*}
moreover, since the possible poles of the zeta function $\zeta(s,U)$ are for $s=1,3,5,\dots, 2p-1$, 
\begin{align*}
{w_{0,\rm sing}^{(2p-1)}}'(0,S)=&\frac{\gamma}{2}\sum_{k=0}^{p-1}\sum_{q=0}^{p-1}(-1)^{q} \Ru_{s=0}\int_0^\infty t^{s-1}\frac{1}{2\pi i}\int_{\Lambda_{\theta,c}}\frac{\e^{-\lambda t}}{-\lambda}\Psi_{q,2k+1}^{2p-1}(\la) d\la dt \Ru_{s=2k+1}\zeta(s,U)\\
&+\frac{1}{2}\sum_{k=0}^{p-1}\sum_{q=0}^{p-1}(-1)^{q} \Rz_{s=0}\int_0^\infty t^{s-1}\frac{1}{2\pi i}\int_{\Lambda_{\theta,c}}\frac{\e^{-\lambda t}}{-\lambda} \Psi_{q,2k+1}^{2p-1}(\la) d\la dt \Ru_{s=2k+1}\zeta(s,U)\\
&+{\sum_{k=0}^{p-1}}{^{\displaystyle'}}\sum_{q=0}^{p-1}\Ru_{s=0}\int_0^\infty t^{s-1}\frac{1}{2\pi i}\int_{\Lambda_{\theta,c}}\frac{\e^{-\lambda t}}{-\lambda} \Psi_{q,2k+1}^{2p-1}(\la)d\la dt 
\Rz_{s=2k+1}\zeta(s,U),
\end{align*}
since $j=2k+1$ is always odd, substitution of the previous formula gives
\[
w^{(2p-1)}_{0,{\rm sing}}(s)= t^{(2p-1)}_{0,{\rm sing}}(s,l_1)+t^{(2p-1)}_{0,{\rm sing}}(s,l_2).
\]

Second, assume $m=2p$ even, then
\begin{align*}
\hat\psi_{q,j,-}(l_2,l_1,\la)&-\hat\psi_{q,j,+}(l_1,l_2,\la)+(-1)^{m-1} \left(\hat\psi_{q,j,+}(l_2,l_1,\la)
-\hat\psi_{q,j,-}(l_1,l_2,\la)\right)\\
=&(-1)^j\left(- \phi_{q,j,+}(l_1,\la)+\hat\phi_{q,j,-}(l_1,\la)
-\phi_{q,j,-}(l_1,\la)+\hat\phi_{q,j,+}(l_1,\la)\right)\\
& +\phi_{q,j,-}(l_2,\la)-\hat\phi_{q,j,+}(l_2,\la)
-\phi_{q,j,+}(l_2,\la)+\hat\phi_{q,j,-}(l_2,\la).
\end{align*}

Now  the possible poles of the zeta function $\zeta(s,U)$ are for $s=2,4,6,\dots, 2p-2$, and hence $j=2k$ is always even. Proceeding as above we obtain that
\[
w^{(2p)}_{0,{\rm sing}}(s)= t^{(2p-1)}_{0,{\rm sing}}(s,l_1)+t^{(2p-1)}_{0,{\rm sing}}(s,l_2),
\]

In conclusion, we have proved that, for all $m$, 
\begin{align*}
w^{(m)}_{0,{\rm sing}}(s)&= t^{(m)}_{0,{\rm sing}}(s,l_1)+t^{(m)}_{0,{\rm sing}}(s,l_2),\\
w^{(m)}_{1,{\rm sing}}(s)&= t^{(m)}_{1,{\rm sing}}(s,l_1)+t^{(m)}_{1,{\rm sing}}(s,l_2)
\end{align*}

\subsection{The boundary part of the torsion of the cone}
\label{s7.2}

The Br\"{u}ning-Ma anomaly boundary term for the frustum is
\[
A_{\rm BM,mixed}(\b C^{(m)}_{[l_1,l_2]}(W))=\int_{\{l_1\}\times W} B+\int_{\{l_2\}\times W} B,
\]
where $B$ is a form construct starting from the metric of the frustum (see for example \cite{HS3}  for details). 
By Proposition \ref{mainfrustum} and the relation proved in the previous section, we have  proved that
\[
{t^{(m)}_{0,{\rm sing}}}'(0,l_1)+{t^{(m)}_{1,{\rm sing}}}'(0,l_1)
+{t^{(m)}_{0,{\rm sing}}}'(0,l_2)+{t^{(m)}_{1,{\rm sing}}}'(0,l_2)=\int_{\{l_1\}\times W} B+\int_{\{l_2\}\times W} B.
\]

Since both the expressions
\[
{t^{(m)}_{0,{\rm sing}}}'(0,l)+{t^{(m)}_{1,{\rm sing}}}'(0,l),
\]
and
\[
\int_{\{l\}\times W} B,
\]
are continuous functions of $l$ (for positive $l$), it follows (for example taking the limit $l_1\to l_2$) that
\[
{t^{(m)}_{0,{\rm sing}}}'(0,l)+{t^{(m)}_{1,{\rm sing}}}'(0,l)=\int_{\{l\}\times W} B.
\]

Now, the quantity 
\[
\int_{\{l\}\times W} B,
\]
is precisely the quantity that result computing the anomaly boundary term using the formulas of \cite{BM1} if the cone were a smooth manifold, therefore we have proved that
\[
{t^{(m)}_{0,{\rm sing}}}'(0,l)+{t^{(m)}_{1,{\rm sing}}}'(0,l)=A_{\rm BM,abs}(\b C^{(m)}_{0,l}(W)).
\]


\section{De Rham maps and Ray-Singer intersection torsion of a cone}
\label{hodgederhamcone}

\subsection{Hodge Theorem and  De Rham maps}

Using the isomorphisms between the intersection homology groups and the harmonic forms of the cone and of the section, we may prove an extension of the classical Hodge Theorem for the cone and  define the De Rham maps from the spaces of harmonic forms  onto the intersection homology groups for a cone.

\begin{theo}\label{derham1} Let $W$ be a compact connected  orientable Riemannian manifold of dimension $m$, with metric $g$ and without boundary. Let either $\pf=\mf$ or $\pf=\mf^c$.   Let $C_{0,l}(W)$ the geometric cone over $W$. Let $\rho_0:\pi_1(W)\to O(k,V)$ a real orthogonal representation of the fundamental group of $W$. Then, there are chain maps that induce isomorphisms: 
\begin{align*}
I^\pf \A_{{\rm abs},q}&: \Ha^\bu_{\rm abs,\pf}(C_{0,l}(W),V_{\rho_0})\to I^{\pf} H_\bu(C_{0,l}(W);V_{\rho_0}),\\
I^\pf \A_{{\rm rel},q}&: \Ha^\bu_{\rm rel,\pf}(C_{0,l}(W),V_{\rho_0})\to I^{\pf} H_\bu(C_{0,l}(W),W;V_{\rho_0}).
\end{align*}

These maps are called De Rham maps, and are described in the course of the proof.
\end{theo}
\begin{proof} Let $i:W\to C_{0,l}(W)$ be the natural inclusion. Let $N$ be a CW decomposition of $W$, and denote by the same letter $i:N\to C(N)$ the inclusion induced by $i$.

We may construct the following commutative diagram of isomorphisms, for all $0\leq q\leq \left[\frac{m}{2}\right]$:

\centerline{
\xymatrix{\Ha^q_{\rm abs, \mf}(C_{0,l} (W),V_{\rho_0})\ar[r]^{\star}&\Ha^{m-q+1}_{\rm rel, \mf^c}(C_{0,l}(W),V_{\rho_0})\ar@{..>}[r]^{I^\pf \A_{ {\rm rel}}^{m-q+1}}&H^{m-q+1}(I^{\mf^c}\CS_\bu(C_{0,l} (\check N), \check N);V_\rho)&H_q(I^{\mf}\CS_\bu(C_{0,l}(N);V_\rho)\ar[l]_{\hspace{40pt}I^\mf \QQ'_{*,q}}\\
\Ha^q( W,V_{\rho_0})\ar[r]^{\tilde\star}\ar[u]_{k^*_q}&
\Ha^{m-q}( W,V_{\rho_0})\ar[r]^{\tilde\A^{m-q}}\ar[u]_{(-1)^q h^{m-2q}(x)dx\wedge k_q^*}&H^{m-q}(\check N;V_{\rho_0})\ar[u]_{(-1)^q \gamma_q}&H_q( N;V_{\rho_0})\ar[l]_{\QQ_{*, q}}\ar[u]_{(-1)^q \gamma_q}
}}
\noindent where the maps are defined as follows.  The map $k_q^*$ is the isomorphism described in the propositions of Section \ref{harmonicforms}, the inverse of restriction $i^*_q$. The map $\tilde \A^q$ is the De Rham map on $W$ defined in Section \ref{secRS}. $ \QQ'_{*,\bu}$ is the Poincar\'e map described in Section \ref{sec666}. $I^\pf \QQ'_{*,q}$ is the Poincar\'e map for the intersection chain complex introduced in Proposition \ref{dualcone}. The vertical maps on the right, sends cells of the base into cells of the cone, in the dual case as in the construction of the dual space in Section \ref{Poincare}. Commutativity of the left square  is proved in Proposition \ref{Hodge-duality}. Commutativity of the right square is clear by definition. The map $I^\pf \A_{ {\rm rel}}^{m-q+1}$ is defined by closing the central square.

In order to better understand the diagram, take $\tilde\omega_q\in \Ha^q( W,V_{\rho_0})$, then 

\centerline{
\xymatrix{k_q^*(\tilde\omega_q)\ar[r]^{\star}&(-1)^q h^{m-q}(x)\d x\wedge\tilde \star k_q^*(\tilde\omega_q)\ar@{..>}[rr]^{I^\pf \A_{ {\rm rel}}^{m-q+1}}&&(-1)^q\int_0^l h^{m-2q}(x) dx\left(\int_{\check c}\tilde\star\tilde\omega_q\right)\check c^\da
&(-1)^q\int_0^l h^{m-2q}(x) dx\left(\int_{\check c}\tilde\star\tilde\omega_q\right)\check c^\da\ar[l]_{\hspace{40pt}I^\mf \QQ'_{*,q}}\\
\tilde \omega_q\ar[r]^{\tilde \star}\ar[u]_{k_q^*}&\tilde\star\tilde\omega_q 
\ar[rr]^{\tilde\A^{m-q}}\ar[u]_{(-1)^q h^{m-2q}(x)dx\wedge k_q^*}&&\left(\int_{\check c}\tilde\star\tilde\omega_q \right)\check c^\da\ar[u]_{(-1)^q \gamma_q}
&\left(\int_{\check c}\tilde\star\tilde\omega_q\right) c\ar[l]_{\QQ_{*, q}}\ar[u]_{(-1)^q \gamma_q}
}}

What happens on the left is clear. On the right, the image ${\tilde\A^{m-q}}(\tilde\star\tilde\omega_q)$ is an homomorphism on the chains of $\check N$, and precisely, the a multiple of the dual of the  cycle $\check c$ that generates the corresponding homology class, i.e.:
\[
{\tilde\A^{m-q}}(\tilde\star\tilde\omega_q)=\left(\int \tilde\star\tilde\omega\right)=x\check c^\da.
\]

Since, $\check c^\da (\check c)=1$, 
\[
{\tilde\A^{m-q}}(\tilde\star\tilde\omega_q)(\check c)=\left(\int_{\check c} \tilde\star\tilde\omega\right)=x.
\]

The same argument on the top line, gives the coefficients 
\[
\gamma_q=\frac{\| k_q^*(\tilde \omega)\|^2_{C_{0,l}(W)}}{\|\tilde\omega\|^2_W}
=\int_0^l h^{m-2q}(x) dx.
\]


Since the homology does not depend on the cell decomposition, the result follows, with
\[
I^\mf\A_{{\rm abs},q}=(I^\mf \QQ'_{*,q})^{-1} I^{\mf^c} \A_{\rm rel}^{m-q+1} \star.
\]

\end{proof}


\begin{corol}\label{HodgeCone} (Hodge Theorem) There exist natural isomorphisms
\begin{align*}
I^\pf \A^q_{{\rm abs}}&: \Ha^q_{\rm abs,\pf}(C_{0,l}(W),V_{\rho_0})\to I^{\pf} H^q(C_{0,l}(W);V_{\rho_0}),\\
I^\pf \A^q_{{\rm rel}}&: \Ha^q_{\rm rel,\pf}(C_{0,l}(W),V_{\rho_0})\to I^{\pf} H^q(C_{0,l}(W),W;V_{\rho_0}).
\end{align*}
\end{corol}

\subsection{RS intersection torsion of the cone of a manifold}

Let $(W,g)$ be a compact connected oriented Riemannian manifold without boundary. 
Let $C(W)$ be the cone over $W$. Since $W$ is a smooth manifold, it admits a smooth triangulation, and hence it is a regular CW complex. Moreover, any two such triangulations admit a common subdivision. It follows that we may select any one triangulation and all the result of the previous sections hold. The intersection cellular chain complex $I^\pf \CS_\bu(C(N);V_{\rho_0})$ is a complex of based vector spaces, with graded bases induced by the cells, by Lemma \ref{l7.3} and Corollary \ref{c7.6}. It follows that for any given graded homology basis, its R torsion is well defined and independent on the cellular decomposition, see Theorem \ref{t7.1}

We proceed by assuming either $\pf=\mf$ or $\pf=\mf^c$, and that $W$ has a Riemannian structure $g$ and $C(W)$ has the induced Riemannian structure, i.e.   the Cheeger metric, as defined in Section \ref{geocone}, and so we denote it by $C_{0,l}(W)$, and call it geometric cone. Then, using the $L^2$ theory of the Laplace operator on spaces with conical singularities developed in the second part of the work, we may define square integrable harmonic forms on $C_{0,l}(W)$ with coefficients in $V_{\rho_0}$,  and boundary conditions $\rm bc$ in the boundary: 
$\H^\bu_{\rm bc, \pf}(C_{0,l}(W),V_{\rho_0})=V\otimes \H^\bu_{\rm bc, \pf}(C_{0,l}(W))$ (see Section \ref{torman}). In particular,  we have the De Rham maps $I^\pf \A_{{\rm abs},q}$ and $I^\pf \A_{{\rm rel}, q}$ 
that are isomorphism, see Theorem \ref{derham1}.

\begin{defi}\label{torC} Let $W$ be a compact connected  orientable Riemannian manifold without boundary. Let either $\pf=\mf$ or $\pf=\mf^c$.  Let $N$ be any regular  cellular decomposition of $W$. Let $C(W)$ the cone over $W$. We call {\it intersection RS torsion of $C(W)$ with perversity $\pf$ with respect to the representation $\rho_0$}, the positive real number
\[
I^\pf \tau_{\rm RS}(C_{0,l}(W);V_{\rho_0})=\tau_{\rm R} (I^\pf \CS_\bu(C(N);V_{\rho_0});I^\pf\A_{{\rm abs},\bu}(I^\pf \alphas_\bu)), 
\]
where  $I^\pf \alphas_\bu$ is a graded orthonormal basis  for the harmonic forms $\H^\bu_{\rm abs, \pf}(C_{0,l}(W),V_{\rho_0})$, and 
$I^\pf \A_{{\rm abs},\bu}$ the De Rham map. We call relative  intersection RS torsion of $C(W)$, the number
\[
I^\pf \tau_{\rm RS}(C_{0,l}(W), W;V_{\rho_0})=\tau_{\rm R} (I^\pf \CS_\bu(C(N),N;V_{\rho_0});I^\pf\A_{{\rm rel},\bu}(I^\pf \betas_\bu)), 
\]
where  $I^\pf \betas_\bu$ is a graded orthonormal basis  for the harmonic forms $\H^\bu_{\rm rel, \pf}(C_{0,l}(W),V_{\rho_0})$, and 
$I^\pf \A_{{\rm rel},\bu}$ the De Rham map.
\end{defi}

\begin{theo}\label{t7.29} Let $W$ be a compact connected orientable Riemannian manifold without boundary, of dimension $m$. Let either $\pf=\mf$ or $\pf=\mf^c$. Let $\tilde\alphas_\bu$ a graded orthonormal basis of $\Ha^q(W,V_{\rho_0})$, and $\ns_\nu$ the standard graded basis of $H_\bu(W;V_{\rho_0})$. Then,

\begin{align*}
I^\pf \tau_{\rm RS}(C_{0,l}(W);V_{\rho_0})
=& \prod_{q=0}^{\af-2} \gamma_q^{\frac{(-1)^{q}}{2}r_q }
 \left|\det (\tilde\A_{q}(\tilde\alphas_q)/\ns_q)\right|^{(-1)^q }
\left(\# TH_q(W;\Z)\right)^{(-1)^q },\\
I^\pf \tau_{\rm RS}(C_{0,l}(W), W;V_{\rho_0})
=& \prod_{q=\af-1}^{m}  \gamma_q^{\frac{(-1)^{q}}{2}r_q }
 \left|\det (\tilde\A_q(\tilde\alphas_q)/\ns_q)\right|^{(-1)^q}
 \left(\# TH_q(W;\Z)\right)^{(-1)^{q+1}},
\end{align*}

where
\[
\gamma_q=\| k_q^*(\tilde \alphas_q)\|^2_{C_{0,l}(W)}
=\int_0^l h^{m-2q}(x) dx,
\]
$k^*$ the constant extension on the cone, and $r_q=\rk H_q(W;\Z)$. 
\end{theo}
\begin{proof} Consider an orthonormal basis  $\tilde\alphas_q$ of  $\Ha^q(W,V_{\rho_0})$,  then we obtain  a basis $k_q^*(\tilde\al_q)$ of 
$\H^\bu_{\rm abs, \pf}(C_{0,l}(W),V_{\rho_0})$, that however is not normalised. For if $\tilde \omega$ is a $q$ form on $W$,  and $k_q^*(\tilde \omega)$ its extension on the cone, then, $k_q^*(\tilde \omega)=\tilde \omega_1$, in the notation of Section \ref{hodge}. Thus,
\[
\star k_q^*(\tilde \omega)=(-1)^q h^{m-2q}(x) dx\wedge \tilde \star \tilde \omega_1,
\]
and
\[
\|k_q^*(\tilde \omega)\|_{C_{0,l}(W)}^2=\int_0^l h^{m-2q}(x) dx \int_W \tilde\omega \wedge\tilde\star \tilde \omega_1
=\gamma_q\|\tilde \omega\|^2_{W},
\]
where (recall $q<\af-2$)
\[
\gamma_q=\frac{\| k_q^*(\tilde \omega)\|^2_{C_{0,l}(W)}}{\|\tilde\omega\|^2_W}
=\int_0^l h^{m-2q}(x) dx.
\]

Then, the basis $I^\pf \dot{\tilde\al}_q=\gamma_q^{-\frac{1}{2}}k_q^*(\tilde\alphas_q)$ is an orthonormal basis of 
$\H^q_{\rm abs, \pf}(C_{0,l}(W),V_{\rho_0})$. The construction is similar for $\mf^c$ and for the relative case. 

Now, proceeding as in the proof of Theorem \ref{t7.1},
\[
I^\pf \tau_{\rm RS}(C_{0,l}(W);V_{\rho_0})
= \prod_{q=0}^{\af-2} \left|\det (I^\pf \A_q(I^\pf \dot{\tilde\alphas}_q)/I^\pf \dot\ns_q)\right|^{(-1)^{q} }
\prod_{q=0}^{\af-2} \left(\# TH_q(W;\Z)\right)^{(-1)^q}.
\]

Since by commutativity of the diagram in the proof of Theorem \ref{derham1}, up to sign, 
\[
I^\pf \A_{{\rm abs},q}(I^\pf \dot{\tilde\alphas}_q)
=\gamma_q^{-\frac{1}{2}}I^\pf \A_{{\rm abs},q}k_q^*(\tilde\alphas_q)
=\gamma_q^{\frac{1}{2}} \tilde\A_{q}(\tilde\alphas_q),
\]
we have that 
\[
\prod_{q=0}^{\af-2} \left|\det (I^\pf \A_{{\rm abs},q}(I^\pf \dot{\tilde\alphas}_q)/I^\pf \dot\ns_q)\right|^{(-1)^q}
=\prod_{q=0}^{\af-2} \gamma_q^{\frac{(-1)^{q}}{2}r_q}
\prod_{q=0}^{\af-2} \left|\det (\tilde\A_q(\tilde\alphas_q)/\ns_q)\right|^{(-1)^q},
\]
where $r_q=\rk H_q(W;\Z)$.

\end{proof}

\begin{corol} If $m=2p-1$ is odd, then 
\begin{align*}
I^\mf \tau_{\rm RS}(C_{0,l}(W);V_{\rho_0})&=I^\mf \tau_{\rm RS}(C_{0,l}(W),W;V_{\rho_0})\\
&= 
\sqrt{\|\Det k^*_\bu(\tilde\alphas_\bu)\|_{\Det I^\mf H_\bu(C_{0,l}(W);V_{\rho_0})} \tau_{\rm RS}(W;V_{\rho_0})}.
\end{align*}
\end{corol}

\begin{theo}\label{t15.8} Let $W$ be a compact connected orientable Riemannian manifold without boundary, of dimension $m$. Let either $\pf=\mf$ or $\pf=\mf^c$. Then,
\begin{align*}
I^\pf \tau_{\rm RS}(C(W);V_{\rho_0})
=\left(I^{\pf^c} \tau_{\rm RS}(C(W), W;V_{\rho_0})\right)^{(-1)^m}.
\end{align*}
\end{theo}
\begin{proof} This follows by  Proposition \ref{dualtorcone}, since any two cell decompositions of $W$ have a common subdivision, and since  the graded  homology basis induced applying the De Rham gives the absolute and the relative graded homology bases of the cone. 
\end{proof}

\section{The analytic torsion and the  Cheeger-M\"{u}ller theorem for a cone}
\label{cheegercone}

Let $W$ be a compact connected oriented Riemannian manifold of dimension $m$ without boundary. Let $C_{0,l}(W)$ be the deformed finite metric cone over $W$ as defined in Section \ref{geocone}. Given a (trivial) representation $\rho:\pi_1(C_{0,l}(W))\to O(\R^k)$, let $E_\rho$ be the associated vector bundle, as defined at the beginning of Section \ref{torman}. Then, all the theory developed in this part of the work may be remade assuming forms with coefficients in $E_\bu$, i.e., $\Omega^{(q)}(C_{0,l}(W),E_\rho)$, simply by changing the notation. In particular, since the representation is trivial, so will be the bundle. 
We will use this notation in all the relevant objects.

\begin{theo}\label{t1} 
Let $W$ be a compact connected oriented Riemannian manifold of dimension $m$ without boundary. Let $C_{0,l}(W)$ be the deformed finite metric cone over $W$ as defined in Section \ref{geocone}. Let   $T_{\rm abs,\pf}(C_{0,l}(W),E_{\rho_0})$ be the analytic torsion associated to the Hodge-Laplace operator $\Delta_{\rm abs, \pf}$ on $C_{0,l}(W)$ defined in Section \ref{sec4.6}, with coefficients in the real vector bundle $E_{\rho_0}$ induced by the trivial representation $\rho_0:\pi_1(C_{0,l}(W))\to O(\R^k)$  ($\pf=\mf,\mf^c$). Then,
\[
\log T_{\rm abs, \pf}(C_{0,l}(W),E_{\rho_0})=\log T_{\rm global,  abs, \pf}(C_{0,l}(W),E_{\rho_0})+\log T_{\rm bound, \rm abc, \pf}(C_{0,l}(W),E_{\rho_0}),
\]
where 
\[
\log T_{\rm bound, \rm abc, \pf}(C_{0,l}(W),E_{\rho_0})=\frac{1}{4}\chi(W;E_{\rho_0})\log 2+A_{\rm BM, abc}(W;E_{\rho_0}|_W),
\]
is the boundary anomaly term, that only depends on the boundary,  coincides with the anomaly boundary term described  in \cite{BM1,BM2} for smooth manifolds, and vanishes if the boundary is totally geodesic.

The global term is different depending on the parity of the dimension $m$ of $W$. If $m=2p-1$, $p\geq 1$, then
\begin{align*}
\log T_{\rm global, abs, \pf}(C_{0,l}(W);E_{\rho})= &\frac{1}{2}\log T(W,E_{\rho_0}|_W) 
+ \frac{1}{2} \log \|\Det k^*_\bu(\tilde\alphas_\bu)\|_{\Det I^\pf H_\bu(C_{0,l}(W);V_{\rho_0})}.
\end{align*}
If $m=2p$, $p\geq 1$, either $\pf=\mf$ or $\pf=\mf^c$, then
\begin{align*}
\log T_{\rm global, abs, \pf}(C_{0,l}(W),E_{\rho_0}) =& \frac{1}{4} \chi(W;E_{\rho_0})\log 2
+\frac{1}{2}\log \|\Det k^*_\bu(\tilde\alphas_\bu)\|_{\Det I^\pf H_\bu(C_{0,l}(W);V_{\rho_0})}\\
&+\sum_{q=0}^{p-1}(-1)^{q+1} r_q \log (2p-2q-1)!!\\
&+\frac{1}{2}\sum_{q=0}^{p-1}(-1)^{q} \sum_{n=1}^{\infty} m_{{\rm cex},q,n} \log \frac{\mu_{q,n}+\alpha_{q}}{\mu_{q,n}-\alpha_{q}}.
\end{align*}
where $\tilde \alphas_q$ is an orthonormal basis of $\H^q(W)$ and $k^*:\H^q(W)\to \H^q_{{\rm abs},\pf}(C_{0,l}(W))$ the map induced by inclusion (see Section \ref{harmonicforms}), $r_q=\rk H_q(W;\Z)$.
\end{theo}

Note that 
\[
\|\det k^*_q(\tilde\alphas_q)\|_{\det I^\pf H_q(C_{0,l}(W);V_{\rho_0})}=\gamma_q^{r_q}=\left(\int_{0}^{l} h(x)^{m-2q} dx\right)^{r_q}.
\]


\begin{theo}\label{tcm} Let $W$ be a compact connected oriented Riemannian manifold of dimension $m$ without boundary. Let $C_{0,l}(W)$ be the deformed finite metric cone over $W$ as defined in Section \ref{geocone}. Let   $T_{\rm bc,\pf}(C_{0,l}(W))$ be the Analytic Torsion associated to the Hodge-Laplace operator $\Delta_{\rm bc, \pf}$ on $C_{0,l}(W)$ defined inn Section \ref{cone}
($\rm bc=abs, rel$, $\pf=\mf,\mf^c$). If $m=2p-1$, $p\geq 1$, then ($\chi(W)$ is trivial in this case)
\begin{align*}
\log T_{\rm abs, \pf}(C_{0,l}(W);E_{\rho_0})=&\log I^\pf \tau_{\rm RS}(C_{0,l}(W);E_{\rho_0})+\frac{1}{4}\chi(W;E_{\rho_0}|_{W})\log 2\\
&+A_{\rm BM, abs}(W),\\
\log T_{\rm rel, \pf}(C_{0,l}(W);E_{\rho_0})=&\log I^\pf \tau_{\rm RS}(C_{0,l}(W),W;E_{\rho_0})+\frac{1}{4}\chi(W;E_{\rho_0}|_{W})\log 2\\
&+A_{\rm BM, rel}(W).
\end{align*}

If $m=2p$, $p\geq 1$, then
\begin{align*}
\log T_{\rm abs, \pf}(C_{0,l}(W);E_{\rho_0})=&\log I^\pf \tau_{\rm RS}(C_{0,l}(W);E_{\rho_0})+\frac{1}{4}\chi(W;E_{\rho_0}|_{W})\log 2\\
&+A_{\rm BM, abs}(W)+A_{\rm comb, \pf}(W)+A_{\rm analy}(W),\\
\log T_{\rm rel, \pf}(C_{0,l}(W);E_{\rho_0})=&\log I^\pf \tau_{\rm RS}(C_{0,l}(W),W;E_{\rho_0})+\frac{1}{4}\chi(W;E_{\rho_0}|_{W})\log 2\\
&+A_{\rm BM, rel}(W)+A_{\rm comb, \pf}(W)+A_{\rm analy}(W),
\end{align*}
where the combinatoric anomaly term are:
\begin{align*}
A_{\rm comb, \mf}(W)&=-\sum_{q=0}^{p}(-1)^q 
\log \left(\# TH_q(W;\Z)\left|\det (\tilde\A_{q}(\tilde\alphas_q)/\ns_q)\right|\right),\\
A_{\rm comb, \mf^c}(W)&=-\sum_{q=0}^{p-1}(-1)^q 
\log \left(\# TH_q(W;\Z)\left|\det (\tilde\A_{q}(\tilde\alphas_q)/\ns_q)\right|\right),
\end{align*}
the analytic anomaly terms is:
\begin{align*}
A_{\rm analy}(W)=&\sum_{q=0}^{p-1}(-1)^{q+1} r_q \log (2p-2q-1)!!+\frac{1}{2}\sum_{q=0}^{p-1}(-1)^{q} \sum_{n=1}^{\infty} m_{{\rm cex},q,n} \log \frac{\mu_{q,n}+\alpha_{q}}{\mu_{q,n}-\alpha_{q}}\\
&+\frac{1}{4}\chi(W;E_{\rho_0}|_{W})\log 2.
\end{align*}
\end{theo}

\begin{proof} For the absolute torsion, the odd case follows by Theorem \ref{t1} and the Corollary of Theorem \ref{t7.29}. The even case by Theorem \ref{t1} and  Theorem \ref{t7.29}. For the relative torsion, by Proposition \ref{p7.24} and Theorem \ref{dt}.
\end{proof}

Note that, using either Hodge duality on the section or Proposition \ref{t15.8}, the combinatorial anomaly term may be rewritten s
\begin{align*}
A_{\rm comb, \mf}(W)&=-\sum_{q=0}^{p}(-1)^q 
\log \left(\# TH_q(W;\Z)\left|\det (\tilde\A_{q}(\tilde\alphas_q)/\ns_q)\right|\right)\\
&=-\sum_{q=p+1}^{m} (-1)^{q+1}
\log \left(\# TH_q(W;\Z)\left|\det (\tilde\A_{q}(\tilde\alphas_q)/\ns_q)\right|\right).
\end{align*}

%% file: part3.11.tex


\section{Analysis on a space with conical singularities}
\label{spaceX}

A topological space $X$  is called a {\it space with an isolated metric conical singularity} if 
\[
X=C_{0,l}(W)\cup_{\b Y} Y,
\]
where $Y$ is a compact connected smooth Riemannian (we denote its metric by $g_Y$ when necessary) manifold of dimension $n$ with boundary $\b Y$ isometric with $(W,h(l)^2 g)$,  for some $l>0$, the union is smooth along the boundary, and the metric is glued smoothly along the boundary \cite[pg. 575]{Che2}. We say that $X$ has dimension $n=m+1$.

In this setting, we may identify a neighborhood of the boundary with the collar 
\[
\CC=[0,\iota)\times \b Y,
\]
where $\iota=\iota(Y)$ is the injectivity radius of $Y$ \cite[2.7]{Gil}. Whence, using the local system of coordinates $(x,y)$, where $x$ is the global coordinate on the interval $(0,l]$, as described in Section \ref{geocone}, $[0,\iota )$ reads $[l,l+\iota)$, and the metric of $Y$ in $\CC$ coincides with
\[
\g_h=dx\otimes dx+h^2(x)\tilde  g.
\]

Note that we used the same function $h$ as introduced in Section \ref{geocone} for the restriction of the metric of $Y$ onto $\CC$. This is possible without loss of generality. For assume we have any smooth metric $g_Y$ on $Y$, its restriction onto $\CC$ will be as above but with some different smooth function, say $f$, instead of $h$. However, we set no hypothesis on $h$ near $x=l$, beside requiring that it is smooth, and therefore we may as well deform $h$ smoothly near $x=l$ in such  a way that it glues smoothly with $f$ on the other side. 

This is a key point, because then we can affirm that all the formal differential operators defined on $Y$, and in particular the exterior derivative $\d_Y$ and the Hodge-Laplace operator $\Delta_Y$,  when restricted onto $\CC$ coincide with the ones described in Section \ref{ancone}. 

Let 
\[
i:\CC\to Y,
\]
denote the identification of the collar with a neighbourhood of the boundary of $Y$. This map is a smooth injective isometry, and induces an injective isometry 
\[
i^*:L^2(Y) \to L^2(\CC), 
\]
that sends a $q$ form $\vv_Y$ in $\Omega^q(Y)$ into a form $i^*(\vv_Y)$ in $ \Omega^q(\CC)$. We may then compose it with the adjoint map to get a form $\ad ~i^*(\vv_Y)$ in $C^\infty ([l,l+\iota), \Omega^q (W)\times \Omega^{q-1}(W))$, as in Section \ref{hodge}. Observe that $\b Y$ is diffeomorphic to $W$, and $(\b Y, g_\b)$ is isometric to $(W,h^2(l) g)$, where $g_\b$ is the restriction of $g_Y$ on the boundary.

If $\vv$ is a form defined on $X$, it is clear that the same process may be executed producing a smooth inclusion 
\[
 i:(0,l+\iota)\times W,
\]
where $i$ is the identity map on $(0,l]$, and therefore an inclusion
\[
\ad_q  i^*:\Omega^q(X)\to C^\infty((0,\iota), \Omega^q (W)\times \Omega^{q-1}(W)), 
\]
and an injective isometry
\[
\ad_q i^*:L^2(\Omega^q(X)) \to L^2((0,\iota), \Omega^q (W)\times \Omega^{q-1}(W)),
\]
where the inner product on the left side was defined in Section \ref{hodge}.

We may then consider the restriction $\PF^q$ of the formal Hodge  Laplace operator on $X$ making the following square of injective isometries to commute
\[
\xymatrix{
 \Omega^q(X) \ar[d]_{\Delta_X^q} \ar[rrr]^{\ad_q i^*} &&&C^\infty((0,\iota), \Omega^q (W)\times \Omega^{q-1}(W))\ar[d]^{\PF^q}\\
 \Omega^q(X) \ar[rrr]^{\ad_q i^*} &&&C^\infty((0,\iota), \Omega^q (W)\times \Omega^{q-1}(W))
}
\]

As observed the restriction of $\PF^q$ on forms on $C=C_{(0,l]}(W)=(0,l)\times W$ coincides with $\AF^q$, whence we have the identification
for any given  form $\vv$ in $\Omega^q(X)$:
\[
\Delta_X^{(q)} \vv=\AF^q \ad_q \vv|_C+ \Delta_Y^{(q)} \vv|_Y.
\]

In particular, this means that
\[
\langle\vv,\psi\rangle_{L^2(X,\Omega^q(X))}=\langle\ad_q  \vv_C, \ad_q \psi|_C\rangle_{L^2((0,l], \Omega^q (W)\times \Omega^{q-1}(W))}+\langle\vv|_Y,\psi|_Y\rangle_{L^2(Y,\Omega^q(Y))}.
\]

In order to economise notation, we will just write
\beq\label{innerprod}
\langle\vv,\psi\rangle_X=\langle\ad_q  \vv_C, \ad_q \psi|_C\rangle_C+\langle\vv|_Y,\psi|_Y\rangle_Y.
\eeq

We  associate to  the formally self-adjoint  formal Hodge-Laplace operator $\Delta_X^{(q)}$ on $X$ (associated to the metric $g_X$), the following minimal and maximal operators:
\begin{align*}
D(\Delta^{(q)}_{X, \rm min})&=H^2_0(X),\\
D(\Delta^{(q)}_{X, \rm max})&=H^2_{\Delta_{X}}(X),
\end{align*}
that by definition are unbounded and densely defined in $L^2(X)$ (we could take $C^\infty_0(X)$ instead of $H^2(X)$ in the definition of the minimal operator). Since $\Delta^{(q)}_{X, \rm min}\subseteq \Delta^{(q)}_{X, \rm max}$, $\Delta^{(q)}_{X, \rm min}$ is symmetric, and since $\Delta^{(q)}_{X}$ is formally self-adjoint, $(\Delta^{(q)}_{X, \rm min})^\da=\Delta^{(q)}_{X, \rm max}$, and the last is closed. Moreover, the self-adjoint extensions of $\Delta^{(q)}_{X, \rm min}$ are restrictions of $\Delta^{(q)}_{X, \rm max}$. We look for boundary values that characterise these extensions. 

For we recall  the Hodge decomposition of the space of forms on $W$ introduced in Section \ref{declap} in order to obtain a direct sum decomposition of the space
$C^\infty((0,l+\iota), \Omega^q (W)\times \Omega^{q-1}(W))$, and consequently a direct sum decomposition of $\PF^q$. 

Let $\tilde \la_{q,n}$ an eigenvalue of the Hodge-Laplace operator $\tilde\Delta^{(q)}$ in $W$, with $n=0,1,\dots$, and $\tilde \la_{q,0}=0$. Let $\vv_{\tilde\la_{q,n},w_n,j_n}$ an eigenform of $\tilde\la_{q,n}$ in a given complete orthonormal basis of eigenforms of $\tilde \Delta^{(q)}$, where $w_0={\rm har}$ denotes an harmonic forms, $w_{n>0}={\rm ex, cex}$ identifies exact and coexact forms. Let $\nu_{q,n}=\sqrt{\tilde\la_{q,n}^2+\al_q^2}$,  where $\al_q=\frac{1}{2}(1+2q-m)$ ($m=\dim W$). 
Then, we have the direct sum decomposition
\begin{align*}
C^\infty((0,l+\iota), &\Omega^q(W)\times \Omega^{q-1}(W))=\bigoplus_{  n_1,n_2} C^\infty \left((0,l+\iota), \tilde \E^{(q)}_{\tilde\la_{q,n_1}}\times \tilde\E^{(q-1)}_{\tilde\la_{q-1,n_2}}\right)
\end{align*}
and

\[
\PF^q=\bigoplus_{  n_1,n_2}\PF^q_{n_1,n_2},
\]
where
\[
\PF^q_{n_1,n_2}
=\left(\begin{matrix} \tf_{\tilde\la_{q,n_1},\al_q} & -2\frac{h'}{h^2} \tilde \d\\
 -2\frac{h'}{h^2} \tilde \d^\da& \tf_{\tilde\la_{q-1,n_2},-\al_{q-2}}\end{matrix}\right),
\]
acts on the space
\[
C^\infty \left((0,l+\iota), \tilde \E^{(q)}_{\tilde\la_{q,n_1}}\times \tilde\E^{(q-1)}_{\tilde\la_{q-1,n_2}}\right),
\]
and $\tf_{\la,\al}$ is the formal operator
\[
\tf_{\la,\al} = -\frac{d^2}{dx^2} - \frac{h''(x)}{h(x)}\left(\alpha-\frac{1}{2}\right) + \frac{\la^2 +\left(\alpha^2-\frac{1}{4}\right)(h'(x)^2}{h(x)^2}.
\]

Take $\vv$ in $C^\infty(X, \Omega^q(X))$, then the  decomposition of  $\ad_q i^*(\vv)$ in 
\[
C^\infty((0,l+\iota), \Omega^q(W)\times \Omega^{q-1}(W))=\bigoplus_{  n_1,n_2} C^\infty \left((0,l+\iota), \tilde \E^{(q)}_{\tilde\la_{q,n_1}}\times \tilde\E^{(q-1)}_{\tilde\la_{q-1,n_2}}\right),
\]
is a linear combination
\[
\ad_q i^*(\vv)=\sum_{n_1,n_2,w_{n_1},w_{n_2},j_{w_{n_1}},j_{w_{n_2}}} \left( u_{\tilde\la_{q,n_1},w_{n_1},j_{n_1}} ,
u_{\tilde\la_{q-1,n_2},w_{n_2},j_{n_2}} \right),
\]
where $u_j$, $j=1,2$,  are two smooth functions on $(0,l+\iota)$ with values in $\tilde \E^{(q)}_{\tilde\la_{q,n_1}}$ and $\tilde\E^{(q-1)}_{\tilde\la_{q-1,n_2}}$, respectively. We denote by $P_{\tilde\la_{q,n_1},\tilde\la_{q-1,n_2}}$ the projection
\[
P_{\tilde\la_{q,n_1},\tilde\la_{q-1,n_2}}\ad_q i^*(\vv)=\left(u_{\tilde\la_{q,n_1},w_{n_1},j_{n_1}} , u_{\tilde\la_{q-1,n_2},w_{n_2},j_{n_2}} \right).
\]

This decomposition induces a direct sum decomposition
\[
\Delta^{(q)}_X=\bigoplus_{  n_1,n_2}\Delta^{(q)}_{X,n_1,n_2},
\]
where 
\[
\Delta^{(q)}_{X,n_1,n_2}\vv=\ad_q \AF_{n_1,n_2}\vv|_C+\Delta^{(q)}_{Y,n_1,n_2}\vv|_Y.
\]

\begin{rem} Observe that the formula above is not a definition, but a matter of fact: namely the operator $\Delta^{(q)}_{Y,n_1,n_2}$ acts as the operator $\ad_q \AF_{n_1,n_2}$ in the collar $\CC$, therefore, extending on left the collar, we may as well consider the operator $\Delta^{(q)}_{Y,n_1,n_2}$ acting in the same way on it, and this is what the above formula is saying.
\end{rem}

Let consider the operator $\Delta^{(q)}_{X,n_1,n_2}$, for fixed $n_1$  and $n_2$. Applying the analysis described in Section \ref{sec4.6}, there are two linearly independent square integrable solutions of the formal eigenvalues equation for $\Delta^{(q)}_{X,n_1,n_2}$ near $x=0$ if $\nu_{1,q,n_1}\geq 1$ and $\nu_{2,q,n_2}\geq 1$, while there is only one such solutions otherwise. Note that the last situation may happen only for a finite number of values of $n_j$, $j=1,2$. Whence if either $\nu_{1,q,n_1}$ or $\nu_{2,q,n_2}$ is smaller that $1$, we have deficiency indices $(1,1)$ and we need a boundary condition at $x=0$. We may construct a boundary value at $x=0$ as in Section \ref{bv}.  
Let $\uf_\pm$ be the  two normalised solutions of the  differential equation (\ref{eqdiff2}) introduced in Definition \ref{defi1}.   
Fix a point $x_0$, with  $0<x_0<l$, and let $v_\pm$ be two smooth functions on $(0,l]$ vanishing  for $x>x_0$ and equal to 
$\uf_\pm$ near $0$. Use these functions to construct (by extending them with the zero function) two smooth forms $\vv_\pm$ in $C^\infty(X, \Omega^q(X))$, i.e. 
\[
\ad_q i^*(\vv_\pm)=\left( v_\pm \tilde\vv_{\tilde\la_{q,n_1},w_{n_1},j_{n_1}},v_\pm \tilde\vv_{\tilde\la_{q-1,n_2},w_{n_2},j_{n_2}}\right),
\]
and is the zero form other wise.

Then, near $x=0$
\[
(\Delta^{(q)}_{X,n_1,n_2} \vv_\pm)(x)=  (\PF_{n_1,n_2}^q \ad_q i^* (\vv_\pm))(x)=0,
\]
so that $\Delta^{(q)}_{X,n_1,n_2} \vv_\pm$ is square integrable on $X$ and therefore $\vv_\pm\in D(\Delta^{(q)}_{X,n_1,n_2, \rm max})$.  
Thus, by \cite[XII.4.20]{DS2}
\[
BV_{n_1,n_2,\pm}(0)(\vv)=(\Delta^{(q)}_{X,n_1,n_2} \vv,\vv_\pm)-( \vv,\Delta^{(q)}_{X,n_1,n_2} \vv_\pm),
\]
are boundary values for $\Delta^{(q)}_{X,n_1,n_2}$ at $0$. In fact, these are continuous functionals on $D(\Delta^{(q)}_{X,n_1,n_2,{\rm max}})$ that vanish on $D(\Delta^{(q)}_{X,n_1,n_2,\rm min})$. Using the formula in equation (\ref{innerprod}), we may compute
\begin{align*}
BV_{n_1,n_2,\pm}(0)(\vv)&=(\Delta^{(q)}_{X,n_1,n_2}\vv,\vv_\pm)_{L^2(X,\Omega^q(X))}-(\vv,\Delta^{(q)}_{X,n_1,n_2}\vv_\pm)_{L^2(X,\Omega^q(X))}\\
&=(\AF^q_{n_1,n_2}\ad_q  \vv|_C, v_\pm|_C)_C
-(\ad_q  \vv|_C, \AF^q_{n_1,n_2} v_\pm|_C)_{C},
\end{align*}
since the restriction of $\Delta^q_X$ on $M$ is self-adjoint. Whence, proceeding as in Section \ref{bv}, and recalling Remark \ref{mixedter}, we have the following two boundary values for $\Delta^{(q)}_{X,n_1,n_2}$ at $x=0$: for $\vv \in D(\Delta^{(q)}_{X,n_1,n_2,{\rm max}})$, write $\ad_q \vv|_C=(u_1,u_2)$, then
\begin{align*}
BV_{1,n_1,n_2,\pm}(0)(u_1,u_2):& \left\{\begin{array}{l}\left\{\begin{array}{cc}BV_{\nu_{q,n_1},\pm}(0)(u_1),&{\rm ~if~}\nu_{1,q,n_1}<1,\\
{\rm none,}&{~\rm if~}\nu_{1,q,n_1}\geq 1,
\end{array}\right.\\
\left\{\begin{array}{cc}
 BV_{\nu_{2,q-1,n_2},\pm}(0)(u_2),&{\rm ~if~}\nu_{2,q-1,n_2}<1,\\
{\rm none,}&{~\rm if~}\nu_{2,q-1,n_2}\geq 1,
\end{array}\right.\\
\end{array}\right.\\
\end{align*}
where
\[
BV_{\nu,\pm}(0)(u)=\lim_{x\to 0^+}\left( u'(x)v_\pm(x)-u(x)v'_\pm(x)\right), 
\]
and recall that $v_\pm$ depends on the value of $\nu$.

We are ready to introduce our extensions of the formal Hodge-Laplace operator on $X$.

\begin{defi} \label{XLaplace} 

Let $\Delta^{(q)}_X$ denote the formal Hodge-Laplace operator on $X$. We define the  operators $\Delta^q_{X, n_1, n_2, \pm}$  acting on $L^2(X)$ as follows: 
\begin{align*}
D(\Delta^{(q)}_{X, n_1, n_2, \pm})&=\left\{ \vv\in D(\Delta^{(q)}_{X,{\rm max}})~|~BC^q_{n_1,n_2, {\rm +}}(0)(P_{\tilde\la_{q,n_1},\tilde\la_{q-1,n_2}}\ad_q i^*(\vv))=0\right\}.
\end{align*}

\begin{align*}
\Delta^{(q)}_{X,\pm}&=\bigoplus_{  n_1,n_2} \Delta^{(q)}_{X,n_1,n_2,  \pm}.
\end{align*}

If $m=2p-1$, $p\geq 1$:
\begin{align*}
\Delta^{(\bu)}_{X, \mf}=\Delta^{(\bu)}_{X, \mf^c}=&\bigoplus_{q=0}^{p-2} \Delta^{(q)}_{X,\rm max}\oplus \Delta^{(p-1)}_{X,+}\oplus \Delta^{(p)}_{X,\rm max}
\oplus \Delta^{(p+1)}_{X,+}\oplus \bigoplus_{q=p+2}^{2p} \Delta^{(q)}_{X,{\rm max}}.
\end{align*}

If $m=2p$, $p\geq 1$:
\begin{align*}
\Delta^{(\bu)}_{X, \mf^c}=&\bigoplus_{q=0}^{p-2} \Delta^{(q)}_{X,\rm max}\oplus \Delta^{(p-1)}_{X,+}\oplus \Delta^{(p)}_{X,+}
\oplus \Delta^{(p+1)}_{X,0,0,-}\oplus\bigoplus_{n_1>0,n_2>0} \Delta^{(p+1)}_{X,n_1, n_2,+} \\
&\oplus \Delta^{(p+2)}_{X,+}\oplus\bigoplus_{q=p+3}^{2p+1} \Delta^{(p-1)}_{X,{\rm max}},\\
\Delta^{(\bu)}_{X, \mf}=&\bigoplus_{q=0}^{p-2} \Delta^{(q)}_{X,\rm max}\oplus \Delta^{(p-1)}_{X,+}
\oplus \Delta^{(p)}_{X,0,0,-}\oplus\bigoplus_{n_1>0,n_2>0} \Delta^{(p)}_{X,n_1, n_2,+} 
\oplus \Delta^{(p+1)}_{X,+}\\
&\oplus \Delta^{(p+2)}_{X,+}\oplus
\bigoplus_{q=p+3}^{2p+1} \Delta^{(p-1)}_{X,{\rm max}}.
\end{align*}

All these operators act by means of the formal Hodge-Laplace operator $\Delta^{(q)}_X$.
\end{defi}

\begin{prop}\label{greenY} Let $\te$ be a square integrable form on $Y$, then
\begin{align*}
\langle\Delta_Y \te,\te\rangle=&\|d \te\|^2+\|d^\da_Y \te\|^2\\
&-h^{1-2\al_q}(0)f_1(0)f_1'(0)\|\t\omega_1\|^2-(h^{1-2\al_{q-1}}f_2)'(0)f_2(0)\|\t\omega_2\|^2\\
&+2 h^{1-2\al_q}(0)f_1(0)f_2(0)\langle \t\omega_1 ,\tilde d\t\omega_2\rangle.
\end{align*}
\end{prop}

\begin{proof} Let $\te$ be a square integrable form on $Y$. Since  
\[
d d_Y^\da\theta\wedge\star_Y\theta=d(d_Y^\da \theta\wedge\star_Y \theta)+d^\da\theta\wedge\star_Y d^\da\theta,
\]
so
\[
\int_Y d d_Y^\da\theta\wedge\star_Y\theta=
\int_W i^*(d^\da \theta\wedge\star_Y \theta)+\int_Y d_Y^\da\theta\wedge\star_Y d_Y^\da\theta.
\]

Similarly, 
\[
d^\da_Y d\theta\wedge\star_Y\theta=-d( \theta\wedge\star_Y d \theta)+d\theta\wedge\star_Y d\theta,
\]
so
\[
\int_M d^\da_Md_M\theta\wedge\star_M\theta
=-\int_W i^*( \theta\wedge\star_Y d \theta)+\int_Y d\theta\wedge\star_Y d\theta,
\]
then
\[
\langle\Delta_Y \te,\te\rangle=\|d \te\|^2+\|d^\da_Y \te\|^2+\int_Wi^*(d_M^\da \theta\wedge\star_M \theta)
-\int_W i^*( \theta\wedge\star_M d_M \theta).
\]

Decomposing 
$\theta=\theta_{\rm tan}+ \theta_{\rm norm}$,
we obtain (compare with \cite[2.1.4]{Sch})
\begin{align*}
\int_Wi^*(d_M^\da \theta\wedge\star_M \theta)
-&\int_W i^*( \theta\wedge\star_M d_M \theta)\\
&=\int_W \left.(d_M^\da \theta)_{\rm tan}\wedge\star_M \theta_{\rm norm}\right|_{W}
-\int_W \left.\theta_{\rm tan}\wedge \star_M (d_M\theta)_{\rm norm}\right|_{W}.
\end{align*}

In the local system, setting
\[
\theta|_\CC=\omega_1+dx\wedge\omega_2,
\]
we find that (see also \cite[2.27]{Sch})
\begin{align*}
(\theta|_\CC)_{\rm tan}&=\omega_1,\\
d_Y \theta|_\CC&=\tilde d \omega_1+dx\wedge(\omega_1'-\tilde d\omega_2),\\
(d_Y\theta|_\CC)_{\rm norm}&=dx\wedge(\omega_1'-\tilde d\omega_2),\\
\star_Y (d_Y\theta|_\CC)_{\rm norm}&\star_Y d x\wedge \omega_1'-\tilde d\omega_2
= h^{1-2\al_q}(\tilde \star \omega_1'-\tilde\star\tilde d\omega_2),
\end{align*}
\[
\int_W \left.\theta_{\rm tan}\wedge \star_Y (d_Y\theta)_{\rm norm}\right|_{x=l}
=h^{1-2\al_q}|_W\int_W \omega_1\wedge(\tilde \star \omega_1'-\tilde\star\tilde d\omega_2).
\]

Separating the variables, with $x\in [0,\iota)$, 
\[
\theta|_\CC(x,y)=f_1(x)\omega_1(y)+f_2(x) dx\wedge \omega_2(y),
\]
we have
\begin{align*}
\int_W \left.\theta_{\rm tan}\wedge \star_Y (d_Y\theta)_{\rm norm}\right|_{x=0}
=&h^{1-2\al_q}(0)f_1(0)f_1'(0)\int_W\omega_1\wedge\tilde\star\omega_1\\
&-h^{1-2\al_q}(0)f_1(0)f_2(0)\int_W\omega_1 \wedge\tilde\star\tilde d\omega_2,\\
\int_W \left.(d_Y^\da \theta)_{\rm tan}\wedge\star_Y \theta_{\rm norm}\right|_{x=0}
=&-(h^{1-2\al_{q-1}}f_2)'(0)f_2(0)\int_W\omega_2\wedge\tilde\star\omega_2\\
&+h^{-1-2\al_{q-1}}(0) f_1(0)f_2(0) \int_W\tilde d^\da \omega_1\wedge\tilde\star\omega_2.
\end{align*}
\end{proof}

\begin{prop} \label{prop1.3} The operators $\Delta^{(\bu)}_{X, \pf}$ are non negative self-adjoint operators. 
\end{prop}
\begin{proof} $\Delta_X$ is formally self-adjoint, $\Delta_{X,\rm min}$ is densely defined and therefore symmetric, with adjoint $\Delta_{X,\rm max}$, so  $\Delta^{(\bu)}_{X, \pf}$ that is a restriction of $\Delta_{X,\rm max}$ is self-adjoint. 

We show that it is non negative. For take $\te\in D(\Delta^{(\bu)}_{X, \pf})$ and consider the inner product 
\begin{align*}
\langle \Delta^{(q)}_{X,\pf}\theta,\theta \rangle&=\int_X \Delta_X^{(q)}\theta\wedge\star_X \theta=
\int_C\Delta^{(q)}_{C}\theta|_C\wedge \star_C \theta |_C
+ \int_Y \Delta^{(q)}_Y \theta|_Y\wedge\star_Y \theta|_Y\\
&=\langle \Delta^{(q)}_{C}\theta|_C,\theta|_C \rangle+\langle \Delta^{(q)}_{Y}\theta|_Y,\theta|_Y \rangle.
\end{align*}

Using  Propositions \ref{green cone2} and \ref{greenY}, we see that the boundary terms cancel each other and the remaining part is a sum of positive quantities. 
\end{proof}

\begin{prop}\label{compresX} The operators $\Delta^{(\bu)}_{X, \pf}$ have compact resolvent. 
\end{prop}
\begin{proof} Let $r(z_1, z_2)$ be the  standard kernel for the resolvent of the Hodge-Laplace operator on the compact manifold $Y$ (see for example \cite{Gil}). On the collar, we may use local coordinate and separate the variables, moreover we can use the decomposition on the complete system of eigenforms on the section, so we reduce exactly to the situation described for the cone in the proof of Theorem  \ref{teo1}, so $k$ takes the explicit form described in Theorem \ref{teo1}. Therefore the minimal operators (on forms with compact support) coincide. Since both the maximal extensions are continuous (and densely defined) they coincide. Since the resolvent on the cone is defined by the extension of the resolvent of the collar, this defines the  resolvent of $\Delta^{(\bu)}_{X, \pf}$, and shows that the it has a continuous square integrable kernel, since it is square integrable on $Y$ by construction, and it is square integrable on the cone by Theorem \ref{teo1}.
\end{proof}

These propositions imply the following fact.

\begin{corol} The operators $\Delta^{(\bu)}_{X, \pf}$ have discrete non negative spectrum with simple eigenvalues. The eigenfunctions determine a complete orthonormal  basis of $L^2(X)$. 
\end{corol}

\begin{rem} As observed in Corollary \ref{tracecorol}, there exists a power of the resolvent of $\Delta^{(\bu)}_{X, \pf}$ that is of trace class.
\end{rem}

\begin{corol}\label{heattraceX} The heat operator $\e^{-t\Delta^{(\bu)}_{X, \pf}}$ associated to $\Delta^{(\bu)}_{X, \pf}$ is of trace class. Its trace has an asymptotic expansion for small $t$ of the following form
\[
\Tr \e^{-t\Delta^{(\bu)}_{X, \pf}}=\sum_{k=0}^\infty a_k(\Delta^{(\bu)}_{X, \pf}) t^\frac{m+1-k}{2}.
\]
The coefficients $a_k(\Delta^{(\bu)}_{X, \pf})$ are local, i.e. there exist smooth integrable functions such that
\[
a_k(\Delta^{(\bu)}_{X, \pf})=\int_X e_k(x, \Delta^{(\bu)}_{X, \pf}) d vol_X
\]

\end{corol}
\begin{proof} By the theorem, the resolvent has a continuous square integrable kernel whose restriction on $M$ is the kernel of the Hodge-Laplace on $Y$, and therefore is of trace class and has got the stated local asymptotic expansion,  whose restriction on the cone is the resolvent of the Hodge-Laplace operator on the cone discussed in Theorem \ref{teo1}. This kernel is of trace class since its trace is given by Laplace transform of  the associated logarithmic Gamma function \cite[2.3]{Spr9}, and the last is the one studied in Section \ref{s6.1}. Moreover, the asymptotic expansion exists and is given by the explicit twisted product of the one on the section and the one on the line. The coefficients of the expansion on $X$ are thus given by the sum of the integrals of the local coefficients on the two parts. 
\end{proof}

In order to describe the spaces of the harmonic forms on $X$ we recall a few facts about the De Rham cohomology groups of a manifold with boundary. So let $Y$ be a compact connected oriented Riemannian manifold of dimension $m+1$ with smooth connected boundary $W$. 
Let $i:W=\b Y \to Y$ denote  the inclusion of the boundary.  
The map 
\[
i^q:\Omega^q(Y)\to \Omega^q(W),
\]
induced by $i$ is surjective. Set
\[
\Omega^q(Y,W)=\ker i^q,
\]
we verify that 
\[
\Omega^q(Y,W)=\Omega_{\rm rel}^q(Y)=\{ \omega \in \Omega^q(Y)~|~\omega_{\rm tg}|_W=0\}.
\]
and, we have a short exact sequence
\[
\xymatrix{
0\ar[r]&\Omega^q(Y,W)\ar[r]^{p^q}&\Omega^q(Y)\ar[r]^{i^q}&\Omega^q(W)\ar[r]&0.
}
\]

Let $H_{\rm DR}^q$ denote the De Rham cohomology, since the maps $j^q$ and $i^q$ commute with the differentials, the sequence above induces a long exact cohomology sequence in the De Rham cohomology. 

\[
\xymatrix{
\dots \ar[r]&H_{\rm DR}^q(Y,W)\ar[r]^{p_*^q}&H_{\rm DR}^q(Y)\ar[r]^{i_*^q}&H_{\rm DR}^q(W)\ar[r]^{\delta^q}&
H^{q+1}_{\rm DR}(Y,W)\ar[r]&\dots.
}
\]


Moreover, the map
\[
\A^q:\omega \to \int \omega
\]
defined on the spaces of smooth forms induces isomorphisms

\begin{align*}
\H^q(W)&\to H^q_{\rm DR}(W)\to H^q(W),\\
\FF^q_{\rm abs}(Y)&\to H^q_{\rm DR}(Y)\to H^q(Y),\\
\FF^q_{\rm rel}(Y)&\to H^q_{\rm DR}(Y,W)\to H^q(Y,W),
\end{align*}
where $\FF_{\rm bc}(Y)\leq \H^q_{\rm bc}(Y)$ are the subspaces of the harmonic fields (closed and co closed harmonic forms) \cite[Definition 2.2.1]{Sch}.

We may now characterise the image of $j^q_*$. By exactness this is the kernel of $i^q_*$, i.e. 
\[
\Im j^q_*=\ker i_*^q=\left\{[\omega]\in H_{\rm DR}^q(Y)~|~ i_*^q([\omega])=[0]\in H_{\rm DR}^q(W)\right\}.
\]

Since by the Hodge decomposition  \cite[2.6.1]{Sch}, we may write univocally 
\[
\omega=\theta+d\al,
\]
with $\theta\in \FF_{\rm abs}^q(W)=\{\theta\in \Omega^q(Y)~|~d\theta=d^\da\theta=0,\theta_{\rm norm}|_W=0\}$,  the module of the harmonic fields with absolute bc, and
\[
i^q(\omega)=\omega_{\rm tg}|_W=\theta_{\rm tg}|_W+d\al_{\rm tg}|_W,
\]
the equation above has the solution
\[
i^q(\theta)=\theta_{\rm tg}|_W=\tilde d\tilde \vv,
\]
where $\tilde d\tilde \vv$ is not an harmonic on $W$, i.e. $\tilde d^\da\tilde d\tilde \vv\not=0$.

We now proceed to a local characterisations of the harmonic forms representing the cohomology classes.

Repeating the analysis of the formal Hodge-Laplace operator made on the cone, but with $h=1$, we have the following description of the solutions $\te_Y$ of the harmonic equation for the formal Hodge-Laplace operator on $Y$ in the collar $\CC$, decomposed on the spectral family of eigenfunctions of the Hodge-Laplace operator on the section: 
\[
\te_\CC(t)= \left( a_1 \sh (\sqrt{\t \la_1} t) +b_1 \ch (\sqrt{\t\la_1}t)\right) \t \vv_1+ \left( a_2 \sh (\sqrt{\t \la_2} t) +b_2 \ch (\sqrt{\t\la_2}t)\right) \t \vv_2.
\]
where  $t\in [0,\iota)$ is the standard local coordinate near the boundary on $Y$ (out going geodesic distance from the boundary).

By \cite[2.4.8]{Sch}, 
\[
\FF^q(Y)=\FF^q_{{\rm abs}}(Y)\oplus \FF_{\rm ex}(Y)= \FF^q_{{\rm rel}}(Y)\oplus \FF_{\rm cex}(Y),
\]
and $ \FF^q_{{\rm abs}}(Y)\cap \FF^q_{{\rm rel}}(Y)=\{0\}$. Therefore, if $\te_Y\in  \FF^q_{{\rm abs}}(Y)$, then 
$\te_Y\not\in  \FF^q_{{\rm rel}}(Y)$, so $\te_Y=d^\da \be$. Viceversa, if $\te_Y\in  \FF^q_{{\rm rel}}(Y)$, then 
$\te_Y\not\in  \FF^q_{{\rm abs}}(Y)$, so $\te_Y=d \vv$. So there are three possibilities
\[
\begin{aligned}
\te_Y&=d^\da\be\in \FF^q_{\rm abs}\cap \FF_{\rm cex}(Y),\\
\te_Y&=d\vv\in \FF^q_{\rm rel}\cap \FF_{\rm ex}(Y),\\
\te_Y&=d\vv=d^\da \be\in (\FF^q_{\rm abs}\cup\FF^q_{\rm rel})^\bot,
\end{aligned}
\]
where the orthogonal complement is meant inside the space of harmonic fields. 

Imposing these requests,  for $\te_Y$ to be an harmonic field we have that on the collar $\te_\CC$ is a linear combination of the following four types:
\begin{align*}
\te_{\rm E}(t)&=\tilde\te_{\rm harm},\\
\te_{\rm II}(t)&=\ch(\sqrt{\t\la}t) \t d\t\vv+\sqrt{\t\la}\sh(\sqrt{\t\la}t) d t\wedge\t\vv,\\
\te_{\rm O}(t)&=dt \wedge \tilde\te_{\rm harm},\\
\te_{\rm III}(t)&=\sh(\sqrt{\t\la}t) \t d\t\vv+\sqrt{\t\la}\ch(\sqrt{\t\la}t)d t\wedge\t\vv.
\end{align*}

Note that $\te_{\rm E}, \te_{\rm II}\in \FF^{\bu}_{\rm abs}(Y)$, while $\te_{\rm O}, \te_{\rm III}\in \FF^{\bu}_{\rm rel}(Y)$. 

However, on the collar $\CC$, $\te_{\rm E}\in \FF^{\bu}_{\rm abs}(\CC)$,  $\te_{\rm O}\in \FF^{\bu}_{\rm rel}(\CC)$
but , $\te_{\rm II}, \te_{\rm III}$ do not satisfy any boundary condition (recall there are two boundary on the collar). Indeed, 
it is easy to verify that 
\begin{align*}
\te_{\rm II}(t)&=-\frac{1}{\sqrt{\t\la}}d^\da(\sh(\sqrt{\t\la}t) dt \wedge \t d\t\vv),\\
\te_{\rm III}(t)&=d(\sh(\sqrt{\t\la}t)  \t\vv),\\
\te_{\rm II}(t)&=d (\ch (\sqrt{\t\la} t)\t\vv),\\
\te_{\rm III}(t)&=-\frac{1}{\sqrt{\t\la}}d^\da(\ch(\sqrt{\t\la}t) dt \wedge \t d\t\vv)
\end{align*}
and this shows that $\te_{\rm II}, \te_{\rm III}\in (\FF^q_{\rm abs}\cup\FF^q_{\rm rel})^\bot$. Moreover
\beq\label{x1}
\te_{\rm O}(t)=d(t\t\te_{\rm harm}),
\eeq

These equations tell us that the forms of types $O, II$ and $III$ represents trivial cohomology classes on the collar, whose cohomology is indeed that of $W$, whose classes are all represented by harmonics of type $E$ in $\FF^\bu_{\rm abs}(Y)$.

However, this is no longer true on $Y$, where there may exist non trivial classes in the relative cohomology. This is due to the fact that the equation (\ref{x1}) may be not satisfied on the whole $Y$, in other words the local form $t\t\te_{\rm harm}$ may be not extendable on $Y$. In particular, forms of types II and III either coincides, and in such a case they represent a trivial homology class, or one is exact and the other coexact, and in such a case one represents a trivial cohomology class and the other a non trivial class.

Next, observe that if $\te_Y\in \FF^q_{\rm abs}(Y)$ represents a cohomology class $[\te_Y]\in H_{\rm DR}^q(Y)$, then either 
$[i^q_*(\te_Y)]=[0]\in H_{\rm DR}^q(W)$ or not, and in the second case, this means that $[\te_Y]$ is not in the image of the map $p^q_*$, i.e. it does not come from a relative cohomology class. We have the following local description.

\begin{lem} With the notation introduced above, given  $\te_Y\in  \FF^q_{\rm abs}(Y)$, then $[\te_Y]\in \ker i^q_*$ if and only if on the collar 
\[
\te_\CC(t)=d^\da(\sh(\sqrt{\t\la}t) \t d\t\vv)=-\sqrt{\t\la}\ch(\sqrt{\t\la}t) \t d\t\vv-\t\la\sh(\sqrt{\t\la}t) d t\wedge\t\vv.
\]

On the other side, given $\te_Y\in  \FF^q_{\rm rel}(Y)$, then $[\te_Y]\in \ker j^q_*$ if and only if on the collar 
\[
\te_\CC(t)=d(\sh(\sqrt{\t\la}t) \t\vv)=\sh(\sqrt{\t\la}t) \t d\t\vv+\sqrt{\t\la}\ch(\sqrt{\t\la}t)d t\wedge\t\vv.
\]
\end{lem}

\begin{prop}\label{HarmonicX} There are the following natural isomorphisms between the  spaces of the harmonic forms $\H^q_{\rm  \pf}(X)$ of the operator $\Delta_{X,\pf}^{(q)}$ (here $\pf$ stands for either $\mf$ or $\mf^c$) and those of the harmonic forms of operators $\Delta^{(q)}_{Y,\rm abs}$ and $\Delta^{(q)}_{Y,\rm rel}$:
\begin{align*}
\H^q_{\mf}(X)&=\begin{cases}
\H^{(q)}_{\rm abs}(Y), &0\leq q\leq p-1,\\
\ker(i^p_*:H_{\rm DR}(Y)\to H_{\rm DR}(W)), &  q=p,\\
 \H^{(q)}_{\rm rel}(Y), & p+1\leq q\leq 2p,\end{cases}&\dim W=2p-1,\\
\H^q_{\mf^c}(X)&=
\begin{cases} 
\H^{(q)}_{\rm abs}(Y), &  0\leq q\leq p-1,\\
\ker (i^p_*:H_{\rm DR}(Y)\to H_{\rm DR}(W)), &  q=p,\\
\H^{(q)}_{\rm rel}(Y), &  p+1\leq q\leq 2p+1,
\end{cases}&\dim W=2p,\\
\H^q_{\mf}(X)&=
\begin{cases} 
\H^{(q)}_{\rm abs}(Y), &  0\leq q\leq p,\\
\ker(i^p_*:H_{\rm DR}(Y)\to H_{\rm DR}(W)), &  q=p+1,\\
\H^{(q)}_{\rm rel}(Y), &  p+2\leq q\leq 2p+1,
\end{cases}&\dim W=2p,
\end{align*}
where the map $i:W\to Y$ is the inclusion, and the isomorphism is  given by restriction plus an isometry (as described in the course oft he proof)
\begin{align*}
j^q&:  \H^q_{\rm \pf}(X)\to \H^q_{\rm bc}(Y),\\
j^q&:\theta\mapsto   \theta|_Y,
\end{align*}
induced by the inclusion $j:Y\to X$.
\end{prop}

\begin{proof} In the course of this proof, we will consider without saying that explicitly only square integrable forms. Let $\theta$ be an harmonic of $\Delta^{(q)}_{X,\pf}$.  Then, $\te_Y=\te|_Y$ is a solution of the harmonic  equation on $Y$, while $\te_C=\te|_C$ is a solution of the harmonic equation on the cone, and
\[
0=\langle \Delta^{(q)}_{X,\pf}\theta,\theta \rangle
=\langle \Delta^{(q)}_{C}\theta_C,\theta_C \rangle_C+\langle \Delta^{(q)}_{Y}\theta_Y,\theta_Y \rangle_Y.
\]

Consider the space $\SS=C_{0,l+\iota}(W)=C_{0,l}\cup_W \CC$, and the restriction $\te_\SS=\te|_\SS$.  
On  $\SS$  we have the local system used in Part II for the cone, and we may consider the tangential and the normal component of $\te$, and decompose on the spectral resolution of the Hodge-Laplace operator on the section.  Thus, $\te_C$ has the form of one of the solutions described in Lemma \ref{harmonics}: of type O, E, I, or IV.

Observe that there is an isometry $\Psi: C_{l,l+\iota}(W)=\SS-C_{0,l}(W)\to \CC$, that sends the inward collar of the boundary of the cone over $W$ of length $l+\iota$ onto the the collar $\CC$ of the boundary of $Y$ in $Y$, and preserves the type of the solutions of the harmonic equation. This isometry essentially consists in a change from generalised spherical coordinates to cartesian coordinate composed with a translation by $l$, and extend to an isometry of $X-C_{0,l}(W)\to Y$. Let $x$ denotes the radial coordinate on the cone as in Part II, and $t\in [0,\iota)$ the local coordinate on $\CC$ as described above. Then, on 
\begin{align*}
\Psi&: C_{l,l+\iota}(W)\to \CC,\\
\Psi&:\te_{T}(x)\mapsto \te_T(t),
\end{align*}
where $T=E,O,I,II,III,IV$, $\te_{T}(x)$ is as described in Proposition \ref{harmonics} and $\te_{T}(t)$ as described above the statement of this proposition.

We consider now the different types of solutions. It is clear that those of types I and IV are not harmonic fields on the cone and therefore may not be harmonic fields on $X$. So we consider the other types.

First, suppose that $\te_\SS$ is either of type II or III. Then, $\Psi(\te|Y)$ is either of type II or III on Y, and therefore, by what seen above,  it is either exact or it represents a non trivial class of $H^q(Y)$ in the kernel of $i^q_*$. This gives an injection of the solutions $\te$ of the harmonic equation on $X$ onto $\ker i^q_*$, explicitly determined by the restriction on $Y$ and the isometry $\Psi$. Since this forms satisfy the boundary condition at the tip of the cone defining the Hodge-Laplace operator on $X$, they are harmonic of $X$.

Next,  suppose $\te_\SS$ is of type  E, O, I or IV. Using the Green formula on the cone and on $Y$, Propositions \ref{green cone1}, and \ref{greenY},
\begin{align*}
0=\langle \Delta^{(q)}_{X,\pf}\theta,\theta \rangle
=&\|d\te_C\|^2+\|d^\da_C\te_C\|^2+\|d \te_Y\|^2+\|d^\da_Y \te_Y\|^2,
\end{align*}
because the mixed  terms vanished since the forms on the section are co exact,  the boundary terms cancel each other, and the terms at $x=0$ vanish since the harmonics belong to the domain of the operator. This means that solutions of type $I$ and $IV$ can not be harmonics. So there remains solutions of type $E$ and $O$ with the further that the coefficients are constant functions, namely that $f_1=1$, and $f_2=h^{2\al_{q-1}-1}$. These solutions are on the cone:
\begin{align*}
&\left\{\begin{array}{l}\theta_{E,0,\pm}^{(q)}(x,0)=\L_{|\alpha_q|,\alpha_q,\pm}(x,0)h^{\al_q-\frac{1}{2}}(x)\tilde\vv^{(q)}_{\rm har}, \\\left(\L_{|\alpha_q|,\alpha_q,\pm}(x,0)h^{\al_q-\frac{1}{2}}(x)\right)'=0, \end{array}\right.\\
&\left\{\begin{array}{l}\theta^{(q)}_{O,0,\pm}(x,0)=\L_{|\alpha_{q-2}|,-\alpha_{q-2},\pm}(x,0)h(x)^{\alpha_{q-1}-\frac{1}{2}} \d x\wedge\tilde\vv_{\rm har}^{(q-1)}, \\
\left(\L_{|\alpha_{q-2}|,-\alpha_{q-2},\pm}(x,0)h(x)^{\frac{1}{2}-\alpha_{q-1}}\right)'=0.
\end{array}\right.
\end{align*}

It is now convenient to distinguish odd and even dimensions.

Assume $m=2p-1$. In the first case, the condition  
\[
\L_{|\alpha_q|,\alpha_q,\pm}(x,0)=h^{\frac{1}{2}-\alpha_q},
\]
since $\L_{|\alpha_q|,\alpha_q,\pm}(x,0)$ must be square integrable, requires $q\leq p-1$. So  there are no harmonics of this type when $q\geq p$. Next, recalling Remark \ref{nu=alpha}, the condition above 
gives that $\L_{|\alpha_q|,\alpha_q,\pm}(x,0)=\L_{|\alpha_q|,\alpha_q,+}(x,0)$ if $q\leq p-1$, while 
$\L_{|\alpha_q|,\alpha_q,\pm}(x,0)=\L_{|\alpha_q|,\alpha_q,-}(x,0)$ if $q\geq p$.  So, when $q\leq p-1$, $\L_{|\alpha_q|,\alpha_q,\pm}(x,0)=\L_{|\alpha_q|,\alpha_q,+}(x,0)$ is square integrable and satisfies the $+$ boundary condition at $x=0$, so it belongs to the domain of $\Delta_{X,\pf}^{(q)}$, and is an harmonic of $\Delta_{X,\pf}^{(q)}$. 
Observe that this gives the form
\[
\theta_C=\theta_{E,+}^{(q)}(x,0)=\L_{|\alpha_q|,\alpha_q,+}(x,0)h^{\al_q-\frac{1}{2}}(x)\tilde\vv^{(q)}_{\rm har}=\tilde\vv^{(q)}_{\rm har},
\]
that is also the restriction on $\te_\CC$ of the form $\te_Y$ that satisfies the absolute boundary condition for $\Delta^{(q)}_Y$, and this shows that the form $\theta_Y$ itself satisfies this condition, namely that $\te_Y\in \H^q_{\rm abs}(Y)$. Moreover, the restriction of the tangential component is an harmonic form on $W$. Since any closed form on $W$ is either harmonic or exact, combining this injection with the one described above, we have established, for $0\leq q\leq p-1$,  a bijection
\begin{align*}
j^q&:  \H^q_{\rm \pf}(X)\to \H^q_{\rm abs}(Y),\\
j^q&:\theta\mapsto   \theta|_Y.
\end{align*}

In the second  case,  the condition  
\[
\L_{|\alpha_{q-2}|,-\alpha_{q-2},\pm}(x,0)=h^{\alpha_{q-1}-\frac{1}{2}}=h^{\frac{1}{2}-(-\al_{q-2})},
\] 
since $\L_{|\alpha_{q-2}|,-\alpha_{q-2},\pm}(x,0)$ must be square integrable, requires $q\geq p+1$. So  there are no harmonics of this type when $q\leq p$. Next, recalling Remark \ref{nu=alpha}, the condition above 
gives that $\L_{|\alpha_{q-2}|,-\alpha_{q-2},\pm}(x,0)=\L_{|\alpha_{q-2}|,-\alpha_{q-2},+}(x,0)$ if $q\geq p+1$, while 
$\L_{|\alpha_{q-2}|,-\alpha_{q-2},\pm}(x,0)=\L_{|\alpha_{q-2}|,-\alpha_{q-2},-}(x,0)$ if $q\leq p$. 
So, when $q\geq p+1$, $\L_{|\alpha_{q-2}|,-\alpha_{q-2},+}(x,0)$ is square integrable and satisfies the $+$ boundary condition at $x=0$, so it belongs to the domain of $\Delta_{X,\pf}^{(q)}$, and is an harmonic of $\Delta_{X,\pf}^{(q)}$. 
Observe that this gives the form
\[
\theta_C=\theta_{O,+}^{(q)}(x,0)=h^{\al_{q-1}-\frac{1}{2}}(x)\L_{|\alpha_{q-2}|,-\alpha_{q-2},+}(x,0)dx\wedge \tilde\vv^{(q-1)}_{\rm har}=h^{2\al_{q-2}+1}dx \wedge\tilde\vv^{(q-1)}_{\rm har},
\]
that is also the restriction on $\te_\CC$ of the form $\te_Y$ that satisfies the relative boundary condition for $\Delta^{(q)}_Y$, and this shows that the form $\theta_Y$ itself satisfies this condition, namely $\te_Y\in \H^q_{\rm rel}(Y)$. Moreover, the restriction of the normal  component is an harmonic form on $W$. Since any closed form on $W$ is either harmonic or exact, combining this injection with the one described above, we have established,  for $p+1\leq q\leq 2p-1$, a bijection
\begin{align*}
j^q&:  \H^q_{\rm \pf}(X)\to \H^q_{\rm rel}(Y),\\
j^q&:\theta\mapsto   \theta|_Y.
\end{align*}

Next, consider $m=2p$. We may proceed as in the odd case $m=2p-1$ for all $q$ except that when $q=p$ and $q =p+1$, so we discuss only these two cases. The forms of type $E$ satisfying  these conditions are
\begin{align*}
&\left\{\begin{array}{l}\theta_{E,0,\pm}^{(p)}(x,0)=\L_{|\alpha_p|,\alpha_p,\pm}(x,0)\tilde\vv^{(q)}_{\rm har}, \\
\left(\L_{|\alpha_p|,\alpha_p,\pm}(x,0)\right)'=0, 
\end{array}\right.
&\left\{\begin{array}{l}\theta_{E,0,\pm}^{(p+1)}(x,0)=\L_{|\alpha_{p+1}|,\alpha_{p+1},\pm}(x,0) h(x) \tilde\vv^{(q)}_{\rm har}, \\
\left(\L_{|\alpha_{p+1}|,\alpha_{p+1},\pm}(x,0) h(x) \right)'=0, \end{array}\right.\\
\end{align*}
where $\al_p = \frac{1}{2}$ and $\alpha_{p+1} = \frac{3}{2}$. The conditions tell us that 
\[
\L_{|\alpha_p|,\alpha_p,\pm}(x,0) = 1,\qquad \text{and }\qquad \L_{|\alpha_{p+1}|,\alpha_{p+1},\pm}(x,0) = h(x)^{-1}.
\]

Since $\L_{|\alpha_p|,\alpha_p,\pm}(x,0)$ and $\L_{|\alpha_{p+1}|,\alpha_{p+1},\pm}(x,0)$ must be square integrable, we do not have forms of this type in dimension $q=p+1$ and  we have only the form of this type in dimension $q=p$. Using the remark \ref{nu=alpha} we have only the minus solution. Therefore, this form is in the domains of $\Delta_{X,\mf}^{(p)}$ (these forms satisfy the minus boundary condition at $x=0$) and if we are considering the plus boundary condition at $x=0$, i.e., forms in the domain of  $\Delta_{X,\mf^c}^{(p)}$ we do not have harmonic of this type in dimension $q=p$. Note that, $\theta_{E,0,-}^{(p)}(x,0)$ satisfy the absolute boundary condition at $x=l$.

The forms of type $O$ satisfying  these conditions are
\begin{align*}
&\left\{\begin{array}{l}\theta^{(p)}_{O,0,\pm}(x,0)=\L_{|\alpha_{p-2}|,-\alpha_{p-2},\pm}(x,0)h(x)^{\alpha_{p-2}+\frac{1}{2}} \d x\wedge\tilde\vv_{\rm har}^{(p-1)}, \\
\left(\L_{|\alpha_{p-2}|,-\alpha_{p-2},\pm}(x,0)h(x)^{-\frac{1}{2}-\alpha_{p-2}}\right)'=0.
\end{array}\right.\\
&\left\{\begin{array}{l}\theta^{(p+1)}_{O,0,\pm}(x,0)=\L_{|\alpha_{p-1}|,-\alpha_{p-1},\pm}(x,0)h(x)^{\alpha_{p-1}+\frac{1}{2}} \d x\wedge\tilde\vv_{\rm har}^{(p)}, \\
\left(\L_{|\alpha_{p-1}|,-\alpha_{p-1},\pm}(x,0)h(x)^{-\frac{1}{2}-\alpha_{p-1}}\right)'=0.
\end{array}\right.
\end{align*}
where $\al_{p-1} = -\frac{1}{2}$ and $\alpha_{p-2} =- \frac{3}{2}$. The conditions tell us that 
\[
\L_{|\alpha_{p-2}|,-\alpha_{p-2},\pm}(x,0) = h(x)^{-1},\qquad \text{and }\qquad \L_{|\alpha_{p-1}|,-\alpha_{p-1},\pm}(x,0) = 1.
\]

Since $\L_{|\alpha_{p-2}|,-\alpha_{p-2},\pm}(x,0)$ and $\L_{|\alpha_{p-1}|,-\alpha_{p-1},\pm}(x,0)$ must be square integrable, we do not have forms of this type in dimension $q=p$ and  we have only the form of this type in dimension $q=p+1$. Using the remark \ref{nu=alpha} we have only the minus solution. Therefore this form is in the domains of $\Delta_{X,\mf^c}^{(p+1)}$ (these forms satisfy the minus boundary condition at $x=0$) and if we are considering the plus boundary condition at $x=0$, i.e., forms in the domain of  $\Delta_{X,\mf}^{(p+1)}$, we do not have harmonic of this type in dimension $q=p+1$. Note that, $\theta_{O,0,-}^{(p+1)}(x,0)$ satisfy the relative boundary condition at $x=l$. So we have in the case $m=2p$ bijections similar to those of the case $m=2p-1$. \end{proof}

\section{The De Rham complex on the cone}
\label{drcone}
 
Consider the de Rham complex $(\Omega^\bu(C_{0,l}(W)), \d)$ described in Section \ref{hodge}. Recall, we have the decomposition of $\d$ and $\d^{\dag}$  in $C^\infty ((0,l], \Omega^{q}( W)\times \Omega^{q-1}(W))$:
\begin{align*}
\df^q= \ad_{q+1} \d^q~ \ad_q^{-1} &= 
\frac{1}{h(x)}\left(\begin{matrix} \tilde\d^q  & 0 \\
\left(\al_q-\frac{1}{2}\right)h'(x)+h(x)\frac{d}{d x}& -\tilde\d^{q-1}\end{matrix}\right)\\
&=\left(\begin{matrix} 0&0\\1&0\end{matrix}\right)\frac{d}{dx}
+\left(\begin{matrix} \frac{1}{h(x)}\tilde\d^q  & 0 \\ \left(\al_q-\frac{1}{2}\right)\frac{h'(x)}{h(x)}& -\frac{1}{h(x)}\tilde\d^{q-1}\end{matrix}\right),
\end{align*}
and
\begin{align*}
(\df^\da)^q = \ad_{q-1} (\d^\da)^q \ad_q^{-1} &= 
\frac{1}{h(x)}\left(\begin{matrix} 
(\tilde{\d}^{\da})^q & 
\left(\al_{q-1}-\frac{1}{2}\right)h'(x)-h(x)\frac{d}{d x}\\
0 & -(\tilde\d^\dag)^{q-1}
\end{matrix}\right)\\
&=-\left(\begin{matrix} 0&0\\1&0\end{matrix}\right)\frac{d}{dx}
+\left(\begin{matrix} \frac{1}{h(x)}(\tilde{\d}^{\da})^q  & \left(\al_{q-1}-\frac{1}{2}\right)\frac{h'(x)}{h(x)} \\ 
0& -\frac{1}{h(x)}(\tilde{\d}^{\da})^{q-1}\end{matrix}\right).
\end{align*}

These operators are  each other formal adjoints, according to the \cite[XIII.2.1, pg. 1287]{DS2}, noting that the last matrix in the definition of the $\df^\da$ is the transposed of the one in the definition of $\df$. They are formal adjoints as well according to \cite[pg. 67]{Wei}, since
\beq\label{formaladj}
\begin{aligned}
\langle \df(\omega_1,\omega_2), (\vv_1,\vv_2)\rangle_{I_{a,b}}
-&\langle (\omega_1,\omega_2), \df^\da(\vv_1,\vv_2)\rangle_{I_{a,b}}\\
&=\int_0^l \langle (\omega_1\vv_1)'(x)\rangle_W dx
+\int_0^l \langle (\omega_2\vv_2)'(x)\rangle_W dx,
\end{aligned}
\eeq
vanishes on the suitable domains. Separating the variables, $\omega_j=u_j \al_j$, $\vv_j=v_j \be_j$, the equation above reads
\begin{align*}
\langle \df(u_1\al_1,u_2\al_2),& (v_1\be_1,v_2\be_2)\rangle_{I_{a,b}}
-\langle (u_1\al_1,u_2\al_2), \df^\da(v_1\be_1,v_2\be_2)\rangle_{I_{a,b}}\\
&=\langle \al_1,\be_1\rangle_W \int_0^l (u_1v_1)'(x) dx  
+\langle \al_2,\be_2\rangle_W \int_0^l (u_2v_2)'(x)dx\\
&=\lim_{\ep\to 0^+}\left(\langle \al_1,\be_1\rangle_W \left[u_1(x)v_1(x)\right]_\ep^l 
+\langle \al_2,\be_2\rangle_W \left[u_2(x)v_2(x)\right]_\ep^l\right).
\end{align*}

We need a concrete definition of the operator $\df$ and of its dual. Define the minimal and the maximal operators associated to $\df$ (and $\df^\da$) as follows. Let $\tilde \d$ and $\tilde \d^\da$ denote the  usual closed extensions of the exterior derivative operator and of its dual on a  compact manifold. So it remains to define the domain of the operators acting on the line segment. We set:
\begin{align*}
D(\DF_{\rm min}^q)&=D((\DF^\da)_{\rm min}^q)= C_0^\infty ((0,l), D(\tilde\d^q)\times D(\tilde \d^{q-1})), \\
D(\DF_{\rm max}^q)&=\left\{ u\in A^2 ((0,l],D(\tilde\d^q)\times D(\tilde \d^{q-1})~|~u,\df u\in L^2 ((0,l], D(\tilde\d^q)\times D(\tilde \d^{q-1})\right\},\\
D((\DF^\da)_{\rm max}^q)&=\left\{ u\in A^2 ((0,l],D(\tilde\d^\da{}^q)\times D(\tilde \d^\da{}^{q-1})~|~u,{\df^\da}^q u\in L^2 ((0,l], D(\tilde\d^\da{}^q)\times D(\tilde \d^\da{}^{q-1})\right\}.
\end{align*}

We denote by $\d_{\rm min}^q$, $(\d^\da)^q_{\rm min}$,  $\d^q_{\rm max}$, and $(\d^\da)^q_{\rm max}$ the corresponding minimal and maximal operators associated to $\d^q$ and $(\d^\da)^q$. All these operators are densely defined.

It is clear by equation (\ref{formaladj}), that  $\df$ and  $\df^\da$ are each other formal adjoints, more precisely
\[
(\df^q)^\da=(\df^\da)^{q+1}.
\]  

The same is true for $\d$ and $\d^\da$:
\[
(\d^q)^\da=(\d^\da)^{q+1}.
\]

\begin{lem}\label{extder} We have that $ (\DF^\da)_{\rm max}^{q+1}=(\DF_{\rm min}^q)^\da$, i.e. $(\DF^\da)^{q+1}_{\rm max}$ is the adjoint of $\DF_{\rm min}^q$, and $\DF_{\rm max}^{q+1}=((\DF^\da)_{\rm min}^q)^\da$, i.e. $\DF_{\rm max}^{q+1}$ is the adjoint of $(\DF^\da)^q_{\rm min}$. 
Similarly: $(\d^\da)_{\rm max}^{q+1}=(\d_{\rm min}^q)^\da$, and $(\d^\da)_{\rm max}^{q+1}=((\d^\da)_{\rm min}^q)^\da$.
\end{lem}

\begin{proof} By definition, since the operators $\tilde \d$ and $\tilde \d^\da$ are each other adjoint and are closed, it just remains to prove that the result holds for the operators on the line. This follows for example by \cite[XIII.2.10, pg. 1294]{DS2}. More precisely, using equation  (\ref{formaladj}), if $(\omega_1,\omega_2)\in D((\DF^\da)_{\rm min})$, then
\[
\langle \df^\da(\omega_1,\omega_2), (\vv_1,\vv_2)\rangle_{I_{a,b}}
-\langle (\omega_1,\omega_2), \df(\vv_1,\vv_2)\rangle_{I_{a,b}}=0,
\]
for all $(\vv_1,\vv_2)\in D(\DF_{\rm max})$. So $(\DF^\da)_{\rm min})$ and $\DF_{\rm max}$ are formal adjoints. Whence, $(\DF^\da)_{\rm min}\subseteq \DF_{\rm max}^\da$. Mutans mutandis, we could have showed that $\DF_{\rm min})$ and $(\DF^\da)_{\rm max}$ are formal adjoints. Whence, $\DF_{\rm min}\subseteq (\DF^\da)_{\rm max}^\da$. Taking the adjoint  
$ \DF_{\rm max}\subseteq \DF_{\rm max}^{\da\da}\subseteq (\DF^\da)_{\rm min}^\da$, and 
$ (\DF^\da)_{\rm max}\subseteq(\DF^\da)_{\rm max}^{\da\da}\subseteq \DF_{\rm min}^\da$, 
by \cite[4.13 and pg. 72]{Wei}, since all operators are densely defined.

It remains to prove the converse implication, namely that $(\DF^\da)_{\rm min}^\da \subseteq  \DF_{\rm max}$ (and that 
$\DF_{\rm min}^\da\subseteq (\DF^\da)_{\rm max}$). Essentially, this is equivalent to prove that the maximal operators are indeed maximal. The proof is technical and we omit it here. 
It consists in showing that given any square integrable pair $(\omega_1, \omega_2)$ there exists a square integrable pair $(\varphi_1, \varphi_2)$ such that for all $(\psi_1,\psi_2)\in (\DF^\da)_{\rm min}$
\[
\langle (\omega_1,\omega_2), \df^\da (\psi_1,\psi_2)\rangle_{I_{a,b}}
-\langle (\vv_1,\vv_2), (\psi_1,\psi_2)\rangle_{I_{a,b}}=0.
\]

Then we may follow the same steps as in \cite[XIII.2.10, pg. 1294]{DS2} or \cite[6.29]{Wei}.
\end{proof}

\begin{corol} The operators  $\DF^\bu_{\rm max}$, $(\DF^\da)^\bu_{\rm max}$, and $\d_{\rm max}^\bu$, $(\d^\da)_{\rm max}^\bu$ are closed.
\end{corol}


We want to define a closed extension of the minimal operators. 
Let $\bar D^q$ and $ (\bar D^\da)^q$ denote any two closed extensions of the operators 
${\d}^q$ and ${\d^\da}^q$, respectively. Such extensions exist by Lemma \ref{extder} \cite[5.3, 5.4]{Wei}, and all lie between the closure of the minimal and the maximal extensions.

\begin{lem} \label{complex} For all $q$, $R(\bar D^q)\subseteq D(\bar D^{q+1})$, 
$R((\bar D^\da)^{q+1})\subseteq D((\bar D^\da)^q)$, where $R$ denotes the range, and
\[
\bar D^{q+1}\bar D^q=(\bar D^\da)^q(\bar D^\da)^{q+1}=0.
\]
\end{lem}

\begin{proof} Take $u\in D(\bar D^q)$. Since $\bar D^q$ is a densly defined closed extension of $^q_{\rm min}$, there exists a sequence $u_n$ in $C^\infty_0(0,l)$ such that $u_n\to u$ in the norm of $L^2(0,l)$, and $\bar D^q(u)=\lim_{n} \d^q u_n$. But $\d^q u_n=D^q u_n$ belongs to the range of $D^q$ and to the domain of $\bar D^{q+1}$. Since the last is closed, it follows that the limit too belong to the domain namely that $D^q u\in D(\bar D^{q+1})$. The fact that $\bar D^{q+1}\bar D^q=0$, follows immediately by the definition. The proof for the other operator is the same. 
\end{proof}

This shows that the complexes $(D^\bu, D(D^\bu)$ and  $((D^\da)^\bu,D((D^\da)^\bu))$ are Hilbert complexes in the sense of \cite{BL0}. We may now associate a Laplace operator to these complexes following \cite{BL0}. Decomposing this complex in its even and odd part, $H_{\rm ev}=\bigoplus_{q=0}^{m+1} L^2(\Omega^{2q}(C_{a,b}(W))$, 
$H_{\rm odd}=\bigoplus_{q=0}^{m+1} L^2(\Omega^{2q+1}(C_{a,b}(W))$, we have the closed operator
\begin{align*}
\bar D:&D(\bar D)\subseteq H_{\rm ev}\to H_{\rm odd},\\
\bar D:&(u_0, u_2, \dots, u_{m+1})\mapsto (\bar D^0 u_0+(\bar D^1)^\da u_1, \bar D^0 u_0+(\bar D^1)^\da u_1, \dots),
\end{align*}
with domain
\[
D(\bar D)=\bigoplus_{q=0}^{m+1} D(\bar D^{2q})\cap D((\bar D^{2q-1})^\da).
\]

Let $\bar D^\da$ the adjoint of $\bar D$. Then, we have the following self-adjoint operator
\begin{align*}
\bar \Delta:&D(\bar \Delta)\subseteq H_{\rm ev}\oplus H_{\rm odd},\\
\bar \Delta:&(u_0, u_1, \dots)\mapsto (\bar D^\da\bar D\oplus \bar D\bar D^\da) (u_0, u_1, \dots, u_{m+1}).
\end{align*}
with domain
\[
D(\bar \Delta)=(D(\bar D)\cap D(\bar D^\da))\oplus (D(\bar D^\da)\cap D(\bar D).
\]

\begin{prop}\label{laplap} The operator $\bar \Delta$ is a self-adjoint extension of the minimal operator 
$\Delta_{\rm min}=\Delta^{(\bu)}_{\rm min}$ define in Section 4.6.
\end{prop}
\begin{proof} A direct verification shows that, even if the operator $\bar \Delta$ is not graded, the associated formal operator is graded, and it is the operator
\[
(\d^\da)^{q+1}\d^q+\d^{q-1}(\d^\da)^q=\Delta^{(q)}.
\]
\end{proof}

The self-adjoint extensions of the operator $\Delta^{(\bu)}_{\rm min}$ have been described in Proposition \ref{selfext}. By Proposition \ref{laplap},  there exist  closed extensions of the operator $\d^\bu_{\rm min}$ whose Laplace operator $\bar \Delta$ coincides with the operators $\Delta_{\rm \mf, bc}$ and $\Delta_{\rm \mf^c, bc}$ introduced in Definition \ref{HodgeLaplace}. 

\begin{defi} We denote by $\d^\bu_{\rm max, \mf}$, $\d^\bu_{\rm max, \mf^c}$, $\d^\bu_{\rm min, \mf}$, and $\d^\bu_{\rm min, \mf^c}$ the closed extensions of the operator $\d^\bu_{\rm min}$ whose induced Laplacian are the operators $\Delta^{(\bu)}_{\rm abs, \mf}$, $\Delta^{(\bu)}_{\rm abs, \mf^c}$, $\Delta^{(\bu)}_{\rm rel, \mf}$, and $\Delta^{(\bu)}_{\rm rel, \mf^c}$.

\end{defi}

\begin{rem}
In the following we write $\pf$ for either $\mf$ or $\mf^c$, and ${\rm mm}$ for either ${\rm max}$ or ${\rm min}$. 
\end{rem}
By the result of Section \ref{secspectral}, the heat operator of $\Delta^{(\bu)}_{\pf,\rm mm}$ is of trace class and has a full asymptotic expansion of the form described in \cite{Les2}, whence we have the following result (the definition of operator with discrete dimension spectrum is as well in \cite{Les2}).

\begin{prop}\label{p18.6} 
The complex $(\d^{(\bu)}_{ \rm mm,\pf}, D(\d^{(\bu)}_{\rm  mm,  \pf}))$ is an Hilbert complex of discrete dimension spectrum.
\end{prop}

\begin{rem}\label{open1} Observe that the results of this section were performed on the space $C_{0,l}(W)=(0,l]\times W$. However, it is clear that everything works exactly in the same way on the space $C_{0,l}(W)=(0,l]\times W$.
\end{rem}

According to the definition in \cite{Les2}, the above result may be also stated saying that $\Delta^{(\bu)}_{\rm bc,\pf}$ is an ideal boundary condition with discrete dimension spectrum for the De Rham complex $(\Omega^\bu(\oo C_{0,l}(W))$.

\section{Analytic Torsion and the glueing formula}
\label{XX}

By the results of Section \ref{spaceX},  $\Sp_0 (\Delta_{X,\pf}^{(q)})$ is a sequence of real positive numbers with unique accumulation point at infinity, and since some power of the resolvent is of trace class, the logarithmic Gamma function 
$\log \Gamma(\la,\Sp_0 (\Delta_{X,\pf}^{(q)}))$ is well defined \cite{Spr9}. Since the heat operator has an asymptotic expansion for small $t$, it follows that the logarithmic Gamma function has an asymptotic expansion for large Gamma. In particular, the associated zeta function is well defined \cite{Spr9}, 
\[
\zeta(s,\Delta_{X,\pf}^{(q)})=\sum_{\la\in \Sp_0 (\Delta_{X,\pf}^{(q)})} \la^{-s},
\]
and has an analytic continuation regular  at $s=0$. We define the torsion zeta function
\[
t_{X,\pf}(s)=\sum_{q=0}^{m+1} (-1)^q q \zeta(s,\Delta_{X,\pf}^{(q)}),
\]
and the analytic torsion by
\[
\left. \log T_{\pf}(X)=\frac{d}{d s}t_{X,\pf}(s)\right|_{s=0}.
\]

In the rest of this section we assume that the function $h$ is constant with value $1$ in an open neighbourhood of $l$ (as in Section \ref{spaceX}, we consider $h$ extended on $(0,l+\iota)$). Recall that
\[
X=C_{0,l}(W)\cup_{W} Y.
\]

Denote for simplicity $X_+=Y$, $X_-=C_{0,l}(W)$. We introduce some complexes on these spaces.

On $X_+$, let $(\Omega^\bu(X_+),\d^\bu_{X^+})$ be the classical De Rham complex. Let $\d^\bu_{X_+, {\rm rel}}$ and $\d^\bu_{X_+, {\rm abs}}$ denote the classical extensions of the minimal operator induced by $\d^\bu_{X_+}$, and $\Delta^{(\bu)}_{X_+,{\rm abs}}$ and $\Delta^{(\bu)}_{X_+,{\rm rel}}$ the classical Laplacians with absolute and relative boundary conditions, respectively, as defined for example in \cite[2.7.1]{Gil}. Observe that all this construction is the same if developed on the interior $\mathring X_+$ of $X_+$ (compare \cite{CY} and \cite{Gaf2}). 
On the other side, let $\d^\bu_{\oo X_+, {\rm min}}$ and $\d^\bu_{\oo X_+, {\rm max}}$ denote  (closure of the) minimal and maximal operators associated to the formal operator $\d^\bu_{X_+}$, as defined  in \cite[3.1]{BL0}. Note that these operators coincide with the ones defined in equation (3-25) of \cite{Les2}, with $\vv=0$.  Let $\Delta^{(\bu)}_{\oo X_+,{\rm min}}$ and $\Delta^{(\bu)}_{\oo X_+,{\rm max}}$ denote the corresponding Laplacians,  as defined in Section \ref{drcone} {\cite[(2.14a)]{BL0}. Then, by \cite[4.1]{BL0}, since the Laplacian is the square of the Gauss Bonnet operator, and since $\oo X_+$ is the interior of a smooth manifold, the Laplacian $\Delta^{(\bu)}_{\oo X_+,{\rm min}}$ and  $\Delta^{(\bu)}_{\oo X_+,{\rm max}}$ coincide with the classical Laplacian with absolute and relative boundary conditions $\Delta^{(\bu)}_{X_+,{\rm abs}}$ and $\Delta^{(\bu)}_{X_+,{\rm rel}}$, respectively, as defined above. Note that the exterior differential them selves do  not coincide, in particular the domain of the maximal extension is defined without imposing any boundary condition on the boundary.

By the identification of the Laplacian, it follows that the complexes  $\d^\bu_{\oo X_+, {\rm min}}$ and $\d^\bu_{\oo X_+, {\rm max}}$ are Hilbert complexes with discrete dimension spectrum \cite[pg. 239]{Les2}, and therefore with discrete dimension spectrum outside  $U_+=[l,l+\iota)\times W\subseteq \oo X_+$. 


On $X_-$, let $\d^\bu_{\oo X_-, \rm mm, \pf}$ denote the Hilbert complex with finite dimension spectrum defined in Section \ref{drcone}. By the result in Section \ref{ancone},  this complex has finite dimension spectrum on $X_-$ and therefore as well outside $U_-=(l-\iota,l]\times W\subseteq \oo X_-$. As for $X_+$,  the boundary condition for $\d^\bu_{\oo X_+, \rm mm,\pf}$ at $W$ coincides with the boundary condition introduced in \cite[3-25]{Les2}, but in this case with a function $\vv\in C^\infty((0,l))\times W$, which is  equal to $1$ near $x=0$ and equal to $0$ near $x=l$.

We conclude by defining a complex on $X$.  Denote by $\d^\bu_{1,1}$ the subcomplex of the direct sum of the two complexes $\d^\bu_{X_-, \rm bc,\pf}$ and $\d^\bu_{X_+, \rm bc}$ defined by 
\[
D(\d^q_{X, \pf, 1,1})=\left\{(u_1, u_2)\in D(\d^q_{\oo X_-, \rm mm, \pf})\oplus D(\d^q_{\oo X_+, \rm mm})~|~i_-^* u_1=i_+^*u_2\right\},
\]
where the $i_\pm:X_\pm\to X_+\sqcup X_-$ are the inclusions into the disjoint union. This is an Hilbert complex with finite dimension spectrum by construction  \cite[(4-5)]{Les2}. Let $\Delta^{(\bu)}_{X, 1,1}$ denote the Laplace operator associated to the complex $\d^\bu_{X,1,1}$, as defined in \cite{BL0}. 

The next result is an adaptation of Proposition 1.1 of \cite{Vis} to the actual setting (compare also with \cite{Les2}, end of page 240).

\begin{prop} \label{Vis} The Laplace operators $\Delta^{(\bu)}_{X, \pf, 1,1}$ and the Laplace operator $\Delta^{(\bu)}_{X,\pf}$ have the same spectral resolution.
\end{prop}

\begin{theo}\label{les} Let $X=C_{0,l}(W)\sqcup_W Y$ be a space with a conical singularity of dimension $n=m+1$,  where $(Y,W)$ is a compact connected orientable smooth Riemannian manifold of dimension $n$, with boundary $W$. Assume the metric $\g_0$ on $X$ is a product in collar neighbourhood of $W$.  
Let either $\pf=\mf$ or $\pf=\mf^c$. Then, 
\[
\log T_{\pf}(X,\g_0)=\log  T_{\rm abs, \pf}(C_{0,l}(W),\g_0|C)+\log T_{\rm rel}(Y,\g_0|Y)-\frac{1}{2}\chi(W)\log 2+\log \tau(\H_0),
\]
where $\H_0$ is the long homology exact sequence induced by the inclusion $C_{0,l}(W)\to X$, with some orthonormal graded basis. 
\end{theo}
\begin{proof} All the hypothesis necessary in order to apply Theorem 6.1 of \cite{Les2} have been verified. By definition, the analytic torsion $T_{\rm bc}(M)$ coincide with the torsion of the Hilbert complexes $\d^{\bu}_{\oo X,\rm min}$, and $\d^{\bu}_{\oo X,\rm max}$, respectively, and by Proposition \ref{Vis}, $T_{\pf}(X)$ coincides with the analytic torsion of the Hilbert complex $\d^q_{X, \pf, 1,1}$, so the we have that
\[
\log T_{\pf}(X)=\log  T_{\rm abs, \pf}(C_{0,l}(W))+\log T_{\rm rel}(Y)-\frac{1}{2}\chi(W)\log 2+\log \tau(\H).
\]

The graded basis in the long homology sequence must be coherent with the graded bases used for analytic torsion, that are any orthonormal ones. 
\end{proof}

Observe that this theorem holds true (with the necessary notational changes) for forms with coefficients in a flat bundle, and it will be used in the following in this more general setting, see the remarks at the beginning of Section \ref{hodgederhamX}.

\section{Variation of the Analytic Torsion}
\label{varX}

In this section we study the variation of  the analytic torsion of a space with conical singularity under a variation of the metric far from the singular point. Let $\g_{X,\mu}$ be a smooth family of metrics on $X$ such that the restriction on $Y$ is a smooth family of metrics on $Y$, and the restriction on $C_{0,l}(W)$ is
\[
\g_{W,\mu}|_C=\g_{h_\mu}=\d x\otimes \d x+h_\mu^2(x) \tilde g_\mu,
\]
where $h_\mu$ is a smooth family of functions on $[0,l]$ that coincide on $[0,l-\ep]$, for some positive $\ep<\iota$, and satisfy in $[0,l]$ all the hypothesis introduced in Section \ref{geocone}. Denote the analytic torsion of the induced Hodge-Laplace operators  by $T_{\pf}(X,\g_{X,\mu})$, respectively. Then, we have the following result, to be compared with \cite{RS}.

\begin{theo}\label{variationX} Let $X=C_{0,l}(W)\sqcup_W Y$ be a space with a conical singularity of dimension $n=m+1$,  where $(Y,W)$ is a compact connected orientable smooth Riemannian manifold of dimension $n$ with boundary $W$. Let either $\pf=\mf$ or $\pf=\mf^c$.  Then, 
\begin{align*}
\frac{\d}{\d\mu}\log T_{\pf}(X,\g_{X,\mu})= -\frac{1}{2} \log \|\Det I^\pf \alphas_{X,\g_{X,\mu}, \bullet}\|_{\Det \H^{\bullet}_{\pf}(X,\g_{X,\mu};V_{\rho})}^2
\end{align*}
where $I^\pf \alphas_{X,\g_{X,\mu},q}$ is an orthonormal basis of $\H^{(q)}_{\pf}(X,\g_{X,\mu};V_{\rho})$.
\end{theo}

\begin{proof}  The proof is essentially contained in \cite{RS}.  Let $\Delta^{(\bu)}_{X, \pf, \mu}$ denote the operator for the metric $\g_{X,\mu}$. Consider the variation 
\[
\frac{\partial}{\partial\mu} {\rm Tr} (\e^{-t \Delta^{(q)}_{X, \pf, \mu}}).
\]

By Corollary \ref{heattraceX}, the heat operator has an asymptotic expansion, for small $t$
\[
\Tr \e^{-t\Delta^{(q)}_{X, \pf, \mu}}=\sum_{k=0}^\infty a_k(\Delta^{(q)}_{X, \pf, \mu}) t^\frac{m+1-k}{2},
\]
where 
\[
a_k(\Delta^{(q)}_{X, \pf, \mu})=\int_X e_k(\Delta^{(q)}_{X, \pf, \mu}) dvol_X.
\]

Now observe that for all $\mu$, $\Delta^{(q)}_{X, \pf, \mu}=\Delta^{(q)}_{X, \pf,0}$ in $\UU=(0,l-\ep]\times W$. So, for small $t$, 
\[
\frac{\partial}{\partial\mu} {\rm Tr} (\e^{-t \Delta^{(q)}_{X, \pf, \mu}})
=\sum_{k=0}^\infty  \tilde a_k(\Delta^{(q)}_{X, \pf, \mu}) t^\frac{m+1-k}{2},
\]
where 
\[
\tilde a_k(\Delta^{(q)}_{X, \pf, \mu})=\int_\UU \tilde e_k(x,\Delta^{(q)}_{X, \pf, \mu}) dvol_X.
\]

Now $X-\UU$ may be considered inside $Y$, and therefore the local invariants are just those of the restriction of the kernel on $Y$, namely 
\[
\tilde e_k(x,\Delta^{(q)}_{X, \pf, \mu})=e_k(x, \Delta^{(q)}_{Y,  \mu}). 
\]

For these invariants, we have that all the ones with odd index vanish, and duality, i.e. \cite[Section 2.5.2]{Gil}
\[
a_k(\Delta^{(m+1-q)}_{Y, \pf, \mu})=a_k(\Delta^{(q)}_{Y, \pf, \mu}).
\]

Note that this is the interior contribution to the coefficients of the asymptotic expansion for the trace of the heat kernel. There exists also a contribution from the boundary, that does not satisfy these properties, but it is not relevant here.

Note also that looking at $\UU$ as a subspace of the cone, we may proceed to a similar analysis, and arrive to an explicit description of the heat kernel coefficients. Since duality follows from commutativity between Hodge star and Laplacian, on the cone this would give duality between the coefficients in the asymptotic expansion of the kernel of the heat operator associated to $\Delta^{(\bu)}_{X, \mf, \mu}$ and those associated to $\Delta^{(\bu)}_{X, \mf^c, \mu}$. However, this would not compromise the result in the present case, since this contribution would depend on the integration of the local coefficients on the line near $x=0$, and in the present case this contribution vanishes.

This proves the theorem, by the same argument as in \cite{RS}. In detail, since

\[
\frac{\partial}{\partial\mu} {\rm Tr} (\e^{-t \Delta^{(q)}_{X, \pf, \mu}}) = - t {\rm Tr} \left( \left(\frac{\partial}{\partial \mu} 
\Delta^{(q)}_{X, \pf, \mu}\right) \e^{-t \Delta^{(q)}_{X, \pf, \mu}}\right).
\]
proceeding formally as in \cite[pg. 152-153]{RS},
\begin{equation*}
\begin{aligned}
\sum_{q=0}^{m+1} (-1)^{q+1} q {\rm Tr} \left( \left(\frac{\partial}{\partial \mu} 
\Delta^{(q)}_{X, \pf, \mu}\right) \e^{-t \Delta^{(q)}_{X, \pf, \mu}}\right)
 &= \sum_{q=0}^{m+1} (-1)^{q+1} {\rm Tr} (\beta^{(q)}_\mu \Delta^{(q)}_{X, \pf, \mu} \e^{-t \Delta^{(q)}_{X, \pf, \mu}})\\
&=\frac{\partial}{\partial t}  \sum_{q=0}^{m+1} (-1)^{q} {\rm Tr} (\beta^{(q)}_\mu \e^{-t \Delta^{(q)}_{X, \pf, \mu}}),
\end{aligned}
\end{equation*}
where
\[
\beta_\mu^{(q)}= (\star^{(q)}_\mu)^{-1}  \frac{\partial}{\partial_\mu} \star^{(q)}_\mu.
\]

By definition \cite[2.3]{Spr9}, the torsion zeta function of the family of operators $\Delta^{(q)}_{X, \pf, \mu}$ is
\[
 t_{X,\pf}(s,\mu)= \frac{1}{2} \frac{1}{\Gamma(s)}\sum_{q=0}^{n+1} (-1)^q q \int_0^\infty t^{s-1} {\rm Tr} \left(\e^{-t \Delta^{(q)}_{X, \pf, \mu}} - P_{\ker \Delta^{(q)}_{X, \pf, \mu}}\right) dt. 
\] 

Then 
\[
\begin{aligned}
\frac{\partial }{\partial\mu}  t_{X,\pf}(s,\mu) &= \frac{1}{2} \frac{1}{\Gamma(s)}\sum_{q=0}^{m+1} (-1)^q  \int_0^\infty  t^{s} \frac{\partial}{\partial t} {\rm Tr} \left(\beta_{\mu}^{(q)}\left(\e^{-t \Delta^{(q)}_{X, \pf, \mu}} - P_{\ker \Delta^{(q)}_{X, \pf, \mu}}\right)\right) dt\\
&= \frac{1}{2} \frac{s}{\Gamma(s)}\sum_{q=0}^{m+1} (-1)^q  \int_0^\infty  t^{s-1}  {\rm Tr} \left(\beta_\mu^{(q)} \left(e^{-t \Delta^{(q)}_{X, \pf, \mu}} - P_{\ker \Delta^{(q)}_{X, \pf, \mu}}\right)\right) dt.
\end{aligned}
\]

By definition of analytic torsion, using the asymptotic expansion of the heat kernel, 
\[
\frac{\b}{\b\mu} \log T_{\pf}(X,\g_{X,\mu})=\frac{\partial }{\partial\mu}   t'_{X,\pf}(0,\mu)
=\frac{1}{2} \sum_{q=0}^{m+1} (-1)^{q+1} \left(a_\frac{m+1}{2}(\Delta^{(q)}_{X, \pf, \mu})-
\Tr \beta_\mu^{(q)}  P_{\ker \Delta^{(q)}_{X, \pf, \mu}}\right).
\]

By the relation on the coefficients $a_k$, the first terms sum to zero, and we have

\[
\frac{\b}{\b\mu}\log T_{\pf}(X,\g_{X,\mu})=\frac{\partial }{\partial\mu}  t'_{X,\pf}(0,\mu)
=\frac{1}{2} \sum_{q=0}^{m+1} (-1)^q \Tr \beta_\mu^{(q)}  P_{\ker \Delta^{(q)}_{X, \pf, \mu}}.
\]
\end{proof}

\section{Hodge Theorem, De Rham maps and Ray-Singer intersection torsion for a space with conical singularities}
\label{hodgederhamX}

Let $X=C_{0,l}(W)\sqcup_W Y$ be a space with a conical singularity of dimension $n=m+1$,  where $(Y,W)$ is a compact connected orientable smooth Riemannian manifold of dimension $n$, with boundary $W$. Given a  representation $\rho:\pi_1(X)\to O(\R^k)$, let $E_\rho$ be the associated vector bundle, as defined at the beginning of Section \ref{torman}. Then, as observed at the beginning of Section \ref{cheegercone}, all the theory developed so far may be remade assuming forms with coefficients in $E_\rho$, i.e. $\Omega^{(q)}(X,E_\rho)$, simply by changing the notation. In particular, the construction is funtorial under restriction. We will apply this notation in all the relevant objects. Observe that the restrictions of the representation $\rho$ onto the cone $C_{0,l}(W)$ and onto $Y$ coincide actually with the trivial representation and a representation of the space $Y/W$, respectively, by Lemmas \ref{pi1} and \ref{pic2}.

\subsection{Hodge Theorem and  De Rham maps}

Using the isomorphisms between the intersection homology groups and the harmonic forms of the cone and of the section, we may prove an extension of the classical Hodge Theorem  and  define the De Rham maps from the spaces of harmonic forms  onto the intersection homology groups.

\begin{theo}\label{derham2} Let $X=C_{0,l}(W)\sqcup_W Y$ be a space with a conical singularity of dimension $n=m+1$, with metric $\g_X$, where $(Y,W)$ is a compact connected orientable smooth Riemannian manifold of dimension $n$, with boundary $W$.  Let $\rho:\pi_1(X)\to O(V)$ be an orthogonal representation on a $k$ dimensional vector space $V$. Then, there are chain maps that induce isomorphisms: 
\begin{align*}
I^\pf \A_{q}&: \Ha^\bu_{\pf}(X,V_{\rho})\to I^{\pf} H_\bu(X;V_{\rho}).
\end{align*}

These maps are called De Rham maps, and are described in the course of the proof.
\end{theo}
\begin{proof} Let $K=C_{0,l}(N)\sqcup_N M$ a coherent cell decomposition of $X$. The proof depends on the parity of the dimension, and on the perversity, we describe in details the even case $m=2p$ with $\pf=\mf$, the other cases are similar. Let $0\leq q\leq p-1$, then we can construct the following commutative diagram of isomorphisms

\centerline{
\xymatrix{\Ha^q_{ \mf}(X,V_{\rho})\ar[r]^{\star}&\Ha^{m-q+1}_{ \mf^c}(X,V_{\rho})\ar@{..>}[r]^{I^{\mf^c} \A^{m-q+1}}&I^{\mf^c} H^{m-q+1}(X;V_\rho)&I^\mf H_q(X;V_\rho)\ar[l]_{\hspace{40pt}I^\mf \QQ_{*,q}}\\
\Ha^q_{\rm abs}( Y,V_{\rho})\ar[r]^{\star_Y}\ar[u]_{k^*}&
\Ha^{m-q}_{\rm rel}( Y,V_{\rho})\ar[r]^{\A^{m-q}_{\rm rel}}\ar[u]_{ k^*}&H^{m-q}(Y,W;V_{\rho})\ar[u]&H_q( Y;V_{\rho})\ar[l]_{\hat\QQ_{*, q}}\ar[u]
}}

The map $k^*$ is induced by the inclusion and is an isomorphism by Proposition \ref{HarmonicX}. 
The map $\hat \QQ_{*,q}$ is standard Poincar\'e map, see Section \ref{dualpair}, and the map $I^\mf \QQ_{*,q}$ the corresponding extension for the singular case, see Proposition \ref{dualcone}. The vertical right arrows are the isomorphisms described in Proposition \ref{p2.3}, up to some coefficient. The coefficient depends on the cells as follows. On the cells that do not meet the boundary, the coefficient is just $1$. On the cells that meet the boundary, a renormalisation constant appears, in the proof of Theorem \ref{derham1}. This constant appears since the dual of the cells that are or meet the boundary of $Y$ are defined in the definition of the dual complex $K^*$ gluing with the cone on the boundary part. Thus, when evaluating the harmonic forms on these cells, the integral is the integral on the part of the cell in $Y$ plus the integral on the cone. 
The map $\A^{q}_{\rm rel}$ is the classical De Rham map, see Section \ref{secRS}.

Commutativity of the left and the right squares follows by definition. 
So closing the diagram defines the desired map:
\[
I^\mf\A_q=(I^\mf \QQ_{*,q})^{-1} I^{\mf^c} \A_{\rm rel}^{m-q+1} \star.
\] 
\end{proof}

\begin{rem}
Let $\omega_Y=\omega|_Y$ be the restriction of an harmonic form in $\H^q_\mf(X,V_\rho)$ onto $Y$. Then, $i^*(\omega_Y)=\omega$, and
\begin{align*}
\|\omega\|^2_X&=\int_X\omega\star\omega
=\int_{C_{0,l}(W)}\omega|_{C}\star_C\omega|_C+\int_Y\omega|_Y\star_Y\omega|_Y
=\|\omega|_C\|^2_C+\|\omega|_Y\|^2_Y\\
&=\gamma_q\|i_q^*( \omega|_C)\|_W^2+\|\omega|_Y\|^2_Y,
\end{align*}
where
\[
\gamma_q=\frac{\| k_q^*(\tilde \omega)\|^2_{C_{0,l}(W)}}{\|\tilde\omega\|^2_W}
=\int_0^l h^{m-2q}(x) dx.
\]
\end{rem}

\begin{corol}\label{hodgetX}(Hodge theorem) There exists a natural isomorphism
\[
\H^q_\pf(X,E_\rho)=I^\pf H^q(X;E_\rho).
\]
\end{corol}

\subsection{RS intersection torsion of a space with conical singularities}

Let $X$ be a space with a conical singularity, and dimension $n$, as defined in Section \ref{spaceX}. Then, $X$ is the smooth glueing of a manifold with boundary $(Y,W)$, with the cone over $W$. So, there exists a cellular decomposition $K$ of $X$, such that $K$ is an $n$ pseudomanifold with one isolated singularity (the tip of the cone). We have the standard decomposition $K=C(N_0)\sqcup_{N_0} M$.

The intersection cellular chain complex $I^\pf \CS_\bu(C(N);V_{\rho_0})$ is a complex of based vector spaces, with graded bases induced by the cells, by Lemma \ref{l7.10}. It follows that for any given graded homology basis, its R torsion is well defined and is independent on the cellular decomposition, see Theorem \ref{t7.2} and its corollary.

We proceed by assuming either $\pf=\mf$ or $\pf=\mf^c$, that $Y$ has a Riemannian structure $g_Y$, and $X$ the induced Riemannian structure as defined in Section \ref{spaceX}, so we write $X=C_{0,l}(W)\cup_W Y$. 

Then, using the $L^2$ theory of the Laplace operator on spaces with conical singularities developed in the second part of the work, we may define square integrable harmonic forms on $X$ with coefficients in $V_{\rho}$, i.e.  
$\H^\bu_{ \pf}(X,V_{\rho})=V\otimes_\rho \H^\bu_{\rm bc, \pf}(X)$ (see Section \ref{torman}). In particular,  we have the De Rham map 
\begin{align*}
I^\pf \A_{q}&:\H^\bu_{\rm abs, \pf}(X,V_{\rho})\to I^\pf H_q (X;V_{\rho}),
\end{align*}
that is an  isomorphism, see Theorem \ref{derham2}. So given an orthonormal basis $I^\pf\alphas_q$ for $\H^\bu_{ \pf}(X,V_{\rho})$, we have an homology basis, and the RS torsion is well defined. 

\begin{defi}\label{torX} Let $X=C_{0,l}(W)\sqcup_W Y$ be a space with a conical singularity of dimension $n=m+1$, with metric $\g_X$, where $(Y,W)$ is a compact connected orientable smooth Riemannian manifold of dimension $n$, with boundary $W$.  Let $\rho:\pi_1(X)\to O(V)$ be an orthogonal representation on a $k$ dimensional vector space $V$. Let either $\pf=\mf$ or $\pf=\mf^c$. 
Let $K=M\sqcup_{N_0} C(N_0)$ be any coherent cellular decomposition of $X$. We call {\it intersection RS torsion of $X$ with perversity $\pf$ with respect to the representation $\rho$}, the positive real number
\[
I^\pf \tau_{\rm RS}(X;V_{\rho})=\tau_{\rm R} (I^\pf \CS_\bu(K;V_{\rho});I^\pf\A_{\bu}(I^\pf \alphas_\bu)), 
\]
where  $I^\pf \alphas_\bu$ is a graded orthonormal basis  for the harmonic forms $\H^\bu_{ \pf}(X,V_{\rho})$, and 
$I^\pf \A_{\bu}$ the De Rham map, see Theorem \ref{derham2}. 
\end{defi}

Therefore we obtain the following theorem.

\begin{theo}\label{t7.37} Let $X=C_{0,l}(W)\sqcup_W Y$ be a space with a conical singularity of dimension $n=m+1$, with metric $\g_X$, where $(Y,W)$ is a compact connected orientable smooth Riemannian manifold of dimension $n$, with boundary $W$.  Let $\rho:\pi_1(X)\to O(V)$ be an orthogonal representation on a $k$ dimensional vector space $V$. Let either $\pf=\mf$ or $\pf=\mf^c$. Then,

\begin{align*}
I^\pf \tau_{\rm RS}(X;V_{\rho})
=& I^\pf \tau_{\rm RS}(C_{0,l}(W);V_{\rho})  \tau_{\rm RS}(Y,W;V_{\rho})\\
&\tau_{\rm R}(I^\pf \ddot\Ha;I^\pf\A_{{\rm abs},\bu}(I^\pf \alphas_{C,\bu}), I^\pf \A_\bu(I^\pf \alpha_{X,\bu}),\A_{Y,{\rm rel}, \bu}(\alpha_{Y,{\rm rel},\bu})),
\end{align*}
where $I^\pf \alphas_{C,q}$ is an orthonormal basis of $I^\pf \H^q(C(W);V_{\rho_0})$, $\alpha_{Y,{\rm rel},q}$ is an orthonormal basis of 
$\H^q_{\rm rel}(Y;V_{\rho''})$, and $I^\pf \alpha_{X,q}$ is an orthonormal basis of $\H^q(X;V_{\rho})$, $I^\pf \A_{{\rm abs},q}$ is the absolute De Rham map on $C_{0,l}(W)$ as defined in Theorem \ref{derham1},  $\A_{Y, {\rm rel}, q}$ the classical relative De Rham map on $Y$,  $I^\pf \A_q$ the De Rham map on $X$ defined in Theorem \ref{derham2}, and 
\[
\xymatrix{
I^\pf \ddot\Ha:&  \dots\ar[r]&I^{\pf} H_q(C_{0,l}(W);V_{\rho}) \ar[r]&I^\pf H_q(X;V_{\rho})\ar[r]&H_q(Y,W;V_{\rho})\ar[r]&\dots.
}
\]
\end{theo}
\begin{proof} Theorem \ref{t7.2}. \end{proof}

As consequence of Proposition \ref{p7.24} we have.

\begin{theo} Let $X=C_{0,l}(W)\sqcup_W Y$ be a space with a conical singularity of dimension $n=m+1$, with metric $\g_X$, where $(Y,W)$ is a compact connected orientable smooth Riemannian manifold of dimension $n$, with boundary $W$.  Let $\rho:\pi_1(X)\to O(V)$ be an orthogonal representation on a $k$ dimensional vector space $V$. Then,
\[
I^\mf \tau_{\rm RS}(X;V_{\rho})=\left(I^{\mf^c} \tau_{\rm RS}(X;V_{\rho})\right)^{(-1)^m}.
\]
\end{theo}
A direct implication of the last theorem is:
\begin{corol} If the dimension of $X$ is even, then
\[
I^\mf \tau_{\rm RS}(X;V_{\rho})=1.
\]
\end{corol}

\subsection{Variation of Ray-Singer intersection torsion}

We investigate the variation of the Ray and Singer intersection torsion under a variation of the metric. 

\begin{theo} \label{variationRS} Let $X=Y\sqcup_W C_{0,l}(W)$ be a space with a conical singularity of dimension $n=m+1$,  and with a smooth family of metrics $\g_{X,\mu}$ (see Section \ref{varX}), where $(Y,W)$ is a compact connected orientable smooth Riemannian manifold of dimension $n$, with boundary $W$.  Assume that the restriction of $\g_{X,\mu}$ onto $C_{0,l}(W)$ is $dx\otimes dx +h^2_\mu(x) \tilde g_\mu$, where $h_\mu$ is a family of functions on $[0,1]$ satisfying the assumptions in Section \ref{ss1.1}, and $\tilde g_\mu$ the smooth family of metrics on $W$ induced by the restriction of $\g_{X,\mu}$. Let $\rho:\pi_1(X)\to O(V)$, be an orthogonal representation on a $k$ dimensional vector space $V$. Let either $\pf=\mf$ or $\pf=\mf^c$. Then, 
\begin{align*}
\frac{\b}{\b\mu} \log  I^\pf \tau_{\rm RS}(X,g_{X,\mu};V_{\rho})
&= -\frac{1}{2} \log \|\Det I^\pf \alphas_{X,\g_{X,\mu}, \bullet}\|_{\Det \H^{\bullet}_{ \pf}(X,\g_{X,\mu};V_{\rho})}^2
\end{align*}
where $I^\pf \alphas_{X,\g_{X,\mu},q}$ is an orthonormal basis of $\Ha_{\pf}^{(q)}(X,\g_{X,\mu};V_{\rho})$.

\end{theo}

\begin{proof} We can split $X$ into the cone and $Y$, with induced metrics. On $Y$ we may follow exactly the same argument as in the proof of Theorem 7.6 of \cite{RS}. On the cone, we have the explicit formula for the torsion, and we see that the unique dependence on $\mu$ is in the homology basis, and precisely in 

\begin{align*}
\frac{\b}{\b\mu} \log  I^\pf \tau_{\rm RS}(X;V_{\rho})
&= \frac{1}{2} \sum_{q=0}^n (-1)^q  \Tr\star_\mu^{-1}(\b_\mu\star_\mu)P_{I^\pf\Ha^{(q)}_{X,\g_{X,\mu}, \pf, \mu}}\\
&= -\frac{1}{2} \sum_{q=0}^n (-1)^q  \sum_{j=0}^{r_q}\int_X \|I^\pf \alpha_{X,\g_{X,\mu},q,j}\|_X^2
\end{align*}
where $I^\pf \alphas_{X,\g_{X,\mu},q}=\{I^\pf \al_{X,\g_{X,\mu},q,j}\}$
\end{proof}

\section{The analytic torsion and the  Cheeger-M\"{u}ller theorem for a space with conical singularities}
\label{cheegerX}

\begin{theo}\label{tcmX} Let $X=C_{0,l}(W)\sqcup_W Y$ be a space with a conical singularity of dimension $n=m+1$, with metric $\g_X$, where $(Y,W)$ is a compact connected orientable smooth Riemannian manifold of dimension $n$, with boundary $W$.  Let $\rho:\pi_1(X)\to O(V)$ be an orthogonal representation on a $k$ dimensional vector space $V$, and $E_\rho$ the associated vector bundle. Let either $\pf=\mf$ or $\pf=\mf^c$. If $m=2p-1$, $p\geq 1$, then 
\begin{align*}
\log T_{ \pf}(X,E_\rho)&=\log I^\pf \tau_{\rm RS}(X;V_{\rho}).
\end{align*}

If $m=2p$, $p\geq 1$, then
\begin{align*}
\log T_{ \pf}(X,E_\rho)=&\log I^\pf \tau_{\rm RS}(X;V_{\rho})
+A_{\rm comb, \pf}(W)+A_{\rm analy}(W),\\
\end{align*}
where the  anomalies are given in Theorem \ref{tcm}.
\end{theo}

\begin{proof} Let $\g_{X,\mu}$ be a smooth family of metrics on $X$ as in Section \ref{varX}. Assume that $\g_{X,0}$ is a product near $W$. Then, by Theorem \ref{les}, 
\begin{align*}
\log T_{\pf}(X,\g_{X,0},E_\rho)=&\log  T_{\rm abs, \pf}(C_{0,l}(W),\g_{X,0}|_C,E_\rho|_C)+\log T_{\rm rel}(Y,\g_{X,0}|_Y,E_\rho|_Y)\\
&-\frac{1}{2}\chi(W)\log 2+\log \tau(\H_0),
\end{align*}
where $\H_0$ is the long homology exact sequence induced by the inclusion $C_{0,l}(W)\to X$, with $E_\rho$ coefficients, and with some orthonormal graded bases, in the metric $\g_{X,0}$. By Theorem \ref{tcm}, and the Cheeger-M\"{u}ller theorem for $Y$, \begin{align*}
\log T_{\pf}(X,\g_{X,0},E_\rho)=&\log  I^\pf \tau_{\rm RS}(C_{0,l}(W),\g_{X,0}|_C, V_{\rho}|_C)
+\log \tau_{\rm RS}(Y,W, \g_{X,0}|Y;V_\rho|_Y)\\
&+\log \tau(\H_0)+A_{\rm comb, \pf}+A_{\rm analy},
\end{align*}
where the anomaly terms vanish if $m=2p-1$ is odd. By Theorem \ref{t7.37},
\begin{align*}
\log T_{\pf}(X,\g_{X,0},E_\rho)=&\log  I^\pf \tau_{\rm RS}(X,\g_{X,0}|_C, V_{\rho}|_C)
+A_{\rm comb, \pf}+A_{\rm analy}.
\end{align*}

By Theorem \ref{variationRS}, 
\begin{align*}
\log T_{\pf}(X,\g_{X,0},E_\rho)=&\log  I^\pf \tau_{\rm RS}(X,\g_{X,1}|_C, V_{\rho}|_C)
+A_{\rm comb, \pf}+A_{\rm analy}\\
&+\frac{1}{2}\int_0^1 \log \|\Det I^\pf \alphas_{X,\g_{X,\mu}, \bullet}\|_{\Det \H^{\bullet}_{\pf}(X,\g_{X,\mu};V_{\rho})}^2 d\mu,
\end{align*}
where $I^\pf \alphas_{X,\g_{X,\mu},q}$ is an orthonormal basis of $ \H_{\pf}^{(q)}(X,\g_{X,\mu};V_{\rho})$.

By Theorem \ref{variationX}
\begin{align*}
\log T_{\pf}(X,\g_{X,1},E_\rho)=&\log  I^\pf \tau_{\rm RS}(X,\g_{X,1}|_C, V_{\rho}|_C)
+A_{\rm comb,\pf}+A_{\rm analy}.
\end{align*}

 \end{proof}

As a corollary of the previous theorem, we have the glueing formula for analytic torsion.

\begin{theo}\label{tX} Let $X=C_{0,l}(W)\sqcup_W Y$ be a space with a conical singularity of dimension $n=m+1$,  where $(Y,W)$ is a compact connected orientable smooth Riemannian manifold of dimension $n$, with boundary $W$.  Let $\rho:\pi_1(X)\to O(V)$ be an orthogonal representation on a $k$ dimensional vector space $V$. Let either $\pf=\mf$ or $\pf=\mf^c$. Then, 
\begin{align*}
\log T_{\pf}(X,E_\rho)=&\log  T_{\rm abs, \pf}(C_{0,l}(W),E_\rho|_C)+\log T_{\rm rel}(Y,E_\rho|_Y)\\
&-\frac{1}{2}\chi(W;E_{\rho}|_{W})\log 2+\log \tau(\H),
\end{align*}
where $\H$ is the long homology exact sequence induced by the inclusion $C_{0,l}(W)\to X$, with $E_\rho$ coefficients, and with some orthonormal graded bases. 
\end{theo}

Moreover, we have the following generalisations of the glueing formula and of the anomaly boundary term.

\begin{theo}\label{LV} The Lesh Vishick glueing formula for the analytic torsion given in  \cite{Vis,Les2}holds for a space with conical singularities for any metric structure near the glueing, provided the glueing is not close to  the singularities.
\end{theo}

\begin{theo}\label{BMext} The formula for the anomaly boundary term in the analytic torsion given in \cite{BM1}  holds for a  space with conical singularities, provided that the boundary is not close to the singularities.
\end{theo}

%% file: appendices.11.tex

\section{Standard bases}
\label{standardbasis}

\subsection{Basis for module and submodules}
\label{based-modules}

Let $A$ be an integral domain ($A$ modules have the invariant dimension property \cite[IV.2.12]{Hun}). Let $M$ be a free $A$-modules of finite rank $m$, and $N$ a submodule of $M$. Then, $N$ is free of rank $r$, with $r\leq m$, and we have the following result.

\begin{prop}\label{app.p1} Let $M$ be a free $A$-modules of finite rank $m$, and $N$ a submodule of $M$ of rank $r$. Then, $N$ is free of rank $r$, with $r\leq m$.
\end{prop}
\begin{proof}  \cite[IV.6.1]{Hun} \cite[11.1]{Mun} Let 
\[
N_k=N\cap \langle z_1,\dots, z_k\rangle.
\]

Let $p_k:M\to A$ denote the projection on the $k$-th coordinate. Then, $N_k$ consists of all $x\in n$ such that $p_j(x)=0$ for all $j>k$. In particular, $N_m=N$.

The homomorphism
\[
p_k:N_k\to A,
\]
carries $N_k$ onto a principal ideal of $A$. If this ideal it the trivial ideal, set $u_k=0$, and $j_k=k$; otherwise, chose $u_k\in N_k$ such that $p(u_k)$ generates $p_k(N_k)$. We assert that the non trivial elements of the set
\[
\{u_1,\dots, u_m\},
\]
form a basis for $N$. Then, $r=|\{u_1,\dots, u_m\}|$, and we can complete this set of linearly independent elements of $M$ to a basis of $M$, taking the elements $z_{j_k}$ (whose indices do not appear in the indices of the $u_k$). This gives the basis stated in the lemma. 

We prove the assertion. First, we prove that the elements $\{u_1,\dots, u_k\}$ generate $N_k$, for each $k$. $u_1$ generates $N_1$ by construction. Assume $\{u_1,\dots, u_{k-1}\}$ generate $N_{k-1}$. Let $x\in N_k$. Then, $p_k(x)=ap_k(u_k)$, for some $a\in A$. Thus,
\[
p_k(x-au_k)=0,
\]
and hence $x-au_k\in N_{k-1}$, and
\[
x-au_k=a_1u_1+\dots a_{k-1}u_{k-1},
\]
by the induction hypothesis. Next, we show that the non zero elements in the set $\{u_1,\dots, u_k\}$ are linearly independent. This is trivial if $k=1$. Assume this is true for $k-1$. Assume that
\beq\label{pop}
a_1u_1+\dots +a_ku_k=0,
\eeq
and apply $p_k$. We obtain 
\[
a_k p_k(u_k)=0,
\]
and hence either $a_k=0$ or $u_k=0$. Thus, equation (\ref{pop}) became
\[
a_1u_1+\dots +a_{k-1}u_{k-1}=0,
\]
that by the induction hypothesis implies that $a_j=0$ for all $x_j\not=0$.
\end{proof}

Let 
\[
i:N\to M,
\]
be an injective homomorphism of free finitely generated left $A$ modules. Consider the exact sequence
\[
\xymatrix{
0\ar[r]&N\ar[r]^i&M\ar[r]^{p}&M/N\ar[r]&0.
}
\]
Let $h+n=\dim (M)$, $n=\dim (N)$. Let $h=\rk(M/N)$ (the dimension of the free submodule), $t$ the number of  the cyclic submodules \cite[VI.6.12]{Hun}.

\begin{prop}\label{ap2} There exists a basis $\cs=\{c_1,\dots, c_{h+n}\}$ of $M$ such that the set 
\[
\zs=\{u_1 c_{h+1},\dots, u_{h+n-t} c_{h+n-t},a_1 c_{h+n-t+1},\dots, a_t c_{h+n}\},
\]
is a basis for the submodule $i(N)$ of $M$, the $u_j$ are unit of $A$ and  the $a_j$ are non unit elements of $A$.
\end{prop}
\begin{proof} Let $\zeta$ be any basis of $M$. Then,  $i(N)$ is generated by $n$ vectors $\{v_1,\dots,v_n\}$, and, for each $k$,
\[
v_k=r_{k,1}\zeta_1+\dots +r_{k,m}\zeta_m,
\]
where $m=h+n$. Let $y_k$ be the mcd of the $r_{k,j}$, then 
\[
v_k=y_k(p_{k,1}\zeta_1+\dots +p_{k,m}\zeta_m),
\]
where the $p_{k,j}$ have not common divisors. Consider the matrix
\[
X=\left(\begin{matrix} x_{1,1}&\dots&x_{1,m}\\
\dots &\dots&\dots\\
x_{h,1}&\dots &x_{h,m}\\
p_{1,1}&\dots&p_{1,m}\\
\dots &\dots&\dots\\
p_{n,1}&\dots &p_{n,m}\end{matrix}\right).
\]

Thus the set $\cs=\{c_{k}\}$, where
\begin{align*}
c_k&=x_{k,1}\zeta_1+\dots+x_{k,m}\zeta_m, & 1\leq &k\leq h,\\
c_k&=p_{k,1}\zeta_1+\dots+p_{k,m}\zeta_m, & h+1\leq &k\leq h+n,
\end{align*}
is a basis of $M$ and clearly satisfies the requirement in the thesis, ordering first the possible $y_k$ that are units.
\end{proof}

Using the basis given in the lemma, we have
\begin{align*}
\frac{M}{N}&=\frac{R[c_1,\dots, c_h,c_{h+1},\dots, c_{h+l}, c_{h+l+1}, \dots c_{h+n}]}{R[u_1 c_{h+1},\dots, u_{h+n-t} c_{h+n-t},a_1 c_{h+n-t+1},\dots, a_t c_{h+n}]}\\
&=R[c_1,\dots,c_h]\oplus \frac{R}{a_1R}[c_{h+n-t+1}]\oplus \frac{R}{a_tR}[c_{h+n-t+1}].
\end{align*}


Note, if there is at least one $a_j$, then   the set
\[
\iota=\{c_1,\dots, c_h\}\zs=\{c_1,\dots, c_h,u_1 c_{h+1},\dots, u_{h+n-t} c_{h+n-t},a_1 c_{h+n-t+1},\dots, a_t c_{h+n}\},
\]
is a set of $n+h$ linearly independent elements of $M$, that again is not a basis if there exists at least one $a_j$, with matrix (since $A$ is an integral domain the Whitehead class is the module of the image of the determinant)
\[
(\{c_1,\dots, c_h\}\zs/\cs)=|\det(\iota/\cs)|=\prod_{j=1}^{n-t} u_j \prod_{j=1}^{t} a_j.
\]

The class in the Whitehead group $\tilde K_{A^\times}(A)$ is 
\[
[(\iota/\cs)]= \prod_{j=1}^{t} a_j.
\]

Let $\F$ a field extension of $A$. Taking the tensor product, we obtain a short exact sequence of vector spaces
\[
\xymatrix{
0\ar[r]&\F\otimes N\ar[r]^i&\F\otimes M\ar[r]^{p}&\F\otimes M/N\ar[r]&0.
}
\]

Now, the set 
\[
\iota=\{c_1,\dots, c_h,u_1 c_{h+1},\dots, u_{h+n-t} c_{h+n-t},a_1 c_{h+n-t+1},\dots, a_t c_{h+n}\},
\]
is a basis of the vector space $\F\otimes M$. We call the equivalence classes of basis
\[
[\iota]=[\{c_1,\dots, c_h,c_{h+1},\dots,  c_{h+n-t},a_1 c_{h+n-t+1},\dots, a_t c_{h+n}\}],
\]
the integral basis of $M$. Note that the definition does not depend on the choice of the initial basis $\cs$. For starting with any other basis $\cs'$ of $M$
\[
[(\iota'/\iota)]=[(\iota'/\cs)][(\cs/\cs')][(\cs'/\iota)]=\frac{\prod_{j=1}^{t} a_j}{\prod_{j=1}^{t} a_j}=1.
\]

Next, let $M=\CS_q$ be the $q$-chain module of the  $A=\Z$ chain complex $\CS_\bu$. Then, $\ker \b_q$ is a free submodule of $\CS_q$. Let  $\cs_q$ be the chain basis of $\CS_q$, and $\iota_q$ the set constructed above. Since the change of basis matrix of the change of basis $(\cs/\cs_q)$ has determinant $\pm 1$, we have that  
\[
[(\iota_q/\cs_q)]=\prod_{j=1}^t a_j=\OO_q,
\]
where $\OO_q$ denotes  the order of the torsion subgroup of $H_q(\CS_\bu)$ \cite{Che0}.

\subsection{The standard basis in an integral chain complexes}\

Let $\CS_\bu$ a free chain complex of $\Z$ modules. Then, we have the following splitting \cite[11.4]{Mun}
\[
\CS_q=U_q\oplus V_q\oplus W_q,
\]
where $U_q=\hat H_{{\rm free},q}$ is a lift of the free part of the homology, $W_q=\hat H_{{\rm tr}, q}$ is a lift of the torsion part of the homology, $\b_q(U_q)=W_{q-1}$, and $Z_q=V_q\oplus W_q$, and is uniquely determined. 
Let $u_q=\dim (U_q)=\dim (W_{q-1})$, and $m_q=\dim (\CS_q)$.
Then, there exists a basis $\es_q$ of $\CS_q$ such that $\{e_{q,1},\dots, e_{q,u_q}\}$ is  a basis of $U_q$, $\{e_{q-1,m_{q-1}-u_q},\dots,e_{q-1,m_{q-1}}\}$ is a basis of $W_{q-1}$, $\{e_{q,u_q+1},\dots,e_{q,m_q}\}$ is a basis for $Z_q$, and the boundary operator acts as follows:
\[
\b_q(e_{q,j})=k_{q-1,j} e_{q-1,m_{q-1}-u_q+j},
\]
for $1\leq q\leq u_q$, with $k_{q-1,j}\in \Z$, $k_{q-1,j}\not=0$, $k_{q-1,m_{q-1}-u_q+1}|k_{q-1,m_{q-1}-u_q+2}|\dots | k_{q,m_{q-1}}$. 
The homology $H_q(\CS_\bu)$ of the  complex $\CS_\bu$ has a  free part  $FH_q(\CS_\bu)$ generate by the set
\[
\{e_{q,u_q+1},\dots,e_{q,m_q-u_{q+1}}\},
\]
and a torsion part $TH_q(\CS_\bu)$ generated by the set
\[
\{e_{q,m_q-u_{q+1}+t_q-1},\dots,e_{q,m_q}\},
\]
where the elements in the second set are those whose indices correspond to the $k_{q,j}\not=\pm 1$, and $t_q=\#TH_q(\CS_\bu)$.

However the torsion of the complex $\CS_\bu$ is not well defined, since the homology is not free, unless all the $k_{q,j}$ are $\pm 1$, and in this case the torsion is trivial (as expected since this is a complex of $\Z$ modules).

On the other side,  if we change the coefficients to the reals, namely if we consider the complex $\R\otimes \CS_\bu$, then the homology $H_q(\R\otimes \CS_\bu)$ of the twisted complex is free and is generate by the set
\[
\hat\ns_q=\{e_{q,u_q+1},\dots,e_{q,m_q-u_{q+1}}\},
\]

We may then choose the set 
\[
\bs_q=\{e_{q,1},\dots, e_{q,u_q}\},
\]
as a lift of the set of the generators of $\Im \b_q$, that is generated by the set 
\[
\b_{q}(\bs_{q})=\{k_{q-1,m_{q-1}-u_q}e_{q-1,m_{q-1}-u_q},\dots,k_{q-1,m_{q-1}}e_{q-1,m_{q-1}}\}.
\]

Then the set 
\begin{align*}
\b_{q+1}(\bs_{q+1}) \hat \ns_q \bs_q=&\{k_{q,m_{q}-u_{q+1}}e_{q,m_{q}-u_{q+1}},\dots,k_{q,m_{q}}e_{q,m_{q}}\}\cup \{e_{q,u_q+1},\dots,e_{q,m_q-u_{q+1}}\}\\
&\cup \{e_{q,1},\dots, e_{q,u_q}\},
\end{align*}
is a basis of $\R\otimes \CS_q$, and it is clear that
\[
\det (\b_{q+1}(\bs_{q+1}) \hat \ns_q \bs_q/\es_q)=\prod_{j=1}^{t_q} k_{q,j}=\# T H_q(\CS_\bu).
\]

\section{Some results on Sturm-Liouville differential equations}
\label{SturmLouville}

 \subsection{Smooth case}

Consider the following second order regular singular equation
\beq\label{sl0}
f''(x)+\frac{p(x)}{x}f'+\frac{q(x)}{x^2}f=0,
\eeq
where $p$ and $q$ are real continuous functions  in $[0,r]$, for some $r>0$. We set
\begin{align*}
p_0&=p(0),&q_0&=q(0).
\end{align*}

The indicial equation for equation (\ref{sl0}) is
\[
s^2 +(p_0-1)s+q_0=0,
\]
with solutions:
\[
s_\pm=\frac{1}{2}\left(1-p_0\pm\sqrt{(p_0-1)^2-4q_0}\right).
\]

We adopt the convention  that $s_+\geq s_-$.

The change of variable 
\[
f(x)=P(x)u(x)=\e^{-\frac{1}{2}\int \frac{p(x)}{x}}u(x),
\]
transforms equation (\ref{sl0}) into the following equation
\beq\label{sl1}
u''+\frac{P''(x)}{P(x)}u+\frac{p(x)}{x}\frac{P'(x)}{P(x)} u
+\frac{q(x)}{x^2}u=0.
\eeq

For small $x$, $P(x)\sim x^{-\frac{1}{2}p_0}$, and direct substitution gives
\[
\frac{P''(x)}{P(x)}+\frac{p(x)}{x}\frac{P'(x)}{P(x)}
\sim \frac{1}{2}p_0\left(1-\frac{1}{2}p_0\right)\frac{1}{x^2},
\]

Whence, the indicial equation for equation (\ref{sl1}) is
\[
\mu(\mu-1) + \frac{1}{2}p_0\left(1-\frac{1}{2}p_0\right)+q_0=0,
\]
with solutions
\[
\mu_\pm=\frac{1}{2}\left(1\pm\sqrt{(p_o-1)^2-4q_0}\right).
\]

If we adopt the same convention as above, in this case we always have $\mu_1=\mu_+$ and $\mu_2=\mu_-$. Also note the following relation between the solutions of the two indicial equations:
\[
s_\pm=\mu_\pm-\frac{p_0}{2}.
\]

\begin{theo}\label{boc} \cite{Boc} Consider the regular singular second order differential equation
\beq\label{sl2}
f''(x)+\frac{p(x)}{x}f'+\frac{q(x)}{x^2}f=0,
\eeq
where $p$ and $q$ are real  continuous functions in the interval $[0,r]$, for some $r>0$,  with finite limits at $x=0$ 
\begin{align*}
p_0&=\lim_{x\to 0^+} p(x),& q_0&=\lim_{x\to 0^+} q(x),
\end{align*}
and satisfying the assumption
\begin{align*}
\int_0^r |p(x)-p(0)|\log^2 x dx&<\infty, &\int_0^r \frac{|q(x)-q(0)|}{x} \log^2 x dx&<\infty,
\end{align*}

Let $s_+\geq s_-$ be the solutions of the indicial equation
\[
s^2 +(p_0-1)s+q_0=0.
\]

Then, the equation (\ref{sl2}) has  two linearly independent solutions $f_\pm$  according to the following description:
\begin{enumerate}
\item if $s_+\not= s_-$, then
\begin{align*}
f_\pm(x)&=x^{s_\pm}\psi_\pm(x),&f'_\pm(x)&=x^{s_\pm-1}\Psi_\pm(x),
\end{align*}
where the $\psi_\pm$ and the $\Psi_\pm$ are continuous  in some interval $[0,r_0]$, $r_0\leq r$,   $\psi_\pm(0)=1$, and $\Psi_\pm(0)=s_\pm$; 
\item if $s_+=s_-$, then
\begin{align*}
f_+(x)&=x^{s_+}\psi_+(x),&f'_+(x)&= x^{s_+-1}\Psi_+(x),\\
f_-(x)&=x^{s_+}\psi_-(x)\log x,&f'_-(x)&=x^{s_+-1}\Psi_-(x)\log x
\end{align*}
where the $\psi_\pm$,   and the $\Psi_\pm$ are continuous  in some interval $[0,r_0]$, $r_0\leq r$,    $\psi_\pm(0)=1$, and $\Psi_\pm(0)=s_+$;
\end{enumerate}

The corresponding solutions of equation (\ref{sl1}) are 
\[
u_\pm(x)=P^{\frac{1}{2}}(x)f_\pm(x)=\e^{\frac{1}{2}\int \frac{p(x)}{x}}f_\pm(x).
\]
\end{theo}

\begin{rem} \label{case0} In the case of equal exponent, there is the following alternative description of the second solution:
\[
f_-(x)=f_+(x)\log x+x^{s_+}F(x),
\]
where $F$ is continuous  in some interval $[0,r_0]$, $r_0\leq r$,    $F(0)=0$.
\end{rem}

\begin{rem}\label{Appsmooth} If the function $p$ and $q$ are of class $C^k((0,l])$, then the solutions $\ff_\pm$ are of class $C^{k+2}((0,l])$ \cite[XIII.1.3, pg. 1281; XIII.1.4, pg. 1283]{DS2}. In particular, if $k=1$, then 
\[
\lim_{x\to 0^+}\frac{|p(x)-p(0)|}{x}=|p'(0)|<\infty,
\]
and therefore the hypothesis of Theorem \ref{boc} are satisfied.
\end{rem}

\begin{rem} Note that the hypothesis of Theorem \ref{boc} are satisfied if we assume 
\begin{align*}
p(x)&=P(x^\ep), &q(x)&=Q(x^\ep),
\end{align*}
with smooth $P$ and $Q$ in $[0,r]$, and $\ep>0$.
\end{rem}

If the coefficients of the differential equation are analytic, then the solutions are analytic, but a larger number of logarithmic contributions must be take in account, according to the following theorem due to Fucs \cite[4.5]{Tes}.

\begin{theo}\label{theoA2} Assume that the coefficients $p$ and $q$  in equation \ref{sl2} have a power series expansion in $[0,r]$, for some $r>0$. Then,  there exists a fundamental system of solutions with  following series expansions in $[0,r]$:
\begin{enumerate}
\item if $s_+-s_-\not\in \Z$, then
\[
f_{\pm}(x)= x^{s_\pm}\sum_{k=0}^\infty a_{k,\pm}x^k,
\]
with $a_{0,\pm}=1$;
\item if $s_+-s_-\in\Z$, then
\begin{align*}
f_{+}(x)=&x^{s_+}\sum_{k=0}^\infty a_{k,+}x^k,\\
f_{-}(x)=& x^{s_-} \sum_{k=0}^\infty a_{k,-}x^k+b \ff_{+}(x)\log x,
\end{align*}
with $a_{0,\pm}=1$, and where $b$ may be zero unless $s_+-s_-=0$.
\end{enumerate}
\end{theo}

\begin{rem} \label{ra6}
Observe that if $s_+-s_-$ is not an integer, then  the solutions $\pm$ of Theorem \ref{boc} coincide with those of Theorem \ref{theoA2}. If $s_+-s_-$ is an integer different from zero, this is true only for the  plus solutions, for logarithmic terms appear in the power series expansion that are hidden in the function $\psi_-$ in the smooth description. If $s_+=s_-$, then also the minus solutions coincide if we use the alternative formula for it,  given in Remark \ref{case0}.
\end{rem}

\subsection{Analytic case}
\label{series}

In this section we give details on the power series expansions of the solutions available if the coefficients of the differential equation are analytic. 

\begin{lem}\label{l2.2-a1}  Consider the equation
\[
u''(x)+\left(\la-\frac{\nu^2-\frac{1}{4}}{h^2(x)}-p(x)\right) u(x)=0,
\]

Assume either $\nu=\frac{1}{2}$ or $\nu\not\in \frac{1}{2}\Z$, and $h$ and $p$ are analytic in $[0,l]$. Then  the above equation reads
\[
u''(x)+\left(\la-\frac{\nu^2-\frac{1}{4}}{x^2}- P(x)\right) u(x)=0,
\]
with
\begin{align*}
P(x)&=x^{-1}\left(\nu^2-\frac{1}{4}\right)\sum_{k=0}^\infty b_j x^j+ x^{-1}\sum_{k=0}^\infty c_j x^j,
\end{align*}
and has has two linearly independent solutions $\uf_{\pm}(x,\la,\nu)$ for $x\in (0,l]$, that are analytic in $\lambda$, corresponding to the two solutions $\mu_\pm=\frac{1}{2}\pm\nu$ of the indicial equation 
\[
\mu(\mu-1)+\frac{1}{4}-\nu^2=0.
\]

These solutions  may be fixed univocally by the following series expansions. In the regular case, $\nu=\frac{1}{2}$, $\mu_-=0$, $\mu_+=1$: 
\begin{align*}
\uf_{-}(x,\la,\nu )&= \sum_{k=0}^\infty a_{k}(\la,\nu ) x^k,\\
a_0(\la,\nu )&=1,&a_1(\la,\nu )&=0,\\ 
a_{k+2}(\la,\nu )& =-\frac{ a_{k}(\la,\nu )}{(k+1)(k+2)}\la + \frac{ \sum_{j=0}^k a_{k}(\la,\nu )  b_{k-j}}{(k+1)(k+2)};
\end{align*}
\begin{align*}
\uf_{+}(x,\la,\nu )&= \sum_{k=0}^\infty a_{k}(\la,\nu ) x^k,\\
a_0(\la,\nu )&=0,&a_1(\la,\nu )&=1,\\
a_{k+2}(\la,\nu )& =-\frac{ a_{k}(\la,\nu )}{(k+2)(k+3)}\la + \frac{ \sum_{j=0}^k a_{k}(\la,\nu )  b_{k-j}}{(k+2)(k+3)}.
\end{align*}

In the singular case, $\nu\not=\frac{1}{2}$:
\[
\uf_{\pm}(x,\la,\nu )=x^{\frac{1}{2}\pm \nu} \sum_{k=0}^\infty a_{k,\pm}(\la,\nu ) x^k,
\]
with the following coefficients ($k=0,1,2,\dots$):  
\begin{align*}
a_{0,\pm}(\la,\nu )&=1,&\\
a_{1,\pm}(\la,\nu )&=\frac{\left(\nu^2-\frac{1}{4}\right) b_0-  c_0}{2\mu},\\ 
a_{k+2,\pm}(\la,\nu )& =-\frac{ a_{k,\pm}(\la,\nu )}{(k+2)(\pm 2\nu+k+2)}\la + \frac{   
\sum_{j=0}^{k+1}\left(\left(\nu^2-\frac{1}{4}\right) b_{k+1-j}+  c_{k+1-j}\right)a_{j,\pm}(\la,\nu )}{(k+2)(\pm2\nu+k+2)}.
\end{align*}
\end{lem}

\begin{rem}\label{zzz} Observe that all the coefficients $b_k$, and $c_k$  vanish if $p=0$ and $h(x)=x$. 
\end{rem}

\begin{lem}\label{l2.3-a1} Assume that  $\al\in \frac{1}{2}\Z$, $\al\not=0,\frac{1}{2}$, and that the function $h$ has the following series expansion 
\[
h(x) = x \left(1 + \sum_{j=1}^\infty h_j x^j\right).
\]

Then, the equation 
\begin{equation}\label{h1-1}
f''+(1-2\al)\frac{h'}{h} f'+ \la f=0,
\end{equation}
has two linearly independent solutions $\ff_{\pm}(x,\la,\nu)$ for $x\in (0,l]$, that are analytic in $\lambda$, corresponding to the two solutions $\mu_\pm=\frac{1}{2}\pm\nu$ of the indicial equation 
\[
\mu(\mu-1)+(1-2\al)\mu=0,
\]
and that may be fixed univocally by the following series expansions:
\begin{equation}\label{es5}
\begin{aligned}
\ff_{+}(x) &= x^{\al + |\al|} \sum_{n=0}^\infty a_{n,+} x^n\\
\ff_{-} (x)&= x^{\al - |\al|}\sum_{n=0}^\infty a_{n,-} x^n + b\ \ff_{+}(x) \log x.
\end{aligned}
\end{equation}
with coefficients
\begin{equation*}
\begin{aligned}
a_{0,+}&=1,\\
a_{1,+} &= -\frac{(\al \pm | \al |)  (1-2\al)  h_1  a_{0,+}}{(1+\al+ | \al |)(1+\al + | \al | -2\al)},\\
a_{n,+} &= \frac{-\la a_{n-2} - (1-2\al) \sum_{k=1}^{n} \tilde h_k a_{n-k}(n-k+\al+|\al|)}{(n+\al\pm|\al|)(n+\al+|\al|-2\al)},
\end{aligned}
\end{equation*}
\begin{align*}
a_{0,-}&=1,\\
a_{n,-} &=\frac{- \la a_{n-2,-} - (1-2\al ) \sum_{k=1}^n \tilde h_k a_{n-k,-} (n-k+ \al - | \al |)}{(n+\al -| \al |)(n+\al - |\al | - 2 \al)}, &0< n<2| \al |,\\
a_{2|\al|,-} &= \left\{\begin{array}{cl}-\frac{b}{2} H_{|\al|}, & \al \in \Z-\{0\},\\
 0,  &\al\in \frac{1}{2}(2\Z+1) - \{1/2\},\end{array}\right.\\
a_{n,-} &=\frac{- \la a_{n-2,-} - (1-2\al ) \sum_{k=1}^n \tilde h_k a_{n-k,-} (n-k+ \al - | \al |)}
{(n+\al -| \al |)(n+\al - |\al | - 2 \al)}\\
&+\frac{b \left(2a_{n-2|\al|,+}(n-|\al|)+(1-2\al)\sum_{k=1}^{n-2|\al|}\tilde h_k a_{n-2|\al|-k,+}\right)}
{(n+\al -| \al |)(n+\al - |\al | - 2 \al)},& n>2|\al|, \\
b  &=-\frac{1}{2|\al| }\left( \la a_{2| \al | - 2,-} + (1-2\al) \sum_{k=1}^{2| \al |} \tilde h_k a_{2| \al |-k,-}( \al +| \al | - k)\right).
\end{align*}
where
\begin{align*}
\frac{h'(x)}{h(x)}&=\frac{1}{x}\left(1+\sum_{j=1}^\infty \tilde h_j x^j\right),&H_{|\al|} &= \sum_{j=1}^{|\al|} \frac{1}{j}.
\end{align*}
\end{lem}

\begin{rem}\label{xxx} Observe that in the flat case $h(x)=x$, all the coefficients $\tilde h_j$ vanish, and therefore all add coefficients $a_{2k+1,\pm}$ vanish as well.
\end{rem}

\begin{corol} In the hypothesis of the previous lemma, the equation 
\begin{align*}
u''+\left(\al-\frac{1}{2}\right) \frac{ h''}{h} u
+\left(\left(\frac{1}{4}-\alpha^2\right) \frac{(h')^2}{h^2}+\la\right) u=0,
\end{align*}
has two linearly independent solutions $\uf_{\pm}(x,\la,\nu)$ for $x\in (0,l]$,  analytic in $\lambda$, and univocally determined by the following series expansions:
\[
\uf_{\pm}=h^{\frac{1}{2}-\al}\ff_\pm.
\]
\end{corol}

We conclude with the case $\nu=\al=0$.

\begin{lem}\label{l2.3.1} Assume that  the function $h$ has the following series expansion 
\[
h(x) = x \left(1 + \sum_{j=1}^\infty h_j x^j\right).
\]

Then, the equation 
\begin{equation}\label{h1}
f''+\frac{h'}{h} f'+ \la f=0,
\end{equation}
has two linearly independent solutions $\ff_{\pm}(x,\la,0)$ for $x\in (0,l]$, that are analytic in $\lambda$, 
and that may be fixed univocally by the following series expansions:
\begin{align*}
\ff_{+}(x) &=  \sum_{n=0}^\infty a_{n,+} x^n\\
\ff_{-} (x)&= \sum_{n=0}^\infty a_{n,-} x^n +  \ff_{+}(x) \log x,
\end{align*}
with coefficients ($n>0$)
\begin{equation*}
\begin{aligned}
a_{0,+}&=1,\\
a_{n,+} &=  \frac{-\la a_{n-2,+} -  \sum_{k=1}^{n} \tilde h_k a_{n-k,+}(n-k)}{n^2},\\
a_{n,-} &=\frac{- \la a_{n-2,-} - \sum_{k=1}^n \tilde h_k a_{n-k,-} (n-k)
+\left(2a_{n,+}n+\sum_{k=1}^{n}\tilde h_k a_{n-k,+} \right)
}
{n^2}.
\end{aligned}
\end{equation*}

\end{lem}

Observe that formula for the $-$ solution coincides with the one given in the smooth case.

\subsection{Determinacy of the solutions}
\label{a12}

Consider first the case $\nu>0$. It is easy to see that the solution $\ff_+$ described in Theorem \ref{boc} is univocally determined. For any solution may be written as 
\[
f=c_+\ff_++c_-\ff_-,
\]
with some constants $c_\pm$. For $f$ to be of type $\ff_+(x)=x^{s_+}\psi_+(x)$, it is necessary that $c_-=0$. On the other side, $\ff_-$ is not univocally determined if described as in Theorem \ref{boc}, since 
\[
\ff_-(x)+c\ff_+(x)=x^{s_-}\left( \psi_-(x)+c x^{s_+-s_-}\psi_+(x)\right)=x^{s_-}F_-(x),
\]
with $F_-$ continuous and $F_-(0)=1$. Observe that in the case $s_+=s_-$, $\ff_-$ is univocally determined if we use the alternative description given in Remark \ref{case0}, this is discussed in the next section.

On the other side, if the solutions may be expanded in power series, then they are both determined by the series expansions stated  in Theorem \ref{theoA2} and described in details in Section \ref{series}. It is not difficult to prove that the power series expansions are univocally determined by the following initial value conditions (recall $\nu>0$):

\begin{align*}
IC_{0}(\ff_{+})&=0,&
IC_{0}(\ff_{-})&=1,\\
IC_{0}'(\ff_{+})&=1,
&IC_{0}'(\ff_{-})&=0,\\
\end{align*}
where
\begin{align*}
IC_{0}(f)&=\Rz_{x=0} x^{-\al+\nu} f(x),&IC_{0}'(f)&=\Rz_{x=0} x^{-\al-\nu} f(x),
\end{align*}
plus the condition
\begin{align*}
a_{2|\al |,-}&=H_{|\al|}, \hspace{10pt}\al\in \Z-\{0\},& a_{2|\al |,-}&=0, \hspace{10pt}\al\in \frac{1}{2}(2\Z+1)-\left\{\frac{1}{2}\right\},&
\end{align*}
if $\nu=|\al|$ is an half integer different from $0$. These conditions give for the functions $\uf$:
\begin{align*}
IC_{0}(\uf_{+})&=0,&IC_{0}'(\uf_{+})&=1,\\
IC_{0}(\uf_{-})&=1,&IC_{0}'(\uf_{-})&=0,\\
\end{align*}
where
\begin{align*}
IC_{0}(u)&=\Rz_{x=0} x^{\nu-\frac{1}{2}} u(x),&IC_{0}'(u)&=\Rz_{x=0} x^{-\nu-\frac{1}{2}} u(x).
\end{align*}

Second, consider  $\nu=0$. In this case the description of the minus solution given in Remark \ref{case0} determine it univocally. For adding a multiple of the $+$ solution to the $-$ solution clearly modify the value of the limit
\[
\lim_{x\to 0^+} x^{-s_+}\left(\ff_-(x)-\ff_+(x)\log x\right).
\]

We now investigate the  the limit case $h(x)=x$.  A direct calculation  gives
\[
\uf_+^0=u_+(x)=\sqrt{x}J_0(\sqrt{\la}x)=\sqrt{x}I_0(\sqrt{-\la}x).
\]

The second solution requires a bit more work, since in the classical functions the coefficients in the power series are fixed in a different way. Taking the power series expansion 
\[
u_-(x)=b\log x u_+(x)+\sqrt{x}\sum b_{k}(\sqrt{\la}x)^k,
\]
and proceeding as for example in \cite[Section 4.1]{Tes}, we chose $b=\frac{2}{\pi}$. Then, the following relation holds if $\la=1$:
\[
\sqrt{x}Y_0(x)=u_-(x)+\left(-1+(\gamma-\log 2)\right)u_+(x),
\]
whence in such a case, using \cite[8.444]{GZ}, we find that
\[
u_-(x)=\frac{2}{\pi}\sqrt{x}\log x J_0(x)-\frac{2}{\pi} \sqrt{x}\sum_{k=1}^\infty \frac{(-1)^k}{(k!)^2 2^{2k}} \sum_{l=1}^k\frac{1}{k} x^{2k}+\sqrt{x}J_0(x).
\]

Comparing with the previous expression for $u_-$ and restoring the general value for $\la$, we have

\[
\frac{\pi}{2}u_-(x)=-\sqrt{x}K_0(\sqrt{-\la} x)-\frac{1}{2}\sqrt{x}\log (-\la) I_0(\sqrt{-\la}x)+\left(\frac{\pi}{2}+\log 2-\ga\right)\sqrt{x}I_0(\sqrt{-\la}x). 
\]

Comparing with the normalisation given in \ref{case0}, we find that
\[
\uf^0_-=\frac{\pi}{2}u_-(x)  -\frac{\pi}{2}\uf_+=-\sqrt{x}K_0(\sqrt{-\la} x)+(\log 2 - \gamma -\log\sqrt{-\la}) I_0(\sqrt{-\la} x),
\]
and also the $-$ solutions is expressed in terms of Bessel functions.

We conclude with the function $\vf$. We define it by
\[
\vf^0=-\uf^0_- -\left(\log 2-\ga-\frac{1}{2}\log(-\la)\right)\uf^0_+,
\]
so that
\[
\vf^0=\sqrt{x}K_0(\sqrt{-\la} x).
\]

\section{Some technical lemmas}
\label{fu}

Consider the  equation 

\[
f''+\frac{(1-2\alpha) h'}{h}f'+\left(\la+\frac{\alpha^2-\mu^2}{h^2}\right)f=0.
\]

With the substitution $f=h^{\be} u$, we have
\begin{align*}
f=&h^{\be} u,\\
f'=&\be h^{\be-1} h' u+h^\be u',\\
f''=&\be(\be-1)h^{\be-2} h' u+\be h^{\be-1} h'' u+2\be h^{\be-1}h' u'+h^\be u'',
\end{align*}
and hence
\begin{align*}
\be(\be-1)h^{\be-2} (h')^2 u&+\be h^{\be-1} h'' u+2\be h^{\be-1}h' u'+h^\be u''\\
&+\frac{(1-2\alpha_q) h'}{h}\be h^{\be-1} h' u
+\frac{(1-2\alpha_q) h'}{h}h^\be u'
+\left(\frac{\alpha^2-\mu^2}{h^2}+\la\right)h^\be u=0.
\end{align*}

If $\be=\al-\frac{1}{2}$, 
\begin{align*}
u''+\left(\al-\frac{1}{2}\right) \frac{ h''}{h} u
+\left(\frac{\alpha^2-\mu^2-\left(\alpha^2-\frac{1}{4}\right) (h')^2}{h^2}+\la\right) u=0.
\end{align*}

\begin{lem} Let $u_{\pm\mu}$ (assume $\mu\geq 0$) be the two linearly independent solutions of the equation
\begin{align*}
u''+\left(\al-\frac{1}{2}\right) \frac{ h''}{h} u
+\left(\frac{\alpha^2-\mu^2+\left(\frac{1}{4}-\alpha^2\right) (h')^2}{h^2}+\la\right) u=0,
\end{align*}
satisfying the conditions
\[
u_{\pm\mu}(x)\sim x^{\frac{1}{2}\pm \mu},
\]
for $x\to 0^+$. Then, the equation 
\[
f''+\frac{(1-2\alpha) h'}{h}f'+\left(\la+\frac{\alpha^2-\mu^2}{h^2}\right)f=0,
\]
has the following two linearly independent solutions
\[
f_{\al,\pm\mu}(x)=h(x)^{\al-\frac{1}{2}} u_{\pm\mu}(x).
\]
\end{lem}

\begin{lem}\label{applem1} Let $\al$ and $b$ be two real numbers with $b\not=0$. Let $\ff_{\al,b,\pm}$ a complete system of solutions  of the equation
\beq\label{es0}
f''+(1-2\al)\frac{h'}{h} f'+ \frac{b^2}{h^2} f=0,
\eeq
normalised according to Definition \ref{defi1}. Then, 
\beq\label{es00}
\ff'_{\al,b,\pm}=\left(\al\pm\sqrt{\al^2-b^2}\right) \ff_{-\al,b,\pm} h^{2\al-1}.
\eeq
\end{lem} 
\begin{proof}  Let $f_{\al,b}$ be a solution of equation (\ref{es0}). Define the function
\[
g=h^{1-2\al} f'_{\al,b}.
\]

Then, 
\begin{align*}
f_{\al,b}&=\int g h^{2\al-1}, & f'_{\al,b}&=h^{2\al-1}g,& f''_{\al,b}&=g' h^{2\al-1}+(2\al-1)h^{2\al-2}h'g.
\end{align*}

Substitution into equation (\ref{es0}) gives
\[
h^{2\al+1} g'+b^2 \int g h^{2\al-1}=0.
\]

Derivation of this equation gives the equation satisfied by the $\ff_{-\al,b,\pm}$, proving that $g$ is a linear combination of its solutions,
i.e. 
\[
f'_{\al,b}=h^{2\al-1}\left(A_+ \ff_{-\al,b,+} +A_- \ff_{-\al,b,-} \right),
\]
for some constants $A_\pm$. The indicial equation associated to the equation (\ref{es0}) is
\[
s(s-1)+(1-2\al)s+b^2=0,
\]
with roots $s_\pm=\al\pm\sqrt{\al^2-b^2}$. Since the  roots are different and non zero, by Definition \ref{defi1}, 
\[
\ff_{-\al,b,\pm}(x)= x^{-\al+\sqrt{\al^2+b^2}} F_\pm(x),
\]
with $F_\pm(0)=1$, and hence
\[
f'_{\al,b}(x)=A_+ x^{\al+\sqrt{\al^2+b^2}-1} F_+(x)H^{2\al-1}(x)+A_- x^{\al-\sqrt{\al^2+b^2}-1} F_-(x)H^{2\al-1}(x).
\]

On the other side, by Theorem \ref{boc},
\[
\ff'_{\al,b,\pm}(x)=s_\pm x^{s_\pm-1} \Psi_\pm(x),
\]

Comparing the two formulae we have the result. Observe that the indetermination of the $\ff_{-\al,b,-}$ does not compromise  the result. For if $f=\ff_{-\al,b,-}+c\ff_{-\al,b,+}$, for some constant $c$, then 
\[
f'_{\al,b}(x)=(A_++cA_-) x^{\al+\sqrt{\al^2+b^2}-1} F_+(x)H^{2\al-1}(x)+A_- x^{\al-\sqrt{\al^2+b^2}-1} F_-(x)H^{2\al-1}(x).
\]

Comparing with $\ff'_{\al,b,-}$ gives $A_-=s_-$, comparing with $\ff'_{\al,b,+}$ gives $A_-=0$, and $A_+=s_+$.
\end{proof}

\begin{lem}\label{applem2} Let $\al$ be a  real number different from $0$.   
Let $\ff_{\al,\pm}$ denote a fundamental system of solutions of the equation 
\beq\label{es1-1}
f''+(1-2\al)\frac{h'}{h} f'+ \la f=0.
\eeq

Then, 
\begin{align*}
\ff'_{\al,\pm}&=A_- h^{2\al-1}\ff_{-\al+1,\mp},& {\it if}~~ 0<\al<1,\\
\ff'_{\al,\pm}&=A_+ h^{2\al-1}\ff_{-\al+1,\pm},& {\it if}~~ \al\leq 0, ~\al\geq1,
\end{align*}
for some constant $A_\pm$.
\end{lem} 
\begin{proof}     Let $f_{\al}$ be a solution  of the equation \eqref{es1-1}. Define the function
\[
g=h^{1-2\al} f'_{\al}.
\]

Then, 
\begin{align*}
f_{\al}&=\int g h^{2\al-1}, & f'_{\al}&=h^{2\al-1}g,& f''_{\al}&=g' h^{2\al-1}+(2\al-1)h^{2\al-2}h'g.
\end{align*}

Substitution into equation \eqref{es1-1} gives
\[
h^{2\al-1} g'+\la \int g h^{2\al-1}=0.
\]

Derivation of this equation gives the equation 
\[
f''+(2\al-1)\frac{h'}{h} f'+ \la f=0,
\]
proving that $g$ is a linear combination of its  fundamental solutions $\ff_{-\al+1,\pm}$. The associated indicial equation is 
\[
s(s-1)+(2\al-1)s=0,
\]
with solutions $s_0=0$, $s_1=2-2\al$. Whence, if $\al\not=1$, by Definition \ref{defi1}
\begin{align*}
\ff_{-\al+1,0}&=\psi_0(x),&\ff_{-\al+1,1}=x^{2-2\al}\psi_1(x)
\end{align*}
with $\psi_{0/1}(0)=1$. We compute
\begin{align*}
h^{2\al-1} \ff_{-\al+1,0}(x)&=x^{2\al-1} F_0(x),\\
h^{2\al-1} \ff_{-\al+1,1}(x)&=x F_1(x).
\end{align*}

On the other side,  the indicial equation of of equation \eqref{es1-1} is
\[
x(x-1)+(1-2\al)x=0,
\]
with solutions $x_0=0$, $x_1=2\al$. Then, if $\al\not=0$, by Theorem \ref{boc}
\begin{align*}
\ff'_{\al,0}(x)&=\Psi_0(x), & \ff'_{\al,1}(x)&=x^{2\al-1}\Psi_1(x),
\end{align*}
with $\Psi_{0/1}(0)=0/1$. Comparing the expression, we prove the lemma when $\al\not=0,1$.
Observe that, as in the proof of the previous lemma, the indetermination of the solution with smaller exponent does not compromise the result.

If $\al=0$ then the associated indicial equation to $f''-\frac{h'}{h} f'+ \la f=0$ is
\[
s(s-1)-s=0,
\]
with solutions $s_0=0$ and $s_1 = 2$. Whence, by Definition \ref{defi1}
\begin{align*}
\ff_{1,0}(x)&=\psi_0(x),&\ff_{1,1}=x^{2}\psi_1(x),
\end{align*}
with $\psi_{0/1}(0) = 1$. We compute
\begin{align*}
h^{-1}(x) \ff_{1,0}(x)&=x^{-1} F_0(x),\\
h^{-1}(x) \ff_{1,1}(x)&=x F_1(x).
\end{align*}
On the other side, the indicial equation of \eqref{es1-1} is
\[
x(x-1)+x=0,
\]
with solutions $x_0=x_1=0$. Then,  by Theorem \ref{boc}
\begin{align*}
\ff_{0,0}(x)&=\psi_0(x), & \ff_{0,1}(x)&=\log x\psi_1(x),\\
\ff'_{0,0}(x)&=\psi'_0(x), & \ff'_{0,1}(x)&=x^{-1}\psi_{1}(x)\left(1+\log x \frac{\psi_1'(x)}{\psi_1(x)}\right),
\end{align*}
with $\psi_{0/1}(0)=1$ e $\psi'_{0/1} = 0$. Comparing the expressions we obtain the result.

If $\al=1$ then the associated indicial equation to $f''+\frac{h'}{h} f'+ \la f=0$ is
\[
s(s-1)+(2\al-1)s=0 \Rightarrow s^2 = 0.
\]
with solutions $x_0=x_1=0$. Whence, by Definition \ref{defi1}
 \begin{align*}
\ff_{0,0}(x)&=\psi_0(x),&\ff_{0,1}=\log x \psi_1(x),
\end{align*}
with $\psi_{0/1} = 1$. We compute
\begin{align*}
h (x)\ff_{0,0}(x)&=x F_0(x),\\
h (x)\ff_{0,1}(x)&=x \log x F_1(x).
\end{align*}
On the other side, the indicial equation \eqref{es1-1} is
\[
x(x-1)+(1-2\al)x=0,
\]
with indicial solutions $x=0$ and $x=2$. Then, by Theorem \ref{boc}
\begin{align*}
\ff_{1,0}(x)&=\psi_0(x), & \ff_{1,1}(x)&=x^2 \psi_1(x),\\
\ff'_{1,0}(x)&=\psi'_0(x), & \ff'_{1,1}(x)&=x \Psi_1(x),
\end{align*}
with $\psi_{0/1}(0) = 1$, $\psi'(0) = 0$ and $\Psi_1(0) = 2$. Comparing the expressions we obtain the result.

\end{proof}


\begin{lem}\label{meta} Let $p$ be a smooth function on $(0,1)$. The equation
\[
u''+ Qu=0,
\]
where
\[
Q=-\frac{1}{2}\left(\frac{1}{2} p^2+p'\right),
\]
on the space of smooth functions on the interval $(0,1)$ has the following  two l.i. solutions 
\begin{align*}
u_1(x)&=\e^{\frac{1}{2}\int p(x) dx },\\
u_2(x)&=\e^{\frac{1}{2}\int p(x) dx} \int\e^{-\int p(x) dx},
\end{align*}
\end{lem}
\begin{proof} (suggested by G. Metafune) Setting   
$f=u \e^{-\frac{1}{2}\int p}$, we have  the equation 
$f''+p f'=0$, that has the  solutions  $f_1=1$ and $f_2'(x)=\e^{-\int p(x) dx}$. 
\end{proof}

\section{Elements of zeta regularisation technique} 
\label{backss}

\subsection{Simple sequences and zeta determinant}

Let $S=\{a_n\}_{n=1}^\infty$ be a sequence
of  non negative real numbers with unique accumulation point at infinity. 
We denote by $S_0$ the positive part of $S$, i.e $S_0=S-\{0\}$. We assume that $S$ has finite exponent of convergence $\es(S)$, so that the associated Weierstrass canonical product converges uniformly and absolutely in any bounded closed region and is an integral function of order the genus $\gs(S)$ of the sequence $S$. In this setting, we define the function
\[
\frac{1}{\Gamma(-\lambda,S)}=\prod_{a_n\in S_0}\left(1+\frac{-\lambda}{a_n}\right)\e^{\sum_{j=1}^{\gs(S)}\frac{(-1)^j}{j}\frac{(-\lambda)^j}{a_n^j}},
\]
for $\la\in \rho(S)=\C-S$, that we call Gamma function associated to $S$.  Here $-\la$ denotes the complex variable defined on $\C- [0,+\infty)$, with $\arg{-\lambda=0}$ on $(-\infty, 0]$. 
We also introduce the zeta function, defined by the uniformly convergent series 
\[
\zeta(s,S)=\sum_{a_n\in S_0} a_n^{-s},
\]
for $\Re(s)>\es(S)$, and by analytic continuation elsewhere. We  assume that the sequence $S$ is a regular sequence of spectral type of non positive order, as defined in \cite[2.1, 2.6]{Spr9}. In this situation, the analytic extension of the zeta function associated to $S$ is regular at $s=0$, and we may define the zeta regularised determinant of $S$ by
\[
\det_\zeta S=\e^{-\zeta'(0,S)}.
\]

Note that if $S_0\not= S$, the determinant is defined as the determinant of the positive part of $S$, since the determinant of $S$ in this case would obviously vanish. 

Moreover, we have the following result \cite[2.11]{Spr9}.

\begin{theo} \label{teo1-1}
If $S$ is a regular sequence of spectral type of non positive order, the logarithmic Gamma function has an asymptotic expansion for large $\la$ in $\C-\Sigma_{c,\theta}$, where $\Sigma_{c,\theta}$ is some sector $\left\{z\in \C~|~|\arg(z-c)|\leq \frac{\theta}{2}\right\}$, with $c>0$ and $0<\theta<\pi$, that contains $S$, and  we have 
\[
\log \det_\zeta S=-\zeta'(0,S)=\Rz_{\la=+\infty} \log \Gamma(-\la,S).
\]
\end{theo}

\subsection{Double sequences and the spectral decomposition lemma}

We give some formulas to deal with zeta invariants of double series that follow as particular instances of the general results given in \cite{Spr9} (see also \cite{HS3}).

Given a double sequence $S=\{\lambda_{n,k}\}_{n,k=1}^\infty$  of non
vanishing complex numbers with unique accumulation point at the
infinity, finite exponent $s_0=\es(S)$ and genus $p=\gs(S)$, we use the notation $S_n$ ($S_k$) to denote the simple sequence with fixed $n$ ($k$), we call the exponents of $S_n$ and $S_k$ the {\it relative exponents} of $S$, and we use the notation $(s_0=\es(S),s_1=\es(S_k),s_2=\es(S_n))$; we define {\it relative genus} accordingly.

\begin{defi}\label{spdec} Let $S=\{\lambda_{n,k}\}_{n,k=1}^\infty$ be a double
sequence with finite exponents $(s_0,s_1,s_2)$, genus
$(p_0,p_1,p_2)$, and positive spectral sector
$\Sigma_{\theta_0,c_0}$. Let $U=\{u_n\}_{n=1}^\infty$ be a totally
regular sequence of spectral type of infinite order with exponent
$r_0$, genus $q$, domain $D_{\phi,d}$. We say that $S$ is
spectrally decomposable over $U$ with power $\kappa$, length $\ell$ and
asymptotic domain $D_{\theta,c}$, with $c={\rm min}(c_0,d,c')$,
$\theta={\rm max}(\theta_0,\phi,\theta')$, if there exist positive
real numbers $\kappa$, $\ell$ (integer), $c'$, and $\theta'$, with
$0< \theta'<\pi$,   such that:
\begin{enumerate}
\item the sequence
$u_n^{-\kappa}S_n=\left\{\frac{\lambda_{n,k}}{u^\kappa_n}\right\}_{k=1}^\infty$ has
spectral sector $\Sigma_{\theta',c'}$, and is a totally regular
sequence of spectral type of infinite order for each $n$;
\item the logarithmic $\Gamma$-function associated to  $S_n/u_n^\kappa$ has an asymptotic expansion  for large
$n$ uniformly in $\lambda$ for $\lambda$ in
$D_{\theta,c}$, of the following form
\beq\label{exp}
\log\Gamma(-\lambda,u_n^{-\kappa} S_n)=\sum_{h=0}^{\ell}
\phi_{\sigma_h}(\lambda) u_n^{-\sigma_h}+\sum_{l=0}^{L}
P_{\rho_l}(\lambda) u_n^{-\rho_l}\log u_n+o(u_n^{-r_0}),
\eeq
where $\sigma_h$ and $\rho_l$ are real numbers with $\sigma_0<\dots <\sigma_\ell$, $\rho_0<\dots <\rho_L$, the
$P_{\rho_l}(\lambda)$ are polynomials in $\lambda$ satisfying the condition $P_{\rho_l}(0)=0$, $\ell$ and $L$ are the larger integers 
such that $\sigma_\ell\leq r_0$ and $\rho_L\leq r_0$.
\end{enumerate}
\end{defi}

Define the following functions, ($\Lambda_{\theta,c}=\left\{z\in \C~|~|\arg(z-c)|= \frac{\theta}{2}\right\}$, oriented counter clockwise):
\beq\label{fi1}
\Phi_{\sigma_h}(s)=\int_0^\infty t^{s-1}\frac{1}{2\pi i}\int_{\Lambda_{\theta,c}}\frac{\e^{-\lambda t}}{-\lambda} \phi_{\sigma_h}(\lambda) d\lambda dt.
\eeq

By Lemma 3.3 of \cite{Spr9}, for all $n$, we have the expansions:
\beq\label{form}\begin{aligned}
\log\Gamma(-\lambda,S_n/{u_n^\kappa})&\sim\sum_{j=0}^\infty a_{\alpha_j,0,n}
(-\lambda)^{\alpha_j}+\sum_{k=0}^{p_2} a_{k,1,n}(-\lambda)^k\log(-\lambda),\\
\end{aligned}
\eeq
for large $\lambda$ in $D_{\theta,c}$. We set (see Lemma 3.5 of \cite{Spr9})
\beq\label{fi2}
\begin{aligned}
A_{0,0}(s)&=\sum_{n=1}^\infty \left(a_{0, 0,n} 
-\sum_{h=0}^\ell b_{\sigma_h,0,0}u_n^{-\sigma_h}\right)u_n^{-\kappa s},\\
A_{0,1}(s)&=\sum_{n=1}^\infty \left(a_{0, 1,n} 
-\sum_{h=0}^\ell b_{\sigma_h,j,1}u_n^{-\sigma_h}\right)
u_n^{-\kappa s},& 0&\leq j\leq p_2.
\end{aligned}
\eeq

\begin{theo} \label{sdl} Let $S$ be spectrally decomposable over $U$ as in Definition \ref{spdec}. Assume that the functions $\Phi_{\sigma_h}(s)$ have at most simple poles for $s=0$. Then,
$\zeta(s,S)$ is regular at $s=0$, and
\begin{align*}
\zeta(0,S)=&\zeta_{\rm reg}(0,S)+\zeta_{\rm sing}(0,S),\\
\zeta'(0,S)=&\zeta'_{\rm reg}(0,S)+\zeta'_{\rm sing}(0,S),
\end{align*}
where the regular and singular part are 
\begin{align*}
\zeta_{\rm reg}(0,S)=&-A_{0,1}(0),\\
\zeta_{\rm sing}(0,S)=&\frac{1}{\kappa}{\sum_{h=0}^\ell} \Ru_{s=0}\Phi_{\sigma_h}(s)\Ru_{s=\sigma_h}\zeta(s,U),\\
\zeta_{\rm reg}'(0,S)=&-A_{0,0}(0)-A_{0,1}'(0),\\
\zeta'_{\rm sing}(0,S)=&\frac{\gamma}{\kappa}\sum_{h=0}^\ell\Ru_{s=0}\Phi_{\sigma_h}(s)\Ru_{s=\sigma_h}\zeta(s,U)\\
&+\frac{1}{\kappa}\sum_{h=0}^\ell\Rz_{s=0}\Phi_{\sigma_h}(s)\Ru_{s=\sigma_h}\zeta(s,U)+{\sum_{h=0}^\ell}{^{\displaystyle
'}}\Ru_{s=0}\Phi_{\sigma_h}(s)\Rz_{s=\sigma_h}\zeta(s,U),
\end{align*}
and the notation $\sum'$ means that only the terms such that $\zeta(s,U)$ has a pole at $s=\sigma_h$ appear in the sum, and
\end{theo}



\begin{rem}\label{last} Observe that  in the expansion at point (2) of Definition \ref{spdec}, the terms where the functions $\phi_{\sigma_h}(\la)$ and the polynomial $P_{\rho_l}(\la)$ are constant (in $\la$) do not enter in the subsequent  results. For this reason it is sufficient in the enumeration of this terms (namely in the indices $\sigma_h$ and $\rho_l$) to consider only the other terms.
\end{rem}

\begin{corol} \label{c} Let $S_{(j)}=\{\lambda_{(j),n,k}\}_{n,k=1}^\infty$, $j=1,...,J$, be a finite set of  double sequences that satisfy all the requirements of Definition \ref{spdec} of spectral decomposability over a common sequence $U$, with the same parameters $\kappa$, $\ell$, etc., except that the polynomials $P_{(j),\rho}(\lambda)$ appearing in condition (2) do not vanish for $\lambda=0$. Assume that some linear combination $\sum_{j=1}^J c_j P_{(j),\rho}(\lambda)$, with complex coefficients, of such polynomials does satisfy this condition, namely that $\sum_{j=1}^J c_j P_{(j),\rho}(\lambda)=0$. Then, the linear combination of the zeta function $\sum_{j=1}^J c_j \zeta(s,S_{(j)})$ is regular at $s=0$ and satisfies the linear combination of the formulas given in Theorem \ref{sdl}.
\end{corol}